\documentclass[12pt]{memo-l}

\usepackage{amsmath, amsfonts, amssymb,amsthm, amscd}
\usepackage{pictexwd,dcpic}
\usepackage{graphicx}
\usepackage{epsf}
\usepackage{epsfig}
\usepackage{psfrag}
\usepackage{color}
\usepackage{hyperref}
\usepackage{marginnote}

\newtheorem{theorem}{Theorem}[chapter]
\newtheorem{proposition}[theorem]{Proposition}
\newtheorem{lemma}[theorem]{Lemma}

\theoremstyle{definition}
\newtheorem{definition}[theorem]{Definition}

\theoremstyle{remark}
\newtheorem{remark}[theorem]{Remark}

\numberwithin{section}{chapter}
\numberwithin{equation}{chapter}
\numberwithin{figure}{chapter}

\newtheorem*{P4.17}{Proposition \ref{PROPX}}
\newtheorem*{P4.19}{Proposition \ref{PROP4.19}}
\newtheorem*{LXX}{Lemma \ref{TORTE}}
\newtheorem*{L2.6}{Lemma \ref{lem2.4}}
\newtheorem*{L3.4}{Lemma \ref{GROMOVCONV}}

\newcommand{\nr}{\pmb{|}\negthinspace\pmb{|}}
\newcommand{\bb}{\beta'_{b(a,v)}}
\newcommand{\ga}{\gamma_{a}}
\newcommand{\bba}{\beta_{a}}
\newcommand{\bd}{B_{\delta}}
\newcommand{\pav}{\phi_{a, v}}
\newcommand{\psav}{\psi_{a, v}}
\newcommand{\pt}{\partial_t}
\newcommand{\ps}{\partial_s}
\newcommand{\bt}{B_{\frac{1}{2}}}
\newcommand{\htl}{H^2_{\rm{loc}}}
\newcommand{\av}{\text{av}}
\newcommand{\hb}{B_{\frac{1}{2}}}
\newcommand{\C}{{\mathbb C}}

\newcommand{\norm}[1]{\left\Vert#1\right\Vert}
\newcommand{\abs}[1]{\left\vert#1\right\vert}

\newcommand{\ssc}{\text{sc}}

\renewcommand{\epsilon}{\varepsilon}
\newcommand{\tl}{\triangleleft}
\newcommand{\what}{\widehat}
\newcommand{\wh}{\widehat}

\newcommand{\wt}{\widetilde}

\newcommand{\call}{{\mathcal L}}

\newcommand{\ov}{\overline}
\newcommand{\si}{\Sigma}
\newcommand{\cf}{{\mathcal F}}
\newcommand{\cg}{{\mathcal G}}

\newcommand{\cl}{\operatorname{cl}}
\newcommand{\ce}{{\mathcal E}}

\newcommand{\calm}{{\mathcal M}}

\newcommand{\Z}{{\mathbb Z}}

\newcommand{\id}{\operatorname{id}}

\newcommand{\loc}{\text{loc}}

\newcommand{\ind}{\textrm{Ind\ }}

\newcommand{\supp}{\operatorname{supp}}

\newcommand{\cs}{{\mathcal S}}
\newcommand{\cae}{{\mathcal E}}
\newcommand{\co}{{\mathcal O}}
\newcommand{\mcr}{{\mathcal R}}
\newcommand{\bx}{{\bf X}}
\newcommand{\be}{{\bf E}}
\newcommand{\R}{{\mathbb R}}

\providecommand{\ker}[1]{$\text{ker}\ {#1}$}

\newcommand{\N}{{\mathbb N}}
\newcommand{\Q}{{\mathbb Q}}
\newcommand{\pr}{\text{pr}}

\makeindex

\begin{document}

\frontmatter

\title{Applications of Polyfold Theory {I}:\\ The Polyfolds of
 Gromov--Witten Theory}


\author{ H. Hofer}
\address{Institute for Advanced Study,  USA}
\thanks{Research partially supported
by NSF grant  DMS-0603957 and DMS-1104470}
\email{hofer@ias.edu}
\author{K. Wysocki}
\address{Penn State University}
\email{wysocki@math.psu.edu}
\thanks{Research
partially supported by  NSF grant DMS-0906280}

\author{ E. Zehnder}
\address{ETH-Zurich\\Switzerland}
\email{zehnder@math.ethz.ch}





\begin{abstract}
In this paper we start with the construction of the symplectic field theory (SFT). As a general theory of symplectic invariants, SFT has been outlined in \cite{EGH}  by Y. Eliashberg, A. Givental and H. Hofer who have predicted its formal properties. The actual construction of SFT is a hard analytical problem which will be overcome be means of the polyfold theory due to the present authors. The current paper addresses  a significant amount of the  arising  issues and the general theory will be completed in part II of this paper. To illustrate the polyfold theory we shall use the results of the present paper to describe an alternative construction of the Gromov-Witten  invariants for general compact symplectic manifolds. 
\end{abstract}

\maketitle

\tableofcontents


\mainmatter
%
%
%


\chapter{Introduction and Main Results}
In this paper  we start with the application of the polyfold theory to  the symplectic field theory (SFT)  outlined in \cite{EGH}. It turns out that the polyfold structures near noded stable curves lead also naturally to a polyfold description of the Gromov-Witten theory which is  a by-product of  the analytical foundation of SFT, presented here. The polyfold constructions for SFT will be completed in \cite{HWZ5}. 
The Gromov-Witten invariants (GW-invariants) are invariants of symplectic manifolds deduced from the structure of stable pseudoholomorphic maps from noded Riemann surfaces to the symplectic manifold. The construction of GW invariants of  general symplectic manifolds goes back to Fukaya-Ono in \cite{FO} and Li--Tian in \cite{LT}. Cieliebak--Mohnke  studied the genus zero case  in \cite{CM}. Earlier work for special symplectic manifolds are due to Ruan in  \cite{R1} and \cite{R2}. We suggest \cite{MW} for a discussion of some of the inherent difficulties in these approaches.

Our approach to the GW-invariants is quite different from the approaches in the literature. We shall apply the  general Fredholm theory developed in  \cite{HWZ2,HWZ3,HWZ3.5} and surveyed in \cite{Hofer} and \cite{H2}. {A comprehensive discussion of the abstract theory will be contained in the upcoming lecture note \cite{HWZ10}}. The theory is designed for the analytical foundations of the SFT in \cite{HWZ5}.

 We recall from \cite{HWZ3.5} that a polyfold $Z$ is a metrizable space equipped with  with an equivalence class of polyfold structures. A polyfold structure $[X,\beta]$ consists of an ep-groupoid $X$ which one could  consider as a generalization of an  \'etale proper Lie groupoid whose object and morphism sets have M-polyfold structures instead of manifold structures, and whose structure maps are sc-smooth maps. Moreover, $\beta:\abs{X}\to Z$ is a homeomorphism between the orbit space of $X$ and the topological space $Z$. The relevant concepts here are recalled in Section \ref{sc-smoothness} below.  

Our strategy to obtain the GW-invariants is as follows. We first construct the ambient space $Z$ 
of stable curves,  from noded  Riemann surfaces to the symplectic manifold,  which are not required to be pseudoholomorphic. The space $Z$ has a natural paracompact {Hausdorff} topology and we construct an equivalence class of natural polyfold structures $[X,\beta]$ on $Z$. The second step constructs a so called strong bundle $p:W\to Z$ which will be equipped with an sc-smooth strong polyfold bundle structure. In the third step we shall show that the Cauchy-Riemann operator $\ov{\partial}_J$ defines an sc-smooth section of the bundle $p$ which is a particular case of SFT. We shall prove that $\ov{\partial}_J$ is an sc-smooth proper Fredholm section of the bundle $p:W\to Z$. The solution sets of the section $\ov{\partial}_J$ are the Gromov compactified moduli spaces which, as usual, are badly behaving sets. However, the three ingredients already established at this point  immediately allows one to apply  the abstract Fredholm perturbation theory from  \cite{HWZ3,HWZ3.5}. {After a  generic perturbation,  the solution sets of the perturbed Fredholm problem become smooth objects, namely compact, weighted, smooth  branched sub-orbifolds. They also have a natural orientation,  so that the branched integration theory from \cite{HWZ7} allows one  to integrate the sc-differential forms over the perturbed solution sets to obtain the GW-invariants in the form of integrals.}

Our main concern in the following is the construction of the polyfold structures which allows one  to deal with noded objects in a smooth way. For this purpose we describe, in particular, the normal forms for families of  noded  Riemann surfaces in the Deligne-Mumford theory used in our constructions.  We also include some related technical results needed for the SFT in \cite{HWZ5}.

\section{The Space Z of Stable Curves}
We start with the construction of the ambient space $Z$ of stable curves. The stable curves are not required to be pseudoholomorphic. We consider maps defined on noded Riemann surfaces $S$ having their images in the closed symplectic manifold $(Q,\omega)$ and possessing various regularity properties.

We shall denote by 
$$
u:{\mathcal O}(S,x)\rightarrow Q
$$
a germ of a map  defined on (a piece) of Riemann
surface $S$ around $x\in S$. 
\noindent {Throughout  the paper we identify $S^1$ with $\R/\Z$ unless otherwise noted.}
\noindent {Moreover, smooth (in the classical sense) means $C^\infty$-smooth.}

\begin{definition}
Let $m\geq 2$ be an integer and $\delta\geq 0$.  A germ of
a continuous map $u:{\mathcal O}(S,x)\rightarrow Q$ is called of class
$(m,\delta)$ near the point $x$,\index{map of class $(m,\delta)$}  if for a smooth chart
$\psi:U(u(x))\rightarrow {\mathbb R}^{2n}$ mapping $u(x)$ to $0$ and for 
holomorphic polar coordinates $ \sigma:[0,\infty)\times
S^1\rightarrow S\setminus\{x\}$ around $x$, the map
$$
v(s,t)=\psi\circ u\circ \sigma(s,t)
$$
which is defined for $s$ large,  has {weak} partial derivatives up to order
$m$, which weighted by $e^{\delta  s}$ belong to the space 
$L^2([s_0,\infty)\times S^1,{\mathbb R}^{2n})$  for $s_0$  sufficiently
large. We  call the germ  of class $m$ around the  point $z\in S$,  if 
$u$ is of class $H^m_{loc}$ near $z$.
\end{definition}
One easily verifies that if $\sigma$ is a germ of biholomorphic map
mapping $x\in S$ to $y\in S'$ then  $u$ is of class
$(m,\epsilon)$ near $x$ if and only if the same is true for
$u\circ\sigma^{-1}$ near $y$. Moreover, the above
definition does not depend on the choices involved, like charts and
holomorphic polar coordinates.

\begin{definition}
A {\bf noded Riemann surface}  with marked points is a tuple $(S,j,M,D)$\index{noded Riemann surface with marked points} in which 
{$(S,j)$ is  an oriented closed smooth surface $S$ equipped with a smooth almost complex structure $j$. }The subset $M$ of $S$  is  a finite collection of marked points which can be ordered or un-ordered, and $D$ is a finite collection of un-ordered pairs $\{x,y\}$ of points in $S$ so that $x\neq y$ and two pairs which intersect are identical. The union of all sets $\{x,y\}$ belonging to $D$, denoted by $|D|$,  is disjoint from $M$. We call $D$ the set of {\bf nodal pairs}  and $|D|$ the set of  {\bf nodal points.}
\end{definition}
{It is a classical result, proved for example in \cite{SD}, Theorem 3.2, that the pair $(S, j)$ determines a unique compatible Riemann surface  structure in the sense of complex manifolds.}

The Riemann surface $S$ can consist of different connected components $C$ called domain components.  
The noded Riemann surface  $(S,j,M,D)$ is called {\bf connected}  if  the topological space $\ov{S}$,  obtained by identifying  the points $x$ and $y$ in the nodal pairs $\{x,y\}\in D$,  is connected.

So in our terminology  it is possible that the  noded  surface  $(S,j,M,D)$  is  connected, but  the Riemann surface $S$ has  several connected components, namely  its domain components.

The {\bf arithmetic genus}  $g_a$ of a connected noded Riemann surface $(S, j, M, D)$  is the integer $g_a$ defined by 
$$g_a=1+\sharp D+\sum_{C}[ g(C)-1]$$
where $\sharp D$ is the  number of nodal pairs in $D$ and where the sum  is taken over the finitely many domain components $C$ of the Riemann surface $S$, and where $g(C)$ denotes the genus of $C$. The  arithmetic genus $g_a$ agrees  with the genus of the connected {closed}  Riemann surface obtained by taking disks around the nodes in every nodal pair and replacing the two disks by a connecting tube. 
In the following we refer to the elements of $M\cup |D|$ as to  the special points. The set of special points lying on the domain component $C$ is abbreviated by $\Sigma_C:=C\cap (M\cup|D|)$.

Two  connected  noded  Riemann surfaces
$$(S,j,M,D)\quad \text{and}\quad (S',j',M',D')$$
are called isomorphic (or equivalent) if there exists a biholomorphic map 
$$\phi:(S,j)\rightarrow (S',j')$$
(i.e., the diffeomorphism satisfies $T\phi\circ j=j'\circ T\phi$) 
mapping the marked points onto the marked points  and the nodal pairs onto the nodal pairs, hence satisfying $\phi(M)=M'$  and $\phi_\ast(D)=D'$ where 
$$
\phi_\ast(D)=\bigl\{ \{\phi(x),\phi(y)\}\in D' \vert  \,  \{x,y\}\in D \bigr\}.
$$
If the marked points $M$ and $M'$ are ordered it is required that $\phi$ preserves the order. 
If the two noded Riemann surfaces are identical, the isomorphism above is called an automorphism of the noded surface $(S,j,M,D)$. In the following we denote by 
$$[(S, j, M, D)]$$
the equivalence class of all connected noded Riemann  surfaces isomorphic to the {\bf connected noded Riemann  surface}  $(S,j,M,D)$.
\begin{definition} The connected noded Riemann surface $(S,j,M,D)$  is called {\bf stable}  if its automorphism group $G$  is finite \index{stable! connected noded Riemann surface}
\end{definition}
One knows that a connected noded Riemann  surface  $(S,j,M,D)$ is  stable if and only if 
every domain component $C$  of $S$ satisfies 
$$
2\cdot g(C) +\sharp\Sigma_C\geq 3
$$
where $g(C)$ is the genus of $C$.

Next we describe the tuples $\alpha=(S, j, M, D, u)$ in which $(S, j, M, D)$ is a, not necessarily stable,  noded Riemann surface with ordered marked points, and $u:S\to Q$ a continuous map,  in more detail.

\begin{definition} [{\bf Stable maps and stable curves}] \index{stable! map} \index{stable! curve}The tuple 
$$\alpha=(S, j, M, D, u)$$ is called  a {\bf stable map} (of class $(m,\delta)$, where  $m\geq 3$ and {$\delta\geq 0$}),
if it has the following properties, 
\begin{itemize}
\item[$\bullet$] The underlying topological space,  obtained by identifying
the two points in every  nodal pair,  is connected. 
\item[$\bullet$] The map $u$ is  of class $(m,\delta)$ around the nodal points in $|D|$ and 
of class $m$ around all other points. (For certain
applications it is useful to require  the map $u$ around the marked points in $M$ to be of class $(m,\delta)$ as well; 
the   minor modifications  are left to the
reader.)
\item[$\bullet$] $u(x)=u(y)$  at  every nodal pair $\{x,y\}\in D$.
\item[$\bullet$]  {$\int_C u^\ast\omega\geq 0$ for every component $C$ of $S$.}
\item[$\bullet$] {\bf Stability condition:}\index{stability condition}\,   if a domain component $C$ of $S$ has genus $g (C)$ and 
$\sharp \Sigma_C$ special points, and satisfies $2\cdot g(C) +\sharp \Sigma_C\leq 2$,  {i.e., $C$ is not stable} , then  
$$\int_C u^\ast\omega >0.$$

\end{itemize}
Two {stable maps}  $\alpha=(S,j,M,D,u)$ and $\alpha'=(S',j',M',D',u')$ are called
equivalent if there exists an isomorphism 
$\phi:(S,j,M,D)\rightarrow (S',j',M',D') $  between  the noded Riemann surfaces  satisfying 
 $$u'\circ\phi=u.$$ 
 Here $\phi$ preserves the ordering of the marked points.
 An equivalence class is called a {\bf stable curve} of class
$(m,\delta)$. Hence a stable curve  is an equivalence class of stable maps and will be  denoted by $[\alpha]$ if $\alpha$ is a representative.
\end{definition}
Next we introduce the space $Z$ which will be equipped with
a polyfold structure.
\begin{definition}
Fix $\delta_0\in (0,2\pi)$. The collection of all equivalence
classes $[\alpha]$ of stable maps  $\alpha$ of class $(3,\delta_0)$ is called the
space of stable curves in $Q$ of class $(3,\delta_0)$ and is denoted
by $Z$ or by $Z^{3,\delta_0}(Q,\omega)$. \index{space of stable curves}
\end{definition}
{We shall equip the set $Z$ with  a natural topology. By  ``natural" we indicate  that the topology is independent
of the choices involved in its construction. The topology is related to the topology on the
Sobolev space of $H^3$-maps on a punctured Riemann surfaces with
exponential decay near the nodes. In Section \ref{natural_topology_Z_section} we shall prove the following theorem.}

\begin{theorem}\label{th-top}
For given $\delta_0\in (0,2\pi)$ the associated space $Z=Z^{3,\delta_0}(Q, \omega)$
of stable curves in $Q$ of class $(3,\delta_0)$ has a natural second countable paracompact Hausdorff topology.
\end{theorem}
We shall
prove that  the topological space $Z$ carries a  natural polyfold structure provided some
other pieces of data are fixed. A polyfold is  quite similar to a possibly
infinite-dimensional orbifold, where however, and this is crucial for applications,
the local models are sc-smooth retracts divided by  finite group actions.
The notion of sc-smoothness is a new notion of smoothness which in infinite dimensions
is much weaker than the notion of Fr\'echet differentiability. Whereas a smooth
retract for the latter is a split submanifold, an sc-smooth retract can be a very
wild set of locally varying dimensions. The polyfold theory,
among other things,  generalizes differential geometry
to spaces which locally look  like sc-smooth retracts or quotients thereof.
The more fancy local models are needed since the spaces have
to incorporate  analytical limiting behaviors,  like bubbling-off and
breaking of trajectories,  which cannot be satisfactorily described
in the classical set-up of manifolds. We refer to \cite{HWZ3.5}
for the definition of a polyfold structure. 
For the convenience of the reader,  the concepts are recalled in Section \ref{sc-smoothness} below.

In order to describe a preliminary version of our main result we observe that the gluing construction at the nodes requires the conversion of the absolute value $\abs{a}$ of a non-zero complex number $a$ into a real number $R>0$ which is essentially the modulus of the associated cylindrical neck. The conversion is defined by a so called gluing profile $\varphi$ which is a diffeomorphism $\varphi:(0, 1]\to [0, \infty)$.

The main result is the following theorem,  {proved in Section \ref{polyfoldstructure}.} 

\begin{theorem}\label{pfstructure}
Given a strictly increasing sequence $(\delta_m)$, starting at the
previously chosen $\delta_0$ and staying below $2\pi$, and the
gluing profile $\varphi(r)=e^{\frac{1}{r}}-e$,  the space $Z=Z^{3, \delta_0}(Q, \omega)$ of stable curves into $Q$ has {in a
natural way the structure of a polyfold}  for which  the $m$-th level consists of
equivalence classes of stable maps $(S,j,M,D,u)$ in which $u$ is of class
$(m+3,\delta_m)$.
\end{theorem}

We shall formulate  a more precise statement later on after  some more preparations. The constructions will show that there
are natural polyfold charts which depend on the choice of the gluing
profile $\varphi$, which we take as specified above (other choices
would be possible). These charts work for all strictly increasing sequences
$\delta_m$ starting at $\delta_0$ and staying below $2\pi$.
Staying below $2\pi$ is an important requirement in  the Fredholm theory.

Let us also remark that we can define similar spaces using maps of
Sobolev class $W^{m,p}$ provided $W^{m, p}$ is continuously embedded  into
$C^1$. This regularity will be needed for the transversal constraint
construction. We will not pursue this  here and leave the details to the
reader.  One has to keep in mind that dealing with the Banach spaces  $W^{m,p}$, $p\neq 2$, one has to check the  existence of sc-smooth partitions of unity, which are guaranteed in  case  the models are sc-Hilbert spaces.

We point out that $Z$ has many connected components
as well as many interesting open subsets.
If $g,m$ are  nonnegative integers we denote by $Z_{g,m}$ the subset of $Z$
consisting of all equivalence  classes $[S,j,M,D,u]$ in which  the underlying noded Riemann surface $(S,j,M,D)$ has arithmetic genus $g$ and
$m$ marked points. This subset is open in $Z$ and therefore has an
induced polyfold structure. If $A\in H_2(Q,{\mathbb Z})$ is a second homology class, we can also
consider the set $Z_{A,g,m}$ which is the open subset of $Z_{g,m}$  consisting of elements $[S, j, M, D, u]$ in which  the map $u$ represents $A$.

There are  natural maps which play an important role in the 
GW-theory and the SFT.  Consider for a  fixed pair $(g,m)$ of nonnegative integers  satisfying $2g+m\geq 3$ the space $Z_{g,m}$.
  The evaluation map at
the $i$-th marked point is defined by 
$$ 
ev_i:Z_{g,m}\rightarrow Q,\quad [S,j,M,D,u]\mapsto u(x_i)
$$
for $i=1,\ldots, m$. 
Further, if $2g+m\geq 3$, we have the forgetful map associating with the stable curve $[\alpha]$ the underlying stable part of the domain. It is
obtained as follows. We take a representative $(S,j,M,D,u)$ of our
class $[\alpha]$. First we forget the map $u$ and consider, if it exists, a component $C$
satisfying $2g(C)+\sharp (C\cap (M\cup |D|))<3$. Then we have the
following cases. Firstly $C$ is a sphere without marked points and with  one nodal point, say $x$. Then  we remove
the sphere, the nodal point $x$ and its partner $y$, where $\{x,y\}\in D$. Secondly $C$ is a sphere with two nodal points. In this case there are two  nodal pairs $\{x,y\}$ and $\{x',y'\}$, where $x$ and
$x'$ lie on the sphere.  We remove the sphere and the two nodal pairs  but add the nodal
pair $\{y,y'\}$. Thirdly $C$ is a sphere with one node and one
marked point. In that case we  remove the sphere but replace the
corresponding nodal point  on the other component by the marked point.
Continuing this way we end up with a stable noded marked Riemann
surface whose biholomorphic type does not depend on the order we
`weeded out'  the unstable components. This way we obtain the map
$$
\gamma:Z_{g,m}\rightarrow \overline{\mathcal
M}_{g,m},\quad  [S,j,M,D,u]\mapsto  [(S,j,M,D)_{stab}].
$$
 Here $\overline{\mathcal M}_{g,m}$ is the standard Deligne-Mumford
compactification  of the space of (ordered) marked stable Riemann surfaces
with its holomorphic orbifold structure. {The map $\gamma$ is illustrated in  Figure \ref{Fig1}}.

\begin{figure}[!htb]
\psfrag{s}{$\gamma$}
\centering
\includegraphics[width=3.8in]{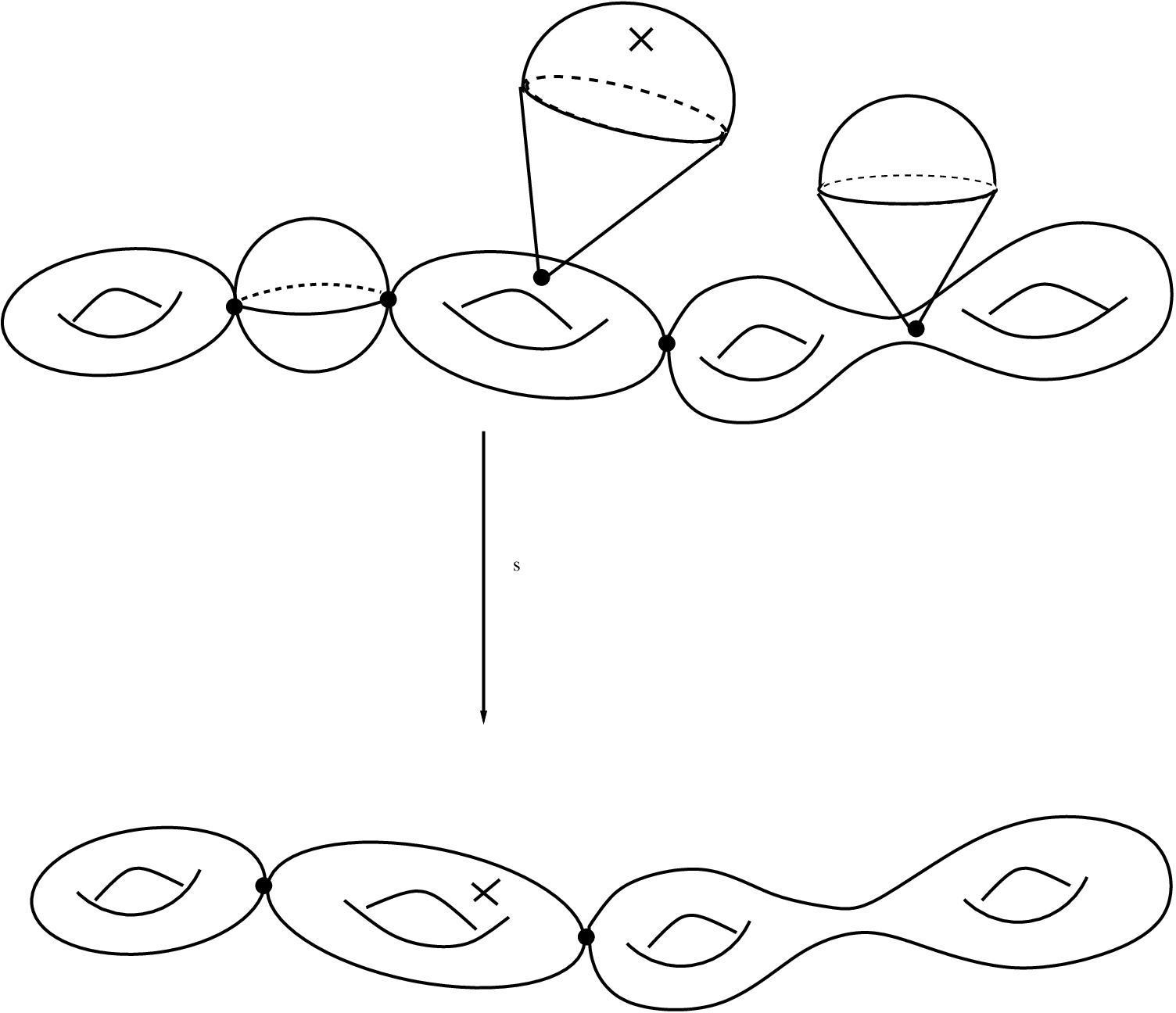}
\caption{The forgetful map $\gamma:Z_{g,m}\to \ov{\calm }_{g,m}$}
\label{Fig1}
\end{figure}

\begin{theorem} \label{evaluation_thm}
 If $2g+m\geq 3$, then 
the maps $ev_i:Z_{g,m}\rightarrow Q$ for $1\leq i\leq m$ and $\gamma:Z_{g,m}\rightarrow
\overline{\mathcal M}_{g,m}$  are sc-smooth.
\end{theorem}

As a consequence we can pull-back differential forms on $Q$ and $
\overline{\mathcal M}_{g,m}$ to obtain sc-differential forms on the polyfold 
$Z_{g,m}$ which,  suitably wedged together,  can be  integrated over
the smooth moduli spaces obtained as solution sets of  transversal Fredholm (multi-)sections
of suitable strong bundles over $Z$. Here one makes use of  the branched integration
theory  developed in \cite{HWZ7}.

\section{The Bundle W}\label{sect1.2}
Next we introduce the  strong polyfold bundle $W$ over $Z$.
We consider the closed symplectic manifold $(Q,\omega)$ and  choose a 
compatible almost complex structure $J$ on $Q$ so that  $\omega\circ (\id \oplus J)$ is a Riemannian metric on $Q$.

The points
of $W$ are defined as follows. We consider tuples 
$$
\wh{\alpha}=(\alpha,\xi)=(S,j,M,D,u,\xi)
$$
in which the stable map 
$\alpha=(S,j,M,D,u)$ is a representative of an element in $Z$.
Moreover, $\xi$ is a continuous section along $u$ such that the map
 $$
 \xi(z):T_zS\rightarrow T_{u(z)}Q,\quad \text{for $z\in
S$},
$$
 is a complex anti-linear map. Its domain of definition  is  equipped with  the complex structure $j$ and
its target space is equipped with the almost complex structure $J$. Moreover, on $S\setminus |D|$, the map
 $z\rightarrow \xi(z)$ is of class $\htl$. At the nodal  points in
 $|D|$ we require that $\xi$ is of class $(2,\delta_0)$. This requires, 
 taking holomorphic polar coordinates $\sigma$ around the  point $x\in |D|$ and
 a chart $\psi$ around $u(x)$ in $Q$,  that  the map
 $$
 (s,t)\mapsto 
 T\psi(u(\sigma(s,t)))\xi(\sigma(s,t))(\partial_s\sigma (s,t)), 
 $$
and  its weak  partial derivatives up to order $2$,  weighted
 by $e^{\delta_0 |s|}$ belong to the space  $L^2([s_0,\infty)\times S^1,{\mathbb
 R}^{2n})$ for $s_0$ large enough. The definition does not depend on the
 choices involved. We call two such tuples
 equivalent if  there exists an isomorphism  
 $$
 \phi:(S,j,M,D)\rightarrow (S',j',M',D')
 $$
satisfying 
$$
\xi'\circ T\phi=\xi\ \ \hbox{and}\ \ u'\circ\phi=u.
$$
The 
equivalence class  of $\wh{\alpha}=(S,j,M,D,u,\xi)$  is denoted  by $[\wh{\alpha}]$. 

The collection of all such equivalence classes constitutes the space  $W$.

We have defined what it means that an
element $\alpha$ represents an element on level $m$. Let us observe
that if the map $u$ has regularity $(m+3,\delta_m)$ it makes sense to talk
about elements $\xi$ along $u$ of regularity $(k+2,\delta_k)$ for
$0\leq k\leq m+1$. In the case $k=m+1$ the fiber
regularity is $(m+3,\delta_{m+1})$ and the underlying base
regularity is $(m+3,\delta_m)$. The requirement of a faster
exponential decay in the fiber towards a nodal point than the
exponential decay of the underlying base curve is well-defined and
independent of the charts chosen to define it. Our conventions for defining the levels are governed by the overall
convention that sections should be horizontal in the sense that they
preserve the level structure, i.e.,  an element on level $m$ is mapped
by the section $\xi$ to an element on bi-level $(m,m)$. Hence if the
section comes from a first order differential operator we need
precisely the  convention  we have just introduced.  {We  call   an element} 
$$
\wh{\alpha}=(S,j,M,D,u,\xi)
$$
of  (bi)-regularity $((m+3,\delta_m),(k+2,\delta_k))$ as long as $k$
satisfies the above restriction $0\leq k\leq m+1$. We say that $[\wh{\alpha}]\in
W$ is on level $(m,k)$ provided $(u,\xi)$ has the above regularity.
We keep in mind that  bi-levels $(m,k)$ are defined under the restriction $0\leq k\leq m+1$.
\begin{theorem}
The set $W$ has a natural second countable paracompact {Hausdorff} topology so
that the projection map
$$
p:W\rightarrow Z, \quad [\wh{\alpha}]\mapsto [\alpha], 
$$
(forgetting the $\xi$-part),  is continuous.
\end{theorem}

Our  main result concerning the polyfold set-up is the following theorem{, proved in Section \ref{polstrbundle}.}

\begin{theorem}\label{main1.10}
Let $Z=Z^{3,\delta_0}(Q, \omega)$ be the previously introduced space  of stable curves with its polyfold
structure associated with  the increasing sequence $(\delta_m)\subset
(0,2\pi)$ and the exponential gluing profile $\varphi$. Then the bundle $p:W\rightarrow Z$ has in a natural way the
structure of a strong polyfold bundle in which the $(m,k)$-bi-level
(for $0\leq k\leq m+1$) consists of elements of base regularity $(m+3,\delta_m)$ and of  fiber
regularity $(k+2,\delta_k)$.
\end{theorem}
We refer the reader to \cite{HWZ3.5}, \cite{H2}, and \cite{HWZ7} for the background material on
polyfold theory.  For the reader's convenience we shall also review  the main concepts during the proofs.

\section{Fredholm Theory}
For a compatible smooth almost complex
structure $J$ on $(Q,\omega)$ we define  the Cauchy-Riemann section $\ov{\partial}_J$
of the strong polyfold bundle $p:W\rightarrow Z$ by
$$
\ov{\partial}_J([S,j,M,D,u])=[S,j,M,D,u,\ov{\partial}_{J,j}(u)]
$$
where
$$
\ov{\partial}_{J,j}(u) =\frac{1}{2}\bigl[ Tu+J(u)\circ Tu\circ j\bigr].
$$
We  call a Fredholm section of a strong polyfold bundle
component-proper provided the restriction to every connected
component of the domain is proper.
\begin{theorem}\label{maincauchy-riemann}
The  Cauchy-Riemann section $\ov{\partial}_{J}$ of the strong polyfold bundle  $p:W\rightarrow
Z$ is an sc-smooth component-proper Fredholm section, which is naturally oriented. On the component 
$Z_{A, g, m}$ of the polyfold $Z$ the Fredholm index of  $\ov{\partial}_{J}$ is equal to 
$$\text{Ind}\ (\ov{\partial}_{J})=2c_1(A)+(2n-6)(1-g)+2m,$$
with $2n=\text{dim}\ Q$, where $g$ is the arithmetic genus of the noded Riemann surfaces, and $m$ the number of marked points and $A\in H_2(Q).$
\end{theorem}
{Theorem \ref{maincauchy-riemann} is restated and proved up to the orientability as 
Theorem \ref{FREDPROP} in Section \ref{c-r-results}. 
 The orientation problem is explained in the Appendices \ref{orientations-abstract} and \ref{orientations}}. It is a special  case of the more elaborate
orientation problems  in SFT, which will be studied in  \cite{HWZ5}.

\section{The GW-invariants}
At this point, in view of Theorem \ref{pfstructure}, Theorem \ref{main1.10} and 
Theorem \ref{maincauchy-riemann},  we are in the position to apply the general   functional analytic Fredholm theory developed in
\cite{HWZ2,HWZ3,HWZ3.5}.   The theory provides, in particular, an abstract perturbation theory to achieve transversality of the Fredholm section. In general, however, transversality can only be  achieved by means of multivalued perturbations.

As a side remark we observe that on components of the polyfold  $Z$,  on which   the solutions are somewhere injective, the more classical theory is still applicable and the transversality of the Fredholm section can be achieved 
by choosing a suitable  almost complex structure  $J$ on the symplectic manifold $(Q,\omega)$. So,  the classical theory of the pseudoholomorphic curve equation is part of our  general
constructions, see \cite{MS2}.

 In general, following the recipe of the abstract perturbation theory in \cite{HWZ3.5}, one {chooses}  a generic 
 $\text{sc}^+$-multisection of the polyfold bundle $W\to Z$ such that the components of the solution spaces of the perturbed section are  oriented, compact, weighted, branched suborbifolds of the polyfold $Z$ of even 
dimension. Depending on the components $Z_{A,g,m}$, the solution spaces can also  be
smooth oriented manifolds or orbifolds. The  solution spaces of 
different perturbations are connected by oriented cobordisms in the same category.

To be precise we consider the  polyfold $Z$ of Theorem \ref{pfstructure}, the strong polyfold bundle $p:W\to Z$
of Theorem \ref{main1.10} and the sc-smooth component-proper Fredholm section $\ov{\partial}_J$ of Theorem 
\ref{maincauchy-riemann}. Given the homology class $A\in H_2(Q, Z)$  of the closed symplectic manifold $(Q,\omega)$ and two integers $g,m\geq 0$,
with $2g+m\geq 3$, we focus on the polyfold $Z_{A, g,m}\subset Z$ of equivalence classes 
$[(S, j, M, D, u)]$ in which the noded Riemann surface $S$ has arithmetic  genus $g$ and is equipped with $m$ marked points; the map $u$ represents the homology class $A$ of $Q$. The sc-smooth evaluation maps $ev_i:Z_{A, g,m}\to Q$ and $\gamma\colon Z_{A,g,m}\to \ov{\mathcal M}_{g,k}$ of Theorem \ref{evaluation_thm} allow to pull back the differential forms on $Q$ and $\ov{\mathcal M}_{g,m}$ to sc-differential forms on the polyfold  $Z_{A,g,m}$.
The solution set $S( \ov{\partial}_J)\subset Z_{A,g,m}$ of the Fredholm section $\ov{\partial}_J$ of the strong polyfold bundle $p:W\to Z$ is the set 
$$S(\ov{\partial}_J)=\{ z\in Z_{A, g,m}\vert \,  \ov{\partial}_J (z)=0 \}$$
where $0$ is the zero section of the bundle $p$. We point out that the fibers of a strong polyfold bundle do not {possess  a linear structure. However,  the bundle possesses a preferred section $0$. }

In view of the perturbation theory in \cite{HWZ3.5} there exists a small generic $\ssc^+$-multisection $\lambda$ of the  bundle $p$ in the sense of Definition 3.41 in \cite{HWZ3.5} so that the pair $(\ov{\partial}_J,\lambda)$  {possesses the properties listed in  Definition 4.7 in \cite{HWZ3.5}}.  We define the solution set of the pair 
$(\ov{\partial}_J,\lambda)$ as the subset
$$S(\ov {\partial}_J, \lambda )=\{z\in Z_{A, g,m}\vert \, \lambda (\ov{\partial}_J(z) )>0\}$$
and the weight function $\vartheta: Z_{A, g,m}\to \Q^+$ by 
$$\vartheta (z)=\lambda (\ov {\partial}_J(z))\in \Q^+.$$ 
{Here $\Q^+$ stands for the set of non-negative rational numbers}. 
By the Theorems 4.13 and 4.14 in \cite{HWZ3.5}, the pair $(S(\ov {\partial}_J, \lambda ),\vartheta)$ is a compact, branched and oriented suborbifold of the polyfold $Z_{A,g,m}$ in the sense of Definition 3.22 in \cite{HWZ3.5}. Its dimension is equal to $\text{Ind}\ (\ov{\partial}_J)$. Therefore, we can apply the branched integration theory of \cite{HWZ7} and conclude from Theorem 4.23 in \cite {HWZ3.5} immediately the following result. 

\begin{theorem}\label{GW-thm}
Let $(Q,\omega)$ be a closed symplectic manifold of dimension $2n$. For a
given  homology class $A\in H_2(Q)$ and for given  integers  $g,m\geq 0$ satisfying  $2g+m\geq 3$,  
there exists  a multi-linear map
$$
\Psi^Q_{A,g,m}: H^\ast(Q;{\mathbb R})^{\otimes m}\rightarrow
H^\ast(\overline{\mathcal M}_{g,m};{\mathbb R})
$$
which,  on $H^*(Q; \R)^{\otimes m}$,   is super-symmetric with respect to the grading by even and odd forms. This map
is uniquely characterized by the following formula. For a given compatible almost complex structure $J$ on $Q$  and a given
small generic perturbation by an $\ssc^+$-multisection $\lambda$,  we have the representation 
\begin{eqnarray*}
&\langle\Psi^Q_{A,g,m}([\alpha_1]\otimes\ldots \otimes[\alpha_m]),[\tau]\rangle&\\
&=\int_{(S(\ov{\partial}_J, \lambda) ,\vartheta)}\gamma^\ast (PD(\tau ))\wedge \text{ev}_1^\ast(\alpha_1)\wedge\ldots \wedge
\mbox{ev}_k^\ast(\alpha_m)&
\end{eqnarray*}
in which  $\alpha_1,\ldots, \alpha_m\in H^*(Q)$,  $\tau \in H_*(\ov{\calm}_{g,m})$, $PD$ denotes the Poincar\'{e} dual, and $\gamma:Z_{g,m}\to \ov{\mathcal M}_{g,m}$ is the map in 
Theorem \ref{evaluation_thm}
\end{theorem}
The integration theory used is the ``branched integration'' from  \cite{HWZ7}. The integral over the empty set  is defined to be zero.

The a priori real number $\langle\Psi^Q_{A,g,m}([\alpha_1]\otimes\ldots \otimes[\alpha_m]),[\tau]\rangle$ is called a  {\bf Gromov-Witten invariant}. \index{GW--invariant} It is zero if  the Fredholm index does not  agree with the degree of the differential form being integrated. It follows from the construction that the Gromov-Witten invariants are invariants of the symplectic deformation type of the manifold $Q$. Indeed, a deformation produces  an arc of Cauchy-Riemann sections. These  are component-proper since the almost complex structures
occurring are tamed. {With additional efforts one can show that the GW-invariants  are rational numbers  if the (co)homology classes  are integer and the above integral has an interpretation as a rational count of solutions of some intersection problems. 
}\\

\noindent{\bf Acknowledgement} We thank the referee for useful comments.
We thank Urs Fuchs for pointing out a bug in the topology part,
which required a stronger analytical argument comparing two uniformizers.
The first author also would like to thank Nate Bottman and Jiayong Li and the more than 
30 participants of the MIT RTG-sponsored polyfold workshop at Pajaro Dunes 
at which the ideas of polyfold theory as well as this paper was discussed in detail.
This event took place, while the authors  were preparing the final version implementing the suggestions
of the referee,  and the questions and discussion with participants helped to improve
several parts of the presentation. We thank Stefan Suhr and Kai Zehmisch for their comments.
Finally we would like to thank Joel Fish and Katrin Wehrheim
for their interest in this project and their valuable  feedbacks.

Kris Wysocki and Eduard Zehnder would like  thank the  Institute for Advanced Study (IAS) in Princeton and the ForschungsInstitut f\"ur Mathematik (FIM) in Zurich for their hospitality.

%
%
%

\chapter{Recollections and Technical Results}\label{chapter_2_technical results}
A basic ingredient in the construction of the polyfold structure on the space of stable curves is the Deligne-Mumford theory. There is, however, a twist.  In the plumbing constructions of Riemann surfaces we have to apply  a recipe which is quite different from the one used in the DM-theory. The deeper reasons lie in  the analysis of function spaces. The DM-theory, as needed here, has been developed in detail in \cite{HWZ-DM}.  The other ingredients are  nonlinear analysis facts related to  our sc-version of nonlinear analysis. The results needed here have been proved in \cite{HWZ8.7}.
We start by recalling  the results used  in
the polyfold construction later on.

\section{Deligne--Mumford type Spaces}\label{dm-subsect}
In this section we present a Lie groupoid version of the
classical Deligne-Mumford theory.

The gluing construction at the nodes requires a conversion of the absolute value $|a|$ of a non-zero complex number $a$, the gluing parameter,  into a positive real number $R>0$,
which is essentially the modulus of the associated cylindrical neck. {For this purpose we introduce the notion of a gluing profile.
\begin{definition}\label{gluing_profile}
A {\bf gluing profile} is a smooth diffeomorphism 
$$\varphi:(0,1]\to [0, \infty ).
$$
\end{definition}
}

The conversion $|a|\mapsto R$ in the classical Deligne-Mumford case is defined  by the logarithmic gluing profile
\index{gluing profile!logarithmic }
$$
(0,1]\rightarrow [0,\infty), \quad r\mapsto  \varphi(r):=-\frac{1}{2\pi}\ln(r).
$$

In our study of stable curves we  need, in view of the  functional analytic requirements,  other gluing profiles. A very useful one is the exponential gluing profile
\index{gluing profile!exponential}
 $$
(0,1]\rightarrow [0,\infty),\quad  r\mapsto \varphi(r):=e^{\frac{1}{r}}-e.
$$
As it turns out, only in the case of the logarithmic
gluing profile, the Deligne-Mumford space of stable
noded and marked Riemann surfaces has a holomorphic
 orbifold structure. In the case of the exponential
 gluing profile we merely obtain a smooth naturally
 oriented orbifold structure.

We begin with the classical case.
We denote by ${\overline{\mathcal N}}$ the space of biholomorphic
equivalence classes of connected, stable, noded Riemann surfaces
with un-ordered marked points. More precisely
we consider equivalence classes $[\alpha]$ of
nodal Riemann surfaces 
$$\alpha=(S,j,M,D)$$
 in which $(S,j)$ is
a closed Riemann surface (which can  consist
of different connected components), $M$ is a finite subset of $S$ of  marked points
 which we assume to be ordered, un-ordered or partially ordered,
 depending on the situation.
 Isomorphisms between such Riemann
 surfaces have to preserve the marked points,
 and when they are ordered even their ordering.
 For the first part of the discussion we assume $M$ to be un-ordered. The set $D$ consists of  finitely many un-ordered pairs  $\{x,y\}$ of distinct points $x$ and $y$ in $S$ and we  
 require that  two such  pairs are either identical or disjoint, i.e.,  if $\{x,y\}\cap\{x',y'\}\neq \emptyset$,  then $\{x,y\}=\{x',y'\}$.  We call $D$  the set of nodal pairs and  write $|D|\subset S$ for the union of all sets $\{x,y\}$, 
 $$
 |D|=\bigcup_{\{x,y\}\in D} \{x,y\}. 
 $$
The sets $M$ and $|D|$ are  required to be disjoint.  
 We recall that  the nodal Riemann surface  $\alpha$ is said to be connected if the topological space,  obtained by identifying the two points of every  nodal pair is  a connected topological space. We call $\alpha$ stable if  every connected component $C$ of  the Riemann surface $S$  having genus $g(C)$ satisfies
$$
2\cdot g(C)+\sharp (C\cap(M\cup|D|))\geq 3,
$$
where $\sharp(\cdot)$ denotes the cardinality of a set. Two such tuples $\alpha$ and $\alpha'$  are called isomorphic  (or equivalent) provided there exists a biholomorphic map $\phi:(S,j)\rightarrow (S',j')$ satisfying $\phi(M)=M'$ and $\phi_{\ast}(D)=D'$, where $\phi_{\ast}(D)=\{\{\phi(x),\phi(y)\} | \ \{x,y\}\in D\}$. The group $G$ of the automorphisms of the  noded Riemann  surface $(S, j, M, D)$ is finite if and only if $\alpha$ is stable. 
The set $\overline{\mathcal N}$ consists of all equivalence classes of connected, stable, noded Riemann surfaces
with un-ordered marked points. This set has interesting subsets. If $g$ and $m$ are non-negative integers satisfying $2g+m\geq 3$,  we denote by $\overline{\mathcal N}_{g,m}$ the subset consisting of all equivalence classes $[\alpha]$ for which the representative  $\alpha$ has arithmetic genus $g$ and $m$ marked points.
\begin{theorem}
The space $\overline{\mathcal N}$ has a natural second countable
paracompact Hausdorff topology for which every connected component is compact.
Moreover,  if $2g+m\geq 3$, then  the subset $\overline{\mathcal N}_{g,m}$ is a compact subset which also is a connected component.
\end{theorem}
A proof of this classical result can be found in \cite{Hu}.
We shall call the underlying topology the DM-topology.  We shall next  describe a basis for this topology. 

We fix {one of the above}  gluing profiles  $\varphi$ and let  the noded Riemann surface $\alpha=(S,j,M,D)$ represent the equivalence  class $[\alpha]\in \overline{\mathcal N}$. 
{Then we take a small disk structure associated with $\alpha$ and defined as follows. 
\begin{definition}[{\bf Small disk structure}]\label{small_disk_structure}
A {\bf small disk structure}  associated with the stable noded Riemann surface $\alpha=(S, j, M, D)$, denoted by ${\bf D}$, assigns to every nodal point $x\in \abs{D}$ a disk like neighborhood $D_x\subset S$ centered at $x$ and having a smooth boundary. The disks are mutually disjoint and disjoint from the set $M$ of marked points (unless otherwise specified). Moreover, the disks are equipped with a choice of holomorphic polar coordinates as introduced below. 
In addition, the union 
$$
|{\bf D}|:= \bigcup_{x\in |D|} D_x.
$$
of the disks is invariant under the  group action of  the automorphism group $G$ of $\alpha$.\index{small disk structure}
\end{definition}
} 

 {There always exists  a small disk structure in any open neighborhood of the nodal points  as proved in   \cite{HWZ-DM}. } 

With every nodal pair $\{x, y\}\in D$ we associate a gluing parameter $a_{\{x, y\}}\in \C$\index{gluing! parameter}  which is a complex number varying in  the closed unit disk. By  $a=\{a_{\{x,y\}}\}_{\{x, y\}\in D}$ we denote the collection of all gluing parameters. Given a fixed  $a=\{a_{\{x,y\}}\}$, we denote by 
$D_a$ the  set of all pairs $\{x,y\}$ for which $a_{\{x,y\}}=0$.

We first consider a nodal pair $\{x, y\}\in D\setminus D_a$ for which, by definition, $a_{\{x, y\}}\neq 0$. In this case we shall connect the boundaries of the associated disks $D_x$ and $D_y$ by a finite cylinder 
$$Z^{\{x, y\}}_{a_{\{x, y\}}}$$
defined as follows. Near $x$ we choose the positive holomorphic polar coordinates\index{positive holomorphic polar coordinates}
$$h_x:[0,\infty)\times S^1\rightarrow D_x\setminus\{x\},\qquad 
h_x(s, t)=\ov{h}_x(e^{-2\pi (s+it)}),$$
and near $y$ we choose the 
negative holomorphic polar coordinates \index{negative  holomorphic polar coordinates}
$$h_y:(-\infty, 0]\times S^1\rightarrow D_y\setminus\{y\},\qquad 
h_y(s, t)=\ov{h}_y(e^{2\pi (s+it)})$$
where
\begin{align*}
&\ov{h}_x:\{w\in \C\vert \, \abs{w}\leq  1\}\to D_x\\
&\ov{h}_y:\{w\in \C\vert \, \abs{w}\leq  1\}\to D_y
\end{align*}
are biholomorphic mappings $(T\ov{h}_x\circ i=j\circ T\ov{h}_x$ and similarly for $\ov{h}_y$) satisfying 
$$\ov{h}_x(0)=x\quad \text{and}\quad \ov{h}_y(0)=y.$$
We remove the  points $z\in D_x$
which are  of the form $z=h_x(s,t)$ for $s>\varphi(|a_{\{x,y\}}|)$  and the points $z'\in D_y$ of the form  $z'=h_y(s',t')$ for  $s'<-\varphi(|a_{\{x,y\}}|)$. Now we identify the remaining annuli as follows. The points 
$$z=h_x(s,t)\quad \text{and}\quad z'=h_y(s',t')$$
are equivalent if 
$$s=s'+\varphi (\abs{a_{\{x,y\}}})\quad \text{and}\quad t=t'+\vartheta    \pmod 1$$
where the polar form of $a_{\{x,y\}}\in \C$ is given by 
$$a_{\{x,y\}}=|a_{\{x,y\}}|e^{-2\pi i\vartheta}.$$
We shall denote the equivalence classes by $[h_x(s, t)]$ or simply $[s,t]$. The equivalence classes define the glued finite cylinder $Z_{a_{\{x, y\}}}^{\{x, y\}}$ connecting $\partial D_x$ with $\partial D_y$. It possesses two  distinguished  global coordinate  systems, namely the positive coordinates
$$\{ (s, t)\vert\, 0\leq s\leq \varphi \bigl(\abs{a_{\{x, y\}}}\bigr) \}$$
and  the negative coordinates
$$\{ (s', t')\vert\,  -\varphi \bigl(\abs{a_{\{x, y\}}}\bigr) \leq s'\leq 0\}. $$
{The gluing construction near a nodal pair $\{x, y\}$ for which $a_{\{x, y\}}\neq 0$ is illustrated in Figure \ref{Fig2}.}

We next consider a nodal pair $\{x, y\}\in D_a$ which satisfies $a_{\{x, y\}}=0$. In this case we do not change the associated disks 
and define 
$$Z^{\{x, y\}}_0=D_x\cup D_y.$$

\begin{figure}[h!]
\psfrag {r}{\small{$R$}}
\psfrag {mr}{\small{$-R$}}
\psfrag {0}{\small{$0$}}
\psfrag {a0}{\small{$0$}}
\psfrag {rp}{\small{${\mathbb R}^+\times S^1$}}
\psfrag  {xy}{\small{$\{x,y\}$}}
\psfrag {hx}{\small{$h_x$}}
\psfrag {rp}{\small{${\mathbb R}^+\times S^1$}}
\psfrag  {rm}{\small{${\mathbb R}^-\times S^1$}}
\psfrag{Z}{\small{$Z_{a_{\{x,y\}}}^{\{x,y\}}$}}
\centering
\includegraphics[width=3.6in]{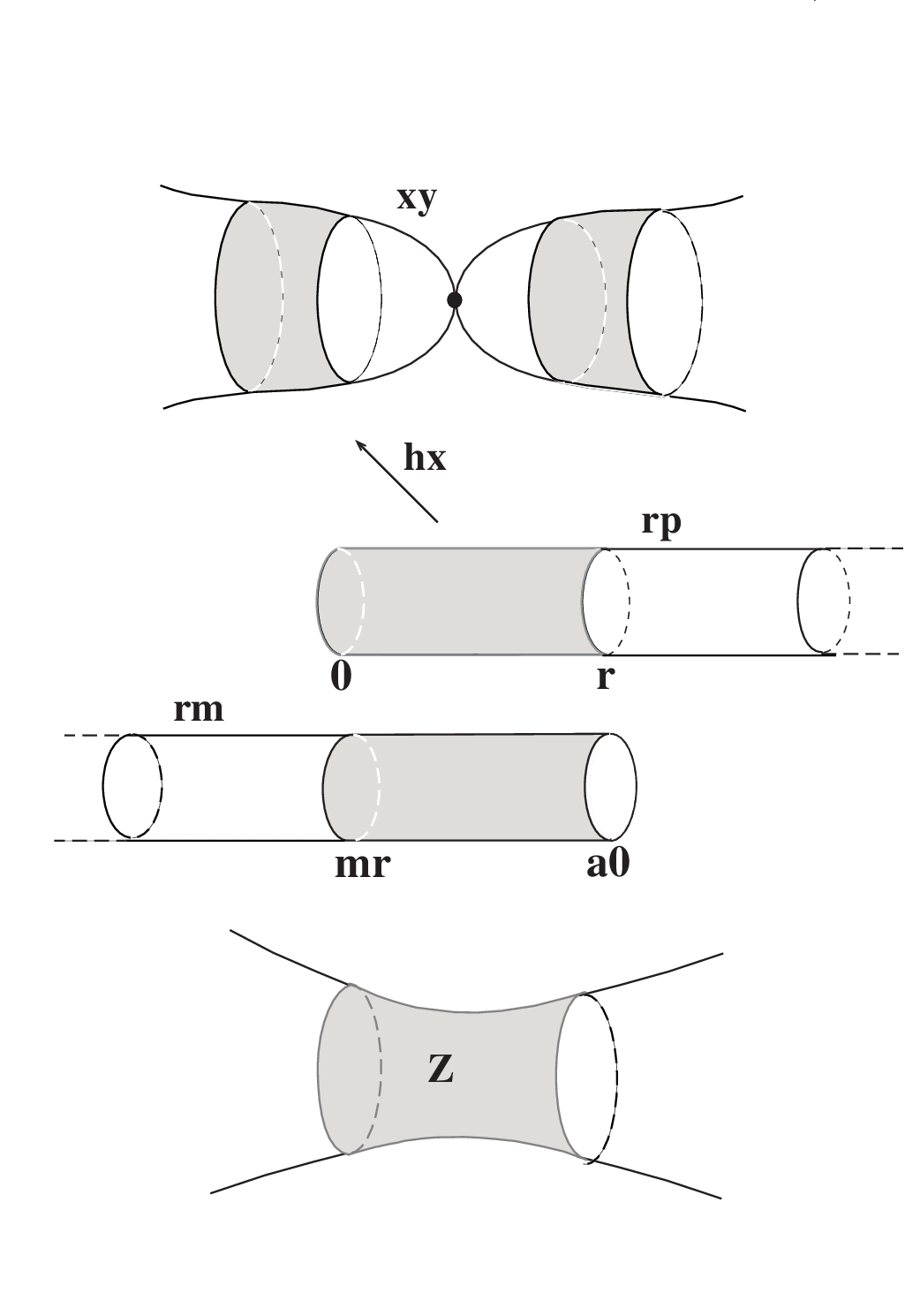}
\caption{Gluing}
\label{Fig2}
\end{figure}

We carry out the above gluing construction at every nodal pair $\{x, y\}\in D$ and obtain this way the new glued Riemann surface  $S_a$.  Because we did not change the complement of the disks $(D_x)$ we can identify the marked points $M$ on $S$ with points on $S_a$, denoted by $M_a$. {An almost}  complex structure $k$ on $S$ coinciding with $j$ on the set $\abs{{\bf D}}$ induces the { almost } complex structure $k_a$ on the glued surface $S_a$.

\begin{figure}[ht]
\psfrag {a}{$(S, k, M, D)$}
\psfrag {b}{$(S_a, k_a, M_a, D_a)$}
\psfrag {a1}{\small{$a_{\{x,y\}}=0$}}
\psfrag  {a2}{\small{$a_{\{x',y'\}}=0$}}
\psfrag {b1}{\small{$a_{\{x,y\}}\neq 0$}}
\psfrag  {b2}{\small{$a_{\{x',y'\}}=0$}}
\centering
\includegraphics[width=4.2in]{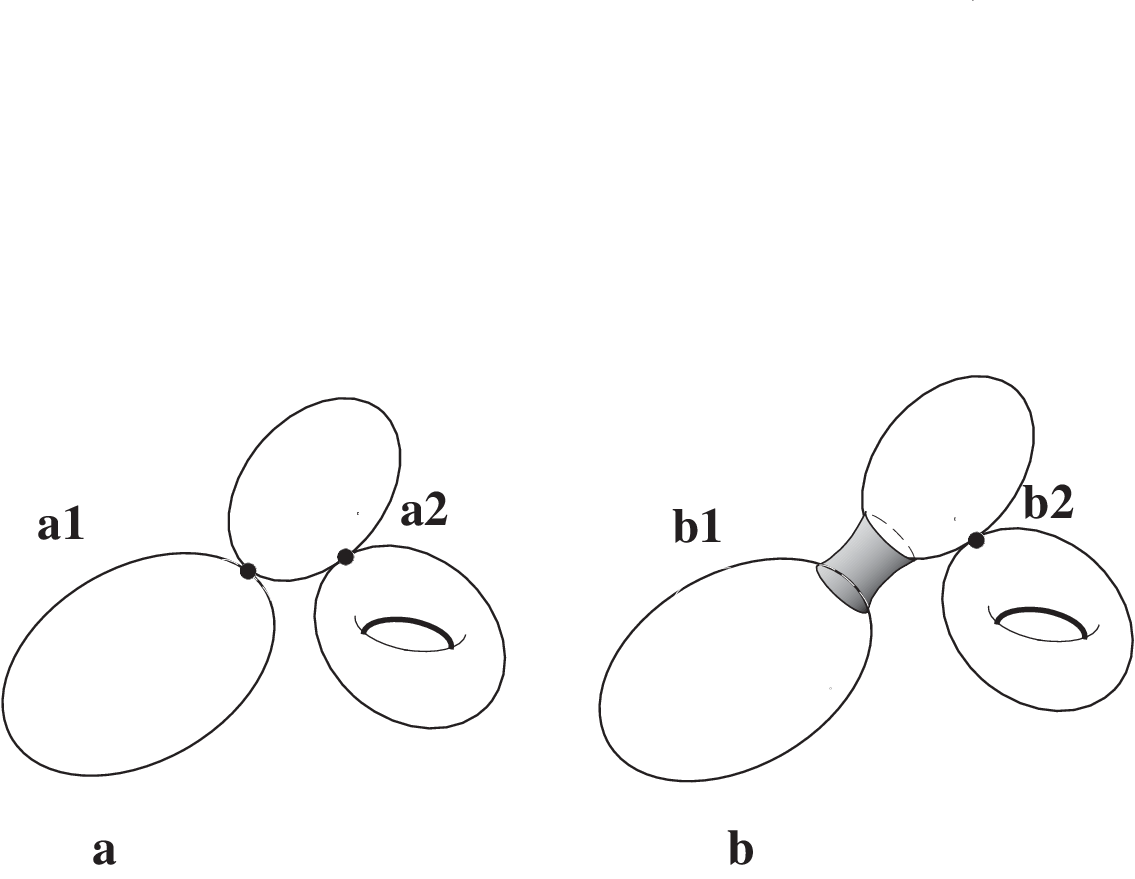}
\caption{Gluing}
\label{Fig3}
\end{figure}

Summarizing, given the stable noded Riemann surface $\alpha=(S, k, M, D)$,
 {any of the above two gluing profiles $\varphi$}, and  given the small disk structure $(D_x)_{x\in \abs{D}}$ and the gluing parameters $a$, the above gluing construction defines the new stable noded  Riemann  surface $\alpha_a$ defined by
$$\alpha_a=(S_a,k_a,M_a,D_a),$$
{and illustrated in Figure \ref{Fig3}.}
Starting from the noded Riemann surface $\alpha=(S,j,M,D)$, having fixed the disks $\{D_x\}_{x\in \abs{D}}$ of the small disk structure, we let $U$ be an open 
$C^\infty$-neighbor\-hood of $j$ consisting of complex structures on $S$ giving the same orientation. For every $\varepsilon\in (0,\frac{1}{2})$ we define the subset  $V(\alpha,U,{\bf D},\varepsilon)$ of $\ov{\mathcal N}$ by
$$
V(\alpha,U,{\bf D},\varepsilon)
=\{[S_a,k_a,M_a,D_a]\ |\ |a|<\varepsilon,\ k\in U,\ k=j\ \hbox{on}\ |{\bf D}|\}.
$$
\begin{proposition}
The collection of all sets $V(\alpha,U,{\bf D},\varepsilon)$\index{basis for DM-topology}
constitutes  a basis for the natural DM-topology on $\ov{\mathcal N}$. The topology does not depend on the choice of the gluing profile.
\end{proposition}
We refer to \cite{HWZ-DM} for a proof. In view of  the above construction we shall say that the noded Riemann surface $S_a$ is obtained from $S$  by plumbing or by gluing.

The classical Deligne-Mumford result in  \cite{DM} equips the topological space $\overline{\mathcal N}$ with a holomorphic orbifold structure. This is also proved from a more differential geometric perspective in \cite{RS}. Assume that $(S,j,M,D)$ represents
an element in $\overline{\mathcal N}$ and denote by $G$ its finite
automorphism group. By $\Gamma_0(\alpha)$ we denote  the vector space of
smooth sections of $TS\rightarrow S$ which vanish at the special points in
$|D|\cup M$. Moreover, we denote by $\Omega^{0,1}(\alpha)$ the space of
sections of $\hbox{Hom}_{\mathbb R}(TS,TS)\rightarrow S$ which are
complex antilinear. There is an associated linear Cauchy-Riemann\index{linear Cauchy-Riemann operator $\ov{\partial}$}
operator
$$
\bar{\partial}:\Gamma_0(\alpha)\rightarrow \Omega^{0,1}(\alpha)
$$
which, in holomorphic coordinates, is represented by
$$\ov{\partial }f(z)=\frac{\partial f}{\partial \ov{z}}(z)d\ov{z}=\frac{1}{2}\bigl(\partial_xf+i\partial_y f\bigr)(z)(dx-idy).$$
 {First order elliptic theory shows that $\ov{\partial}$  is a Fredholm operator. Its index can be computed with the help from the Riemann-Roch theorem and we refer to \cite{MS2}. The analysis can be based on  H\"{o}lder  or Sobolev spaces. Since the precise analytical set-ups are standard,  we will allow ourselves to be somewhat imprecise in our notation. We work, in particular, most of the time in the smooth category.  With smooth we mean $C^\infty$ smooth. However, everything can be made precise by taking the appropriate functional-analytic set-up.}
 
Due to the stability condition, the operator $\bar{\partial}$ is
injective. More precisely, as a consequence of the Riemann-Roch theorem the following holds. 
\begin{proposition}
Let $\alpha=(S, j, M, D)$ be a stable and connected nodal Riemann surface. Then the operator
$\bar{\partial}:\Gamma_0(\alpha)\rightarrow \Omega^{0,1}(\alpha)$ is
a complex linear injective Fredholm operator whose  index is equal to 
$$
\ind (\ov{\partial})=-(3g_a+\sharp M-\sharp D-3).\index{Fredholm index of $\ov{\partial}$}
$$
\end{proposition}
We define the linear quotient space  $H^1(\alpha)$ by
$$
H^1(\alpha)= \Omega^{0,1}(\alpha)/\text{im}(\ov{\partial}).
$$
As the notation already indicates it can be viewed as a cohomology
group. This space has complex dimension $3g_a+\sharp M-\sharp D-3$.
Here $\sharp D$ is the number of nodal pairs which is half the
number $\sharp |D|$ of nodal points.

The automorphism group $G$ {whose elements we also denote by $g\in G$}, acts on $H^1(\alpha)$ \index{$G$-action on $H^1(\alpha)$}in a natural way
via
$$
G\times H^1(\alpha)\rightarrow
H^1(\alpha),\quad (g, r)\mapsto  g\ast[r]:=[(Tg)r(Tg)^{-1}].
$$
We call it the natural representation of $G$. If $E$ is a complex vector space, then a representation of $G$ on $E$ is called natural
if  it is  equivalent to the natural representation of $G$. Given a smooth family
$v\mapsto  j(v)$ of complex structures on $S$ we can consider the
family
$$
v\mapsto  \alpha_v=(S,j(v),M,D)
$$
of nodal Riemann surfaces. 
If $j(0)=j$, the family is called a deformation of $\alpha$. \index{deformation of $\alpha$} Assume that $v$
belongs to an open neighborhood of $0$ of some finite-dimensional
complex vector space $E$. The differential $\delta
v\mapsto  Dj(v)\delta v$ induces a linear map
$$
E\rightarrow H^1(\alpha_v),\quad \delta v\mapsto [Dj(v)\delta v].
$$
This linear map,  denoted by $[Dj(v)]:E\rightarrow H^1(\alpha_v)$,  is
called the Kodaira differential. \index{Kodaira differential}The following definitions are very useful.

\begin{definition}\label{def_complex_effective_symmetric}
Let $\alpha=(S,j,M,D)$ be a stable nodal Riemann surface and let $E$ be a complex vector
space  carrying a natural representation of the automorphism group $G$ of $\alpha$.
Let $v\mapsto  j(v)$ be a smooth deformation of $j$ parameterized
by $v$ belonging to an open neighborhood $V$ of $0\in E$. Then the family
$v\mapsto  j(v)$ is called
\begin{itemize}
\item[$\bullet$] complex if $[Dj(v)]$ is complex linear for all $v\in V$.\index{deformation of $\alpha$! complex}\index{complex deformation}
\item[$\bullet$] effective if $[Dj(v)]$ is a real linear isomorphism at
every point $v\in V$. \index{deformation of $\alpha$! effective} \index{effective deformation}
\item[$\bullet$]  symmetric if $V$ is invariant under the $G$-action on $E$
and for every $g\in G$,  the diffeomorphism $g:S\rightarrow S$ is a
biholomorphic map
$$
g:\alpha_v\rightarrow \alpha_{g\ast v}.\index{deformation of $\alpha$! symmetric}\index{symmetric deformation}
$$
\end{itemize}
\end{definition}
{In the symmetric case we have
$$
[Dj(g\ast v)(g\ast\delta v)] =  [(Tg)(Dj(v)\delta v)(Tg)^{-1}].
$$
}
A {\bf good deformation} of $j$\index{good deformation} consists of a complex vector
space $E$ with a  natural representation of $G$, a $G$-invariant
open neighborhood $V$ of $0$ and a smooth family $v\mapsto  j(v)$
of complex structures on $S$ satisfying $j(0)=j$,  which  is
 effective and symmetric and such that,  in addition,  there exists an
open neighborhood $U$ of $|D|\cup M$ on which  $j(v)=j$ for all $v\in
V$. {A good deformation $v\mapsto j(v)$ is called a {\bf good complex deformation}  if the family is, in addition, a  complex family in the sense of the above Definition \ref{def_complex_effective_symmetric}.}

There always exists a good complex deformation according to the next proposition proved in 
 \cite{HWZ-DM}.
\begin{proposition}
{For every stable noded Riemann surface $\alpha=(S, j, M, D)$ there exists a good complex deformation $v\mapsto j(v)$ of $j$. 
In fact we can arrange with $j=j(0)$ that $Dj(v)(j\delta v)= j(v)\circ (Dj(v)\delta v)$ for all $v\in V$  and $\delta v\in E$.}
\end{proposition}

We have already introduced the notion of a
small disk structure. In our gluing construction later on it is convenient to deal with good families which are constant near the nodes. They are  guaranteed by the following result from  \cite{HWZ-DM}.
\begin{proposition}\label{prop2.6-n}
 Let  $\alpha=(S, j, M, D)$ be  a noded Riemann surface
and $v\mapsto  j(v)$ a good deformation of $j$ satisfying $j(v)=j$ on an open neighborhood
$U$ of $|D|$. Then there exists a small disk structure ${\bf D}$ so
that $|{\bf D}|$ is contained in $U$.
\end{proposition}

We now fix a stable noded Riemann surface $\alpha=(S, j, M, D)$ having the good deformation $v\mapsto j(v)$ parametrized by $v\in V\subset E$ and the small disk structure ${\bf D}$ satisfying Proposition \ref{prop2.6-n}.

For every nodal point $x$ we  fix
a biholomorphic map $\ov{h}_x:(D^1,0)\rightarrow (D_x,x)$ where 
$D^1=\{w\in \C\vert \, \abs{w}\leq 1\}$.  Such a
choice is unique up to rotation. For every nodal pair $\{x,y\}\in D$ there exists a uniquely determined complex anti-linear map
$\varphi_{x,y}:T_yS\rightarrow T_xS$ such that if $\varphi_{y,x}$
is its inverse,  the map  
$$
T\ov{h}_y(0)^{-1}\circ\varphi_{y,x}\circ T\ov{h}_x(0):{\mathbb
C}\rightarrow {\mathbb C}
$$
is the complex conjugation $z\rightarrow \overline{z}$. If  $\{x,y\}\in
D$ is a nodal pair, then for an automorphism $g\in G$ of $\alpha$,  the composition
$$
\varphi_{x,y}\circ Tg(y)^{-1}\circ\varphi_{g(y),g(x)}\circ
Tg(x):T_xS\rightarrow T_xS
$$
is  a multiplication by an element in the unit circle as proved in  \cite{HWZ-DM}.  {Geometrically, given a nodal pair $\{x, y\}$ and an automorphism $g\in G$, then $\sigma_{\{x,y\}}(g)$ is the relative twist of the gluing parameter induced from the action of $g$.}  One easily verifies that this  element in
$S^1$ is the same if we exchange in the previous expression the order
of $x$ and $y$. Let
us denote this element in $S^1$ by $\sigma_{\{x,y\}}(g)$. 
One also verifies immediately that
$$
\sigma_{\{x,y\}}(hg)=\sigma_{\{g(x),g(y)\}}(h)\cdot\sigma_{\{x,y\}}(g).
$$

Now we are
in the position to extend the natural $G$-action on $E$ as
follows. Let $N$ be the finite-dimensional complex vector space of maps
$D\rightarrow {\mathbb C}$ which associate to the nodal pair $\{x,y\}$
a complex number $a_{\{x,y\}}$. We  define the extended $G$-action
$$
\sigma:G\rightarrow GL(N)\times GL(E)
$$
by
$$
g\ast (a, e)=\sigma(g)(a,e)=(\sigma_N(g)a,g\ast e),
$$
where $\sigma_N(g)a=b$ is defined by
$$
b_{\{g(x),g(y)\}} = \sigma_{\{x,y\}}(g)\cdot a_{\{x,y\}}.\index{extended $G$-action}
$$
From the properties of $\sigma_{\{x,y\}}(g)$ one sees that the
latter defines indeed a representation of $G$ on $N\times E$. The element $a_{\{x,y\}}$ will
occur as a gluing parameter at the nodal pair $\{x,y\}$. With the data
 so far fixed,  we choose for every nodal pair $\{x,y\}$ an ordered pair,
say $(x,y)$, and take  for $x$ and $y$ the holomorphic polar coordinates
$$
h_x:[0,\infty)\times S^1\rightarrow D_x,\quad 
h_x(s,t)=\ov{h}_x(e^{-2\pi(s+it)})
$$
and
$$
h_y:(-\infty,0]\times S^1\rightarrow
D_y,\quad h_y(s',t')=\ov{h}_y(e^{2\pi(s'+it')})
$$
as introduced in Section \ref{dm-subsect}. We have chosen positive holomorphic polar coordinates around
$x$ and negative ones around $y$. In the following $(s,t)$ will
always denote positive polar coordinates and $(s',t')$ negative
ones. In order to carry out the gluing we need a gluing profile.
As already pointed out,  there will be two gluing profiles of interest for us. The
logarithmic gluing profile occurring in the classical holomorphic
Deligne-Mumford theory is defined by $\varphi(r)=-\frac{1}{2\pi}\ln(r)$
and the exponential gluing profile which will be used in  our
constructions is defined by $\varphi(r)=e^{\frac{1}{r}}-e$. Our approach in \cite{HWZ-DM} allows one to derive Deligne-Mumford type results for any gluing profile satisfying suitable derivative bounds. Alternatively,  assuming the DM-theory for the logarithmic gluing profile,  a calculus lemma allows one  to derive the necessary results for the exponential gluing profile as shown in \cite{HWZ-DM}.

For the moment, let us fix the logarithmic gluing profile,  which we denote by $\varphi$. Given $\alpha$, a small disk structure ${\bf D}$, and a good deformation $v\mapsto  j(v)$ of $j$, which coincides with $j$ on $|{\bf D}|$, we obtain as previously described the family of noded Riemann surfaces
$$
(a,v)\mapsto  \alpha_{(a,v)}:=(S_a,j(a,v),M_a,D_a)\, \text{where  $(a,v)\in (B_{\frac{1}{2}})^{\# D}\times V$}.
$$
Here $j(a,v)=j(v)_a$ is the complex structure induced from $j(v)$ on the glued surface $S_a$.
The automorphisms $g\in G$ of $(S, j, M, D)$  induce the canonical isomorphisms 
$$g_a:\alpha_{(a, v)}\to \alpha_{g\ast (a, v)}.$$
{At this point there should be no confusion with the arithmetic genus which is also denoted by $g_a$.}

If the nodal pair $\{x, y\}\in D$ satisfies $a_{\{x, y\}}\neq 0$, we have introduced in Section \ref{dm-subsect}  
the finite cylinder $Z^{\{x,y\}}_{a_{\{x,y\}}}$  consisting of points $z\in S_a$
satisfying  $z=[h_x(s,t)]$ where $(s,t)\in [0,R]\times S^1$ and $R=\varphi(|a_{\{x,y\}}|)$.

If the nodal pair $\{x, y\}$ satisfies $a_{\{x, y\}}=0$, we have introduced $Z^{\{x,y\}}_{a_{\{x,y\}}}=Z^{\{x,y\}}_0=D_x\cup D_y$. We recall that  our gluing construction at the nodal pair  does not change the 
pair $(D_x, x)\cup (D_y, y)$ of disks if $a_{\{x, y\}}=0$ and if  $a_{\{x, y\}}\neq 0$ replaces the pair of disks by the finite cylinder $Z^{\{x,y\}}_{a_{\{x,y\}}}$  connecting the  boundaries   of the disks. 
This cylinder has two sets of
distinguished coordinates namely $[s,t]$ by extending the positive
holomorphic coordinates coming from $x$ or $[s',t']'$ coming from the
negative holomorphic coordinates from $y$. 

If $a_{\{x, y\}}\neq 0$ and $h>0$,  we  define the 
sub-cylinder $Z^{\{x,y\}}_{a_{\{x,y\}}}(-h)\subset Z^{\{x,y\}}_{a_{\{x,y\}}}$ by 
$$
Z^{\{x,y\}}_{a_{\{x,y\}}}(-h)=\{[s,t]\in Z^{\{x,y\}}_{a_{\{x,y\}}}\
|\ s\in [h,R-h]\}
$$
and introduce the subsets $Z_a$  and  $Z_a(-h)$ of $S_a$  defined by 
$$
Z_a = \bigcup_{\{x,y\}\in D} Z^{\{x,y\}}_{a_{\{x,y\}}}\quad  \text{and}\quad 
Z_a(-h) = \bigcup_{\{x,y\}\in D} Z^{\{x,y\}}_{a_{\{x,y\}}}(-h).
$$

{If $a_{\{x, y\}}=0$ and $h>0$,  we  define  
\begin{equation*}
Z_0^{\{x, y\}}(-h)=D_x(-h)\cup D_y(-h),
\end{equation*}}
with the sub-disks 
$D_x(-h)=\{z\in D_x\vert \, \text{$z=x$  or  $z=h_x(s, t)$ for $s\geq h$}\}$ and 
$D_y(-h)=\{z'\in D_y\vert \, \text{$z'=y$  or  $z'=h_y(s', t')$ for $s'\leq -h$}\}.$

\begin{definition}[{\bf Core of $S_a$}]\label{def_core}
{Let $\alpha=(S, j, M, D)$ be a noded Riemann surface with a small disk  structure $\{D_x\}_{x\in \abs{D}}$ and let $\alpha_a=(S_a, j_a, M_a, D_a)$ be  the Riemann surface obtained from $\alpha$ by means of the aforementioned gluing construction using the logarithmic or exponential gluing profile associated to $a\in (B_{\frac{1}{2}})^{\# D}$. Then we define  the {\bf core}  of  $S_a$ \index{core of $S_a$} to be $S_a\setminus \text{int}\ Z_a$, where $Z_a\subset S_a$  is the collection of cylinders defined above.} 

{Hence, in the special case $a\equiv 0$,  the core of $S$ is the complement of the union of the disks of the small disk structure, namely, $S\setminus \bigcup_{x\in \abs{D}}\text{int}\ D_x$.}
\end{definition}
The cores of $S$ and $S_a$ are naturally identified.
The following observation will be useful in the study of
sc-smoothness.

\begin{figure}[ht]
\psfrag {z1}{\small{$Z_a(-h)$}}
\psfrag  {z2}{\small{$Z_a$}}
\psfrag {a1}{\small{$a_{\{x,y\}}=0$}}
\psfrag  {a2}{\small{$a_{\{x',y'\}}=0$}}
\psfrag {b1}{\small{$a_{\{x,y\}}\neq 0$}}
\psfrag  {b2}{\small{$a_{\{x',y'\}}=0$}}
\centering
\includegraphics[width=5in]{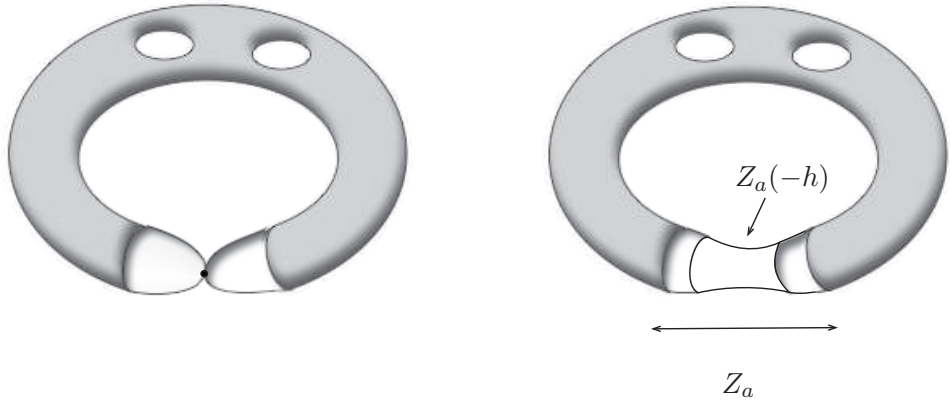}
\caption{The left-hand side  figure  shows a connected nodal Riemann surface of arithmetic genus $3$. The right-hand side  figure shows a glued version having genus $3$ with the distinguished cylinder $Z_a$ and the sub-cylinder $Z_a(-h)$.}
\label{Fig4}
\end{figure}

\begin{remark}\label{google}A  point $z\in S_a$ can be viewed as a point in $S$ as follows. If
$z$ belongs to the core it is obvious. If $z$ is outside of the core
we have two cases. Either it belongs to $D_x$ or to $D_y$, where
$\{x,y\}$ is an unglued nodal pair. In that case the identification
is  also clear. Otherwise it must belong to a ``neck''  $Z^{\{x,y\}}_{a_{\{x,y\}}}$ and hence can be
written in two ways as
$$
z=h_x(s,t)\quad \text{and} \quad z=h_y(s',t'),
$$
where $(s, t)\in [0, R]\times S^1$ and  $(s', t')\in [-R,0]\times S^1$ satisfying $s=s'+R$ and {$t=t'+\vartheta$} where $a_{\{x, y\}}=\abs{a_{\{x, y\}}}\cdot e^{-2\pi i\vartheta}$ and $R=\varphi (\abs{a_{\{x, y\}}})$ is the gluing length. In this case the point $z$ can be identified with two points in $S$, one in the disk $D_x$ and the other in the disk $D_y$. We denote by $\ov{z}$ the lift of a point $z\in S_a$ to $S$ if the lift is unique and by $\ov{z}_x$ and $\ov{z}_y$ if the lift is not unique and  the 
 subscript indicates
into which disk the point is  lifted, see Figure \ref{Fig5}.
\end{remark}

\begin{figure}[ht]
\psfrag {dx}{\small{$D_x$}}
\psfrag {dy}{\small{$D_y$}}
\psfrag  {x}{\small{$x$}}
\psfrag {y}{\small{$y$}}
\psfrag {hx}{\small{$h_x$}}
\psfrag {hy}{\small{$h_y$}}
\psfrag {zx}{\small{$Z_x$}}
\psfrag  {zy}{\small{$Z_y$}}
\psfrag {z}{\small{$Z$}}
\psfrag  {c}{\small{$z$}}
\psfrag  {a}{\small{$\ov{z}_x$}}
\psfrag {b}{\small{$\ov{z}_y$}}
\psfrag  {st}{\small{$(s, t)$}}
\psfrag  {stp}{\small{$(s',t')$}}
\psfrag {sa}{\small{$S_a$}}
\psfrag  {s}{\small{$S$}}
\centering
\includegraphics[width=4.3in]{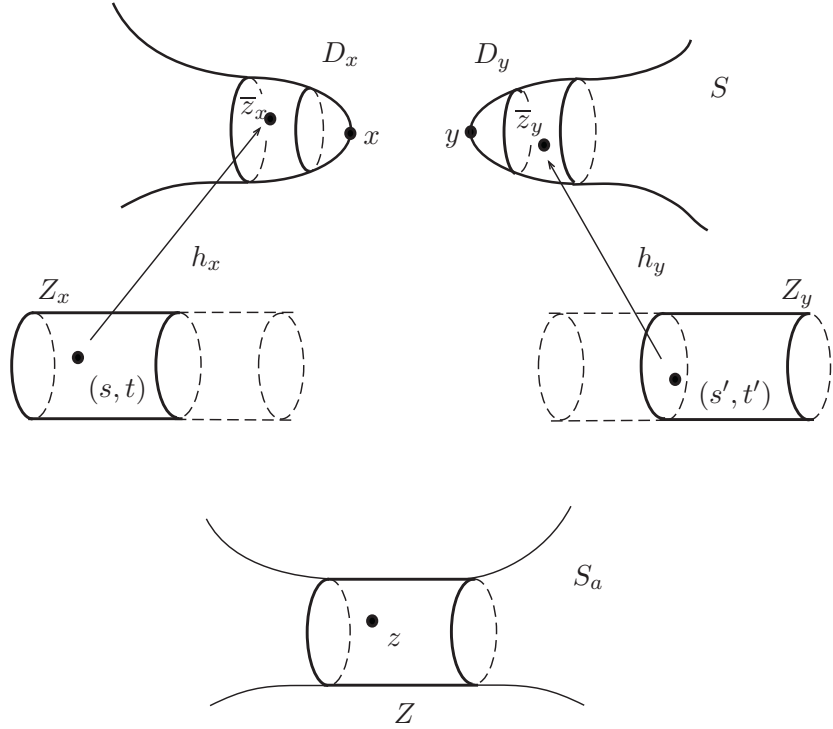}
\caption{Lifts of  a point $z\in S_a$.}
\label{Fig5}
\end{figure}

\begin{definition}\label{def_finite_distance_core}
Let $a^l=(a^l_{\{x, y\}})_{\{x, y\}\in D}$ be a sequence of gluing parameters and let 
$(z_l)\subset  S_{a^l}$ be a sequence of points.  We say that  the sequence
$(z_l)$ stays in a {\bf finite distance to the core} \index{finite distance to the core}  if there exists a real number 
$h>0$ satisfying  $z_l\not \in Z_{a^l}(-h)$ for all large $l$.
\end{definition}

Assume that $(a^0,v^0)$ is given and consider the nodal pairs
$\{x,y\}\in D$ for which $a_{\{x,y\}}^0\neq 0$. If one  varies
only these non-zero gluing parameters, then  one quite easily
constructs a new smooth complex structure $j^\ast(b,v)$ on the noded  Riemann surface $S_{a^0}$ so that the noded surfaces  $(S_{a^0},j^\ast(b,v),M_{a^0},D_{a^0})$  and
$\alpha_{(a^0+b,v^0)}$ are isomorphic by a map which {restricted to the respective cores  agrees with the natural identification}.  Here $b$  is small and has a nonzero component  at the nodal pair $\{x, y\}\in D$ only  if
$a_{\{x,y\}}^0\neq 0$. In  \cite{HWZ-DM} this 
construction is called the ``freezing of gluing parameters''.  
We can  take the Kodaira differential at $(b,v)=(0,v^0)$ with respect
to $(b,v)$ and obtain a linear map
$$
PDj(a^0,v^0):N(a^0)\times E\rightarrow H^1(\alpha_{(a^0,v^0)})
$$
called a  partial Kodaira differential. \index{partial Kodaira differential} The space  $N(a^0)$ is the
subspace of $N$ consisting of all $\{a_{\{x,y\}}\}$ satisfying 
$a_{\{x,y\}}=0$ if  $a^0_{\{x,y\}}=0$.

\begin{definition}[{\bf Good uniformizing family}]\label{citiview}\index{good uniformizing family}
We consider the stable Riemann surface $\alpha=(S,j,M,D)$  whose finite automorphism group is denoted by $G$. Let $v\mapsto j(v)$ be a good complex deformation of $j$ and consider the family 
$(a,v)\mapsto  \alpha_{(a,v)}$ of glued noded Riemann surfaces, whose construction uses the  {\bf logarithmic  or the exponential gluing profile}. The parameters 
$(a,v)$ vary in the open neighborhood $O \subset N\times E$ of the origin which is invariant under the natural representation of $G$. The family $(a, v)\mapsto \alpha_{(a, v)}$ is called a uniformizing family with domain $O$, if the following conditions are satisfied. 
\begin{itemize}
\item[$\bullet$] The set ${\mathcal U}=\{[\alpha_{(a,v)}]\,  \vert \ (a,v)\in O\}$ of equivalence classes of stable Riemann surfaces 
is open in the space $\overline{\mathcal N}$.
\item[$\bullet$]  The map $p:O\rightarrow {\mathcal U}$ defined by
$$
p(a,v)=[\alpha_{(a,v)}]
$$
induces a homeomorphism from the orbit space $G\backslash O$  onto ${\mathcal U}$.
 \item[$\bullet$]  If 
$\phi:\alpha_{(a,v)}\rightarrow \alpha_{(a',v')}$ is an isomorphism of noded Riemann surfaces where $(a, v)\in O$ and $(a',v')\in O$, then there exists an automorphism    $g\in G$   satisfying  $(a',v')=g\ast(a,v)$. Moreover,  $\phi=g_a$ is the induced canonical isomorphism.
\item[$\bullet$]  The partial Kodaira differentials 
$$
PDj(a,v):N(a)\times E\rightarrow H^1(\alpha_{(a,v)})
$$
at the point  $(a,v)\in O$ are all  linear isomorphisms (resp. complex  linear if the logarithmic gluing profile is used).
\end{itemize}
\end{definition}

The following key result  is a reformulation of parts of the classical Deligne-Mumford theory in the case of the logarithmic gluing profile. It holds true also  in the case of the exponential gluing profile and we refer to  \cite{HWZ-DM} for the proof. We also discuss some of the issues later on.
\begin{theorem}\label{existence-x}
For each stable noded Riemann surface  $\alpha=(S,j,M,D)$ there exists a good uniformizing family
$$
(a,v)\mapsto  \alpha_{(a,v)}\quad \text{on the domain $ (a,v)\in O.$}
$$
\end{theorem}

 We now consider  the uniformizing family
$(a,v)\mapsto  \alpha_{(a,v)}$  on the  domain $ (a,v)\in O$ and,  abbreviating 
$o=(a,v)$,  simply write 
$$
o\rightarrow \alpha_o\quad \text{for $o\in O$}
$$
for the uniformizing family. In order to simplify the notation we shall abbreviate  by $\cg$ or
$\cg_\alpha$ the graph of the family $o\rightarrow \alpha_o$,
$$
\cg=\{(o,\alpha_{o})\ |\ o\in O\}
$$
which we equip with the structure of a complex manifold by requiring
that the map $\cg\rightarrow O$ defined by $(o,\alpha_o)\mapsto  o$ is a biholomorphic diffeomorphism.  We shall also refer to $\cg$ as to  a uniformizing
family. The natural map $p:\cg\rightarrow \overline{\mathcal
N}$ is defined by $p(o,\alpha_o)=[\alpha_o]$.

If  $\cg$ and $\cg'$ are given, we consider for the two uniformizing families $(o,\alpha_o)\in
\cg$ and $(o',\alpha_{o'}')\in \cg'$ the triple 
$$
((o,\alpha_{o}),\phi,(o',\alpha_{o'}') )\equiv (o,\alpha_o,\phi,o',\alpha_{o'}')
$$
in which  $\phi:\alpha_{o}\rightarrow \alpha_{o'}'$ is an
isomorphism between the nodal Riemann surfaces. We
denote the  set  of all such triples by ${\bf
M}(\cg,\cg')$.  There are  two natural maps $s,t$ defined on ${\bf
M}(\cg,\cg')$. Namely,  the source map
$$
s:{\bf
M}(\cg,\cg')\rightarrow\cg, \quad \text{defined by  $s(o,\alpha_o,\phi,o',\alpha_{o'}')=
(o,\alpha_{o})$,}
$$
and the target map 
$$
t:{\bf
M}(\cg,\cg')\rightarrow\cg',\quad \text{defined by $t(o,\alpha_o,\phi,o',\alpha_{o'}')=
(o',\alpha_{o'}')$}.
$$
\begin{figure}[htbp]
\psfrag {s}{\small{$\cg'$}}
\psfrag {s1}{\small{$\cg$}}
\psfrag {a}{\small{$(o,\alpha_o)$}}
\psfrag  {a}{\small{$(o,\alpha_o)$}}
\psfrag {b}{\small{$(o',\alpha_{o'})$}}
\psfrag {c}{\small{morphism} {$((o,\alpha_o), \psi, (o',\alpha_{o'}))$}}
\centering
\includegraphics[width=3.5in]{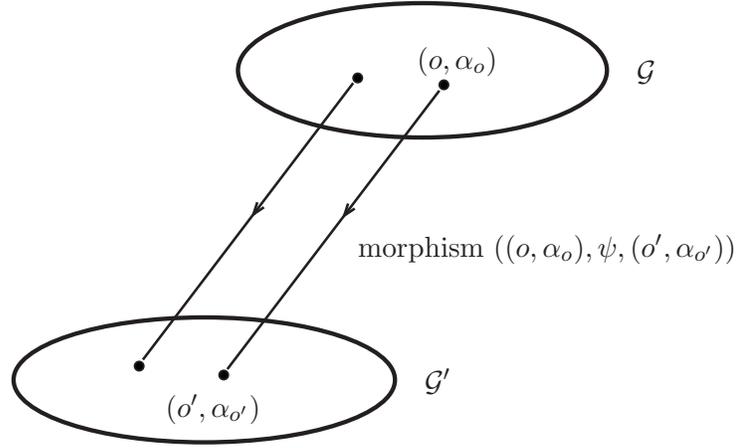}
\caption{One can view the points in $\cg$ and $\cg'$ as objects of a category and the points in ${\bf M}(\cg,\cg')$ as morphisms
between these objects.}
\label{Fig6}
\end{figure}

The following result, {proved in \cite{HWZ-DM},}  is a
reformulation  of parts of the DM-theory.
\begin{theorem}
{We have fixed the logarithmic gluing profile}.
 The set ${\bf
M}(\cg,\cg')$ has a natural structure as a complex manifold so that
the source and the target maps are locally biholomorphic
maps and the map
$${\bf
M}(\cg,\cg')\xrightarrow{s\times t}\cg\times\cg'
$$
is proper. For these complex manifold structures
the inversion map
$$
{\bf M}(\cg,\cg')\rightarrow {\bf M}(\cg',\cg),
$$
defined by 
$$
(o,\alpha_o,\phi,o',\alpha_{o'}')\mapsto
(o',\alpha_{o'}',\phi^{-1},o,\alpha_{o}),
$$
and the $1$-map
$
\cg\rightarrow {\bf M}(\cg,\cg)$, defined by 
$$(o,\alpha_{o})\mapsto
(o,\alpha_{o},\id,o,\alpha_{o}),
$$
are holomorphic maps. Moreover, given $\cg,\cg',\cg''$  the fibered product
$$
{\bf M}(\cg',\cg''){{_s}\times_t}{\bf M}(\cg,\cg')
$$
has, in view  of the first part of the theorem,the structure of
a complex manifold, and the multiplication map, 
$$
{\bf M}(\cg',\cg''){{_s}\times_t}{\bf M}(\cg,\cg')\rightarrow {\bf
M}(\cg,\cg''),
$$
 defined by 
\begin{equation*}
((o',\alpha_{o'}',\phi,o'',\alpha_{o''}''),
(o,\alpha_{o},\psi,o',\alpha_{o'}'))\mapsto 
(o,\alpha_{o},\phi\circ\psi,o'',\alpha_{o''}''), \end{equation*}
is a holomorphic  map.
\end{theorem}

 The above theorem is the building block for a
complex orbifold structure on  the space $\overline{\mathcal N}$ of equivalence classes of connected and stable noded Riemann surfaces as explained next.

We begin with the complex structure. If $g,m$ are nonnegative integers satisfying 
 $2g+m\geq 3$, we  denote by $\overline{\mathcal N}_{g,m}$ the
connected component of  $\overline{\mathcal N}$ consisting of elements of noded Riemann surfaces having arithmetic genus equal to $g$ and
$m$ marked points. This component is a compact topological space. 

\begin{figure}[h!]
\psfrag {ax}{$\abs{X}$}
\psfrag {x}{\small{$X=\Sigma_1\cup \Sigma_2\cup \Sigma_3\cup \Sigma_4$}}
\psfrag{s1}{\small{$\Sigma_1$}}
\psfrag {s2}{\small{$\Sigma_2$}}
\psfrag {s3}{\small{$\Sigma_3$}}
\psfrag   {s4}{\small{$\Sigma_4$}}
\centering
\includegraphics[width=4in]{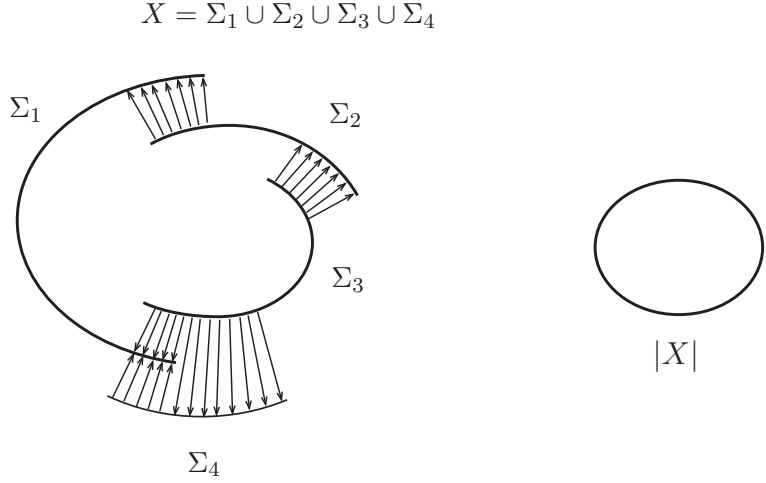}
\caption{The object set is $X$. Only some of the arrows are given in the figure. This figure shows how the circle $S^1$ can be obtained as the orbit space of an \'etale proper Lie groupoid. These groupoids play the same  role an atlas plays in the manifold theory. Modeling the notion of an equivalence of smooth atlases in the present context gives equivalent orbifold structures.}
\label{Fig7}
\end{figure}

We thus
can find finitely many sets $\cg_1,\ldots ,\cg_l$ so that the union of their
images under the natural maps $p:\cg \to \ov{\mathcal N}$,  defined by $p(o,\alpha_o)=[\alpha_o]$,    is equal to $\overline{\mathcal N}_{g,m}$. We denote by $X_{g,m}$ the 
disjoint  union of the sets $\cg_i$, where $1\leq i\leq l$, and by  ${\bf X}_{g,m}$ the disjoint union of the sets  ${\bf M}(\cg_i,\cg_j)$, where $1\leq i,j\leq l$. We  view $X_{g,m}$  as the
object set  in a small category having the set  ${\bf X}_{g,m}$ as its morphism set.
The object and the morphism sets are equipped with the  structure of complex
manifolds, so that in view of the previous theorem all category
operations are holomorphic. The source and target maps $s,t$
are surjective locally biholomorphic maps and the map
$$
{\bf X}_{g,m}\xrightarrow{s\times t} X_{g,m}\times X_{g,m}
$$
is proper. Therefore the category  $X_{g,m}$ is
a proper \'etale Lie groupoid as described  in  \cite{Mj,MM}.
 Finally,  we have a natural homeomorphism
$$
\beta:|X_{g,m}|\rightarrow \overline{\mathcal N}_{g,m},\quad 
|x|\mapsto  p(x)
$$
between the orbit space $|X_{g,m}|$ of $X_{g,m}$,  obtained by identifying two objects if they are connected by a morphism, and the set $\ov{\mathcal N}_{g,m}$. 
Hence the pair $(X_{g,m},\beta)$ defines a holomorphic orbifold structure on the connected component
$\overline{\mathcal N}_{g,m}$. Different choices of finite collections $(\cg_i)$ as introduced  above
define equivalent orbifold structures.  For a proof we refer to \cite{Mj} or \cite{MM}.
 An orbifold is a paracompact Hausdorff space equipped with an equivalence class of orbifold structures.
\begin{remark}\label{LIFT}
Let us also note that by a slight modification one obtains a natural
orbifold structure on the space   $\overline{\mathcal M}$ of equivalence classes of 
stable noded Riemann surfaces with ordered marked points. It can be
obtained from the above construction by admitting  only biholomorphic
isomorphisms preserving the order of the marked points. The associated space   $\overline{\mathcal M}_{g,m}$ is the classical Deligne-Mumford space of equivalence classes of stable Riemann surfaces of arithmetic genus $g$ and $m$ ordered marked points.
\end{remark}

In the previous discussion we have described a version  of the classical
DM-theory {assuming that the gluing profile is the logarithmic gluing profile}. However,  if we study  maps on the Riemann surfaces we are forced 
to choose a different  gluing profile. This is a priori not obvious but follows from rather subtle analytical considerations. The
modification is  carried out  in  \cite{HWZ-DM} and is as follows. 
 For every
family $(a,v)\mapsto  \alpha_{(a,v)}$ we replace the logarithmic
gluing profile by the exponential gluing profile. Therefore, we
consider the family
$$
(a,v)\mapsto  (\wt{a},v)\mapsto  \alpha_{(\wt{a},v)},
$$
of stable noded Riemann surfaces parametrized by the gluing parameters  $\wt{a}=(\wt{a}_1,\ldots ,\wt{a}_k)$ for the exponential gluing profile which  are related to the gluing parameters 
$a=(a_1,\ldots, a_k)$ belonging to the logarithmic gluing profile by the following formula.  If $a_i=0$ then $\wt{a}_i=0$. If
$a_i\neq 0$, we set $a_i=|a_i|\cdot e^{-2\pi i\vartheta}$ and define
${\wt{a}}_i$ by ${\wt{a}}_i=|{\wt{a}}_i|\cdot e^{-2\pi
i\vartheta}$ where 
$$
\abs{a_i} = e^{-2\pi(e^{\frac{1}{\abs{\wt{a}_i}}}-e)}.
$$
If we define the set $X_{g,m}$ of objects 
as before we obtain a naturally oriented smooth orbifold structure
on $\overline{\mathcal N}_{g,m}$. This is the smooth orbifold structure
we shall need for the  further constructions.

One could also start the whole construction with a good but not necessarily complex family $v\mapsto j(v)$.
Then in the corresponding construction of the manifolds ${\bf M}(\Xi,\Xi')$  the source and target maps are smooth, and so are the multiplication, inversion and $1$-maps. The main result here is the following theorem.
\begin{theorem}\label{theorem2.15}
Carrying out the previous gluing constructions  with the   {\bf exponential gluing profile} replacing the logarithmic gluing profile, one obtains  naturally oriented smooth orbifold structures  on the sets $\overline{\mathcal N}$ and $\ov{\mathcal M}_{g,m}$. We denote the spaces equipped with these orbifold  structures  by $\ov{\mathcal N}^{\textrm{exp}}$ and  
$\ov{\mathcal M}^{\textrm{exp}}_{g,m}$.
\end{theorem}
 In \cite{HWZ-DM} this result is proved directly by using an implicit function theorem.
However, it  follows alternatively from the classical (holomorphic)
  DM-theory as described above by means of a calculus lemma, also shown in \cite{HWZ-DM}.

Let $\alpha=(S,j,M,D)$  be a connected  noded  Riemann surface with (ordered) marked points. We do not require $\alpha$ to
be stable, but assume that $2g+m\geq 3$, where $g$ is the arithmetic
genus of $\alpha$ and $m$ the number of marked points on $S$. 
{Let $G$ be a finite subgroup of the automorphism group of $\alpha$. In   case $\alpha$ is unstable this latter group will be infinite.  Now we take a collection of compact small disks around the nodal points and require that they have smooth boundaries and their
union is invariant under $G$. The construction is the same as the construction of  a small disk structure. Let $j(w)$ be a
deformation of $j$ satisfying  $j(w)=j$ on these disks. Here $w$ varies in a finite dimensional complex vector space $E$.  Let $N$ be the  finite dimensional complex vector space of the gluing parameters 
$a_{\{x,y\}}$, $\{x, y\}\in D$. Using the exponential gluing profile we obtain a  family
$$
(b,w)\mapsto \beta_{(b,w)}=(S_b,j(b,w),M_b,D_b), \qquad (b,w)\in O,
$$
where $O\subset N\times E$ is an open neighborhood of the origin. Now we apply the forgetful map $\gamma$ to obtain a map
$$
O\rightarrow \overline{\mathcal M}_{g,m}, \quad (b,w)\mapsto
[(\beta_{(b,w)})_{\textrm{stable}}]
$$
into the classical Deligne-Mumford space $\overline{\mathcal M}_{g,m}$ 
space  constructed by means of  the logarithmic gluing profile.  We already explained in the first section how the underlying stable 
Riemann surface is constructed.}

\begin{lemma}\label{thm2.15-n}
The above  map
$$
O\rightarrow \overline{{\mathcal M}}_{g,m}, \quad (b,w)\mapsto 
[{(\beta_{(b,w)})}_{\textrm{stable}}]
$$
into the classical Deligne-Mumford space equipped with the smooth structure related to the holomorphic structure, is a smooth map between orbifolds.
\end{lemma}

 Note that we suppress here  the obvious `overhead'
for $\overline{\mathcal M}_{g,m}$  described  by the sets $\Xi$.
The map in {Lemma} \ref{thm2.15-n} is,  of course, represented near a point in $O$ by a map into some  set $\Xi_i$.
\begin{proof}
We start with the above family 
$$
(b,w)\mapsto  \beta_{(b,w)}:=(S_b,j(b,w),M_b,D_b),\quad (b, w)\in O
$$ 
of noded Riemann surfaces constructed by means of   the exponential gluing profile and a smooth deformation $w\mapsto  j(w)$ which is independent of $w$ near the nodal points. Let us add to the unstable domain components of $(S,j,M,D)$ a few points outside of the small disk structure in order to make them stable. Call this additional set of marked points $M^\ast$.
We may assume that $M^\ast$ is ordered so that the set $M^{\ast\ast}=M\cup M^\ast$ of marked points is ordered.
{We now consider the family 
$$
(b,w)\mapsto \beta^{\ast\ast}_{(b,w)}:=(S_b,j(b,w),M_b^{\ast\ast},D_b)
$$
which is smooth 
into the space $\overline{\mathcal M}_{g,m^{\ast\ast}}^{\textrm{exp}}$ equipped with the smooth structure coming from the exponential gluing profile
as shown in \cite{HWZ-DM}}.  
Here $m^{\ast\ast}=\sharp M^{\ast\ast}$ is the number of marked points.
On the other hand, the identity map 
$$\overline{\mathcal M}_{g,m^{\ast\ast}}^{\textrm{exp}}\xrightarrow{\text{id}} \overline{\mathcal M}_{g,m^{\ast\ast}}$$
 is smooth which is again  proved  in \cite{HWZ-DM}.  The inverse, in fact, is not.
 
 Now we can apply  the following well-known result in the classical DM-theory.
\begin{proposition}\label{classical_prop}
The map which forgets the $l$-th marked point and throws away possibly arising unstable domain components
$$
\text{\em forget}_l:\overline{\mathcal M}_{g,m^{\ast\ast}}\rightarrow
\overline{\mathcal M}_{g,m^{\ast\ast}-1}
$$
is  holomorphic.
 \end{proposition}
Continuing with the proof of Lemma \ref{thm2.15-n},  we set $m=\sharp M$ and $m^\ast=\sharp M^\ast$,  so that  $m^{\ast\ast}=m+m^\ast$. Then the map 
$
O\rightarrow \overline{\mathcal M}_{g,m}$,  defined by $(b,w)\mapsto  [{(\beta_{(b,w)})}_{\textrm{stable}}]$,
is the following composition of maps, 
\begin{gather*}
O\to \ov{\mathcal M}_{g,m^{\ast\ast}}^{\textrm{exp}}\xrightarrow{\text{id}} \ov{\mathcal M}_{g,m^{\ast\ast}}\to \ov{\mathcal M}_{g,m^{\ast\ast}-1}\to \ldots \to \ov{\mathcal M}_{g,m}\\
 (b,w)\mapsto  \text{forget}_{m+1}\circ \text{forget}_{m+2}\circ \ldots \circ  \text{forget}_{m+m^{\ast}}([\beta^{\ast\ast}_{(b,w)}]).
\end{gather*}
{In view of Proposition \ref{classical_prop}, 
this map is  a composition of smooth maps and hence the map $O\rightarrow \overline{\mathcal M}_{g,m}$ is  smooth,  completing the proof of Lemma \ref{thm2.15-n}.}
\end{proof}

Finally we  recall  the so called ``universal property'' of a good uniformizing family $(a, v)\mapsto 
\alpha_{(a, v)}$ of stable noded Riemann surfaces. The property claims for every family $(b, w)\mapsto \beta_{(b, w)}$ of 
 stable noded Riemann surfaces, containing for some parameter value $(b_0, w_0)$ a 
 surface $\beta_{(b_0, w_0)}$ isomorphic to  a surface $\alpha_{(a_0, v_0)}$ belonging to a  uniformizing family, that there exists a well-defined map $(b, w)\mapsto \mu (b, w)=(a(b, w), v(b, w))$ defined near $(b_0, w_0)$ and satisfying $\mu (b_0, w_0)=(a_0, v_0)$ so that all the surfaces $\beta_{(b, w)}$ are isomorphic to the  surfaces $\alpha_{\mu (b, w)}$. In order to make this statement precise we choose a  stable noded and connected Riemann surface $\alpha=(S,j,M,D)$ 
with unordered marked points and let 
$$
(a,v)\mapsto  \alpha_{(a,v)}=(S_a, j(a, v), M_a, D_a)$$
be a good uniformizing family with domain $(a, v)\in O$ as defined in Definition \ref{citiview}.\index{universal property}

Let $(S', j', M', D')$  be  another stable connected noded  Riemann surface with  unordered marked points, and let  ${\bf D}'$  be a small disk structure ${\bf D}'$ on $S'$. This time we allow the marked points in $M'$ to belong to $\abs{{\bf D}'}$. Of course, $M'$ is still assumed to be disjoint from the nodal points $\abs{D'}$. 
In order to define a deformation $M'(\sigma)$ of $M'$ we choose around every point $z\in M'$ an open neighborhood $V(z)$ whose closure does not intersect $\abs{D'}$.  {Moreover, these sets $V(z)$ are chosen to be mutually disjoint. 
If $M'=\{z_1,\ldots, z_{m'}\}$, we introduce the product space 
$$V(M'):=V(z_1)\times \cdots V(z_{m'}).$$
A point $\sigma=(\sigma_1,\ldots ,\sigma_{m'})\in V(M')$ where $\sigma_j\in V(z_j)$, defines the set $M(\sigma')$ of marked points on $S'$, in the following called a 
{\bf deformation of the marked points $M'$}.} \index{deformation of marked points} 
If $\sigma$ changes smoothly,  the set of marked points $M'(\sigma)$ changes, by definition, smoothly. We now take a $C^{\infty}$ neighborhood $U$ of the complex structure $j'$ and denote by $U_{{\bf D}'}$ the collection of complex structures $k$ in $U$ which coincide with $j'$ on the discs $\abs{{\bf D}'}$ of the small disk structure.
With the triple $(b, k, \sigma)\in B_{1/2} \times U_{{\bf D}'}\times V(M')$ we associate  the noded Riemann surface 
$\beta_{(b, k, \sigma)}=(S'_b, k_b, M' (\sigma)_b, D'_b)$ and study  
the  surfaces for  $\abs{b}$  small, $\sigma$  close to $\sigma_0$ satisfying  $M(\sigma_0)=M'$ and $k$  close to $j'$. If we take a finite-dimensional family 
$$w\mapsto k(w)$$
of complex structures $k$ which depend  smoothly  on a parameter $w$ varying  in a finite dimensional Euclidean space, then we call  the map 
$$(b, w, \sigma)\mapsto  \beta_{(b, k(w), \sigma)}=(S'_b, k(w)_b, M'(\sigma)_b, D'_b)$$
a  smooth family.  The following result is proved in  \cite{HWZ-DM}.

\begin{theorem}[{\bf Universal property}] \label{smoothfamily}
We fix the good uniformizing family $(a, v)\mapsto \alpha_{(a, v)}$.\\
(1) Consider the family 
$(b, k, \sigma)\mapsto  \beta_{(b, k, \sigma)}$ of noded Riemann surfaces as introduced above
($(b,k,\sigma)\in B_{\frac{1}{2}}\times U_{{\bf D}'}\times V(M')$) and assume that there exists an 
isomorphism
$$
\psi:\beta_{(0, j', \sigma_0)}=(S',j',M'(\sigma_0),D')\rightarrow (S_{a_0},j(a_0, v_0),M_{a_0}, D_{a_0})=\alpha_{(a_0,v_0)}
$$
between the two noded Riemann surfaces for some parameter $(a_0, v_0)$. 
Then there exists a uniquely determined continuous germ 
$$
(b,k,\sigma)\mapsto  \mu (b, k, \sigma)=(a(b,k,\sigma),v(b,k,\sigma))
$$
near $(b, k, \sigma)=(0, j', \sigma_0)$ satisfying 
$\mu  (0,j',\sigma_0)=(a_0,v_0)$ and an associated core-continuous germ
$(b, k, \sigma)\to \phi (b, k, \sigma)$ of isomorphisms 
$$\phi(b, k, \sigma):\beta_{(b, k, \sigma)}\to \alpha_{\mu (b, k, \sigma)}$$
between the noded Riemann surfaces satisfying $\phi (0, j', \sigma_0)=\psi$.\\
(2) Moreover, if $w\mapsto k(w)$ is a smooth family of complex structures on $S'$ satisfying $k(0)=j'$, then the family $(b, w, \sigma)\mapsto \phi (b, k(w), \sigma)$ of isomorphisms
$$\phi(b, k(w), \sigma):\beta_{(b, k(w), \sigma)}\to \alpha_{\mu(b, k(w), \sigma)}$$
is core-smooth.

\end{theorem}

\begin{figure}[h!]
\psfrag {z}{${z}$}
\psfrag {zp}{ ${\phi_{(b_0, w_0, \sigma_0)}(z)}$ }
\psfrag {p}{${\phi_{(b, w, \sigma)}}$}
\centering
\includegraphics[width=3.9in]{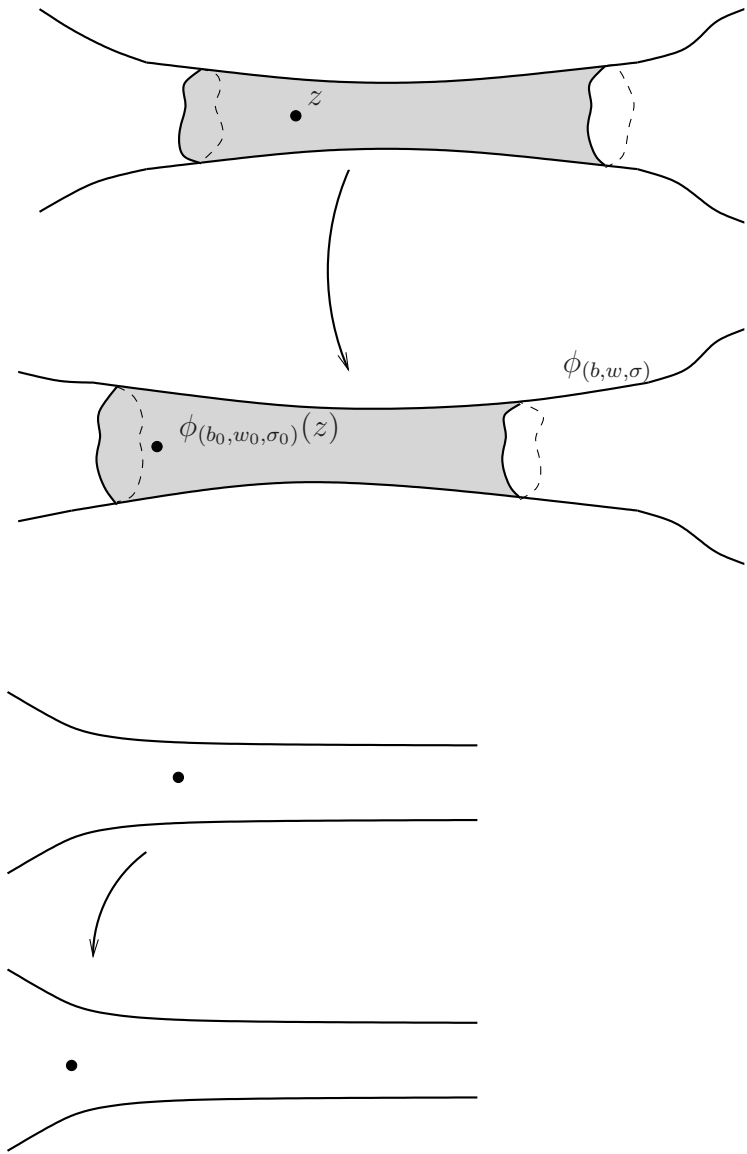}
\caption{Since the domain and target are varying one needs the notion of a family of maps being core smooth.}
\label{Fig8}
\end{figure}

The notions  of core-continuous and core-smooth, illustrated in Figure \ref{Fig8}, are defined as follows. If $(b, w, \sigma)\to \phi (b, k(w), \sigma)$ is a family of isomorphisms 
$$
\phi (b, k(w),\sigma):\beta_{(b,k(w),\sigma)}\rightarrow \alpha_{\mu (b, k(w), \sigma)}
$$
between the noded Riemann surfaces, we fix $(b_0, w_0, \sigma_0)$ in the parameter space. If the point $z$ in $S'_{b_0}$ is not a nodal point, then its image $\zeta$ under 
$\phi (b_0, k(w_0),\sigma_0)$ in $S_{a(b_0, k(w_0),\sigma_0)}$ is also not a nodal point of the target surface $\alpha_{\mu (b_0, k(w_0),\sigma_0)}.$ If $z$ is in the core, a neighborhood of $z$ can canonically be identified with a neighborhood of $z$ viewed as a point in $S'$. If $z$ does not belong to a core, we find a nodal pair $\{x', y'\}\in D'$ and can identify $z$ with a point in $D_{x'}\setminus \{x'\}$ or with a point in $D_{y'}\setminus \{y'\}$ where $D_{x'}$ and $D_{y'}$ are discs belonging to the chosen small disk structure of the noded surface $(S', j', M', D')$. The same alternatives hold for the image $\zeta$ of $z$ under the isomorphism $\phi (b_0, k(w_0),\sigma_0)$. Via these identifications, the family $\phi (b, k(w),\sigma)$ of isomorphisms gives rise to a  family of diffeomorphisms defined on a neighborhood of $z$ in $S'$ into a neighborhood of $\zeta$ in $S$. Being   core-continuous respectively core-smooth requires that all these germs of families  of isomorphisms are continuous respectively  smooth families of  local diffeomorphisms in the familiar sense,  between the fixed Riemann surfaces $S'$ and $S$.\index{core-continuous }\index{core-smooth}
\section{Sc-smoothness, Sc-splicings, and Polyfolds}\label{sc-smoothness}
The theory of polyfolds developed in \cite{HWZ2}-\cite{HWZ7} generalizes the theory of Banach manifolds. The theory is based on a new concept  of differentiability which allows, in particular, to view analytically intricate limiting phenomena like bubbling off or stretching the neck, as smooth phenomena. In this subsection we recall for the convenience of the reader some of the fundamental concepts  of this new theory. We start recalling the underlying smooth structure of Banach spaces, called an sc-structure. 

An {\bf sc-structure} \index{sc-structure}  for the  Banach space $E$  consists of a nested   sequence of Banach spaces 
$$E=:E_0\supset E_1\supset \ldots \supset E_{\infty}:=\bigcap_{i\geq 0}E_i$$
having the properties that the inclusion operators  $E_{i+1}\to E_i$ are compact operators and  that the vector space $E_{\infty}$ is dense in every Banach space $E_i$. 

A point belonging to $E_{\infty}$ is called a  {\bf smooth point}.\index{smooth point}  In the following a Banach space equipped with an sc-structure is called an sc-Banach space. \index{sc-Banach space}
A subset $U\subset E$ inherits the induced  sc-structure defined by the sequence $U_i=U\cap E_i$  equipped with the topology of $E_i$ for all 
$i\geq 0$. Given such an  induced sc-structure on $U$, the set $U_i$ has the induced sc-structure defined by the 
nested sequence  $(U_{i})_{j}:=U_{i+j}$ for all $j\geq 0$, and we write $U^i$ for the set $U_i$ equipped with this sc-structure. If $E$ and $F$ are sc-Banach spaces and $U\subset E$ and $V\subset F$ are open subsets, then $U\oplus V$ carries the sc-structure $(U\oplus V)_i=U_i\oplus V_i$.

 {A subspace $F$ of an sc-Banach space $E$ is  called an {\bf sc-subspace}  if $F$  is closed and the sequence $(F_i)_{i\geq 0}$, where  $F_i: = F \cap E_i $, defines an sc-structure on $F$. An sc-subspace $F$ of an sc-Banach
$E$  has an {\bf sc-complement}   if there exists an algebraic complement $G$  of $F$  equipped with an sc-structure $(G_i)_{i\geq 0}$ so that on every level $i$  we have a topological direct sum
$E_i = F_i \oplus G_i$.}

In the classical theory the tangent space at a point $x\in U\subset E$ has a natural identification with the  Banach space $E$. In the sc-theory, points in $U\setminus U_1$ do not have a tangent space for the given sc-structure and  the tangent space $TU$ of the open set $U\subset E$ is defined as 
$$TU=U^1\oplus E.$$

A linear operator $T:E\to E$ between two sc-Banach spaces is called an {\bf sc-operator}\index{sc-operator}  if  it maps $E_m$ into $F_m$ for every $m$ such that $T:E_m\to F_m$ is a bounded linear operator. An sc-isomorphism $T:E\to F$ is a bijective sc-operator whose inverse is also an sc-operator. In order to construct spaces with  boundary with corners, the concept of a partial quadrant $C$ in the sc-Banach space $E$ is useful.\index{partial quadrant} It is a closed convex subset having the property that there exists an sc-Banach space $W$ and an  linear sc-isomorphism $T:E\to \R^n\oplus W$ mapping $C$ onto  $[0, \infty )^n\oplus W$.

The tangent space of a relatively open subset $U\subset C\subset E$ of a partial quadrant $C$ in the sc-Banach space is defined as $TU=U^1\oplus E$. It carries the sc-structure $(TU)_i=U_{i+1}\oplus E_i$.

The notion of continuity of a map generalizes naturally to $\ssc^0$-maps as follows. 
A map $f:U\rightarrow U'$   between two  relatively open  subsets of the  partial quadrants in sc-Banach spaces is an $\ssc^0$-map, if   $f(U_i)\subset U'_i$  for all $i\geq 0$ and if the induced maps  $f:U_i\rightarrow U'_i$ are  continuous. \index{$\ssc^0$-map} Next we recall the generalization of a smooth map between sc-Banach spaces, called sc-smoothness.

\begin{definition}\label{sscc1} 
Let $U$ and $U'$ be  relatively 
open subsets  of partial quadrants $C$ and $C'$ in the sc-Banach spaces $E$ and $E'$, respectively. 
An  $\ssc^0$-map $f:U\rightarrow U'$ is called an 
$\ssc^1$-map,    or of class $\ssc^1$,   if the following conditions are satisfied.

\begin{itemize}\label{sc-1}
\item[(1)] For every $x\in U_1$ there exists a bounded 
linear  map $Df(x)\in  {\mathcal L}(E_0, E'_0)$ satisfying for $h\in
E_1$, with $x+h\in U_1$,
$$\frac{1}{\norm{h}_1}\norm{f(x+h)-f(x)-Df(x)h}_0\to 0\quad
\text{as\ $\norm{h}_1\to 0$.}\mbox{}\\[4pt]$$
\item[(2)]  The tangent map  $Tf:TU\to TU'$,
defined by
$$Tf(x, h)=(f(x), Df(x)h),
$$
is an $\ssc^0$-map between the tangent spaces. \index{$\ssc^1$-map}
\end{itemize}
\end{definition}

If $Tf:TU\to TU'$ is of class $\ssc^1$, then $f:U\to
E'$ is called of  class $\ssc^2$. Proceeding inductively, the map $f:U\to E'$ is called of class 
$\ssc^k$ if the
$\ssc^0$-map $T^{k-1}f:T^{k-1}U\to T^{k-1}E'$ is of class $\ssc^1$.
Its tangent map $T(T^{k-1}f)$ is then denoted by $T^kf$. It is an
$\ssc^0$-map $T^kU\to T^kE'$.
 A map  which is of class $\ssc^k$ for every $k$ is called  {\bf sc-smooth}  or of  class $\ssc^{\infty}$. \index{sc-smooth! map}
   
 The cornerstone  of the sc-calculus,  the chain rule, holds true and 
 one concludes by induction that the composition of $\ssc^{\infty}$-maps is again of class $\ssc^{\infty}$.

A fundamental  role in the definition  of a M-polyfold is  played by the sc-smooth retractions introduced next. 
\begin{definition}\label{retract}
Let $U\subset C\subset E$ be a relatively open subset of  the partial quadrant $C$ in the sc-Banach space $E$. An $\ssc^{\infty}$-map $r:U\to U$ is called an {\bf sc-smooth retraction} 
or 
an $\ssc^\infty$-retraction  if $r$ satisfies $r\circ r=r$. \index{sc-smooth!retraction}  The image $O=r(U)$ is called 
an  {\bf $\ssc^{\mathbf \infty}$-retract}. \index{sc-smooth!retract} 
\end{definition}

In view of  the chain rule, the tangent map $Tr:TU\to TU$  of an sc-smooth retraction $r:U\to U$ satisfies $Tr\circ Tr=Tr$, and hence  is again an sc-smooth retraction. 

Given a partial quadrant  $C$ in an sc-Banach space $E$,  we define   the degeneration index
$$
d_C:C\to  \N_0
$$
as follows. We choose a linear $\ssc$-isomorphism $T:E\to  \R^k\oplus W$ satisfying $T(C)=[0,\infty)^k\oplus W$. Hence for $x\in C$, we have  
$$T(x)=(r_1, \ldots, r_k, w)$$
where   $(r_1, \ldots, r_k)\in [0,\infty )^k$ and $w\in W$. Then we  define  the integer $d_C(x)$ by
$$
d_C(x)=\sharp\{i\in \{1,\ldots, k\}\vert \, r_i=0\}.$$
It is not difficult to see that this definition is independent of  the choice of an sc-linear isomorphism $T$.

The following concept will replace open subsets of Banach spaces in the classical definition of a local chart of a   
Banach manifold. {In \cite{HWZ10} we spend considerable efforts  to find (among other things) the best notions
for future developments. The local models for M-polyfolds with possible boundaries are sc-smooth retracts $(O,C,E)$, where $E$ is an sc-Banach space, $C\subset E$ a partial quadrant and $O\subset C$ an sc-smooth retract relative to $C$. As it turns out, using this definition,  a M-polyfold (with nonempty boundary) has a boundary with little structure.
In our applications, for example SFT, the boundary has a lot of structure and one can take local models which reflect these. This leads to the definition
of tame sc-smooth retracts. In the current paper, where the main illustration is Gromov-Witten theory, the full strength of the theory in \cite{HWZ10} is not needed
and we shall now introduce for the purpose of this paper the notion of a local M-polyfold model. The reader will immediately verify that this is a special case
of a tame sc-smooth retract.}

\begin{definition}\label{locmodel}
A local M-polyfold model is a triple $(O,C,E)$ in which  $E$ is an sc-Banach space, $C$ is  a partial quadrant of $E$, and $O$ is a  subset of $C$  having the following properties:
\begin{itemize}
\item[(1)] There is an sc-smooth retraction $r:U\to  U$ defined on a relative open subset $U$ of $C$ so that 
$O=r(U).$
\item[(2)] For every  smooth point $x\in O_\infty$,  the kernel of the map  $(\id-Dr(x))$ possesses  an sc-complement which is  contained in $C$.
\item[(3)] For every $x\in O$,  there exists a sequence of smooth points $(x_k)\subset O_\infty$ converging to $x$ in $O$ and satisfying $d_C(x_k)=d_C(x)$.\index{local M-polyfold model}
\end{itemize}
\end{definition}
The choice of $r$ in the above definition is irrelevant as long as it is an sc-smooth retraction onto $O$ defined on a relatively open subset $U$ of $C$. \\

{\bf Note that from now on the local models for our M-polyfolds will be the sc-smooth retractions of the kind just described. Local bundles $K\rightarrow O$
introduced later will have base spaces $O$ using the same type of retracts.}
\begin{remark} {Let us comment on the properties (2) and (3) in Definition \ref{locmodel}.
If $(O,C,E)$ and $(O',C',E')$ are local M-polyfold models and if $f$ is a germ of sc-diffeomorphism 
$f:{\mathcal O}(O,x)\rightarrow {\mathcal O}(O',x')$, at the smooth point  $x$, then Property (2) implies that
$d_C(x)=d_{C'}(x')$. In fact,  one can conclude this  for all elements $x$ on level at least $1$. However, it does not follow
if $x$ is only on  level $0$. Adding now Property (3)  we can show that the conclusion also holds for $x$ on level $0$. 
It follows that $d_C\vert O$ is a local diffeomorphism invariant, which,  describes
numerically the corner  structure of $O$ at the point $x$. Definition \ref{locmodel} is  a useful definition 
for the notion of a M-polyfold with boundary with corners. 
In the current paper we can restrict ourselves to the case without boundary. We point out that  for SFT
we shall need a set-up with boundaries with corners. The above definition is sufficient for SFT, though the more general notion
of a tame retract introduced in \cite{HWZ10}, which looks more natural, works equally well.}
\end{remark}

We would like to point out that in contrast to the smooth retracts in Banach manifolds, which are submanifolds, the sc-smooth retracts can be very wild sets having locally varying dimensions. {We refer  to the paper \cite{HWZ8.7} for illustrations. }

Note that an $\ssc^\infty$-retract $O\subset C\subset E$ has a tangent space $TO\subset TC\subset TE$, defined by
$TO=Tr(TU)$, where we can take any $\ssc^\infty$-retraction $r:U\rightarrow U$, $U\subset C\subset E$, which has $O$ as its image.
Consequently,  we can define the tangent $T(O,C,E)$ of a local M-polyfold model $(O,C,E)$ as 
$$
T(O,C,E):=(TO,TC,TE).\index{tangent of a local M-polyfold model}
$$

The constructions in this present paper, unlike the second part, only involves polyfolds without boundaries. Hence
the relevant local M-polyfold models are of the form $(O,E):=(O,E,E)$, where $O$  is an $\ssc^\infty$-retract, i.e. the image
of an $\ssc$-smooth map $r:U\rightarrow U$ defined on an open subset $U$ of $E$ and satisfying $r\circ r=r$.

In our applications of the polyfold theory later on,  the retractions are derived   from sc-smooth splicings, which are the distinguished sc-smooth retracts defined as follows.  

We consider a relatively open subset $V$ of a partial quadrant $C$ in the sc-Banach space $W$ and let $E$ be another sc-Banach space. Then we consider a family $\pi_v:E\to E$ of linear bounded projections parametrized by $v\in V$ such that the map 
$$\pi:V\oplus E\to E$$ 
defined by $\pi (v, e)=\pi_v (e)$ is sc-smooth. We point out that it is not required that the map $v\mapsto \pi_v$ is continuous in the operator topology. The triple  ${\mathcal S}=(\pi, E, V)$ is called an {\bf  sc-smooth splicing}. \index{sc-smooth! splicing} The  associated {\bf splicing core}  $K=K^{\mathcal S}$ is the set 
$$K=\{ (v, e)\in V\oplus E\vert \, \pi_v (e)=e\}. \index{splicing core} 
$$
Clearly, $V\oplus E$ is a relatively open subset of the partial quadrant $C\oplus E$ in the sc-Banach space $W\oplus E$ and the sc-smooth map 
$$r:V\oplus E\to V\oplus E$$ 
defined by $ r (v, e)=(v, \pi_v (e))$ 
satisfies $r\circ r=r$ and hence is an sc-smooth retraction onto  the retract $r(V\oplus E)=K.$

So far we have defined the sc-smoothness of maps between open subsets of sc-Banach spaces and use it now to define the sc-smoothness  of mappings between local models as follows.
\begin{definition}\label{map}
A map $f:(O, C, E)\to (O', C', E')$ between two  local M-polyfold models is of class $\ssc^1$ if the composition  $f\circ r:U\to E'$ is of class $\ssc^1$ where $U\subset C\subset E$ is a relatively open subset of the partial quadrant $C$ in the sc-Banach space $E$ and where $r:U\to U$ is an sc-smooth retraction onto $r(U)=O$. 
\end{definition}

{The tangent (or linearization)  of the map $f:O\to O'$ at the point $o\in O_1=r(U_1)$ on  level $1$ is defined as 
$$
Tf(o):T_oO\rightarrow T_{f(o)}O'
$$
by $T(f\circ r)|T_oO$.}

The chain rule  generalizes to $\ssc^1$-maps between sc-smooth local models.

Having the local models available we now generalize the concept of a manifold and introduce the notion of an 
{\bf M-polyfold}.

Let $X$ be a 
{paracompact} Hausdorff topological space.  
An  M-polyfold chart  for $X$ is a triple $(U,\varphi, (O, C, E))$, in which $U$ is an open subset of $X$, $(O, C, E)$ is a local  M-polyfold model and $\varphi:U\to O$ is a homeomorphism. Two such  charts $(U,\varphi, (O, C, E))$ and $(U',\varphi', (O', C', E'))$  are called  sc-smoothly compatible  if the transition maps  
\begin{align*}
\varphi'\circ \varphi^{-1}&:\varphi (U\cap U')\to  \varphi' (U\cap U')\\
\varphi \circ (\varphi')^{-1}&:\varphi' (U\cap U')\to \varphi (U\cap U')
\end{align*}
are sc-smooth. \index{M-polyfold! chart} Note that the sets $\varphi (U\cap U')$ and $\varphi'(U\cap U')$ are $\ssc^{\infty}$- retracts for sc-smooth retractions defined on relatively open subsets of partial quadrants $C$ and $C'$ in sc-Banach spaces $E$ and $E'$, respectively.  An sc-smooth atlas for $X$ consists of a family of sc-smoothly compatible charts  so that their domains cover $X$. 

\begin{definition}\label{DEF38}
Let $X$ be a  paracompact  Hausdorff topological space.   
A maximal atlas of sc-smoothly compatible M-polyfold charts is called an {\bf  M-polyfold structure} on $X$.\index{M-polyfold! structure}
\end{definition} 
{
\begin{remark} In a previous paper we have defined a M-polyfold by requiring the topology to be second countable Hausdorff.
The M-polyfold structure of the space then makes the space completely regular. By Urysohn's Theorem
a second countable, completely regular Hausdorff space is metrizable. We have realized later on that everything works under the assumption that  the topology is  paracompact and Hausdorff. As far as the Fredholm theory is concerned
the results we have proved in the previous papers are valid in this context. For certain future constructions it might be necessary to 
impose the condition of  second countability. This is the case, for example,  if the solution space of the unperturbed Fredholm section is not compact
and one constructs a Baire set of perturbations such  that every point has an open neighborhood on which the solution set is generic. In the case of a second countable topology 
one then will be able to obtain a Baire set of perturbations such  that globally the solution set is generic. This is already well-understood in the classical context of the Sard-Smale 
perturbation theory. The difficulties do not arise in the polyfold Fredholm theory,  if the solution set is compact and the perturbations are  sc$^+$-sections.
\end{remark} }
We note that $X$ can have a  boundary with corners or no boundary depending on whether in the charts 
$(U,\varphi, (O, C, E))$ we have $C=E$ or not. M-polyfold versions of  \'etale and proper Lie-groupoids are called ep-groupoids.
 In order to define them it is useful to first recall the concept of a groupoid.

\begin{definition}
A {\bf groupoid} $\mathfrak{G}$ is a small  category whose 
morphisms are all invertible.\index{groupoid}
\end{definition}
Recall that  the category $\mathfrak{G}$ consists of the set  of objects $G$, the set  
${\bf G}$ of morphisms  (or arrows) and the   five  structure maps $(s, t, i, u, m)$. \index{structure maps of a groupoid} Namely, the source and the target maps $s,t:{\bf G}\rightarrow G$ assign to every morphism, denoted by $g:x\to y$,  its source $s(g)=x$ and 
its target $t(y)=y$, respectively.  The associative multiplication (or composition) map 
$$m: {\bf G}{{_s}\times_t}{\bf G}\to {\bf G}, \quad m(h, g)=h\circ g$$
is defined on the fibered product 
$$ {\bf G}{{_s}\times_t}{\bf G}=\{(h, g)\in  {\bf G}\times {\bf G}\, \vert \, s(h)=t(g)\}.$$
For every object $x\in G$, there exists the unit morphism $1_x:x\mapsto  x$ in $\bf{G}$ which is a $2$-sided unit for the composition, that is, $g\circ 1_x=g$ and $1_x\circ h=h$ for all morphisms $g, h\in \bf{G}$ satisfying $s(g)=x=t(h)$. These unit morphisms together define the unit map $u:G\to \bf{G}$ by $u(x)=1_x$. Finally, for every morphism $g:x\mapsto  y$ in $\bf{G}$, there exists the inverse morphism $g^{-1}:y\mapsto  x$ which is a $2$-sided inverse for the composition, that is, $g\circ g^{-1}=1_y$ and $g^{-1}\circ g=1_x$. These inverses together define the inverse map $i:\bf{G}\to \bf{G}$ by $i (g)=g^{-1}.$
The {\bf orbit space} of a groupoid $\mathfrak{G}$,  
$$\abs{\mathfrak{G}}=G/\sim,$$ 
is  the quotient of the set of objects $G$ by the equivalence relation $\sim$ defined by 
$x\sim y$ if and only if  there exists a morphism $g:x\mapsto  y$. \index{orbit space of a groupoid} The equivalence class $\{y\in G\vert \, y\sim x\}$ will be denoted by 
$$\abs{x}:=\{y\in G\vert \, y\sim x\}.$$
If $x, y\in G$ are two objects, then ${\bf G}(x, y)$ denotes the set of all morphisms $g:x\mapsto y$. In particular, for $x\in G$ fixed, we denote by ${\bf G}(x)={\bf G}(x, x)$ the {\bf stabilizer} (or {\bf isotropy }) group of $x$,
$${\bf G}(x)=\{\text{morphisms}\ g:x\mapsto x\}.$$
For the sake of notational economy
we shall denote in the following a groupoid, as well as its object set by the same letter $G$ and its morphism set  by the bold letter ${\bf G}$.

\newenvironment{Myitemize}{%
\renewcommand{\labelitemi}{$\bullet$}%
\begin{itemize}}{\end{itemize}}

Ep-groupoids, as defined next, can be viewed as M-polyfold versions of  \'etale and proper Lie-groupoids discussed e.g. in \cite{Mj} and \cite{MM}.

\begin{definition}\index{ep-groupoid}
{\em An  {\bf ep-groupoid}  $X$ is a groupoid $X$ together with a M-polyfold
structures on the object set $X$ as well as on the morphism set ${\bf X}$
so that all the structure maps  $(s, t, m, u, i)$ are sc-smooth maps and the following
holds true.
\begin{Myitemize}
\item {\em ({\bf \'etale})} The source and target maps
$s$ and $t$ are surjective local sc-diffeomorphisms.
\item {\em ({\bf proper})} For every point $x\in X$,    there exists an
open neighborhood $V(x)$ so that the map
$t:s^{-1}(\overline{V(x)})\rightarrow X$ is a proper mapping.
\end{Myitemize}
}
\end{definition}

We  point out that if  $X$  is a groupoid  equipped  with
M-polyfold structures on the object set $X$ as well as on the
morphism set ${\bf X},$ and  $X$  is \'etale,  then  the
fibered product ${\bf X}{{_s}\times_t}{\bf X}$ has a natural
M-polyfold structure so that the multiplication map $m$ is defined
on an  M-polyfold.  Hence its
sc-smoothness is well-defined, see Lemma \ref{ert} below. 

In an ep-groupoid every morphism $g:x\to y$ can be extended to
a unique local sc-diffeomorphism $t\circ s^{-1}$ satisfying $s(g)=x$
and $t(g)=y$. The properness assumption implies that  the isotropy groups
${\bf G}(x)$ are finite groups.

The orbit space $\abs{X}$ of the ep-groupoid  $X$ is equipped with the quotient topology.

{A M-polyfold  $X$ is equipped with a filtration 
$$
X=X_0\supset X_1\supset X_2\supset \cdots \supset X_\infty:=\bigcap_{j\geq 0}X_j
$$
 into subsets such that the embeddings  $X_{k+1}\to X_k$ are  continuous and $X_\infty$ is dense in every subset $X_k$. Since in an ep-groupoid  the object set $X$ and the morphism set ${\bf X}$ are M-polyfolds, they are both equipped with filtrations. The source and the target maps are local sc-diffeomorphisms,  hence they preserve the levels of the filtrations. Consequently, the filtration of the object set $X$ induces the filtration on the orbit space $\abs{X}$, namely, 
$$
\abs{X}=\abs{X_0}\supset \abs{X_1}\supset \abs{X_2}\supset \cdots \supset \abs{X_\infty}=\bigcap_{j\geq 0}\abs{X_j}.
$$
}
An sc-smooth functor $F:X\to Y$ between two ep-groupoids is called an {\bf equivalence}   if it possesses the following properties.\index{equivalence}
\begin{itemize}
\item[(1)] $F$ is a local sc-diffeomorphism on objects as well as on
morphisms.
\item[(2)] The induced map $\abs{F}:\abs{X}\to \abs{Y}$ between the
orbit spaces {is  a homeomorphism.}
\item[(3)] For every $x\in X$,  the functor $F$ induces an isomorphism 
${\bf G}(x)\rightarrow {\bf G}(F(x)) $ between the isotropy groups.
\end{itemize}
{Let us note that (1)-(3) implies that $|F|(|X_m|)\subset |X_m|$ and that 
$|F|:|X_m|\rightarrow |X_m|$ is a homeomorphism for all $m\geq 0$.}

A {\bf polyfold structure}  on the {paracompact Hausdorff}  topological space  $Z$
is a pair $(X,\beta)$ in which $X$ is an ep-groupoid and
$$\beta:\abs{X}\rightarrow Z$$
a homeomorphism from the orbit space $\abs{X}$ of the ep-groupoid $X$ onto $Z$.\index{polyfold! structure}
Two such polyfold structures $(X,\beta)$ and
$(X',\beta')$ are called  equivalent,
$$(X, \beta )\sim (X', \beta'),$$
if there exists a  third polyfold structure $(X'', \beta'')$ on $Z$ and two equivalences
$$
X\xleftarrow{F} X''\xrightarrow{F'} X'
$$
between the ep-groupoids satisfying
$$\beta''=\beta\circ \abs{F} = \beta' \circ \abs{F'}.$$ 

\begin{definition}
A {\bf polyfold}  is a {paracompact Hausdorff}  topological space 
$Z$ equipped with an equivalence class of polyfold structures.\index{polyfold}
\end{definition}
\begin{remark}
{In our definition we require that a polyfold is a paracompact Hausdorf space.
It is presumably enough to require $Z$ to be a topological space. It is easy to show that
$|X|$ is Hausdorff if $X$ is an ep-groupoid and this implies if $Z$ is homeomorphic to $|X|$
that $Z$ is Hausdorff. At this point it does not seem to be clear whether  $|X|$ is also paracompact 
in general. However, this knowledge is not needed for the Fredholm theory if we know
that the sc-Fredholm section is component-wise proper. We shall investigate this in more detail in \cite{HWZ10}.
There seems to be nothing lost by just requiring in the definition $Z$ to be a topological space.
However, it is important that M-polyfolds are paracompact Hausdorff spaces. In the case of an ep-groupoid
we presently do not know whether  the orbit space $|X|$, which easily is seen to be Hausdorff
as consequence of the properness assumption, has better properties, e. g. is paracompact. }
\end{remark}
In order to recall the concept of a strong bundle over a M-polyfold, we consider two sc-Banach spaces $E$ and $F$ and let $U\subset C\subset E$ be  a relatively  open subset of a partial quadrant $C$ of $E$. By $U\tl F$ we denote the space $U\oplus F$ equipped, however, with the double filtration 
$$(U\tl F)_{m,k}=U_m \oplus F_k,\qquad \text{for $m\geq 0$ and $0\leq k\leq m+1.$}$$
Associated with $U\tl F$ we have the two sc-spaces
$$U\oplus F\qquad \text{and}\qquad U\oplus F^1.$$
An {\bf sc-smooth strong bundle map}  $\Phi:U\tl F\to V\tl G$ is a map of the form 
$$\Phi (u, h)=(f(u), \phi (u, h))$$
which is linear in $h$ and for which the two maps
$$\Phi:U\oplus F^i\to V\oplus G^i,\qquad \text{$i=0$ and  $i=1,$}$$
are sc-smooth.\index{sc-smooth! strong bundle map}
\begin{definition}
An {\bf  sc-smooth strong bundle retraction}  $R:U\tl F\to U\tl F$ is an sc-smooth strong bundle map satisfying $R\circ R=R$. \index{sc-smooth! strong bundle retraction}
\end{definition}
An sc-smooth strong bundle retraction is of the form $R(u, h)=(r(u), \rho (u, h))$  where $\rho$ is linear in $h$ and $r\circ r=r$ so that $r$ is a retraction of $U\subset C\subset E$. We shall abbreviate the associated retracts by 
\begin{align*}
O&=r(U)\subset C\subset  E\\
K&=R(U\triangleleft F)\subset C\triangleleft F \subset E\triangleleft F
\end{align*}
and denote by 
$$p:K\to O, \qquad p(u, h)=u$$
the projection map. 
\begin{definition}\label{local_model_strong_bundle}
{The  sc-smooth  map $p:K\to O$ above   is called a {\bf strong local bundle model}. }
\end{definition}
The retract $K$ is equipped with the double filtration $K_{m,k}$ for $0\leq k\leq m+1$ and $p$ induces the maps $K_{m,k}\to O_m$. By definition, the retraction $R:U\oplus F\to U\oplus F$ is an sc-smooth map and we denote by $K(0)$ the associated retract equipped with the filtration $K(0)_m=K_{m,m}$. Also the retraction $R:U\oplus F^1\to U\oplus F^1$ is, by definition, an sc-smooth map and we denote by $K(1)$ the corresponding retract equipped  with the  filtration $K(1)_m=K_{m, m+1}$. The projection $p:K\to   O$ defines the two sc-smooth maps 
$$p:K(0)\to O\qquad \text{and}\qquad p:K(1)\to O.$$
Hence the bundle map $K(1)\to K(0)$ over the identity map $O\to O$  is a fiberwise compact map. This fact is a very useful tool in the perturbation theory of Fredholm sections. 
\begin{definition}
An {\bf sc-smooth section}  $f$ of the local strong bundle model $p:K\to O$ is an sc-smooth map $f:O\to K(0)$  satisfying $p\circ f=\id$, where $K(0)$ is equipped with the filtration $K(0)_m=K_{m,m}$. \index{sc-smooth! section} An {\bf $\ssc^+$-section} of $p:K\to O$ is an sc-smooth section which is also an sc-smooth map $O\to K(1)$, where $K(1)$ is equipped with the filtration $K(1)_m=K_{m,m+1}$.\index{$\ssc^+$-section}
\end{definition}

We now consider the {paracompact Hausdorff} spaces $X$ and $Y$ and assume that there exists a continuous and surjective map 
$$P:Y\to X$$
having  the property that for every $x\in X$, the preimage $P^{-1}(x)$ posses  the structure of a Banach space.

Then a {\bf strong bundle chart}  $(\Psi, V, K)$ of the bundle $P:Y\to X$ consists of an open subset $V\subset X$, a strong local bundle model $p:K\to O$ and a homeomorphism 
$\Psi:P^{-1}(V)\to K$ which covers a homeomorphism $\psi:V\to O$ and which between every fiber is a bounded linear operator of Banach spaces. \index{strong bundle! chart}
\mbox{}\\
\begin{equation*}
\begin{CD}
P^{-1}(V)@>\Psi>>K\\
@VPVV   @VVpV \\
V@>\psi>>O.\\ 
\end{CD}
\end{equation*}
\mbox{}\\

If $K\to O$ and $K'\to O'$ are two strong local bundle models, then an {\bf sc-smooth strong local bundle map} $f:K\to K'$ is of the form 
$$f(u, h)=(f_0(u), f_1(u, h)),$$
linear in $h$, where $f_0:O\to O'$ is an sc-smooth map and where $f$ induces the two sc-smooth maps 
$$f:K(i)\to K'(i)\qquad \text{for $i=0$ and $i=1$.}$$
{If the above maps $f:K(i)\to K'(i)$ for $i=0$ and $i=1$ are sc-diffeomorphisms, we call $f:K\to K'$ an {\bf sc-smooth strong local bundle diffeomorphism.}}

Two strong bundle charts are sc-smoothly compatible if the transition map is an 
sc-smooth strong local bundle diffeomorphism.

\begin{definition}[{\bf Strong bundle over a M-polyfold}] 
The continuous surjection  $P:Y\to X$ equipped with a maximal atlas of sc-smoothly compatible strong bundle charts is called a strong bundle over the  M-polyfold $X$.\index{strong bundle! over a M-polyfold}
\end{definition}

{Since transition maps of strong bundle charts of the bundle $P:Y\to X$ are sc-smooth and hence sc-continuous, the total space $Y$ inherits a double filtration 
$$Y_{m, k}, \quad \text{for $m\geq 0$ and $0\leq k\leq m+1$.}$$
In particular we obtain the M-polyfolds $Y(0)$ and $Y(1)$ with the filtrations
 $Y(0)_m=Y_{m, m}$ and $Y(1)_m=Y_{m, m+1}$.}

An sc-smooth section of $P$ is a map $f:X\to Y$ satisfying $P\circ f=\id$. Moreover, $f(x)\in Y_{m,m}=Y(0)_m$ if $x\in X_m$ and $f:X\to Y(0)$ is sc-smooth. An {\bf $\ssc^+$-section} of $P$ is a section which satisfies $f(x)\in Y_{m, m+1}=Y(1)_m$ if $x\in X_m$ and the induced map $f:X\to Y(1)$ is sc-smooth. We shall denote these two classes of sections by $\Gamma (P)$ and $\Gamma^+(P)$, respectively.
For a detailed discussion of these concepts we refer to \cite{H2}, \cite{HWZ3.5} and \cite{HWZ7}. 

Finally,  in order to recall from \cite{HWZ3.5} (section 2.4) the notion of a strong bundle over an ep-groupoid we  start with the  ep-groupoid $X=(X, {\bf X})$ and with a strong bundle $p:E\rightarrow X$ over the object set of the ep-groupoid, and we recall the following lemma from \cite{HWZ3.5} (Lemma 2.8).
\begin{lemma}\label{ert}
Let $X$, $Y$ and $Z$ be M-polyfolds. Assume that $f:X\rightarrow Y$ is
a local sc-diffeomorphism and $g:Z\rightarrow Y$ is an $\ssc$-smooth
map. Then the fibered product
$$X{{_f}\times_g}Z=\{(x, z)\in X\times  Z\vert f(x)=g(z)\}$$
 has a natural
M-polyfold structure and the projection map
 $\pi_2:X{{_f}\times_g}Z\rightarrow Z$ is a  local
$\ssc$-diffeomorphism.
\end{lemma}
Since  the source map $s:{\bf X}\to X$ is, by definition, a local sc-diffeomorphism, the fibered product 
$$
{\bf X}{{_s}\times_p}E=\{(g,e)\in {\bf X}\times E\vert \,  s(g)=p(e)\}
$$
is an $M$-polyfold in view of the lemma above. Moreover, the bundle 
$${\bf E}:={\bf X}{{_s}\times_p}E\xrightarrow{\pi_1} {\bf X},$$
defined by the projection $\pi_1(g, e)=g$ onto the first factor is, as the pull-back of the strong $M$-polyfold bundle $p:E\to X$ by the source map $s:{\bf X}\to X$, also a strong $M$-polyfold bundle, in view of Proposition 4.11 in \cite{HWZ2}. 
{\begin{remark}
The results in \cite{HWZ2} concerning M-polyfolds and strong bundles originally assume that these spaces were build on local models 
coming from so-called splicings and bundle splicings, respectively. Splicings and bundle splicings are special examples of  
retractions and strong bundle retractions. The proofs  of the results in \cite{HWZ2}  go through without modifications for these more general retractions.
\end{remark}}

Next  we assume that there is an   {sc-smooth}  strong bundle
map  
$$\mu:{\bf E}\to E$$  covering  the target map $t:{\bf X}\to X$, so that   the diagram

\begin{equation*}
\begin{CD}
{\bf E}@>\mu>>E\\
@V\pi_1VV     @VVp V \\
{\bf X}@>t>> X\\
\end{CD}
\end{equation*}
commutes. 

We require that this {sc-smooth}  strong bundle map  $\mu:{\bf E}\to E$  has the following additional properties.
\begin{itemize}
\item[(1)]  $\mu$ is a local sc-diffeomorphism from ${\bf E}(i)$ onto $E(i)$ for $i=0,1$ and is linear on the fibers $E_x$.
\item[(2)] $\mu(1_x,e_x)=e_x$ for all $x\in X$ and $e_x\in E_x$.
\item[(3)] $\mu(h\circ g,e)=\mu(h,\mu(g,e))$ for all $e\in E$ and $g,h\in {\bf X}$ for which $s(h)=t(g)$ and $s(t)=p(e)$.
 \end{itemize}
 Actually  property (1) is a consequence of the properties (2) and (3). Indeed, the bundle map $\mu:{\bf E}\to E$ is, by definition linear and, in view of (2) and (3),  fiberwise a bijection. In order to verify that $\mu$ induces local sc-diffeomorphisms for $i=0,1$, we choose a point $(g, e)\in {\bf E}(i)$ viewed as an element of ${\bf E}$, and choose an open  neighborhood ${\bf U}\subset {\bf X}$  around $g$, so that  the restriction of the target map
 $$\tau:=t\vert{\bf U}:{\bf U}\to U:=t({\bf U})$$
 is an sc-diffeomorphism. Then $\pi_1^{-1}({\bf U})$ is an open neighborhood of $(g,e)$ in ${\bf E}$ and the restricted map 
 \begin{equation}\label{mu1}
 \mu:\pi^{-1}_1({\bf U})\to p^{-1}(U)\tag{$\ast$}
 \end{equation}
 induces the sc-smooth and bijective maps ${\bf E}(i)\vert{\bf U}\to E(i)\vert U.$ From the properties (2) and (3) again, it follows that the map 
 $$
e\mapsto \bigl(  \tau^{-1}\circ p(e),\mu([ \tau^{-1}\circ p(e)]^{-1},e)\bigr) 
$$
defined for $e\in E\vert U$,  is the inverse of the map  ($\ast$). Since it is an sc-smooth map 
$E(i)\vert U\to {\bf E}(i)\vert {\bf U}$ for $i=0,1$, the map $\mu$ is indeed a local sc-diffeomorphism and property (1) is proved. 

As proved in \cite{HWZ3.5} (section 2.4), the above two strong bundles $p:E\to X$ and $\pi_1:{\bf E}\to {\bf X}$ define, in view of the properties of the map $\mu$, the ep-groupoid 
$$E=(E,{\bf E})$$
having $E$ is its objects and ${\bf E}$ as its morphisms. The map $\mu:{\bf E}\to E$ is the target map and the projection $\pi_2:{\bf E}\to E$ onto the second factor is the source map of this ep-groupoid. This  ep-groupoid could be called a strong bundle ep-groupoid over the ep-groupoid $(X, {\bf X})$ via the functor
$$P:=(p, \pi_1):  E=(E,{\bf E})\to X=(X, {\bf X})$$
between the two ep-groupoids. 
\begin{definition}[{\bf Strong bundle over the ep-groupoid}]\label{stbundleep}
A pair $(E,\mu)$ in which  $p:E\rightarrow X$ is a strong bundle  over  the object $M$-polyfod $X$  of the ep-groupoid $(X, {\bf X})$ and $\mu:{\bf E}\to E$   is an  {sc-smooth} strong bundle map possessing 
the above properties (1)--(3) and {the commutativity condition $p\circ \mu =t\circ \pi_1$},  is called a {\bf  strong bundle over the ep-groupoid } $X=(X,{\bf X})$.\index{strong bundle! over an ep-groupoid}
 \end{definition}
 
Later on we shall sometimes refer to the above functor $P:E\to X$, which is deduced from $(E,\mu)$, as  the strong bundle over the ep-groupoid.
 Next we recall from \cite{HWZ3.5}  the notion of  a strong polyfold bundle structure for the continuous surjection map $\wh{p}:W\to Z$ between two paracompact Hausdorff 
 spaces.
 
 \begin{definition}[{\bf Strong polyfold bundle structure}]\label{stbundleep1}
 A strong polyfold bundle structure on $\wh{p}:W\to Z$ consists of a triple 
 $$(P:E\to X, \Gamma, \gamma)$$ in which $P:E\to X$ is a strong bundle over the ep-groupoid $X$ inducing the map $\abs{P}:\abs{E}\to \abs{X}$ between the orbit spaces, $\Gamma:\abs{E}\to W$ is a homeomorphism between the orbit space $\abs{E}$ and $W$, and $\gamma:\abs{X}\to Z$ is a homeomorphism  between  the orbit space $\abs{X}$  and $Z$. Further, we require that  
 $$\wh{p}\circ \Gamma=\gamma\circ \abs{P}$$
 so that the diagram
\begin{equation*}
\begin{CD}
\abs{E}@>\abs{P}>>\abs{X}\\
@V\Gamma VV @VV\gamma V \\
W@>\wh{p}>>Z\\ \\
\end{CD}
\end{equation*}
commutes. \index{strong polyfold bundle structure}
\end{definition}
 
The reader should consult \cite{HWZ3.5} for the  equivalence of strong bundle
structures over ep-groupoids.
 
\section{Polyfold Fredholm Sections of Strong Polyfold Bundles}\label{sectionpolyfred}

We next recall the notion of a (polyfold) Fredholm section from  \cite{HWZ8.7}. The notion is much more general than the classical notion of a Fredholm section. The first  property of a Fredholm section we require is  the regularization property,  which models the outcome of the elliptic regularity theory.

\begin{definition}
Let $P:Y\rightarrow X$ be  a strong M-polyfold bundle over the M-polyfold $X$.  A section $f\in\Gamma(P)$ is said to be {\bf  regularizing}  if  the following holds. If $x\in X_m$ and $f(x)\in Y_{m,m+1}$, then $x\in X_{m+1}$.\index{regularizing property of a section}
\end{definition}
We observe that if $f\in \Gamma(P)$ is regularizing and $s\in\Gamma^+(P)$,  then $f+s$ is  also regularizing.

We now consider a strong local bundle $K\rightarrow O$.
Here $K=R(U\tl F)$ is  the sc-smooth strong bundle retract  of the {sc-smooth strong bundle retraction}
$$R (u, h)=(r(u),\rho (u)h)$$
where $\rho (u):F\to F$ is a bounded linear map. Moreover, $O$ stands for the retract $O=r(U)$. We assume that $0\in O$ and we are interested in germs of sections $(f, 0)$ of the strong local bundle $K\to O$ defined near $0$. Identifying the local section germ with its principal part we view $(f,0)$ as a germ $\co (O, 0)\to F$. { Here we denote by $\co (O, 0)$ an sc-germ of open neighborhoods of $0$ in $O$ consisting  of a decreasing sequence 
$$U_0\supset U_1\supset U_2\supset \cdots \supset U_m\supset \cdots\ .$$
} of relatively open neighborhoods of $0$ in  $C \subset E$.

In the next step we introduce the useful  notion of a  filling of a  sc-smooth section germ $(f,0)$  of  a  strong local bundle $K\rightarrow O$ near the given smooth point  $0$.  
We do knot require that $f(0)=0$!. The notion of  a filling is a new concept specific to the world of retracts. In all known applications it deals successfully with bubbling-off phenomena and similar singular phenomena.

\begin{definition}[{\bf Filling}]\label{filled-def}
We consider a strong local bundle $K\to O$, where $K=R(U\triangleleft F)$  and the set $U\subset C\subset E$ is a relatively open neighborhood  of $0$ in the partial quadrant $C$ of the  sc-Banach space $E$.  Here  $F$ is a sc-Banach space and $R$ is a strong bundle retraction of the form 
$$R(u, h)=(r(u), \rho(u)(h))$$
covering the  retraction $r\colon U\to U$ onto $O=r(U)$
and $\rho (u)\colon F\to F$ is a bounded linear operator.  We also assume that $r(0)=0$. 

A sc-smooth section germ 
$(f, 0)$ of the bundle $K\to O$ possesses a  {\bf filling}
if there exists a  sc-smooth section germ $(g, 0)$ of the bundle $U\triangleleft F\to U$ extending $f$ and having  the following properties.
\begin{itemize}
\item[(1)]  $f(x)=g(x)$  for $x\in O$ close to $0$.
\item[(2)] If $g(y)=\rho (r(y))g(y)$ for a point $y\in U$ near $0$, then $y\in O$.
\item[(3)] The linearization of the map
$
y\mapsto  [\id -\rho(r(y))]\cdot g(y)
$
(which vanishes at $0$) at the point $0$, restricted to $\ker(Dr(0))$, defines a topological linear  isomorphism
$$
\ker(Dr(0))\rightarrow \ker(\rho (0)).
$$
\end{itemize}
\end{definition}
The crucial property of a filler is the fact that the 
solution sets $\{y\in O\, \vert \, f(y)=0\}$ and $\{y\in U\, \vert \, g(y)=0\}$ coincide near $y=0$. 
Indeed, if $y\in U$ is a solution of the filled section $g$ so that $g(y)=0$, then it follows from (2) that $y\in O$ and from (1) that $f(y)=0$. The section $g$ is, however, much  easier to analyze than  the section $f$,  whose domain of definition has a rather complicated structure. It turns out that in the applications these extensions $g$ are surprisingly easy to  detect. In the Gromov-Witten theory and the SFT they  seem almost canonical.

The condition (3) plays an important  role in the comparison of the linearizations $f'(0)$ and $Dg(0)$ if  $f(0)=0=g(0)$ holds, as we are going to explain next.

It  follows from the definition of a retract that
$\rho (r(y))\circ \rho (r(y))=\rho (r(y)).$ Hence, since $y=0\in O$ we have $r(0)=0$ and $\rho(0)\circ \rho(0)=\rho (0)$ so that $\rho (0)$ is a linear sc-projection in $F$ and we obtain the sc-splitting
$$
F=\rho (0)F\oplus (\id -\rho (0))F.
$$
Similarly, it follows from $r(r(y))=r(y)$ for $y\in U$ that $Dr (0)\circ Dr (0)=Dr(0)$ so that $Dr (0)$ is a linear sc-projection in $E$ which gives rise to the sc-splitting
$$\alpha \oplus \beta \in  E=Dr (0)E\oplus (\id -Dr (0))E.$$
We recall that the linearization $f'(0):T_0O\rightarrow K_0$ of the section $f\colon O\to K$ at $y=0=r(0)$ (assuming $f(0)=0$)  is defined as the restriction of the derivative
$D(f\circ r)(0)$  of the map $f\circ r\colon U\to F$ to $T_0O$. We note that it takes its image in $K_0=\ker(Id-\rho(0))=\rho(0)F$,  and that we have
the identity
$T_0O=Dr(0)E$.
 From $\rho (r(y))f(r(y))=f(r(y))$ for $y\in U$  close to $0$,  we obtain, using $f(0)=0$,  by linearization 
at $y=0$ the relation $\rho (0)Df(0)=Df(0)$. From $g(r(y))=f(r(y))=f(r(r(y))$ for $y\in U$ near $0$ we deduce 
$$
Dg(0)\circ Dr(0)=f'(0)\circ Dr(0)
$$
and therefore $Dg(0)|T_0O=f'(0):Dr(0)E\rightarrow \rho(0)F$.
 From the identity
 \begin{gather*}
(\id -\rho (r(y))g(r(y))=0\quad \text{for all $y\in U,$}
\end{gather*}
we deduce, using $g(0)=0$, the relation
$(\id -\rho (0))Dg (0)\circ Dr(0)=0$. 
Hence  the matrix representation of $D g (0)\colon E\to F$ with respect to the above splittings of $E$ and $F$ looks as follows,$$
Dg(0)\begin{bmatrix}\alpha \\ \beta
\end{bmatrix}=
\begin{bmatrix}f'(0)&\rho (0)Dg(0)(\id-Dr(0))\\0&(\id -\rho (0)) Dg (0)(\id-Dr(0))\end{bmatrix}\cdot
\begin{bmatrix}\alpha\\ \beta
\end{bmatrix}.
$$
In view of property (3), the linear map $\beta \mapsto (\id -\rho (0))\circ Dg (0)(\id-Dr(0))\beta$ from $(\id -Dr (0))E$ to $(\id -D\rho(0))F$ is an isomorphism of Banach spaces. Therefore, 
$$
\ker (Dg (0))=\ker(Df (0))\oplus \{0\}.$$
Moreover the filler has the following additional properties assuming $f(0)=0$.

{\begin{proposition}[{\bf Filler}]\label{filler_new_1}\mbox{}
\begin{itemize}
\item[(1)] The operator $f'(0)\colon Dr(0)E\to \rho (0)F$ is surjective if and only if the operator $Dg(0)\colon E\to F$ is surjective.
Further $\ker(f'(0))=\ker(Dg(0))$.
\item[(2)] $f'(0)$  is a Fredholm operator (in the classical sense) if and only if $Dg(0)$ is a Fredholm operator and $\ind f'0)=\ind Dg(0)$.
\end{itemize}
\end{proposition}
\begin{proof}
(1)\, We assume that $f'(0)$ is surjective and let $\gamma \in F$. Since the operator $(\id-\rho (0))Dg(0)(\id-Dr(0))\colon (\id-Dr(0)E\to (\id-\rho (0))F$ is an isomorphism, we find $\beta\in (\id-Dr(0)E$ satisfying $(\id-\rho (0))Dg(0)\beta=(\id-\rho (0))\gamma$. As $f'(0)$ is surjective, there exists $\alpha\in Dr(0)E$ satisfying $f'(0)\alpha=\rho(0)\gamma-\rho (0)Dg(0)\beta$. Therefore,
\begin{equation*}
\begin{split}
Dg(0)(\alpha+\beta)&=f'(0)\alpha+\rho(0)Dg(0)\beta+(\id-\rho (0))Dg(0)\beta\\
&=\rho(0)\gamma+(\id-\rho (0))\gamma=\gamma,
\end{split}
\end{equation*}
showing that $Dg(0)$ is surjective.\\[0.5ex]
Conversely, assuming that  $Dg(0)\colon E\to F$ is surjective, we let $\gamma \in \rho (0)F$ and find $\alpha+\beta\in Dr(0)E\oplus (\id-Dr(0))E$ solving $Dg(0)(\alpha+\beta)=\gamma.$
Since $\gamma \in \rho (0)F$, we have 
$(\id-\rho(0))Dg(0)(\alpha+\beta)=0$. From $(\id-\rho(0))Dg(0)\beta=0$ we conclude $\beta=0$ and hence $(\id-\rho(0))Dg(0)\alpha =0$. Therefore, $Dg(0)(\alpha+0)=\rho (0)Dg(0)\alpha=f'(0)\alpha=\gamma$, proving that $f'(0)$ is surjective.\\

Assume that $(\alpha,\beta)\in \ker(Dg(0))$. Then $\beta=0$ since $\beta \mapsto (\id -\rho (0))\circ Dg (0)(\id-Dr(0))\beta$ is an isomorphism.
This  implies that $f'(0)\alpha=0$. If on the other hand $f'(0)\alpha=0$, then it is immediate that $(\alpha,0)\in\ker(Dg(0))$.\\[0.5ex]

(2)\, To simplify the notation we abbreviate the above matrix representing $Dg(0)$ by 
$$
Dg(0)=\begin{bmatrix}A&B\\0&C\end{bmatrix}
$$
and abbreviate the above splittings by $E=E_0\oplus E_1$ and $F=F_0\oplus F_1$.
The operators in the matrix are bounded between corresponding Banach spaces and $C\colon E_1\to F_1$ is an isomorphism of Banach spaces. Therefore, if $B=0$, the operator $A=f'(0)\colon E_0\to F_0$ is Fredholm if and only if the operator 
$$\begin{bmatrix}
A&0\\0&C\end{bmatrix}\colon E\to F
$$
is Fredholm in which case their indices agree. The statement now follows from the composition formula
$$
\begin{bmatrix}
\id&BC^{-1}\\
0&\id
\end{bmatrix}
\begin{bmatrix}
A&0\\
0&C
\end{bmatrix}=
\begin{bmatrix}
A&B\\
0&C
\end{bmatrix}
$$
since the first factor is an isomorphism from $E$ to $F$, and hence has index equal to $0$ and the Fredholm indices of a composition are additive. This completes the proof of Proposition \ref{filler_new_1}
\end{proof}


To sum up the role of a filler, instead of studying the solution set of the section $f\colon O\to K$ we can as well study the solution set of the filled section $g\colon U \to U\triangleleft F$, which is defined on the relatively open set $U$ of the partial quadrant $C$ in the sc-space $E$ and easier to analyze.

\begin{proposition}\label{prop2.28}
If  the section  germ  $(f,0)$ of the strong local bundle $K\to O$ has the  filled version $(g,0)$, then the section germ  
$(f+s, 0)$ has the filled version $(g+s\circ r, 0)$ for every $\ssc^+$-section  $s$ of the bundle $K\rightarrow O$.
\end{proposition}

\begin{proof}
We show  that if $t$ is the $\ssc^+$-section of $U\tl F\to F$ defined by $t(y)=s(r(y))$, then 
$(g+t, 0)$ is a filled version of $(f+s, 0)$. In order to verify property (1) for $g+t$ we take $x\in O$ and  obtain, using $r(x)=x$, that $(g+t)(x)=g(x)+s(r(x))=f(x)+s(x)$ which is property (1).

If  $(g+t)(y)=\rho(r(y))(g+t)(y)$ for some $y\in U$, we  conclude, using $\rho (r(y))t(y)=t(y)$, that 
$g(y)=\rho (r(y))g(y)$. Therefore, $y\in O$ by property (2) of $g$. 
Finally, using $[\id -\rho(r(y))]t(y)=0$,  we obtain
$$
[\id-\rho(r(y))](g(y)+t(y))=[\id-\rho(r(y))]g(y).
$$
Hence, in view of property (3) for $g$,  the linearisation of the left-hand side at $y=0$,  if restricted to  the kernel of $Dr(0)$,  satisfies the property (3) for  $g+t$.  
\end{proof}

Next  we introduce a class of so-called basic germs denoted by
$\mathfrak{C}_{basic}$.
\begin{definition}[{\bf Basic germ}]
An element in $\mathfrak{C}_{basic}$ is an sc-smooth germ
$$
f:{\mathcal O}({\mathbb R}^n\oplus W,0)\rightarrow ({\mathbb R}^N\oplus W,0),
$$
where $W$ is an sc-Banach space, so that  if $P:{\mathbb R}^N\oplus W\rightarrow W$  denotes  the projection, then the  germ $P\circ f:{\mathcal O}({\mathbb R}^n\oplus W,0)\rightarrow (W, 0)$ has the form
$$
P\circ f(a,w)=w-B(a,w)
$$
for $(a,w)$ close to $(0,0)\in {\mathbb R}^n\oplus W_0$. Moreover, for every $\varepsilon>0$ and $m\geq 0$,  we have the estimate
$$
\abs{B(a,w)-B(a,w')}_m\leq \varepsilon\cdot\abs{w-w'}_m
$$
for all $(a,w)$, $(a,w')$ close to $(0,0)$ on level $m$.\index{$\ssc^0$-contraction germ}
\end{definition}

We are in the position to define the notion of a Fredholm germ.
\begin{definition}
Let $P:Y\rightarrow X$ be a strong bundle, $x\in X_\infty$, and
$f$ a germ of an sc-smooth section of $P$ around $x$. We call $(f,x)$ a {\bf Fredholm germ}  provided there exists a germ of $\ssc^+$-section $s$ of $P$ near $x$ satisfying
$s(x)=f(x)$ and such that  in suitable strong bundle coordinates mapping $x$ to $0$, the  push-forward  germ $g=\Phi_\ast(f-s)$ around $0$ has a filled version
$\ov{g}$ so that the germ $(\ov{g},0)$ is equivalent  to a germ belonging to  $\mathfrak{C}_{basic}$.\index{Fredholm germ}
\end{definition}

Let us observe that tautologically if $(f,x)$ is a Fredholm germ
and $s_0$ a germ of $\ssc^+$-section around $x$, then $(f+s_0,x)$ is a Fredholm germ as well. Indeed,  if $(f-s,0)$ in suitable coordinates has a filled version with the desired properties, then $((f+s_0)-(s+s_0),0)$ has the same filled version.

Finally,  we can introduce the Fredholm sections of  strong M-polyfold bundles.
\begin{definition}
Let $P:Y\rightarrow X$ be a strong M-polyfold bundle and $f\in\Gamma(P)$ an sc-smooth section. The section $f$ is called a {\bf  polyfold Fredholm section}  provided  it has the following properties:
\begin{itemize}
\item[(1)] $f$ is regularizing.
\item[(2)] At every smooth point $x\in X$,  the germ $(f,x)$ is a Fredholm germ.\index{polyfold! Fredholm section}
\end{itemize}
\end{definition}
If $(f,x)$ is a Fredholm germ and $f(x)=0$,  then the  linearisation $f'(x):T_xX\rightarrow T_{f(x)}Y$ is a linear sc-Fredholm operator. The proof can be found in \cite{HWZ3}.
If,  in addition,  the linearization $f'(x):T_xX\rightarrow  T_{f(x)}Y$ is surjective, then our implicit function theorem gives a natural smooth structure on the solution set of $f(y)=0$ near $x$ as the following theorem from \cite{HWZ3} shows.
\begin{theorem}
Assume that $P:Y\rightarrow X$ is a strong M-polyfold bundle and  let $f$ be a  Fredholm section of the bundle $P$. If the point $x\in X$ solves $f(x)=0$ and if the linearization  at this point $f'(x):T_xX\rightarrow T_{f(x)}Y$ is surjective,  there exists an open neighborhood $U$ of $x$ so that the solution set
$$
S_U:=\{y\in U\vert \ f(y)=0\}
$$
has in a natural way a smooth manifold structure induced from $X$. In addition, $S_U\subset X_\infty$.
\end{theorem}

\section{Gluings and Anti-Gluings}\label{gluinganti-sect}

In this section we carry out the  gluing and anti-gluing constructions for maps from conformal cylinders into $\R^{2n}$.  We start with the formal aspects of these  constructions for continuous maps. For this part we can use any gluing profile, however, to prove sc-smoothness of the expressions for gluing and anti-gluing it will be  important to use 
the {\bf exponential  gluing profile}  given by 
$$\varphi (r)=e^{1/r}-e, \quad r\in (0, 1].$$

We first introduce the  domains on which the glued and anti-glued maps will be defined. 
Let $a\in \hb \subset \C$  be  the disk of radius $\frac{1}{2}$.  
If $a=0$,  we define the set $Z_0$ as the disjoint union of two  half-cylinders
$$
Z_0=({\mathbb R}^+\times S^1)\bigsqcup ({\mathbb R}^-\times S^1).
$$
If $a\neq 0$,  we represent $a$ as  $a=\abs{a}\cdot e^{-2\pi i \vartheta}$
and define the gluing length $R$ by means of the gluing profile as 
$$R =\varphi(\abs{a}).\index{gluing!length}$$
Now we introduce the  abstract  infinite cylinder $C_a$.  We  take the disjoint union of  the half-cylinders $\R^+\times S^1$ and $\R^-\times S^1$ and identify the  points $(s,t)\in [0,R]\times S^1$ with  the points $(s',t')\in [-R,0]\times S^1$  if they satisfy the  relation
$$
s=s'+R\quad \text{and}\quad t=t'+\vartheta.
$$

\begin{figure}[htbp]
\psfrag  {a1}{$(s', t')$}
\psfrag  {a2}{$(s,t)$}
\psfrag {a3}{$[s,t]$}
\psfrag  {a}{$\R^-\times S^1$}
\psfrag {b}{$\R^+\times S^1$}
\psfrag {c}{$C_a$}
\psfrag {01}{$0$}
\psfrag {02}{$0$}
\psfrag  {r1}{$R$}
\psfrag  {mr1}{$-R$}
\centering
\includegraphics[width=3.8in]{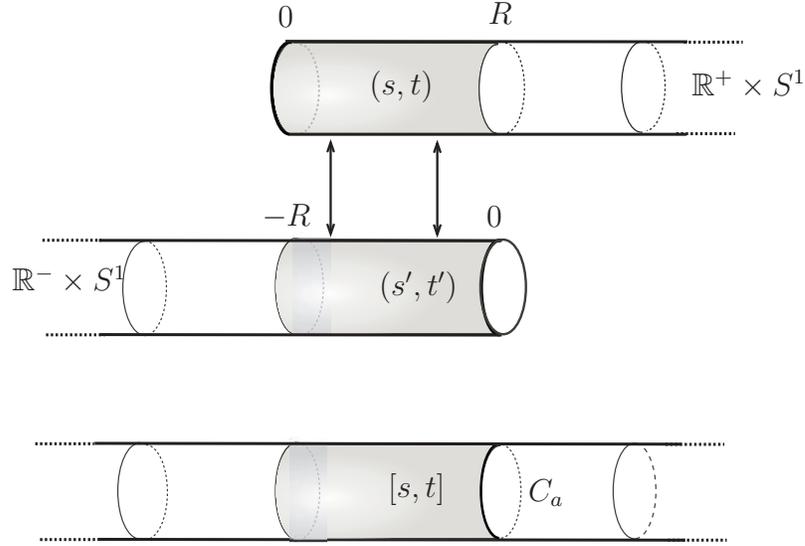}
\caption{Glued  infinite cylinders $C_a$.}
\label{Fig9}
\end{figure}

This identification  is compatible with the standard conformal structure on the cylinders. There are two global conformal coordinates
on $C_a$ defined by
\begin{align*}
&C_a\rightarrow {\mathbb R}\times S^1, \quad [s,t]\mapsto (s,t),\\
\intertext{and}
&C_a\rightarrow {\mathbb R}\times S^1,  \quad [s',t']'\mapsto (s',t').
\end{align*}
The first coordinates are  the extension of the  coordinates  $(s,t)$ for $s\geq 0$  and the second are  the extension of the coordinates  $(s', t')$ for $s'\leq 0$. 
Clearly, if  $[s,t]=[s',t']'$, then  $s=s'+R$ and $t=t'+\vartheta.$
The infinite cylinder $C_a$ contains the finite  sub-cylinder $Z_a$ defined by 
$$Z_a=\{[s, t]\vert\  (s, t)\in [0, R]\times S^1\}=\{[s', t']'\vert \ (s', t')\in [-R, 0]\times S^1\}.$$

 \begin{figure}[htbp]
\psfrag {01}{$0$}
\psfrag {02}{$0$}
\psfrag {a1}{$(s', t')$}
\psfrag {a2}{$(s,t)$}
\psfrag {a3}{$[s,t]$}
\psfrag  {mr1}{$-R$}
\psfrag {r1}{$R$}
\psfrag {a}{$\R^-\times S^1$}
\psfrag {b}{$\R^+\times S^1$}
\psfrag {c}{$Z_a$}
\centering
\includegraphics[width=3.8in]{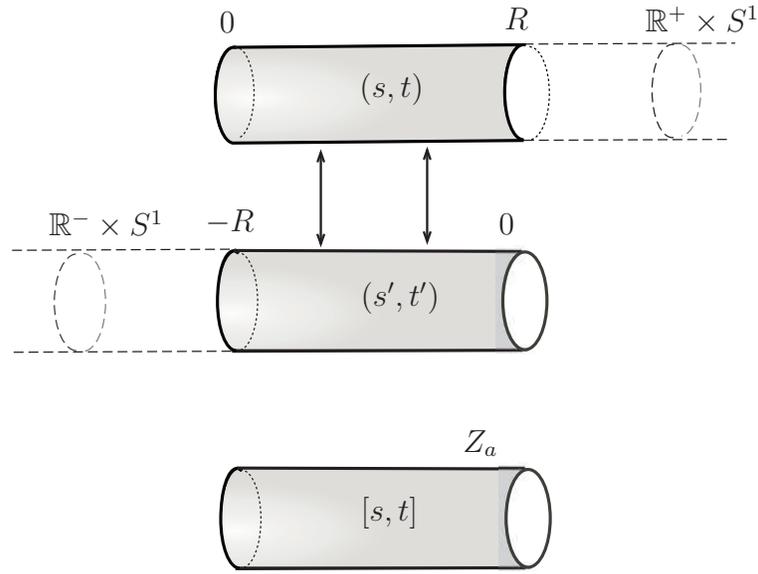}
\caption{Glued  finite cylinders $Z_a$.}
\label{Fig10}
\end{figure}

One  should think of the coordinates $(s, t)\in \R^+\times S^1$ and $(s', t')\in \R^-\times S^1$
as holomorphic polar coordinates on the annuli $D_x\setminus \{x\}$ and $D_y\setminus \{y\}$ in 
the disks of the small disk structure belonging to the nodal pair  $\{x,y\}$. In this case  the gluing  parameter $a=a_{\{x, y\}}$  described in Section \ref{dm-subsect} defines  the neck region $Z_{a_{\{x, y\}}}^{\{x, y\}}$.
The infinite cylinder $C_a$ will be used to describe data which otherwise would be lost during the gluing process. This is not too important for the Deligne-Mumford theory but  it is crucial once we consider maps defined on  the Riemann surfaces. {For the gluing parameter  $a=0$ we define
$$C_0=\emptyset.$$
Observe that given any codomain $A$ there is precisely one map $\emptyset \rightarrow A$. }

In order to define the gluing and anti-gluing construction for maps we need to fix additional data. We choose 
a smooth cut-off function $\beta:{\mathbb R}\rightarrow [0,1] $ having the following properties:
\begin{equation}\label{properties_beta_1}
\begin{aligned}
\bullet &\:\:\beta (-s)+\beta (s)=1\quad \text{for all $s\in \R$}\\
\bullet &\:\:\beta (s)=1\quad \text{for all $s\leq -1$}\\
\bullet &\:\:\beta'(s)<0\quad \text{for all $s\in (-1, 1)$}.
\end{aligned}
\end{equation}

If $a\in \C$ is a gluing parameter and $R=\varphi (\abs{a})$ the associated gluing length,  we introduce the translated function  $\beta_a=\beta_R:\R\to \R$ defined by 
\begin{equation}\label{translation_function}
\beta_a (s):=\beta \biggl( s-\frac{R}{2}\biggr).
\end{equation}

In order to  define the basic ${\mathbb R}^{2n}$-gluing $\oplus_a$, we take  a pair  $(u^+,u^-)$ of continuous maps 
$$u^{\pm}:{\mathbb R}^\pm\times S^1\rightarrow {\mathbb R}^{2n}
$$
satisfying
$$
\lim_{s\to \infty} u^\pm(s,t)=:u^{+}(\infty )=u^{-}(-\infty ):=\lim_{s\to  -\infty} u^-(s, t)
$$
uniformly in $t$.   We call the constant $c:=u^+(\infty )=u^-(-\infty )$ the common asymptotic constant of the pair $(u^-, u^+)$ and we say that $u^\pm$ possess  asymptotic matching conditions. 
For such a pair of maps we define the glued map 
$$\oplus_a(u^+,u^-):Z_a\to \R^{2n}$$
for a gluing parameter $a$  as follows. 
If $a=0$, we define
$$\oplus_0(u^+, u^-):=(u^+, u^-):Z_0\to \R^{2n}.$$
If $0<\abs{a}\leq \frac{1}{2}$,  we define 
\begin{equation*}
\oplus_a(u^+,u^-)([s,t])=\beta_a (s)\cdot u^+(s, t)+(1-\beta_a(s))\cdot u^-(s-R, t-\vartheta).\index{nonlinear  ${\mathbb R}^{2n}$-gluing $\oplus_a$}
\end{equation*}
We will always  use  the coordinates $[s,t]$  on $Z_a$ for $a\neq 0$. The formula in the coordinates $[s',t']'$ looks similar. The above procedure  is a ``nonlinear gluing'' of maps into ${\mathbb R}^{2n}$. We shall  see  below how to replace ${\mathbb R}^{2n}$ by a manifold.

Fixing a pair $(u^+, u^-)$ of maps as above, we shall define the glued and anti-glued maps  for pairs of vector fields  $(h^+, h^-)$ along the pair  $(u^+, u^-)$ of curves  so that 
$h^\pm (s,t)\in T_{u^{\pm}(s,t)}{\mathbb R}^{2n}={\mathbb R}^{2n}$.\index{gluing of vector fields}
We denote by $E$  the vector space of pairs $(h^+, h^-)$ of continuous maps 
$$h^{\pm}:{\mathbb R}^\pm\times S^1\rightarrow {\mathbb R}^{2n}$$
satisfying 
$$\lim_{s\to  \infty}h^{+}(s, t)=:h^+(\infty)=h^-(-\infty ):=\lim_{s\to  -\infty}h^{-}(s, t)$$
uniformly in $t$,  where $h^{\pm}(\pm \infty)$ are the asymptotic  constants  in $ \R^{2n}$. 
For a pair $(h^+, h^-)\in E$, we define the glued vector field  $\oplus_a(h^+, h^-):Z_a\to \R^{2n}$  by the same formulae  as before. Namely, if 
$a=0$, we set 
$$\oplus_0(h^+, h^-)=(h^+, h^-), $$
and if $0<\abs{a}<\frac{1}{2}$,  we define
$$\oplus_a(h^+, h^-)([s, t])=\beta_a(s)\cdot h^+(s, t)+(1-\beta_a(s))\cdot h^-(s-R, t-\vartheta).$$
We may view $\oplus_a(h^+, h^-)$ as a  vector field along the glued map $\oplus_a(u^+,u^-)$. If  $\exp$ is  the exponential map on $\R^{2n}$ with respect to the Euclidean
metric, then, given the  pair of vector fields $(h^+, h^-)\in E$ along the pair $(u^+, u^-)$, so that  
$$
\exp_{(u^+,u^-)}(h^+,h^-)=(u^ ++h^+,u^- +h^-), 
$$
we obtain  the  formula,
\begin{equation}\label{exp-formula}
\oplus_a(\exp_{(u^+,u^-)}(h^+,h^-)) =
\exp_{\oplus_a(u^+,u^-)}(\oplus_a(h^+,h^-)).
\end{equation}

Next we introduce the operation of anti-gluing for vector
fields.\index{anti-gluing of vector fields}
The anti-glued vector field  
$$
\ominus_a(h^+,h^-):C_a \rightarrow \R^{2n}
$$
is defined as  follows.  {If $a=0$, we set 
$$\ominus_a(h^+,h^-)=0.$$ 
Here $0$ is the unique map $\emptyset\rightarrow {\mathbb R}^{2n}$.}
If $0<\abs{a}<\frac{1}{2}$,  we define
\begin{equation*}
\begin{split}
\ominus_a (h^+,h^-)([s,t])&=-(1-\beta_a(s))\cdot [ h^+(s,t)-\av_a(h^+, h^-)]\\
&\phantom{=\ }+\beta_a(s)\cdot[ h^-(s-R,t-\vartheta)-\av_a(h^+, h^-)].
\end{split}
\end{equation*}
The average is the number 
$$
\av_a(h^+, h^-):=\frac{1}{2}\left(  [h^+]_R+[h^-]_R\right) 
$$
where
$$
 [h^\pm ]_R:=\int_{S^1} h^{\pm} \left (\pm \frac{R}{2}, t\right)\  dt.
$$
The anti-glued map possesses  antipodal asymptotic constants, 
$$
\ominus_a(h^+,h^-)(-\infty)=-\ominus_a(h^+,h^-)(+\infty).
$$
\begin{remark}\label{verynewremark1}
Let us summarize the above discussion. The gluing and anti-gluing constructions will be later on implanted  into the manifold setting. The following view point about the  
$\R^{2n}$-case will be useful. 
One fixes a smooth pair $(u^+_0, u^-_0)$ with vanishing asymptotic constant. We refer to   $(u^+_0, u^-_0)$  as the base pair.  We consider nearby pairs $(u^+, u^-)$ with  matching asymptotic constants. For these pairs the nonlinear gluing  is defined. Moreover,  the nearby pairs $(u^+, u^-)$ can be expressed, using the exponential map $\exp$ on $\R^{2n}$,  as  
$$\exp_{(u^+_0, u^-_0)}(h^+, h^-)=(u_0^++h^+, u_0^-+h^-)$$
for some pair of vector fields $(h^+, h^-)$ along $ (u^+_0, u^-_0).$ For vector fields $(h^+, h^-)$ along the base pair $(u^+_0, u^-_0)$,  the linear gluing $\oplus_a (h^+, h^-)$ is defined and  produces a vector field along the glued base pair $\oplus_a (u^+_0, u^-_0)$. 
In addition, the following holds
\begin{equation*}
\oplus_a(\exp_{(u_0^+,u_0^-)}(h^+,h^-)) =
\exp_{\oplus_a(u_0^+,u_0^-)}(\oplus_a(h^+,h^-)).
\end{equation*}
The negative gluing $\ominus_a (h^+, h^-)$  of a pair of vector fields $(h^+, h^-)$  takes values in the tangent space $T_0\R^{2n}$ of $\R^{2n}$ at $0$,  where $0$ is the common asymptotic limit of the base pair  $(u^+_0, u^-_0)$. 
\end{remark}

The gluing constructions introduced above are intimately related  to  the concept of
splicing as we shall describe next. 

Recall that  we have denoted by  $E$  the  space of pairs  of continuous vector fields $(h^+, h^-)$ along $(u^+,u^-)$ satisfying  the asymptotic matching conditions 
$h^+(\infty )=h^-(-\infty)$.  We introduce the space  $E^a_+$ consisting of continuous maps $v:Z_a\to \R^{2n}$ and the space
$E^a_{-}$  consisting of continuous maps $w:C_a\to \R^{2n}$  satisfying $w(-\infty)=-w(\infty)$. 
We define,  for $0<\abs{a}<\frac{1}{2}$,  the {\bf  total glued map}  $\boxdot_a :E\to E^a_+\oplus E^a_{-}$ 
by 
$$\boxdot_a(h^+, h^-)=\left(\oplus_a (h^+, h^-), \ominus_a (h^+, h^-)\right).$$
If $a=0$,   the map $\boxdot_0:E\to  E=E\oplus\{0\}$ is defined  by 
$$\boxdot_0(h^+, h^-)=((h^+, h^-),0).\index{total glued map  $\boxdot_a$} $$
\begin{lemma}\label{direct-lem} 
The map  $\boxdot_a:E\to E^a_+\oplus E^a_{-}$ is an isomorphism. 
\end{lemma}
In order to  prove the lemma we take a pair $(v, w)\in E^a_+\oplus E^a_{-}$ and solve the system of two equations for $(h^+, h^-)$,
\begin{equation}\label{system1}
\oplus_a (h^+, h^-)=v\quad \text{and}\quad \ominus_a (h^+, h^-)=w.
\end{equation}
The first equation is defined for the  points $[s,t]\in Z_a$ and the second for all points $[s,t]\in C_a$.
Integrating  the first equation over the circle $\{\frac{R}{2}\}\times S^1$, we find that 
$$
\av_a(h^+,h^-)=   \frac{1}{2}\left( [h^+]_R+[h^-]_R\right)=[v]:=\int_{S^1} v\left(\frac{R}{2},t\right)\ dt .
$$
It will be convenient to introduce the so called {\bf hat version of the gluing and anti-gluing}. If  $(h^+, h^-)\in E$, 
we define 
$$\wh{\oplus}_a(h^+, h^-):=\oplus_a (h^+, h^-)$$
and
$$\wh{\ominus}_a (h^+, h^-)[s,t]=-(1-\beta_a)\cdot h^+ (s, t)+\beta_a \cdot h^- (s -R, t -\vartheta)$$
where we have abbreviated $\beta_a=\beta_a(s)$. 
The relation between $\ominus_a (h^+, h^-)$ and $\wh{\ominus}_a(h^+, h^-)$ is  the following,
$$\ominus_a(h^+, h^-)=\wh{\ominus}_a(h^+, h^-)-\wh{\ominus}_a ([v], [v])=\wh{\ominus}_a(h^+, h^-)-(2\beta_a-1)[v]$$
where we have used that  $[v]= \av_a(h^+,h^-)$.
In view of  the hat anti-gluing, the second equation in  \eqref{system1} can be rewritten as
\begin{equation*}
\wh{\ominus}_a(h^+, h^-)=w+(2\beta_a-1)[v]=:\wh{w}
\end{equation*}
so that  the two equations in  \eqref{system1} take the form,
\begin{equation*}
\wh{\oplus}_a(h^+, h^-)=v\quad \text{and}\quad \wh{\ominus}_a(h^+, h^-)=\wh{w},
\end{equation*}
or, in the matrix form,  
\begin{equation}\label{system2}
\begin{bmatrix}
\phantom{-}\beta_a&1-\beta_a\\
-(1-\beta_a)&\beta_a
\end{bmatrix}\cdot
\begin{bmatrix}
h^+ (s, t)\\   h^-(s-R, t-\vartheta)\end{bmatrix}=
\begin{bmatrix}
v(s, t)\\ \wh{w}(s, t)\end{bmatrix}.
\end{equation}
Observe that the second equation $\wh{ \ominus}_a(h^+, h^-)=\wh{w}$  determines $h^+$ for $s\geq \frac{R}{2}+1$ and $h^-$ for $s\leq \frac{R}{2}-1$. Indeed, in view of the properties of the cut off functions $\beta_a$, 
\begin{align*}
h^+(s, t)&=-\wh{w}[s,t]=-w([s,t])+[v], \quad s\geq \frac{R}{2}+1\\
h^-(s-R, t-\vartheta)&=\wh{w}([s, t])=w([s,t])+[v], \quad s\leq \frac{R}{2}-1.
\end{align*}
We conclude that  the asymptotic constant  $h^+(\infty)$ of $h^+$ is equal to $-w(\infty)+[v]$ and  the asymptotic constant $h^-(-\infty)$ of $h^-$ is equal to $w(-\infty)+[v]$. 
Since, by assumption $w(-\infty)=-w(\infty)$, we conclude  
$$h^+(\infty)=h^-(-\infty).$$

Abbreviating by $\gamma_a:=\beta_a^2+(1-\beta_a)^2$ the determinant of the square matrix on the left-hand side above,  we find for the unique solutions  $h^\pm$  the formulae 
\begin{equation}\label{system3p}
h^+(s, t)=\frac{1}{\gamma_a}\left( \beta_a \cdot v(s, t)-(1-\beta_a)\cdot \wh{w}(s, t)\right)
\end{equation}
 for  $s\geq 0$ and $t\in S^1$,   and 
\begin{equation}\label{system3m}
h^-(s-R, t -\vartheta)=\frac{1}{\gamma_a}\left( (1-\beta_a)\cdot  v(s, t)+\beta_a\cdot \wh{w}(s, t)\right)
\end{equation}
 for $s\leq R$ and $t\in S^1$. Since $h^\pm$ have identical asymptotic constants, the pair
$(h^+, h^-)$ belongs to the space $E$. This completes the proof that  the map $\boxdot_a:E\to E^a_+\oplus  E^a_{-}$ is an isomorphism. \hfill $\blacksquare$

The injectivity of $\boxdot_a$ implies $(\ker \oplus_a )\cap (\ker \ominus_a )=\{(0, 0)\}$ and the surjectivity shows that 
$E=\ker \oplus_a \oplus  \ker \ominus_a $. Consequently, the space $E$ has the decomposition
\begin{equation}\label{decomp}
E=(\ker \oplus_a )\oplus (\ker \ominus_a ).
\end{equation}
We now introduce the projection 
$$\pi_a:E\rightarrow E$$
 onto $\ker \ominus_a$ along $\ker \oplus_a$.  If $a=0$, we set $\pi_0=\id$. Using our calculations  above we shall  derive a formula for this projection in the case $a\neq 0$.  Given a pair 
 $(h^+, h^-)\in E$,   we  set 
$$\pi_a (h^+,h^-)=(\eta^+,\eta^-).$$
In view of the decomposition \eqref{decomp}, the pair $(\eta^+,\eta^-)$ is uniquely determined by the requirement
$$
\oplus_a(\eta^+,\eta^-)=\oplus_a(h^+,h^-)\quad \text{and}\quad \ominus_a(\eta^+,\eta^-)=0.
$$
Setting $v=\oplus_a(h^+, h^-)$ and $w=0$,   the system of equations  \eqref{system1} takes the form 
$$
\oplus_a(\eta^+,\eta^-)=v\quad \text{and}\quad \ominus_a(\eta^+,\eta^-)=0
$$
and we find,  in view of \eqref{system3p}, 
\begin{equation}\label{sol1}
\eta^+=\frac{1}{\gamma_a}\left[ \beta_a\cdot  v-(1-\beta_a)\cdot \wh{w}\right]
\end{equation}
where $\wh{w}=(2\beta_a-1)[v]$  and  $[v]$ is equal to 
$$[v]=\int_{S^1}v \left(\frac{R}{2}, t\right)\ dt= \frac{1}{2}\left(  [h^+]_R+ [h^-]_R\right) =\av_a (h^+, h^-).$$ 
Substituting $\oplus_a(h^+, h^-)$ for $v$  and $(2\beta_a-1)[v]$ for $\wh{w}$ into  the equation  \eqref{sol1}, we obtain the explicit formula
\begin{equation}\label{eta-plus}
\begin{split}
\eta^+(s, t)&=\frac{1}{\gamma_a}\left( \beta_a\cdot  \oplus_a(h^+, h^-)(s, t) -(2\beta_a-1)(1-\beta_a)[v]\right)\\
&=\left(1-\frac{\beta_a}{\gamma_a}\right)\cdot \av_a(h^+, h^-) +\frac{\beta_a^2}{\gamma_a}\cdot h^+(s, t)\\
&\phantom{=}+ \frac{\beta_a (1-\beta_a)}{\gamma_a}\cdot h^-(s-R, t -\vartheta)
\end{split}
\end{equation}
for $(s, t)\in \R^+\times S^1$.  We have abbreviated $\beta_a=\beta_a(s)$ and $\gamma_a=\gamma_a(s)$. 
A similar calculation uses \eqref{system3m} and leads to the  following formula for $\eta^-$,
\begin{equation*}
\begin{split}
\eta^-(s-R, t-\vartheta)&=\left( 1-\frac{1-\beta_a}{\gamma_a}\right)\cdot \av_a (h^+, h^-)  +\frac{\beta_a (1-\beta_a)}{\gamma_a}\cdot h^+(s, t)\\
&\phantom{==}+\frac{(1-\beta_a)^2}{\gamma_a}h^-(s -R,t-R)
\end{split}
\end{equation*}
for $(s, t)$ satisfying $s\leq R$. 
To express $\eta^-$ is terms of  the coordinates $(s', t')$ on $\R^-\times S^1$, we note that from  $\beta (s)=1-\beta (-s)$ and $\beta_a (s)=\beta (s-\frac{R}{2})$ and $\gamma_a=\beta_a^2+(1-\beta_a)^2$, it follows that 
$\beta_a (s'+R)=1-\beta_a (-s')\quad \text{and}\quad \gamma_a (s'+R)=\gamma_a (-s').$
Consequently,
\begin{equation}\label{eta-minus}
\begin{split}
\eta^-(s', t')&=\left( 1-\frac{\beta_a (- s')}{\gamma_a (-s')}\right)\cdot \av_a (h^+, h^-)\\
&\phantom{}  +\frac{\beta_a (-s' ) (1-\beta_a (-s' ))}{\gamma_a(-s' )}\cdot h^+(s'+R, t' +\vartheta) +\frac{\beta_a (-s' )^2}{\gamma_a(-s' )}h^-(s', t')
\end{split}
\end{equation}
for $(s', t')\in \R^-\times S^1$.

We next  study  the above constructions  in the analytical framework of
weighted  Sobolev spaces.  For these studies  it is  crucial to  use  the exponential gluing profile
$$\varphi (r)=e^{\frac{1}{r}}-e, \quad r\in (0, 1].$$

We fix   a strictly increasing sequence $(\delta_m)_{m\geq 0}$  satisfying 
$0<\delta_m<2\pi$ for all $m\geq 1$. If $m\geq 0$,  the space $E_m^\pm=H^{3+m,\delta_m}(\R^\pm \times S^1)$ 
consists of $L^2$-maps 
$$u:\R^\pm \times S^1\to \R^{2n}$$
whose  weak partial derivatives $D^{\alpha}u$   up to order $3+m$ belong,   if weighted by $e^{\delta_m \abs{s}}$,   to the space $L^2(\R^\pm \times S^1)$ so that  $e^{\delta_m \abs{s}} D^{\alpha}u\in L^2(\R^\pm \times S^1)$ for all $\abs{\alpha}\leq 3+m$. 
The space $E_m$ consists  of pairs $(u^+, u^-)$ of maps defined on half-cylinders
$$u^\pm :\R^\pm \times  S^1\to \R^{2n}$$
for which there is a constant  $c\in \R^{2n}$ (depending on $(u^+, u^-)$)  so that 
$$u^\pm -c\in E^\pm_m.$$
If $(u^+, u^-)\in E_m$, then $u^\pm=c+r^\pm$ and its $E_m$-norm is defined by 
$$\abs{(u^+, u^-)}^2_{E_m}=\abs{c}^2+\norm{r^+}^2_{3+m,\delta_m}+\norm{r^-}_{3+m,\delta_m}^2$$
where
$$\norm{r^\pm}^2_{3+m,\delta_m}=\sum_{\abs{\alpha}\leq m+3}\int_{\R^\pm \times S^1}\abs{D^{\alpha}r^\pm (s, t)}^2e^{2\delta_m \abs{s}}\ ds dt.$$
Since the sequence $(\delta_m)$ is {strictly}  increasing, it follows from the Sobolev embedding theorem  that the inclusions $E_{m+1}\to E_m$ are compact maps and the vector space $E_{\infty}:=\bigcap_{m\geq 0}E_m$ is dense in every $E_m$.  In our terminology, the space $E:=E_0$ is an sc-Banach space  and the nested sequence $(E_m)$ defines an sc-structure on $E$.
  {We would like to point out that  we could assume, as in our previous papers,  that the sequence $(\delta_n)$ starts with $\delta_0=0$. However, since  the Cauchy-Riemann operator defined on the infinite  cylinder is not Fredholm if $\delta_0=0$,  we will always assume that $0<\delta_m<2\pi$.} 
 
Moreover, we define for every $a\in \hb\subset  \C $, the sc-Banach  space $G^a$  as follows.
If $a=0$, we set  
$$G^0=E\oplus \{0\},$$
and if $a\neq 0$,   we define
$$G^a=H^3(Z_a, \R^{2n})\oplus H^{3,\delta_0}_c(C_a, \R^{2n})$$
where the  space $ H^{3,\delta_0}_c(C_a, \R^{2n})$ consists of functions $u:C_a\to \R^{2n}$ for which there exists an asymptotic constant $c\in \R^{2n}$ such that  $u$ converges to $c$ as $s\to \infty$ and to $-c$ as $s\to -\infty$ and  
$$ u- (1-2\beta_a)\cdot c\in H^{3, \delta_0}(C_a, \R^{2n}).$$
We recall that $\beta_a(s)=0$ if $s\geq \frac{R}{2}+1$ and $\beta_a(s)=1$ for $s\leq \frac{R}{2}-1$. The Hilbert space space $H^3(Z_a, \R^{2n})$ is equipped with the 
sc-structure $H^{3+m}(Z_a, \R^{2n})$ for all $m\geq 0$ and the sc-structure for $H^{3,\delta_0}_c(C_a, \R^{2n})$  is given by 
$H^{3+m,\delta_m}_c(C_a, \R^{2n})$. Thus the sc-structure on $G^a$ is defined by  the sequence 
$$G^a_m=H^{3+m}(Z_a, \R^{2n})\oplus H^{3+m,\delta_m}_c(C_a, \R^{2n}), \quad m\geq 0.$$
The total glued map 
$$\boxdot_a(h^+,h^-)=( \oplus_a(h^+,h^-),\ominus_a(h^+,h^-))$$
maps the $m$-level $E_m$ onto the $m$-level $G^a_m$. Recalling the projection $\pi_a$  onto $\ker \ominus_a$ along $\ker \oplus_a$ and the formulae \eqref{eta-plus} and \eqref{eta-minus} 
for 
$$\pi_a(h^+, h^-)=(\eta^+, \eta^-), $$
one  sees  that $\pi_a:E\to E$. 
{Both parts of the following theorem  are proved in  \cite{HWZ8.7}; part (1) is proved as Theorem 2.23 and part (2) is proved in Section 2.4 of  \cite{HWZ8.7}.}

\begin{theorem}\label{sc-splicing-thm}
\mbox{}
\begin{itemize}
\item[{\em (1)}] For every $a\in \hb $, the total  gluing  map 
$$\boxdot_a=( \oplus_a,\ominus_a):E\to G^a$$
is an sc-linear isomorphism. 
\item[{\em (2)}] The map 
$$\pi: \hb \oplus E\to E,\quad  (a, (h^+, h^-))\mapsto \pi_a(h^+, h^-)$$
is sc-smooth. Consequently, the triple $\cs=(\pi,E,\hb )$  is an sc-smooth splicing.
\end{itemize}
\end{theorem}
The splicing  ${\mathcal S}=(\pi,E,\hb )$ defines the splicing core $K^{\mathcal S}$ which is  the set
$$
K^{\mathcal S}=\{(a,(h^+,h^-))\in \hb \times E\, \vert \quad
\pi_a (h^+,h^-)=(h^+,h^-)\}.
$$
In view of Theorem \ref{sc-splicing-thm}, the map 
$r:\hb \oplus E\to \hb \oplus E$,  defined by 
$$r(a, (h^+, h^-))=(a, \pi_a (h^+, h^-)),$$
is an sc-retraction  and the splicing core $K^{{\mathcal S}}=r (\hb \oplus E)$ is   the associated sc-smooth retract.

In our construction of strong bundles  below we shall make use of the following version of the hat gluing and hat anti-gluing.

We introduce the space   $F$ of pairs $(\xi^+,\xi^-)$  in which   the maps 
$$\xi^+:\R^+\times S^1\to \R^{2n}\quad \text{and}\quad \xi^-:\R^-\times S^1\to \R^{2n}$$
are continuous and have the common asymptotic constant equal to $0$,
$$\lim_{s\to \infty}\xi^+(s, t)=\lim_{s\to -\infty}\xi^-(s, t)=0,$$
where the  convergence  is uniform in $t\in S^1$. Given $(\xi^+, \xi^-)\in F$, the hat-glued map $\wh{ \oplus}_a (\xi^+, \xi^-):Z_a\to \R^{2n}$ is  defined by
the same formula as  $\oplus_a (\xi^+, \xi^-)$. Namely,  
if $a=0$, then 
$$\wh{\oplus}_0(\eta^+, \eta^-):=(\eta^+,\eta^-)$$
and, if $0<\abs{a}<\frac{1}{2}$, then 
\begin{equation*}
\wh{\oplus}_a(\xi^+,\xi^-)([s,t])=\beta_a(s)\cdot \xi^+(s,t)+
(1-\beta_a(s))\cdot \xi^-(s-R,t-\vartheta).\index{hat--gluing map}
\end{equation*}

The hat-anti-glued map  $\wh{\ominus}_a(\xi^+,\xi^-):C_a\to \R^{2n}$ is defined as follows. 
If  $a=0$, we set 
$$\wh{\ominus}_a(\xi^+,\xi^-)=0.$$
However, if $a\neq 0$, the definition of  $\wh{\ominus}_a (\xi^+, \xi^-)$ differs from the anti-glued map $\ominus_a$. 
If $0<\abs{a}<\frac{1}{2}$,  the hat anti-glued map $\wh{\ominus}_a(\xi^+,\xi^-):C_a\to \R^{2n}$ is defined by  
\begin{equation*}
\wh{\ominus}_a(\xi^+,\xi^-)([s,t])\\
=-(1-\beta_a(s))\cdot \xi^+(s,t)+
\beta_a(s)\cdot \xi^-(s-R,t-\vartheta).\index{hat--anti- gluing map}
\end{equation*}
As  before we have  the  direct sum decomposition 
$$F=\ker \wh{\oplus}_a \oplus \ker \wh{\ominus}_a$$
for every $a\in \hb $,  and introduce  the projection map
$$\wh{\pi}_a:F\to F$$
onto $\ker \wh{\ominus}_a$  along  $\ker \wh{\oplus}_a$.  Representing the projection as 
\begin{equation*}
\wh{\pi}_a(\xi^+, \xi^-)=(\eta^+, \eta^-), 
\end{equation*}
a  calculation similar to the one we performed above  leads to the formulae 
\begin{equation}\label{smooth-splicing-10}
\begin{aligned}
\eta^+ (s, t)& = \frac{\beta_a^2}{\gamma_a}\cdot \xi^+(s, t) +\frac{\beta_a (1-\beta_a)}{\gamma_a}\cdot \xi^-(s -R, t -\vartheta)\\
\eta^- (s', t')&=\frac{\beta_a (-s' ) (1-\beta_a (-s' ))}{\gamma_a(-s' )}\cdot \xi^+(s' +R,t'+\vartheta) +\frac{\beta_a (-s')^2}{\gamma_a(-s' )}\xi^-(s', t')
\end{aligned}
\end{equation}
where $\eta^+$ is defined for $(s, t)\in \R^+\times S^1$ and $\eta^-$  for $(s', t')\in\R^-\times S^1$.  We have abbreviated $\beta_a=\beta_a(s)$ and $\gamma_a=\gamma_a(s)$.

The analytical framework of the hat gluing and anti-gluing is similar to the Sobolev framework above;  the  weighted Sobolev  spaces will however be different.
For  the  strictly increasing sequence $(\delta_m)_{m\geq 0}$  satisfying $0<\delta_m<2\pi$ we let $F$ be the space of pairs $(\xi^+,\xi^-)$ of maps 
$$\xi^\pm :\R^\pm \times S^1\to \R^{2n}$$
satisfying  $\xi^\pm\in H^{2, \delta_0}(\R^\pm \times S^1, \R^{2n})$. In particular, we note,  that the common asymptotic constant  of every pair $(\xi^+, \xi^-)\in F$ is equal to $0$. 
We equip  the space $F$ with 
an sc-smooth structure $(F_m)_{m\geq 0}$ where $F_m$ consists of those pairs $(\xi^+, \xi^-)\in F$ in which $\xi^\pm \in H^{2+m, \delta_m}(\R^\pm \times S^1, \R^{2n})$.  The sc-Banach  space $\wh{G}^a$  is defined as follows. 
If $a=0$,  we set $\wh{G}^0=F\oplus \{0\}$, and for  $a\neq 0$  we  define
$$\wh{G}^a=H^2(Z_a, \R^{2n})\oplus H^{2,\delta_0}(C_a, \R^{2n}).$$
Then $\wh{G}^a$ is an sc-Banach space with an sc-smooth structure for which the  level $m$ consists of pairs 
$(u, v)\in \wh{G}^a$  in which $u$ is of 
Sobolev regularity $2+m$ and $v$ of regularity $(2+m, \delta_m)$.
Recalling the projection $\wh{\pi}_a$ onto $\ker \wh{\ominus}_a$  along  $\ker \wh{\oplus}_a$ defined
by \eqref{smooth-splicing-10},   the following theorem is proved in   \cite{HWZ8.7}.
\begin{theorem}\label{sc-splicing-thm1}\mbox{}\\
\begin{itemize}
\item[{\em (1)}] For every $a\in \hb $, the total hat-gluing 
$$\wh{\boxdot}_a=( \wh{\oplus}_a,\wh{\ominus}_a):E\to \wh{G}^a \index{total hat-glued map $\wh{\boxdot}_a$}$$
is an sc-linear isomorphism. 
\item[{\em (2)}] The map 
$$\wh{\pi}: \hb \oplus F\to F,\quad  (a, (\xi^+, \xi^-))\mapsto \wh{\pi}_a(\xi^+, \xi^-)$$
is sc-smooth. Consequently, the triple $\wh{\cs}=(\wh{\pi},F,\hb )$  is an sc-smooth splicing.
\end{itemize}
\end{theorem}

\section{Implanting Gluings and Anti-gluings into a Manifold}\label{section-implant}
Working in a chart we shall implant our gluing constructions  into the symplectic manifold $(Q, \omega)$ of dimension $2n$. We consider  a piece of the Riemann surface $S$ consisting of two disjoint disks $D_x$ and $D_y$ having smooth boundaries and centered at the points $x$ and $y$ in $S$ belonging to the nodal pair $\{x, y\}$. As a reference map we take a continuous map  $u_0:S\rightarrow Q$ 
satisfying the matching condition 
$$u_0(x)=u_0(y)=q\in Q$$ and choose 
a diffeomorphic (in particular onto)  chart  
$\psi: {\bf R}(q) \rightarrow \R^{2n}$
around the point $q$ which satisfies $\psi (q)=0$. The open sets ${\bf R}_r (q)\subset {\bf R}(q)$   are defined as $\psi^{-1}(B_r(0))$ where $B_r(0)\subset \R^{2n}$ is the open ball of radius $r$ centered at the origin.  We assume that ${\bf R}(q)$ is equipped with a Riemannian metric  which over ${\bf R}_4(q)$ is  the pull-back of the Euclidean metric under the map $\psi$ and denote its exponential map by $\exp$.  Then  
$$\psi (\exp_p V(p))=\exp_{\psi (p)}(T\psi (p)\cdot V(p))=\psi (p)+T\psi (p)\cdot V(p)$$
for $p\in {\bf R}_2(q)$, where the tangent vectors $V(p)\in T_pQ$ satisfy $\abs{V(p)}_p\leq 2$.
We assume that $u_0(D_x\cup D_y)\subset {\bf R}_1(q)$. We denote by $O$ a $C^0$-neighborhood of the map $u_0$ consisting of  continuous maps $v:S\to Q$ satisfying $v(x)=v(y)$ and $v(D_x\cup D_y)\subset {\bf R}_2(q)$. We now take a smooth map $u:S\to Q$ in $O$ and choose positive and negative holomorphic polar coordinates centered at $x$ and $y$, and denoted by 
$$h_x:\R^+\times S^1\to D_x\quad \text{and} \quad  h_y:\R^-\times S^1\to D_y.$$ 
This leads to the smooth maps $u^\pm:\R^\pm \times S^1\to \R^{2n}$ defined  by 
$$u^+(s, t)=\psi\circ u\circ h_x(s, t)$$
and 
$$
u^-(s', t')=\psi\circ u\circ h_y(s', t').
$$
Such maps we have already met in the gluing constructions in the  previous section.
If $a=\abs{a}e^{-2\pi i \vartheta}$  is a  gluing parameter satisfying $0<\abs{a}<\frac{1}{2}$,  we identify  
$$
z=h_x(s, t)\in D_x, \quad 0\leq s\leq R
$$
and
$$
z'=h_y(s', t')\in D_y,\quad -R\leq s'\leq 0
$$
if $s'=s-R$ and $t'=t-\vartheta$ where $R=\varphi (\abs{a})$ is the  gluing length. 
If $a=0$, we set 
$$Z_0=D_x\cup D_y.$$
Denoting the gluing operations of the maps into $\R^{2n}$  of the previous section by $\oplus_a^0$, $\ominus^0_a$ and $\wh{\oplus}_a^0$, $\wh{\ominus}_a^0$, we define the glued map 
$$\oplus_a (u):Z_a\to Q$$
 by
$$
\psi (\oplus_a(u)[s,t] )= \oplus_a^0(u^+, u^-)[s,t]=\oplus_a^0(\psi\circ u\circ h_x, \psi\circ u\circ h_y)[s,t],
$$
for all $[s,t]\in Z_a$. The definition of the gluing operation $\oplus_a^0$ implies that 
$$\oplus_a(u)=u$$
near the boundary $\partial Z_a$ of the  glued cylinder $Z_a$.

If $\eta$ is a section of the pull-back bundle $u^*TQ$ with matching data at the nodal pair we denote its principal part by 
$$\eta (z)\in T_{u(z)}Q$$
and hence require that $\eta(x)=\eta (y)\in T_pQ$ where $p=u(x)=u(y)$. The glued 
section  $\oplus_a(\eta)$ is a section of the pull-back bundle $\oplus_a(u)^*TQ$. The vector $\oplus_a(\eta)[s,t]\in T_{\oplus_a(u)[s,t]}Q$ is defined by 
$$T\psi (\oplus_a(u)[s,t])\cdot 
\oplus_a (\eta)[s,t]=\oplus_a^0 (T\psi (u\circ h_x)\cdot \eta\circ h_x, T\psi (u\circ h_y)\cdot \eta\circ h_y)[s,t]$$
where on the right hand-side we have used the $\oplus_a^0$-gluing in $\R^{2n}$ for two vector fields.

Since the Riemannian metric  on ${\bf R}_4(q)$ is the  pull-back of the Euclidean metric in $\R^{2n}$ by the map 
$\psi$ we have the following formula.
\begin{proposition}\label{exp-formula1}
If  $\eta$ is a section of the pull-back bundle $u^*TQ$ satisfying $\eta (x)=\eta (y)$ and $\exp_u(\eta)\in O$, then 
$$
\exp_{\oplus_a(u)}(\oplus_a(\eta))=\oplus_a(\exp_u(\eta)).
$$
\end{proposition}
\begin{proof}
Denoting the principal part of the section $\eta$ by $\eta (z)\in T_{u(z)}Q$ we recall that 
$$\psi \bigl(\exp_{u(z)}(\eta (z))\bigr)=\psi (u(z))+T\psi (u(z))\cdot \eta (z).$$ Therefore,
\begin{align*}
\psi (\exp_{u\circ h_x}(\eta\circ h_x))&=\psi (u\circ h_x)+T\psi (u\circ h_x)\cdot \eta \circ h_x\\
\psi (\exp_{u\circ h_y}(\eta\circ h_y))&=\psi (u\circ h_y)+T\psi (u\circ h_y)\cdot \eta \circ h_y,
\end{align*}
and, since the  operation $\oplus_a^0$ is linear,
\begin{equation*}
\begin{split}
\psi \bigl(\oplus_a(\exp_u \eta)\bigr)&=\oplus_a^0\bigl(\psi (\exp_{u\circ h_x}(\eta\circ h_x)), \psi (\exp_{u\circ h_y}(\eta\circ h_y)) \bigr)\\
&=\oplus_a^0(\psi\circ u\circ h_x, \psi\circ u\circ h_y)\\
&\phantom{=}+
\oplus_a^0\bigl(T\psi (u\circ h_x)\cdot \eta \circ h_x, T\psi (u\circ h_y)\cdot \eta \circ h_y\bigr)\\
&=\psi (\oplus_a(u))+T\psi (\oplus_a(u))\cdot \oplus_a (\eta)\\
&=\psi (\exp_{\oplus_a(u)}\cdot \oplus_a (\eta))
\end{split}
\end{equation*}
and the proposition is proved.
\end{proof}

As above we denote by $C_a$ the infinite cylinder consisting of two half-cylinders $z=h_x(s, t)\in D_x$ and $z'=h_y(s', t')\in D_y$ in which the points with $0\leq s\leq R$ and $-R\leq s'\leq 0$ are identified if $s'=s-R$ and $t'=t-\vartheta$. Let $\eta$ be a section of the  pull-back bundle $u^*TQ$ satisfying 
$\eta(x)=\eta (y)\in T_pQ$ where $p=u(x)=u(y)$. Then the anti-glued map 
$$\ominus_a (\eta):C_a\to T_pQ$$
 is defined by 
$$T\psi (p)\cdot \ominus_a  (\eta)[s,t]=\ominus^0_a(T\psi (u\circ h_x)\cdot \eta\circ h_x, T\psi (u\circ h_y)\cdot \eta\circ h_y)[s,t].$$
Let   $J$ be  a  compatible almost complex structure on the symplectic manifold $(Q, \omega)$,  {that is,   $\omega (\cdot , J\cdot )$  is a Riemannian metric on $Q$. 
}
If $u:D_x\cup D_y\to {\bf R}_1(q)$ is a smooth map belonging to  our neighborhood $O$,  hence satisfying $u(x)=u(y)$,  and if we have at every point $z\in D_x\cup D_y$ a complex anti-linear map
$$\xi (z): (T_z(D_x\cup D_y), j(z))\to (T_{u(z)}Q, J(u(z)))$$
satisfying $\xi (x)=0=\xi (y)$,  we define the  two vector fields $\xi^\pm$ in $\R^{2n}$ by 
\begin{align*}
\xi^+(s, t)&=T\psi (u\circ h_x(s, t))\cdot \xi (h_x(s, t))\cdot \partial_sh_x (s, t)\\
\xi^-(s', t')&=T\psi (u\circ h_y(s', t'))\cdot \xi (h_y(s', t'))\cdot \partial_{s'}h_y (s', t')
\end{align*}
where $(s, t)\in\R^+\times S^1$ and $(s', t')\in \R^-\times S^1$.  {We point out  the fact that a complex anti-linear map defined on a 1-dimensional complex vector space  is determined by its action on a single  nonzero vector.}
{\begin{remark} The condition $\xi^+(x)=\xi^-(y)=0$ is natural in this context. Namely,  if we have a pair  $(u^+,u^-)$ of maps 
 which converge at the nodal points exponentially fast to the nodal image, then applying the Cauchy-Riemann operator
we obtain the pair  $(\xi^+,\xi^-)$ of vector fields which converges exponentially fast to $0$.
\end{remark}}

The hat glued map $\wh{\oplus}_a(\xi)$ then associates with the point $z=[s,t]\in Z_a$ of the glued cylinder, the   complex anti-linear map along the glued map $\oplus_a (u)$,
$$\wh{\oplus}_a(\xi)(z):(T_zZ_a, j(z))\to (T_{\oplus_a(u)(z)}Q, J(\oplus_a(u)(z)))$$
which  is, {using the previous remark about  complex anti-linearity},  defined by 
$$
T\psi (\oplus_a (u)[s, t])\cdot \wh{\oplus}_a(\xi)([s,t]){\cdot \frac{\partial}{\partial s}}:=\wh{\oplus}^0_a(\xi^+, \xi^-)[s, t].
$$
The hat anti-glued  map $\wh{\ominus}(\xi)$ associates with every point $z=[s,t]$ of the infinite cylinder $C_a$ the  complex anti-linear map 
$$\wh{\ominus}_a(\xi)([s, t]):(T_zC_a, j)\to (T_{p}Q, J(p))$$
which is,  defined by 
$$
T\psi (p)\cdot \wh{\ominus}_a(\xi)(z) {\cdot  \frac{\partial }{\partial s}}:=\wh{\ominus}^0_a (\xi^+, \xi^-)[s, t].
$$
We have considered  above the smooth map $u:D_x\cup D_y\to Q$ satisfying $u(x)=u(y)=p$ and the smooth section $\eta$ along the map $u$ of the  pull-back bundle $u^*TQ$ satisfying $\eta (x)=\eta (y)\in T_pQ$. Moreover, we have introduced the glued map $\oplus_a (u):Z_a\to Q$ and the glued section $\oplus_a (\eta):Z_a\to T_{\oplus_a (u)}Q$ of the pull-back bundle $\oplus_a(u)^*TQ$. We also recall that with every $z=[s,t]\in Z_a$ we have associated the hat-glued complex anti-linear map $\wh{\oplus}_a(\xi):(T_zZ_a, j(z))\to  (T_{\oplus_a (u)(z)}Q, J(\oplus_a (u)(z))).$ 

Finally,  the splicing projections  $\pi_a^0$ and $\wh{\pi}^0_a$ can also be implanted as follows. In the case of $\pi_a^0$ we take a section $\eta$ 
along $u$. This defines,  using holomorphic polar coordinates and  a chart $\psi$ for the manifold $Q$, the pair $(\eta^+,\eta^-)$ of sections along the pair $(u^+,u^-)$.
For the  gluing parameter $a$ the projection $\xi=\pi_a(\eta)$ is defined by its representative $(\xi^+,\xi^-)$ obtained as the unique solution of the equations 
$$
\oplus_a^0(\xi^+,\xi^-)=\oplus^0_a(\eta^+,\eta^-)\ \ \text{and}\ \ \ \ominus_a^0(\xi^+,\xi^-)=0.
$$
If $D_x\cup D_y\subset S$ for a Riemann surface $S$ and if $\eta$ is a section along the map  $u:S\rightarrow Q$,
then $\pi_a(\eta)$ is well-defined acting as the identity over the complement of $D_x\cup D_y$.
By implanting the splicing projection associated to a single nodal pair $\{x,y\}$ we will always assume that the splicing projection
will be extended by the identity in the complement of the relevant disks of the small disk structure.
If we deal with several nodal pairs $\{x,y\}\in D$ we obtain for each $\{x,y\}$ a splicing projection $\pi_{a_{\{x,y\}}}^{\{x,y\}}$ as just described.
Since the disks of the small disk structures are mutually disjoint any  two of these projections commute. By taking the composition
of all of them (whatever the order) we obtain a well-defined projection $\pi_a$, where $a={(a_{\{x,y\}})}_{\{x,y\}\in D}$. This is the splicing
projection associated to the tuple $a$ of gluing parameters. 

\section{More Sc-smoothness Results.}

We next apply the following criterion for the sc-smoothness of a map $f:E\to F$ between  sc-Banach spaces which is  formulated in terms of  classical derivatives.

\begin{proposition}\label{prop2.11}
Let $E$ and $F$ be two sc-Banach spaces and let $U\subset C\subset E$ be an open set in a partial quadrant $C$. We assume
that the map $f:U\rightarrow F$ is sc$^0$-continuous  and that the induced maps 
$f:U_{m+k}\rightarrow F_m$ are of class  $C^{k+1}(U_{m+k}, F_m)$ for every $m\geq 0$ and $k\geq 0$.
Then  the map $f:U\rightarrow E$ is sc-smooth.
\end{proposition}
{Recalling that  an $\ssc^0$-map $f:U\rightarrow F$ is of class $\ssc^\infty$ if it is of class $\ssc^{k+1}$ for all $k\geq 0$,  Proposition \ref{prop2.11} is a consequence of Proposition 2.4 in \cite{HWZ8.7}.}

We shall make use of Proposition \ref{prop2.11} in the special  situations, in which $F$ is equal to some Euclidean space equipped with the constant sc-structure.  First we shall   carry out a  transversal constraint construction used later on.

We consider the closed disk  $D$ in ${\mathbb C}$ and  a $C^1$-embedding $u:D\rightarrow {\mathbb R}^{2n}$  satisfying $u(0)=0$. Let $H\subset \R^{2n}$ be the
orthogonal complement of the image set  $Tu(0)({\mathbb C})$ so that $\R^{2n}=Tu(0)({\mathbb C})\oplus H$. Let us note that,
in fact, we only need $H$ to be a complement. By means of the implicit function theorem we find an open neighborhood $O$ of the map $u$ in $C^1(D, \R^{2n})$ having the following properties for $v\in O$. 
\begin{itemize}
\item[$\bullet$] The set $v^{-1}(H)$ consists of precisely one point $z_v$
and  $z_v$ belongs to  the open disk $B_{\frac{1}{2}}\subset \C$ of radius $\frac{1}{2}$ and centered at the origin.
\item[$\bullet$] The image point $v(z_v)$ lies in the unit-ball $B^{2n}(1)$ of ${\mathbb
R}^{2n}$.
\item[$\bullet$] The image set $Tv(z_v)({\mathbb C})$ is transversal to $H$ in $\R^{2n}$.
\item[$\bullet$]  The map $\beta:O\to \hb$, defined by $\beta (v)=z_v$, is of class $C^1(D, \R^{2n})$ and if we restrict the map $\beta$ to the space $ O\cap C^m(D, \R^{2n})$, then the map 
$\beta :O\cap C^m(D, \R^{2n})\to \hb$ is of class $C^m$.
\end{itemize}
We  now introduce the sequence of Sobolev spaces $E_m=H^{3+m}(D, \R^{2n})$ for $m\geq 0$ and set $E=E_0$, which is an sc-smooth Banach space equipped  with the filtration $(E_m)_{m\geq 0}$. 
The set $O\cap E$ is an open subset of $E$. In view of  the Sobolev embedding theorem,  there is a continuous {linear}  embedding $E_m\rightarrow C^{m+1}$ and consequently the map $\beta  (v)=z_v$ defines an $\ssc^0$-continuous map $\beta:E\cap O\to \R^{2}$ whose induced maps $E_m\cap O\to \R^{2}$ are of class $C^{m+1}$ for every $m\geq 0$. In view of  Proposition \ref{prop2.11} we obtain the following result.
\begin{proposition}\label{n-prop2.24}
The map $\beta:O\cap E\rightarrow B_{\frac{1}{2}}$ defined above by $\beta (v)=z_v$ is 
sc-smooth.
\end{proposition}

The previous transversal constraint construction can easily be  implanted into a manifold.
We consider the Riemann surface $(S,j)$  and the $C^{\infty}$-map 
$u:S\rightarrow Q$  into the  symplectic manifold
$(Q,\omega)$  having the property that at some point $z\in S$ the linearized map $Tu(z)$   is injective. We set $u(z)=q\in Q$. Then we find a disk like neighborhood $D_z$ of $z$ having a smooth boundary, a biholomorphic map 
$h:(D,0)\rightarrow (D_z,z)$  where, as above, $D\subset \C$ is the  closed unit disk, and we take  a diffeomorphic chart 
$\psi: ({\bf R}(q), q)\to ({\mathbb R}^{2n},0)$ around the point $q$. We can assume that 
$u(D_z)\subset {\bf R}_1(q)=\psi^{-1}(B^{2n}(1))$ and that $T\psi  (q)$ maps the image set $Tu(z)(\R^{2})$ onto $\R^{2}\times \{0\}$ in $\R^{2n}$.    We assume that ${\bf R}(q)$ is equipped with the Riemannian metric  which over ${\bf R}_4(q)=\psi^{-1}(B^{2n}(4))$ is  the pull-back of the  Euclidean metric of $\R^{2n}$ by means of the map $\psi$. 
We abbreviate $H=\{0\}\times \R^{2n-2}$. Then the composition 
$\wh{u}:D\to \R^{2n}$,  defined  by 
$$
\wh{u}= \psi\circ u\circ h
$$
is smooth, satisfies $\wh{u}(0)=0$, and $H$ is transversal to $T\wh{u}(0)(\R^2)$ in $\R^{2n}$. Hence by the above discussion there is 
a $C^1$-neighborhood $U(\wh{u})$ in $C^1(D, \R^{2n})$ of the  smooth map $\wh{u}$ so that every map $\wh{v}\in U(\wh{u})$ meets the properties listed above.

Hence, by Proposition \ref{n-prop2.24}, the map $\beta:U(\wh {u})\cap E\to \hb$,  defined by 
$\beta (\wh{v})=z_v$, is sc-smooth.  Consequently, defining the submanifold $Q_0$ of $Q$ by 
$$Q_0=\psi^{-1}(H\cap B^{2n}(2)),$$
we find a $C^1$-neighborhood $U(u)$ in $C^1(D_z, Q)$ of the smooth map $u:D_z\to Q$ such that if $v\in U(u)$,  then $v(D_z)$ intersects $Q_0$ in precisely one point $q_v$,
$$q_v\in v(D_z)\cap Q_0$$
and the intersection is transversal. Moreover, the map $U(u)\to Q_0$, associating with $v\in U(u)$ the intersection point $q_v$, is sc-smooth if  restricted to the sc-structure defined by the nested sequence $H^{3+m}$, $m\geq 0$. 

Later on we shall refer to the map $v\mapsto q_v\in Q_0$ as {``the  transversal constraint $Q_0$''},  and we shall refer to above neighborhood $U(u)$ as to the {\bf  transversal constraint construction}  associated with the smooth map $u$ and the chart $\psi$.\index{transversal constraint construction}

The next result is concerned with the action of parameterized diffeomorphisms. Let $D$ be the closed unit disk in ${\mathbb C}$ and
assume $V$ is an open subset in some ${\mathbb R}^N$. We assume that we are given a smooth family $v\mapsto \phi_v$, $v\in V$ of diffeomorphisms of $D$ having  compact supports  contained in the interior of $D$.
We take  the sc-Banach space $E$ with $E_m=H^{3+m}(D,{\mathbb R}^m)$ and  consider the composition
$$
\Phi:V\oplus E\rightarrow E, \quad (v,u)\mapsto u\circ \phi_v.
$$
We would like to point out that the map $\Phi$ is far from being smooth in the classical sense. 
{In contrast,  we have demonstrated   in \cite{HWZ8.7} (Theorem 1.26)   the following  result.}
\begin{theorem}\label{action-diff}
The map $\Phi:V\oplus E\to E$, $\Phi (v, u)=u\circ \phi_v$,  is sc-smooth.
\end{theorem}
\begin{remark}
{We point out that in the statement of Theorem 1.26 in \cite{HWZ8.7} the crucial assumption is missing. The correct statement of the theorem requires, in addition to the assumption that the family of maps 
$v\rightarrow \phi_v$ is  smooth,  that the maps $\phi_v:D\rightarrow {\mathbb C}$ are smooth embeddings.}
\end{remark} 
{One can deduce  many  corollaries  from this  result. To describe such a corollary,  we assume  that   $S$ and $S'$ are pieces of Riemann surfaces containing the distinguished points $z$ and $z'$.  We denote by $O(z)$ and $O(z')$  open neighborhoods of $z$ and $z'$ in $S$ and $S'$, respectively, and   assume that there exists   a smooth map 
$$
V\times O(z')\rightarrow S, \quad (v,s)\mapsto \phi_v(s)
$$
into 
a neighborhood  of  $z$ satisfying
$$
\phi_v(O(z'))\subset O(z)
$$
for all $v\in V$. {Moreover,  we assume that $\phi_v:O(z')\rightarrow S$ is a smooth embedding
for every $v$.} We consider the Banach space $E$ of maps $u$ of Sobolev class $H^3$ defined on ${O(z)}$ and having their images  in some ${\mathbb R}^n$.  We equip $E$   with the usual sc-structure defined by the sequence $H^{3+m}$ for  $m\geq 0$. We  let  $F$  be the sc-Banach space of the same type of maps defined on some compact subdomain $K$ of $O(z')$  having a smooth boundary. Then the map 
$$\Psi:V\oplus E\to  F,$$
defined by 
$$
(v,u)\mapsto  (u\circ\phi_v)\vert K, 
$$
is an sc-smooth map.}

Later on we will need  the following  version of this result in the case of  noded surfaces.
We assume that $D_x$ and $D_y$ as well as $D_{x'}'$ and $D_{y'}'$ are disk-like Riemann surface 
with smooth boundaries and associated with the  nodal pairs $\{x,y\}$ and $\{x',y'\}$. Using the exponential gluing profile  and proceeding as in Section \ref{dm-subsect}, 
we construct  the glued surfaces $Z_a$ and $Z_{b}'$. For $h>0$ and $a$ (or $b$) small enough we
have  defined the sub-cylinders $Z_a(-h)$ and $Z_b'(-h)$ in Section \ref{dm-subsect} and recall that $Z_0(-h)=D_x(-h)\cup D_y(-h)$ if $a=0$, 
and similarly for $Z_0'$.

Let  $V$ be an open neighborhood of $0$ in some $R^N$ and let $D_{\varepsilon}$  be the closed  disk in $\C$ centered at the origin and of radius $\varepsilon$.  We assume that the  following data {are} 
 given:
\begin{itemize}
\item[(1)]  A smooth map $(a,v)\mapsto  b(a,v)\in {\mathbb C}$ defined for $v\in V$ and  $a\in D_{\varepsilon}$  with  $\varepsilon$  sufficiently small,  
and satisfying  $b(0,v)=0$ for all $v\in V$.
\item[(2)] For sufficiently large $h>0$, 
a core-smooth  family of  holomorphic embeddings
$$
\phi_{(a,v)}:Z_a(-h)\rightarrow Z_{b(a,v)}'
$$
parameterized by $(a,v)\in D_{\varepsilon}\oplus V$.
\item[(3)]  There exists  $H>0$  such  that 
$$
Z'_{b(a,v)}(-H)\subset \phi_{(a,v)}(Z_a(-h))
$$
for all $(a, v)\in  D_{\varepsilon}\oplus V$ where  $\varepsilon$  is sufficiently small.
\end{itemize}
We take two smooth cut-off functions $\beta$ and $\beta'$ as in Section \ref{gluinganti-sect} and use them,  together with  the exponential gluing profile,  to define glued maps on the two finite cylinders $Z_a$ and $Z_b'$. Then we  consider the sc-Banach space $E'$ of maps $\xi:D_{x'}'\cup D_{y'}'\rightarrow {\mathbb R}^{2n}$ of class $(3,\delta_0)$  {satisfying $\xi (x')=\xi (y')$}.

The sc-structure is such that the level m corresponds to the regularity $(m + 3, \delta_m)$, where  the strictly increasing sequence of weights $(\delta_m)_{m\geq 0}$ satisfies $0<\delta_m <2\pi$.
Similarly,  we introduce the sc-Banach space $E$ consisting of maps $\eta: D_x(-h)\cup D_y(-h)$ of regularity  $(3,\delta_0)$ whose level $m$ consists of maps of the same regularity $(m+3,\delta_m)$ 
{and satisfying $\eta (x)=\eta (y)$}. With these spaces we define the map 
$$\Phi:D_{\varepsilon}\oplus V\oplus E'\to E,\quad (a,v,\xi)\mapsto  \eta$$
where the map $\eta\in E$ is defined, using Theorem \ref{sc-splicing-thm},  as the unique solution of the  two equations 
$$
\oplus_a(\eta) = \oplus_{b(a,v)}'(\xi)\circ \phi_{(a,v)}\quad \text{and}\quad  \ominus_a(\eta)=0.
$$ 
Here the prime on $ \oplus'$ and $\ominus'$ indicates that we use  the cut-off functions $\beta'$ for the gluing.
The following theorem which will be employed later when dealing with sc-smoothness issues of transition maps.
\begin{theorem}\label{theorem_neck}
The map $\Phi$ introduced above  is sc-smooth.
\end{theorem}
The result is proved  Appendix \ref{QWE}.

{\begin{remark} 
There is an analogous  result for the sc-Banach space set-up, in which  the regularity levels are given by $(k+2,\delta_k)$, 
the asymptotic constants vanish,  and the hat-gluings are used.  This will be relevant in  the construction of 
strong bundles later on.
\end{remark}
Finally we  shall need the following  result which can also be reduced to Theorem 1.31 in \cite{HWZ8.7}. 
In order to formulate it, we
assume that $((D_x,D_y),(x,y))$ as well as $((D_{x'}',D_{y'}'),(x',y'))$ are disk pairs so that we obtain via gluing the associated
cylinders $Z_a$ and $Z_{a'}'$.  We assume that $E$ is the sc-Banach space as described above associated to the first pair
and $E'$ to the second.  We assume that $(a_k)$ and $(a_k')$ are small gluing parameters converging to $0$ and $$\phi_k:Z_{a_k}\rightarrow Z_{a_k'}'$$
is a sequence of holomorphic embeddings having the following properties.  The restrictions of the maps to 
a fixed annulus-type neighborhood of the boundary of the disc $D_x$ converge, viewed as the maps into the discs $D_{x'}'$, in the $C^\infty$-sense. Similarly, the maps $\phi_k$ restricted to a fixed annulus-type neighborhood of the boundary of $D_y$ converge as maps into $D'_{y'}$ in the $C^\infty$-sense.
Then we consider a sequence $(\xi_k)$ in $E'$ converging on level $0$ to some $\xi_0=(\xi^+,\xi^-)\in E'$ and define the maps 
$\eta_k\in E$ as the unique solutions of the two equations
$$
\oplus_{a_k}(\eta_k) =(\oplus_{a_k'}'(\xi_k))\circ \phi_k\ \ \text{and}\ \ \ \ominus_{a_k}(\eta_k)=0.
$$
\begin{proposition}\label{gotham}
Under the assumptions stated above the sequence $(\phi_k)$ converges in $C^\infty_{loc}$ in finite distance to the left and the  right boundary
to holomorphic embeddings $\phi^+:D_x\rightarrow D_{x'}'$ and $\phi^-:D_y\rightarrow D_{y'}'$ satisfying  $\phi^+(x)=x'$ and $\phi^-(y)=y'$. 
The sequence $(\eta_k)\subset E$ converges on level $0$ to  the map  $\eta_0=(\eta^+,\eta^-)\in E$ given by 
$$
\eta^+ = \xi^+\circ \phi^+\ \ \text{and}\ \ \eta^-=\xi^-\circ\phi^-.
$$
\end{proposition}
\begin{remark}
The proof is lengthy, but  uses only arguments and constructions occurring already in  the proof of Theorem \ref{theorem_neck}.
We sketch the proof and leave the details to the reader.
The first step is to show that,  without loss of generality,  the maps $\phi_k$ are biholomorphic maps $Z_{a_k}\rightarrow Z_{a_k'}'$,
where the domain is equipped with the standard {almost complex} structure,  and the target with a family of almost complex structures $(j_k)$, 
which on suitable sub-cylinders  $Z_{a_k}'(-h)$,  are all  standard structures and converge near the boundaries in the $C^\infty$-sense.
In the next step one shows that,  without loss of generality,  one may assume that the  maps $\phi_k$ map the   distinguished
boundary points on $Z_{a_k}$ to the distinguished boundary points on $Z_{a_k'}'$. Then the proof  is indeed reduced 
to Theorem 1.31 in \cite{HWZ8.7}. We point out that the sequence $(\phi_k)$ converges in a strong sense to $(\phi^+,\phi^-)$. This  is 
discussed in detail, in \cite{HWZ8.7}, Theorem 1.46. 
\end{remark}}


%
%
%

\chapter{The Polyfold Structures }\label{four}
This chapter is devoted to the construction of the polyfold structures on the space 
$Z=Z^{3,\delta_0}(Q, \omega)$ of stable curves into the symplectic manifold $(Q, \omega)$,  and on the bundle $W\to Z$ introduced in Section \ref{sect1.2}.  
In particular, we shall give the proofs of the Theorems \ref{th-top}, \ref{pfstructure} and
 \ref{main1.10} announced in the introduction. The polyfold construction is based on the concept of  good uniformizing families of stable curves introduced next.
 \section{Good Uniformizing Families of Stable Curves}\label{xA}
Fixing  a number $\delta_0\in (0,2\pi) $  we consider the 
$(3,\delta_0)$-stable map
$$
\alpha=(S,j,M,D,u).
$$

We recall that $(S, j, M, D)$ is a connected nodal Riemann surface with ordered marked points, and $u:S\to Q$ is a map into the closed symplectic manifold $(Q,\omega)$ which is of class $(3,\delta_0)$ near the nodal points in $\abs{D}$ and of class $H^3_{\loc}$ around all the other points of $S$. Moreover, $u(x)=u(y)$ for every nodal pair $\{x, y\}\in D$. We shall denote the set of such mappings $u:S\to Q$ by 
$$H^{3,\delta_0}(S, Q).$$
If $C$ is an unstable domain component of $S$, the map $u:C\to Q$ satisfies the stability condition 
$\int_Cu^*\omega >0$.  Two stable maps $(S,j,M,D,u)$ and $(S',j',M',D',u')$ are isomorphic or equivalent, if there exists an isomorphism $\phi:(S,j,M,D)\to (S',j',M',D')$ between the underlying noded surfaces satisfying $u'\circ \phi =u$. An equivalence class is called a stable curve. We shall denote the space of equivalence classes of  stable maps of class $(3, \delta_0)$ by 
$$Z=Z^{3,\delta_0}(Q,\omega ).$$

A consequence of the stability condition $\int_Cu^*\omega >0$ is the finiteness of the automorphism group  $G$ of the stable map $(S, j, M, D, u)$.

The underlying  connected nodal  Riemann surface $(S, j, M, D)$ is not necessarily stable. In order to achieve stability, we shall add more marked points.
\begin{definition}\label{stab-f}
A finite set
$\si$ of points in $S$ is called a {\bf stabilization} of $\alpha=(S,j,M,D,u)$  if  the following holds true.\index{stabilization}
\begin{itemize}
\item[(1)] The set $\si$ lies in the complement of $M\cup |D|$.
\item[(2)]  {Every element $g$ of the automorphism group $G$ of $\alpha$  maps  the set $\si$ onto itself.}
\item[(3)] If we denote by $M^\ast$ the un-ordered set $M\cup\si$, then the nodal Riemann surface $(S,j,M^\ast,D)$ is stable.
\item[(4)] If $u(z)=u(z')$ for two points in $\si$, then there exists an automorphism  $g$ in $G$ satisfying $g(z)=z'$.
\item[(5)] The image $u(\si)$ does not intersect $u(|D|\cup M)$.
\item[(6)]  {At  the points $z\in \si$, the tangent map $Tu(z)$ is  injective, the $2$-form $u^\ast\omega (z)$ is non-degenerate and determines on $T_zS$ the same orientation  as the almost complex structure $j(z)$ does.}
\end{itemize}
\end{definition}
The automorphism group $G^\ast$ of  the nodal Riemann  surface $(S,j,M\cup\si,D)$  is finite and contains the automorphism $G$ of  the stable map  $\alpha$.
{The latter property about $u^\ast\omega$ will be used solely in the  discussion of the  orientability of the Cauchy-Riemann section later on.  Here we adopt the convention  that the orientation of the tangent space $T_zS$ determined by the non-degenerate $2$-form $u^\ast\omega (z)$ is defined by an {ordered} basis $\{e_1, e_2\}$ satisfying $u^\ast\omega(z)(e_1,e_2)>0$ and the orientation determined by the almost complex structure $j(z)$ is defined by an  {ordered}  basis $\{v, j(z)v\}$ in $T_zS$.}

\begin{lemma}\label{ax}
{ The stable map $\alpha=(S,j,M,D,u)$ possesses a stabilization $\si$. 
}
\end{lemma}

\begin{proof}
We choose a domain component $C$ of $S$ which together with its special points $C\cap (M\cup |D|)$ is unstable.
 Since $\alpha$ is stable, we know that $\int_C u^\ast\omega>0$. Hence we find 
 a point $z\in C\setminus (M\cup |D|)$ at which   $Tu(z)$ is injective,  {$u^\ast\omega (z) $ is non-degenerate and   determines on $T_zS$ the same orientation as $j(z)$ does.} By moving $z$ slightly we can keep this property, and  can achieve that, in addition,  
 $$
 u(z)\not\in u(M\cup |D|).
 $$
 Now we use the automorphism group $G$ to move $z$ around hitting  possibly  other domain components.
 If $\si_1$ is  the collection of points of the orbit $\{g(z)\ |\ g\in G\}$, then  $\si_1$ satisfies all properties of a stabilization with the exception that $(S,j,M\cup \si_1,D)$ might not be stable yet. If the latter still has unstable domain components we can pick a point $z'\in S$ in such a component in the complement of $M\cup\si_1\cup |D|$ so that $Tu(z')$ is injective, {
 $u^\ast\omega (z')$ is non-degenerate and   determines on $T_{z'}S$ the same orientation as $j(z')$ does.} Moreover, 
$u(z')$ is not contained in $u(|D|\cup M\cup\si_1)$. Then the set $\si_2=\si_1\cup \{g(z')\ |\ g\in G\}$ satisfies all properties of a stabilization with the exception that $(S,j,M\cup \si_2,D)$ might not be stable yet.
 At every step of this procedure an unstable domain component
 obtains {at least one}  additional point. Since we only have a finite number of domain components  and since we have to add at most three points to an unstable domain component to obtain stability, it is clear that after a finite number of steps we arrive at  the desired stabilization $\si$  of $\alpha$.
\end{proof}

Let $\alpha=(S,j,M,D,u)$ be a stable map equipped with the stabilization $\si$. 
Recalling that $u$ maps $S$ into the symplectic manifold $Q$, we label the images of the nodal points in $\abs{D}$ and of the points in the stabilization $\si$  by
$u(\abs{D})=\{w_1,\ldots, w_k\}$ and $u(\si)=\{w_{k+1},\ldots ,w_{k+l}\}$,  {respectively}. By definition of the stabilization, $u(|D|)\cap u(\si)=\emptyset$.
For every $i\in \{1,\ldots, k+l\}$ we fix a smooth bijective chart
$$
\psi_i:({\bf R}(w_i),w_i)\rightarrow ({\mathbb R}^{2n},0)
$$
so that the closures of the domains are mutually disjoint. We denote
by ${\bf R}_r(w_i)$ the preimage of the $r$-ball $B_r(0)$ around $0$ under the chart map $\psi_i$.
Next we fix a small disk structure ${\bf D}$,  and, in addition, closed disks $D_z$ with
smooth boundaries  centered at the points $z$ in $\si$ so that all disks are
mutually disjoint and so that the following holds. \\[0.3ex]

\noindent $\boldsymbol{(\ast)}$ For every $z\in |D|\cup \si$ the image under $u$ of the
disk $D_z$ is contained in ${\bf R}_1(w_i)$, where $w_i=u(z)$. Moreover the union $\bigcup_{z\in \Sigma} D_z$ is invariant under $G^\ast$.\\
\\
We note that invariance under $G$ would actually be sufficient for the later constructions.
We choose  a Riemann metric $g$ on the manifold $Q$ which,  over  the open neighborhood ${\bf R}_4(w_i)$, is the
pull-back of the standard metric on ${\mathbb R}^{2n}$ under the map  $\psi_i.$
At this point we consider the stable noded Riemann surface $(S,j,M\cup\si,D)$ in which the marked points $M\cup \si$ are not ordered.  We denote its  finite automorphism  group by $G^\ast$. Using the exponential gluing
profile,  we use our small disk structure and take a good {complex} deformation
$v\mapsto  j(v)$ of $j$ so that for a suitable $G^\ast$-invariant
open neighborhood $O$ of $(0,0)\in N\oplus E$ we have the  good uniformizing
family
$$
(a,v)\mapsto  (S_a,j(a,v),(M\cup\si)_a,D_a)
$$
possessing  all the properties listed in Definition \ref{citiview}.

We denote by $U(u)$ a $C^0$-open neighborhood of the map $u$ in the space of continuous maps from $S$ into $Q$, which is so small that the images of the discs
$D_z$ under the maps belonging to $U(u)$ lie  in ${\bf R}_2(w_i)$ for $w_i=u(z)$.  Hence we can define
a family
$$
(a,v,u')\rightarrow (S_a,j(a,v),M_a,D_a,\oplus_a(u')),
$$
where $(a,v)$ belongs to the $G^\ast$-invariant open neighborhood
$O$ of $(0,0)\in N\times E$ and  where the continuous map $u':S\to Q$ belongs to $U(u)$. Let us fix for every $z\in
\si$ a subdisk $SD_z$ which is contained in the interior of $D_z$
and also has smooth boundary so that their union is invariant under
the $G$-action. There is no loss of generality to assume that the subdisks are concentric, i.e. of the form
$$
SD_z = D_z(-h)=\{\zeta \in D_z\vert \, \text{$ \zeta =z$ or $\zeta=h_x(s, t)$ for $s\geq h$}\},$$
where $h>0$ and where $h_x$ are positive holomorphic polar coordinates on $D_z$ centered at $z$.
We also fix subdisks having  the same invariance property for the  discs $D_x$ and $D_y$ associated with the nodal points $\{x, y\}$. Recall that given the  gluing parameter $a={(a_{\{x, y\}})}_{\{x, y\}\in D}$ we have introduced the union
$$
Z_a=\bigcup_{\{x,y\}\in D} Z_{a_{\{x,y\}}}^{\{x,y\}}
$$
of the cylinders $Z_{a_{\{x,y\}}}^{\{x,y\}}$ which connect the boundaries of the discs $D_x$ and $D_y$ if $a_{\{x, y\}}\neq 0$ and which are equal to $Z_0^{\{x, y\}}=D_x\cup D_y$ if 
$a_{\{x, y\}}=0$. Recalling the subcylinders  $SZ_{a_{\{x,y\}}}^{\{x,y\}}=Z_{a_{\{x,y\}}}^{\{x,y\}}(-h)$ from section \ref{dm-subsect}, we have abbreviated  
$$
SZ_a=Z_a(-h)=\bigcup_{\{x,y\}\in D} Z_{a_{\{x,y\}}}^{\{x,y\}}(-h).
$$
 {We recall that $G$ is the automorphism group of the stable map $\alpha=(S, j, M, D, u)$, not necessarily having a stable domain $(S, j, M, D)$, but equipped with a stabilization $\si$. Moreover,  $G^\ast$ is   the  group of automorphisms  of the  stable nodal  Riemann  surface $(S,j,M\cup\si,D)$, where
$M\cup\si$ is viewed as an unordered set. The group  $G^*$ is finite and, by the properties of a stabilization, contains $G$.} 

{The behavior of isomorphisms  near the nodes  is described in the following proposition}.
\begin{proposition}\label{ay}
We consider the stable map  $(S, j, M, D, u)$ and recall the above definitions. Then given $h>0$, there exists  a $G^\ast$-invariant
open neighborhood $O'\subset O$ of $(0,0)\in N\oplus E$ and a $G$-invariant $C^0$-neighbor\-hood $U'(u)\subset U(u)$ of the map  $u$ so that the following holds. If $u',u''\in U'(u)\cap H^{3,\delta_0}(S, Q)$ are  maps from $S$ to $Q$ and 
$(a,v),(b,w)\in O'$ and if 
$$
\phi:(S_a,j(a,v),M_a,D_a,\oplus_a(u'))\rightarrow
(S_b,j(b,w),M_b,D_b,\oplus_b(u''))
$$
is an isomorphism, then there exists an automorphism $g$ {in the automorphism group $G$ of the  original stable map $\alpha=(S, j, M, D, u)$} satisfying 
$$
\phi(Z_{a_{\{x,y\}}}^{\{x,y\}}(-h))\subset Z_{b_{\{g(x),g(y)\}})}^{\{g(x),g(y)\}}
$$
at  all nodal pairs $\{x,y\}\in D$. In addition,   if  $z\in \si$,  then
$$
\phi(D_z(-h))\subset D_{g(z)}.
$$
\end{proposition}
\begin{proof}
Arguing indirectly we may assume that we have two sequences
$u_k',u_k'' \rightarrow u$ of maps  in $U'(u)\cap H^{3,\delta_0}(S, Q)$ converging in $C^0$,  sequences $(a_k,v_k)\rightarrow (0,0)\in
O$ and $(b_k,w_k)\rightarrow (0,0)\in O$, and a sequence of isomorphisms
\begin{eqnarray*}
&\phi_k:(S_{a_k},j(a_k,v_k),M_{a_k},D_{a_k},\oplus_{a_k}(u'_k))\rightarrow\ \ \ \ &\\
&\ \ \ \ \ \ \ \ \ \ \ \ (S_{b_k},j(b_k,w_k),M_{b_k},D_{b_k},\oplus_{b_k}(u''_k))&
\end{eqnarray*}
which violates the conclusion of the proposition.
The stability assumption on the stable maps $(S, j, M, D, u)$ excludes bubbling off in  the sequence $(\phi_k)$ of mappings and one concludes by Gromov compactness that any subsequence of $(\phi_k)$ has a subsequence converging in $C^{\infty}_{\text{loc}}$ away from the nodes and the neck to some isomorphism 
$$\phi_0 :(S, j, M, D, u)\to (S, j, M, D, u).$$
This fact is a  consequence of the following lemma which is proved  in Appendix  \ref{section5.1}.
\begin{lemma}\label{GROMOVCONV}
We consider the stable map  $(S, j, M, D, u)$ and two sequences $u_k, u_k'\in H^{3,\delta_0}(S, Q)$  converging  to $u$ in $C^0$. We assume 
that $(a_k,v_k)\rightarrow (0,0)\in O$ and $(b_k,w_k)\rightarrow (0,0)\in O$ and  assume that  
\begin{eqnarray*}
&\phi_k:(S_{a_k},j(a_k,v_k),M_{a_k},D_{a_k},\oplus_{a_k}(u_k))\rightarrow\ \ \ \ \ \ \ &\\
&\ \ \ \ \ \ \ \ (S_{b_k},j(b_k,w_k),M_{b_k},D_{b_k},\oplus_{b_k}(u'_k))&
\end{eqnarray*}
is a sequence of isomorphisms. Then there is a subsequence of $(\phi_k)$ which converges in $C^{\infty}_{\text{loc}}$ away from the nodes to an automorphism $\phi_0$ of $(S, j, M, D, u)$. 
\end{lemma}
Consequently, $\phi_0=g\in G$ and hence $\phi_0(D_z(-h))=D_{g(z)}(-h)\subset D_{g(z)}$ for $z\in \Sigma$ and $\phi_0(D_x(-h) \cup D_y(-h))\subset D_{g(x)}\cup D_{g(y)}$ for the nodal pairs  $\{x, y\}$. After passing to a subsequence, we see that for $k$ large enough
$\phi_k (D_z(-h))\subset D_{g(z)}$, and 
$\phi_k (D_x(-h)\cup D_y(-h))\subset D_{g(x)}\cup D_{g(y)}$ contradicting the choice of the sequence and implying the result.
\end{proof}

The same arguments prove also the next proposition.

\begin{proposition}\label{az}
Given $h>0$ there exists $H>0$ and a $G^\ast$-invariant open
neighborhood $O''\subset O'$ of $0$ and a $C^0$-neighborhood $U''(u)\subset U(u)$  of the map 
$u$ so that the following holds. If $u', u''\in U''(u)$, $(a,v),(b,w)\in
O''$ and if 
$$
\phi:(S_a,j(a,v),M_a,D_a,\oplus_a(u'))\rightarrow
(S_b,j(b,w),M_b,D_b,\oplus_b(u''))
$$
is an isomorphism,  then there exists an automorphism $g\in G$ satisfying 
$$
\phi(D_z(-H))\subset D_{g(z)}(-h)
$$
for all $z\in \si$, and 
$$\phi(Z^{\{x,y\}}_{a_{\{x,y\}}}(-H))\subset
Z^{\{g(x),g(y)\}}_{b_{\{g(x),g(y)\}}}(-h))
$$
for every nodal pair $\{x,y\}\in D$.
\end{proposition}
In order  to define the notion of a good uniformizing family for $Z=Z^{3,\delta_0}$, we start with the notion of good data.

\begin{definition}[{\bf Good Data}]\label{GD}\index{good data (definition)}
Let $\alpha=(S,j,M,D,u)$ be a stable map  representing the class $[\alpha]\in Z$ {where $(S, j, M, D)$ is a noded, not necessarily stable, Riemann surface}, and let $G$ be the automorphism group of $\alpha$. 
{\bf Good data}  centered at $\alpha$  are the following data.
\begin{itemize}
\item[(1)] A stabilization $\si$ of $\alpha$ and a good uniformizing family
$$
(a,v)\mapsto  (S_a,j(a,v),(M\cup\si)_a,D_a),\ \ (a,v)\in O,
$$
where the gluing is associated with  a small disk structure ${\bf D}$ around the points in $|D|$. In addition,  disks with smooth boundaries have been fixed around the points in $\si$, so that all the disks are mutually disjoint. Further,  the union of the disks is invariant under the $G^\ast$-action, where $G^\ast$ is the automorphism group of the stable noded   Riemann  surface $(S,j,M\cup\si,D)$. Here $M\cup\si$ is viewed as an un-ordered set.
\item[(2)]  Abbreviating the set $W=u(|D|\cup\si)$, there is  a collection of bijective charts around every point $w\in W$, 
    $$
    \psi_w:({\bf R}(w),w)\rightarrow ({\mathbb R}^{2n},0),
     $$
     so that the closures of the domains are mutually disjoint. Moreover, 
     for every $z\in |D|\cup\si$ the image $u(D_z)$ of the disk $D_z$ is contained in 
     ${\bf R}_1(u(z))$.
\item[(3)] The exponential map $\exp$ on the manifold $Q$ is associated with  a Riemannian metric $g$, which on the open sets ${\bf R}_4(w)$ is the pull-back of the standard metric on ${\mathbb R}^{2n}$ by $\psi_w$. Let $\wt{\mathcal O}$ be an open neighborhood of the zero-section in $TQ$, which is fiber-wise convex and has the property 
  that $\exp:\wt{\mathcal O}_q:=\wt{\mathcal O}\cap T_qQ\rightarrow Q$ is an embedding for every $q\in Q$.
\item[(4)]  Every point  $w\in u(\si)$ lies in  a  $(2n-2)$-dimensional submanifold $M_w\subset Q$ such that  $\psi_w(M_w)\subset {\mathbb R}^{2n}$ is a linear subspace.  Moreover, if $w=u(z)$ for $z\in \si$, then the tangent space  $H_w:=T_{w}M_w$ is {a complement} in $T_wQ$ of the image of $Tu(z)$. We point out that near $w$ the manifold  $M_w$  is the image of vectors in $T_wM_w\subset T_wQ$ under the exponential map.
\item[(5)] An open $G$-invariant neighborhood $U$ of the zero-section $0\in H^{3,\delta_0}_c(S, u^\ast TQ)$ so that every section $\eta\in U$ has its image in
    {$u^*\wt{\mathcal O}$ where $\wt{\mathcal O}$ is the  above open neighborhood of the zero-section in $TQ$.}
\item[(6)] {Concentric} subdisks $SD_z$ of $D_z$ with smooth boundaries for all $z\in \si$ so that their union is invariant under $G$.
    \end{itemize}

\begin{itemize}
\item[(7)] If  $z\in \si$, then  the restriction $u|D_z$ is an embedding transversal to $M_{u(z)}$. Moreover, if $z'\in D_z$ satisfies $u(z')\in M_{u(z)}$, then   $z'=z$.
\item[(8)]  For every  section $\eta\in U$ and every $z\in \si$, the map $f:=\exp_u(\eta):S\to Q$
satisfies  $f(D_z)\subset {\bf R}_2(u(z))$ and the restriction $f\vert D_z$ is an embedding transversal to $M_{u(z)}$ and intersects  $M_{u(z)}$ at a single point
$f(z')$ for a unique $z'\in SD_z$.
\item[(9)] Take the small disk structure of (1)  and let $j(a, v)$ be as given in (1).  If for $(a,v,\eta)$ and $(a',v',\eta')$ in $O\oplus U$ the tuples
    $$
    \alpha_{(a,v,\eta)}:=(S_a,j(a,v),M_a,D_a,\oplus_a(\exp_u(\eta)))
    $$
     and
$\alpha_{(a',v',\eta')}$ are isomorphic by some isomorphism  $\phi$, then there exists an automorphism  $g\in G$ satisfying  $\phi(SD_z)\subset D_{g(z)}$ for all $z\in\si$.
\item[(10)] If the isomorphism $\phi$ in (9) is,  in addition,  an isomorphism
$$
\phi:(S_a,j(a,v),(M\cup\si)_a,D_a)\rightarrow (S_{a'},j(a',v'),(M\cup\si)_{a'},D_{a'})
$$
between  the stable nodal  Riemann surfaces,  then $\phi=g_a$ for some $g\in G$ and $(a', v')=g\ast (a, v).$
\end{itemize}
\end{definition}
{The above space $H^{3,\delta_0}_{\textrm{c}}(S, u^*TQ)$ is the   space of sections having matching asymptotic constants near the nodes}. {As a complete norm we can take the sum of finitely many semi-norms obtained
as follows. Around  a point which is not a nodal point,  we fix a small open disk-like neighborhood ${\mathcal U}$  having a smooth boundary and 
not containing  a nodal point. Then we take a biholomorphic map $\psi:{\mathbf D}\rightarrow \cl({\mathcal U})$
and a smooth trivialization $\Gamma: (u^{\ast}TQ)|{\mathcal U}\rightarrow {\mathbf D}\times {\mathbb R}^{2n}$. 
If $\eta$ is a section of $u^\ast(TQ)$,  then $v=pr_2\circ\Gamma\circ \eta\circ \psi$ defines a map on the open unit disk
with image in ${\mathbb R}^{2n}$. Take a smooth function  $\beta:{\bf D}\rightarrow [0,1]$ which has compact support in
the open unit disk and is identically to $1$ on some $\delta$-disk for  $\delta\in (0,1)$ close to $1$. Then the map
$$\eta\mapsto \norm{\beta v}_{H^3({\bf D},{\mathbb R}^{2n})}$$
defines one of the semi-norms we are going to use.  Such constructions already occur in \cite{El}.
Near a nodal pair $\{x,y\}$ we take two open disk-like neighborhoods $D_x$ and $D_y$ with smooth boundaries
so that the images of their closures under the map $u$ lie in the domain of a smooth diffeomorphic chart $\psi$ of the manifold $Q$, 
$$\psi:(O(u(x)),u(x))\rightarrow ({\mathbb R}^{2n},0).$$
Then we take holomorphic polar coordinates on $D_x\setminus\{x\}$ (positive) and $D_y\setminus\{y\}$ (negative), 
denoted by $\sigma_x$ and $\sigma_y$.  Given a section $\eta$ with common asymptotic limit over
$\{x,y\}$,  we consider the pair $(v^+,v^-)$ of maps,  defined by
$$v^+(s,t)=pr_2\circ T\psi(u\circ\sigma_x(s,t))\eta(\sigma_x(s,t))$$
and 
$$v^-(s,t)=pr_2\circ T\psi(u\circ\sigma_y(s',t'))\eta(\sigma_y(s',t')).$$
{The maps} $v^+$ and $v^-$ have the same asymptotic constant, say $c$. Taking a suitable smooth function
$\beta:{\mathbb R}\rightarrow [0,1]$ with support in $(0,\infty)$ and satisfying $\beta(s)=1$ for $s\geq 1$
we define $\beta^+(s)=\beta(s)$ and $\beta^-(s')=\beta(|s|)$.  Now  we define the square of the  semi-norm 
using the usual $(3,\delta_0)$-norms by
$$|c|^2 +\norm{\beta^+(v^+-c)}_{H^{3,\delta_0}}^2 +\norm{\beta^-(v^--c)}_{H^{3,\delta_0}}^2.$$
Using a suitable choice of finitely many  semi-norms of this  type, one can define complete norms
on $H^{3,\delta_0}_{\textrm{c}}(S, u^*TQ)$,   and any such construction will lead to an equivalent norm.}

\begin{proposition}
There exists  good data centered at a stable map  $\alpha$ representing $[\alpha]\in Z$.
\end{proposition}
\begin{proof}
Since the existence of a good uniformizing family is guaranteed by 
Theorem \ref{existence-x}, the proposition follows in view of  Lemma \ref{ax} and Propositions \ref{ay}, up to property (10) which requires the following additional argument. 
If the assumption (10) holds, then we conclude by the properties of a good uniformizing family that 
$\phi=g^\ast_a$ for an automorphism  $g^\ast\in G^\ast$.  However, if we assume,  in addition,  that the open neighborhood 
$U$ of the zero-section of $u^*TQ$ in (5) and the neighborhood $O$ of $(0, 0)$ in (1) are sufficiently small, then $\phi$ is necessarily $C^{\infty}$-close on the core to an automorphism $g$ in $G$. Since $G\subset G^\ast$ is a finite group, we conclude that $\phi=g\in G$. This completes the proof of the proposition.
\end{proof}

 If we have good data centered at  the stable map $\alpha=(S,j,M,D,u)$, with automorphism group $G$,  where the map $u$ is only of class { $(3,\delta_0)$} we can find arbitrarily $H^{3,\delta_0}$-close,  a  smooth map $u'$, which is of class $(m+3,\delta_m)$, for all $m\geq 0$, such  that we have good data centered at $(S,j,M,D,u')$ with the same choice of small disk structure, the same stabilization $\si$ etc. and such  that $u=\exp_{u'}(\eta)$ for an arbitrarily small $\eta\in U'$ and the same automorphism group.  This is  true, since our conditions in the above definition are $C^0$-smallness conditions and,  at the disks around points in $\si$ they are  $C^1$-smallness conditions. We shall refer to good data centered at $(S, j, M, D, u')$ as to {\bf smooth data}. Hence we  {have proved} the following {basic} result.

\begin{theorem}\label{blog}
For every stable map $\alpha=(S,j,M,D,u)$ representing an element in $Z$ there exist good data centered at some smooth stable map  $\alpha'$ of the form
$$
\alpha'=(S,j,M,D,u')
$$
such that  $u=\exp_{u'}(\eta)$ and $\eta$ is an arbitrarily small section in the open neighborhood $U$ of the zero-section of $H^{3,\delta_0}_c((u')^\ast TQ).$
Moreover $\alpha$ and $\alpha'$ have the same automorphism group.
\end{theorem}

\begin{figure}[ht]
\psfrag {u}{$u$}
\psfrag  {uz}{$u(z)$}
\psfrag  {ez}{$\eta (z)$}
\psfrag  {h}{$H_{u(z)}$}
\psfrag  {m}{$M_{u(z)}$}
\psfrag  {ex}{$\exp_u(\eta)$}
\centering
\includegraphics[width=4.1in]{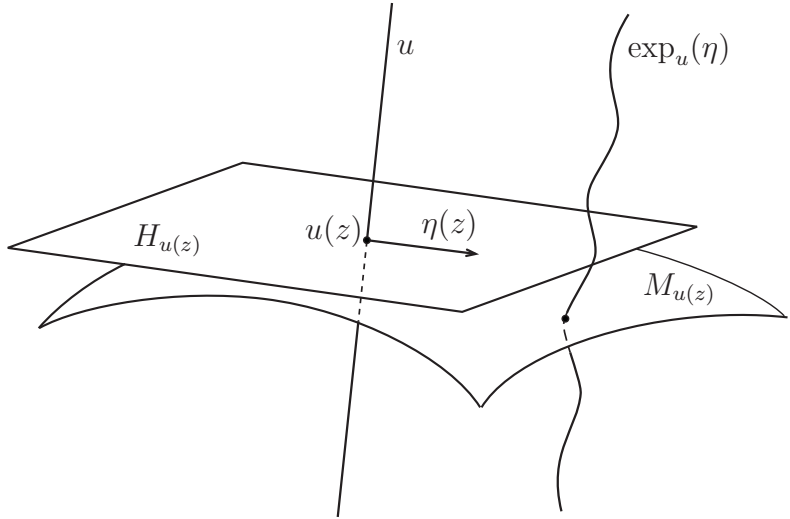}
\caption{The linear constraint belonging to $z\in \Sigma$.}
\label{Fig11}
\end{figure}

Next we consider good data centered at a smooth stable map  $\alpha=(S,j,M,D,u)$. If  $z$ belongs to the stabilization set $ \si$, we  introduce the plane  $H_{u(z)}=T_{u(z)}M_{u(z)}\subset T_{u(z)}Q$ in  the tangent space, which we refer to as the linear constraint associated with the point $z\in\si$. 
We note that points $z,z'\in\si$ satisfying  $u(z)=u(z')$ have the same associated constraint. 

In Section \ref{section-implant}  we have described the gluing and anti-gluing into the manifold $Q$.  Starting with the sc-Banach space
$H^{3, \delta_0}_c(S, u^\ast TQ)$ of sections $\eta$ (with matching nodal values) along the smooth map $u:S\to Q$ we define an sc-subspace of finite co-dimension $E_u$ by
$$
E_u=\{\eta\in H_c^{3,\delta_0}(u^\ast TQ)\vert \, \text{$\eta(z)\in H_{u(z)}$ for $z\in \si$}\}
$$
consisting of sections (along the smooth map $u:S\to Q$) satisfying the linear constraint $H_{u(z)}$ at the  points  $z\in \si$.
For every nodal pair $\{x,y\}\in D$ we can implant the family of splicing projections, which initially where defined for sections with matching nodal values 
$(\eta_x,\eta_y)$, where $\eta_x$ is a section of $(u|D_x)^\ast TQ$ and similarly for $D_y$. Near the boundaries of $D_x$ and $D_y$ the projection acts as the identity
and one can extend it to $E_u$ canonically, so that outside of the disks associated to the nodal pair it acts as the identity.
We shall denote it by $\pi^{\{x,y\}}_{a_{\{x,y\}}}$. The projections obtained for the various nodal pairs $\{x,y\}\in D$ commute and we denote by
$\pi_a$ with $a= {(a_{\{x,y\}})}_{\{x,y\}\in D}$ their product. It is defined for all $a$ with $|a_{\{x,y\}}|<\frac{1}{2}$, and therefore  for  $a\in {(B_\frac{1}{2})}^{\# D}$ where ${(B_\frac{1}{2})}^{\# D}$ is the cartesian product of $\# D$-copies of $B_{\frac{1}{2}}$.

We view the parameters $(a, v)$ as splicing parameters,  though $v$ is actually ineffective,  and  define the splicing core $K$ on $E_u$ as the set 
$$
K=\{(a, v,\eta)\vert\, \text{$ (a, v)\in O,\, \eta\in E_u$ and $\pi_a(\eta)=\eta$}\}
$$
{where $O$ is an open subset of the complex linear parameter space $(a, v)$ and $a\in (B_{\frac{1}{2}})^{\# D}$.}
In the following we shall abbreviate by  ${\mathcal O}$ the subset 
\begin{equation}\label{paraset}
{\mathcal O}=\{(a,v,\eta)\in K\vert \, \text{$(a,v)\in O$ and $\eta\in U\cap E_u $ }\},
\end{equation}
{where $U$ is the open neighborhood of the zero section $0\in H^{3,\delta_0}(S, u^*TQ)$ postulated in property (5) of the good data in Definition \ref{GD}.} The set  ${\mathcal O}$ is an open subset of the splicing core $K$.

\begin{definition}[{\bf Good uniformizing family of stable maps}]\label{def_g_uni_f_stable_maps}
{We assume we have good data (according to Definition \ref{GD}) centered at the smooth stable map $\alpha=(S, j, M, D, u)$ with automorphism group $G$, and consider the good uniformizing family 
$$
(a,v)\mapsto  \alpha_{(a,v)}:=(S_a,j(a,v),(M\cup \Sigma)_a,D_a), \quad  (a,v)\in O,
$$
of stable noded Riemann surfaces according to condition (1) in the Definition \ref{GD}.
Then the associated family of stable maps 
$$
(a,v,\eta)\mapsto  \alpha_{(a,v,\eta)}:=(S_a,j(a,v),M_a,D_a,\oplus_a\exp_u(\eta)), \quad  (a,v,\eta)\in {\mathcal O}, 
$$
is called a {\bf good uniformizing family of stable maps  centered at the smooth stable map} $\alpha=(S, j, M, D, u)$. } \index{good uniformizing family}
\end{definition}

From  
Theorem \ref{blog} one concludes immediately the following proposition. 
\begin{proposition}\label{imremark}
Given an element $[\alpha']\in Z$ in the space of stable maps,  there exists a good uniformizing family
$
(a,v,\eta)\mapsto  \alpha_{(a,v,\eta)}$, 
$(a,v,\eta)\in {\mathcal O}$, 
centered at a smooth $[\alpha]\in Z$,  and there exists a parameter $(a_0, v_0, \eta_0)\in {\mathcal O}$ for which  
$$[\alpha']=[\alpha_{(a_0,v_0,\eta_0)}].$$
\end{proposition}

\begin{definition}
If $G$ is the automorphism group of the stable map $(S, j, M, D, u)$,  the action of $G$ on ${\mathcal O}$ is defined by 
$$g\ast (a, v, \eta)=(a', v', \eta'),\qquad g\in G,$$
where 
$(a', v')=g\ast (a, v)$ and $\eta'\circ g=\eta$.
\end{definition}
\begin{proposition}\label{reremark1}
If  $(a, v, \eta)\mapsto \alpha_{(a, v, \eta)}$, $(a, v,\eta)\in {\mathcal O}$,  is a  good uniformizing family {of stable maps}, and 
$$[\alpha_{(a, v, \eta)}]=[\alpha_{(a', v' , \eta')}]$$
for two parameters $(a, v, \eta)$ and $(a', v', \eta')$ in ${\mathcal O}, $ 
then there exists an automorphism $g\in G$ satisfying
$$(a', v', \eta')=g\ast (a, v, \eta)$$
and  every isomorphism  between the two stable maps is of the form
$$g_a:\alpha_{(a, v,\eta)}\to \alpha_{(a', v', \eta')}.$$
\end{proposition}
\begin{proof}
By assumption there exists an isomorphism 
\begin{equation*}
\begin{split}
&\phi:\alpha_{(a, v, \eta)}=(S_a, j(a, v), M_a, D_a, \oplus_a (\exp_u (\eta)))\\
&\phantom{=======}\to 
\alpha_{(a', v', \eta')}=(S_{a'}, j(a', v'), M_{a'}, D_{a'}, \oplus_{a'} (\exp_u (\eta')))
\end{split}
\end{equation*}
satisfying 
$$\oplus_a (\exp_u (\eta))= \oplus_{a'} (\exp_u (\eta'))\circ \phi.$$
Take a point $z\in \si$ in the stabilization. Such a point lies, by construction, at a  finite distance  to the core 
part of $S$ which can be identified with the core part of $S_a$. Since the gluing takes place in the discs of the small disk structure, we have 
$\oplus_a (\exp_u (\eta)) (z)=\exp_{u (z)} (\eta (z))$ at $z\in\si$, and $\eta (z)\in H_{u(z)}\subset T_{u(z)}M_{u(z)}\subset T_{u(z)}Q.$

In view of the data (2), (6),  and (8), the restricted map $\exp_u (\eta)\vert D_z$ is an embedding of the disk $D_z$ which is transversal to the manifold $M_{u(z)}$ and intersecting $M_{u(z)}$ in the single point $z$. Here we have used that  the manifold $M_{u(z)}$ is totally geodesic. By the data (9), there exists an automorphism $g\in G$ satisfying
$$\phi (SD_z)\subset D_{g(z)}.$$
By definition of $G$, $u(g(z))=u(z)$ and hence,
$$\exp_u (\eta')\circ \phi(z)=\exp_u (\eta) (z)\in M_{u(z)}=M_{u(g(z))}.$$
Similarly, the map $\exp_u (\eta')\vert D_{g(z)}$ is an embedding of $D_{g(z)}$ and intersects the manifold $M_{u(g(z))}$ in precisely one point, namely in $g(z)$. Since $\phi (z)\in D_{g(z)}$, we therefore conclude that 
$$\phi (z)=g(z).$$
Hence, if $z\in \si$, then also $\phi (z)\in \si$ and consequently $\phi$ is an isomorphism
$$(S_a, j(a, v), (M\cup \si)_a, D_a)\to (S_{a'}, j(a', v'), (M\cup \si)_{a'}, D_{a'})$$
of nodal Riemann surfaces. In view of the data (10), there exists an automorphism $g\in G$ satisfying  $(a', v')=g\ast (a, v)$ and $\phi=g_a$. On the core part of $S$ we therefore have $g_a=g$ and hence, using $ u\circ g=u$,
$$
\exp_u (\eta')\circ g=\exp_u (\eta'\circ g)=\exp_u (\eta).
$$
Therefore, $\eta'\circ g=\eta$ on the core of $S$. The isomorphism $g_a$ maps 
$Z_{a_{\{x, y\}}}^{\{x, y\}}$ onto $Z_{a_{\{g(x), g(y)\}}}^{\{g(x), g( y)\}}$ for all nodal pairs $\{x, y\}\in D$.  
Using the definition of $g_a$, the fact that $u\circ g=u$ and the assumptions $\pi_a(\eta)=\eta$ and $\pi_a(\eta')=\eta'$, it follows that $\eta'\circ g=\eta$ on all of $S$.
This finishes the proof of Proposition \ref{reremark1}.
\end{proof}

{The next theorem  is the key result for proving the sc-smooth compatibility of good uniformizers. It is also important 
for the topological considerations later on.}

\begin{theorem}\label{key-z}
We consider two good uniformizing families of stable maps $q\mapsto  \alpha_{q}$ parametrized by  $q\in {\mathcal O}$ and $q'\mapsto  \alpha_{q'}'$ parametrized $q'\in {\mathcal O}'$. We assume that for two points  $q_0\in O$ and $q_0'\in O'$ there exists an isomorphism 
$$
\phi_0:\alpha_{q_0}\rightarrow \alpha_{q_0'}'
$$
between the associated stable maps. 
Then there exist a unique local germ of an sc-diffeomorphism 
$$f:({\mathcal O}, q_0)\to ({\mathcal O}', q_0'),\quad q\mapsto q'=f(q)$$
between the parameter spaces  satisfying $f(q_0)=q_0'$ and there exists a core-smooth germ of a family $q\mapsto  \phi_q$ of isomorphisms
$$\phi_q:\alpha_q\to \alpha_{f(q)}'$$
satisfying  $\phi_{q_0}=\phi_0$.
\end{theorem}

\begin{proof}

The proof proceeds in several steps. We first choose the set $\si'\subset S'$ of stabilizing points. 
The set satisfies, by definition, $\si'\cap (M'\cup \abs{D'})=\emptyset$ and $u'(\si')\cap u'(M'\cup \abs{D'})=\emptyset$.  The preimage of $\si'$ under the postulated diffeomorphism 
$\phi_0:S_{a_0}\to S_{a_0'}'$ is denoted by 
$$\Theta_0=\phi^{-1}_0(\si')\subset S_{a_0}.$$
The subset $\Theta_0$ is different from the nodal points $D_{a_0}$ and different from the marked points $M_{a_0}$, and the nodal Riemann surface 
$$(S_{a_0}, j(a, v), M_{a_0}\cup \Theta_0, D_{a_0})$$
is stable. Moreover, it is possible that $\Theta_0$ contains points lying in the neck regions of $S_{a_0}$. Considering $(a, v)$ close to $(a_0, v_0)$, we denote by $\Theta\subset S_a$ a deformation of $\Theta_0$ as defined in Section \ref{dm-subsect}. Therefore, we deduce from the universal family property of the stabilized target family (Theorem \ref{smoothfamily}) the following result. 
\begin{lemma}\label{lem4.19}
There exists a germ of a smooth map
$$(a, v, \Theta)\to (a' (a, v,\Theta), v'(a, v, \Theta))$$
satisfying 
$$(a' (a_0, v_0,\Theta_0), v'(a_0, v_0, \Theta_0))=(a_0', v_0'),$$
and there exists an associated core-smooth germ of isomorphisms 
$$\phi_{(a, v, \Theta)}:(S_a, j(a, v), M_a\cup \Theta, D_a)\to 
(S_{a'}, j'(a', v'), (M'\cup \si')_{a'}, D_{a'}), $$
where 
$a'=a'(a, v, \Theta)$ and $v'=v'(a, v, \Theta)$, satisfying 
$$\phi_{(a_0, v_0, \Theta_0)}=\phi_0.$$
\end{lemma}
By definition of an isomorphism between stable maps,  the diffeomorphism $\phi_0:S_{a_0}\to S'_{a_0'}$ satisfies
$$\oplus_{a_0}\exp_u (\eta_0)=\oplus_{a_0'}'\exp_{u'}'(\eta'_0)\circ \phi_0.$$
If $z\in \Theta_0$, then $\phi_0(z)=z'\in \si'$ and therefore
$$\oplus_{a_0}\exp_u (\eta_0)(z)=\oplus_{a_0'}'\exp_{u'}'(\eta'_0)(z')\in M'_{u'(z')}.$$
Hence the map $\oplus_{a_0}\exp_u (\eta_0):S_{a_0}\to Q$  intersects the embedded  submanifold $M'_{u(z')}$ at the point $z$ transversally, and we call the transversal intersection point $z=z(a_0, v_0,\eta_0)\in S_{a_0}$. 
If $(a, v,\eta)$ is close to $(a_0, v_0,\eta_0)$ on level $0$  and hence in the $C^1$-topology, it follows from Proposition \ref{n-prop2.24} that also the nearby maps 
$\oplus_a\exp_u(\eta):S_a\to Q$ intersect the manifold $ M'_{u'(z')}\subset Q$  transversally at the unique points $z(a, v, \eta)\in S_a$ so that 
$$\oplus_a\exp_u (\eta)(z(a, v, \eta))\in  M'_{u'(z')}.$$
Moreover, the map
$$(a, v,\eta)\mapsto z(a, v,\eta)\in S_a$$
is sc-smooth in the following sense,  considering the  fact that the target surface depends on $a$. Abbreviating $q=(a, v,\eta)$ and assuming that $z(q_0)$ lies in the core region of $S$, we can  consider the map $q \to z(q)$ as a map into the fixed surface $S$. If the point $z(q)$ lies in  a nontrivially glued neck, we describe the map with respect to the two distinguished polar coordinate systems 
$[s,t]$ or $[s',t']$ and the maps obtained this way are sc-smooth. Hence we obtain the sc-smooth germ 
$$q\to \Theta_q=\{z(q)\}$$
for $q=(a, v,\eta)$ near $q_0$ and we  can summarize the discussion as follows.
\begin{lemma}\label{lem4.20}
Let $\Theta_q$ consist of the unique points $z(q)=z(a, v,\eta)\in S_a$ at which 
$$
\oplus_a\exp_u(\eta)(z(q))\in M'_{u'(z')},\ \text{some}\ \ z'\in\Sigma'.
$$
Then the germ $q\to \Theta_q$ at $q_0$ is sc-smooth in the sense defined above.
\end{lemma} 
Recalling from Lemma \ref{lem4.19} the germ of a  smooth map 
$$(a, v,\Theta)\mapsto (a'(a, v, \Theta), v'(a, v, \Theta))$$  and the associated germ $(a, v, \Theta)\to \phi_{(a, v, \Theta)}$ of isomorphism, we define 
the germ of a smooth map as the following composition of maps 
\begin{align*}
a'(a,v,\eta):&=a'(a, v, \Theta_{(a, v,\eta)})\\
v'(a,v,\eta):&=v'(a, v, \Theta_{(a, v,\eta)})
\end{align*}
and introduce the associated core-smooth germ $q\to \phi_q$ of isomorphisms 
\begin{equation*}
\begin{split}
\phi_{(a, v,\eta)}&:=\phi_{(a, v, \Theta_{(a, v, \eta)})}:\\
&(S_a, j(a, v), M_a\cup \Theta_{(a, v, \eta)}, D_a)\to 
(S_{a'}, j'(a', v'), (M'\cup \si')_{a'}, D'_{a'})
\end{split}
\end{equation*}
where 
$a'=a'(a,v,\eta)$ and $v'=v'(a,v,\eta)$ and where $(a, v, \eta)$ is close  to $(a_{0},v_{0}, \eta_{0})$. Then 
\begin{gather*}
\phi_{(a_0, v_0, \eta_0)}=\phi_0\\
(a'(a_0,v_0,\eta_0), v'(a_0,v_0,\eta_0)=(a'_0, v'_0).
\end{gather*}
In view of Lemma \ref{lem4.19} and Lemma \ref{lem4.20}, and using the chain rule, we obtain the following lemma.
\begin{lemma}\label{lem4.21}
The germ 
$$(a, v,\eta)\to (a'(a, v, \eta), v'(a, v, \eta))$$
is sc-smooth near $(a_0, v_0, \eta_0)$.
\end{lemma}
At this point we have constructed near $q_0=(a_0, v_0, \eta_0)$ the sc-smooth germ 
$$q\to (a'(q), v'(q))$$
of functions and the associated core-smooth germ $q\to \phi_q$ of isomorphisms 
\begin{eqnarray*}
&\phi_q:
(S_a, j(a, v), M_a\cup \Theta_q, D_a)\to\ \ \ \ \ \ \ \ \ \ \ \ \ \ \ \ \ \ \ \ \ \ \ \ \ \ \ &\\ 
&\ \ \ \ \ \ \ \ \ \ \ \ (S_{a'(q)}, j'(a'(q), v'(q)), (M'\cup \si')_{a'(q)}, D'_{a'(q)})&
\end{eqnarray*}
where $q=(a, v,\eta)$. It satisfies, in particular, 
$$
\phi_q(z(q))=z'\in \si'\ \ \text{for some}\ \ z'\in\Sigma'.
$$
By construction of $z(q)$, the maps 
$$(\oplus_a\exp_u(\eta))\circ \phi_{(a, v,\eta)}^{-1}:S'_{a'(a, v,\eta)}\to Q$$
are at the points $z'\in \si'$ transversal to the submanifolds $M_{u'(z')}'$. 

We recall that at $q_0=(a_0, v_0, \eta_0)\in {\mathcal O}$, the map $\phi_{q_0}=\phi_0$ satisfies
$$\oplus_{a'(q_0)} \exp'_{u'}(\eta'_0)=\oplus_{a'_0} \exp_{u}(\eta_0)\circ \phi_{q_0}^{-1}.$$
With the following lemma the proof of Theorem \ref{key-z} is complete.
\begin{lemma}\label{lem4.23}
For $q=(a, v,\eta)$ near $q_0=(a_0, v_0, \eta_0)$ there exists a uniquely defined section $\eta'(q)$ of the bundle $(u')^\ast TQ$ near the section $\eta_0'$, satisfying 
$$
\oplus_{a'(q)} \exp'_{u'}(\eta'(q))=\oplus_{u} \exp_{u}(\eta)\circ \phi_{q}^{-1},
$$
and  the linear constraint conditions,  and $(a'(q), v'(q),\eta'(q))\in {\mathcal O}'$. In addition, 
$$q\to \eta'(q)$$
is an sc-smooth map.
\end{lemma}
\begin{proof}
The  formula 
$$
\oplus_{a'(q)} \exp'_{u'}(\eta')=\oplus_{a} \exp_{u}(\eta)\circ \phi^{-1}_q\ \ \text{and}\ \  \ominus_{a'(q)} \eta'(q)=0
$$
determines the section $\eta'$ as a function of $q$ in a local way and the following observations simplify our task.\\

{\bf Observation 1:}  Using a partition of unity argument in the target $S'$ it suffices to show that $\eta'$ restricted to suitable subsets, as a function of $(a,v,\eta)$ is sc-smooth. The suitable subsets are either small open subsets whose closures do not contain nodal points or small open neighborhoods of nodal points.\\

 We can consider $\eta$ as a variable or alternatively $\exp_u(\eta)$, since the map $\eta\rightarrow \exp_u(\eta)$ is level-wise smooth,  a classical nonlinear analysis fact, which is used in  the construction of manifolds of maps in \cite{El}.\\

{\bf Observation 2:} If $z_0'$ is a point which is not a nodal point
in $|D_{a_0'}'|\subset S'$,  the section $\eta'$ near $z_0'$ is determined by
$\eta$ near $z_0:=\phi_{(a_0,v_0,\eta_0)}^{-1}(z_0')$ and $z_0$ is not a nodal point in $|D_{a_0}|$.\\

If $z'_0\in S'_{a'}$ lies outside of the discs of the small disk structure of $S'$ we have, in a neighborhood of $z'_0$ for $q$ near $q_0$,
$$\oplus_{a'(q)}\exp'_{u'}(\eta')=\exp_{u'}'(\eta').$$
In this case the section $\eta'$ near $z_0'$ is determined by the formula 
$$\eta'=(\exp'_{u'})^{-1}\circ \oplus_a\exp_u(\eta)\circ \phi_q^{-1}.$$
Since the isomorphism $\phi_q=\phi_{(a, v, \Theta_q)}$ is core-smooth and $q\to \Theta_q$ is smooth, it follows from Theorem \ref{action-diff} and Lemma \ref{lem4.20} using the chain rule, that the right hand side depends sc-smoothly on $q$. 

If $z_0'$ is in a neck we can study the two equations
\begin{align*}
\oplus_{a'(q)}\exp'_{u'}(\eta')&= \oplus_a\exp_u(\eta)\circ \phi_q^{-1}\\
\ominus_{a'(q)}(\eta')&=0
\end{align*}
in an open neighborhood of $z_0'$. Again the right hand side depends sc-smoothly on $q$. In the polar coordinates on the discs one sees that  the unique solution $\eta'=\eta'(q)$ on a neighborhood of $z_0'$ is, 
{in view of Proposition \ref{ay} and Proposition \ref{az}}, guaranteed by Theorem \ref{theorem_neck}.
Moreover, $(a'(q), v'(q),\eta'(q))\in {\mathcal O}$ and $q\to \eta'(q)$ is sc-smooth.

Finally, we look at a nodal point $z_0'=x$ on $S'$. This implies that also $\phi_0^{-1}(z_0')=z_0$ is a nodal point on $S$. In this case $a_{\{x, y\}}=0$ and $a'_{\{x', y'\}}=0.$  A direct application of Theorem \ref{theorem_neck} implies our assertion. This completes the proof of Lemma \ref{lem4.23}.
\end{proof}
With Lemma  \ref{lem4.23} the proof of Theorem \ref{key-z} is complete.
\end{proof}

\section{Compatibility of Good Uniformizers}\label{sect4.3-pol.str}
In the construction of the polyfold structure on the space $Z$ of stable curves  we shall imitate the Lie groupoid approach to the Deligne-Mumford theory.

We take a good uniformizing family 
$$(a, v, \eta)\mapsto  \alpha_{(a, v, \eta)},\quad \text{where $(a, v,\eta)\in {\mathcal O}$}$$
of stable maps  centered at the smooth stable map  $\alpha=(S, j, M, D, u)$, continue to abbreviate 
$$q:=(a, v, \eta)\in {\mathcal O},$$
and define the graph $\cg$ of the family as the set
$$\cg=\{(q, \alpha_q)\vert \, q\in {\mathcal O}\}.$$
We recall that ${\mathcal O}$ is the open subset of a splicing core introduced in the formula \eqref{paraset} in Section \ref{xA}.  We equip the set $\cg$ with the  topology which turns  the bijective projection map 
$$\pi:\cg\to {\mathcal O},\qquad \pi (q, \alpha_q)=q$$
into a homeomorphism. Then the sc-smooth chart $(\pi, {\mathcal O})$ defines the natural sc-structure on $\cg$ so that $\cg$  is an M-polyfold and the map $\pi:\cg\to {\mathcal O}$ is an sc-diffeomorphism. We shall use the M-polyfolds  $\cg$  the same way we have used the graphs of uniformizing families in the Deligne-Mumford theory.

{We consider the two M-polyfolds  
$$
\cg=\{(q, \alpha_q)\vert \, q\in {\mathcal O}\}\ \ \text{and}\ \  
\cg'=\{(q', \alpha'_{q'})\vert \, q'\in {\mathcal O}'\},
$$
 where $q=(a,v,\eta)\in{\mathcal O}$ and $q'=(a', v', \eta')\in {\mathcal O}'$. For two points $(q, \alpha_q)\in \cg$ and $(q', \alpha'_{q'})\in \cg'$ for which there exists an isomorphism 
$$\phi:\alpha_q\to \alpha'_{q'}$$
between the stable maps, we form the triple 
$$\Phi:=((q, \alpha_q), \phi, (q', \alpha'_{q'})).
$$
We view $\Phi$ as a morphism $(q,\alpha_q)\rightarrow (q',\alpha_{q'}')$ and write
$$\Phi:(q,\alpha_q)\rightarrow (q',\alpha_{q'}').$$ }
The set of all {morphisms} $\Phi$ will be denoted by 
$$ M(\cg, \cg').$$
The source  map 
$$
s: M(\cg, \cg')\rightarrow \cg$$
is defined by 
$$s((q, \alpha_q), \phi, (q', \alpha'_{q'}))=(q, \alpha_q)$$
and the target map 
$$t: M(\cg, \cg')\rightarrow \cg'$$
is defined by 
$$
t((q, \alpha_q), \phi, (q', \alpha'_{q'}))=(q', \alpha'_{q'}).
$$

Next we equip the set  $M(\cg, \cg')$ with a topology,  defined by means of a neighborhood basis. Let  
$$\wh{\Phi}= ((\wh{q}, \alpha_{\wh{q}}), \wh{\phi}, (\wh{q}',\alpha'_{\wh{q}'}))$$
 be a morphism in  $M(\cg, \cg')$. We are going to define sets ${\mathcal V}$ of  open neighborhoods of $\wh{\Phi}$.  In view of Theorem \ref{key-z},  there exists  a germ of a  local sc-diffeomorphism 
 $$f: {\mathcal O}\to {\mathcal O}',\qquad q\mapsto q'(q)$$ 
 satisfying $f(\wh{q})=\wh{q}'$ and a germ of a core-smooth family 
$$\phi_q:\alpha_q\to \alpha'_{f(q)}$$ of isomorphisms satisfying $\phi_{\wh{q}}=\wh{\phi}.$  We consider  the sets  ${\mathcal V}$ consisting of the morphisms 
$((q, \alpha_q), \phi_q, (f(q), \alpha'_{f(q)}))$ where   $q\in V\subset {\mathcal O}$ and where $V$ is a {sufficiently small} open neighborhood  of $\wh{q}$. We shall show that the sets ${\mathcal V}$ form  a basis  for a Hausdorff  topology on $M(\cg, \cg')$. 
To see this,  we take two such sets ${\mathcal V}_1$ and ${\mathcal V}_2$  containing a morphism 
$\Phi_{q_0}=((q_0, \alpha_{q_0}), \phi_0, (q_0', \alpha'_{q_0'}))$.
The set ${\mathcal V}_1$  consist of triples 
$$
((q, \alpha_q), \phi^1_q, (f_1(q), \alpha'_{f_1(q)}))
$$
 in which  $q\in V_1\subset {\mathcal O}$  where $V_1$ is an open neighborhood of $q_1$, in which  $f_1:V_1\to f_1(V_1)\subset {\mathcal O}'$ is an sc-diffeomorphism satisfying $f_1(q_1)=q_1'$,  and in which $q\mapsto \phi_q^1$,  for $q\in V_1$,  is a core-smooth family  $\phi^1_q:\alpha_q\to \alpha'_{f_1(q)}$ 
of isomorphisms satisfying $\phi^1_{q_1}=\phi_1.$ The set ${\mathcal V}_2$ is defined similarly. It is  centered at  the point $q_2$,  has  the sc-diffeomorphism $f_2:V_2\to f_2(V_2)\subset {\mathcal O}'$, and the core-smooth family $q\mapsto \phi_q^2$ replacing $q_1$, $f_1$, and $q\mapsto \phi_q^1$, respectively. Since $((q_0, \alpha_{q_0}), \phi_0, (q_0', \alpha'_{q_0'}))\in {\mathcal V}_1\cap {\mathcal V}_2$, it follows that $q_0\in V_1\cap V_2$. We next show  that on a  small open neighborhood of $q_0$ contained in $V_1\cap V_2$, the maps $f_1$ and $f_2$ coincide and the two  maps $q\mapsto \phi_q^1$ and $q\mapsto \phi_q^2$ also coincide. We consider  for $q\in V_1\cap V_2$, the following  isomorphisms between the noded surfaces
$$\phi_q^2\circ (\phi_q^1)^{-1}:\alpha'_{f_1(q)}\to \alpha'_{f_2(q)}.$$
At the point $q_0$ we have $f_1(q_0)=f_2(q_0)=q_0'$ and 
$$\phi_{q_0}^2\circ (\phi_{q_0}^1)^{-1}=\phi_0\circ (\phi_0)^{-1}=\text{id}:\alpha'_{q_0'}\to \alpha'_{q_0'}.$$
We claim that $f_1(q)=f_2(q)$ for $q$ close to $q_0$. Indeed,  arguing by contradiction we 
assume that there exists a sequence $(q_k)\subset V_1\cap V_2$ converging to $q_0$ such that  $f_1(q_k)\neq f_2(q_k)$.  In view of Proposition \ref{reremark1} and the fact that the group of automorphisms is finite,  there exists an automorphism $g$  such that $f_2(q_k)=g\ast f_1(q_k)$ for all $k$ and $\phi_{q_k}^2\circ (\phi_{q_k}^1)^{-1}=g_{a_k}$ where the gluing parameter $a_k$ is the first component of $f_1(q_k)$. Since 
$\phi_{q_0}^2\circ (\phi_{q_0}^1)^{-1}=\text{id}$, it follows that $g=\text{id}$ so that $f_2(q_k)=\text{id}\ast  f_1(q_k)=f_1(q_k)$ which is a contradiction. We conclude that indeed  there exists an open  neighborhood  $V_3$ of $q_0$ contained in $V_1\cap V_2$ such that $f_1(q)=f_2(q)$ for all $q\in V_3$.  Using similar arguments and Proposition \ref{reremark1} one shows that also 
$\phi_q^1=\phi_q^2$ for all $q\in V_3$. Consequently, setting $f(q)=f_1(q)=f_2(q)$ and $\phi_q=\phi_q^1=\phi_q^2$ for all $q\in V_3$ and ${\mathcal V}_3=\{((q, \alpha_q), \phi_q, (f(q), \alpha'_{f(q)}))\vert q\in V_3\}$, we conclude that  ${\mathcal V}_3\subset {\mathcal V}_1\cap {\mathcal V}_2$. Hence the  collection of sets ${\mathcal V}$ defines  a basis of the topology on $M(\cg, \cg')$.

\begin{lemma}\label{huasdorff-lemma}
The  topology on $M(\cg, \cg')$ is Hausdorff.
\end{lemma}
\begin{proof}  We take two distinct points  $\Phi_{1}=((q_1, \alpha_{q_1}), \phi_1, (q_1', \alpha_{q_1'}'))$ and $\Phi_{2}=((q_2, \alpha_{q_2}), \phi_2, (q_2', \alpha_{q_2'}'))$
and show that there are two disjoints sets ${\mathcal V}_1$ and ${\mathcal V}_2$ containing $\Phi_{1}$ and $\Phi_{2}$, respectively.
{We consider the following three cases.\\[0.5ex]
{\bf Case 1}.\, $q_1\neq q_2.$\\
{\bf Case 2}.\,  $q_1=q_2$ and $q_1'\neq q_2'$.\\
{\bf Case 3}.\,  $q_1=q_2$, $q_1'=q_2'$, and $\phi_1\neq \phi_2$.}\\[0.5ex]

{In case 1 we take} two disjoint small open  neighborhoods $V_1$ and $V_2$ of $q_1$ and $q_2$, and define  the sets ${\mathcal V}_1=\{((q, \alpha_q), \phi^1_q, (f_1(q), \alpha'_{f_1(q)}))\vert q\in V_1\}$ and ${\mathcal V}_2=\{((q, \alpha_q), \phi^2_q, (f_1(q), \alpha'_{f_2(q)}))\vert q\in V_2\}$. Clearly, ${\mathcal V}_1\cap {\mathcal V}_2=\emptyset$.

{In case 2 we
find} an open neighborhood $V$ of $q_1=q_2$ such that $f_1(V)\cap f_2(V)=\emptyset$. Then the sets 
 ${\mathcal V}_1=\{((q, \alpha_q), \phi^1_q, (f_1(q), \alpha'_{f_1(q)}))\vert q\in V\}$ and ${\mathcal V}_2=\{((q, \alpha_q), \phi^1_q, (f_2(q), \alpha'_{f_2(q)}))\vert q\in V\}$ are disjoint.
 
 { Finally, in  case 3, we take  a small open neighborhood $V$ of $q_1=q_2$ and apply Proposition \ref{reremark1} to find two  sc-diffeomorphisms $f_1: V\to {\mathcal O}'$ and $f_2:V\to {\mathcal O}'$ satisfying  $f_1(q_1)=q_1'$ and $f_2(q_2)=q_2'$ and the core-smooth families 
 $\phi_q^1:\alpha_{q}\to \alpha'_{f_1(q)}$ and  $\phi_q^2:\alpha_{q}\to \alpha'_{f_2(q)}$ of isomorphisms satisfying $\phi_{q_1}^1=\phi_1$ and $\phi_{q_2}^2=\phi_2$.  } The composition 
 $$\phi_q^{2}\circ (\phi_q^1)^{-1}:\alpha'_{f_1(q)}\to \alpha'_{f_2(q)}$$ 
 is an isomorphism for every $q\in V$ and at $q=q_1=q_2$ we have $\phi_{q_2}^{2}\circ (\phi_{q_1}^1)^{-1}=\phi_2\circ (\phi_1)^{-1}:\alpha'_{q_1'}\to \alpha'_{q_2'}$ which,  by assumption,  is different from the identity map.
Again applying Proposition \ref{reremark1} we find a possibly smaller open neighborhood of $q_1=q_2$, still denoted by $V$, and  an automorphism $g\in G$ different from the identity map satisfying 
$$f_2(q)=g\ast f_1(q)\quad \text{and}\quad  \phi_{q}^{2}\circ (\phi_{q}^1)^{-1}=g_{a(q)}$$
 where the gluing parameter $a$ is the first  component of $f_1(q)$. 
Hence $\phi_q^2=g_{a(q)}\circ \phi_q^1$. Now since the families $q\mapsto \phi^1_q$ and $q\mapsto \phi^2_q$ are core smooth we find an open neighborhood $V$ of $q_1=q_2$ such that $\phi^1_q\neq \phi^2_q$ for all $q\in V$.  Consequently, the sets ${\mathcal V}_1=\{((q, \alpha_q), \phi^1_q, (f_1(q), \alpha'_{f_1(q)}))\vert q\in V\}$ and ${\mathcal V}_2=\{((q, \alpha_q), \phi^2_q, (f_1(q), \alpha'_{f_2(q)}))\vert q\in V\}$ are disjoint. We have verified that our  topology of  $M(\cg, \cg')$ is Hausdorff.
\end{proof}

Recalling  that the topology on $\cg$ is defined by requiring that the projection map $\pi:\cg\to {\mathcal O}$, $\pi (q, \alpha_q)=q$, is a homeomorphism, we see that the source  map $s:M(\cg,\cg')\to \cg$ and the target map $t:M(\cg, \cg')\to \cg'$ are continuous. 
Next we consider, in a small neighborhood of $(q_0, \alpha_{q_0})$,  the  injective map 
$$
\Gamma:\cg\to M(\cg, \cg'),$$
defined by 
$\Gamma (q,\alpha_q)=((q, \alpha_q), \phi_q, (q'(q), \alpha'_{q'(q)}))$, where $q'(q)=f(q)$ with the local sc-diffeomorphism $f:{\mathcal O}\to {\mathcal O}'$ 
guaranteed by Theorem \ref{key-z}.  The compositions $s\circ \Gamma:\cg\to \cg$ and $t\circ \Gamma:\cg\to \cg'$ satisfy 
$$s\circ \Gamma (q, \alpha_q)=(q, \alpha_q)\qquad \text{and}\qquad t\circ \Gamma (q, \alpha_q)=(q'(q), \alpha_{q'(q)}).$$

Therefore, the map $t\circ \Gamma:\cg \to \cg'$ is, in view of Theorem \ref{key-z}, a local sc-diffeomorphism.  We also conclude that the map $\Gamma:\cg\to M(\cg,  \cg')$ is continuous. Since $\Gamma \circ s =\id$, i.e., 
$$\Gamma \circ s ((q, \alpha_q), \phi_q, (q'(q), \alpha'_{q'(q)}))=((q, \alpha_q), \phi_q, (q'(q), \alpha'_{q'(q)})),$$
we see that  $\Gamma$ and $s$ are local homeomorphisms. Therefore, 
the local homeomorphism
$$s:M(\cg, \cg')\to \cg$$
allows us to equip  the space $M(\cg, \cg')$ with an sc-structure which turns the source map $s$ into a local sc-diffeomorphism. Hence also $\Gamma=s^{-1}:\cg\to M(\cg, \cg')$ is a local sc-diffeomorphism and since $t\circ \Gamma (q, \alpha_q)=(q'(q), \alpha'_{q'(q)}):\cg\to \cg'$ is, by Theorem \ref{key-z}, a local sc-diffeomorphism, the target map $t=(t\circ \Gamma )\circ \Gamma^{-1}:M(\cg, \cg')\to \cg'$ is also a local sc-diffeomorphism. We have proved the following proposition.

\begin{proposition}\label{a-key-z}
If $\cg$ and $\cg'$ are equipped with their natural
M-polyfold structures, then the set $M(\cg,\cg' )$ of morphisms has also a natural M-polyfold structure for which the source map $s$ and the target map $t$ 
are local sc-diffeomor\-phisms.
\end{proposition}
Using the same arguments, we obtain the following additional consequence of Theorem \ref{key-z}. 
\begin{proposition}\label{b-key-z}
Consider $\cg$, $\cg',$ and  $M(\cg,\cg')$ equipped with their natural M-polyfold structures. Then,  
\begin{itemize}
\item[{\em (1)}] The unit map
$
u:\cg \rightarrow M(\cg,\cg)$, 
defined by 
$$u(q, \alpha_q)=((q,\alpha_q), \id, (q,\alpha_q)),$$
is an sc-smooth map.
\item[{\em (2)}] The inversion map
$
i: M(\cg,\cg')\rightarrow  M(\cg',\cg)$, defined by 
$$i((q,\alpha_q),\phi,(q',\alpha_{q'}'))= ((q',\alpha_{q'}'),\phi^{-1},(q,\alpha_{q})),$$
is sc-smooth.
\end{itemize}
\end{proposition}
\begin{proof}
By the previous proposition the source map $s:M(\cg, \cg')\to \cg$ is a local sc-diffeomorphism and consequently defines a local sc-smooth chart. The map $u$ is in these local charts represented by $s\circ u=\id$ and hence is an sc-smooth map. The inversion map $i$,  composed with $t$ and the locally inverted $s$,  satisfies
$$t\circ i \circ s^{-1}=\id.$$
Consequently, locally $i=t^{-1}\circ s$ is, by the chain rule, an sc-smooth map, as claimed.
\end{proof}

In order to define the multiplication map, it is useful to first recall the definition of the fibered product and its properties
given in Lemma \ref{ert}.
Now consider the M-polyfolds $\cg$, $\cg'$ and  $\cg''$   of graphs as 
introduced above. Then the multiplication map 
$$
m:M(\cg',\cg''){{_s}\times_t}M(\cg,\cg')\rightarrow M(\cg,\cg'')
$$
on the fibered product  is defined by 
$$m((q', \alpha'_{q'}), \psi, (q'', \alpha''_{q''})), ((q, \alpha_{q}), \phi, (q', \alpha'_{q'})))=
((q', \alpha'_{q'}), \psi\circ \phi, (q'', \alpha''_{q''})).$$
\begin{proposition}\label{c-key-z}
The multiplication map is an sc-smooth map.
\end{proposition}
\begin{proof}
We take two morphisms 
\begin{equation*}
\Phi_1=((q_0, \alpha_{q_0}), \phi_0,(q_0',\alpha'_{q'_0}))\in M(\cg, \cg')
\end{equation*}
and
\begin{equation*}
\Phi_2=((q'_0, \alpha'_{q_0}), \phi'_0,(q_0'',\alpha''_{q''_0}))\in M(\cg', \cg'')
\end{equation*}
satisfying $t(\Phi_1)=s(\Phi_2)$ so that $m(\Phi_2, \Phi_1)\in M(\cg, \cg'')$ is well defined. The source map $s:M(\cg, \cg')\to \cg$ is a local sc-diffeomorphism and the local inverse $s^{-1}$ maps $(q_0, \alpha_{q_0})\in \cg$ into 
$((q_0, \alpha_{q_0}), \phi_0, (q'_{0}, \alpha'_{q_0}))\in M(\cg, \cg')$.  This local diffeomorphism, in view of Theorem \ref{key-z},  is of the  form
$$s^{-1}(q, \alpha_{q})=((q,\alpha_q), \phi_q, (q'(q), \alpha'_{q'(q)}))$$
in  a neighborhood of $(q_0, \alpha_{q_0})$ where $q\mapsto q'(q)$ is a local sc-diffeomorphism of ${\mathcal O}$ mapping $q_0$ into $q'(q_0)=q'_0$. The same holds true for the source map $s_1:M(\cg', \cg'')\to \cg'$. Its local inverse $s_1^{-1}$ mapping 
$(q_0', \alpha_{q'_0}')\in \cg'$ into $((q_0', \alpha'_{q_0}), \phi_0', (q_0'', \alpha_{q_0''}))\in M(\cg', \cg'')$  is of the form 
$$s_1^{-1}(q', \alpha'_{q'})=((q',\alpha'_{q'}), \phi'_{q'}, (q''(q'), \alpha''_{q''(q')}))$$
where 
$q'\mapsto q''(q')$ is a local sc-diffeomorphism mapping $q'_0$ into $q''(q'_0)=q''_0$.  Consequently, 
the inverse $\psi^{-1}$ of an sc-smooth chart 
$$\psi:M(\cg',\cg''){{_s}\times_t}M(\cg,\cg')\to \cg$$ has  the form 
\begin{equation*}
\begin{split}
&\psi^{-1}(q, \alpha_q)\\
&\qquad =\Bigl( 
\bigl((q'(q), \alpha'_{q'(q)}), \phi'_{q'(q)}, (q''(q), \alpha''_{q''(q)})\bigr), 
\bigl( (q, \alpha_{q}), \phi_{q}, (q'(q), \alpha'_{q'(q)})\bigr)
\Bigr).
\end{split}
\end{equation*}
Also the source map $s_2:M(\cg,\cg'')\to\cg$ is a local sc-diffeomorphism and  defines a local sc-smooth chart. 
The representation of the multiplication map $m$ in these local sc-smooth charts becomes $s_2\circ m\circ \psi^{-1}$ which is equal to the identity map. Therefore, the multiplication map $m$ is an sc-smooth map and the proof of Proposition \ref{c-key-z} is complete.
\end{proof}

One often represents a morphism $\Phi\in M(\cg, \cg')$ by the arrow 
$$\Phi:s(\Phi)\to t(\Phi),$$
and if the morphisms  $\Phi\in M(\cg, \cg')$  and $\Psi  \in M(\cg', \cg'')$  satisfy $t(\Phi)=s(\Psi)$ one often writes the multiplication map 
as a composition 
$$m(\Psi, \Phi)=\Psi\circ \Phi:s(\Phi)\to t(\Psi).$$
Moreover, if $i:M(\cg, \cg')\to M(\cg', \cg)$ is the inversion map, one uses for the inverse arrow the notation 
$$\Phi^{-1}:=i(\Phi):t(\Phi)\to s(\Phi).$$

{\section{Compactness 
Properties of $M(\cg,\cg')$}
\label{section_compactness}}
{We next establish compactness properties for a pair of good uniformizers which will be  crucial in the study of the topology of $Z$ later on. We start with the following basic compactness result.}
\mbox{}\\
\begin{theorem}\label{rees}
{Let $\cg=\{(q, \alpha_q)\vert \, q\in \co \} $ be a M-polyfold as defined above, where 
$q\mapsto \alpha_q$, $q=(a, v, \eta)\in \co$, is a good uniformizing family of stable maps, and let  
$\pi:\cg\rightarrow \co $ be the associated sc-diffeomorphism.  Then, for every point $q\in \co$, 
there exists an open neighborhood 
$U^{\cg}(q)\subset \co$ which has the following property.
}

{If $\cg$ and $\cg'$ are two such M-polyfolds and $q\in \co$ and $q'\in \co'$ two points, then every sequence  $\Phi_k \in M(\cg,\cg')$ of morphisms  
 satisfying 
$$\pi\circ s(\Phi_k)\in U^\cg(q)\quad 
 \text{ and} \quad \pi'\circ t(\Phi_k)\rightarrow q'$$
 possesses a subsequence converging  in  $M(\cg,\cg')$.}
\end{theorem}

\begin{remark} 
{The above result will imply that the natural topology ${\mathcal T}$ on the space $Z$ of equivalence classes of stable maps, defined later on,  is Hausdorff and completely regular. These two topological properties of $Z$ will be deduced from the  representation of $Z$ as the orbit space of a M-polyfold  groupoid $(X, {\bf X})$ whose object set $X$ is a disjoint union of M-polyfolds $\cg$ and whose morphism set ${\bf X}$ 
is a disjoint union of morphism sets $M(\cg, \cg')$.}

{It will become apparent later on, that in case we only wish to establish the Hausdorff property of $Z$, it suffices to show for any two given points $q\in \co$ and $q'\in \co'$ that every sequence 
$(\Phi_k)$ of morphisms satisfying $\pi\circ s(\Phi_k)\to q$ and $\pi'\circ t(\Phi_k)\to q'$  possesses a convergent subsequence.} 

{Finally, let us note that we can always replace 
the open neighborhood $U^{\cg}(q)$ quaranteed by Theorem \ref{rees} by a smaller open neighborhood for which  the conclusions of the theorem will still hold.}

\end{remark}


{As a preliminary step in the proof of Theorem \ref{rees} we fix an M-polyfold 
$\cg=\{(q, \alpha_q)\vert \, q\in \co\}$ and a point 
$q_0\in \co$, and start with the construction of the desired open neighborhood $U^\cg(q_0)$.}

{
Our good uniformizer centered around the stable map $(S, j, M, D, u)$   is denoted by 
$$
q\mapsto \alpha_q, 
\quad q=(a, v, \eta)\in \co.
$$
It has the properties listed in the good 
data according to Definition \ref{GD} which will be used heavily in the following constructions.}

{We denote the stabilization set on $(S, j, M, D)$ by $\si$ and view the set $M^\ast:=M\cup \si$ as an unordered set. The noded Riemann surface 
$(S, j, M^\ast, D)$ is stable and has the automorphism group $G^\ast$. The 
automorphism group $G$ of the stable map $(S, j, M, D, u)$ is a subgroup of $G^\ast$.  There exists an open $G^\ast$-invariant neighborhood $O$ of the origin $(0, 0)$  in the parameter space  
$ {\mathbb C}^{\# D}\times E$, where $E$ is the finite dimensional complex vector space which parametrizes the deformations $j(v)$ of the almost complex structure $j$. The map 
$$(a, v)\mapsto (S_a, j(a, v), M^\ast_a, D_a),\quad (a, v)\in O,$$ 
is a good uniformizing family centered at the stable noded surface $(S, j, M^\ast, D)$ according to Definition \ref{citiview}.
}

{
\begin{definition}\label{gf}
 A {\bf family of Riemannian metrics} $g_a$  for the family  $S_a$ of surfaces, where  $a\in B_{(1/2)}^{\# D}$,  consists
 of a choice of a 
Riemannian metric $g_a$ on 
$S_a\setminus |D_a|$ for every parameter value $a$, having the following properties.  The cylinders $Z_{a_{\{x,y\}}}^{\{x,y\}}$  are isometric
to the standard cylinder $[0,R]\times S^1$ equipped with the product metric (if $a_{\{x,y\}}\neq 0$). On  the core regions of $S_a$ (which can can be canonically identified) the metrics are all  the same.  If 
$a_{\{x,y\}}=0$,  the metrics on  the punctured disks $D_x\setminus\{x\}$ and $D_y\setminus\{y\}$ are  isometric to the corresponding half-cylinders $\R^+\times S^1$ and $\R^-\times S^1$, respectively.
\end{definition}
}
{
Every connected component of $S_a$ having its nodal points removed, is now equipped with a metric, which is complete and denoted by $d_a$.
Continuing  with the construction of the open neighborhood $U^\cg(q_0)$ we denote our good uniformizer  by
$$
q\mapsto  \alpha_q=(S_a,j(a,v),M_a,D_a,\oplus_a(\exp_{u}(\eta))),\quad  q\in\co,
$$
where $q=(a, v, \eta)\in \co$ and where $\co$ is an open neighborhood of the origin $(0, 0, 0)$ in the splicing core $K\subset {\mathbb C}^{\# D}\oplus  E\oplus H^{3,\delta_0}_c(u^\ast TQ)$.
We recall that if $(a,v,\eta)\in {\mathcal O}$, then $(a,v)\in O$ and $\eta$ is a section of $u^\ast\tilde{\mathcal O}$ for the specific open neighborhood $\tilde{\mathcal O}$  of the zero-section in $TQ$ introduced in (3) of 
Definition \ref{GD}. }

{Our distinguished point $q_0\in \co$ is given by 
$$q_0=(a_0, v_0,\eta_0).$$
We now take a small open neighborhood $V$ of 
$(a_0, v_0)\in O$ satisfying $\cl (V)\subset O$. 
For the symplectic manifold $(Q,\omega)$ we choose a Riemaniann metric $d_Q$. Then with the metric $g_{a_0}$ already fixed on 
$S_{a_0}\setminus \abs{D_{a_0}}$, the tangent map 
$Tu_0$ of the map $u_0:=\oplus_{a_0}\exp_u(\eta_0):S_{a_0}\setminus \abs{D_{a_0}}\to Q$ is uniformly bounded, in view of the exponential convergence of $\eta_0$ at the ends of the infinite cylinders. }

{For a  given map $w:S_a\rightarrow Q$ we introduce the notation 
$$\norm{Tw}=\sup  \abs{Tw(z)h}_Q,$$
where the supremum is taken over all $z\in S_a\setminus \abs{D_a}$ and $h\in T_zS_a$ having  
norm at most equal to $1$. So, $\norm{Tu_0}<\infty$ as just explained. If $z\in S_a\setminus \abs{D_a}$, we set 
$$\norm{Tw(z)}=\sup \abs{Tw(z)h}_Q,$$
 where the supremum is take over 
all $h\in T_zS_a$ of  norm at most  $1$.
Finally, we choose 
a small open neighborhood $U_0$ of $\eta_0$ in $H^{3,\delta_0}_c(u^\ast TQ)$ having the following properties.
\begin{equation}\label{eq_properties_2}
\begin{aligned}
&\quad (1)\quad  \bigl(\cl (V)\oplus \cl (U_0)\bigr)\cap K \subset \co\\
&\quad (2)\quad 
\text{There exists  a constant $C>0$ such that}
\norm{T(\oplus_a(\exp_u(\eta)))}\leq C\\
&\quad \phantom{(2)}\quad  \text{\ for every $(a,v,\eta)\in V\oplus U_0$.}
\end{aligned}
\end{equation}
The desired neighborhood $U^\cg(q_0)$ is now defined as  
\begin{equation}\label{neighborhood}
U^\cg (q_0) = (V\oplus U_0)\cap K.
\end{equation}
We  note that the closure of $U^\cg(q_0)$ in $\C^{\# D}\oplus  E\oplus H^{3,\delta_0}_c(u^\ast TQ)$ is a subset
of $\co$.
}

\begin{proof}[Proof of Theorem \ref{rees}]
{
We fix $q_0\in \co$ and $U^\cg(q_0)$ as just constructed and consider  a second M-polyfold 
$$\cg'=\{(q', \alpha'_{q'})\vert \, q'=(a', v', \eta')\in \co'\}$$
where $q'\to \alpha_{q'}'$ is a also a good uniformizer around the stable map $(S', j', M', D', u')$ whose automorphism group we denote by $G'$. The stabilization of the noded surface $(S', j', M', D')$ is denoted by $\si'$. We let 
$$q_0'=(a_0', v_0', \eta_0')\in \co'$$
 be the  second distinguished  point. We consider a sequence $(\Phi_k)\in M(\cg, \cg')$ of morphisms 
satisfying
$$
\pi'\circ t(\Phi_k)\rightarrow q_0'\quad \text{and}\quad \pi\circ s(\Phi_k)\in U^\cg(q_0).
$$
and shall show that the sequence $(\Phi_k)$ has a subsequence which converges in $M(\cg, \cg')$. }

{Without loss of generality we may assume, going over to a subsequence, that the sequence  
$$q_k:=\pi \circ s(\Phi_k)=:(a_k,v_k,\eta_k)$$
has the property that 
$$(a_k,v_k)\rightarrow (a,v)\in \cl (V).$$
By assumption, we have the convergence 
$$q_k'=\pi' \circ t(\Phi_k)=(a_k',v_k',\eta_{k}')\to q_0'=(a_0', v_0', \eta_0').$$
 }

{Associated with the sequence $(\Phi_k)\in M(\cg,\cg')$ is the sequence
$$
\phi_k:(S_{a_k},j(a_k,v_k),M_{a_k},D_{a_k})\rightarrow (S_{a_k'}',j'(a_k',v_k'),M_{a_k'}',D_{a_k'}')
$$
of isomorphisms  satisfying 
\begin{equation}\label{condition_5}
(\oplus'_{a_k'}(\exp'_{u'}(\eta_k')))\circ \phi_k = \oplus_{a_k}(\exp_{u}(\eta_k))
\end{equation}
for all $k$. Abbreviating the maps 
$$
u_k'=(\oplus'_{a_k'}(\exp'_{u'}(\eta_k')))\quad \text{and}\quad u_k= \oplus_{a_k}(\exp_{u}(\eta_k)),
$$
\eqref{condition_5} says that 
\begin{equation}\label{condition_6}
u_k'\circ \phi_k=u_k.
\end{equation}
By the assumptions of the theorem,
\begin{equation}\label{condition_7}
u_k'\to u_0'.
\end{equation}
}
{We claim that the gradients of the maps $\phi_k:S_{a_k}\setminus \abs{ D_{a_k} }\to S'_{a_k'}$ are uniformly bounded in $k$. 
\begin{lemma}\label{lemma_3.23}
$$\sup_{ z\in S_{a_k}\setminus \abs{D_{a_k} }, k\geq 1 } \norm{T\phi_k (z)}<\infty.$$
\end{lemma}
}

\begin{proof}[Proof of Lemma 
\ref{lemma_3.23}]
{
Arguing indirectly we
assume that there  is  subsequence (in k) of points $c_k\in S_{a_k}\setminus \abs{D_{a_k}}$ for which 
$$\norm{T\phi_k(c_k)}\to \infty.$$
For this, so called bubbling off sequence $(c_k)$,  we perform a bubbling off analysis. Such an analysis is carried out in detail in Appendix \ref{section5.1}. 
}
{
We start with the first case in which (possibly going over to a subsequence) the points $c_k$ as well as their image points $\phi_k(c_k)$ all lie in a finite distance to the respective core
according to the 
Definition \ref{def_finite_distance_core}. In this case, the result of the bubbling off analysis is a non-constant and injective holomorphic map $\psi:\C\to S_{a_0'}'$ which 
contains at most one marked point from $M_{a_0'}'$ and no nodal points from $\abs{D_{a_0'}'}$. After applying the removable singularity theorem for holomorphic mappings we obtain a non-constant holomorphic map $\wt{\psi}:S^2\to S_{a_0'}'$ whose image, in view of the Hurwitz theorem, has to be a spherical component of $S'_{a_0}$. If this spherical component is stable, it carries at least $3$ special points from 
$M_{a_0'}'\cup \abs{D_{a_0'}'}$, but only $2$ of them can lie in the image of $\wt{\psi}$.
Hence $\wt{\psi}$ cannot be surjective and this case is not possible. If the image of 
$\wt{\psi}$ is an unstable spherical  component of $S_{a_0'}'$, it carries at least one stabilization point 
$z_0'$ from $\si'$. Now we recall that the part of the surfaces $S_{a_k'}'$ at finite distance 
to the core,  can be canonically identified. By definition of a stabilization point, there exists a disk-like neighborhood $D'_{a_0'}(\equiv D'_{a_k'})$  centered at $z_0'$ and  having a smooth boundary $\partial D'_{a_0'}$,  on which the maps $u_k':D'_{a_0'}\to Q$ are all embeddings. We know that 
$u'_k\to u_0'$. In view of $\norm{T\phi_k(c_k)}\to \infty$ we find two sequences $(z^1_k)$ and $(z_k^2)$ on the same component of $S_{a_k}$ which, denoting their images in $S_{a_k'}'$ by  $z_k^{1'}=\phi_k(z_k^1)$ and $z_k^{2'}=\phi_k(z_k^2)$,  have the following properties.
\begin{itemize}
\item[$\bullet$] 
$d_{a_k}(z_k^1,z_k^2)\rightarrow 0$ as $k\to \infty$.
\item[$\bullet$] $z_k^{1'}=z_0'$ and $z^{2'}_k\in \partial D'_{a_0'}.$
\end{itemize}
Hence the distance between  $z_k^{1'}$ and $z_k^{2'}$  is bounded away from $0$
uniformly in $k$.  Since the restrictions of $u_k'$ onto the disk $D'_{a_0'}$ are embeddings and $u_k'\to u_0'$, also the distance
$d_Q(u_k'(z_0'),u_k'(z_k^{2'}))$ is bounded away from $0$ as $k\rightarrow \infty$. }

{
On the other hand, 
$$
d_Q(u_k(z_k^1),u_k(z_k^2))\rightarrow 0
$$
as $k\rightarrow 0$ because, by the assumption \eqref{eq_properties_2} on $U^\cg(q_0)$ the gradients of the maps $u_k$ are  uniformly  bounded in $k$  and
$d_{a_k}(z_k^1,z_k^2)\rightarrow 0$.  However, from $u_k'\circ \phi_k=u_k$ we conclude that 
$$
d_Q(u_k'(z_0'),u_k'(z_k^{2'}))=d_Q(u_k(z_k^1),u_k(z_k^2)).
$$
We again have arrived at a contradiction and see that the first case cannot occur. }

{In the second case we assume that the sequence $(c_k)$ stay in finite distance to the core,   but the images $\phi_k(c_k)$ lie in  cylinders with increasing distance to the core as $k\rightarrow\infty$.
In this case the bubbling off analysis produces an injective non-constant holomorphic map $\psi:\C\to \R\times S^1$, which is impossible.
}

{Now we assume that the  points $c_k$  lie in cylinders,  with increasing distance to the core as $k\rightarrow\infty$. If the images $\phi_k(c_k)$ 
are also in cylinders and  have  increasing distance to the core, we obtain a contradiction as in the previous case.
}

{In the final case the  points $c_k$ lie in cylinders with increasing distance to the core and $\phi (c_k)$ in finite distance
to the core. In this  case the bubbling-off analysis produces an injective and non-constant holomorphic  map $\psi:\C\rightarrow S_{a_0'}'$
which can have at most one nodal point in its  image and no marked points.  Now the arguments of  the first case lead to a contradiction.}

Having exhausted all cases of possible bubbling-off, we conclude that the gradients  $\phi_k$ are indeed uniformly bounded.
\end{proof}

{
Continuing with the proof of  Theorem \ref{rees}, 
the uniform gradient bounds for the sequence $(\phi_k)$ of isomorphisms imply,  by Ascoli-Arzela's theorem,  a $C^\infty_{\textrm{loc}}$ convergence of a subsequence of $(\phi_k)$ in finite distance to the core. 
Using Proposition \ref{gotham} we can pass to the limit $\phi_0=\lim_{k\to \infty}\phi_k$ in $C^\infty_{\textrm{loc}}$. Using 
the postulated convergence $\eta'_k\to \eta'_0$ we conclude from formula 
\eqref{condition_5},  in view of Theorem \ref{theorem_neck},  that $\eta_k$ converges in finite distance 
to the core in the Sobolev norm $H^3$ and then, by Proposition \ref{gotham}, in the neck regions in the $H^{3,\delta_0}_{c}$-norm if the corresponding gluing parameter goes to $0$. We call the limit section $\wh{\eta}_0=\lim_{k\to \infty}\eta_k$. }


{Recalling  that 
$$\pi' \circ t(\Phi_k)=q_k'=(a_k',v_k',\eta_{k}')\to q_0'=(a_0', v_0', \eta_0')\quad \text{in $\co'$,}$$
we have proved so far also 
the convergence 
$$\pi\circ s(\Phi_k)=q_k=(a_k, v_k, \eta_k)\to  \wh{q}=(a, v, \wh{\eta}_0)\quad \text{in $\co$. }$$
Moreover, we have found the element 
$$\Phi_0=\bigl((\wh{q}_0,\alpha_{\wh{q}_0}),  \phi_0, (q_0', \alpha'_{q'_0})\bigr)\in M(\cg, \cg'),$$
in which $\phi_0:\alpha_{\wh{q}_0} \to \alpha'_{q'_0}$ is an isomorphism between the stable maps.}
 
 { 
 To show that $\lim_{k\to \infty}\Phi_k=\Phi_0$ in $M(\cg, \cg')$, we take a set 
 $\wh{\mathcal V}$ in the basis of the topology of $M(\cg,\cg')$ containing 
 $\Phi_0$,  and show that 
 $\Phi_k\in \wh{\mathcal V}$ for $k$ large.  By definition (based on Theorem \ref{key-z}), 
$$\wh{\mathcal V}=\{\bigl((q, \alpha_q), \psi_q, (f(q), \alpha'_{f(q)})\bigr)\vert \, q\in \wh{\co}\subset \co\},$$
where $\wh{\co}\subset \co$ is a small open neighborhood of $\wh{q}_0\in \co$. 
The local germ of sc-diffeomorphism 
$f:(\wh{\co}, \wh{q}_0)\to (\wh{\co}', q_0')$ between the parameter spaces  satisfies 
$f(\wh{q}_0)=q_0'$. The core-smooth germ of a family $q\mapsto \psi_q$ of isomorphisms 
$\psi_q:\alpha_q\to \alpha'_{f(q)}$ satisfies $\psi_{q_0}=\phi_0$.
}

{
Since  
$$
\psi_{q_k}\circ \phi_k^{-1} :\alpha_{q_k'}'\rightarrow \alpha_{f(q_k)}'
$$
is an isomorphism, there exists, by 
Proposition \ref{reremark1}, an isomorphism $g'\in G'$ of the stable map 
$(S', j', M', D', u')$ satisfying 
$$f(q_k)=g'\ast q'_k\quad \text{and}\quad 
\psi_{q_k}\circ \phi_k^{-1}=(g')_{a_k'}.$$
}
{
In finite distance to core, the sequence 
$\psi_{q_k}\circ \phi_k^{-1}$ converges 
in $C^\infty$ to the indentity map on $S'_{a_0'}$. Therefore, $g'=\text{id}\in G'$. Hence $(g')_{a_k'}=\text{id}$ on $S'_{a_k'}$ and so, $\psi_{q_k}=\phi_k$. Consequently, 
$\Phi_k=\bigl((q_k, \alpha_{q_k}), \phi_k, (q_k', \alpha'_{q_k'})\bigr)\in \wh{\mathcal V}$ for $k$ large and the proof of Theorem \ref{rees} is complete.}
\end{proof}

{\section{The Topology on  $Z$}\label{natural_topology_Z_section}}

We consider a good uniformizer $q\rightarrow \alpha_q$ of {stable maps}  defined on the second countable paracompact Hausdorff M-polyfold $\co$. The associated graph 
$\cg=\{(q, \alpha_q)\vert \, q\in \co\}$ is equipped with the M-polyfold structure making the projection
$\pi:\cg\rightarrow {\mathcal O}$ an sc-diffeomorphism. 
The space ${\mathcal O}$ comes with a finite group action by sc-diffeomorphisms
$$
G\times{\mathcal O}\rightarrow {\mathcal O},\quad (g,q)\mapsto  g\ast q.
$$
The quotient space $G\backslash {\mathcal O}$ is again a second countable paracompact Hausdorff space and the projection map
$$
\pi:{\mathcal O}\rightarrow G\backslash{\mathcal O}
$$
is trivially open since $\pi^{-1}(\pi(U))$ is a finite union
of the diffeomorphic images of $U$ under the  action of $g\in G$. 
In view of Proposition \ref{reremark1}, 
the map
$$
G\backslash{\mathcal O}\rightarrow Z,\quad [a,v,\eta]\mapsto  [S_a,j(a,v),M_a,D_a,\oplus_a(\exp_u(\eta))]
$$
into the space of stable curves is injective.
This implies that there is at most  one choice of sets  defining the topology on $Z$ such  that all such maps are homeomorphisms onto  open sets. 

This leads us to the following definition of a collection ${\mathcal T}$ of subsets of $Z$ as follows.
\begin{definition}
A subset $U$ of $Z$ belongs to ${\mathcal T}$ provided for every $z\in U$ there exists a good uniformizer
${\mathcal O}\ni q\mapsto \alpha_q$ satisfying  $z=[\alpha_{q_0}]$ for some $q_0\in {\mathcal O}$,  and an open neighborhood
$V(q_0)\subset {\mathcal O}$ such that $\{[\alpha_q]\ |\ q\in V(q_0)\}\subset U$.
\end{definition}

The main  result in this section is the following theorem which shows that ${\mathcal T}$ is a metrizable topology.

\begin{theorem}\label{nattop}
The collection ${\mathcal T}$ of subsets of the space $Z$ of stable curves  is a second countable, paracompact, Hausdorff topology which also is completely regular. In view of Urysohn's theorem the topology ${\mathcal T}$ is in particular metrizable.
Moreover the un-noded elements in $Z$ constitute an open and dense subset. There exists a sequence of smooth elements in $Z$ representing
un-noded elements which is dense in $Z$. In particular,  $Z$ is separable.
\end{theorem}

The proof of Theorem \ref{nattop} will be split into several parts.

\begin{lemma}\label{topology}
The collection ${\mathcal T}$ of subsets of $Z$ is a topology on $Z$.
\end{lemma}
\begin{proof}
It is clear that $\emptyset\in {\mathcal T}$. Also $Z\in {\mathcal T}$ since for every element $z\in Z$ there exists
a good uniformizer $q\mapsto \alpha_q$ with $z=[\alpha_{q_0}]$ as was already shown. It follows immediately from the definition
of ${\mathcal T}$ that the  union of elements in ${\mathcal T}$ also belongs to ${\mathcal T}$. If   $U_1,U_2\in {\mathcal T}$, 
we shall show that $U=U_1\cap U_2$ belongs to ${\mathcal T}$. Pick $z\in U$. We find two good uniformizers
${\mathcal O}\ni q\mapsto\alpha_q$ and ${\mathcal O}'\ni q'\mapsto \alpha_{q'}'$ such  that for some $q_0\in \co$ and 
$q'_0\in \co'$, 
$$
z=[\alpha_{q_0}]=[\alpha_{q_0'}'].
$$
Moreover,  there exist open sets  $V(q_0)\subset {\mathcal O}$ and $V'(q_0')\subset {\mathcal O}'$ for which 
$$
\{[\alpha_q]\ |\ q\in V(q_0)\}\subset U_1\ \ \text{and}\ \ \{[\alpha_{q'}']\ |\ q'\in V'(q_0')\}\subset U_2.
$$
Let  $\phi_0:\alpha_{q_0}\rightarrow \alpha_{q_0'}'$ be an isomorphism.  In view of the fundamental Theorem \ref{key-z}
there exists a germ of sc-diffeomorphism $f:({\mathcal O},q_0)\rightarrow ({\mathcal O}',q_0')$ satisfying  $f(q_0)=q_0'$ 
and a core smooth family of isomorphisms $\phi_q$, 
$${\phi_q}:
\alpha_q\rightarrow\alpha_{f(q)}, 
$$
defined for $q$ near $q_0$  and satisfying $\phi_{q_0}=\phi_0$. This implies, for a suitable small open neighborhood $V_1(q_0)$, that the following properties hold.
\begin{itemize}
\item[(1)] $q_0\in V_1(q_0)\subset V(q_0)\subset {\mathcal O}$.
\item[(2)] $V'_1(q_0'):=f(V_1(q_0))$ is open and $q_0'\in V_1'(q_0')\subset V'(q_0')\subset {\mathcal O}'$.
\item[(3)] $[\alpha_q]=[\alpha_{f(q)}']$ for all $q\in V_1(q_0)$.
\end{itemize}
Hence 
$$
\{[\alpha_q]\ |\ q\in V_1(q_0)\}=\{[\alpha_{q'}']\ |\ q'\in V_1'(q_0')\}\subset U_1\cap U_2 =U.
$$
This completes the proof that ${\mathcal T}$ is a topology on $Z$.  
\end{proof}

\begin{lemma}\label{cb}
Every point in the topological space $(Z,{\mathcal T})$ has a countable basis of open neighborhoods.
\end{lemma}
\begin{proof}
Fix $z\in Z$ and take a good unitormizer ${\mathcal O}\ni q\rightarrow \alpha_q$ so that $[\alpha_{q_0}]=z$ for a suitable $q_0\in {\mathcal O}$.  Since ${\mathcal O}$ is a subset of a metric space,  we may assume that $\co$  is equipped with a metric $d$. We 
take a sequence of positive numbers $(\varepsilon_k)$ converging to $0$ and consider the associated open balls $B_{\varepsilon_k}(q_0)$.
Then the associated open sets $V_k\subset Z$,  defined by
$$
V_k=\{[\alpha_q]\ |\ q\in B_{\varepsilon_k}(q_0)\},
$$
constitute  a countable neighborhood basis at  the point $z$. Note that given any other family $q'\rightarrow \alpha_{q'}'$ with
$[\alpha_{q_0'}']=z$,  the associated construction $V'_{k'}$ for a suitable sequence $(\varepsilon_{k'}')$ will lead to  an equivalent basis which can 
easily proved by utilizing  Theorem \ref{key-z} as above.
Indeed, via this theorem one proves that for given $k$ there exist $k'$ and $k''$ such that
$$
z\in V_{k''}\subset V_{k'}'\subset V_k.
$$
\end{proof}
\begin{lemma}\label{hausdorff}
The topology ${\mathcal T}$ is Hausdorff.
\end{lemma}
\begin{proof}
Let $z, z'\in Z$ be two different points. Using the fact that ${\mathcal T}$ has for every point a countable basis of neighborhoods
(Lemma \ref{cb}) we can take a basis of open neighborhoods $(V_k)$ of $z$
and $(V_{k'}')$ of $z'$ so that
$$
V_{k+1}\subset V_k\quad  \text{and}\quad  V_{k+1}'\subset V_k'.
$$
Arguing indirectly we may assume that $V_k\cap V_{k}'\neq \emptyset$ for all $k$,  
so that we find  an element $z_k\in V_k\cap V_k'$ for every $k$.

{We choose  a good uniformizer 
$q\mapsto \alpha_q$ on $\co$
such that $[\alpha_{q_0}]=z$ for some $q_0\in \co$ and abbreviate 
$\cg=\{(q,\alpha_q)\vert \, q\in \co\}$. 
Theorem \ref{rees} guarantees the distinguished open neighborhood 
$U^\cg(q_0)\subset {\mathcal O}$  of $q_0$.  If $k$ is  large, we may assume that  $z_k=[\alpha_{q_k}]$ and $q_k\to q_0$. 
Similarly, we choose for $z'$  the second good uniformizer $q'\rightarrow \alpha_{q'}'$ on ${\mathcal O}'$ such that 
$[\alpha_{q_0'}']=z'$ and represent the sequence $z_k$ by $q_k'\in {\mathcal O}'$ satisfying $q_k'\rightarrow q_0'$. We 
abbreviate $\cg'=\{(q',\alpha'_{q'})\vert \, q'\in \co'\}$. 
Since $[\alpha_{q_k}]=[\alpha_{q_k'}']=z_k$,  we find a sequence 
$\phi_k:\alpha_{q_k}\to \alpha'_{q_k'}$ of isomorphisms which defines the sequence 
$$
\Phi_k=\bigl((q_k,\alpha_k), \phi_k, (q_{k}',\alpha_{q_k'}')\bigr)
$$
of morphisms in $M(\cg,\cg')$. By Theorem \ref{rees}  there exists a convergent subsequence $\Phi_{k_j}\rightarrow \Phi$ in $M(\cg,\cg')$. In view of $q_k\rightarrow q_0$ and  $q_k'\rightarrow q_0'$, the limit  has the form 
$$\Phi=((q_0,\alpha_{q_0}),\phi_0,q_0',(\alpha_{q_0'}'))$$ for the isomorphism $\phi_0:\alpha_{q_0}\to \alpha'_{q_0'}$. 
Therefore, $[\alpha_{q_0}]=[\alpha'_{q_0'}]$ in contradiction to the assumption 
$z\neq z'$. Hence the topology ${\mathcal T}$ must be Hausdorff.}
\end{proof}
\begin{lemma}\label{cr}
The topology ${\mathcal T}$ is completely regular.
\end{lemma}
\begin{proof}
We consider a point $z\in Z$ and a closed subset $A\subset Z$ not containing $z$. We have to construct a continuous map
$\gamma:Z\rightarrow [0,1]$ satisfying $\gamma(z)=0$ and $\gamma(a)=1$ for $a\in A$.

We choose a good uniformizer ${\mathcal O}\ni q\rightarrow \alpha_q$  such  that 
$[\alpha_{q_0}]=z$ for a suitable $q_0\in {\mathcal O}$. We find an open neighborhood $V(q_0)\subset {\mathcal O}$ so that
$A$ and the open neighborhood $V=\{[\alpha_q]\ |\ q\in V(q_0)\}$ of $z\in Z$,  are disjoint.  We choose  an open neighborhood $U^\cg(q_0)\subset {\mathcal O}$ of $q_0$
whose closure (in ${\mathcal O})$ is contained in $V(q_0)$, 
$$
\cl_{\mathcal O}(U^\cg(q_0))\subset V(q_0), 
$$
and such that  for $U^\cg(q_0)$ the assertion of Theorem \ref{rees} holds (see the the remark after this theorem).  If $G_{q_0}=\{g\in G\vert g\ast q_0=q_0\}$ is the stabilizer of $q_0$
we may assume by  replacing $U^\cg(q_0)$ and $V(q_0)$  by  possibly smaller neighborhoods,  that,  in addition to the already stated
properties,  the open sets $U^\cg(q_0)$ and $V(q_0)$ are  invariant under $G_{q_0}$. We choose a continuous map
$\tilde{\gamma}:V(q_0)\rightarrow [0,1]$ which vanishes at $q_0$ and is $G_{q_0}$-invariant and takes
the value $1$ on $V (q_0)\setminus U^\cg(q_0)$. Passing to $Z$ we obtain a continuous map $\gamma_1:V\rightarrow [0,1]$ on the 
 open neighborhood $V$ of $z$ which is  disjoint from  $A$. 
 
 The open neighborhood 
$U=\{[\alpha_q]\vert \, q\in U^\cg(q_0)\}$ of $z$ is contained in $V$, and 
$$\text{$\gamma_1(z)=0$ \quad and \quad $\gamma_1(\zeta)=1$ for $\zeta \in V\setminus U$.}$$
Note that we cannot conclude without further work that the closure of $U$ in $Z$ is contained in $V$ despite the fact that
$\cl_{\mathcal O}(U^\cg(q_0))\subset V(q_0)$.
Let us define the map $\gamma:Z\rightarrow [0,1]$ by extending the map $\gamma_1$ outside
of $V$ by $1$. The main work consists in proving the continuity of the map $\gamma$. Since, By Lemma \ref{cb},  every point has
a countable neighborhood basis of open sets,  we can argue with sequences. 

{
Let $b\in Z$ and let $b_k\rightarrow b$.
We distinguish the two cases $b\in V$ and $b\not \in V$,  and start with the case that $b\in V$. Then $b=[\alpha_{q_b}]$ for some $q_b\in {\mathcal O}$ and  $b_k=[\alpha_{q_k}]$ for $k$ large,  where 
$q_k\rightarrow q_b$ in ${\mathcal O}$.
Now,  $\gamma(b_k)=\tilde{\gamma}(q_k)\rightarrow \tilde{\gamma}(q_b)=\gamma(b)$ and $\gamma$ is continuous at such
a point $b$. }

{
Next we consider the second case, assuming that $b\in Z\setminus V$.
Then $\gamma(b)=1$, by definition of 
$\gamma$,  and we  have to show that $\gamma (b_k)\to 1$. If $b_k\not \in V$, then  $\gamma(b_k)=1$ and  so we may assume without loss of generality that $b_k\in V$ for all $k$. We take a 
good uniformizer $q''\rightarrow \alpha_{q''}''$ on $\co''$ satisfying  $b=[\alpha_{q''_0}'']$ for some $q_0''\in \co''$. 
For large
$k$ we can represent the points $b_k$ in the form $b_k=[\alpha''_{q_k''}]$ where $q_k''\rightarrow q_0''$ in ${\mathcal O}''$. 
We can also represent  them  using the good uniformizer $q\mapsto \alpha_q$ on $\co$ in the form $b_k=[\alpha_{q_k}]$ and  $q_k\in {\mathcal O}$. If $q_k\not\in V(q_0)$,  then
$\gamma(b_k)=1$ and if $q_k\in V(q_0)\setminus U^\cg(q_0)$ then also $\gamma (b_k)=1$.  Hence we may assume that $q_k \in U^\cg(q_0)$.}

{ We shall see that this case is impossible,
which then completes the proof that $\gamma$ is continuous at $b$.}

{ Since $[\alpha_{q_k}]=[\alpha''_{q_k''}]$, there exists a sequence 
$
\phi_k:\alpha_{q_k}\rightarrow \alpha_{q_k''}''
$
of isomorphisms. Abbreviating $\cg''=\{(q'', \alpha''_{q''})\vert \, q''\in \co''\},$ we obtain the sequence 
$$\Phi_k=\bigl((q_k, \alpha_{q_k}), \phi_k, (q''_k, \alpha''_{q_k''})\bigr)$$ 
of morphisms 
in $M(\cg, \cg'').$ 
Since $q_k\in U^\cg(q_0)$, we know from Theorem \ref{rees} that every subsequence  $(\Phi_{k_j})$ of the sequence $(\Phi_k)$ possesses a 
a convergent subsequence $(\Phi_{k_{j_l}})$ so that 
$$
\Phi=
\lim_{l\rightarrow\infty} \Phi_{k_{j_l}}\quad \text{in $M(\cg, \cg'')$}.
$$
Consequently, there exists an isomorphism $\phi_0:\alpha_{q}\rightarrow \alpha_{q_0''}''$ for some  $q\in \cl(U^\cg(q_0))\subset V(q_0)\subset \co$.
Therefore, 
$b=[\alpha_{q_0''}'']=[\alpha_{q}]\in V$ contradicting our assumption that $b\in Z\setminus  V$ and completing the proof of Lemma \ref{cr}. 
}
\end{proof}

Continuing with the proof of Theorem \ref{nattop}, we consider triples $(S,M,D)$ consisting of a closed oriented surface $S$ having perhaps several connected components,  equipped with a finite ordered
set $M$ of marked points, and a finite set of nodal pairs $D$, such  that $|D|\cap M=\emptyset$. Two such triples
$(S,M,D)$ and $(S',M',D')$ are called diffeomorphic if  there exists a smooth orientation-preserving diffeomorphism
$\phi:S\rightarrow S'$ mapping $M$ bijectively to $M'$, preserving the ordering of the points, and mapping $D$ to $D'$.
The diffeomorphism class of $(S,M,D)$, denoted by $[[S,M,D]]$,  is the class of triples diffeomorphic to $(S,M,D)$.
We denote by $\Delta$ the set of all diffeomorphism types. It is countable set.

On our space $Z$ of stable curves, we  introduce the  map
$$ 
\text{type}:Z\rightarrow \Delta:[(S,j,M,D,u)]\rightarrow[[S,M,D]].
$$
Given a type $d\in\Delta$,  we denote by $Z_d$ the subset of $Z$ consisting of the elements having the prescribed type $d$.
For example,  if  
$d_{g,k}=[[S,M,\emptyset]]$, where $S$ is a  connected surface of genus $g$,  and  $M$ is an ordered set of $k$ marked points, then 
the set $Z_{d_{g,k}}$ consists of the un-noded elements of genus $g$ with $k$ {ordered} marked points.
Note that for certain types $d$,  the set $Z_d$ can be  empty. For example,  if $Q=T^{2n}$, then $Z_{d_{0,0}}=\emptyset$.
Indeed, a map $u:S^2\rightarrow T^{2n}$ is contractible and hence  the  tuple  $(S^2,j,\emptyset,\emptyset,u)$ cannot be a stable map.

{
For every type $d\in\Delta$,  the subset $Z_d\subset Z$,  equipped with the topology induced  from $(Z,{\mathcal T})$,  is, in view of Lemma \ref{hausdorff}, a Hausdorff
topological space,  denoted by 
$(Z_d,{\mathcal T}_d)$.}

{
We now fix a type $d\in\Delta$ represented  by the triple $(S,M,D)$.  Then every element 
$z\in Z_d$ has a representative of the form 
$(S,j,M,D,u)$, where $j$ is a smooth almost complex structure on $S$ which determines  the orientation of $S$,  and the map $u$ belongs to the space $H^{3,\delta_0}_c(S,Q)$. 
We recall that the map $u$ is 
 of class $H^3_{loc}$ away from the nodes  and at the nodes of class $(3,\delta_0)$ with matching nodal values over the nodal pairs. }
 
{ We note that  $H^{3,\delta_0}_c(S,Q)$  is a smooth Hilbert manifold. If  $f:Q\rightarrow {\mathbb R}^N$ is a smooth embedding of $Q$ into a high dimensional Euclidean space, we obtain the  smooth embedding  
$$F: H^{3,\delta_0}_c(S,Q)\rightarrow H^{3,\delta_0}_c(S,{\mathbb R}^N)$$  given by  the composition 
$
F(u)=f\circ u.
$
The topology on $H^{3,\delta_0}_c(S,Q)$  corresponds to the topology of the embedded 
space $F(H^{3,\delta_0}_c(S,Q))\subset 
H^{3,\delta_0}_c(S,{\mathbb R}^N)$. The embedding induces a complete metric on $H^{3,\delta_0}_c(S,Q)$,  denoted by $\tau$.}

{Also the space 
of smooth almost complex structures on $S$ determining the given orientation of $S$
is a complete metric space,  denoted  by $(J_d,\rho_d)$.  By $H_d\subset H^{3,\delta_0}_c(S,Q)$ we denote  the open subset
consisting}  {of all stable maps 
$(S, j, M, D, u)$.
We equip $H_d$  with the induced metric, denoted by $\tau_d.$
The canonical map
$$
\Phi_d:J_d\times H_d\rightarrow Z_d
$$
is defined by 
$\Phi_d(j,u)=[(S,j,M,D,u)].$}
\begin{lemma}\label{ium}
For every  type $d\in\Delta$ the canonical map $\Phi_d:J_d\times H_d\rightarrow Z_d$ is continuous and open.
Moreover,  the map $\Phi_d$ is surjective.
\end{lemma}
\begin{proof}
We fix a point $(j_0,u_0)\in J_d\times H_d$ and let 
{
$$\alpha_0=
\Phi_d(j_0,u_0)=(S,j_0,M,D,u_0).$$ be the associated stable map which we equip with the stabilization $\si$.}
{  In view of Theorem \ref{blog}, there is a good uniformizing family 
$$q\mapsto \alpha_q=(S_a, j(a, v), M_a, D_a, \oplus_a\exp_{u_0'}\eta)$$
 of stable maps, where 
$q=(a, v,\eta)\in \co$. It is  centered at the stable map
$
\alpha:=(S,j_0,M,D,u_0')
$
for a smooth map $u_0'$ near $u_0$. The family $\alpha_q$ satisfies the properties of good data in 
Definition \ref{def_g_uni_f_stable_maps}. Moreover, 
at $a=0$, 
$$\alpha_{q_0}=\alpha_0\quad \text{and}\quad 
 q_0=(0,0,\eta_0)$$
 where $u_0=\exp_{u_0'}(\eta_0)$.
We recall that if the gluing parameter vanishes, there is no gluing and $\oplus_a$  is the identity operator if $a=0$. We introduce the abbreviated notation, for $a=0$,
$$j(0, v)=j_0(v),\quad S_0=S, $$
so that $j(0, 0)=j_0(0)=j_0$ on $S$.
Given an almost complex structure $j$ on $S$ close to $j_0$ and a deformation $\si'$ of the stabilization $\si$ close to $\si$ (as defined in Section \ref{dm-subsect}) we find by 
Theorem \ref{smoothfamily} a unique pair  $(\phi_{j, \si'}, v(j, \si'))$ such that 
$$\phi_{j, \si'}:(S, j, M\cup \si', D)\to 
(S, j_0(v(j, \si')), M\cup \si,  D)$$ is an isomorphism of marked  Riemann surfaces close to the identity map of $S$. Moreover, 
$\phi_{j, \si'}$ and $v(j, \si')$ depend continuously on $j$ and $\si'$ and satisfy 
$\phi_{j, \si'}=\text{id}$ and $v(j_0, \si')=0$.
Every stable map $u$ on $(S, j_0, M, D)$ close to $u_0$ intersects the constraints at $u_0(\si)$ in $Q$ transversally and, since $u_0$ near $\si$ is an embedding,  defines 
the deformation $\si_u\subset S$ of $\si$.  This way we obtain the continuous map 
$$(j, u)\mapsto (v(j, \si_u), u\circ \phi_{j, \si_u}^{-1}).$$
Defining the section $\eta(j, u)$ close to $\eta_0$ by the equation
$$\exp_{u_0'}(\eta (j, u))=u\circ \phi^{-1}_{j, \si_u}, $$
we have $\eta (j_0, u_0)=\eta_0$, and 
$$[(S, j_0(v(j, \si_u)), M, D, \exp_{u_0'}(\eta (j, u))=[(S, j, M, D, u)].$$
The continuity of the map 
$(j, u)\mapsto (0, v(j, \si_u), \eta(j, \si_u))\in \co$ implies, in view of the definition of the topology of $H_d$, that the map $\Phi_d$ is continuous in a small neighborhood of 
$(j_0, u_0)$. }

{In order to prove that the map $\Phi_d$ is open we take an open subset $U$ of $J_d\times H_d$ and consider its  image $\wh{U}=\Phi_d (U)$. If $[(S, j_0, M, D, u_0)\in \wh{U}$, then $(j_0, u_0)\in U$. Proceeding as above
we choose a small stable map $u_0'$ near $u_0$ and take a good uniformizer 
$q\mapsto \alpha_q$ centered at $(S, j_0, M, D, u_0')$ such that 
$[\alpha_{q_0}]=[(S, j_0, M, D, u_0)]$ at 
the point $q_0=(0, 0, \eta_0)\in \co$, where 
 $\exp_{u_0'}(\eta_0)=u_0$.  For $\abs{v}$ 
 small and $\eta$ near $\eta_0$,  the map 
 $$(v, \eta)\mapsto (j_0(v), \exp_{u_0'}(\eta))$$
 into $U$ is continuous and 
 $$\Phi_d(j_0(v), \exp_{u_0'}(\eta))=
 [(S, j_0(v), M, D, \exp_{u_0'}(\eta))].$$
 In view of the definition of the topology of $Z_d$ we have found an open neighborhood around 
 $ [(S, j_0, M, D, u_0)]$ which belongs to $\wh{U}$. Hence the map $\Phi_d$ is an open map. The surjectivity of $\Phi_d$ is obvious and the proof of Lemma \ref{ium} is complete.}

\end{proof}

\begin{lemma}\label{type-d}
For every type $d\in \Delta$,  the Hausdorff topological space $(Z_d,{\mathcal T}_d)$ admits a countable basis for the topology.
\end{lemma}
\begin{proof}
We fix the type $d\in \Delta$  and recall that,  by Lemma \ref{ium},  the map
$$
\Phi_d:J_d\times H_d\rightarrow Z_d, \quad (j,u)\mapsto [(S,j,M,D,u)]
$$
is open and continuous. Since $J_d \times H_d$ is a separable metric space it posses a countable basis for its topology,
say $\tilde{\mathcal B}_d= {(\tilde{U}_i)}_{i\in {\mathbb N}}$. Define the open sets $U_i=\Phi(\tilde{U}_i)$ and denote its collection by
${\mathcal B}_d$. Let us show that it is a basis for the topology ${\mathcal T}_d$. Pick an open subset $V\subset  Z_d$.
Then let $z\in V$ and pick
$(j_0,u_0)\in J_d\times H_d$ such that $z=\Phi_d(j_0,u_0)$. 
By the continuity of $\Phi_d$ we find an open neighborhood $U$ of $(j_0,u_0)$
such  that $\Phi_d(U)\subset V$. Then we find $\wt{U}_{i_0}\in \wt{\mathcal B}_d$ 
such that  $(j_0,u_0)\in \wt{U}_{i_0}\subset \wt{U}$. This then implies that 
$$
z\in U_{i_0}\subset V.
$$
Since $z$ was arbitrary in $V$ we conclude that
$$
V =\bigcup_{\{i\ |\ U_i\subset V\}} U_i,
$$
which shows  that ${\mathcal B}$ is a basis for the topology ${\mathcal T}_d$ on $Z_d$.
\end{proof} 

\begin{lemma}\label{2nd-count}
The topology ${\mathcal T}$ is second countable.
\end{lemma}
\begin{proof}
Recall that the set $\Delta$ of types $d$ is countable and that every space $(Z_d,{\mathcal T}_d)$ has a countable
basis ${\mathcal B}_d$ for its topology. Hence we have a countable collection $(U_{d,k})$, where $d\in \Delta$ and $k\in {\mathbb N}$. Further the collection ${(U_{d,k})}_{k\in {\mathbb N}}$ is a basis for ${\mathcal T}_d$.
Fix a type $d\in\Delta$. For every $z\in Z_d$ we have a good uniformizing family
$$
{\mathcal O}_z\ni q^z\rightarrow \alpha_{q^z}^z
$$
such that $ [\alpha^z_{q^z_0}]=z$
for a suitable $q^z_0$.  In addition,  the uniformizer is constructed 
with an element of the same type.
In addition, the subset  
$$
V_z=\{[\alpha^z_{q^z}]\ |\ q^z\in {\mathcal O}_z\}\subset Z
$$
is an open neighborhood of $z\in Z$ and $V_z^d=Z_d\cap V_z$ is an open neighborhood of $z$ in $Z_d$.
Also the collections ${(V_z^d)}_{z\in Z_d}$ is an open covering of $Z_d$. Consider the set $\Sigma_d\subset {\mathbb N}$ of integers 
consisting of all $i\in {\mathbb N}$ such  that $U_{d,i}$ belongs to at least one of the  sets $V^d_z$. Since $(V_z^d)$ is an open covering of $Z_d$ and 
${(U_{d,i})}_{i\in {\mathbb N}}$ is a basis for ${\mathcal T}_d$ we conclude that
$$
Z_d =\bigcup_{i\in\Sigma_d} U_{d,i}.
$$
Now we choose for every $i\in\Sigma_d$ a point $z_{d,i}\in Z_d$ such  that $U_{d,i}\subset V_{z_{d,i}}^d$. Then 
the collection of all ${(V_{z_{d,i}}^d)}_{i\in\Sigma_d}$ is an open covering of $Z_d$. We do this for every type $d$ and obtain
the index set
$$
\Sigma : = \bigcup_{d\in\mathfrak{d}} \{d\}\times \Sigma_d
$$
and a map
$$
\Sigma\rightarrow Z:(d,i)\rightarrow z_{d,i}.
$$
Then   we  take the countable collection of all good uniformizers 
$$
{\mathcal O}_{z_{d,i}}\ni q^{z_{d,i}}\rightarrow \alpha^{z_{d,i}}_{q^{z_{d,i}}}.
$$
By construction,  the associated collection of open sets ${(V_{z_{d,i}})}_{(d,i)\in\Sigma}$ covers $Z$.
Every $V_{z_{d,i}}$ carries a metrizable topology and is separable. Hence its admits a countable
basis for its topology. Taking the countable union of all these bases for the $V_{z_{d,i}}$, where $(d,i)$ varies over the countable set
$\Sigma$,  we obtain a countable basis for
$(Z,{\mathcal T})$. 
\end{proof}

\begin{proof}[Proof of Theorem \ref{nattop}]
At this point we know by the previous discussion that ${\mathcal T}$ is a Hausdorff topology (Lemma \ref{topology}, Lemma \ref{hausdorff}), which is second countable (Lemma \ref{2nd-count})
and completely regular (Lemma \ref{cr}). As a consequence of Urysohn's metrization  theorem it is therefore metrizable and consequently also 
paracompact. 

Next we consider the space of un-noded elements in $Z$, denoted  by $\dot{Z}$. From the construction of the uniformizers
it follows that $\dot{Z}$ is open. Indeed, if we start without nodes the good uniformizer does not involve nodes. Hence the original
element has an open neighborhood not containing nodal elements. If $z$ is noded then any neighborhood contains
un-noded elements. This follows immediately from the definition of the topology via good uniformizers. Let $d\in \Delta$ be a type, {of un-noded elements in  $Z$},  and consider the map $\Phi_d:J_d\times H_d\rightarrow Z_d$. The metric space $J_d\times H_d$ is separable and therefore
admits a dense sequence $(j_l,u_l)$. We may assume that the maps $u_l$ have regularity $(3+m,\delta_m)$ for all $m$.
Then the sequence $(\Phi(j_l,u_l))$ consists of smooth elements and is dense in  $Z_d$ because  $\Phi$ is open and surjective.
Since we only have a countable number of such types $d\in \Delta$, 
we conclude that $\dot{Z}$ is separable. Because $\dot{Z}$ is open and dense in $Z$,  we see that  the space $Z$ of stable curves is separable. 
This completes the proof of the Theorem \ref{nattop}.
\end{proof}

\section{The Polyfold Structure on the Space $Z$}\label{polyfoldstructure}
After all these preliminaries we finally construct the polyfold structure of the topological space $Z$ of stable curves. For every $z\in Z$ we find by Proposition \ref{imremark} a good uniformizing family $(a, v, \eta)\to \alpha_{(a, v, \eta)}$,  where $(a, v, \eta)\in {\mathcal O}$. It is  centered at a smooth stable curve $\alpha$ such  that the map 
$$p:{\mathcal O}\to Z,\qquad p(a, v,\eta)=[\alpha_{(a, v,\eta)}]$$
has $z$ in its image $U=p({\mathcal O})$. Let $\cg_{\lambda}$, $\lambda\in \Lambda$, be a family of such good uniformizers having the property that the open sets $U_\lambda=p_\lambda ({\mathcal O}_\lambda)$ cover the space $Z$.

In view of Theorem \ref{nattop}, the space $Z$ is a second countable paracompact Hausdorff topological space. Therefore, there exists a refinement $V_\lambda\subset U_\lambda$ of the covering having the same index set $\Lambda$, which is locally finite. (Of course, some of the sets 
$V_\lambda$ might be empty and we remove all those indices $\lambda$ for which 
$V_\lambda=\emptyset.$) Without loss of generality we may assume that $V_\lambda\neq \emptyset$ for all $\lambda\in \Lambda$. Then we replace the open sets ${\mathcal O}_\lambda$ by the preimages 
$p_\lambda^{-1}(V_\lambda)$, for which we use the same letter ${\mathcal O}_\lambda$.  We therefore may assume that we are given a countable collection of good uniformizing families 
$$q_\lambda \to \alpha_{q_\lambda}^\lambda,\qquad q_\lambda\in {\mathcal O}_\lambda,\, \lambda \in \Lambda, $$
having the property that the images $U_\lambda=p_\lambda ({\mathcal O}_\lambda)$,
$$U_\lambda=\{[\alpha^\lambda_{q_\lambda}]\vert\, q_\lambda\in {\mathcal O}_\lambda\}$$
constitute a locally finite open covering of the space $Z$.

By $X$ we denote the disjoint union of all the graphs 
$$
\cg_\lambda=\{(q_\lambda,\alpha^\lambda_{q_\lambda})\vert \, q_\lambda\in {\mathcal O}_\lambda \},
$$
that is
$$X=\coprod_{\lambda \in \Lambda} \cg_\lambda.$$
This disjoint union  is an M-polyfold  in a natural way and we view $X$ as the object set of a small category. {The morphism set  ${\bf X}$ of the small category is} the disjoint union 
$${\bf X}=\coprod_{\lambda,\lambda'\in \Lambda}M(\cg_\lambda, \cg'_{\lambda'}),$$
which, by Proposition \ref{a-key-z}, is also an M-polyfold.

\begin{theorem}\label{theoremA}
The M-polyfold $X$ is the object set and the M-polyfold set  ${\bf X}$ is the morphism set of a small category. The source and target maps $s, t:{\bf X}\to X$ are surjective local sc-diffeomorphisms. The structure maps $i:{\bf X}\to {\bf X}$, $u:X\to {\bf X}$, and $m:{\bf X}{{_s}\times_t}{\bf X}\to {\bf X}$ are sc-smooth.
In addition, $X$ satisfies the properness assumption, so that $X$ is an ep-groupoid. 
\end{theorem}
\begin{proof}
The theorem follows immediately  from Proposition \ref{a-key-z}, Proposition \ref{b-key-z}, and Proposition \ref{c-key-z} except the  statement about the properness which we shall prove next. 

We choose a point $x\in X=\coprod_{\lambda\in\Lambda} \cg_\lambda$. 
Then $x=(q_{\lambda_0},\alpha_{q_{\lambda_0}}^{\lambda_0})$ belongs to the graph $\cg_{\lambda_0}$. Pick the open neighborhood $U^{\cg_{\lambda_0}}=U^{\cg_{\lambda_0}}(q_{\lambda_0})$ guaranteed
by Theorem \ref{rees}.  If the subset  $\gamma_{\lambda_0}\subset X$ is  the graph of the good uniformizer restricted to  $U^{\cg_{\lambda_0}}$,  we show that
$$
t: s^{-1}(\cl (\gamma_{\lambda_0}))\rightarrow X
$$
is proper.  To this aim we let $K\subset X$ be a compact subset. Then there are only a finite number of  indices $\lambda$  in $\Lambda$ such that 
$K\cap \cg_{\lambda}\neq \emptyset$. Denote them by $\lambda_1,\ldots, \lambda_k$. Next we consider any sequence of morphisms 
$(\Phi_l)$ in ${\bf X}$ satisfying $t(\Phi_l)\in K$ and $s(\Phi_{l})\in \cl (\gamma_{\lambda_0})$. Then,  after taking a subsequence,  we may assume, 
without loss  of generality,  that $t(\Phi_l)\in \cg_{\lambda_{i_0}}$ for a suitable $i_0$ and,  moreover, 
$\pi_{\lambda_{i_0}}(t(\Phi_{l}))\rightarrow q_0$ in ${\mathcal O}_{\lambda_{i_0}}$. In addition, 
$\pi_{\lambda_0}(s(\Phi_{l}))\in \cl({U}^{\cg_{\lambda_0}})$. From Theorem \ref{rees}, we conclude that a 
subsequence of $(\Phi_{l})$ is convergent,  finishing the proof of Theorem \ref{theoremA}.
  \end{proof}

If $\abs{X}$ is the orbit space of the ep-groupoid $X$, then the canonical map $p:X\to Z$ induces the homeomorphism $\abs{p}:\abs{X}\to Z$. Hence the pair $(X, \abs{p})$ defines a polyfold structure on $Z$ and we have proved the following theorem. 
 
 \begin{theorem}\label{theoremB}
 The M-polyfold groupoid $X$ is an ep-groupoid and the pair $(X, \abs{p})$ defines a polyfold structure   on $Z$.
 \end{theorem}
 
 We now assume that we have constructed,   following the same  recipe,  the second  ep-groupoid $X'$ with its  canonical map $p':X'\rightarrow Z$ and the associated polyfold structure $(X', \abs{p'})$ of $Z$. In order to show that these  different choices define equivalent polyfold structures on $Z$  we proceed as follows.  We first define  $X''=X\coprod X'$  as the disjoint union of the corresponding $\cg_\lambda$ and $\cg'_{\lambda'}$ and note that the associated open  covering is still locally finite. Then we take the disjoint union of all morphisms as before. This  way we obtain a third ep-groupoid $X''$. The inclusion maps $F:X\rightarrow X''=X\coprod X'$ and
 $F':X'\rightarrow X''$ are equivalences.  The weak fibered product
 $X'''=X\times_{X''}X'$  is again an ep-groupoid and the projections  
 $\pi_1:X'''\rightarrow X$ and $\pi_2:X'''\rightarrow X'$ are equivalences. {For the notions of weak fibered products and common refinement we refer the reader to Section 2.2 and Section 2.3 in \cite{HWZ3.5}.}

 Therefore, the ep-groupoid   $X'''$ is a common refinement of $X'$ and $X''$.  The relation 
 $$
 p(\pi_1(X'''))=p'(\pi_2(X'''))
 $$
  implies that the pair  $(X''',\abs{p\circ\pi_1})$ defines another   polyfold structure on $Z$. Since 
 $\abs{p\circ \pi_1}=\abs{p}\abs{\pi_1}=\abs{p'}\abs{\pi_2}$, the polyfold structures $(X, \abs{p})$ and  $(X', \abs{p'})$ of $Z$ are indeed equivalent and we have proved the following theorem announced in the introduction as Theorem \ref{pfstructure}.

 \begin{theorem}\label{theoremC}
 Having fixed the exponential gluing profile and a strictly increasing sequence $(\delta_i)_{i\geq 0}\subset (0,2\pi)$,  the second countable paracompact topological space $Z$ of stable curves possesses  a  natural equivalence class of polyfold structures.
 \end{theorem}
{ \begin{remark}``Natural''  here means the following. The set $Z$ is determined by specifying $\delta_0\in (0,2\pi)$,  and by requiring the maps
 to be stable and of class $(3,\delta_0)$. {As already proved},  one can construct a topology on $Z$ which is second countable, paracompact and Hausdorff and  does not depend on the choices involved in its construction. Then picking a strictly increasing sequence $\delta=(\delta_i)$ starting at $\delta_0$ and staying below $2\pi$, and prescribing the exponential gluing profile
 there is a procedure which results in the construction of a polyfold structure $(X,|p|)$ on $Z$. This procedure
 in general involves choices. However, different choices lead to a polyfold structure $(X',|p'|)$, which
 is equivalent to $(X,|p|)$ in the sense defined in \cite{HWZ2}.  In  other words $Z=Z^{3,\delta_0}$ is in a natural way a polyfold
 once we fix a sequence of strictly increasing weights and the exponential gluing profile. Specifically we want to point out 
 that the choice of the cut-off function $\beta$ in the gluing constructions is irrelevant. Different choices lead to equivalent 
 polyfold structures. This is a consequence of Theorem \ref{theorem_neck} which  implies the compatibility of gluing procedures
 for different choices of cut-off functions.
  \end{remark}}

Summarizing we have constructed an equivalence class of polyfold structures on the topological space $Z$.
 In the following we shall refer to $Z$ as the polyfold $Z$ and keep the particular choices of the weights and the exponential gluing profile in mind when needed.

\section{The Polyfold Structure of  the Bundle $W\to Z$}\label{polstrbundle}
 
In this section we shall construct the strong polyfold structure  on the bundle $W\to Z$ introduced in Section \ref{sect1.2},  and prove Theorem \ref{main1.10}.
The construction of the strong polyfold structure for our bundle $W\to Z$  starts from the polyfold structure $(X, \abs{p})$ of the space $Z$ of stable curves into $Q$, constructed in the previous section. The structure consists of the M-polyfold $X$ and the homeomorphism $\abs{p}:\abs{X}\to Z$ between the orbit space  of $X$ and $Z$. In the previous section we have constructed the object set $X$ as the disjoint union
 $$X=\coprod_{\lambda\in \Lambda}\cg_\lambda$$
 of a countable family $\cg_\lambda$ of graphs $\cg$ of good uniformizing families of stable maps  (we omit the index $\lambda$)
 $$(a, v, \eta)\to \alpha_{(a, v,\eta)}=(S_a, j(a,v), M_a, D_a, \oplus_a\exp_u (\eta))$$
 centered at smooth stable curves $\alpha=(S, j, M, D, u)$. Here $(a, v, \eta)\in \co$, where $\co$ is an open subset of a splicing core $K^\cs$.  We can view the sets $\co$ or equivalently, the graphs $\cg$ of good uniformizing  families, as the local models of the M-polyfold $X$.
 
 Abbreviating, as we did before, the good uniformizing family  by $q\mapsto  \alpha_q$ where $q=(a, v, \eta)$,  we shall construct the associated  lifted family $\wh{q}\mapsto  \wh{\alpha}_{\wh{q}}$ such  that its graph $\wh{\cg}$ is a strong M-poyfold bundle $\wh{\cg}\to \cg$. The disjoint countable union of the corresponding graphs $\wh{\cg}_\lambda$,
 $$E:=\coprod_{\lambda\in \Lambda}\wh{\cg}_\lambda$$ is then  a strong bundle $E\to X$ over the M-polyfold $X$.

 To this aim we fix a good uniformizing family $q\to \alpha_q$, $q=(a, v, \eta)\in \co$, centered at the smooth stable map $\alpha=(S, j, M, D, u)$, and denote by $F$ the sc-Hilbert space consisting of maps  $z\mapsto \xi (z)$, $z\in S\setminus \abs{D}$, where $\xi (z):(T_zS, j)\to (T_{u(z)}Q, J(u(z)))$ is a complex anti-linear map. Moreover, the map 
 $z\mapsto  \xi (z)$ is of Sobolev class $(2, \delta_0)$ as introduced in Section \ref{sect1.2}. The level $k$ of the sc-structure on $F$ corresponds to the regularity $(2+k, \delta_k)$. The strong bundle splicing $\mcr$ on $\co\triangleleft F$ is defined by the strong bundle projection
 $$
 \rho:\co \triangleleft F\to F,\quad (a, v,\eta, \xi)\mapsto  \rho_{(a, v, \eta)}(\xi),
 $$
which is parametrized by $(a, v,\eta)\in \co$, but actually only depends on  $a$ so that 
$\rho_{(a, v,\eta)}=\rho_a$. The  projection operator $\rho_a$ is obtained by implanting the splicing projections using the previously introduced local charts $\psi$ on the symplectic manifold $Q$, however, using this time the hat gluing $\wh{\oplus}_a$ and the 
 hat anti-gluing $\wh{\ominus}_a$ operations. This way we obtain the splicing core
 $$
 K^\mcr=\{(a, v,\eta, \xi)\in \co\triangleleft F\, \vert\, \rho_a(\xi)=\xi\}.
 $$
 The splicing core $K^\mcr$ defines the strong local bundle 
 $$K^\mcr\to \co$$
 by 
 $(a, v,\eta, \xi)\to (a, v, \eta).$ We observe that the automorphism group $G$ of the original stable map $\alpha=(S, j, M, D, u)$ acts on $K^\mcr$ compatibly with the action already defined on $\co$,  namely by
 $$g\ast (a, v,\eta, \xi)=(g\ast a, g\ast v, \eta\circ g^{-1}, \xi\circ Tg^{-1})$$
 for every $g\in G$. 
 
 In order to ``lift'' the good family $q\to \alpha_q$ to a good family $\wh{q}\to \wh{\alpha}_{\wh{q}}$ on the bundle level, we introduce for $v\in V$ the complex linear bundle isomorphism 
 \begin{equation}\label{delta_eq_1}
 \delta (v):(TS, j(v))\to (TS, j),
\end{equation}
 defined by,
 $$ \delta (v)h=\frac{1}{2}\bigl(\id -j\circ j(v)\bigr)h.$$
 It satisfies $\delta  (v)\circ j(v)=j\circ \delta (v)$ and depends smoothly on $v\in V$. Moreover, since $j(v)=j$ on the discs of the small disc structure,
 $$ \delta (v)h=h$$
 above the discs. Similarly, the complex  linear map 
 \begin{equation}\label{delta_eq_2}
 \delta(a, v):(TS_a, j(a, v))\to (TS_a, j(a, 0))
 \end{equation}
 is defined by 
 $$\delta (a, v)h=\frac{1}{2}\bigl(\id -j(a, 0)\circ j(a, v))h.$$
 Next we consider  the complex anti-linear map
 $$
 \xi (z):(T_zS, j)\to (T_{u(z)}Q, J(u(z)).
 $$
  By implanting the hat gluing $\wh{\oplus}_a$,  we have defined the glued section 
 $$\wh{\oplus}_a(\xi )(z):(T_zS_a, j(a, 0))\to (T_{\oplus_a u (z)}Q, J(\oplus_a u (z)))$$
 along the glued map $\oplus_au$ into $Q$. The composition $\xi \circ \delta (v)$ is a complex anti-linear map 
 $$(T_zS, j(v))\to (T_{u(z)}Q, J(u(z)))$$
 which satisfies the identity 
 $$\wh{\oplus}_a(\xi \circ \delta (v))=\wh{\oplus}_a(\xi )\circ \delta (a, v).$$
 
Earlier we have introduced in {Definition \ref{GD} about good data}   the exponential map 
 $\exp:\wt{\co}\to Q$, where $\wt{\co}_p=T_pQ\cap \wt{\co}$ is a convex neighborhood of the origin $0_p$ in the tangent space $T_pQ$,  having the property that the restriction of $\exp$ to $\wt{\co}_p$ is an embedding into $Q$. Let us fix a complex connection on the almost complex vector bundle $(TQ, J)\to Q$. For example, if $\nabla_X$ is the covariant derivative on $Q$ belonging to the Riemaniann metric $\omega\circ (\id \oplus J)$, the connection $\wt{\nabla}_X$, defined by $\wt{\nabla}_XY=\nabla_XY-\frac{1}{2}J(\nabla_XJ)Y$, is complex in the sense that it satisfies $\wt{\nabla}_X(JY)=J(\wt{\nabla}_XY)$. If $\eta\in \wt{\co}_p$ is a tangent vector, the parallel transport  (of a complex connection) along the path $t\mapsto \exp_p (t\eta)$ for $t\in [0,1]$, defines the   linear map
 $$\Gamma  (\exp_p(\eta), p):(T_pQ, J(p))\to (T_{\exp_p(\eta)}Q, J( \exp_p(\eta)))$$
 which is complex linear so that 
 $$\Gamma (\exp_p(\eta), p)\circ J(p)= J( \exp_p(\eta))\circ  \Gamma (\exp_p(\eta), p).$$
 Given the point $(a, v,\eta, \xi)\in K^\mcr$ we define the complex anti-linear map 
$$
\Xi (a, v, \eta, \xi)(z):(T_zS_a, j(a, v))\to (T_{\oplus_a\exp_u (\eta)(z)}Q, J( \oplus_a\exp_u (\eta)(z)))
$$
at the points $z\in S_a\setminus \{\text{not-glued nodal points}\}$ by 
 \begin{equation}\label{sigma_complex_linear_1}
\Xi (a, v, \eta, \xi)(z)=\Gamma  [ \oplus_a\exp_u (\eta) (z), \oplus_au(z)]\circ \wh{\oplus}_a(\xi )(z)\circ \delta (a, v)(z).
\end{equation}
Observe that for $g\in G$
$$
\Xi(a,v,\eta,\xi)\circ Tg^{-1}_a = \Xi(g\ast a,g\ast v,\eta\circ g^{-1},\xi\circ Tg^{-1}).
$$
Now, with the given good uniformizing family $q\to \alpha_q$, explicitly given by 
$$(a, v, \eta)\to (S_a, j(a, v), M_a, D_a, \oplus_a\exp_u (\eta)), \quad (a, v, \eta)=q\in \co,$$
we associate the lifted family $\wh{q}\to \wh{\alpha}_{\wh{q}}$, defined by
$$(a, v, \eta, \xi)\to (S_a, j(a, v), M_a, D_a, \oplus_{a}\exp_u (\eta), \Xi (a, v,\eta, \xi))$$
where  $(a, v, \eta, \xi)=\wh{q}\in K^\mcr$. If $g\in G$ then $g$ induces an isomorphism
$$
\wh{\alpha}_{\wh{q}}\rightarrow \wh{\alpha}_{g\ast \wh{q}}
$$
which on the underlying $\alpha_q$ gives the usual isomorphism to $\alpha_{g\ast q}$. Moreover,  
$$
\Xi(g\ast (a,v,\eta,\xi))\circ Tg_a = \Xi(a,v,\eta,\xi).
$$
Let $\wh{\cg}$ be the graph of the lifted family $\wh{q}\to \wh{\alpha}_{\wh{q}} $ and note that the bundle $(a, v, \eta, \xi)\to (a, v,\eta)$ has the structure of a strong M-polyfold bundle. Therefore, the graph $\wh{\cg}$ carries trivially the natural structure of a strong polyfold bundle if we require that the graph map establishes a strong bundle isomorphism. Hence we have an sc-smooth projection
$$\wh{\cg}\to \cg.$$
Carrying out the above construction for all the graphs $\cg_\lambda$ of good families for all $\lambda\in \Lambda$ used to define the M-polyfold $X$, we define the strong polyfold bundle $E\to X$ over the object space $X$ as the disjoint union of the associated graphs $\wh{\cg}_{\lambda}$ of the lifted families 
$$E=\coprod_{\lambda\in \Lambda}\wh{\cg}_\lambda\to X.$$
We recall that the tuples $\wh{\alpha}=(S, j, M, D, u, \xi)$ and $\wh{\alpha}'=(S', j', M', D', u', \xi')$ are called equivalent, if there exists an isomorphism 
$\varphi:(S, j, M, D, u)\to (S', j', M', D', u')$ between the underlying stable maps satisfying, in addition, 
$$\xi'\circ T\varphi=\xi.$$
The set $W$ is defined as the collection of all such equivalence classes. 
The canonical map 
$$\Gamma:E\to W, $$
associating  with a point in $E$ its  isomorphism class, is defined by 
\begin{equation*}
\begin{split}
&\Gamma ((a, v, \eta, \xi), (S_a, j(a, v), M_a, D_a, \oplus_a \exp_u (\eta), \Xi (a, v, \eta, \xi))\\&\qquad =
[S_a, j(a, v), M_a, D_a, \oplus_a \exp_u (\eta), \Xi (a, v, \eta, \xi)].
\end{split}
\end{equation*}
This map $\Gamma$ covers the canonical map $X\to Z$, {so that  the diagram}
\mbox{}\\
\begin{equation*}
\begin{CD}
E@>\Gamma>>W\\
@VpVV  @VVV \\
X@ >>>Z\\ \\
\end{CD}
\end{equation*}
{commutes}. So far we have constructed a strong M-polyfold bundle
$$p:E\to X$$
over the object set $X$ of the ep-groupoid $(X, {\bf X})$. Next we have to construct the strong bundle map
$$
\mu: {\bf E}:= {\bf X}{{_s}\times_p}E\rightarrow E
$$
covering the target map $t:{\bf X}\rightarrow X$ as specified in 
Definition \ref{stbundleep} of a strong bundle over an ep-groupoid.

\begin{proposition}\label{propA}
There exists a strong bundle map $\mu:\be:=\bx{{_s}\times_{p}}E\to E$ which covers the target map $t:\bx\to X$ of the ep-groupoid $X$ satisfying
$$t\circ \pi_1(\Phi, e)=p\circ \mu (\Phi ,e)$$
for all $(\Phi, e)\in \bx{{_s}\times_{p}}E,$ 
so that the diagram
\mbox{}\\
\begin{equation*}
\begin{CD}
\bx{{_s}\times_{p}}E@>\mu>>E\\
@V\pi_1VV @VVp V \\
\bx@>t>>X.\\ \\
\end{CD}
\end{equation*}
commutes. 
Moreover, $\mu$  has the following additional properties.
\begin{itemize}
\item[{\em (1)}] $\mu$ is a surjective local sc-diffeomorphism and linear in the fibers $E_x$
\item[{\em (2)}] $\mu(1_x,e_x)=e_x$ for all $x\in X$ and $e_x\in E_x$.
\item[{\em (3)}] $\mu(\Phi\circ \Psi,e)=\mu(\Phi,\mu(\Psi,e))$ 
for all morphisms $\Phi, \Psi\in \bx$ and $e\in E$ satisfying $s(\Psi)=p(e)$ and $t(\Psi)=s(\Phi)=p(\mu (\Psi, e)).$
\end{itemize}

\end{proposition}
\begin{proof}
In order to define the map $\mu:\bx{{_s}\times_{p}}E\to E$ we abbreviate $q=(a, v, \eta)$, $\wh{q}=(a,v,\eta,\xi)$ and 
\begin{align*}
\alpha_q&=(S_a, j(a, v), M_a, D_a, \oplus_a\exp_u (\eta))\\
\wh{\alpha}_{\wh{q}}&=(S_a, j(a, v), M_a, D_a, \oplus_a\exp_u (\eta), \Xi (a, v, \eta, \xi)).
\end{align*}
If $\Phi\in \bx$ is the morphism
$$\Phi\equiv ((q,\alpha_q), \varphi, (q', \alpha_{q'}')):(q, \alpha_q)\to (q',\alpha_{q'}'),$$
then the isomorphism 
$$\varphi:(S_a, j(a, v), M_a, D_a)\to (S_{a'}', j'(a', v'), M_{a'}', D_{a'}')$$
satisfies 
$$\oplus'_{a'}\exp_{u'}'(\eta')\circ \varphi=\oplus_a\exp_u (\eta).$$
If $e=(\wh{q},\wh{\alpha}_{\wh{q}})\in E$, then $p(e)=(q,\alpha_q)=s(\Phi)$. Hence 
$(\Phi, e)\in \bx{{_s}\times_{p}}E$ and we define 
$$\mu (\Phi, (\wh{q}, \wh{\alpha}_{\wh{q}}))=(\wh{q}', \wh{\alpha}'_{\wh{q}'})$$
where $q'=(a', v', \eta')$ and $\wh{q}'=(a', v', \eta',\xi')=(q',\xi')$. Moreover, $\xi'$ is the unique solution of the two equations
\begin{align*}
\Xi(a, v,\eta, \xi)&=\Xi'(a', v', \eta', \xi')\circ T\varphi\\
\wh{\ominus}_{a'}'(\xi')&=0.
\end{align*}
Explicitly,
\begin{equation*}
\begin{split}
&\wh{\oplus}_{a'}'(\xi')(\varphi (z))\\
&\phantom{===}=\Gamma'[\oplus_{a'}'\exp_{u'}'(\eta')(\varphi (z)), \oplus_{a'}'u'(\varphi (z))]^{-1}\circ \Gamma [\oplus_a\exp_u (\eta)(z), \oplus_au (z)]\\
&\phantom{====\ }\circ \wh{\oplus}_a(\xi)(z)\circ  \delta (a, v)(z)\circ T\varphi (z)^{-1}\circ [ \delta'(a', v')(\varphi (z))]^{-1}
\end{split}
\end{equation*}
and
\begin{equation*}
\wh{\ominus}_{a'}'(\xi')(\varphi (z))=0.
\end{equation*}
We recall that $ \delta (a, v)$ is the complex linear map defined by the formula \eqref{delta_eq_2}.  The second equation  implies that $\rho_{a'}'(\xi')=\xi'$ and thus requires that $\xi'$ is contained in the splicing core.

If $t:\bx\to X$ is the target map and $e=(\wh{q}, \wh{\alpha}_{\wh{q}})$, then $t\circ \pi_1(\Phi, e)=t(\Phi)=(q', \alpha'_{q'})=p\circ \mu (\Phi, e)$ so that $\mu$ covers indeed the target map.

Clearly, the map $\mu$ is surjective and linear on the fibers $E_{(q, \alpha_q)}$. It also preserves the double-filtration $(m, k)$, where $0\leq k\leq m +1$. Moreover, if $\Phi =1_{(q, \alpha_q)}=((q, \alpha_q),\id, (q,\alpha_q))\in \bx$ is the unit morphism, then 
$\mu (1_{(q, \alpha_q)}, (\wh{q}, \wh{\alpha}_{\wh{q}}))=(\wh{q}, \wh{\alpha}_{\wh{q}}).$

To verify the property (3) we take two morphism $\Psi, \Phi\in \bx $ and $e\in E$ satisfying 
$s(\Psi)=p(e)$ and $t(\Psi)=s(\Phi)=p(\mu (\Psi, e)).$  Then with  $e=(\wh{q}, \wh{\alpha}_{\wh{q}})\in E$ we have $p(e)=(q,\alpha_q)=s(\Psi)$ and,  using $t(\Psi)=s(\Phi)$, 
\begin{align*}
\Psi& \equiv ((q, \alpha_q), \phi, (q',\alpha'_{q'})): (q, \alpha_q)\to (q', \alpha'_{q'})\\
\Phi& \equiv ((q', \alpha_{q'}'), \psi, (q'',\alpha''_{q''})): (q', \alpha_{q'}')\to (q'',\alpha''_{q''}). 
\end{align*}
The isomorphisms 
$\psi: (S_a, j(a, v), M_a, D_a)\to (S_{a'}', j'(a', v'), M_{a'}', D_{a'}')$ and $
\phi: (S_{a'}', j'(a', v'), M_{a'}', D_{a'}')\to (S_{a''}'', j''(a'', v''), M_{a''}'', D_{a''}'')$
satisfy 
\begin{equation}\label{twoeq}
\begin{aligned}
\oplus_{a'}'\exp_{u'}'(\eta')\circ \psi&=\oplus_{a}\exp_{u}(\eta)\\
\oplus_{a''}''\exp_{u''}(\eta'')\circ \phi&=\oplus_{a'}'\exp_{u'}'(\eta').
\end{aligned}
\end{equation}
In view of the definition of the map $\mu$, 
$$\mu (\Psi, e)=\mu (\Psi, (\wh{q}, \wh{\alpha}_{\wh{q}}))=(\wh{q'},\wh{\alpha'}_{\wh{q'}})$$
where  $\wh{q'}=(a', v', \eta', \xi')=(q', \xi')$   and where $\xi'$  is the unique solution of  the  two equations
\begin{equation}\label{twoeq1}
\begin{aligned}
\Xi (a, v, \eta, \xi)&=\Xi'(a', v', \eta', \xi')\circ T\psi\\
\wh{\ominus}_{a'}'(\xi')&=0.
\end{aligned}
\end{equation}
Consequently, 
\begin{equation*}
\mu (\Phi, \mu (\Psi, e))=\mu (\Phi, (\wh{q'},\wh{\alpha'}_{\wh{q'}}))=(\wh{q''},\wh{\alpha''}_{\wh{q''}}), 
\end{equation*}
where 
$\wh{q''}=(a'', v'', \eta'', \xi'')$ and where $\xi''$ is the unique solution of the  equations
\begin{equation}\label{twoeq2}
\begin{aligned}
\Xi' (a', v', \eta', \xi')&=\Xi''(a'', v'', \eta'', \xi'')\circ T\phi\\
\wh{\ominus}_{a''}''(\xi'')&=0.
\end{aligned}
\end{equation}
On the other hand, since the composition $\Phi\circ \Psi\in \bx$ is the morphism 
$$\Phi\circ \Psi\equiv ((q, \alpha_q), \phi\circ \psi, (q'', \alpha_{q''}''),$$
where $\phi\circ \psi$ is an isomorphism 
$$\phi\circ \psi:(S_a, j(a, v), M_a, D_a)\to (S_{a''}'', j''(a'', v''), M_{a''}'', D_{a''}'')$$
satisfying, in view of \eqref{twoeq}, 
$$\oplus_{a''}''\exp_{u''}''(\eta'')\circ (\phi\circ \psi)=\oplus_{a'}'\exp_{u'}'(\eta')\circ \psi=
\oplus_{a}\exp_{u}(\eta), $$
we have 
$$\mu (\Phi\circ \Psi, e)=\mu (\Phi\circ \Psi, (\wh{q},\wh{\alpha}_{\wh{q}}))=(\wh{r}, \wh{\alpha}_{\wh{r}}))$$
where 
$r=(a'', v'', \eta'')$ and $\wh{r}=(a'', v'', \eta'', \bar{\xi})$,  and 
$$\wh{\alpha}_{\wh{r}}=(S_{a''}'', j''(a'', v''), M_{a''}'', D_{a''}'' , \oplus_{a''}''\exp_{u''}\exp (\eta''), \Xi''(a'', v'', \eta'', \bar{\xi})).$$
Here $\bar{\xi}$ is the unique solution of the following two equations, 
\begin{align*}
\Xi (a, v, \eta, \xi)&=\Xi''(a'', v'', \eta'', \bar{\xi})\circ T(\phi \circ \psi)\\
\wh{\ominus}_{a''}''(\bar{\xi})&=0.
\end{align*}
In view of \eqref{twoeq1} and \eqref{twoeq2},  
\begin{equation*}
\begin{split}
\Xi (a, v, \eta, \xi)&=\Xi'(a', v', \eta', \xi')\circ T\psi =\Xi''(a'', v'', \eta'', \xi'')\circ T\phi \circ T\psi\\
&=\Xi''(a'', v'', \eta'', \xi'')\circ T(\phi \circ \psi).
\end{split}
\end{equation*}
In view of $\wh{\ominus}_{a''}''(\xi'')=0$,  we conclude by  the uniqueness  that 
$\ov{\xi}=\xi''$ so that $\wh{r}=\wh{q''}$ and $\wh{\alpha}_{\wh{r}}=\wh{\alpha}_{\wh{q''}}''.$
Consequently, $\mu (\Phi\circ \Psi, e)=(\wh{r}, \wh{\alpha}_{\wh{r}} )=(\wh{q''},\wh{\alpha''}_{\wh{q''}})=\mu (\Phi , \mu (\Psi, e))$ as claimed in property (3).


It remains to prove that $\xi'$ is an $\ssc_\triangleleft$--smooth map of $(a, v,\eta, \xi)\in  K^\mcr$.  If $(\Phi_0, (\wh{q}_0,\wh{\alpha}_{\wh{q}_0}))\in \be$ is given, where 
$\Phi_0=(q_0, \alpha_{q_0}), \varphi_0, (q_0',\alpha'_{q_0'}))$ is the morphism in $\bx$ and 
$\varphi_0:\alpha_{q_0}\to \alpha_{q_0'}'$ is an isomorphism, then there exists, in view of Theorem \ref{key-z}, a core smooth germ $q\to \varphi_q$ of isomorphisms 
$$\varphi_q:(S_a, j(a, v), M_a, D_a, \oplus_a\exp_u (\eta))\to (S_b', j'(b, w), M_b', D_b', \oplus'_b \exp'_{u'}(\eta')),$$
where $q=(a, v, \eta)$ and where $b=b(a, v,\eta)$, $w=w(a, v, \eta)$ and $\eta'=\eta'(a,v, \eta)$ are sc-smooth maps, and where $\varphi_{q_0}=\varphi_0$. Then the relationship between $\xi$ and $\xi'$ is defined by 
\begin{equation*}
\begin{split}
&\Gamma [ \oplus_a \exp_u (\eta), \oplus_a u] \circ \wh{\oplus}_a(\xi)\circ \sigma (a, v)\circ T\varphi^{-1}_{(a, v, \eta)}\circ \sigma'(b, w)^{-1}\\
&\qquad =
\Gamma' [\oplus_b'(\exp_{u'}'\eta'), \oplus_b'(u')] \circ \wh{\oplus}_b'(\xi')
\end{split}
\end{equation*}
and 
$$\wh{\ominus}_b'(\xi')=0.$$
Here $(b, w, \eta')$ are the above sc-smooth maps of $(a, v, \eta)$.  As in the previous chapter, a local case by case study shows that $\xi'$ is an $\ssc_\triangleleft$--smooth map of $(a, v,\eta, \xi)\in K^{\mathcal R}$ which is linear in $\xi$. The ingredients and arguments in this study are similar to those in the proof of the sc-smoothness  in the polyfold construction in {Chapter \ref{chapter_2_technical results}}  and the details are left to the reader.  The map $\xi'$  is clearly fiber-wise an isomorphism. A similar discussion shows  that its fiber-wise inverse is also
a strong bundle map, which implies that $\mu$ is a local sc-diffeomorphism since it covers the local sc-diffeomorphism $t$.
This completes the proof that $\mu$ is a strong bundle map.

\end{proof}

At this point we have constructed a strong bundle $(E,\mu)$ over the ep-groupoid $X$ as defined in Definition \ref{stbundleep}. Recalling Definition \ref{stbundleep1}  we are now  in the position to define
the structure of a strong polyfold bundle on the bundle $\wh{p}:  W\rightarrow Z$.
We consider the canonical map $\Gamma:E\to W$ introduced above. Denoting by $\abs{E}$ the space of orbits under the morphisms in $\be$, we see that $\Gamma$ induces the map 
$$\abs{\Gamma}:\abs{E}\to W$$
defined by 
$$\abs{\Gamma} (\abs{(a,v,\eta, \xi, \wh{\alpha}_{(a,v, \eta, \xi)})})=[\wh{\alpha}_{(a,v,\eta,\xi)}].$$
We note that the existence of an isomorphism 
$$\wh{\varphi}:\wh{\alpha}_{(a,v,\eta, \xi)}\to \wh{\alpha'}_{(a',v',\eta',\xi')}$$
implies the existence of an isomorphism $\varphi:\alpha_{(a, v, \eta)}\to \alpha'_{(a',v',\eta')}$ and since $(a, v,\eta)\to \alpha_{(a,v,\eta)}$ is a good uniformizing family, this implies, as we have seen in the previous chapter, that $(a', v', \eta')=g\ast (a, v, \eta)$ and 
$\varphi=g_a$ for a suitable element $g\in G$. It therefore follows from the definition of $\mu$ that the map $\abs{\Gamma}:\abs{E}\to W$ is a homeomorphism. Recalling the natural homeomorphism $\abs{p}:\abs{X}\to Z$ between the orbit space $\abs{X}$ and $Z$, defined as
$$\abs{p}(\abs{(a, v, \eta, \alpha_{(a, v,\eta)})})=[\alpha_{(a, v,\eta)}],$$
and  the projection functor $P:E\to X$ introduced in Section \ref{sc-smoothness},   which induces  the map 
$\abs{P}:\abs{E}\to \abs{X}$ between the  orbit spaces, we conclude that $\abs{p}\circ \abs{P}=\wh{p}\circ \abs{\Gamma},$ so that the diagram
\mbox{}\\
\begin{equation*}
\begin{CD}
\abs{E}@>\abs{P}>>\abs{X}\\
@V\abs{\Gamma} VV @VV\abs{p}V \\
W@>\wh{p}>>Z\\ \\
\end{CD}
\end{equation*}
commutes.
Therefore, the triple $(P:E\to X, \abs{\Gamma}, \abs{p})$ defines the structure of a strong polyfold bundle on the bundle $\wh{p}:W\to Z$ as claimed in Theorem \ref{main1.10} in the introduction.

Herewith the proof of Theorem \ref{main1.10} is complete. \hfill $\blacksquare$

%
%
%

\chapter{The Nonlinear Cauchy-Riemann Operator}\label{fredholm-section-z}
The chapter is devoted to the Cauchy--Riemann section of the strong polyfold bundle $W\to Z$ constructed in the previous section. Its solution sets are the Gromov compactified moduli spaces. We shall prove, in  particular,  
Theorem \ref{maincauchy-riemann} of the introduction.

\section{Fredholm Sections of Strong Polyfold Bundles}
We begin with the definition of a section of a strong bundle.
\begin{definition}
Let $P:E=(E, {\bf E})\to X=(X, {\bf X})$ be a strong bundle over the ep-groupoid $X$. A {\bf section of the strong bundle} $P$ is an sc-smooth functor 
$$F: X\rightarrow E
$$
satisfying $P\circ F=\id$.\index{section of a strong bundle over ep-groupoid}

An $\ssc^+$-section of the strong bundle $P$ is a section $F$ which  induces  an sc-smooth functor $(X, {\bf X})\to (E^{0,1}, {\bf E}^{0,1})$ where $E^{0,1}$ resp. ${\bf E}^{0,1}$ are equipped with the grading $(E^{0,1})_{m}=E_{m,m+1}$  resp. $({\bf E}^{0,1})_{m}={\bf E}_{m,m+1}$ for $m\geq 0$.\index{$\ssc^+$-section of a strong bundle over ep-groupoid}

\end{definition}
The functoriality of the section $F:X\to E$ implies for the section $F:X\to E$ on the objects sets that 
$$\mu (\varphi, F(x))=F(y),$$
and for the section ${\bf F}:{\bf X}\to {\bf E}$ on the morphisms sets that 
$${\bf F}(\varphi)=(\varphi, F(x))\in {\bf X}{_s}\times_p E$$ for every morphism $\varphi \in {\bf X}$ satisfying $s(\varphi)=x$ and $t(\varphi )=y$ where $s, t:{\bf X}\to X$ are the source and the target maps.

\begin{definition}
A {\bf Fredholm section $F$ of the strong bundle}  $P:(E, {\bf E})\to (X, {\bf X})$ over the ep-groupoid $(X, {\bf X})$ is an sc-smooth section $F$ of $P$ which, as a section $F:X\to E$ on the object sets, is an M-polyfold  Fredholm section as defined in Definition 3.6 in \cite{HWZ3.5} and recalled in Section \ref{sectionpolyfred} above.
\index{Fredholm section!of a strong bundle  over an ep-groupoid}
\end{definition}

An example of a Fredholm section is the Cauchy-Riemann operator $\ov{\partial}_J$ dealt with in the next section. The {\bf Cauchy-Riemann section}  $\ov{\partial}_J$ of the bundle $W\to Z$ is defined as 
$$\ov{\partial}_J([S,j, M, D, u])=[S,j, M, D, u, \frac{1}{2}(Tu+J(u)\circ Tu\circ j)].$$
In the strong bundle structure $(P:E\to X, \Gamma, \gamma)$ of the bundle $W\to Z$,  the section $\ov{\partial}_J$ is  represented by the section $F:X\to E$ on the object sets,  defined by 
$$F(q, \alpha_q)=((q,\xi),\wh{\alpha}_{(q,\xi)}).$$Ä
Here, $q=(a,v,\eta)\in {\mathcal O}$ and $(q, \xi)\in K^{\mathcal R}$ where $\xi$ is the unique solution of the two equations 
\begin{equation}\label{cauchy_riemann_eq_0}
\begin{aligned}
\Xi (a, v, \eta, \xi)&=\frac{1}{2}\bigl[T(\oplus_a \exp_{u}(\eta))+J(\oplus_a \exp_{u}(\eta))\circ  T(\oplus_a \exp_{u}(\eta))\circ j(a, v)\bigr] \\
\wh{\ominus}_{a}(\xi)&=0,
\end{aligned}
\end{equation}
{where $\Xi$ is defined in \eqref{sigma_complex_linear_1} as 
$$
\Xi (a, v, \eta, \xi)=\Gamma  [ \oplus_a\exp_u (\eta), \oplus_au]\circ \wh{\oplus}_a(\xi )\circ \delta (a, v).
$$
}

The next section will show that $F$ is a Fredholm section in the sense of Section  \ref{sectionpolyfred}.

If $p:W\to Z$ is a strong polyfold bundle and $(P:E\to X, \Gamma, \gamma)$ a strong polyfold structure of $p$, then the smooth section $F$ of $P:E\to X$ defines a continuous section $f$ of the bundle $p:W\to Z$ by 
$$f\circ \gamma =\Gamma \circ \abs{F}.$$
In diagrams,
\mbox{}\\
\begin{equation*}
\begin{CD}
\abs{X}@>\gamma>>Z\\
@V\abs{F}VV   @VVfV \\
\abs{E}@>\Gamma>> W.\\ 
\end{CD}
\end{equation*}
\mbox{}\\

Let $(P':E'\to X', \Gamma', \gamma')$ be an equivalent strong bundle structure for $p:W\to Z$ in the sense of Definition 3.29 in \cite{HWZ3.5}, and let $(f', F')$ the associated pair of sections where $f'\circ \gamma'=\Gamma'\circ \abs{F'}$. Then the two pairs $(f, F)$ and $(f', F')$ are called equivalent, if $f'=f$ and if there exists an s-{bundle}  isomorphism $\mathfrak{A}:E\to E'$ (as defined in  Section 3.4  of  \cite{HWZ3.5}) whose push-forward satisfies $\mathfrak{A}_\ast (F)=F'$. 
\begin{definition}
An equivalence class $[f, F]$ of sections is called an {\bf sc-smooth section of the polyfold bundle $p:W\to Z$}  and the pair $(f, F)$ is called a representation of the map $f:Z\to W$ in  the model $P:E\to X$.\index{sc-smooth section of the polyfold bundle}
\end{definition}

In view of Proposition 2.25 in \cite{HWZ3.5} these concepts are all well defined. We recall from \cite{HWZ3.5} the following definition.
\begin{definition}
The section $f=[f,F]$ of the  polyfold  bundle $p:W\to Z$ is called a {\bf Fredholm section}  provided  there exists a representation $F$ of $f$ which is a Fredholm section of the strong bundle $P:E\to X$ belonging to the polyfold structure $(P:E\to X, \Gamma, \gamma)$ of the strong bundle. \index{Fredholm section of a polyfold bundle}

The Fredholm section $f$ is called {\bf proper}  if the solution set 
$$S(f)=\{z\in Z\, \vert \, f(z)=0\}$$
is compact in $Z$.\index{proper Fredholm section}
\end{definition}

We end with a useful remark.
\begin{remark}
{
We consider the strong bundle $P:E\to X$ over the object set of the  ep-groupoid $X$.
If $\Phi:e\rightarrow e'$ is a morphism in ${\bf E}$, then 
the associated map $t\circ s^{-1}$ on $E$ defines an isomorphism
between two strong bundles $\Gamma:E\vert U\rightarrow E\vert V$, where $U,V\subset X$
are open subsets of the object space $X$ which are  sc-diffeomorphic by the local diffeomorphism $\gamma$ underlying  the map $t\circ s^{-1}$. 
The points in $E$ have the form
$$
(a,v,\eta,\xi,S_a,j(a,v),M_a,D_a,\oplus_a\exp_u(\eta),\Xi(a,v,\eta,\xi)), 
$$
and are mapped to similar (primed) data, in which 
$(a',v',\eta',\xi')$ are sc-smooth functions of $(a,v,\eta,\xi)$.
Moreover the transformation is linear in $\xi$. The crucial observation
which will allow for a more specialized perturbation theory is the following. If $\xi$ vanishes near the nodal  points in $|D_a|$ then the same
is true for $\xi'$ near $|D_{a'}'|$. This holds in a uniform way with
respect to some open neighborhood around the underlying point $(a,v,\eta)$.
A similar assertion is true with respect to the vanishing near marked  points in $M$. As a consequence we can distinguish the  subclass of $\ssc^+$-multisections which vanish near nodal points or near marked points. If we take the sum of two such multi-sections, as defined in  {Section 3.5}  of  \cite{HWZ3.5}, the same remains true for their  sum}. \end{remark}

\section{The Cauchy-Riemann Section:  Results}\label{c-r-results}
In this section we merely describe the main results of the chapter. Their  proofs are postponed to the  later sections.

As before we denote by 
$(Q,\omega)$ a closed symplectic manifold  of dimension $2n$ and by $J$ a compatible almost complex structure on $Q$. 
The bundle $p:W\to Z$ over the space $Z$ of stable curves from the noded Riemann surfaces into the manifold $Q$ is equipped with the structure of a strong polyfold bundle. The nonlinear Cauchy-Riemann operator defines the continuous section $\ov{\partial}_J$ of the bundle $p$ via 
$$
\ov{\partial}_J([S,j,M,D,u])=[S,j,M,D,u,\ov{\partial}_{J,j}(u)]
$$
where 
$$\ov{\partial}_{J, j}(u)=\frac{1}{2}\bigl(Tu+J(u)\circ Tu\circ j)$$ and 
$u:S\to Q$. The section $\ov{\partial }_{J}$ will be abbreviated by $f=\ov{\partial }_{J}$. The section $f$ is induced from the section $F$ of the model bundle  $P:E\to X$ which is a strong bundle over the ep-groupoid $X$. The pair $(f, F)$ will be called the nonlinear Cauchy-Riemann section of $p$. The key result is the following theorem. 

\begin{theorem}\label{FREDPROP}
The Cauchy-Riemann section $(f, F)$ of $p$ defines an sc-smooth component-proper Fredholm section.  On the component $Z_{g,k,A}$
the  real Fredholm index is  equal to 
$$
2n(1-g)+2c_1(A)+6g-6+2k,
$$
with $2n=\text{dim}\ Q$, where $g$ is the arithmetic genus of the noded Riemann surfaces, $k$ the number of marked points and $A\in H_2(Q)$. Moreover,  $f$ has a natural orientation.
\end{theorem}

The orientation is discussed in the appendices \ref{orientations-abstract} and \ref{orientations}.  From an abstract point of view the orientability is discussed in \cite{HWZ10}.  It is discussed  in the generality relevant for  Symplectic Field Theory in \cite{HWZ5}.  The component-proper assertion involves arguments of Gromov compactness as explained in Remark \ref{GROMOV} below. The index calculation is well-known. The Fredholm index can be deduced from the corresponding  results  in  \cite{Lock} and  \cite{MS}. The technology can be found in  \cite{Floer1,Floer2,MSCH}.  We should mention that in these computations the nodes are treated as singularities so that the punctured Riemann surfaces are manifolds with cylindrical ends. Theorem \ref{FREDPROP} will follow from the three Propositions \ref{SC-SMOOTHREG}, \ref{C0-CONT}, and \ref{comprop} which we present next.

As described in Section \ref{polstrbundle} the Cauchy-Riemann section in the ep-groupoid description has, in the local  sc-coordinates, the form 
$$
(a,v,\eta)\mapsto  (a,v,\eta,\xi)
$$
where $(a,v,\eta)\in {\mathcal O}$ and $(a,v,\eta,\xi)\in K^{\mathcal R}$. The map  $\xi$ is  the unique solution of the two equations\begin{align*}
\Xi(a,v,\eta,\xi)&= \bar{\partial}_{J,j(a,v)} (\oplus_a\exp_{u}(\eta))\\
\wh{\ominus}_{a}(\xi)&=0.
\end{align*}
In view of the definition of $\Xi$ in the construction of the strong bundle structure,  the above equations  become
\begin{equation}\label{con}
\begin{aligned}
\Gamma(\oplus_a\exp_{u}(\eta),\oplus_a(u))\circ\wh{\oplus}_a(\xi)\circ\delta(a,v)&=\ov{\partial}_{J,j(a,v)} (\oplus_a\exp_{u}(\eta))\\
\wh{\ominus}_{a}(\xi)&=0.
 \end{aligned}
\end{equation}

\begin{proposition}\label{SC-SMOOTHREG}
The Cauchy-Riemann section $\ov{\partial}_J$ of the bundle $p:W\to Z$ is  sc-smooth and regularizing. In particular, all the local expressions 
$$
\co\rightarrow K^{\mathcal R},\quad  (a,v,\eta)\mapsto (a,v,\eta,\xi)
$$
are sc-smooth. If $(a,v,\eta,\xi)\in K^{\mathcal R}_{m,m+1}$, then it follows that  $(a,v,\eta)\in \co_{m+1}$.
\end{proposition}

The proof of Proposition \ref{SC-SMOOTHREG} is postponed to Section \ref{sec5.2}.

As recalled in Section \ref{sectionpolyfred}, in order to verify the (polyfold-)Fredholm property of the section $\ov{\partial}_J$ it remains to show that for every smooth point,  the germ of $\ov{\partial}_J$ around this point admits a filled version,  which after correction by a suitable sc$^+$-section is conjugated to a basic germ.

In order to construct the filled section we denote for every nodal pair $\{x, y\}\in D$,  the image of the nodal points  under the map $u$ by 
$$q_{\{x, y\}}=u(x)=u(y)\in Q.$$

If the gluing parameter $a=a_{\{x,y\}}$ does not vanish,  we denote by 
$$\cae^{\{x,y\}}_a=H^{3,\delta_0}_c(C^{\{x,y\}}_{a_{\{x,y\}}},T_{q_{\{x,y\}}}Q)$$
the usual sc-Hilbert space (introduced in Section \ref{gluinganti-sect}) of mappings defined on the 
infinite cylinder $C_a$ having antipodal asymptotic constants. The level $m$ refers to the regularity $(m+3, \delta_m)$. As always, the strictly increasing sequence $(\delta_m)$ belongs to $(0, 2\pi)$. If $a_{\{x, y\}}=0$, we set 
$$\cae^{\{x,y\}}_0=\{0\}.$$
Similarly we abbreviate 
$$\cf ^{\{x,y\}}_a=H^{2,\delta_0}(C^{\{x,y\}}_{a_{\{x,y\}}},T_{q_{\{x,y\}}}Q)$$ 
if $a=a_{\{x,y\}}\neq 0$ and  otherwise set $\cf_0^{\{x,y\}}=\{0\}$. 

At every nodal pair $\{x,y\}\in D$, the linear  Cauchy-Riemann operator 

$$
\ov{\partial}^{\{x,y\}}_0(k) = \partial_s k+J(q_{\{x,y\}})\partial_t k,
$$
is an sc-isomorphism
$$
\ov{\partial}^{\{x,y\}}_0:\cae^{\{x,y\}}_a\rightarrow \cf^{\{x,y\}}_a.
$$
Here $(s, t)$ are the distinguished coordinates on the infinite cylinder $C_{a_{\{x, y\}}}^{\{x,y\}}$ as explained in Section \ref{preliminaries} below. 
If $a_{\{x, y\}}=0$, this linear Cauchy-Riemann operator is the zero operator between two trivial spaces. The filled section is defined as the map 
\begin{equation}\label{filled-d}
(a,v,\eta)\mapsto\xi
\end{equation}
where  it is no {longer required  that  $(a,v,\eta)$ belongs to the subset $\co$ of the splicing core}.  Instead the section $\eta$ of $u^*TQ$, having matching asymptotic constants across the nodes, is merely required to satisfy the constraints 
$$
\eta(z)\in H_{u(z)}
$$
at  the points $z\in \Sigma$ of the chosen stabilization set $\Sigma$ on $S$. Moreover, the map $\xi$, which associates with every point $z\in S\setminus  \abs{D}$ a complex anti-linear map 
$$\xi (z): (T_zS, j(z))\to (T_{u(z)}Q, J(u(z))),$$
is now determined by the equations
\begin{equation}\label{DEFR}
\begin{aligned}
\Gamma(\oplus_a(\exp_{u}(\eta)),\oplus_a(u))\circ\wh{\oplus}_a(\xi)\circ\delta (a,v)&= \ov{\partial}_{J,j(a,v)} (\oplus_a\exp_{u}(\eta))\\
 \wh{\ominus}_{a_{\{x,y\}}}^{\{x,y\}}(\xi) \cdot \frac{\partial }{\partial s}&= \ov{\partial}^{\{x,y\}}_0(\ominus_{a_{\{x,y\}}}^{\{x,y\}}(\eta))\end{aligned}
\end{equation}
for all nodal points $ \{x,y\}\in D.$

Let us denote by $\cae$ the sc-Banach space consisting  of the sections $\eta$ of the bundle $u^{\ast}TQ$ of  class $(3,\delta_0)$, having matching asymptotic conditions at the nodes and satisfying the  constraints  $\eta(z)\in H_{u(z)}$ at all the points $z\in \Sigma$ of the fixed stabilization set $\Sigma$.  We recall from Definition \ref{GD} of good data that $H_{u(z)}=T_{u(z)}M_{u(z)}$ is a complement in $T_{u(z)}Q$ of the image of $Tu(z)$. By $N$  we denote the vector space of gluing parameters 
$(a_{\{x,y\}})_{\{x,y\}\in D}$ and by $E$ the vector space containing $0\in V\subset E$ where $v\in V$ parametrizes  the complex structures $j(v)$ on $S$. Given $a=\{a_{\{x,y\}}\vert\, \{x,y\}\in D\}\in N$  satisfying $\abs{a_{\{x,y\}}}<\frac{1}{2}$ and given $\eta \in \cae$, we denote by  
$\ominus_a(\eta)$ the tuple $$
\ominus_a(\eta)=\left\{ \ominus^{\{x,y\}}_{a_{\{x,y\}}}(\eta)\right\}_{\{x,y\}\in D}.
$$
The total gluing takes the form
$$
\eta\mapsto  \boxdot_a(\eta)= (\oplus_a(\eta),\ominus_a(\eta)).
$$
Similarly, the total hat-gluing has the form 
$$
\xi\mapsto  \wh{\boxdot}_a(\xi)=(\wh{\oplus}_a(\xi),\wh{\ominus}_a(\xi)).
$$
By $\ov{\partial}_0$  we shall denote the tuple of operators
$$
\ov{\partial}_0 =\left\{ \ov{\partial}^{\{x,y\}}_0\right\}_{\{x,y\}\in D}.
$$
For the following considerations we  shall rewrite the  equations \eqref{DEFR} as follows,
\begin{equation}\label{defined}
\begin{aligned}
\Gamma(\oplus_a\exp_{u}(\eta),\oplus_a(u))\circ\wh{\oplus}_a(\xi)\circ\delta (a,v)&= \ov{\partial}_{J,j(a,v)} (\oplus_a\exp_{u}(\eta))\\
 \wh{\ominus}_a(\xi)\cdot \frac{\partial }{\partial s} &=\ov{\partial}_0 (\ominus_a(\eta)).
\end{aligned}
\end{equation}

Let us denote by $O$ a small open neighborhood
of the origin $(0,0,0)\in N\oplus E\oplus \cae $.  By $\cf$ we denote the collection of all maps $z\rightarrow\xi(z)$ of class $(2,\delta_0)$,
which associate to a point $z\in S\setminus|D|$ a map $\xi(z):(T_zS,j(z))\rightarrow (T_{u(z)}Q,J(u(z)))$ which is complex anti-linear.
The sc-structure is given by requiring that the level $m$ consists of maps of regularity $(m+2,\delta_m)$. {The norms on $\cf$
can be defined as in the case of $H^{3+m,\delta_m}(u^\ast TQ)$ as a sum of semi-norms. For example,  near nodal points from a nodal pair $\{x,y\}$ we take positive (negative) holomorphic polar coordinates centered around $x$ and $y$, respectively and 
associate to $\xi$ the pair $(\xi^+,\xi^-)$ defined by
$$
\xi^+(s,t) =\pr_2\circ T\psi(u(\sigma_x(s,t)))\frac{\partial\sigma_x}{\partial s}(s,t)
$$
and
$$
\xi^-(s,t) =\pr_2\circ T\psi(u(\sigma_y(s',t')))\frac{\partial\sigma_y}{\partial s}(s',t')
$$}

If $(a,v,\eta)\in O$,  we compute the map $\xi\in \cf$ as the solution of the equation \eqref{defined}. This defines the   map
\begin{equation}\label{FILLED}
{\bf f}:O\rightarrow \cf,\quad  {\bf f}(a,v,\eta)=\xi.
\end{equation}

\begin{proposition}[{\bf Filled version}]\label{C0-CONT}
The map ${\bf f}:O\rightarrow \cf$ defined by \eqref{FILLED} is sc-smooth. It is a filled version of  the section $F$ in the ep-groupoid representation. If ${\bf s}$ is an $\ssc^+$-section satisfying  ${\bf s}(0,0,0)={\bf f}(0,0,0)$, then
${\bf f}-{\bf s}$ is conjugated near $(0,0,0)$ to an $\ssc^0$-contraction germ. 
\end{proposition}
The proof of Proposition \ref{C0-CONT} is carried out in Section \ref{section4.4}.

Proposition \ref{SC-SMOOTHREG}  together with Proposition \ref{C0-CONT}  demonstrate the Fredholm property of 
$\ov{\partial}_J$. In order to prove these two propositions we shall establish useful local coordinates expressions for the equations \eqref{DEFR},  making use of the fact that all the occurring expressions are local,  in the sense that  
the properties of the solutions $\xi$ near a point $z\in S$ depend only on $\eta$ and $u$ near $z$ (besides on the parameters $(a, v)$ of course). In the study of nonlinear elliptic pde's there are standard ways to obtain global estimates from local estimates. We shall use these ideas as well. However, the crucial  difficulty going beyond the classical estimates, is to obtain estimates in the neighborhoods of the nodal points, which are independent of the gluing parameter $a$. 

\begin{proposition}\label{comprop}
The Cauchy-Riemann section $\ov{\partial}_J$ of the bundle $p:W\rightarrow Z$ is component-proper.
\end{proposition}
\begin{remark}\label{GROMOV} This is essentially the Gromov compactness theorem and as such a special case of the results in \cite{BEHWZ}. There is, however, one issue which has to be pointed out. In principle we need to show that the induced topology from $X$ on the solution set is a compact topology. Now,  as it turns out, this induced topology is the same topology as the topology familiar from  the Gromov compactness proofs.  As long as one stays away from the nodes, one uses the fact that gradient bounds imply $C^\infty$-bounds. At the nodes, the usual discussion has to be sharpened by using  the results (or easy variations of) from 
\cite{HWZ7.1}. We leave the details to the reader.
\end{remark}

\section{Some Technical Results}\label{preliminaries}
In this section we collect some technical tools which will be used in  the proofs of Proposition \ref{SC-SMOOTHREG} and Proposition \ref{C0-CONT}.

As usual the strictly increasing sequence $(\delta_m)_{m\geq 0}$ is fixed and satisfies $0<\delta_m<2\pi$ for all $m\geq 0$. We recall that the Hilbert spaces $F_m^\pm=H^{2+m,\delta_m}(\R^\pm \times S^1)$, $m\geq 0$,  
consist of $L^2$-maps 
$$u:\R^\pm \times S^1\to \R^{2n}$$
whose  weak partial derivatives $D^{\alpha}u$   up to order $2+m$ belong,   if weighted by $e^{\delta_m \abs{s}}$,   to  the space $L^2(\R^\pm \times S^1)$ so that  $e^{\delta_m \abs{s}} D^{\alpha}u\in L^2(\R^\pm \times S^1)$ for all $\abs{\alpha}\leq 2+m$. 
We equip the Hilbert spaces 
$$F_m=F_m^+\oplus F^-_m$$ 
with the  norms 
$
\abs{(\xi^+,\xi^-)}^2_{F_m}:= \norm{\xi^+}^2_{2+m,\delta_m} +\norm{\xi^-}^2_{2+m,\delta_m},
$
where 
$$
{\norm{\xi^\pm}}_{m,\delta_m}^2:=\sum_{|\alpha|\leq m} \int_{{\mathbb R}^\pm \times S^1} |D^{\alpha}\xi^\pm (s,t)|^2 e^{2\delta_m \abs{s}}dsdt.
$$
The  space  $F:=F_0$ is an sc-Hilbert  space  and the nested sequence $(F_m)$ defines an sc-structure on $F$.

{We will also need  the sc-Hilbert space $E$  consisting of pairs $(\eta^+,\eta^-)$ where $\eta^\pm=c+r^\pm$ for some constant $c$ and the pair $(r^+, r^-)$  belongs to 
$$H^{3,\delta_0}({\mathbb R}^+\times S^1,{\mathbb R}^{2n})\oplus
H^{3,\delta_0}({\mathbb R}^-\times S^1,{\mathbb R}^{2n}).$$ The maps $\eta^+$ and $\eta^-$ have matching asymptotic limits in the sense that 
$\lim_{s\to \infty}\eta^+(s, t)=\lim_{s\to -\infty}\eta^-(s, t).$
The sc-structure  on $E$ is given by  the spaces $E_m$ of regularity $(m+3,\delta_m)$. The $E_m$-norm of $(\eta^+,\eta^-)\in E_m$ is defined by 
$$
|(\eta^+,\eta^-)|_{E_m}^2 := |c|^2 +\norm{r^+}^2_{3+m,\delta_m} + \norm{r^-}^2_{3+m,\delta_m}
$$
where $\eta^\pm=c+r^\pm$ for some constant $c$.
}

We  recall from Section \ref{gluinganti-sect} that the total hat-gluing map
$$\wh{\boxdot}_a=(\wh{\oplus}_a, \wh{\ominus}_a):F\to \wh{G}^a=H^2(Z_a, \R^{2n})\oplus H^{2,\delta_0}(C_a, \R^{2n})$$
is an sc-linear isomorphism,  for every gluing parameter $a\in \hb$ satisfying $a\neq 0$. If $a=0$, then 
$$\wh{\boxdot}_0=(\wh{\oplus}_0, 0):F\to \wh{G}_0=F\oplus \{0\}$$
is the identity map. Now we define for $a\neq 0$ the maps 
\begin{equation*}
D^a_t, D^a_s:E\rightarrow F
\end{equation*}
by
\begin{align*}
D^a_t(\eta^+,\eta^-)& =\wh{\boxdot}_a^{-1}\bigl( \partial_t (\oplus_a( \eta^+,\eta^-)), 0\bigr)\\
D^a_s(\eta^+,\eta^-)& =\wh{\boxdot}_a^{-1}\bigl( \partial_s (\oplus_a( \eta^+,\eta^-)), 0\bigr).
\end{align*}
Correspondingly we set, for $a=0$, 
\begin{align*}
D^0_t(\eta^+,\eta^-)& = \partial_t (\eta^+,\eta^-)\\
D^0_s(\eta^+,\eta^-)& = \partial_s (\eta^+,\eta^-).
\end{align*}
Similarly,    the maps 
$$
C^a_t, C^a_s : E\rightarrow F
$$
are defined,  for $a\neq 0$  belonging to $\hb$,  by 
\begin{align*}
C^a_t(\eta^+,\eta^-)& =\wh{\boxdot}_a^{-1}\bigl(0,  \partial_t (\ominus_a( \eta^+,\eta^-))\bigr)\\
C^a_s(\eta^+,\eta^-)& =\wh{\boxdot}_a^{-1}\bigl(0,  \partial_s (\ominus_a( \eta^+,\eta^-))\bigr)
\end{align*}
and, if $a=0$, by 
$$C_s^0(\eta^+, \eta^-)=0=C_t^0(\eta^+, \eta^-).$$

\begin{proposition}\label{XCX1}
The maps $B_\frac{1}{2}\oplus E\rightarrow F$,  defined by 
$$
(a,(\eta^+,\eta^-))\mapsto D^a_{t}(\eta^+,\eta^-)\qquad \text{and}\qquad 
(a,(\eta^+,\eta^-))\mapsto D^a_{s}(\eta^+,\eta^-), $$
are sc-smooth. The same holds  true if we replace the maps $D^a_s, D^a_t$  by the maps $C^a_s, C^a_t$. Moreover,  for  every $m$ there exists a constant $C_m>0$ independent of $a$ such that 
$$
\abs{D^a_s(\eta^+,\eta^-)}_{F_m}\leq C_m\cdot \abs{(\eta^+,\eta^-)}_{E_m},\  \abs{D^a_t(\eta^+,\eta^-)}_{F_m}\leq C_m\cdot \abs{(\eta^+,\eta^-)}_{E_m},
$$
and 
$$
\abs{C^a_s(\eta^+,\eta^-)}_{F_m}\leq C_m\cdot \abs{(\eta^+,\eta^-)}_{E_m},\
\abs{C^a_t(\eta^+,\eta^-)}_{F_m}\leq C_m\cdot \abs{(\eta^+,\eta^-)}_{E_m}.
$$
\end{proposition}
The proof of the proposition is carried out  in Appendix \ref{XCX10}.

Now,  let  $E^{(N)}$, $F^{(K)}$, resp. $F^{(M)}$ be the sc-Hilbert spaces of functions, whose domains of definitions are half-cylinders as above, but whose image spaces are the Euclidean spaces $\R^N,  \R^K$ resp. $\R^M$ instead of $\R^{2n}$. 

The natural evaluation  map
$$
\call (\R^K,\R^M)\oplus \R^K\rightarrow \R^M
$$
will be denoted by 
$$(G, v)\mapsto G\cdot v.$$
With  a smooth map 
$$A:\R^N\rightarrow \call (\R^K,\R^M)$$
we associate the mapping 

$$
\wt{A}:B_\frac{1}{2}\oplus E^{(N)}\oplus F^{(K)}\rightarrow F^{(M)}, 
$$
defined for $a\neq 0$ as 
$$ 
\wt{A}(a,(u^+,u^-),(\eta^+,\eta^-))=(\xi^+, \xi^-), 
$$
where the pair $(\xi^+, \xi^-)\in F^{(M)}$ is the unique solution of the two equations
\begin{equation*}
\begin{aligned}
\wh{\oplus}_a(\xi^+,\xi^-) (z)&= A\bigl(\oplus_a (u^+,u^-)(z)\bigr) \cdot \wh{\oplus}_a  (\eta^+,\eta^-) (z)\\
\wh{\ominus}_a(\xi^+,\xi^-)(z)&=0,
\end{aligned}
\end{equation*}
for $z$ on the cylinder $Z_a$. If $a=0$, we define 
$\wt{A}(0,(u^+,u^-),(\eta^+,\eta^-))=(A(u^+)\eta^+, A(u^-)\eta^-).$

Abbreviating the  notation we shall write $\xi$ instead of $(\xi^+,\xi^-)$ and $u$ instead of $(u^+, u^-)$ etc. so  that the above two equations  
defining $\wt{A}(a, u,\eta)=\xi$ become
\begin{equation*}
\begin{aligned}
\wh{\oplus}_a(\xi) &= A(\oplus_a(u))\cdot \wh{\oplus}_a (\eta)\\
\wh{\ominus}_a(\xi)&=0.
\end{aligned}
\end{equation*}

For $a\in \hb$ fixed,  we denote by 
$$\wt{A}_a:E^{(N)}\oplus F^{(K)}\to F^{(M)}$$
the induced map $\wt{A}_a(u, \eta)=\wt{A}(a, u, \eta)$ and introduce,  for fixed $a\in \hb$ and fixed level $m$,  the map
$$L_a:E_m^{(N)}\to \call (F^{(K)}_m, F_m^{(M)}),$$ 
defined by $L_a(u)=\wt{A}(a, u,\cdot ).$

\begin{proposition}\label{XCX2}
\mbox{}
\begin{itemize}
\item[{\em (1)}] The bundle  map  
$$\bigl(B_\frac{1}{2}\oplus E^{(N)}\bigr)\triangleleft  F^{(K)}\rightarrow \bigl(B_\frac{1}{2}\oplus E^{(N)}\bigr)\triangleleft F^{(M)},
$$
$$
\bigl((a,u),\eta\bigr)\mapsto  \bigl((a,u),\wt{A}(a, u,\eta)\bigr)
$$
is  an $\ssc_\triangleleft$-smooth map.
If $u_0\in  E^{(N)}_m$ and $\varepsilon>0$ are given, then the following estimate  holds, 
$$
\abs{\wt{A}\bigl(a, u, \eta\bigr)-\wt{A}\bigl(a,u_0,\eta\bigr)}_{F_{m+i}^{(M)}}
\leq \varepsilon \cdot |\eta|_{F^{(K)}_{m+i}}
$$  
for all $u\in E_m^{(N)}$ sufficiently close  to
$u_0$ on the level $m$,  for all $\eta \in F^{(K)}_{m+i}$,  for all  $a\in \hb$,  and  for $i=0,1$.
\item[{\em (2)}] For a fixed  gluing  parameter $a\in \hb$ and the fixed  level $m$,  the map
$
L_a:E^{(N)}_m\rightarrow \call (F^{(K)}_m,F^{(M)}_m)$ 
is of class  $C^1$ and, denoting by $DL_a(u) (\wh{u})$ the linearization of $L_a$ at the point $u$ in the direction of $\wh{u}$, the map 
$$\hb \oplus E^{(N)}\oplus E^{(N)} \oplus F^{(K)}\rightarrow F^{(M)},$$
 defined by 
$$\bigl(a,u,\wh{u},\eta\bigr)\mapsto  \bigl[ DL_a(u)(\wh{u})\bigr]\cdot \eta, $$
is of class $\ssc^\infty$. 
\item[{\em (3)}]  Given a level $m$, a point $u_0\in E_m^{(N)}$, and a positive number $\varepsilon>0$, the following estimate holds,
\begin{equation*}
\begin{split}
\abs{\bigl[ DL_a(u)\cdot (\wh{u}) - DL_{a}(u_0)\cdot (\wh{u})\bigr] \cdot \eta}_{F^{(M)}_m}\leq\varepsilon\cdot \abs{\wh{u}}_{E_m^{(N)}}\cdot \abs{\eta}_{E_m^{(K)}}
\end{split}
\end{equation*}
for all $a\in \hb$,  $u$ sufficiently close to $u_0$ on level $m$, and  for all $\wh{u}\in E^{(N)}_m$, and all $\eta\in F_m^{(K)}$.
\end{itemize}
\end{proposition}

The proof of Proposition \ref{XCX2} is postponed to  the Appendix \ref{XCX20}.

Next we recall from \cite{HWZ8.7} the  estimates for the total hat-gluing map $\wh{\boxdot}_a$ and the total gluing map $\boxdot_a$.  If the gluing parameter $a$ satisfies $0<\abs{a}<\frac{1}{2}$,  the space 
$$
\wh{G}^a=H^{2}(Z_a,{\mathbb R}^{2n})\oplus H^{2,\delta_0}(C_a,{\mathbb R}^{2n}),
$$
has the sc-structure given by the sequence $H^{m+2}(Z_a,{\mathbb R}^{2n})\oplus H^{m+2,\delta_m}(C_a,{\mathbb R}^{2n}).$ If 
$a=0$,  we put  $\wh{G}^a=F\oplus\{0\}$.

For a map $q:Z_a\to \R^{2n}$  on the finite cylinder in $H^{k}(Z_a, \R^{2n})$ we define the norms
\begin{equation}\label{preliminaries9}
\nr q\nr^2_{k,\delta} =\sum_{|\alpha|\leq k} \int_{[0,R]\times S^1}   \abs{D^{\alpha}q(s,t)}^2\cdot e^{2\delta \abs{s-\frac{R}{2}}} dsdt
\end{equation}
where $\delta\in \R$. 

{We would like to point out that  here we allow the weight  $\delta$ to be negative;  in fact, this will be used crucially later on.}

If the map $p:C_a\to \R^{2n}$ on the infinite cylinder $C_a$ has  vanishing asymptotic constants we set $p(s, t)=p([s,t])$ and define the norms
\begin{equation}\label{preliminaries9_a}
\nr p\nr^2_{m,\delta} =\sum_{|\alpha|\leq m}\int_{\R\times S^1}   \abs{D^{\alpha}p(s,t)}^2\cdot e^{2\delta \abs{s-\frac{R}{2} }  } dsdt.
\end{equation}
We introduce the $\wh{G}^a_m$-norm of the pair $(q, p)\in \wh{G}^a$
by setting
$$
| (q,p)|_{\wh{G}^a_m}^2:=e^{\delta_m R}\cdot \left[\nr q\nr_{m+2,-\delta_m}^2+\nr p\nr_{m+2,\delta_m}^2\right].
$$

With the total hat gluing $\wh{\boxdot}_a(\xi^+,\xi^-)=\bigl(\wh{\oplus}_a(\xi^+,\xi^-), (\wh{\ominus}_a(\xi^+,\xi^-)\bigr)\in \wh{G}^a$ defined on $(\xi^+,\xi^-)\in F$, the following estimate is proved in \cite{HWZ8.7}, Theorem 2.26.
\begin{proposition}[{\bf Total Hat-Gluing Estimate}]\label{HAT_HAT}
For every $m\geq 0$ there exists a constant $C_m>0$ which does not depend  on $a$, such that
$$
\frac{1}{C_m}\cdot |(\xi^+,\xi^-)|_{F_m}\leq | \wh{\boxdot}_a(\xi^+,\xi^-)|_{\wh{G}^a_m}\leq C_m\cdot |(\xi^+,\xi^-)|_{F_m}
$$
for  $|a|<\frac{1}{2}$ and  $(\xi^+,\xi^-)\in F_m$.
\end{proposition}

A  similar estimate holds true also for  the total gluing map $\boxdot_a$. Recall the previously introduced space  $H^{3,\delta_0}_c$ consisting of all pairs $(\eta^+,\eta^-)$ of maps  $\eta^\pm:{\mathbb R}^\pm\times S^1\rightarrow {\mathbb R}^{2n}$ 
which have the same asymptotic constant $c$ such that  $\eta^\pm-c$ belongs to $H^{3,\delta_0}({\mathbb R}^\pm\times S^1,{\mathbb R}^{2n})$. The sc-structure of this space has the regularity $(3+m,\delta_m)$. 
In this case we introduce for $0<\abs{a}<\frac{1}{2}$ the sc-space 
$G^a$ by $G^a=H^{3}(Z_a,{\mathbb R}^{2n})\oplus H^{3,\delta_0}_c(C_a,{\mathbb R}^{2n})$ with the  sc-structure
$$
G^a_m= H^{3+m}(Z_a,{\mathbb R}^{2n})\oplus H^{m+3,\delta_m}_c(C_a,{\mathbb R}^{2n})
$$
where $p\in H^{m+3,\delta_m}_c(C_a,{\mathbb R}^{2n})$ possesses antipodal contants 
$p_{\infty}=\lim_{s\to \infty}p(s, t)=-\lim_{s\to -\infty}p(s, t)$ and we associate with $p$ the map 
$\wh{p}:C_a\to \R^{2n}$ defined by 
$$
\wh{p}([s,t])=p([s, t])-[1-2\beta_a(s)]\cdot p_{\infty}.
$$

\noindent {We recall that  the function $\beta_a$ is defined as $\beta_a(s)=\beta (s-R/2)$ with the function $\beta$ as  in Section \ref{gluinganti-sect}}. Then $\wh{p}$ has vanishing asymptotic limits and we define the norm on level $m$ of $(q, p)\in G^a_m$ by 
$$
\abs{(q,p)}^2_{G^a_m}:=\abs{[q]_a-p_\infty }^2+e^{\delta_mR}\bigl( \nr q-[q]_a+p_{\infty}\nr^2_{m+2,-\delta_m} +\nr \wh{p}\nr^2_{m+2,\delta_m}\bigr)
$$
where the mean value $[q]_a$ is the number 
$$
[q]_a=\int_{S^1} q\bigl(\bigl[\frac{R}{2},t\bigr]\bigr)\ dt, 
$$
and where $\wh{p} =p-(1-2\beta_a)\cdot p_\infty$.
For the total gluing map   $\boxdot_a(\eta^+,\eta^-):=\bigl(\oplus_a (\eta^+,\eta^-),\ominus_a (\eta^+,\eta^-)\bigr)\in E$, the following estimate is proved in  \cite{HWZ8.7},  Theorem 2.23.
\begin{proposition}[{\bf Total Gluing Estimate}]\label{total-gluing}
For every $m\geq 0$ there exists a constant $C_m>0$ independent of  the gluing parameter $a$ such that
$$
\frac{1}{C_m}\cdot |(\eta^+,\eta^-)|_{E_m}\leq |\boxdot_a(\eta^+,\eta^-)|_{G^a_m}\leq C_m\cdot |(\eta^+,\eta^-)|_{E_m}
$$
for $|a|<\frac{1}{2}$ and all $(\eta^+, \eta^-)\in E_m$.\index{total gluing estimates}
\end{proposition}

Finally  we recall some results about linear Cauchy-Riemann type operators, which   follow easily from
the usual methods for linear elliptic operators on domains with cylindrical ends as presented in  \cite{Lock},  combined with the well-known index calculation
for the Cauchy-Riemann operator on the sphere.
We present  an abbreviated version of the  Appendix 4.4 in  \cite{HWZ8.7}.
Given two copies of a bi-infinite cylinder $\R\times S^1$, say $\Sigma^+$ and $\Sigma^-$, we can construct, given a nonzero gluing parameter $a$, 
the finite cylinder $Z_a$ and two infinite cylinders $C_a$ and $Z_a^\ast$ as illustrated   in Figure \ref{Fig12}. The two cylinders $\Sigma^\pm$ have parts which are 
shaded in different ways in order  to identify the differences between $C_a$ and $Z_a^\ast$.

We denote by $H_c^{3,\delta_0}(C_a,{\mathbb R}^{2n})$ the usual sc-Hilbert space with antipodal asymptotic constants, where
the level $m$ is given by regularity $(m+3,\delta_m)$. As usual the sequence $(\delta_m)$ belongs to $(0,2\pi)$. {In the next proposition $J(0)$ is  a constant almost complex structure on $\R^{2n}$.}

\begin{proposition}\label{c-iso}
 The Cauchy-Riemann operator
$$
\ov{\partial}_0:=\partial_s+J(0)\partial_t: H^{3,\delta_0}_c(C_a,{\mathbb R}^{2n})\rightarrow H^{2,\delta_0}(C_a,{\mathbb R}^{2n})
$$
is an sc-isomorphism,  and for every $m\geq 0$ there exists a constant $C_m>0$ which is independent of $a$,  such that
$$
\frac{1}{C_m}\cdot |\xi|_{H^{3+m,\delta_m}_c}\leq |\bar{\partial}_0\xi|_{H^{2+m,\delta_m}}\leq C_m\cdot |\xi|_{H^{3+m,\delta_m}_c}
$$
for all $\xi\in H^{3+m,\delta_m}_c(C_a, \R^{2n})$.
\end{proposition}

\begin{figure}[!htb]
\psfrag {a}{\small{$\Sigma^+$}}
\psfrag {b}{$\Sigma^-$}
\psfrag {c}{$Z^\ast_a$}
\psfrag {d}{$C_a$}
\psfrag  {e}{$Z_a$}
\psfrag  {f}{$Z_a$}
\centering
\includegraphics[width=3.8in]{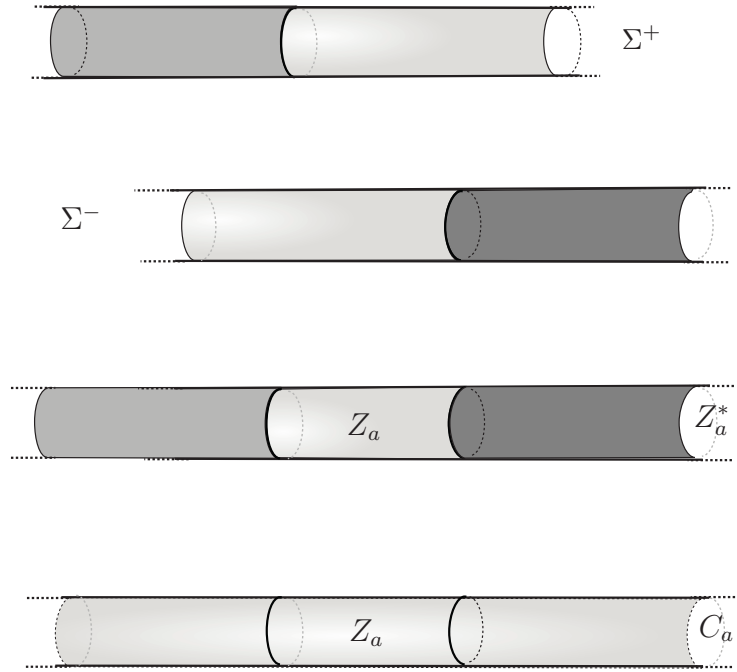}
\caption{The extended cylinder $Z_a^\ast$.}
\label{Fig12}
\end{figure}

Recall that the Hilbert spaces  $H^{3+m,-\delta_m}(Z_a^\ast,\R^{2n} )$ for $m\geq 0$ consist of maps $u:Z_a^\ast\rightarrow\R^{2n}$ for which the associated maps $u:{\mathbb R}\times S^1\rightarrow \R^{2n} $, defined by $ (s,t)\rightarrow u([s,t])$ have partial derivatives up to order $3+m$ which, if  weighted by $e^{-\delta_m|s-\frac{R}{2}|}$ belong to the space $L^2(\R\times S^1)$. 
 The norm of $u\in H^{3+m,-\delta_m}(Z_a^\ast, \R^{2n} )$ is defined  as 
 $$\nr u\nr_{3+m,-\delta_m}^{\ast}=\bigl( \sum_{\abs{\alpha}\leq 3+m }\int_{\R\times S^1}\abs{D^\alpha u}^2e^{-2\delta_m\abs{s-\frac{R}{2}}}\ dsdt\bigr)^{1/2} .$$
 The spaces $H^{2+m,-\delta_m}(Z_a^\ast,\R^{2n})$  and the norms $\nr u\nr_{2+m,-\delta_m}^{\ast}$ are defined analogously.
We denote  by $[u]_a$ the average of a map $u:{\mathbb  R}\times S^1\rightarrow {\mathbb R}^2$ over the  circle at $[\frac{R}{2},t]$,   defined by
$$
[u]_a:=\int_{S^1} u\Bigl(\Bigl[\frac{R}{2},t\Bigr]\Bigr) dt, 
$$
where the number  $R$ is equal to $R =\varphi (\abs{a})$  and $\varphi$ is the exponential   gluing profile.

\begin{proposition}\label{POLTER}
The Cauchy-Riemann operator
$$
\ov{\partial}_0:H^{3+m,-\delta_m}(Z_a^\ast,{\mathbb R}^{2n})\rightarrow H^{2+m,-\delta_m}(Z_a^\ast,{\mathbb R}^{2n})
$$
is a surjective Fredholm operator whose kernel  is of real dimension $2n$. (The kernel consists of the constant functions.) For every $m\geq 0$ there exists a constant $C_m>0$ independent of $a$, such that
\begin{equation}\label{poltereq1}
\frac{1}{C_m}\cdot \nr u-[u]_a\nr^\ast_{m+3,-\delta_m}\leq \nr\bar{\partial}_0(u-[u]_a)\nr^\ast_{m+2,-\delta_m}\leq C_m\cdot \nr u-[u]_a\nr^{\ast}_{m+3,-\delta_m}
\end{equation}
for all $u\in H^{3+m,-\delta_m}(Z_a^\ast,{\mathbb R}^{2n})$.
\end{proposition}

A proof is presented  in Appendix \ref{POLTERGG}.

\section{Regularization and Sc-Smoothness of  $\ov{\partial}_J$}\label{sec5.2}
We start with the regularity for the Cauchy-Riemann section $\ov{\partial}_J$ of the polyfold bundle $p:W\to Z$.
 { In Section \ref{natural_topology_Z_section} we have introduced the pair $(X, \abs{p})$ which defines a polyfold structure on the space  $Z$ of stable curves. Here $X$ is an ep-groupoid and $\abs{p}:\abs{X}\to Z$ is the
 homeomorphism induced from the canonical map $p:X\to Z$.  The orbit space $\abs{X}$ is equipped with the filtration $\abs{X_0}=\abs{X}\supset \abs{X_1}\supset \abs{X_2}\supset \cdots \supset \abs{X_\infty}=\bigcap_{i\geq 0}\abs{X_i}$ where $\abs{X_\infty}$ is dense in every $\abs{X_i}$. This filtration induces via the homeomorphism $\abs{p}:\abs{X}\to Z$ the filtration 
 $Z=:Z_0\supset Z_1\supset  Z_2\supset \cdots \supset  Z_\infty=\bigcap_{i\geq 0}Z_i.$ Similarly, in Section \ref{polstrbundle}  we have defined the polyfold structure $(P:E\to X, \abs{\Gamma}, \abs{p})$ of the bundle $p:W\to Z$ where $P:E\to X$ is an strong bundle over the ep-groupoid $X$ and $\abs{\Gamma}:\abs{E}\to W$ is a homeomorphism between the orbit space of $E$  and $W$. 
 The double filtration $\abs{E_{i, k} }$ for $i\geq 0$ and $0\leq k\leq i+1$, of the orbit space $\abs{E}$ induces via the homeomorphism $\abs{P}:\abs{E}\to W$ the double filtration $W_{i, k}$ for $i\geq 0$ and $0\leq k\leq i+1$.
 }

Now,  in view of the polyfold constructions, it suffices to prove the following theorem.
\begin{theorem}\label{thm-regularity1}
The section $\ov{\partial}_J:Z\to W$ has the following regularity property. If $\zeta=[(S, j, M, D, u)]\in Z_m$ satisfies $\ov{\partial}_J(\zeta)\in W_{m,m+1}$, then $\zeta\in Z_{m+1}.$
\end{theorem}
\begin{proof}
We choose the representative $(S, j, M, D, u)$ of  the equivalence class 
$$\zeta=[(S, j, M, D, u)]\in Z_m.$$
 By assumption,  the map $u:S\to Q$  has regularity  $(m+3,\delta_m)$. By definition of $\ov{\partial}_J$, 
$$\ov{\partial}_{J}[(S, j, M, D, u)]=[(S, j, M, D, u,\ov{\partial}_{J, j}(u))]\in W_{m, m+1}.$$
Hence $\ov{\partial}_{J, j}(u)=w$ has fiber-regularity $(m+3, \delta_{m+1})$. 
The standard elliptic regularity theory implies that  away from the nodal points $u$ is of Sobolev class 
$H^{m+4}_{\loc}$.

In order to analyze the map $u$ near the nodes we choose a nodal point $x\in S$ and take  the holomorphic polar coordinates 
$ 
\sigma:[0,\infty)\times S^1\rightarrow S\setminus\{x\}
$ 
{belonging to the small disk structure (before denoted by $h_x$) and } satisfying $\sigma (s, t)\to x$ as $s\to \infty$. Moreover,  we choose on the manifold $Q$ a smooth chart 
$
\psi:U(u(x))\rightarrow {\mathbb R}^{2n}
$
mapping $u(x)\in Q$ to $0\in \R^{2n}$, and  obtain for the map $u$ in local coordinates the map $v:[0,\infty)\times S^1\to \R^{2n}$,  defined by 
$$
v(s,t)=\psi\circ u\circ \sigma(s,t).
$$
Then $v(s,t)\rightarrow 0$ as $s\rightarrow\infty$ and $v\in H^{m+3,\delta_m}(\R^+\times S^1, \R^{2n})$. In order to show that $v$ solves an elliptic equation we insert $z=\sigma (s, t)$ into the equation 
{
\begin{equation}\label{cr_0}
\ov{\partial }_{J,j}(u)(z)=\frac{1}{2}\bigl[ Tu(z)+J(u(z))\circ Tu(z)\circ j(z)\bigr]=:w(z)
\end{equation}
so that in local coordinates  the equation \eqref{cr_0} becomes
\begin{equation}\label{cr-2}
\begin{split}
\frac{1}{2}&T\psi (u(\sigma ))\bigl[Tu(\sigma)+J(u\circ \sigma)\circ Tu(\sigma)\circ j(\sigma)\bigr]\cdot T\sigma\\
&=T\psi (u(\sigma))\cdot w(\sigma )\cdot T\sigma =:\wh{w}.
\end{split}
\end{equation}
Using $T\psi (u(\sigma))\circ Tu(\sigma) \circ T\sigma=Tv$ and   $j(\sigma)\circ T\sigma=T\sigma \circ i$, the equation \eqref{cr-2} turns into 
\begin{equation}\label{cr-3}
\frac{1}{2}\bigl[Tv+\wh{J}(v)\circ Tv\circ i\bigr]=\wh{w}.
\end{equation}
Here  $\wh{J}$ is the  smooth almost complex structure on $\R^{2n}$ defined by 
$\wh{J}(y)=T\psi (\psi^{-1}(y))\circ J(\psi^{-1}(y))\circ T\psi (\psi^{-1}(y))^{-1}$ at $y\in \R^{2n}$ and $\wh{w}=T\psi (u(\sigma))\circ w (\sigma)\circ T\sigma$. 
Clearly,  $\wh{w}\circ i=-\wh{J}(v)\circ \wh{w}$,  so that 
$\wh{w}$  is complex anti-linear. }

Applying the  complex anti-linear maps in  \eqref{cr-3}  to the vector  $\frac{\partial}{\partial s}$, we obtain the following equation for  the map $v:[0,\infty)\times S^1\to \R^{2n}$,
$$\frac{1}{2}\bigl[v_s+\wh{J}(v)v_t\bigr]=g$$
where $g(s, t)=\wh{w}(s, t)\frac{\partial}{\partial s}.$ By assumption, $v\in H^{m+3,\delta_m}(\R^+\times S^1,\R^{2n})$ and
 $g\in H^{m+3,\delta_{m+1}}(\R^+\times S^1,\R^{2n})$ and we have to prove that $v$ belongs to the space $H^{m+4,\delta_{m+1}}(\R^+\times S^1,\R^{2n})$. This follows immediately from the following proposition which then concludes the proof of  Theorem \ref{thm-regularity1}.

\begin{proposition}\label{lem-reg1}
Assume that $v$ belongs to $H^{m+3,\delta_m}(\R^+\times S^1, \R^{2n})$ and
 $g$  belongs to  $H^{m+3,\delta_{m+1}}(\R^+\times S^1,\R^{2n})$.   Let $J$ be a smooth almost complex structure on $\R^{2n}$. If $v$ is a solution of the equation $v_s +J(v)v_t= g $
on $\R^+\times S^1$, then for every $\varepsilon>0$  the map $v$ belongs to the space $H^{m+4,\delta_{m+1}}((\varepsilon,\infty)\times S^1, \R^{2n})$.
\end{proposition}

\begin{proof}
The standard elliptic regularity theory and the fact that $g\in H^{m+3}_{\loc}$ around points in $(0,\infty)\times S^1$  imply that $v$ is of class $H^{m+4}_{\loc}$ around such points as well.  The main task is to verify the asymptotic properties of $v$ as $s\to \infty$.  We choose a  smooth cut off function
$
\gamma:\R\rightarrow [0,1]
$
satisfying $\gamma(s)=0$ for $s\leq 1$, $\gamma(s)=1$ for $s\geq 2$, and $\gamma '\geq 0$, and introduce the shifted functions $\gamma_R:\R\to [0,1]$ defined by  
$
\gamma_R(s):=\gamma(s-R)
$
for $R\geq 0$. 
The operator
$$
h\mapsto  \gamma_R\cdot (J(v)-J(0))h_t,
$$
induces two $R$-dependent  linear operators between the following  spaces
$$
H^{m+3,\delta_m}_c(\R\times S^1,\R^{2n})\rightarrow H^{m+2,\delta_m}(\R\times S^1,\R^{2n})
$$
and
$$
H^{m+4,\delta_{m+1}}_c(\R\times S^1,\R^{2n})\rightarrow H^{m+3,\delta_{m+1}}(\R\times S^1,\R^{2n}), 
$$
where the functions in $H^{m+4,\delta_{m+1}}_c$ have antipodal asymptotic limits as $s\to \pm\infty$. 
We estimate these operators as $R\rightarrow \infty$ as follows. 
\begin{lemma}\label{TORTE}
If $\varepsilon>0$,  there exists $R_\varepsilon>0$ such that for all $R\geq R_\varepsilon$ and for $i=0,1$ the estimate
$$
\norm{ \gamma_R\cdot (J(v)-J(0))h_t}_{H^{m+2+i,\delta_{m+i}}(\R\times S^1,\R^{2n})}\leq \varepsilon \cdot\norm{h}_{H^{m+3+i,\delta_{m+i}}_c(\R\times S^1,\R^{2n})}
$$
holds for  all $h\in  H^{m+3+i,\delta_{m+i}}_c(\R\times S^1,\R^{2n})$.
\end{lemma}
The proof  is elementary,  but somewhat lengthy, and is given in the Appendix \ref{Lemma-4.18}

In view of  Lemma \ref{TORTE}  and Proposition \ref{c-iso}, 
the linear operators
\begin{equation*}
\begin{gathered}
H^{m+3+i,\delta_{m+i}}_c(\R\times S^1, \R^{2n})\rightarrow H^{m+2+i,\delta_{m+i}}(\R\times S^1, \R^{2n})\\
h\mapsto h_s+J(0)h_t +\gamma_R\cdot (J(v)-J(0))h_t
\end{gathered}
\end{equation*}
are topological linear isomorphisms  for $i=0$ and  for $i=1$ if $R$ is  sufficiently large. 

From  $g\in H^{m+3,\delta_{m+1}}(\R^+\times S^1, \R^{2n})$ and $\gamma_{R+1}(s)=0$ for $s\leq R+2$,   we conclude  that $\gamma_{R+1}g \in H^{m+3,\delta_{m+1}}(\R\times S^1, \R^{2n})$. Hence there exists a unique solution $w\in H^{m+4,\delta_{m+1}}_c(\R\times S^1, \R^{2n})$ of the equation
\begin{equation}\label{klap1}
w_s+J(0)w_t +\gamma_R \cdot (J(v)-J(0)) w_t =\gamma_{R+1}g.
\end{equation}
Because $\delta_m<\delta_{m+1}$, the maps $\gamma_{R+1}g$ and $w$ belong to the spaces $H^{m+2,\delta_{m}}(\R\times S^1, \R^{2n})$ and  $H^{m+3,\delta_m}_c(\R\times S^1,\R^{2n})$, respectively.  Consequently, the map  $w$ is also the unique solution  of \eqref{klap1} belonging to $H^{m+3,\delta_m}_c(\R\times S^1, \R^{2n})$.   

We multiply the equation
$$
v_s+J(v)v_t=g
$$
by $\gamma_{R+1}$ and obtain  for the product $\gamma_{R+1}v$ the equation
\begin{equation}\label{klap2}
\begin{split}
(\gamma_{R+1} v)_s& +J(0)(\gamma_{R+1}v)_t +\gamma_{R}\cdot(J(v)-J(0))(\gamma_{R+1}v)_t \\
&=\gamma_{R+1}'v +\gamma_{R+1}g.
\end{split}
\end{equation}
In view  of \eqref{klap1} and \eqref{klap2}, the map $u:=w-\gamma_{R+1}v$,  solves  the equation
\begin{equation}\label{reg1}
u_s+J(0)u_t +\gamma_{R}\cdot (J(v)-J(0))u_t = -\gamma_{R+1}'v.
\end{equation}
By the assumption of Proposition \ref{lem-reg1}, $v\in H^{m+3,\delta_m}(\R^+\times S^1, \R^{2n})$. Since $\gamma_{R+1}'(s)=0$ for $s\leq R+2$ and $s\geq R+3$, the right-hand side of the  equation  \eqref{reg1} belongs  therefore to $H^{m+3,\delta_{m+1}}(\R\times S^1,\R^{2n})$.   Consequently, in view of Lemma \ref{TORTE} and Proposition  \ref{c-iso}, 
there exists a unique map $\ov{u}\in H^{m+4,\delta_{m+1}}_c (\R\times S^1,\R^{2n})$ solving \eqref{reg1}.  Moreover, the map $u$ belongs, by construction, to the space 
$H_c^{m+3,\delta_{m}}(\R\times S^1,\R^{2n})$. 
The difference $u-\ov{u}$ is an element of the space $H^{m+3,\delta_m}_c (\R\times S^1, \R^{2n})$ and belongs to the kernel of
our linear isomorphism 
$H^{m+3,\delta_m}_c(\R\times S^1, \R^{2n})\rightarrow H^{m+2,\delta_m}(\R\times S^1, \R^{2n})$. Hence, 
$$
0=u-\ov{u} = (w-\gamma_{R+1}v) -\ov{u}
$$
from which we conclude that 
$$
\gamma_{R+1}v= w-\ov{u}\in H^{m+4,\delta_{m+1}}_c (\R\times S^1,\R^{2n}).
$$
In view of definition of the function $\gamma_R$ and the fact that $v$ is in $H^{m+4}_{\loc}$, we conclude that $v\in H^{m+4,\delta_{m+1}}(\R^+\times S^1,\R^{2n}).$  The proof of Proposition  \ref{lem-reg1} is complete.
\end{proof}
With Proposition \ref{lem-reg1} the proof of  Theorem \ref{thm-regularity1}  is complete.
\end{proof}

Next   we shall  prove the {\bf  sc-smoothness}  of the Cauchy-Riemann  section $f=\ov{\partial}_J$ of the strong polyfold bundle $p:Z\to W$. For this purpose we have to analyze the section in the local coordinates of the model $P:E\rightarrow X$  which is a strong bundle over the ep-groupoid $X$. We have constructed the object set $X$ as the disjoint union of graphs $\cg$ of good uniformizing families 
$$
(a,v,\eta)\rightarrow \alpha_{(a, v,\eta)}\equiv  (S_a,j(a,v),M_a,D_a,\oplus_a(\exp_u(\eta))
$$
where  $(a,v,\eta)\in {\mathcal O}$ defines the sc-structure on $\cg$.  The bundle $E$  over the $M$-polyfold $\cg$ is the associated graph 
$\wh{\cg}$ of the family 
$$
(a,v,\eta,\xi)\rightarrow \wh{\alpha}_{(a, v, \eta, \xi)}\equiv (S_a,j(a,v),M_a,D_a,\oplus_a(\exp_u(\eta)),
\Xi(a,v,\eta,\xi)).
$$
The section $F$  over the bundle $\wh{\cg}\to \cg$ is defined by 
$$
F((a,v,\eta),\alpha_{(a,v,\eta)})=
((a,v,\eta,\xi),\wh{\alpha}_{(a,v,\eta,\xi)}).
$$
where  $\xi$ is the unique solution of the equations
\begin{equation}\label{cauchy_riemann_eq_1}
\begin{aligned}
\Gamma  [ \oplus_a\exp_u (\eta),  \oplus_au]\circ \wh{\oplus}_a(\xi )\circ \delta (a, v)
&=\ov{\partial}_{J,j(a,v)}(\oplus_a(\exp_u(\eta)))\\
\wh{\ominus}_a(\xi)&=0
\end{aligned}
\end{equation}
so that $(a,v,\eta,\xi)\in K^{\mathcal R}$. In view of its functoriality properties, the section $F$ represents the Cauchy-Riemann section
$\ov{\partial}_J$ of the strong polyfold bundle $p$. In view of the sc-smooth structures on $\cg$ and $\wh{\cg}$, we have to establish 
 the sc-smoothness of the map
$$
(a,v,\eta)\mapsto \xi.
$$
In the following we may assume that the gluing parameter $a$ is sufficiently small. Since our constructions are local in the sense that the values of $\xi$ near a point $z_0\in S$ only depend on the other data evaluated near $z_0$ we can take a smooth partition of unity on $S$ and split the sc-smoothness into an analysis of $\xi$ in the nodal area and the analysis of $\xi$ in the core area. 

Fixing at first a point $z_0\in S$ which is not a nodal point, then for $\abs{a}$ small, the map $\xi$ near $z_0$ depends only on the properties of $\eta$ near $z_0$. In this case the map $(a, v,\eta)\to \xi$ near $z_0$ is level wise a smooth map in the classical sense. This implies that away from the nodal points, $\xi$ depends sc-smoothly on $(a, v, \eta)$.

 We next analyze  the complex anti-linear section  $\xi$ near the  nodal pair $\{x, y\}\in D$ in local coordinates. The data depend on $(a,v)$ and the properties of $\eta$ along the map $u:S\to Q$ near the nodal pair. 
 {In order to write the equation \eqref{cauchy_riemann_eq_1} for the complex anti-linear section $\xi$ in local coordinates we take  on $S$ the holomorphic coordinates $\sigma^+:\R^+\times S^1\to D_x$ and $\sigma^-:\R^-\times S^1\to D_y$  belonging to the small disk structure  and satisfying  $\sigma^+(s, t)\to x$ as $s\to \infty$ and $\sigma^-(s', t')\to y$ as $s'\to -\infty$. Moreover, we take the smooth chart $\psi:U\subset Q\to \R^{2n}$ around the point $u(x)=u(y)\in Q$ satisfying $\psi (u(x))=\psi (u(y))=0$. 
Then the map $u:S\to Q$ near  the nodal pair  $\{x, y\}$ is represented, in the local coordinates,  by the pair $(v^+,v^-)$ of  smooth maps defined as 
$$v^\pm (s, t)= \psi\circ u \circ \sigma^\pm(s,t):\R^\pm \times S^1\to \R^{2n}.$$
The vector field $\eta$ along $u$ is represented by the pair $(h^+, h^-)$ of vector fields along $v^\pm$ defined by 
$$h^\pm(s,t) = T\psi(u \circ\sigma^\pm(s,t))\circ \eta\circ\sigma^\pm(s,t).$$
We shall abbreviate the glued maps, respectively vector fields,  by 
 $$
 v_a=\oplus_a (v^+,v^-)\quad \text{and}\qquad h_a=\wh{\oplus }_a (h^+,h^-).
 $$ 
According to Section \ref{section-implant} the glued complex anti-linear map $\wh{\oplus}_a(\xi )$ is, in our local coordinates defined by 
$$T\psi (\oplus_a(u))\biggl(\wh{\oplus}_a(\xi) \cdot \frac{\partial}{\partial s}\biggr)=\wh{\oplus}_a(\xi^+, \xi^-)$$
where the pair $(\xi^+,\xi^-):(\R^+\times S^1)\times ( \R^-\times S^1)\to \R^{2n}$ is defined by 
\begin{align*}
\xi^+:=T\psi (u\circ \sigma^+)\cdot \xi (\sigma^+)\cdot \frac{\partial \sigma^+}{\partial s^{\phantom{+}}}\\
\xi^-:=T\psi (u\circ \sigma^-)\cdot \xi (\sigma^-)\cdot \frac{\partial \sigma^-}{\partial s^{\phantom{-}}}
\end{align*}
on the cylinders $\R^\pm\times S^1$.  For the following we abbreviate the vector field 
$$
\wh{\xi}=(\xi^+, \xi^-):(\R^+\times S^1)\coprod ( \R^-\times S^1)\to \R^{2n}.$$}

{We recall that the complex structure $j$ on $D_x\cup D_y$ does, by construction,  not  depend on the parameters $v$ and  $\delta (a, v)=1$. The polar coordinates are holomorphic and hence satisfy  
$$T\sigma^\pm\circ i=j\circ T\sigma^\pm.$$
For the exponential map with respect to the  Euclidean metric in $\R^{2n}$ we know from Proposition \ref{exp-formula1} that 
$$\oplus_a\exp_v(h)=\exp_{v_a}(h_a)=v_a+h_a.$$
Now recalling that the section $\ov{\partial }_{J, j}$ is defined by 
$$
\ov{\partial }_{J, j}(u)=\frac{1}{2}\bigl[ Tu+J(u)\circ Tu\circ j\bigr]
$$
we apply  the  two complex anti-linear maps in the equation \eqref{cauchy_riemann_eq_1} to the vector $\frac{\partial}{\partial s}$ and finally obtain the desired local equations 
\begin{equation}\label{gamma1}
\begin{aligned}
\wh{\Gamma}(v_a+h_a, v_a)\circ \wh{\oplus}_a(\wh{\xi})&=\frac{1}{2}\bigl[ \partial_s (v_a+h_a)+\wh{J}(v_a+h_a)\partial_t (v_a+h_a)\bigr]\\
\wh{\ominus}_a(\wh{\xi})&=0
\end{aligned}
\end{equation}
for the complex anti-linear section $\xi$, now represented by the pair of vector  fields $\wh{\xi}: (\R^+\times S^1)\coprod (\R^-\times S^1)\to \R^{2n}$. 
Here $\wh{J}$ is the smooth almost complex structure on $\R^{2n}$ defined by $\wh{J}(y)=T\psi (\psi^{-1}(y))\circ J(\psi^{-1}(y))\circ (T\psi (\psi^{-1}(y)))^{-1}$ at $y\in \R^{2n}$. 
In addition,  $\wh{\Gamma}:\R^{2n}\times \R^{2n}\to {\mathcal L}(\R^{2n},\R^{2n})$ is the local representation of the map $\Gamma$ arising from the parallel transport associated with  a connection which preserves the almost complex structure,  and hence,  for every $(p, q)$, the linear map $\Gamma (p, q)$ is invertible. Therefore, the map $\wh{\Gamma}$ is smooth and the matrices  $\wh{\Gamma}(y, y')$ are  invertible for every $(y, y')$.
 }

{By assumption, the map $u:S\to Q$ is smooth and satisfies $u(x)=u(y)$. Hence the pair $(v^+, v^-)$ of maps are smooth and have  vanishing asymptotic constants as $s\to \pm \infty$ and belong to the sc-Hilbert space 
$$H^{3,\delta_0}({\mathbb R}^+\times S^1,{\mathbb R}^{2n})\oplus H^{3,\delta_0}({\mathbb R}^-\times S^1,{\mathbb R}^{2n}). $$ 
The pair of vector fields $(h^+, h^-)$ along the curve $(v^+, v^-)$ has matching asymptotic constants and belongs to the sc-Hilbert space $E$ of pairs $(h^+, h^-)$ of the form 
$h^\pm=c+r^\pm$ for some constant $c\in \R^{2n}$, and the pair $(r^+,r^-)$ belongs to the above sc-Hilbert space. }

{In the following we let  $F$  be the sc-Hilbert space 
$$F= H^{2,\delta_0}({\mathbb R}^+\times S^1,{\mathbb R}^{2n})\oplus H^{2,\delta_0}({\mathbb R}^-\times S^1,{\mathbb R}^{2n}).$$
}

We introduce the smooth maps $A, B:\R^{2n}\oplus \R^{2n}\to \call (\R^{2n}, \R^{2n})$ by 
$$A(p, q)=\frac{1}{2}\wh{\Gamma}(p, q)^{-1}\qquad \text{and}\qquad B(p, q)=\frac{1}{2}\wh{\Gamma}(p, q)^{-1}\wh{J}(p).$$
Applying the inverse of the  map $\wh{\Gamma}$ to  both sides of the first equation in \eqref{gamma1}, we obtain 
\begin{equation*}
\begin{aligned}
\wh{\oplus}_a\bigl(\wh{\xi}\bigr) &= A(v_a+h_a, v_a)\circ\ \partial_s (v_a+h_a)+ B(v_a+h_a, v_a)\partial_t (v_a+h_a)\bigr]\\
\wh{\ominus}_a\bigl(\wh{\xi}\bigr)&=0.
\end{aligned}
\end{equation*}
Abbreviating 
\begin{equation*}
\wh{A}(p, q)=\begin{bmatrix}A(p, q)&0\\
0&\id
\end{bmatrix}
\quad \text{and}\quad 
\wh{B}(p, q)=\begin{bmatrix}B(p, q)&0\\
0&\id
\end{bmatrix},
\end{equation*}
the above equations for $\wh{\xi}$ can be written in matrix form as 
\begin{eqnarray*}
&&\wh{\boxdot}_a\bigl(\wh{\xi}\bigr)\\
&=&\wh{A}(v_a+h_a, v_a)\cdot \begin{bmatrix}
\partial_s (v_a+h_a)\\
0
\end{bmatrix}
+
\wh{B}(v_a+h_a, v_a)\cdot 
\begin{bmatrix}
\partial_t (v_a+h_a)\\
0
\end{bmatrix}\\
&=&\wh{A}(v_a+h_a, v_a)\circ \wh{\boxdot}_a\circ D_s^a(v+h)+\wh{B}(v_a+h_a, v_a)\circ \wh{\boxdot}_a\circ D_t^a(v+h)
\end{eqnarray*}
where  the maps $D_s^a$ and $D_t^a$ are the maps  introduced in Section \ref{preliminaries} .
Consequently, the map $\hb\oplus E\to F$, $(a, h)\mapsto  \wh{\xi}$ where $\wh{\xi}=(\xi^+, \xi^-)\in F$  is given by 
\begin{equation*}
\wh{\xi}=\wh{\boxdot}_a^{-1}\circ \wh{A}(v_a+h_a, v_a)\circ \wh{\boxdot}_a\circ D_s^a(v+h)+
\wh{\boxdot}_a^{-1}\circ \wh{B}(v_a+h_a, v_a)\circ \wh{\boxdot}_a\circ D_t^a(v+h).
\end{equation*}
 \begin{proposition}\label{nfilled}
 The map $B_\frac{1}{2}\oplus E\mapsto F$, $(a,h)\mapsto   \wh{\xi},$ introduced above,  is sc-smooth. 
 \end{proposition}
 \begin{proof}
 The translation map $E\to E$, $h\mapsto v+h$ is  sc-smooth. Hence, by Proposition \ref{XCX1}  and the chain rule,  the maps
 $\hb\oplus E\rightarrow F$, defined by 
$$
(a,h)\rightarrow D^a_s(v+h)\quad \text{and}\quad (a,h)\rightarrow D^a_t(v+h),
$$
are sc-smooth.  In view of the chain rule, it suffices  to show that  the maps
$\hb\oplus E\oplus F\to F$, defined by 
$$(a, h, g)\mapsto \wh{\boxdot}_a^{-1}\circ \wh{A}(v_a+h_a, v_a)\circ \wh{\boxdot}_a (g)$$
and 
$$(a, h, g)\mapsto \wh{\boxdot}_a^{-1}\circ \wh{B}(v_a+h_a, v_a)\circ \wh{\boxdot}_a (g), $$
are sc-smooth. 
We only consider the first map and derive the formula for $k=\wh{\boxdot}_a^{-1}\circ \wh{A}(v_a+h_a, v_a)\circ \wh{\boxdot}_a(g)$. In view of the definition of the total hat-gluing 
$ \wh{\boxdot}_a$ and of  the matrix $ \wh{A}(v_a+h_a, v_a)$, 
\begin{equation*}
\begin{aligned}
\wh{\oplus}_a(k)&=A(v_a+h_a)\wh{\oplus}_a (g)\\
\wh{\ominus}_a(k)&=\wh{\ominus}_a(g)
\end{aligned}
\end{equation*}
if $a\neq 0$. If $a=0$, then 
$$k=(k^+, k^-)=(A(v^++h^+, v^+)g^+, A(v^-+h^-, v^-)g^-).$$
Solving  the above  equations for $k=(k^+, k^-)\in F$ we obtain 
\begin{equation*}
k^+(s, t)=\frac{\beta_a}{\gamma_a}\cdot A(v_a+h_a)\cdot \wh{\oplus}_a(g)+\frac{\beta_a-1}{\gamma_a}\cdot \wh{ \ominus}_a(g)
\end{equation*}
for $s\geq 0$ 
and
\begin{equation*}
k^-(s-R, t-\vartheta)=\frac{1-\beta_a}{\gamma_a}\cdot A(v_a+h_a)\cdot \wh{\oplus}_a(g)+\frac{\beta_a}{\gamma_a}\cdot \wh{\ominus}_a(g).
\end{equation*}
for $s\leq R$. 

 {We recall that $\gamma_a=\beta_a^2+(1-\beta_a)^2$ and $\beta_a(s)=\beta (s-R/2)$  with  the  function $\beta$ introduced  
 in Section \ref{gluinganti-sect}.} 
 
  In view of the first statement in  Proposition \ref{XCX2}, the map $\hb\oplus E\oplus F\to H^{2,\delta_0}(\R^+\times S^1, \R^{2n})$, defined by $(a, h, g)\mapsto \frac{\beta_a}{\gamma_a}\cdot A(v_a+h_a)\cdot \wh{\oplus}_a(g)$,  is sc-smooth.  The map $\hb \oplus F\to H^{2,\delta_0}(\R^+\times S^1, \R^{2n})$, defined by 
$$(a,g)\mapsto \frac{\beta_a-1}{\gamma_a}\cdot \wh{\ominus}_a(g)(s, t)=\frac{(\beta_a-1)^2}{\gamma_a}g^+(s, t)+\frac{\beta_a(\beta_a-1)}{\gamma_a}g^-(s-R, t-\vartheta),$$
is sc-smooth in view of {Proposition 2.8}  in \cite{HWZ8.7}. The same arguments  show that  also the map $(a, h, g)\to k^{-}$ is sc-smooth. Consequently, the map 
$\hb\oplus E\oplus F\to F$, $(a, h, g)\mapsto \wh{\boxdot}_a^{-1}\circ \wh{A}(v_a+h_a, v_a)\circ \wh{\boxdot}_a (g)$
is sc-smooth.  This finishes the proof of Proposition \ref{nfilled}.
\end{proof}
 
Summing up,  we have verified that the map ${\bf f}:{\mathcal O} \to \cf$,  $(a, v, \eta)\mapsto {\bf f}(a, v, \eta)$,  is sc-smooth. This finishes the proof of the sc-smoothness and the regularization property of the Cauchy-Riemann section  section  $\ov{\partial}_J$  of the strong polyfold bundle $p:W\to Z$. 
The proof of Proposition \ref{SC-SMOOTHREG} is complete.\hfill $\blacksquare$

\section{The Filled Section, Proof of Proposition \ref{C0-CONT}}\label{section4.4}
Using the arguments  of Section \ref{sec5.2}, we first verify that (locally) the filled version ${\bf f}:O \to \cf$, $(a, v,\eta)\mapsto \xi={\bf f}(a, v, \eta)$ is sc-smooth. 
We recall  from \eqref{DEFR} that the complex anti-linear filled section $\xi$  is defined as the unique solution of the two equations
\begin{equation}\label{eq35}
\begin{aligned}
\Gamma (\oplus_a\exp_u (\eta), \oplus_a( u))\circ \wh{\oplus}_a(\xi)\circ \delta (a, v)&=\ov{\partial}_{J, j(a, v)}(\oplus_a \exp_u (\eta))\\
\wh{\ominus}_a \bigl(\xi)\cdot \frac{\partial}{\partial s} &=\ov{\partial}_{0}(\ominus_a (\eta )).\\
\end{aligned}
\end{equation}
 \begin{proposition}\label{locfilled}
 The map ${\bf f}:O \to \cf$, $(a, v, \eta)\mapsto \xi={\bf f}(a, v, \eta)$,  defined by  the two equations \eqref{eq35},  is sc-smooth.
 \end{proposition}
  \begin{proof}
  {Away from the nodal points,  the proof of the sc-smoothness of the map ${\bf f}$ is identical with  the proof of the sc-smoothness in Proposition \ref{SC-SMOOTHREG}.}
   {Near the nodes we go into the local coordinates as in Section \ref{sec5.2} in which the equation \eqref{eq35} becomes 
    \begin{equation}\label{pom}
  \begin{aligned}
  \wh{\Gamma}(v_a+h_a,v_a)\wh{\oplus}_a(\wh{\xi})& =\frac{1}{2}\bigl[\partial_s (v_a+h_a) +\wh{J}(v_a+h_a)\partial_t (v_a+h_a)\bigr]\\
  \wh{\ominus}_a\bigl(\wh{\xi})&=\ov{\partial}_{0}(\ominus_a(h)).
  \end{aligned}
  \end{equation}
  Here the vector field $\wh{\xi}$ represents locally the filled complex anti-linear section $\xi$. Moreover, $v_a=\oplus_a(v^+, v^-)$ and $h_a=\wh{\oplus}_a(h^+, h^-)$.  The operator 
 $ \ov{\partial}_{0}$ is defined by 
 $$\ov{\partial}_0=\partial_s+\wh{J}(0)\partial_t.$$
Using the decomposition $F=(\ker  \wh{\oplus}_a)\oplus  (\ker  \wh{\ominus}_a)$, we split $\wh{\xi}=\xi_1+\xi_2$ where   $\xi_1\in \ker  \wh{\ominus}_a$ and $\xi_2\in \ker  \wh{\oplus}_a$. Then  the  vector field  $\xi_1$ solves the equations
\begin{equation}\label{pom1}
  \begin{aligned}
 \wh{\Gamma}(v_a+h_a,v_a)\wh{\oplus}_a(\xi_1)& =\frac{1}{2}\bigl[\partial_s (v_a+h_a) +\wh{J}(v_a+h_a)\partial_t (v_a+h_a)\bigr]\\
  \wh{\ominus}_a(\xi_1)&=0.
  \end{aligned}
  \end{equation}
 These are the same equations as studied before, hence the map 
 $\hb\oplus E\to F$, defined by $(a, h)\mapsto \xi_1$ is sc-smooth in view of Proposition \ref{nfilled}.  The  vector field  $\xi_2$ solves the following equations
 \begin{equation}
 \begin{aligned}
\wh{\oplus}_a(\xi_2)& =0\\
\wh{\ominus}_a(\xi_2)&=\ov{\partial}_{0}(\ominus_a(h)).
\end{aligned}
\end{equation}
}
Solving for $\xi_2=(\xi^+, \xi^-)$ 
{one finds
\begin{align*}
(\xi_2)^+(s, t)&=\frac{\beta_a-1}{\gamma_a}\ov{\partial}_0(h_a)\\
(\xi_2)^-(s-R,t-\vartheta)&=\frac{\beta_a}{\gamma_a}\ov{\partial}_0(h_a).
\end{align*}
Using the computations  in Appendix \ref{XCX10}, one obtains the formula 
\begin{equation*}
\xi_2=C_s^a(h)+\wh{J}(0)C_t^a(h).
\end{equation*}
By Proposition \ref{XCX1} the maps $C_s^a$ and $C_t^a$ are sc-smooth, so that  the map $(a, h)\mapsto \xi_2$ is sc-smooth and the proof of Proposition \ref{locfilled} is complete. }
 \end{proof}
 
 \begin{proposition}\label{prop4.16}
 The map ${\bf f}:O\to \cf$  defined by \eqref{eq35} is a filled version of the Cauchy-Riemann section $F=\ov{\partial}_{J}$ in the ep-groupoid representation.
 \end{proposition} 
 \begin{proof}
 We already know that ${\bf f}$ is sc-smooth and it remains to verify the three properties of a filling in Definition \ref{filled-def}. If $(a, v, \eta)\in \co$, then $\ominus_a(\eta)=0$ and hence $\wh{\ominus}_a(\xi)=0$ so that ${\bf f}(a, v, \eta)=F(a, v,\eta)$. In order to verify the second property of a  filler we assume that $\wh{\ominus}_a{\bf }(\xi)=0$. It follows from \eqref{eq35} that $\ov{\partial}_0(\ominus_a (\eta))=0$ and hence $\ominus_a(\eta)=0$ in view of Proposition \ref{c-iso}, so that indeed $(a, v,\eta)\in \co$.
 
 In order to verify the third property of a filler we abbreviate the points in $O$ by $y=(a, v, \eta)$ and let $y_0=(0,0, 0)$. At the point $y_0$, the derivative of the retraction $r$ is the identity map and consequently,  $\ker Dr(y_0)=\{0\}$. Since $\rho(y_0)=\wh{\pi}_0=\id$, we have $\ker \rho (y_0)=\{0\}$. Therefore, the linearization of the map $y\mapsto \bigl[ \id -\rho (r(y))\bigr]{\bf f}(y)$ at the point $y_0$, restricted to $\ker Dr(y_0)$, defines trivially the  isomorphims $\{0\}=\ker Dr(y_0)\to \ker \rho (y_0)=\{0\}$. The proof of Proposition \ref{prop4.16} is complete.
 \end{proof}
 
 In order to complete the proof of Proposition \ref{C0-CONT}, we define the $\ssc^+$-section ${\bf s}$ by its principal part ${\bf s}(a, v, \eta)={\bf f}(0,0, 0)$ so that the section ${\bf h}={\bf f}-{\bf s}$ satisfies ${\bf h}(0,0,0)=0$. We shall now construct a strong bundle isomorphism $\Phi$ having the property that  the push-forward section $\Phi_* ({\bf h})$ is an $\ssc^0$-contraction germ. The difficulty encountered comes from the fact that  although the map $(a, v, \eta)\mapsto {\bf f}(a, v, \eta)$ is sc-smooth and its linearizations $D{\bf f}(a, v, \eta)$ are bounded linear maps, the dependence of these linear operators on the arguments $(a, v,\eta)$ is not continuous in the operator norm. The following proposition will be very helpful.
\begin{proposition}\label{PROPX} 
The filled version $O\to \cf$, $(a, v, \eta)\mapsto {\bf f}(a, v, \eta)$ has the  following properties.
\begin{itemize}
\item[$\bullet$] At every smooth point $(a,v,\eta)$ close to $(0,0,0)$ the linearization 
$$
D{\bf f}(a,v,\eta)\in \call \bigl( (N\oplus E\oplus \ce ), \cf\bigr)
$$
is an sc-Fredholm operator. Their Fredholm indices are all the same, say equal to $d$.  
\item[$\bullet$]  For every level $m$  fixed and small  $(a,v)$, the map 
$$\eta\mapsto {\bf f}(a,v,\eta),\qquad \ce_m\to \cf_m$$ 
is of class $C^1$ in the  Fr\'echet sense.
\item[$\bullet$]  If ${\mathcal K}$ is a finite-dimensional sc-complemented linear subspace of $\ce$,  the following holds  true. 
We consider a sequence $(a_j,v_j,k_j)\in O$ in which $k_j\in {\mathcal K}$  converges  in $\ce_m$ to some point $(a_0,v_0,k_0)\in O$. If  the sequence  $(h_j)\subset \ce_m$ is bounded on level $m$  and satisfies 
$$D_3{\bf f}(a_j,v_j,k_j)\cdot h_j=y_j+z_j$$
where $y_j\in \cf_m$ satisfies 
$y_j\rightarrow 0$ in $\cf_m$ and where  $(z_j)\subset \cf_{m+1}$ is a bounded sequence  in $\cf_{m+1},$ then  the sequence $(h_j)$ has a subsequence which converges in $\ce_m$. 
\end{itemize}
\end{proposition}
 We should note that the points in the finite dimensional sc-complemented subspace ${\mathcal K}$ above, are all smooth.
 We postpone the proof of Proposition \ref{PROPX} to the next section and use it next in the proof of Proposition \ref{C0-CONT}.

We denote by $K\subset \ce$ the kernel of the  sc-Fredholm operator
$$
D_3{\bf f}(0,0,0):\ce\rightarrow \cf. 
$$
{In view of the definition of an sc-Fredholm operator, the finite dimensional kernel $K$ possesses an sc-complement $X$ so that  $\ce=K\oplus X$.}
By $Y=D_3{\bf f}(0,0,0)X$ we denote the range of the  operator and, again by definition of an sc-Fredholm  operator, there is an   
sc-complement $C$ so that $\cf=Y\oplus C$. The projection 
$P:\cf\to \cf$  onto $Y$ along $C$  is an  sc-Fredholm operator whose index is equal to  {$0$}. 
We now consider the family 
$$
(a,v,k)\mapsto P\circ D_3{\bf f}(a,v,k)\vert X
$$
of sc-Fredholm operators from $X$ into $Y$, where $k\in K$.
We emphasize that these linear  bounded operators do not depend as operators 
continuously on the finite dimensional parameter  $(a,v,k)$, which we abbreviate 
by 
$$b\equiv (a, v, k)\in  N\oplus E\oplus K.$$
However, in view of Proposition \ref{locfilled},  the map  
$$
(O\cap (N\oplus E\oplus K))\oplus  \ce \rightarrow  \cf,
$$
defined by 
\begin{equation}\label{map1}
(b, h)\mapsto D_3{\bf f}(b)\cdot h,
\end{equation}
is sc-smooth and raising the regularity index by $1$,  the map \eqref{map1} induces,  {in view of Proposition 2.2 in \cite{HWZ8.7}},  the sc-smooth map 
$$\bigl[ (O\cap (N\oplus E\oplus K))\oplus \ce\bigr]^1 \to \cf^1.$$
Since $K$ is finite dimensional, the left-hand side is equal to 
$$
\bigl(O\cap (N\oplus E\oplus K)\bigr)^1\oplus \ce^1 =\bigl(O\cap (N\oplus E\oplus K)\bigr)\oplus \ce^1 
$$ 
and we conclude that the two maps 
$$
\bigl(O\cap (N\oplus E\oplus K)\bigr)\oplus \ce^i\to \cf^i
$$ 
defined by \eqref{map1} are sc-smooth for $i=0$ and $i=1$. 

We next claim  that the operator $P\circ D_3{\bf f}(b)\vert X_{0}:X_0\rightarrow Y_0$ has a trivial kernel if $b$ is close to $0$. 
Indeed, otherwise  there exist sequences $b_j\to 0$ and $(h_j)\subset X_0\subset \ce_0$  satisfying $\abs{h_j}_{\ce_0}=1$ and 
$P\circ D_3{\bf f}(b_j)h_j=0$. In particular,  $D_3{\bf f}(b_j)h_j = z_j$  is a bounded sequence on every level, and we find by Proposition \ref{PROPX} a converging subsequence $h_j\to h$ in $\ce_0$.
Passing to the limit, $h$ satisfies   $P\circ D_3{\bf f}(0)h=0$. Consequently, $h=0$ in contradiction to  $\abs{h}_{\ce_0}=1$  and hence proving the claim. 

We have found an open neighborhood $\wh{O}$
of $0\in N\oplus E\oplus K$ having the property that for $b\in  \wh{O}$ the sc-Fredholm operator $P\circ D_3{\bf f}(b)\vert X :X\to Y$ is injective.   It is, in addition, an sc-Fredholm operator of index $0$.  To see this, we observe that $D{\bf f}(b)\vert X:X\to \cf$ is the composition of the sc-injection $X\to X\oplus K$ whose Fredholm index is equal to $-\dim K$ and $D{\bf f}(b):\ce \to \cf$ having, in view of Proposition \ref{PROPX}, index $d$. Therefore, the index of the composition is equal to $-\dim K +d$. Using Proposition \ref{PROPX} once more, we obtain for the composition $P\circ D_3{\bf f}(b)\vert X :X\to Y$ the index $\dim C-\dim K+d=-d+d=0$ as claimed. 
We conclude that the Fredholm operators 
$$P\circ D_3{\bf f}(b)\vert X:X\to Y$$
are all sc-isomorphisms if $b$ is close to $0$. 
Moreover,  perhaps replacing $\wh{O}$ by a smaller open  neighborhood of $0$,  we claim that there exists,  for every level $m\geq 0$,  a constant $c_m>0$ such  that
\begin{equation}\label{maps2}
\abs{P\circ D_3{\bf f}(b)h}_{\cf_m}\geq c_m\cdot \abs{h}_{\ce_m}
\end{equation}
for all $h\in X_m$ and for all $b\in \wh{O}$.

In order to prove the estimate \eqref{maps2}, we take an open neighborhood $\wh{O}^\ast$ of $b=0$ having a  compact closure in $\wh{O}$.  Arguing indirectly, we find  a   level $m$ and sequences 
$(h_j)\subset X$  and $(b_j)\subset \wh{O}^\ast$ satisfying $\abs{h_j}_{\ce_m}=1$ and 
 $\abs{P\circ D_3{\bf f}(b_j)h_j}_{\cf_m}\rightarrow 0$.  We may assume, going over to a subsequence that $b_j\to b_0\in \wh{O}^\ast$. We conclude from Proposition \ref{PROPX} the convergence $h_j\to h_0$ in $\ce_m$ of a subsequence, so that 
$
P\circ D_3{\bf f}(b)h_0=0.
$
In view of the injectivity, $h_0=0$, in contradiction to $\abs{h_0}_{\ce_m}=1$ and hence proving the estimates \eqref{maps2}.

In view of our discussion so far, we can now apply the following result from \cite{HWZ8.7},  Proposition 4.8. 
\begin{proposition}\label{WM40}
Assume that  $V$ is an open subset of a finite-dimensional vector space $H$, and $E$ and $F$ are sc-Banach spaces and  consider a family  $v\mapsto L(v)$ of linear operators parametrized by $v\in V$ 
such that  $L(v):E\rightarrow F$ are  sc-isomorphisms having the following properties.
\begin{itemize}\label{property-99}
\item[{\em (1)}] The map $\wh{L}:V\oplus E\rightarrow F$,  defined by 
$$
\wh{L}(v, h):=L(v)h,
$$
is sc-smooth.
\item[{\em (2)}] There exists for every $m$ a constant $C_m$ such that
$$
\abs{L(v)h}_m\geq C_m\cdot \abs{h}_m
$$
for all $v\in V$  and all $h\in E_m$.
\end{itemize}
Let us note that $ L(v)$ is not assumed to be continuously depending on $v$ as an operator.
Since the map $L(v):E\to F$ is an sc-isomorphism, the  equation
$$
L(v)h=k
$$
has  for every pair $(v,k)\in V\oplus F$  a unique solution $h$ denoted by 
$$h=L(v)^{-1}k=:f(v, k),$$
 so that $\wh{L}(v, f(v, k))=k$. Then the 
 map $f:V\oplus F\rightarrow E$ defined above is sc-smooth.
\end{proposition}

From Proposition \ref{WM40} we deduce that the solution operator 
$$
(N\oplus E\oplus K)\oplus Y\to X,\quad (a, v, k, y)\mapsto x, 
$$
defined by $x=\bigl[P\circ D_3{\bf f}(a, v, k)\bigr]^{-1}\cdot y$,  is an sc-smooth map. We shall use this fact in order to define a local strong bundle isomorphism which transforms the section  section ${\bf f}-{\bf s}=:{\bf h}$ into an $\ssc^0$-contraction form. {Here ${\bf s}$ is an $\ssc^+$-section. It is defined as the constant section ${\bf s}(a, v, k)=f(0,0,0)$ for all $(a, v, k)\in N\oplus E\oplus K$ and is an $\ssc^+$-section because ${\bf f}(0, 0, 0)\in {\mathcal F}_\infty$. It then follows that ${\bf h}(0,0, 0)=0$.} 

We first note that 
$B=N\oplus E\oplus K$ is a finite-dimensional sc-subspace and so  isomorphic to $\R^n$ for a suitable $n$ and we may assume that  $B=\R^n$ so that our  sc-space becomes 
$$
N\oplus E\oplus \ce=(N\oplus E\oplus K)\oplus X=B\oplus X=\R^n\oplus X.
$$
Similarly, $\cf =C\oplus Y$ with $Y=P\cf$ where $C$ is a finite dimensional sc- complement  and we may assume $C=\R^N$, so that 
$$\cf=\R^N\oplus Y.$$
We now choose an open neighborhood $O$ of  the origin $(0,0)\in {\mathbb R}^n\oplus X$ having the property that $(b, x)\in O$ implies $b\in \wh{O}^*$, and define the bundle map 
\begin{equation*}
\begin{gathered}
\Phi: O\triangleleft (\R^N\oplus Y)\rightarrow O\triangleleft (\R^N\oplus X)\\
\Phi((b,x),(c,y)) = \bigl( (b,x)\bigl(c, \bigl[ P\circ D_3{\bf f}(b)\bigr]^{-1}y\bigr)\bigr).
\end{gathered}
\end{equation*}
It  is a strong bundle isomorphism in view of Proposition \ref{WM40} and the previous discussions.  Recalling that ${\bf f}={\bf s}+{\bf h}$ for the constant section ${\bf s}(a, v, k)={\bf f}(0,0, 0)$, the principal part of the  push-forward section ${\bf k}=\Phi_\ast {\bf h}$ is given by the following formula, \begin{equation}\label{push1}
{\bf k}(b, x)=\bigl( (\id -P){\bf h}(b, x), \bigl[ P\circ D_2{\bf h}(b, 0)\bigr]^{-1}P{\bf h}(b, x)\bigr).
\end{equation}

The map \eqref{push1} is a germ of a map (near $(0,0)$)
$$\R^n\oplus X\to \R^N\oplus X.$$
If $Q:\R^N\oplus X\to X$ is the sc-projection onto $X$, we consider the map $Q{\bf k}:\R^n\oplus X\to X$, defined by 
$$Q{\bf k}(b, x)= \bigl[ P\circ D_2{\bf h}(b, 0)\bigr]^{-1}P{\bf h}(b, x),$$
and introduce the germ of a map $H:\R^n\oplus X\to X$ as 
$$
H(b, x):=x-Q{\bf k}(b, x).
$$
Then $H(0, 0)=0$,  and we have to prove the following estimate. Given a level $m$ in $X$ and $\varepsilon>0$, then 
\begin{equation}\label{h1}
\abs{H(b, x)-H(b, x')}_{X_m}\leq \varepsilon\cdot \abs{x-x'}_{X_m}
\end{equation}
for all $\abs{b}$, $\abs{x}_{X_m}$ and $\abs{x'}_{X_m}$ small enough. In order to prove this crucial estimate, we make use of the following result, which will be verified in the next section.

\begin{proposition}\label{PROP4.19}
For  the filled version ${\bf f}:O\rightarrow \cf$  defined on the open neighborhood $O$ of  $(0,0,0)\in N\oplus E\oplus {\mathcal E}$ the following estimate  holds.
If the  level $m\geq 0$  and the positive number $\lambda$ are given,  then
$$
\abs{\bigl[ D_3{\bf f}(a,v,0)-D_3{\bf f}(a,v,\eta)\bigr] h}_{\cf_m}\leq \lambda\cdot \abs{h}_{\ce_m}
$$
for all $h\in \ce_m$ and for $(a,v,\eta)\in O$ small enough on level $m$.
\end{proposition}

Abbreviating $b=(a, v, k)$ we deduce from Proposition \ref{PROP4.19} for $x\in X$ small on level $m$ and for $b$ small and for all $h\in X_m$ the estimate 
\begin{equation*}
\begin{split}
&\abs{\bigl[ D_2{\bf f}(b,x)-D_2{\bf f}(b,0)\bigr] h}_{\cf_m}=\abs{D_3{\bf f}(a,v,\eta)h- D_3{\bf f}(a,v,k)h}_{\cf_m}\\
&\leq  \abs{D_3{\bf f}(a,v,\eta)h-D_3{\bf f}(a,v,0)h}_{\cf_m} + \abs{D_3{\bf f}(a,v,0)h-D_3{\bf f}(a,v,k)h}_{\cf_m}\\
&\leq 2\cdot \lambda\cdot \abs{h}_{X_m}.
\end{split}
\end{equation*}
Since ${\bf s}$ is constant, the same estimate holds for the section ${\bf h}$.  Having this estimate we compute,  assuming that all data are small enough, 
\begin{equation*}
\begin{split}
&\abs{ H(b,x)-H(b,x') }_{X_m}\\
&=\abs{ x-x' + Q{\bf k}(b,x')-Q{\bf k}(b,x) }_{X_m}\\
&=\abs{ \bigl[ PD_2{\bf h}(b,0)\bigr] ^{-1} \lbrace PD_2{\bf h}(b,0)(x-x') - \bigl[ P{\bf h}(b,x)-P{\bf h}(b,x')\bigr]\rbrace }_{X_m}\\
&\leq c_m^{-1}\cdot \abs{ PD_2{\bf h}(b,0)(x-x') - \bigl[ P{\bf h}(b,x)-P{\bf h}(b,x')\bigr]  }_{\cf_m}\\
&=c_m^{-1}\cdot \abs{ \int_0^1 \bigl[ PD_2{\bf h}(b,0)- PD_2{\bf h}(b,tx+(1-t)x')\bigr] (x-x')\ dt }_{\cf_m}\\
&\leq  c_m^{-1}\cdot d_m\cdot \int_0^1\abs{\bigl[ D_2{\bf h}(b,0)- D_2{\bf h}(b,tx+(1-t)x')\bigr] (x-x') }_{\cf_m} \ dt\\
&\leq  c_m^{-1}\cdot d_m\cdot \int_0^1 2\lambda \abs{x-x'}_{X_m}\ dt\\ 
&= 2c_m^{-1}\cdot d_m\cdot \lambda\cdot \abs{x-x'}_{X_m}.
\end{split}
\end{equation*}
We have verified the desired estimate \eqref{h1} and so have confirmed the 
$\ssc^0$-contraction germ property of the filled section ${\bf f}$.

To sum up, we have studied the Cauchy-Riemann section $\ov{\partial }_J$ in local coordinates near a smooth point $0$ and have constructed a filled version ${\bf f}$ and an $\ssc^+$-section ${\bf s}$ satisfying ${\bf f}(0)={\bf s}(0)$, and have found a strong bundle isomorphism $\Phi$ such that the push-forward $\Phi_\ast ({\bf f}-{\bf s})$ is a germ belonging to the distinguished  germs $\mathfrak{C}_{basic}$.
This proves that $\ov{\partial }_J$ is a Fredholm  germ at the point $0$.

Apart from the proof of Proposition \ref{PROPX} and Proposition \ref{PROP4.19}, which will be carried out in the next section, the proof of Proposition \ref{C0-CONT} is complete.\\
\mbox{} \hfill $\blacksquare$

We would like to summarize the previous discussion for further references in the following proposition. 

\begin{proposition}
Let $U\subset E$ be an open neighborhood of $0$ in the  sc-Banach space $E$ and  let ${\bf f}:U\rightarrow F$ be an sc-smooth map into the sc-Banach space $F$, which satisfies the following  three conditions. 
\begin{itemize}
\item[{\em (1)}] At every  smooth  point $x\in U$,   the linearisation $D{\bf f}(x):E\rightarrow F$ is an sc-Fredholm operator.  Its Fredholm index
does not depend on $x$.
 \item[{\em (2)}] There is an sc-splitting $E=B\oplus X$ in which $B$ is a finite-dimensional subspace of $E$ containing the kernel  of $D{\bf f}(0)$,  for which the following holds for $b\in B$ small enough. 
 If $(b_j)\subset B$ is a sequence converging to $b\in B$ and if $(\eta_j)\subset X$ a sequence  bounded on level $m$ and satisfying  
 $$D{\bf f}(b_j)\cdot \eta_j=y_j+z_j$$
 where $y_j\to 0$ in $F_m$ and where the sequence $(z_j)$ is bounded in $F_{m+1}$, then the sequence
 of  $(\eta_j)$
possesses a convergent subsequence  in $X_m$.
\item[{\em (3)}]  If  $m\geq 0$ and $\varepsilon>0$ are given,  then 
$$
\abs{\bigl[ D_2{\bf f}(b,0)-D_2{\bf f}(b,x)\bigr] h}_{F_m}\leq \varepsilon\cdot \norm{h}_{E_m}
$$
for all  $|b|$ and $x$ small enough on level $m$.
\end{itemize}
{Define the $\ssc^+$-map ${\bf s}$ as the constant section ${\bf s}(b, x)={\bf f}(0, 0)\in F_\infty$. Then  
under the conditions (1)--(3),  the map ${\bf h}={\bf f}-{\bf s}$ is conjugated near $0$ to an $\ssc^0$-contraction germ.}
\end{proposition}
The proof follows  along the same lines as  above.  The properties (1) and (2) can be used
to define the coordinate change $\Phi$.  The property (3) then allows
to carry out the desired estimate. Let us remark that the conclusion is in fact true for every ${\bf s}$ which satisfies
${\bf s}(0,0)={\bf f}(0,0)$. This follows from one of the results in \cite{HWZ3.5} which is concerned with the perturbation theory.
\begin{remark} We would like to mention that there is a variation of the whole scheme which will be important  for  compactness questions in the context of operations in the SFT.
Namely rather than assuming that $u_0$ is on the infinity level we just assume that it is on the  level zero. In that case we can still define a map ${\bf f}$
$$
{\bf f}:O\rightarrow {\mathcal F},
$$
on  an open neighborhood $O$  of $(0,0,0)$ in $N\oplus E\oplus{\mathcal E}$. But in this case
${\mathcal E}$ is simply  a Banach space. However it still makes sense to distinguish two levels  on ${\mathcal F}$, namely the $0-$ and the $1$-level.
The map ${\bf f}$ is continuous and for fixed $(a,v)$ the map $\eta\rightarrow {\bf f}(a,v,\eta)$ is $C^1$ (even $C^\infty$). We can consider
a bundle $O{\triangleleft}\  {\mathcal F}\rightarrow O$, where the base always has level $0$, but the fiber has two levels. So the bi-levels
$(0,0)$ and $(0,1)$ are well-defined. It makes sense to define continuous $\ssc^+$-sections, namely sections which
map points in $O$ continuously to ${\mathcal F}^1$. We can consider bundle maps quite similarly to the strong bundle maps, but where we only require them
to be continuous and preserving this simple bi-level structure. If we start with an arbitrary strong local bundle we have an underlying structure consisting of  the two bi-levels $(0,0)$ and $(0,1)$. Then it makes sense to  look for  a continuous bundle isomorphism which conjugates ${\bf f}-{\bf s}$, where ${\bf s}$ is 
a continuous $\ssc^+$-section  satisfying  ${\bf f}(0,0,0)={\bf s}(0,0,0)$,  to a map which has the contraction germ property on level $0$ (which is the only level which makes sense).  This  can be used to study compactness questions near elements which are not on the infinity-level,
because the particular contraction type form allows to  control  the  compactness nearby.
\end{remark}

\section{Proofs of  Proposition \ref{PROPX} and Proposition \ref{PROP4.19}}
This section is devoted to the proofs of Proposition \ref{PROPX} and Proposition \ref{PROP4.19} concerning the filled version 
$${\bf f}:O\to {\mathcal F},\quad (a, v, \eta)\mapsto {\bf f}(a, v,\eta),$$
where $(a, v, \eta)\in O$ and $O$ is an open neighborhood of $(0,0, 0)$ in $N\oplus E\oplus {\mathcal E}$. The spaces ${\mathcal E}$ and ${\mathcal F}$ are  introduced in Section \ref{c-r-results}. From the formula for the derivative $D_3{\bf f}(a, v, \eta)$ below in \eqref{identity_25}, it is evident that the proofs can be localized and we shall prove the local versions of the propositions and distinguish between the ``classical case'' and the ``nodal case''. The classical case deals with maps defined on parts of $S$ away from the nodal pairs, while the nodal case deals with maps near the nodal points. Then the propositions follow from the local versions by means of a finite partition of unity argument.

We should  recall that, in general, the Fredholm section ${\bf f}$ is merely sc-smooth and not smooth in the classical sense. This is due to the occurrence of a finite dimensional space of  ``bad'' coordinates, which in our case are the coordinates $(a, v)$. If one keeps $(a, v)$ fixed, the map $\eta\mapsto {\bf f}(a, v, \eta)$ is classically smooth on every level. But in general, the  linear operator $D_3{\bf f}(a, v, \eta)$ does not even depend continuously on $(a, v)$ in the operator norm,  and the polyfold Fredholm property is a consequence of a certain uniformity in the parameters $(a, v)$.

We recall that the filled section map $ (a, v, \eta)\mapsto {\bf f}(a, v,\eta)=\xi$ is defined as the unique solution of the equations 
\begin{equation}\label{defined10}
\begin{gathered}
\Gamma(\exp_{u_a}(\oplus_a(\eta)),u_a)\circ\wh{\oplus}_a(\xi)\circ\delta(a,v)= \bar{\partial}_{J,j(a,v)} (\exp_{u_a}(\oplus_a(\eta)))\\
\wh{\ominus}_a(\xi){\cdot \frac{\partial}{\partial s}}=\bar{\partial}_0 (\ominus_a(\eta)).
\end{gathered}
\end{equation}
Abbreviating  $\Theta(q,p):=\Gamma^{-1}(p,q)$ and $\tau(a,v):=\delta(a,v)^{-1}$, the  equations \eqref{defined10} take  the form 
\begin{equation}\label{defined1000}
\begin{gathered}
\wh{\oplus}_a(\xi)=\Theta(u_a,\exp_{u_a}(\oplus_a(\eta)))\circ \bar{\partial}_{J,j(a,v)} (\exp_{u_a}(\oplus_a(\eta)))\circ\tau(a,v)\\
\wh{\ominus}_a(\xi){\cdot \frac{\partial}{\partial s}}=\bar{\partial}_0 (\ominus_a(\eta)).
\end{gathered}
\end{equation}

In order to deal first  with the classical case we consider only variations of $\eta$ belonging to the space 
 ${\mathcal E}^{\#}$ consisting of all vector fields $\eta$ along $u$ satisfying $\eta (z)=0$ for  all $z$ contained in the union 
 $${\bf D}(-1)= \bigcup_{x\in \abs{D}}D_x(-1)$$
 of the concentric subdisks of the small disk structure as introduced in Section \ref{xA}.  The restriction of ${\bf f}$ to the complement $S\setminus {\bf D}(-1)$ does not depend on the gluing 
 parameter $a$. Therefore, the filled section ${\bf f}$ satisfies for $\eta\in {\mathcal E}^{\#}$ the identity
\begin{equation}\label{MMM}
{\bf f}(a,v,\eta)=\Theta(u,\exp_{u}(\eta))\circ \bar{\partial}_{J,j(v)}(\exp_{u}(\eta))\circ \tau(v),
\end{equation}
where $\tau (v)=\delta (v)^{-1}$. Denoting by ${\mathcal E}^{\#\#}$ the space  of vector fields  $\eta\in {\mathcal E}$
satisfying $\eta (z)=0$ for all $z\in \bigcup_{x\in \abs{D}} D_x(-\frac{1}{2})$, we deduce  the following result by means of the mean value theorem.
\begin{P4.19}[{\bf Classical Case}]
We fix a level $m\geq 0$ and $\lambda>0$. Then 
\begin{equation*}
\abs{D_3{\bf f}(a,v,\eta)h-D_3{\bf f}(a,v,0)h}_{{\mathcal F}_m}\leq \lambda\cdot \abs{h}_{{\mathcal E}_m}.
\end{equation*}
for all $h\in  {\mathcal E}^{\#\#}_m$  and for  $(a,v,\eta)\in O_m$  sufficiently close to $(0,0,0)$ and $\eta\in  {\mathcal E}^{\#\#}_m$.
\end{P4.19}

In  order to deal with the classical case of Proposition \ref{PROPX}   we again consider the section  ${\bf f}$ on the complement $S\setminus  {\bf D}(-1)$ where it is defined by the formula  \eqref{MMM}.  We fix a Hermitian connection $\nabla$ for $TQ\rightarrow Q$, where the fibers are equipped with the complex multiplication $J$, see Appendix \ref{connection} for  more details.
Viewing ${\bf f}$ as a section of the trivial bundle $O\triangleleft{\mathcal F}\rightarrow O$ we fix the trivial connection for the latter so that
$\nabla_{(0,0,h)} {\bf f}(a,v,\eta) =D_3{\bf f}(a,v,\eta)h$. Having fixed these connections we obtain various other connections in standard ways, e.g.
for the function spaces of sections associated to $TQ\rightarrow Q$ and certain Hom-bundles, see \cite{El}.  We compute  the covariant derivative with respect to the third variable  at the section $\eta$ in the direction of the section $h$ of the bundle $u^*TQ\to S$ and obtain  from \eqref{MMM} the formula
\begin{equation}\label{identity_25}
\begin{split}
D_3{\bf f}(a,v,\eta)h&=(\nabla_{(T\exp_{u}(\eta)h, 0)}\Theta)[\exp_{u}(\eta), u] \circ \ov{\partial}_{J,j(v)}(\exp_{u}(\eta))\circ \tau(v)\\
&\phantom{=}+\Theta(\exp_{u}(\eta), u) \circ (\nabla_h\ov{\partial}_{J,j(v)})(\exp_{u}(\eta))\circ \tau(v)\\
&=(\nabla_{(T\exp_{u}(\eta)h, 0)}\Theta)[\exp_{u}(\eta), u] \circ \Gamma(\exp_{u}(\eta),u)\circ {\bf f}(a,v,\eta)\\
&\phantom{=}+\Theta(\exp_{u}(\eta), u) \circ (\nabla_h\ov{\partial}_{J,j(v)})(\exp_{u}(\eta))\circ \tau(v).
\end{split}
\end{equation}
{Note that $\Theta$ depends on two variables and we only differentiate with respect to the first one.
This explains the $0$ in $(T\exp_{u}(\eta)h, 0)$ occurring in $(\nabla_{(T\exp_{u}(\eta)h, 0)}\Theta)[\exp_{u}(\eta), u]$}.
We now  assume, as in the statement of  Proposition \ref{PROPX},  that  the sequence $(a_j,v_j,\eta_j)\subset O$ converges  to $(a_0,v_0,\eta_0)\in O$, where $\eta_j$ and $\eta_0$  belong to the  finite-dimensional sc-complemented subspace $K$ of ${\mathcal E}$ so  that $\eta_j\to \eta_0$ in ${\mathcal E}_k$ for every level  $k\geq 0$. In addition, we assume  that  the sequence $(h_j)\subset {\mathcal E}_m$ is bounded  and has the form 
\begin{equation}\label{convergence_0}
D_3{\bf f}(a_j,v_j,\eta_j)h_j=y_j+z_j,
\end{equation}
where 
$y_j\in {\mathcal F}_m$ satisfies $y_j\to  0$ in ${\mathcal F}_m$ and where the sequence $(z_j)\in {\mathcal F}_{m+1}$ is bounded in ${\mathcal F}_{m+1}$.  In view of the formula \eqref{identity_25}  for the derivative $D_3{\bf f}(a,v,\eta)h$, 
\begin{equation}\label{convergence_1}
\begin{split}
D_3&{\bf f}(a_j,v_j,\eta_j)h_j\\
&=(\nabla_{(T\exp_{u}(\eta_j)h_j,0)}\Theta)[ \exp_{u}(\eta_j), u] \circ\Gamma(\exp_{u}(\eta_j),u)\circ {\bf f}(a_j,v_j,\eta_j)\\
&\phantom{=}+\Theta(\exp_{u}(\eta_j),u) \circ (\nabla_{h_j}\ov{\partial}_{J,j(v_j)})(\exp_{u}(\eta_j))\circ \tau(v_j).
\end{split}
\end{equation}
Since $(a_j, v_j)\to (a_0, v_0)$ and $\eta_j\to \eta_0$ on every level of ${\mathcal E}$ and since the map ${\bf f}(a, v, \eta)$ is sc-continuous, we conclude that ${\bf f}(a_j, v_j, \eta_j)\to {\bf f}(a_0, v_0, \eta_0)$ on every level of ${\mathcal F}$.
We recall  that the  space  ${\mathcal E}^{\#}$  is a  closed subspace of ${\mathcal E}$ consisting of all $\eta$ vanishing on  ${\bf D}(-1)$. Similarly, we define 
the space ${\mathcal F}^{\#}$ consisting of all elements $h\in {\mathcal F}$ which vanish on the set ${\bf D}(-1)$.
The sequence of  linear  operators
\begin{gather*}
L_j:{\mathcal E}^{\#}\rightarrow {\mathcal F}^{\#},
\intertext{defined by}
h\mapsto  (\nabla_{(T\exp_{u}(\eta_j)h, u)}\Theta)[\exp_{u}(\eta_j),u] \circ\Gamma(\exp_{u}(\eta_j),u)\circ {\bf f}(a_j,v_j,\eta_j),
\end{gather*}
consists of $\ssc^+$-operators. The sequence  converges on every level,  as operators from 
${\mathcal E}^{\#}$ to  $\bigl( {\mathcal F}^{\#} \bigr)^1$,  to the $\ssc^+$-operator
$$
h\mapsto   (\nabla_{(T\exp_{u}(\eta_0)h,0)}\Theta)[\exp_{u}(\eta_0),u] \circ\Gamma(\exp_{u}(\eta_0),u)\circ {\bf f}(a_0,v_0,\eta_0).
$$
The boundedness of  the sequence $(h_j)$  in ${\mathcal E}_m$ implies  that the sequence $(L_j(h_j))$ is bounded on  the level $m+1$ in ${\mathcal F}^{\#}$. Hence,  introducing  $z^\ast_j=z_j-L_j(h_j)$ and using  \eqref{convergence_0} and \eqref{convergence_1},  we see  that 
\begin{equation*}
\Theta(\exp_{u}(\eta_j), u) \circ (\nabla_{h_j}\ov{\partial}_{J,j(v_j)})(\exp_{u}(\eta_j))\circ \tau(v_j) = y_j + z_j^\ast
\end{equation*}
where $y_j\rightarrow 0$ in $H^{m+3}$ on $S\setminus {\bf D}(-1)$ and  where $z_j^\ast$ is a bounded sequence
in $H^{m+4}$  on  $S\setminus {\bf D}(-1)$.
We write the above identity as
\begin{equation}\label{convergence_2}
(\nabla_{h_j}\ov{\partial}_{J,j(v_j)})(\exp_{u}(\eta_j)) = \Gamma(\exp_{u}(\eta_j),u)(y_j + z_j^\ast)\circ\delta (v_j)
\end{equation}
where $\Gamma(\exp_{u}(\eta_j),u)=\Theta(\exp_{u}(\eta_j),u)^{-1}$ and $\delta (v)=\tau (v)^{-1}$.

Now we focus  on  a neighborhood of the point $z_0$ in $S\setminus {\bf D}(-1)$.
We  fix  an open disk-like  neighborhood ${\mathcal D}$ around $z_0$ whose  boundary is smooth and whose closure does not intersect ${\bf D}(-1)$.   
We may assume that the disk ${\mathcal D}$ is so small that  the images $\exp_{u}(\eta_j)({\mathcal D})$ for $j$ large  are contained in the image of a chart $\psi:U(u(z_0))\subset Q\rightarrow {\mathbb R}^{2n}$ around $u(z_0)$ mapping $u(z_0)$ to $0$.  In abuse of the notation, we denote by $h_j(z)$ the map $T\psi (\exp_{u}(\eta_j)(z))h_j(z)$. Similarly,  by $y_j(z):T_zS\to \R^{2n}$ and $z^*_j:T_zS\to \R^{2n}$  we denote  the maps $T\psi \circ \exp_{u}(\eta_j)(z)\Gamma(\exp_{u}(\eta_j),u)y_j \circ \delta (v_j) (z)$ and $T\psi (\exp_{u}(\eta_j)(z)\Gamma(\exp_{u}(\eta_j),u)z^*_j \circ \delta (v_j) (z)$. 
Then,  in view of the formula \eqref{covariant2_b} in the Appendix \ref{connection} , the equation \eqref{convergence_2}  can be written   in our local coordinates on $Q$   as  
\begin{equation}\label{convergence_2b}
Dh_j (z)+J_j(z)\circ Dh_j(z)\circ j(v_j)+A_j(z)(h_j(z), \cdot )=y_j +z_j^*.
\end{equation}
We choose  a smooth family
$v\mapsto \varphi_v$ of biholomorphic maps $\varphi_v:(D,i)\rightarrow (\ov{\mathcal D},j(v))$ defined on the standard closed disk $D$ and  mapping $0$ to $z_0$. The family $\varphi_v$ is defined for  $v\in V$ close to $0$. 
Taking the composition $e_j=h_j\circ \varphi_{v_j}:D\to \R^{2n}$ and evaluating  both sides of \eqref{convergence_2b}  at the vectors $T\varphi_{v_j}(\partial_s)$,  we obtain
the identity
$$
\partial_s e_j +J_{j}(\varphi_{v_j}) \partial_t e_j +A_j(\varphi_{v_j})(e_j, T\varphi_{v_j}(\partial_s)) = y^{\ast\ast}_j + z^{\ast\ast}_j
$$
in which the sequences $(y_j^{\ast\ast})$ and $(z_j^{\ast\ast})$ have the following properties. The sequence $(y^{\ast\ast}_j)$ converges to $0$ in $H^{m+2}(D)$ and the sequence 
$(z^{\ast\ast}_j)$ is bounded in $H^{m+3}(D)$. 
Abbreviating 
$$\wt{J}_j=J_{j}(\varphi_{v_j}) \quad \text{and}\quad \wt{A}_j=A_j(\varphi_{v_j})(\cdot,  T\varphi_{v_j}(\partial_s)),$$
the above identity takes the form 
$$
\partial_s e_j +\wt{J}_{j} \partial_t e_j +\wt{A}_je_j= y^{\ast\ast}_j + z^{\ast\ast}_j.
$$
The sequence  $(\wt{J}_j)$  is a sequence of smooth maps $D\rightarrow L({\mathbb R}^{2n})$ associating to a point $z\in D$ the  complex multiplication  $\wt{J}_j(z)$ on ${\mathbb R}^{2n}$. The  sequence  converges  in $C^\infty$ to the  smooth map $z\rightarrow \tilde{J}(z)$ where $\wt{J}(z)$ is  a complex multiplication  on ${\mathbb R}^{2n}$. The sequence $(\wt{A}_j)$  is a sequence of smooth maps $A_j:D\rightarrow L({\mathbb R}^{2n})$ which converges to the map $A:D\rightarrow L({\mathbb R}^{2n})$.  Moreover, since the sequence $(h_j)$ is bounded on level $m$,  the sequence $(e_j)$ is bounded in $H^{m+3}(D, \R^{2n})$.
The sequences $(e_j)$, $(y_j^{\ast\ast})$,  and $(z_j^{\ast\ast})$ meet  the  assumptions of the following lemma  proved as Proposition \ref{appx20} in the Appendix.
\begin{lemma}
We assume that  $(e_j)$ is a bounded sequence in $H^{m+3}(D,\R^{2n})$ satisfying 
$$
\partial_s e_j +\wt{J}_{j} \partial_t e_j +\wt{A}_je_j= y^{\ast\ast}_j + z^{\ast\ast}_j
$$
where 
$(y_j^{\ast\ast})$ is a sequence in $H^{m+2}(D,\R^{2n})$ converging to $0$, and  where $(z_j^{\ast\ast})$  is a bounded sequence in $H^{m+3}(D,\R^{2n})$. Then every $e_j$ decomposes as $e_j=e^1_j+e^2_j$ such  that $e^1_j\to 0$ in $H^{m+3}(D)$ and  such  that the sequence $(e^2_j)$ is bounded in $H^{4+m}(D_{\varepsilon})$, for every $\varepsilon\in (0, 1)$.
\end{lemma}
Here $D_{\varepsilon}\subset D$ denotes the  subdisk $D_{\varepsilon}=\{z\in D\vert \, \abs{z}\leq 1-\varepsilon\}$ for $0<\varepsilon<1$.
In view of  the above  considerations and applying  a finite covering argument  we obtain the following result. 
\begin{P4.17}[{\bf Classical Case}]
Let ${\mathcal K}$ be a finite-dimensional sc-complemented linear subspace of ${\mathcal E}$ and let $(a_0,v_0,k_0)\in O$  with $k_0\in {\mathcal K}$. We consider  a sequence ($a_j,v_j,k_j)\in  O$ in which $k_j\in {\mathcal K}$,  converging to $(a_0,v_0,k_0)$,   and a  bounded sequence  $(h_j)\in {\mathcal E}_m$  satisfying 
$$D_3{\bf f}(a_j,v_j,k_j)h_j =y_j +z_j$$
where  $(y_j)\subset {\mathcal F}_m$ converges to $0$ and where the sequence $(z_j)$  is bounded in ${\mathcal F}_{m+1}$. 

Then  there exists a subsequence of  $(h_j)$ having the property that the sequence  $(h_j \vert (S \setminus  {\bf D}({-1/2} ))$ of restrictions converges in ${\mathcal E}_m$. 
\end{P4.17}

Next we shall consider the nodal part of Proposition \ref{PROP4.19}.
\begin{P4.19}[{\bf Nodal Case}]\label{Nodal_case4.19} 
Let  $\{x,y\}\in D$  be a nodal pair and let $m\geq 0$  and $\varepsilon>0$. Then  
$$
\abs{D_3{\bf f}(a,v,\eta)h-D_3{\bf f}(a,v,0)h}_{{\mathcal F}_m}\leq \varepsilon \cdot\abs{h}_{{\mathcal E}_m}
$$
for all $h\in {\mathcal E}_m$ having their support in $D_x\cup D_y$ and for  $(a,v,\eta)\in O_m$  sufficiently close to $(0,0,0)$ on the level $m$.
\end{P4.19}
\begin{proof}
We work in local coordinates near  the nodal pair $\{x, y\}$.  We start with the defining equation \eqref{defined10}
and focus on the subset  $D_x\cup D_y$ of $S$. On this set the complex structure $j$ does not depend on $v$. Hence
\begin{equation}\label{defined100}
\begin{gathered}
\Gamma(\exp_{u_a}(\oplus_a(\eta)),u_a)\circ\wh{\oplus}_a(\xi)= \ov{\partial}_{J,j} (\exp_{u_a}(\oplus_a(\eta)))\\
\wh{\ominus}_a(\xi) \cdot {\frac{\partial}{\partial s}}=\ov{\partial}_0 (\ominus_a(\eta)).
\end{gathered}
\end{equation}

Here $j$ denotes the standard structure on $Z_a$ for the various $a$, i.e. $j=j(a)$. We are going to use the definitions of the gluing
constructions and  pass to local coordinates near the nodal pair $\{x, y\}$.  We take  the holomorphic polar coordinates $\sigma^+:\R^+\times S^1\to D_x$ and $\sigma^-:\R^+\times S^1\to D_y$   {belonging to the small disk structure  and satisfying}  $\sigma^+(s, t)\to x$ and $\sigma^-(s',t')\to y$ as $s\to \infty$ and $s'\to -\infty$. Moreover, we take a chart $\psi:U\subset Q\to \R^{2n}$ around the image of the nodes $u(x)=u(y)\in Q$ satisfying $\psi (u(x))=\psi (u(y))=0$. The map $u:S\to Q$ is represented  on the set $D_x\cup D_y$ by the pair $(u^+, u^-)$  of  smooth maps $u^\pm(s, t)=\psi \circ u\circ \sigma^\pm (s, t):\R^\pm\times S^1\to \R^{2n}$. The vector field $\eta$ along the map $u$ is represented by the pair $h=(h^+, h^-)$ of  vector fields $h^\pm(s, t)=T\psi (u\circ \sigma^\pm (s, t))\circ \eta \circ \sigma^\pm (s, t)$ along the curves $u^\pm$ in $\R^{2n}$. 
The section 
$$(a, v,\eta)\mapsto {\bf f}(a, v, \eta)$$
on $D_x\cup D_y$ is represented in local coordinates  by the pair of vectors  
$$\wh{\xi}=(\xi^+, \xi^-)=({\bf f}^+(a, v, h), {\bf f}^-(a, v,h))$$  defined by 
\begin{equation*}
\begin{aligned}
{\bf f}^+(a, v,h)(s,t)&=T\psi (u\circ \sigma^+(s, t))\circ {\bf f}(a, v,\eta)(\sigma^+ (s, t))\cdot [\partial_s\sigma^+(s, t)]\\
{\bf f}^-(a, v, h)(s',t')&=T\psi (u\circ \sigma^-(s', t'))\circ {\bf f}(a, v,\eta)(\sigma^- (s', t'))\cdot [\partial_s\sigma^-(s', t')]
\end{aligned}
\end{equation*}
for $(s, t)\in \R^+\times S^1$ and $(s',t')\in \R^-\times S^1$.  Abusing the notation we shall write $u=(u^+, u^-)$ and we shall  use the following abbreviations for   the glued curves and  the glued vector fields
$$u_a=\oplus_a(u)=\oplus_a(u^+, u^-)\quad \text{and}\quad h_a=\wh{\oplus}_a(h)=\wh{\oplus}_a(h^+,h^-).$$
{We recall from Proposition \ref{exp-formula1}  the formula  
$$\oplus_a\exp_u(h)=\exp_{u_a}(h_a)=u_a+h_a$$
for  the  exponential map with respect to the Euclidean metric in $\R^{2n}$.}  Hence  the equation 
\eqref{defined100} on $D_x\cup D_y$ becomes 
\begin{equation}\label{nodalB}
\begin{aligned}
\wh{\Gamma}(u_a+h_a,u_a)\circ \wh{\oplus}_a(\wh{\xi})&=\frac{1}{2}\bigl[  \partial_s (u_a+h_a)+ \wh{J}(u_a+h_a)
\partial_t (u_a+h_a)\bigr]\\
 \wh{\ominus}_a(\wh{\xi})& = \bar{\partial}_0 \bigl( \ominus_a(h)\bigr).
 \end{aligned}
 \end{equation}
 We introduce two maps $A, B:\R^{2n}\oplus \R^{2n}\to {\mathcal L}(\R^{2n},\R^{2n})$ by 
 $$A(p, q)=\frac{1}{2}\wh{\Gamma}(p, q)^{-1}\quad \text{and}\quad B(p, q)=\frac{1}{2}\wh{\Gamma}(p, q)^{-1}\wh{J}(p).$$
Applying the inverse of $\wh{\Gamma}$ to both sides of \eqref{nodalB} we obtain 
\begin{equation}\label{nodalBeq2}
\begin{aligned}
\wh{\oplus}_a(\wh{\xi})&=A(u_a+h_a,u_a) \partial_s (u_a+h_a)+B(u_a+h_a,u_a)\partial_t (u_a+h_a)\\
 \wh{\ominus}_a^0(\wh{\xi})& = \bar{\partial}_0 \ominus^0_a(h).
 \end{aligned}
 \end{equation}
We decompose the vector field $\wh{\xi}$  into the sum of two vector fields  
$$\wh{\xi}=(\xi^+,\xi^-)=\xi_1+\xi_2$$
 where  the pair $\xi_1=(\xi^+_1, \xi^-_1)$ satisfies $\wh{\ominus}_a (\xi_1)=0$ and the pair 
$\xi_2=(\xi^+_2,\xi^-_2)$ satisfies $\wh{\oplus}_a(\xi_2)=0$.  In view of\,  $\wh{\oplus}_a(\wh{\xi})=\wh{\oplus}_a(\xi_1)$ and $\wh{\ominus}_a(\xi_1)=0$ and \eqref{nodalBeq2},  the pair $\xi_1$ solves the equations 
\begin{equation*}
\begin{aligned}
\wh{\oplus}_a(\xi_1)&=A(u_a+h_a,u_a) \partial_s (u_a+h_a)+B(u_a+h_a,u_a)\partial_t (u_a+h_a)\\
\wh{\ominus}_a(\xi_1)& = 0.
 \end{aligned}
 \end{equation*}
Introducing the matrices 
 \begin{equation*}
\wh{A}(p, q)=\begin{bmatrix}A(p, q)&0\\
0&\id
\end{bmatrix}
\quad \text{and}\quad 
\wh{B}(p, q)=\begin{bmatrix}B(p, q)&0\\
0&\id
\end{bmatrix},
\end{equation*}
and recalling the total hat gluing map $\wh{\boxdot}_a=(\wh{\oplus}_a,\wh{\ominus}_a)$ and the maps $D_s^a$, $D^a_s$ from Section \ref{preliminaries}, the above equations can be written as 
\begin{equation*}
\begin{split}
\wh{\boxdot}_a(\xi_1)&=\bigl(A(u_a+h_a,u_a) \partial_s (u_a+h_a), 0\bigr)+\bigl(B(u_a+h_a,u_a) \partial_t (u_a+h_a), 0\bigr)\\
&=\wh{A}(u_a+h_a,u_a)\begin{bmatrix} \partial_s (u_a+h_a)\\0\end{bmatrix}+
\wh{B}(u_a+h_a,u_a)\begin{bmatrix} \partial_t (u_a+h_a)\\0\end{bmatrix}\\
&=\wh{A}(u_a+h_a,u_a)\cdot \wh{\boxdot}_a\circ D^a_s(u+h)+\wh{B}(u_a+h_a,u_a)\cdot \wh{\boxdot}_a\circ D^a_t(u+h).
\end{split}
\end{equation*}
Abbreviating 
$$L(a, v, h):=(\wh{\boxdot}_a)^{-1}\wh{A}(u_a+h_a,u_a)\cdot \wh{\boxdot}_a\quad \text{and}\quad  M(a, v, h):=(\wh{\boxdot}_a)^{-1}\wh{B}(u_a+h_a)\cdot \wh{\boxdot}_a,$$
the solution $\xi_1$ is equal to 
$$\xi_1=L(a, v, h)\circ D^a_s(u+h)+M(a, v, h)\circ D^a_t(u+h).$$
We already know that,  for fixed $(a,v)$,  the map
$$
h\mapsto \xi_1 (a, v, h)
$$
is of class $C^1$. Its derivative with respect to $h$ in the direction $k$ is given by
\begin{equation*}
\begin{split}
D_3\xi_1(a, v,h)k&=\bigl[ \bigl( D_3L(a, v, h)k\bigr)\circ D^a_s(u+h) +L(a, v, h)\circ D^a_s(k)\bigr] \\
&\phantom{=}+ \bigl[ \bigl( D_3M(a, v, h)k\bigr)\circ D^a_t(u+h) +M(a, v, h)\circ D^a_t(k)\bigr]\\
&=I (a, v, h, k)+II (a, v, h, k).
\end{split}
\end{equation*}
We claim that given $\varepsilon>0$,
$$\abs{D_3\xi_1(a, v, h)k-D_3\xi_1(a, v, 0)k}_{F_m}\leq \varepsilon\abs{k}_{E_m}$$
for all $k\in E_m$ and for $h\in E_m$ sufficiently close to $0$.
In order to prove this estimate  we shall show that 
\begin{equation*}
\begin{aligned}
\abs{I (a, v, h, k)-I(a, v, 0, k)}_{F_m}&\leq  \varepsilon\abs{k}_{E_m}\\
\abs{II (a, v, h, k)-II(a, v, 0, k)}_{F_m}&\leq  \varepsilon\abs{k}_{E_m}
\end{aligned}
\end{equation*}
for all $k\in E_m$ if $h\in E_m$ is sufficiently close to $0$.
We only prove the first estimate involving $I(a, v, h, k)$ since the same arguments apply to the term $II (a, v, h, k)$. 
The difference $I (a, v, h, k)-I(a, v, 0, k)$ consists of two summands,
\begin{equation*}
\begin{split}
&I (a, v, h, k)-I(a, v, 0, k)\\
&=\bigl[ \bigl(D_3L(a, v, h)k\bigr)\circ D^a_s(u+h) -\bigl(D_3L(a, v, 0)k\bigr)\circ D^a_s(u)\bigr]\\
&\phantom{=}+\bigl[ L(a, v, h)\circ D^a_s(k)-  L(a, v, 0)\circ D^a_s(k)\bigr]
\end{split}
\end{equation*}
and we estimate each of the terms.

{We take $\delta>0$. It plays the role of $\varepsilon$ in the statement (3) of Proposition \ref{XCX2}. } Then, using Proposition \ref{XCX1} and  statements (2) and (3) of Proposition \ref{XCX2},  the first summand of the difference  $I (a, v, h, k)-I(a, v, 0, k)$ can be estimated as 
\begin{equation}\label{identity_26}
\begin{split}
&\abs{\bigl( D_3L(a, v, h)k\bigr)\circ D^a_s(u+h) -\bigl( D_3L(a, v, 0)k\bigr)\circ D^a_s(u)}_{F_m}\\
&\phantom{==}\leq \abs{\bigl[  \bigl( D_3L(a, v, h) -D_3L(a, v, 0)\bigr) k\bigr] \circ D^a_s(u+h)  }_{F_m}\\
&\phantom{===} +\abs{\bigl( D_3L(a, v, 0)k \bigr )\circ D^a_s(h) }_{F_m}\\
&\phantom{==}\leq C_m \abs{k}_{E_m}\cdot \bigl[ \delta\abs{u}_{E_m}+\delta \abs{h}_{E_m}+\abs{h}_{E_m}\bigr]
\end{split}
\end{equation}
for all $k\in {\mathcal E}_m$ and  $h$ sufficiently small  {with a constant $C_m$ independent of $(a, v)$.   That the estimate of the second term is independent of $a$ follows from the sc-smoothness of the map $(a, v, h, k,\eta)\to [D_3L(a,v,h)k]\cdot \eta$ in  statement (2) of Proposition \ref{XCX2}. }

 The second term  of $I (a, v, h, k)-I(a, v, 0, k)$ is estimated by the Propositions \ref{XCX1} and \ref{XCX2} as
\begin{equation*}
\begin{split}
\abs{\bigl[ L(a, v,h)-L(a, v,0)\bigr]\circ D^a_s(k)}_{F_m}
&\leq \delta \cdot C_m\cdot \abs{k}_{E_m}
\end{split}
\end{equation*}
if $h$ is sufficiently close to $0$ in $E_m$.
Consequently, if $\varepsilon>0$ is given, we can choose $\delta$ sufficiently small  and $h$ sufficiently  close to $0$ in $E_m$ such that 
\begin{equation*}
\abs{I (a, v, h, k)-I(a, v, 0, k)}_{E_m}\leq \varepsilon \abs{k}_{E_m}
\end{equation*}
for all $k\in E_m$. 
The estimate for the term $II$ is  the same.   Together the estimates imply that, given $\varepsilon>0$, then 
\begin{equation}\label{identity_28}
\abs{D_3\xi_1(a, v, h)k-D_3\xi_1(a, v, 0)k}_{F_m}\leq \varepsilon \abs{k}_{E_m}
\end{equation}
for all $k\in E_m$ and for  $(a, v, h)\in O_m$ sufficiently close to $(0, 0, 0)$ on the level $m$.

We next consider the map $h\mapsto \xi_2(a, v, h)$. Recall that $\xi_2$ is a solution of the equations 
\begin{equation}\label{mynodal3}
\begin{aligned}
\wh{\oplus}_a(\xi_2)&=0\\
\wh{\ominus}_a(\xi_2)&= \frac{1}{2}\bigl( \partial_s \ominus_a(h) +J_0\partial_t\ominus_a(h)\bigr).
\end{aligned}
\end{equation}
Linearizing the map $h\mapsto \xi_2(a, v, h)$ at $h$ in the direction of $k$ we find 
\begin{equation*}
\begin{aligned}
\wh{\oplus}_a(D_3\xi_2(a, v, h)k)&=0\\
\wh{\ominus}_a(D_3\xi_2(a, v, h)k)&= \frac{1}{2}\bigl( \partial_s \ominus_a(k) +J_0\partial_t\ominus_a(k)\bigr).
\end{aligned}
\end{equation*}
{Observing that $D_3\xi_2(a, v, 0)k$ is also a solution of the two equations, we conclude from Theorem \ref{sc-splicing-thm1} that} 
$$
D_3\xi_2(a, v, h)=D_3\xi_2(a, v, 0).$$
In view of \eqref{identity_28} and $D_3{\bf f}(a, v, h)=D_3\xi_1(a, v, h)+D_3\xi_2(a, v, h)$, we finally obtain,  for given $\varepsilon>0$, the estimate
$$
\abs{ D_3{\bf f}(a, v, h)k-D{\bf f}(a, v,0)k}_{F_m}\leq \varepsilon \abs{k}_{E_m}
$$
for all $k\in E_m$ and for $(a, v, h)$ close to $(a, v, 0)$. This proves the desired estimate in 
Proposition \ref{PROP4.19} (Nodal Case).
\end{proof}
Finally, we shall prove the nodal case of Proposition \ref{PROPX}.

\begin{P4.17}[{\bf Nodal Case}]\label{propa_nodal}
Let ${\mathcal K}$ be a finite-dimensional sc-comple\-mented linear subspace in ${\mathcal E}$ and let  $(a_0,v_0,\eta_0)\in O$ in which $\eta_0\in {\mathcal K}$.
Fix the level $m\geq 0$ and consider a sequence $(a_j,v_j,\eta_j)$ in $O$ such that  $\eta_j\in {\mathcal K}$ and $(a_j,v_j,\eta_j)\to (a_0, v_0, \eta_0)$. 
Then the following holds.  If  $(\wh{\eta}_j)$  is a bounded sequence in  ${\mathcal E}_m$ satisfying  
$$D_3{\bf f}(a_j,v_j,\eta_j)\wh{\eta}_j=y_j+z_j,$$
where  $y_j\to 0$ in ${\mathcal F}_m$ and $(z_j)$ is bounded in ${\mathcal F}_{m+1}$,  then  there exists a subsequence of $(\wh{\eta}_j)$
such that   the sequence $\bigl(\wh{\eta}_j\vert {{\bf D}(-\frac{1}{2})   }\bigr)$ of restrictions converges in ${\mathcal E}_m$.
\end{P4.17}
\begin{proof}  It suffices to study the problem near the nodal pairs. We fix the  nodal pair $\{x, y\}\in D$ and work in local coordinates on  $D_x\cup D_y$  where $D_x$ and $D_y$ are disks of the small disk structure.
As in the previous proposition we choose the holomorphic polar coordinates $\sigma^+:\R^+\times S^1\to D_x$ and $\sigma^-:\R^+\times S^1\to D_y$  satisfying $\sigma^+(s, t)\to x$ and $\sigma^-(s',t')\to y$ as $s\to \infty$ and $s'\to -\infty$, and  we take a chart $\psi:U\subset Q\to \R^{2n}$ around the image of the nodes $u(x)=u(y)$ satisfying $\psi (u(x))=\psi (u(y))=0$. The smooth map $u:S\to Q$ is represented  on the set $D_x\cup D_y$ by the pair $(u^+, u^-)$  of  smooth maps $u^\pm(s, t)=\psi \circ u\circ \sigma^\pm (s, t):\R^\pm\times S^1\to \R^{2n}$ and the vector field $\eta$ along the map $u$ is represented by the pair $h=(h^+, h^-)$ of  vector fields $h^\pm(s, t)=T\psi (u\circ \sigma^\pm (s, t))\circ \eta \circ \sigma^\pm (s, t)$  along the maps $u^\pm$. The section 
$$(a, v,\eta)\mapsto {\bf f}(a, v, \eta)$$
on $D_x\cup D_y$ is represented in local coordinates  by the pair of vectors  
$$\wh{{\bf f}}(a, v, h)=({\bf f}^+(a, v, h), {\bf f}^-(a, v,h))$$
 defined by 
\begin{equation*}
\begin{aligned}
{\bf f}^+(a, v,h)(s,t)&=T\psi (u\circ \sigma^+(s, t))\circ {\bf f}(a, v,\eta)(\sigma^+ (s, t))\cdot [\partial_s\sigma^+(s, t)]\\
{\bf f}^-(a, v, h)(s',t')&=T\psi (u\circ \sigma^-(s', t'))\circ {\bf f}(a, v,\eta)(\sigma^- (s', t'))\cdot [\partial_s\sigma^-(s', t')]
\end{aligned}
\end{equation*}
for $(s, t)\in \R^+\times S^1$ and $(s',t')\in \R^-\times S^1$. Since $j$ does not depend on the parameter $v$  on $D_x\cup D_y$, the section ${\bf f}(a, v,\eta)$ and hence the maps ${\bf f}^\pm(a,v,\eta)$ are independent of 
the parameter $v$ on $D_x\cup D_y$. 
Abusing the notation we write $u=(u^+, u^-)$ and 
abbreviating the  glued map and  the glued vector fields by 
$$u_a=\oplus_a(u)=\oplus_a(u^+, u^-)\quad \text{and}\quad h_a=\wh{\oplus}_a(h)=\wh{\oplus}_a(h^+,h^-), $$
and introducing $\xi^\pm= {\bf f}^\pm (a,v, h)$,   the vector  field  $\wh{\xi}=(\xi^+,\xi^-)$   is a solution of the  equations 
\begin{equation}\label{nodalBeq1}
\begin{aligned}
\wh{\oplus}_a(\wh{\xi})&=\Theta(u_a+h_a,u_a)\bigl[  \partial_s (u_a+h_a)+ \wh{J}(u_a+h_a)
\partial_t (u_a+h_a)\bigr]\\
 \wh{\ominus}_a(\wh{\xi})& = \bar{\partial}_0 \bigl( \ominus_a(h)).
 \end{aligned}
 \end{equation}
\noindent By  our assumptions, we  are given a  sequence $(a_j,v_j,\eta_j)$ in $O$ converging to  $(a_0,v_0,\eta_0)$ and such that $\eta_j$ belongs to a finite-dimensional sc-complemented  vector space ${\mathcal K}$ of ${\mathcal E}$. 
In addition, there is a bounded sequence  $(\wh{\eta}_j)$ in ${\mathcal E}_m$  satisfying 
$$
D_3{\bf f}(a_j,v_j,\eta_j)\wh{\eta}_j=y_j+z_j,
$$ 
where 
$y_j\rightarrow 0$ in ${\mathcal F}_m$ and  the sequence $(z_j)$ is bounded in ${\mathcal F}_{m+1}$. 
In  our  local coordinates on $D_x\cup D_y$, the vector fields  $\eta_j$ and $\eta_0$ are represented by pairs of smooth maps  $h_j=(h^+_j, h^-_j)$ and $h_0=(h^+_0,h^-_0)$ belonging to a finite-dimensional sc-complemented space ${\mathcal K}$ of the sc-Hilbert space $F$. We observe that $h_j\to h_0$ in $E_k$ for every level $k\geq 0$. 
The elements  $y_j$ and $z_j$ are in our local coordinates and in abuse of the notation  represented by the pairs ${y}_j=(y^+_j, y^-_j)\in F_{m}$ and 
${z}_j=(z^+_j, z^-_j)\in F_{m+1}.$
Then,  denoting by $k_j\in E_m$ the  local representatives of $\wh{\eta}_j$  on $D_x\cup D_y$,  we have 
\begin{equation}\label{bubble}
D_3\wh{{\bf f} }(a_j,v_j,h_j)k_j=y_j+z_j.
\end{equation}
By assumption,  the sequence $(k_j)$ is  bounded in $E_m$, the sequence $(y_j)$ converges to  $0$ in $F_m$  and $(z_j)$ is bounded in $F_{m+1}$.

\noindent We linearize  the equations \eqref{nodalBeq1}  with respect to the variable  $h$ and obtain
\begin{equation}\label{identity1}
\begin{split}
&\wh{\oplus}_{a_j} (D_3\wh{{\bf f}}(a_j, v_j, h_j)k_j)\\
&=  \bigl[ D_1\Theta(u_{a_j}+h_{a_j},u_{a_j}) k_{a_j}\bigr]\cdot
\bigl[  \partial_s  (u_{a_j}+h_{a_j}) +\wh{J}(u_{a_j}+h_{a_j}) \partial_t  (u_{a_j}+h_{a_j} )\bigr]\\
&\phantom{=} + \Theta(u_{a_j}+h_{a_j},u_{a_j})\cdot\bigl[ D\wh{J}(u_{a_j}+h_{a_j})\cdot k_{a_j}\bigr]\cdot  \partial_t (u_{a_j}+h_{a_j})\bigr]\\
&\phantom{=} + \Theta(u_{a_j}+h_{a_j},u_{a_j})\cdot
\bigl[  \partial_s  k_{a_j}+\wh{J}(u_{a_j}+h_{a_j}) \partial_t  k_{a_j}\bigr]\\
\end{split}
\end{equation}
and
\begin{equation}\label{identity1_2}
\wh{\ominus}_{a_j}\bigl( D_3\wh{{\bf f}}(a_j, v_j, h_{a_j}) k_{j}\bigr)=\ov{\partial}_0\bigl(\ominus_{a_j}({k}_{j})\bigr).
\end{equation}
In the equation \eqref{identity1}
we have used the abbreviation $k_{a_j}=\oplus_{a_j}(k_j)$. Abbreviating by $A_j, B_j$ and $C_j$ the three terms on the right-hand side of \eqref{identity1},  we have
\begin{equation}\label{identity1_3}
\wh{\oplus}_{a_j} (D_3\wh{{\bf f}}(a_j, v_j, h_{j}) k_j)=A_j+B_j+C_j.
\end{equation}
Recalling the identity  \eqref{nodalBeq1},
\begin{equation*}
\wh{\Gamma} (u_{a_j}+h_{a_j}, u_{a_j})\cdot\bigl[\wh{ \oplus}_{a_j}\wh{{\bf f}}(a_j, v_{j}, h_j)\bigr] =\frac{1}{2}\bigl[   \partial_s  (u_{a_j}+h_{a_j}) +\wh{J}(u_{a_j}+h_{a_j}) \partial_t  (u_{a_j}+h_{a_j})\bigr],
\end{equation*}
the term $A_j$ can be written as 
\begin{equation}\label{identity1_4}
A_j= \bigl[ D_1\Theta(u_{a_j}+h_{a_j},u_{a_j}) k_{a_j}\bigr] \cdot \wh{\Gamma}(u_{a_j}+h_{a_j}, u_{a_j})\cdot \bigl[ \wh{ \oplus}_{a_j}{\bf f}(a_j, v_{j}, h_j)\bigr].
 \end{equation}
By our assumption, 
\begin{equation}\label{identity_7}
D_3\wh{{\bf f}}(a_j,v_j,  h_j) k_j=y_j+z_j, 
\end{equation}
where $(y_j)$ is a sequence converging to $0$ in $F_m$ and $(z_j)$ is a bounded sequence in $F_{m+1}$.
We decompose $y_j$ and $z_{j}$ as  
$$y_j=y_{1,j}+ y_{2,j}\quad \text{and}\quad z_j=z_{1,j}+z_{2,j}$$ according to the splitting $ \ker \wh{\ominus}_{a_j}\oplus  \ker\wh{ \oplus}_{a_j}$.
Since $\wh{\oplus}_{a_j}(y_j)=\wh{\oplus}_{a_j}(y_{1,j})$ and 
$\wh{\ominus}_{a_j}(y_{1,j})=0$, we observe, using  Proposition \ref{HAT_HAT} and  the definition of the $\wh{G}^a_m$-norm,  that 
\begin{equation*}
\begin{split}
C_m\abs{y_j }_{F_m}&
\geq \abs{\bigl(  \wh{\oplus}_{a_j}(y_j), \wh{\ominus}_{a_j }( y_j)   \bigr)}_{\wh{G}_m^{a_j} }
\geq \abs{\bigl( \wh{\oplus}_{a_j}(y_j), 0\bigr) }_{\wh{G}_m^{a_j} } \geq \frac{1}{C_m}\abs{y_{1,j}}_{F_m}.
\end{split}
\end{equation*}
We conclude, in particular, that the sequence $(y_{1,j})$ converges to $0$ in $F_{m}$. Since 
$y_{2,j}=y_{j}-y_{1,j}$, also the sequence $(y_{2,j})$ converges to $0$ in $F_m$.
Similar arguments applied to the sequence $(z_{1,j})$ show that the sequences $(z_{1,j})$ and $({z}_{2,j})$ are bounded in $F_{m+1}.$

\noindent Since $\wh{\oplus}_{a_j}(y_{j}+z_j)=\wh{\oplus}_{a_j}(y_{1,j}+z_{1,j} )$ and  $\wh{\ominus}_{a_j}(y_{1,j}+z_{1,j})=0$,  the maps $y_{1,j}+z_{1,j}$ are, in view of  \eqref{identity1}, the  solutions of the equations  
\begin{equation}
\begin{aligned}
\wh{\oplus}_{a_j}(y_{1,j}+z_{1,j})&=A_j+B_j+C_j\\
\wh{\ominus}_{a_j}(y_{1,j}+z_{1,j})&=0.
\end{aligned}
\end{equation}
Using the total hat-gluing map $\wh{\boxdot}_{a_j}=(\wh{\oplus}_{a_j},\wh{ \ominus}_{a_j})$, the above two equation can be written as 
\begin{equation}\label{identity1_5}
\wh{\boxdot}_{a_j}(y_{1,j}+z_{1,j})=\bigl(A_j+B_j+C_j, 0\bigr)
\end{equation}
Similarly, since $\wh{\oplus}_{a_j}(y_{2,j}+z_{2,j} )=0$ and  $\wh{\ominus}_{a_j}(y_{j}+z_j)=\wh{\oplus}_{a_j}(y_{2,j}+z_{2,j} )$, the maps $ y_{2,j}+z_{2,j}$ solve, in view of  \eqref{identity1_2}, the following two equations 
\begin{equation}\label{identity_8}
\begin{aligned}
\wh{\oplus}_{a_j}(y_{2,j}+z_{2,j})&=0\\
\wh{\ominus}_{a_j}(y_{2,j}+z_{2,j})&=\ov{\partial}_0(\ominus_{a_j}(k_j))
\end{aligned}
\end{equation}
which can be written,   using the total hat-gluing map $\wh{\boxdot}_{a_j}$,  as 
\begin{equation}\label{identity1_9}
\wh{\boxdot}_{a_j}(y_{2,j}+z_{2,j})=\bigl(0, \ov{\partial}_0(\ominus_{a_j}(k_j)) \bigr).
\end{equation}

\noindent By Proposition \ref{sc-splicing-thm1} the total hat gluing map $\wh{\boxdot}_{a_j}$ is an sc-isomorphism,  so that 
\eqref{identity1_5} implies 
\begin{equation}\label{identity1_6}
y_{1,j}+z_{1,j} =\bigl(\wh{\boxdot}_{a_j}\bigr)^{-1}(A_j,0)+\bigl(\wh{\boxdot}_{a_j}\bigr)^{-1}(B_j,0)+
\bigl(\wh{\boxdot}_{a_j}\bigr)^{-1}(C_j,0).
\end{equation}
We consider  each of the sequences on the right-hand side separately and start with 
$(\wh{\boxdot}_{a_j}\bigr)^{-1}(A_j,0)$.
Abbreviating 
$$A(u_{a_j}, h_{a_j}, k_{a_j})= \bigl[ D_1\Theta(u_{a_j}+h_{a_j},u_{a_j}) k_{a_j}\bigr] \cdot \wh{\Gamma}(u_{a_j}+h_{a_j}, u_{a_j}),$$
the maps $A_j$ can be written, in view of \eqref{identity1_4}, as 
$$
A_j= A(u_{a_j}, h_{a_j}, k_{a_j})\cdot \bigl[ \wh{ \oplus}_{a_j}{\wh{\bf f}}(a_j, v_{j}, h_j)\bigr].
$$
In view of our assumption, the sequence $(h_j)$ belongs to the finite dimensional sc-complemeted subspace ${\mathcal K}$ of $E$ and converges to $h_0\in {\mathcal K}$. Hence $(h_j)$ converges to $h_0$ in every space $E_k$, $k\geq 0$,  and  in particular, in the space $E_{m+1}$. Consequently, the sequence $\bigl(\wh{{\bf f}}(a_j,v_j, h_j) \bigr)$ converges to $\wh{{\bf f}}(a_0,v_0, h_0)$ in $F_{m+1}$.  We conclude, applying Proposition \ref{XCX2} , that  the sequence $\bigr( (\wh{\boxdot}_{a_j})^{-1}(A_j,0)\bigr) $ is bounded in $F_{m+1}$.

\noindent Next we consider the sequence $\bigr( (\wh{\boxdot}_{a_j})^{-1}(B_j,0)\bigr).$ 
We recall that  $B_j$ is defined by 
$$B_j= \Theta(u_{a_j}+h_{a_j},u_{a_j})\cdot\bigl[ D\wh{J}(u_{a_j}+h_{a_j})\cdot k_{a_j}\bigr]\cdot  \partial_t (u_{a_j}+h_{a_j})\bigr].$$
We rewrite the sequence  $(B_j,0)$ using the map  $B_{\frac{1}{2}}\oplus E\to F$,   introduced in Section \ref{preliminaries},  and defined by $(a, h)\mapsto D^a_t (h)$.  Introducing the matrix 
$$\wh{A}(u_{a_j}+h_{a_j}, k_{a_j})=\begin{bmatrix}\Theta (u_{a_j}+h_{a_j}, u_{a_j})\cdot\bigl[ D\wh{J}(u_{a_j}+h_{a_j})\cdot k_{a_j}\bigr]&0\\0&\id\end{bmatrix},$$
the sequence $(B_j, 0)$ can be written as 
\begin{equation*}
(B_j,0)=\wh{A}(u_{a_j}+h_{a_j}, k_{a_j})\circ \wh{\boxdot}_{a_j}\circ  D_t^{a_j}(u+h_j)
\end{equation*}
so that 
$$(\wh{\boxdot}_{a_j})^{-1}(B_j,0)=(\wh{\boxdot}_{a_j})^{-1}\circ \wh{A}(u_{a_j}+h_{a_j}, k_{a_j})\circ \wh{\boxdot}_{a_j}\circ D_t^{a_j}(u+h_{j}).
$$
By Proposition  \ref{XCX1},  the map $(a, h)\mapsto D^a_t (h)$ from $B_{\frac{1}{2}}\oplus E$ to $F$ is  sc-smooth. 
In view of the convergence $u+h_j\to u+h_0$ in $E_k$ for every level  $k$,  the sequence  $\bigl( D^{a_j}_t(u+h_j)\bigl)$ converges to $D^{a_0}_t(u+h_0)$ in $F_{m+1}$.  Applying Proposition \ref{HAT_HAT}, we conclude that the 
sequence $ (\wh{\boxdot}_{a_j})^{-1}(B_j,0)$ is bounded in $F_{m+1}$.
Summing up, we have shown that the sequences $(\wh{\boxdot}_{a_j})^{-1}(A_j,0)$  and $ (\wh{\boxdot}_{a_j})^{-1}(B_j,0)$  are bounded in $F_{m+1}.$ Thus,  abbreviating 
$$\wh{z}_{1,j}:=z_{1,j}-(\wh{\boxdot}_{a_j})^{-1}(A_j,0)-(\wh{\boxdot}_{a_j})^{-1}(B_j,0)$$
the sequence $\wh{z}_{1,j}$ is bounded in $F_{m+1}$ and,  in view of \eqref{identity1_6},
$$
y_{1,j}+\wh{z}_{1,j} =\bigl(\wh{\boxdot}_{a_j}\bigr)^{-1}(C_j,0).
$$
The above equation can be written as 
\begin{equation}\label{identity_11}
\begin{aligned}
\wh{\oplus}_{a_j}y_{1,j}+\wh{\oplus}_{a_j}\wh{z}_{1,j}&=C_j\\
\wh{\ominus}_{a_j}y_{1,j}+\wh{\ominus}_{a_j}\wh{z}_{1,j}&=0.
\end{aligned}
\end{equation}
We recall the equations \eqref{identity_8},
\begin{equation}\label{identity_12}
\begin{aligned}
\wh{\oplus}_{a_j}y_{2,j}+\wh{\oplus}_{a_j}z_{2,j}&=0\\
\wh{\ominus}_{a_j}y_{2,j}+\wh{\ominus}_{a_j}z_{2,j}&=\ov{\partial}_0(\ominus_{a_j}(k_j)).
\end{aligned}
\end{equation}
Abbreviating $z'_j=\wh{z}_{1,j}+z_{2,j}$ and recalling that $y_j=y_{1,j}+y_{2,j}$, we obtain after adding the corresponding rows  of \eqref{identity_11} and \eqref{identity_12},
\begin{equation}\label{identity_20}
\begin{aligned}
C_j&=\wh{\oplus}_{a_j}y_{j}+\wh{\oplus}_{a_j}z_{j}'\\
\ov{\partial}_0(\ominus_{a_j}(k_j))&=\wh{\ominus}_{a_j}y_{j}+\wh{\ominus}_{a_j}z_{j}'.
\end{aligned}
\end{equation}
Since 
$$C_j=\Theta(u_{a_j}+h_{a_j},u_{a_j})
\cdot
\bigl[  \partial_s  k_{a_j}+\wh{J}(u_{a_j}+h_{a_j}) \partial_t k_{a_j}\bigr]$$
and
$$\Theta(u_{a_j}+h_{a_j},u_{a_j})=\frac{1}{2}\wh{\Gamma} (u_{a_j}+h_{a_j},u_{a_j} )^{-1},$$
the first equation  in \eqref{identity_20} can be written as 
$$\wh{\Gamma} (u_{a_j}+h_{a_j},u_{a_j})\cdot \bigl( \wh{\oplus}_{a_j}y_{j}+\wh{\oplus}_{a_j}z'_{j}\bigr)=\frac{1}{2}\bigl[  \partial_s  k_{a_j}+\wh{J}(u_{a_j}+h_{a_j}) \partial_t  k_{a_j}\bigr].$$
The left-hand side can be written as 
$$\wh{\Gamma} (u_{a_j}+h_{a_j},u_{a_j})\cdot \bigl( \wh{\oplus}_{a_j}y_{j}+\wh{\oplus}_{a_j}z'_{j}\bigr)=\wh{\oplus}_{a_j}\zeta_j+\wh{\oplus}_{a_j}\zeta_j', $$
where the sequence $(\zeta_j)$  converges to $0$ in $F_m$,  and the sequence $(\zeta_j')$ is bounded in $F_{m+1}$. 
Consequently,  the equations \eqref{identity_20} become
\begin{equation*}
\begin{aligned}
\frac{1}{2}\bigl[  \partial_s \oplus_{a_j}( k_j)+\wh{J}(u_{a_j}+h_{a_j}) \partial_t  \oplus_{a_j} (k_j)\bigr]&=
\wh{\oplus}_{a_j}\zeta_j+\wh{\oplus}_{a_j}\zeta'_j\\
\ov{\partial}_0(\ominus_{a_j}(k_j))&=\wh{\ominus}_{a_j}y_{j}+\wh{\ominus}_{a_j}z_{j}.
\end{aligned}
\end{equation*}
Now we can apply Proposition \ref{ELLIPTIC-X} from  Appendix \ref{appx19} to  deduce that if $\chi^+$ and $\chi^-$ are cut-off functions  near the boundary, 
the sequence $((\chi^+ k_j^+,\chi^-k_j^-))$ has a convergent subsequence in 
$E_{m}$. The proof of Proposition \ref{PROPX}(nodal case) is complete.  
\end{proof}

Adding up the classical cases and the nodal cases of the present  section by means of a finite partition of unity on $S$,  {the proof of Proposition \ref{PROP4.19} and  the proof of the last statement in Proposition \ref{PROPX} follow.  }

  {It remains to prove the first two statements  of Proposition \ref{PROPX}.
 If we keep $(a,v)$ fixed,  the classical linearisation with respect to $\eta$ defines, as is well known,  a linear Cauchy-Riemann type
 operator on a domain with cylindrical ends. These operators have been used and studied by A. Floer in his seminal papers, see
 e.g. \cite{Floer1}. The weights are in the spectral gaps and the associated operator is Fredholm on every level
 in the classical sense and the Fredholm indices are independent of the smooth $(a,v,\eta)$.
 These operators are in particular sc-Fredholm. If we take the full sc-linearisation,  the real Fredholm index increases
 by the real dimension of the parameter space containing $(a,v)$.  This proves that the Fredholm index of $D{\bf f}(a,v,\eta)$
 is independent of the choice of $(a,v,\eta)$, where $\eta$ is assumed to be smooth.
  The second statement about the $C^1$-character of the map $\eta \to {\bf f}(a, v,\eta)$ for fixed parameter $(a, v)$ follows from the classical results in \cite{El}. In fact, with $(a,v)$ fixed, we are dealing with a nonlinear differential operator on a fixed domain. The proof of  Proposition \ref{PROP4.19} and Proposition \ref{PROPX}  are complete.}\hfill $\blacksquare$

%
%
%

\chapter{Appendices}
In the following we shall provide  the proofs  of the results used
earlier. The Appendices \ref{orientations-abstract} and \ref{orientations}  are devoted to the orientation of sc-Fredholm sections. 

\section{Proof of Theorem 2.56}\label{QWE}
The crucial point of the theorem is the sc-smoothness at the points where $a=0$. Hence 
we assume that the parameters $(a,v)$ are close
to $(0,v_0)$.  Since $\pav$ is a core smooth family of holomorphic embeddings, the implicit function theorem  implies that he preimages of the two boundary circles
of the finite cylinders $Z'_{b(a,v)}(-H)$ under the maps $\pav$ are smoothly varying curves in $Z_a(-h)$ with respect to the natural coordinates on $Z_a$.
Here the left curves vary smoothly for the $(s,t)$-coordinates and the right curves for the $(s',t')$-coordinates.

We find, assuming that $(a,v)$ is close to $(0,v_0)$,  a constant $h_0$ such  that these circles do not intersect $Z_a(-h-h_0)$.  Then we construct  a new family of diffeomorphisms $\psav:Z_a\to Z_{b(a, v)}'(-H)$  which is core-smooth and satisfying 
$$
\psav =\pav\quad \text{on\quad  $Z_a(-h-h_0-1)$}.
$$
In addition, there exists a germ of complex structure $j_{a,v}$, namely the pull-back $\psav^*i$ of the standard complex structure,  agreeing with  the standard complex structure $i$ on $Z_a(-h-h_0-1)$
such that 
$$
\psi_{a,v} : (Z_a,j_{(a,v)})\rightarrow (Z_{b(a,v)}(-H),i)
$$
are biholomorohic maps. 

Finally,  we note that instead of $Z_{b(a,v)}(-H)$ we can take $Z_{b'(a,v)}$
so that both are equipped with the standard complex structure and 
are biholomorphic by the identity map. 
The complex numbers $b(a, v)$ and $b'(a,v)$ has the same angular part and the moduli of $b(a, v)$ and $b'(a, v)$ are related by the formula
$$
\varphi(\abs{b(a,v)})-2H =\varphi(\abs{b'(a,v)}),
$$
where $\varphi$  is  the exponential  gluing profile. In view of Lemma 4.4 and Lemma 4.5 in \cite{HWZ8.7}, smoothness of $b(a, v)$ implies  smoothness of  $b'(a, v)$. Since the modification from $\pav$ to $\psav$ does not change the sc-smoothness properties in view of Theorem \ref{action-diff}, 
 it suffices to study  the map 
$$\Phi\colon \bd\oplus V\oplus E\to E,\quad (a, v, \xi)\mapsto \eta,$$
where $\xi$ is the unique solution 
of the  two equations 
\begin{equation}\label{sy_1a}
\oplus_{a}(\eta)= \oplus'_{b'(a,v)}(\xi)\circ \psav\quad \text{and}\quad \ominus_{a}(\eta)=0.
\end{equation}
By abuse of the notation we use the old notation $\pav$ for $\psav$ and $b(a, v)$ for $b'(a, v)$ so that the above system of equations is as follows 
\begin{equation}\label{sys_1a}
\oplus_{a}(\eta)= \oplus'_{b(a,v)}(\xi)\circ \pav\quad \text{and}\quad \ominus_{a}(\eta)=0.
\end{equation}
If $a=0$, we set  
$
\Phi (0, v, \xi)=(\xi^+\circ \phi^+_{0, v}, \xi^-\circ \phi^-_{0, v}).
$

Decomposing  $\xi^\pm$ as $\xi^\pm=u^\pm +c$,  where $c$ is a common asymptotic constant and 
$u^\pm \in H^{3+m, \delta_m}(\R^\pm \times S^1)$, and   proceeding as in Section 2.4, the component  $\eta^+$  of the solution $\eta$ of the  equation \eqref{sy_1a} is given by
\begin{equation}\label{sol_eq_1}
\begin{split}
\eta^+&= c+\biggl( 1-\dfrac{\bba}{\ga}\biggr)
\cdot \int_{S^1}\oplus'_{b(a, v)}(u) \bigl(\pav (R(a)/2, t)\bigr)\ dt\\
&\phantom{=\,\,  c}+\dfrac{\bba}{\ga} \cdot \bb (\pav) \cdot u^+(\pav)\\
&\phantom{=\,\, c}+ \dfrac{\bba}{\ga} \cdot \bigl(1-\bb (\pav)\bigr) \cdot u^-\bigl(\pav -(R(b,v), \vartheta (b,v)\bigr).\\
\end{split}
\end{equation}
  
There is an analogous  formula for the second components $\eta^-$ of $\eta$ in terms of variables $(s', t')$.  We omit the proof of sc-smoothness of the map $\eta^-$ since it  uses the same arguments as the proof of sc-smoothness   $\eta^+$. The formula \eqref{sol_eq_1} defines the following  four maps: \\[1ex]
\noindent {{\bf M1.}} The map 
$$\bd\oplus V\oplus H^{3,\delta_0}_c(\R \times  S^1)\to \R^n,\quad \xi\mapsto c$$
associating  with the  pair $\xi=(\xi^+, \xi^-)$ its asymptotic constant $c$.\\[0.5ex]
\noindent  {{\bf M2.}}  The map 
\begin{gather*}
\bd\oplus V\oplus H^{3,\delta_0}(\R^+\times  S^1)\oplus H^{3,\delta_0}(\R^-\times S^1) \to H^{3,\delta_0}(\R^+\times S^1),\\
(a, v, u)\mapsto\biggl( 1-\dfrac{\bba}{\ga}\biggr)
\cdot \int_{S^1}\oplus'_{b(a,v)}(u)(\pav (R(a)/2, t)) dt.
\end{gather*}
\noindent {{\bf M3.}} The map 
\begin{gather*}
\bd\oplus V\oplus H^{3,\delta_0}(\R^+\times  S^1)\to H^{3,\delta_0}(\R^+\times S^1),\\
(a, v, u^+)\mapsto \dfrac{\bba}{\ga} \cdot \bb (\pav) \cdot u^+(\pav).\\
\end{gather*}
\noindent  {{\bf M4.}}  The map 
\begin{gather*}
\bd\oplus V\oplus H^{3,\delta_0}(\R^-\times  S^1)\to H^{3,\delta_0}(\R^+\times S^1),\\
(a, v, u^-)\mapsto \dfrac{\bba}{\ga} \cdot \bigl(1-\bb (\pav)\bigr) \cdot u^-\bigl(\pav -(R(b,v), \vartheta (b,v)\bigr).
 \end{gather*}
 
  The map  defined in (M1),   associating  with $\xi$  its asymptotic constant,  does not depend on $(a, v)$ and,  since this map  is an sc-projection,  it is sc-smooth. 
The maps (M2)-(M4) are smooth in the classical sense at every point in which $a\neq 0$.  Hence we study sc-smoothness of these  maps in a neighborhood of the  point  with $(0, v_0)$.

In order to simplify our considerations further,  we recall the conclusions  of Theorem 1.46 from  \cite{HWZ8.7} adapted to our notation.  Given $\Delta>1$, there exists a map
$$
\wt{\phi}\colon \bd\oplus V\to \bigcap_{m\geq 3, 0<\varepsilon<2\pi} {\mathcal D}^{m, \varepsilon},
$$
defined in a neighborhood of $(0, v_0)\in \bd\oplus V$,  with the following properties. For every $m\geq 3$ and every $0<\varepsilon <2\pi$, the map 
$(a, v)\mapsto  \wt{\phi}_{a, v}\in {\mathcal D}^{m, \varepsilon}$  is  smooth and satisfies
\begin{equation}\label{wt_eq_1}
\wt{\phi}_{a,v}(s, t) =\pav (s, t)\quad \text{for $(s, t)\in [0, R(a)/2+\Delta ]\times S^1$.}
\end{equation}
The space $ {\mathcal D}^{m, \varepsilon}$ is   the space of diffeomorphisms $u:\R^+\times S^1\to \R^+\times S^1$ having the form
$u(s, t) = (s, t) +  d + r(s, t) $, where $d=(d', d'')\in \R^+\times S^1$ is a constant and the map $r$ has weak derivatives up to order $m$ which weighted by $e^{\varepsilon s}$ belong to
$L^2(\R^+\times S^1,\R^2).$
In particular,  the  map $\wt{\phi}(a,v)$ has the form  
\begin{equation}\label{wt_eq_2}
\wt{\phi}_{a, v}(s, t)=(s, t)+d(a, v)+r(a, v)(s, t)
\end{equation}
in which the map of constants $\bd\oplus V\to \R^2$, $(a, v)\mapsto d(a, v)$ and,  for every $m\geq 3$ and $0<\varepsilon<2\pi$,  the map $r\colon\bd \oplus V\to H^{m, \varepsilon}(\R^+\times S^1)$, $(a, v)\mapsto r(a, v)$, are smooth.

Fixing $\Delta>1$, we observe that, in view of 
\eqref{wt_eq_1},  we can replace the map $\pav$ by the map $\wt{\phi}_{a, v}$ in the definition of the map in (M2). We can do the same with the maps  in (M3) and (M4)  since $\beta_a(s)=0$ for $s\geq R(a)/2+1$. In the following, instead of writing $\wt{\phi}(a, v)$ we use the old notation $\pav$ and assume that $\pav$ has the form \eqref{wt_eq_2} with the maps $d(a, v)$ and $r(a, v)$ having properties as described above.
 
The crucial result used in the study of smoothness of the maps (M2)-(M4) is the following theorem.

\begin{theorem}\label{thm_m3_1}
Assume that $\phi_{a, v}$ is the family of diffeomorphisms 
$\pav\colon \R^+\times S^1\to \R^+\times S^1$, parametrized by $(a, v)\in \bd\oplus V$, having  the form 
$$\pav(s, t)=(s, t)+d(a, v)+r(a, v)(s, t),$$
where $(a, v)\mapsto d(a, v)\in \R^2$ is smooth  and, for every $m\geq 3$ and $0<\varepsilon<2\pi$, the map 
$r\colon \bd\oplus V\to H^{m,\varepsilon}(\R^+\times S^1)$ 
is smooth. Then the the composition
$\wt{\Phi}\colon \bd\oplus V\oplus H^{3,\delta_0}(\R^+\times S^1)\to  H^{3,\delta_0}(\R^+\times S^1)$,  defined by 
\begin{equation}\label{def_comp}
\wt{\Phi} (a, v,u)=u\circ \pav,
\end{equation}  
is well-defined and $\ssc$-smooth. 
\end{theorem}

Postponing the proof,  we first show that  each of the maps defining $\eta^+$ is sc-smooth. We begin with the map defined in (M3).

\noindent {\bf M3.}\,  Abbreviating  by $L$  the sc-Banach space $H^{3,\delta_0}(\R^+\times S^1)$, 
we   consider the map 
$\Phi_1\colon \bd \oplus V\oplus L\to L$, defined by
\begin{equation}\label{eq_m3_0}
\Phi_1 (a, v, u)=\dfrac{\bba}{\ga}\cdot \bb(\pav) \cdot u(\pav)
\end{equation} if $a\neq 0 $, and $\Phi_1 (0, v, u)=u(\phi_{0, v})$ if $a=0$. 

The map $\Phi_1$ is the   composition of the following maps,
\begin{equation*}
\bd\oplus V\oplus L\xrightarrow{\,\,  A_1\,\,  } \bd\oplus V\oplus L 
 \xrightarrow{\,\, A_2\,\,}\bd\oplus L \xrightarrow{\,\, A_3\,\,}\bd\oplus L \xrightarrow{\,\,  A_4\,\,  }L.
\end{equation*}
The map $A_1\colon \bd\oplus V\oplus L\to \bd\oplus V\oplus \bd\oplus L
$ is given  by 
$$A_1(a, v, h)=(a, v, b(a, v), h)$$
and is clearly smooth. The map 
$A_2\colon \bd\oplus V\oplus \bd\oplus L\to \bd\oplus V\oplus L$, defined by 
$$A_2(a, v,b,  h)\mapsto (a, v,\beta'_b\cdot  h),$$
is sc-smooth by Proposition 2.8 (a) in \cite{HWZ8.7}. The same proposition implies  sc-smoothness of the last map 
$A_4\colon  \bd\oplus L\to L$ given by 
$$A_4(a, h)= \dfrac{\bba}{\ga}\cdot h. $$
The map  $A_3\colon \bd\oplus V\oplus L\to \bd\oplus L$ is defined by 
$$A_3(a, v, h)=(a, \wt{\Phi} (a, v, h)), $$
where $\wt{\Phi}$ has been introduced in Theorem \ref{thm_m3_1}.  
Assuming that Theorem \ref{thm_m3_1} holds true, the map  $A_2$ is sc-smooth and hence, by the chain rule,  the map $\Phi_1$ is also sc smooth. \\[1ex]

\noindent {{\bf M4.}}\, Here we  abbreviate by $L'$  the sc-Banach space $H^{3,\delta_0}(\R\times S^1)$ and by $L^\pm$ the sc-Banach spaces $H^{3,\delta_0}(\R^\pm\times S^1)$.  
We consider  the map $\Phi_2:\bd \oplus V\oplus L^-\to L^+$, defined by
\begin{equation}\label{eq_M3_1}
\Phi_2 (a, v, u)=\dfrac{\bba}{\ga}\cdot \bigr(1-\bb(\pav) \bigr)\cdot u\bigl(\pav -(R(b,v), \vartheta (b,v))\bigr)
\end{equation} if $a\neq 0 $, and $\Phi_2 (0, v, u)=0$ if $a=0$. 
We choose a smooth map $\chi:\R\to [0,1]$ satisfying $\chi (s)=1$ for $s\leq 0$ and $\chi (s)=0$ for $s\geq -1/2$. By multiplying $u\in L^-$ by $\chi$ we may assume that $u$ belongs to $L'$.  Moreover, recalling  that  $\abs{R(a)-R(b(a, v))}\leq \Delta'$ for $(a, v)$ close to $(0, v_0)$ and choosing  a constant  $\Delta''$ satisfying $\Delta'+2<\Delta''$, we have 
$$\dfrac{\bba}{\ga}\cdot \bigl(1-\bb (\pav)\bigr)=\dfrac{\bba}{\ga}\cdot \bigl(1-\bb (\pav )\bigr)\cdot  \bb(\pav -\Delta'')$$
for all $(s, t)\in \R^+\times S^1$. Hence  we may assume that $\Phi_2$ is defined as 
\begin{equation}\label{eq_psi2_1}
\Phi_2(a, v, u)=\dfrac{\bba}{\ga}\cdot \bigl(1-\bb (\pav )\bigr)\cdot  \bb(\pav -\Delta'')\cdot 
 u\bigl(\pav -(R', \vartheta')\bigr)
 \end{equation}
 for all $a\neq 0$, where we have abbreviated $(R', \vartheta')=(R(b(a, v), \vartheta(b(a, v))$.

The map $\Phi_2$, as in the case of the map  $\Phi_1$,  can  be represented as a composition of sc-smooth maps,  
\begin{equation*}
\bd\oplus V\oplus L'\xrightarrow{\,\,  B_1\,\,  } \bd\oplus V\oplus\bd  \oplus L'\xrightarrow{\,\,  B_2\,\,  } \bd\oplus V\oplus L^+ \xrightarrow{\,\, B_3\,\,}\bd\oplus L^+ \xrightarrow{\,\,  B_4\,\,  }\bd\oplus L^+.
\end{equation*}
The map  $B_1$ is the same as the map $A_1$ above except that $L$ is replaced by $L'$, and hence it is sc-smooth. The next map 
$B_2\colon \bd\oplus V\oplus\bd  \oplus L'\to \bd\oplus V\oplus L^+$ is given by 
$$B_2(a, v, h)=(a, v, (1-\beta'_b)\cdot \beta_b'(\cdot -\Delta'')\cdot h(\cdot -(R(b), \vartheta (b)).$$
Its sc-smoothness, observing that the  function  $(1-\beta' )\cdot \beta' (\cdot -\Delta'')$ has a compact support, follows from Proposition 2.8 (b) in \cite{HWZ8.7}.  The map 
$B_3\colon \bd\oplus V\oplus L^+\to \bd\oplus V\oplus L^+$,   given 
by 
$$B_3(a, v, h)=(a, v, \wt{\Phi} (a, v, h)), $$
where $\wt{\Phi}(a, v, h)=h\circ \pav$,  is the same as the map $A_3$ above and  is sc-smooth by Theorem \ref{thm_m3_1}.  The last  map $B_4\colon 
\bd\oplus  L^+\to \oplus  L^+$, defined by 
$$B_4(b, h)=\dfrac{\bba}{\ga}\cdot h, $$ 
is sc-smooth by Proposition 2.8 (a) from \cite{HWZ8.7}. In conclusion,  we have proved that the map $\Phi_2$  is sc-smooth.

\noindent {{\bf M2.}}\,  In this part  $L^\pm$ stands for  the sc-Banach space $H^{3,\delta_0}(\R^\pm \times S^1)$.
We consider the map $\Phi_3\colon \bd \oplus V\oplus L^+\oplus L^-\to \R^n$, defined by
\begin{equation}\label{eq_M3_0}
\Phi_3 (a, v, u)= \biggl(1-\dfrac{\bba}{\ga}\biggr )\int_{S^1}\oplus'_{b(a, v)}(u)(\pav (R(a)/2, t )\ dt
\end{equation} if $a\neq 0 $, and $\Phi_3 (0, v, u)=0$ if $a=0$. .
Recalling that $u=(u^+, u^-)\in L^+\oplus L^-$, the map $\Phi_3$ is, in view of the definition of  the gluing map $\oplus'_b$,  the sum of  two maps,
\begin{equation*}
\begin{gathered}
\Phi_3'\colon \bd \oplus V\oplus L^+\to \R^n,\\
\Phi_3' (a, v, u^+)=\biggl(1-\dfrac{\bba}{\ga}\biggr )\int_{S^1}\bb \bigl(\pav (R/2, t)\bigr)u^+(\pav (R/2, t)\ dt,
\end{gathered}
\end{equation*}
and 
$$
\Phi_3''\colon \bd \oplus V\oplus L^-\to \R^n,
$$
\begin{equation*}
\begin{split}
&\Phi_3'' (a, v, u^-)\\
&\quad =\biggl(1-\dfrac{\bba}{\ga}\biggr ) \int_{S^1}\bigl(1-\beta'_b(\pav (R/2, t)\bigr)u^-\bigl(\pav (R/2, t)-(R', \vartheta')\bigr)\ dt
\end{split}
\end{equation*}
if $a\neq 0$ and $\Phi_3'(0, v, u)=\Phi_3''(0, v, u)=0$ if $a=0$. Here we abbreviated $b=b(a, v)$, $R=R(a)$,  and $(R', \vartheta')=\bigl(R(b(a, v), \vartheta (b(a, v))\bigr)$.

The first map $\Phi_3'$ is a composition of the following two sc-smooth maps,
\begin{equation*}
\bd\oplus V\oplus L^+\xrightarrow{\,\,  C_1\,\,  } \bd\oplus L^+ \xrightarrow{\,\, C_2\,\,}\R^n.
\end{equation*}
The map $C_1\colon \bd\oplus V\oplus L^+\to \bd\oplus L^+$, defined 
by 
$$C_1(a, v, h)=(a, \bb(\pav) h(\pav ) \bigr),$$
is a composition of the sc-smooth maps $A_1$, $A_2$, and $A_3$ defined in (M3), and hence  is sc-smooth.

The second  map  $C_2\colon \bd\oplus L^+\to  \R^n$,   defined by 
$$C_2(a, h)=\biggl(1-\dfrac{\bba}{\ga}\biggr)\cdot [h]_a,$$
is  sc-smooth by Lemma 2.19 and Lemma 2.20 in \cite{HWZ8.7}.  Consequently, the map $\Phi_3'$ is sc-smooth. 

Next we consider the map $\Phi_3''$. This map can be represented as a composition of the following sc-smooth maps,
\begin{equation*}
\bd\oplus V\oplus L^-\xrightarrow{\,\,  D_1\,\,  } \bd\oplus L^+ \xrightarrow{\,\, D_2\,\,}\R^n.
\end{equation*}
The map $D_1\colon \bd\oplus V\oplus L^-\to \bd\oplus L^+$ is defined by 
$$D_1 (a, v, h)=\bigl(a,  \bigl(1-\bb(\pav)\bigr)\cdot \bb (\pav -\Delta'') \cdot h(\pav -(R', \vartheta') \bigr)\bigr)$$
where $(R', \vartheta')=\bigl(R(b(a, v), \vartheta (b(a, v))\bigr)$ and a constant $\Delta''$ is defined in (M4). This map is sc-smooth since it is a composition of the sc-smooth map 
$B_1, B_2$, and $B_3$ introduced  in M4.
The map $D_2$  is the same as the map  $C_2$ above,  and hence it is sc-smooth. 

We have proved that the map $\Phi_3''$ is sc-smooth.   This  together with the sc-smoothness of $\Phi_3'$ shows that map $\Phi_3$ is of class $\ssc^\infty$.  

At this point we have proved, assuming Theorem  \ref{thm_m3_1}, that all the maps defining $\eta^+$ are sc-smooth.
Hence it  remains to prove Theorem  \ref{thm_m3_1}.\\[0.3ex]

\begin{proof}[{\bf Proof of Theorem  \ref{thm_m3_1}}]
Let $L$ be the sc-Banach space $H^{3,\delta_0}(\R^+\times S^1)$ and let $\abs{\cdot}_m=\norm{\cdot}_{H^{3+m, \delta_m}(\R^+\times S^1)}$ for $m\geq 0$. We recall that $\wt{\Phi}\colon  \bd\oplus V\oplus L\to L$ 
is the composition 
$$\wt{\Phi} (a, v, u)=u\circ \pav.$$
In the proof we denote by $C$ a generic constant which depends only on 
the order $m$ of differentiation of the maps involved. 

We begin by showing that the map $\wt{\Phi}$ is well-defined.  We fix  $m\geq 0$ and a point $(a, v, u)\in \bd\oplus V\oplus L_m$. It suffices to show that 
\begin{equation}\label{int_eq_1}
\int_{\R^+\times S^1}\abs{D^\alpha (u(\pav))}^2e^{2\delta_m s}\ dsdt
\end{equation}
are finite for all multi-indices $\abs{\alpha}\leq 3+m$. Denoting by 
$\pav^1$ and $\pav^2$ the components of the diffeomorphisms $\pav$, 
we will be useful to introduce the following notation, 
$${\bf D}^{\mu, \nu}(\pav)={\bf D}^\mu(\pav^1)\cdot {\bf D}^\nu (\pav^2),$$
where 
\begin{align*}
{\bf D}^\mu(\pav^1)&=D^{\mu_1}(\pav^1)\cdot \ldots \cdot D^{\mu_k}(\pav^1),\\
{\bf D}^\nu (\pav^2)&=[D^{\nu_1}(\pav^2)\cdot \ldots \cdot D^{\nu_l}(\pav^2),
\end{align*} 
and $\mu=(\mu_1,\ldots,\mu_k)$ and $\nu=(\nu_1,\ldots,\nu_l)$.
With this notation,  the derivative 
$D^\alpha  (u(\pav))$ is a linear combination of the following expressions,
\begin{equation}\label{int_eq_2}
(D^\gamma u)(\pav)\cdot{\bf D}^{\mu, \nu}(\pav)=
(D^\gamma u)(\pav)\cdot{\bf D}^\mu(\pav^1)\cdot {\bf D}^\nu (\pav^2), 
\end{equation}
where  the multi-indices satisfy  $\abs{\gamma}\leq \abs{\alpha}$, 
$\alpha=\mu_1+\ldots +\mu_k+\nu_1+\ldots +\nu_l$,  and $k+l\leq \abs{\gamma}.$  Moreover, $1\leq \abs{\mu_i}, \abs{\nu_j}$, for $1\leq i\leq k$, $1\leq j\leq l$. 
Accordingly the integral \eqref{int_eq_1} is a linear combination of the terms
\begin{equation}\label{int_eq_1a}
\int_{\R^+\times S^1}
\abs{(D^\gamma u)\cdot {\bf D}^{\mu, \nu}(\pav)}^2e^{2\delta_ms}\ dsdt
\end{equation}
with multi-indices satisfying the conditions listed above.
Since, in view of the properties of the map $r(a, v)$, there exists a constant $C$ such that 
\begin{equation}\label{int_eq_2a}
\abs{D^\beta r(a, v)}\leq C \quad \text{on $\R^+\times S^1$,}
\end{equation} for all $\abs{\beta}\leq 3+m$, the terms \eqref{int_eq_1a} can be estimated,  using the change of variable formula,  as
\begin{equation}\label{change_1}
\begin{split}
&\int_{\R^+\times S^1}
\abs{(D^\gamma u)\cdot {\bf D}^{\mu, \nu}(\pav)}^2e^{2\delta_ms}\ dsdt\\
&\quad \leq 
C\int_{\R^+\times S^1}\abs{(D^\gamma u)(\pav))}^2e^{2\delta_m s}\ dsdt\\
&\quad =C\int_{\R^+\times S^1}\abs{(D^\gamma u)(s, t)}^2\cdot \abs{\det (D\psav)(s, t)}e^{2\delta_m \psav^1s}\ dsdt.
\end{split}
\end{equation}
Here we abbreviated by $\psav$ the inverse of $\pav$ and by $\psav^1$ its first component.  Hence we need  to estimate $\psav^1$ and $\abs{\det D\psav}$ on $\R^+\times S^1$.  We claim that there are constants $c_0$ and $c_1$ such that 
\begin{subequations}\label{properties}
\begin{gather}
\abs{\psav^{1}(s, t)-s}\leq c_0,\label{first}\\
\abs{\text{det}\ (D \psav)(s, t)}\leq c_1\label{second}
\end{gather}
\end{subequations}
for all $(s, t)\in \R^+\times S^1$.
Since $\pav(\psav (s, t))=(s, t)$ and 
\begin{equation}\label{int_eq_3}
\pav(s, t)=(s, t)+d(a, v)+r(a, v)(s, t),
\end{equation}
we obtain
\begin{equation}\label{int_eq_4}
\psav(s, t)=(s, t)-d(a, v)-r(a, v)\bigl(\psav (s, t)).
\end{equation}
So, the first assertion  follows from \eqref{int_eq_4} and $\abs{r(a, v)}\leq C$ on $\R^+\times S^1$. For the second claim we use  the identity $\text{det} (D\psav)=\bigl( \text{det}(D \pav)(\psav)\bigr)^{-1}$.  It is enough to  show that  
$\abs{\text{det}(D \pav)(\psav)}\geq 1/c_1$ on $\R^+\times S^1$. From \eqref{int_eq_3} we get,  
$$(D\pav)(\psav)=\text{id}+Dr(a, v)(\psav),$$
and since 
$\abs{Dr(a, v)(s, t)}\leq ce^{-\delta_ms}$, by  \eqref{first}, 
there exists $s_0>0$ such that $\abs{(D\pav)(\psav)}\geq 1/2$  on  $[s_1,\infty)\times S^1$. Since $\pav$ is a diffeomorphism, $\abs{(D\pav)(\psav)}>0$ on 
$[0,s_0]\times S^1$, finishing the prove of the claim \eqref{second}.

With \eqref{first} and \eqref{second}, the right-hand side of 
\eqref{change_1} is dominated  by 
\begin{equation}\label{int_eq_3a}
C\int_{\R^+\times S^1}\abs{D^\gamma u(s, t)}^2e^{2\delta_m s}\ dsdt.
\end{equation}
Since this holds true for every multi-index $\gamma$ satisfying $
\abs{\gamma}\leq 3+m$, we have proved that 
the integral \eqref{int_eq_1} is finite and that the map $\wt{\Phi}$ is well-defined. Also, the above discussion shows   that  
\begin{equation}\label{int_eq_3ab}
\abs{u\circ \pav}_m\leq C\abs{u}_m
\end{equation}
where the constant $C$ depends on $m$ and the parameter $(a, v)$.

We make the following observation which will be  used in the proof of the next lemma. Given a point $(a_0, v_0)\in \bd\oplus V$, the estimates 
\eqref{int_eq_2a}, \eqref{first}, \eqref{second} as well as  \eqref{int_eq_3ab} still hold true, perhaps with bigger constants,  for all $(a, v)$ close to $(a_0, v_0)$. This follows from the fact that the map $(a, v)\mapsto r(a, v)$ is smooth as a map from $\bd\oplus V\to H^{m',\varepsilon}(\R^+\times S^1)$ for every $m'\geq 3$ and $0<\varepsilon<2\pi$. 

Now we are ready to prove sc-continuity of the map $\wt{\Phi}$.

\begin{lemma}\label{lem_M_3}
The map $\wt{\Phi} $ is sc-continuous.
\end{lemma}

\begin{proof}
We fix level $m\geq 0$,  a point $(a_0, v_0,u_0)\in \bd\oplus V\oplus L_m$,   
and take $\rho>0$. We choose a compactly supported smooth function $h_0\colon \R^+\times S^1\to \R^n$ satisfying $\abs{u_0-h_0}_m<\rho$ and $\supp h_0\subset [0, s_1]\times S^1$. 
Then we  estimate,
\begin{equation}\label{eq_lem_M_3}
\begin{split}
\abs{\wt{\Phi} (a, v, u)-\wt{\Phi}  (a_0, v_0, u_0)}_m&\leq \abs{\wt{\Phi}  (a, v, u)-\wt{\Phi} (a, v, h_0)}_m\\
&\phantom{\leq }+\abs{\wt{\Phi}   (a, v, h_0)-\wt{\Phi}  (a_0, v_0, h_0)}_m\\
&\phantom{\leq }+\abs{\wt{\Phi} (a_0, v_0, h_0)-\wt{\Phi}  (a_0, v_0, u_0)}_m\\
&=I+II+III.
\end{split}
\end{equation}

We consider the term $I$. By \eqref{int_eq_3ab} and recalling that  
$\abs{u_0-h_0}_m<\rho$, we obatin

\begin{equation}\label{eq_I_a}
\begin{split}
I&=\abs{\wt{\Phi} (a, v, u-h_0)}_m\leq C\abs{u-h_0}_{m}\\
&\leq C\abs{u-u_0}_{m}+
C\abs{u_0-h_0}_{m}\\
&\leq C\abs{u-u_0}_{m}+C\rho.
\end{split}
\end{equation}
for $(a, v)$ sufficiently close to $(a_0,  v_0)$ with the constant $C$ depending only on $m$.

Using \eqref{int_eq_3ab} again, we obtain the estimate for the term $III$, 
\begin{equation}\label{int_est_M3_6}
III=\abs{ \wt{\Phi} (a_0, v_0, u-h_0)}_m  \leq C\abs{u_0-h_0}_{m}\leq C\rho.
\end{equation}

It remains to estimate the term $II=\abs{\wt{\Phi}   (a, v, h_0)-\wt{\Phi} (a_0, v_0, h_0)}_m=\abs{h_0(\pav)-h_0(\phi_{a_0, v_0})}_m$. The square of the norm 
$\abs{h_0(\pav)-h_0(\phi_{a_0, v_0}}_m$ is a linear combination of 
the terms 
\begin{equation}\label{int_est_M3_15}
\int_{\R^+\times S^1}\abs{D^\alpha \bigl[h_0(\pav)-h_0(\phi_{a_0, v_0})\bigr]}^2e^{2\delta_m s}\ ds dt
\end{equation}
where the sum is taken over all multi-indices $\alpha$ satisfying $\abs{\alpha}\leq 3+m$.  Recalling the notation introduced in the proof that $\wt{\Phi}$ is well-defined, 
the derivatives  $D^\alpha \bigl[h_0(\pav)-h_0(\phi_{a_0, v_0})\bigr]$ can be written as a  linear combinations of the expressions of the form
\begin{equation}\label{int_est_M3_8}
\begin{split}
(D^\gamma h_0)(\pav)\cdot &\bigl[ {\bf D}^{\mu, \nu}(\pav)-{\bf D}^{\mu, \nu}(\phi_{a_0,v_0})\bigr]\\
&\phantom{=}+\bigl[(D^\gamma h_0)(\pav)-(D^\gamma h_0)(\phi_{a_0, v_0})\bigr]\cdot {\bf D}^{\mu, \nu}(\phi_{a_0,v_0})\\
&=II_1+II_2
\end{split}
\end{equation}
where $\abs{\gamma}\leq \abs{\alpha}\leq 3+m$, $\alpha=(\mu_1+\ldots +\mu_k)+(\nu_1+\ldots +\nu_l)$, and $k+l\leq \abs{\gamma}.$
We consider  the term $II_1$. The factors ${\bf D}^{\mu, \nu}(\pav)-{\bf D}^{\mu, \nu}(\phi_{a_0,v_0})$ can be estimated as follows,
\begin{equation*}
\begin{split}
&\abs{{\bf D}^{\mu, \nu}(\pav)-{\bf D}^{\mu, \nu}(\phi_{a_0,v_0})}\\
&\quad \leq \abs{{\bf D}^\mu(\pav^1)\cdot {\bf D}^\nu (\pav^2)-{\bf D}^\mu(\phi_{a_0,v_0}^1)\cdot {\bf D}^\nu (\phi_{a_0,v_0}^2)}\\
&\quad  \leq \abs{{\bf D}^\mu(\pav^1)-{\bf D}^\mu(\phi_{a_0,v_0}^1)}\cdot\abs{ {\bf D}^\nu (\pav^2)}\\
&\quad \quad +\abs{{\bf D}^\mu(\phi_{a_0,v_0}^1)}\cdot \abs{{\bf D}^\nu (\pav^2)-{\bf D}^\nu (\phi_{a_0,v_0}^2)}\\
&\quad  \leq C\cdot \bigl[  \abs{{\bf D}^\mu(\pav^1)-{\bf D}^\mu(\phi_{a_0,v_0}^1)}+
\abs{{\bf D}^\nu (\pav^2)-{\bf D}^\nu (\phi_{a_0,v_0}^2)}\bigr]
\end{split}
\end{equation*}
since  $\abs{D^\beta (\pav)}\leq C$ on $\R^+\times S^1$ for all $\abs{\beta}\leq 3+m$.\\
\noindent From 
\begin{equation*}
\begin{split}
&{\bf D}^\mu(\pav^1)-{\bf D}^\mu(\phi_{a_0,v_0}^1)=
{\bf D}^\mu(r^1(a, v))-{\bf D}^\mu(r^1(a_0,v_0))=\\
&\sum_{1\leq i\leq k}D^{\mu_1}(r^1(a, v))\cdots\bigl[ D^{\mu_i}\bigl(r^1(a, v))-r^1(a_0, v_0)\bigr)\bigr]\cdots D^{\mu_k}(r^1(a, v)),
\end{split}
\end{equation*}
we conclude that 
\begin{equation*}
\begin{split}
\abs{{\bf D}^\mu(\pav^1)-{\bf D}^\mu(\phi_{a_0,v_0}^1)}&\leq C
\sum_{1\leq i\leq k}\bigl[ D^{\mu_i}\bigl(r^1(a, v))-r^1(a_0, v_0)\bigr)\bigr]\\
&\leq C
\sum_{\abs{\beta}\leq 3+m}\bigl[ D^{\beta}\bigl(r(a, v))-r(a_0, v_0)\bigr)\bigr].
\end{split}
\end{equation*}
The estimate for the factor ${\bf D}^\nu(\pav^1)-{\bf D}^\nu(\phi_{a_0,v_0}^2)$ is the same, so that 
$$
\abs{{\bf D}^{\mu, \nu}(\pav)-{\bf D}^{\mu, \nu}(\phi_{a_0,v_0})}\leq C\sum_{\abs{\beta}\leq 3+m} \bigl[ D^\beta (r(a, v))-
D^\beta(r(a_0, v_0))\bigr].
$$
Now we recalling that $h_0$ is a smooth function having support contained in $[0,s_1]\times S^1$, we let $C_0=\abs{h_0}_{C^{3+m}([0,s_1]\times S^1)}$. Then the  contribution of the term $II_1$ to the integral \eqref{int_est_M3_15} is bounded above by 
\begin{equation*}
\begin{split}
&C_0\cdot C\int_{\R^+\times S^1}\abs{ \sum_{\abs{\beta}\leq 3+m}  D^\beta \bigl[ r(a, v))-
r(a_0, v_0)\bigr] }^2e^{2\delta_ms}\ dsdt\\
&\quad \leq C_0\cdot C\sum_{\abs{\beta}\leq 3+m}
\int_{\R^+\times S^1}\abs{D^\beta \bigl[ r(a, v))-
r(a_0, v_0)\bigr]}^2e^{2\delta_ms}\ dsdt\\
&\quad \leq  C_0\cdot C\abs{r(a, v)-r(a_0, v_0)}^2_m.
\end{split}
\end{equation*}
For the term  $II_2$, we have 
\begin{equation*}
\begin{split}
&\abs{\bigl[ (D^\gamma h_0)(\pav)-(D^\gamma h_0)(\phi_{a_0, v_0})\bigr]\cdot {\bf D}^{\mu, \nu}(\phi_{a_0, v_0})}\\
&\quad \leq C\cdot 
\abs{(D^\gamma h_0)(\pav)-(D^\gamma h_0)(\phi_{a_0, v_0})} \leq C_0\cdot C \cdot 
\abs{\pav-\phi_{a_0, v_0}}\\
&\quad=C_0\cdot C\cdot \bigl[ \abs{d(a, v)-d(a_0, v_0)}+\abs{r(a, v)-r(a_0, v_0)}\bigr].
\end{split}
\end{equation*}
Hence  the contribution of this term  to 
the integral \eqref{int_est_M3_15} is bounded  by 
\begin{equation*}
\begin{split}
C_0^2\cdot \dfrac{(e^{2\delta_m s_1}-1)}{2\delta_m}\cdot  \abs{d(a, v)-d(a_0, v_0)}^2+C_0^2\cdot \abs{r(a, v)-r(a_0, v_0)}_m^2.
\end{split}
\end{equation*}
Combining both estimates,  the term $II$ is bounded by 
$$
II\leq C_0\cdot C\cdot (e^{\delta_m s_1}-1)\cdot \abs{d(a, v)-d(a_0, v_0)} +C_0\cdot C\cdot \abs{r(a, v)-r(a_0, v_0)}_m
$$
and hence, by \eqref{eq_I_a} and \eqref{int_est_M3_6},  we get 
\begin{equation}\label{int_est_M3_11}
\begin{split}
&\abs{\wt{\Phi} (a, v,u)-\wt{\Phi} (a_0, v_0,u_0)}_m\leq C\abs{u-u_0}_m+2C\rho\\
&\phantom{\leq }  +
C_0(1-e^{\delta_m s_1})\abs{d(a, v)-d(a_0, v_0)} +C_0C\cdot \abs{r(a, v)-r(a_0, v_0)}_m.
\end{split}
\end{equation}
The number $\rho$ can be chosen to be as small as we wish. Also, the terms $\abs{d(a, v)-d(a_0,v_0)}$ and $\abs{r(a, v)-r(a_0, v_0)}$ are small for $(a, v)$ sufficiently close to $(a_0, v_0)$.  Therefore,  \eqref{int_est_M3_11} shows  that 
$\wt{\Phi} $ is continuous on every level $m$. This means that $\wt{\Phi} $ is of class $\ssc^0$ and the  proof of the lemma is complete. 
\end{proof} 

Next we show that $\wt{\Phi} $  is of class $\ssc^1$. To shorten out notation we write $b=(a, v)$. The candidate for the linearization $D\wt{\Phi} ( b,  u)\colon \bd\oplus V\oplus L\to L$ at the point $(b, u)$ is given by 
\begin{equation}\label{eq_lin_1}
\begin{split}
D\wt{\Phi}  (b, u)(c,   w)=\wt{\Phi} (b, w)+
\wt{\Phi} (b,  u_s)\cdot D\phi_b^1 \cdot c+\wt{\Phi} (b, u_t)\cdot D\phi_b^2 \cdot c,
\end{split}
\end{equation}
where $(c, w)\in \R^2\oplus \R^N\oplus L$ and $D\phi_b^1$ (resp. $D\phi_b^2$) denotes the derivative of $\phi_b^1$ (resp. $\phi_b^2$)  with respect to the variable $b$. 

We have already proved  that the map $\wt{\Phi} \colon \bd \oplus V\oplus L\to L$ is of class $\ssc^0$. The maps 
$L^1\to L$, defined by $u\mapsto u_s, u_t$,  are sc-operators. Hence, by  the chain rule,  
 the maps $\bd \oplus V\oplus L^1\to L$, defined $(b, u)\mapsto\wt{\Phi}  (b, u_s)$ (resp. $\wt{\Phi} (b,u_t)$) are of class $\ssc^0$. 
The map  $D\phi_b^1(s, t)\cdot c$ is  linear in $c$ and since 
$D\phi_b^1(s, t)=Dc(b) +Dr(b)(s, t)$,  where $D$ denotes  derivative with respect to $b$,   we deduce from the properties of $b\mapsto c(b)$ and $b\mapsto r(b)$, that the map $(b, c, s, t)\mapsto D\phi_b^1(s, t)\cdot  c$ is continuous. This implies that 
the map $(b,  u, c, s, t)\mapsto \Psi  (b, u_s)D\phi_b^1(s, t)\cdot c$ is continuous on every level $m$. The same holds for the map $\wt{\Phi} (b, u_t)D\pav^2\cdot c$. 

We have proved that the map $T(\bd\oplus V\oplus L)\to  TL$, defined by 
\begin{equation}\label{guess1}
T(b,u,b, w)\mapsto  \bigl(\wt{\Phi} (b, u),D\wt{\Phi} (b,u)(b, w)\bigr),
\end{equation}
 is $\ssc^0$.
Next we show that the right-hand side of \eqref{eq_lin_1} defines the linearization of $\wt{\Phi}$.
With $(b,u,c,w)\in T(V\oplus E)$ and using
 $$\wt{\Phi} (b+c, u)-\wt{\Phi} (b, u)=\int_0^1\frac{d}{d\tau }\wt{\Phi} (b+\tau c, u)\ d\tau,$$ 
 we have 
\begin{equation*}
\begin{split}
& \wt{\Phi} (b+c, u+w)-\wt{\Phi} (b,u)-D\wt{\Phi} (b, u)(c,w)\\
&\phantom{===}=\wt{\Phi} (b+c,w)-\wt{\Phi} (b, w) \\
&\phantom{====}+\int_0^1 \bigl[ \wt{\Phi} (b+\tau c, u_s)D\phi_{b+\tau c}\cdot c-\wt{\Phi} (b,u_s)D\phi^1_{b}\cdot c\bigr]\ d\tau\\
&\phantom{====}+\int_0^1 \bigl[\wt{\Phi} (b+\tau c,u_t)D\phi^2_{b+\tau c}\cdot c-\wt{\Phi} (b,u_t)D\phi_b^2\cdot c]\ d\tau\\
&\phantom{===}=I+II+III.
\end{split}
\end{equation*}
We consider the term $I$. Since $\wt{\Phi}$ is linear with respect to the second variable, we have
\begin{equation*}
\begin{split}
&\dfrac{1}{\abs{b}+\abs{w}_1}\cdot \abs{\wt{\Phi} (b+c,w)-\wt{\Phi} (b,w) }_0\\
&\quad =\dfrac{\abs{w}_1}{\abs{b}+\abs{w}_1}\cdot \left| \wt{\Phi} i\left( b+c,\frac{w}{\abs{w}}_1\right)-\wt{\Phi} \left(b, \frac{w}{\abs{w}}_1\right) \right|_0
\end{split}
\end{equation*}
for $w\neq 0$. The inclusion  $L_1\to  L_0$  is compact and hence we may assume that
$\frac{w}{\abs{w}_1}\to w_0$ in $L_0$. Since $\wt{\Phi}$ is $\ssc^0$,  we conclude that
$$\frac{1}{\abs{b}+\abs{w}_1}\cdot \abs{\wt{\Phi} (b+c,w)-\wt{\Phi} (b,w) }_0\to 0$$
as $\abs{b}+\abs{w}_1\to 0$.  Next we consider the second term $II$. We have,   for $c\neq 0$,
\begin{equation*}
\begin{split}
&\frac{1}{\abs{c}+\abs{w}_1}\left| \int_0^1 [\wt{\Phi} (b+\tau c,u_s)\cdot D\phi^1_{b+\tau c}\cdot c-\wt{\Phi} (v,u_s)\cdot D\phi^1_{b}\cdot c]\ d\tau\right|_0\\
&\leq \frac{\abs{c}}{\abs{c}+\abs{w}_1}\int_0^1 \left| \wt{\Phi} (b+\tau c,u_s)\cdot D\phi^1_{b+\tau c}\cdot \frac{c}{\abs{c}}-\wt{\Phi} (b,u_s)\cdot D\phi^1_{b}\cdot  \frac{c}{\abs{c}}\right|_0 \ d\tau
\end{split}
\end{equation*}
Since $\wt{\Phi} $ is $\ssc^0$ and  $(b, c, (s, t))\mapsto D\phi^1_{b} (s, t)(c/\abs{c})$ is smooth,  we conclude that
the above expression converges to $0$ as $\abs{c}+\abs{h}_1\to 0$. The same holds for the term $III$.
We have proved that
$$\frac{1}{\abs{c}+\abs{w}_1}\cdot \abs{ \wt{\Phi} (b+c,u+w)-\wt{\Phi} (b, u)-D\wt{\Phi} (b,u)(c, w)  }_0\to 0  $$
as $\abs{c}+\abs{w}_1\to 0$ so that the right-hand side of \eqref{eq_lin_1} is indeed  the linearization of $\wt{\Phi} $ in the sense of Definition \ref{sscc1}.

Therefore, the sc-continuous map given by equation \eqref{guess1} defines the tangent map $T\Phi :T(\bd\oplus V\oplus L)\to TL$. This finishes the  proof  that the map $\wt{\Phi} $ is of class $\ssc^1$.

Sc-smoothness of the map $\wt{\Phi} $ is a consequence of the following lemma which gives the form of the iterated tangent map $T^k\wt{\Phi} $. 
\begin{lemma}\label{structure}  For every $k$, the map 
$$\wt{\Phi} \colon \bd\oplus V\oplus L\to  L,\quad (b,u)\mapsto  u\circ \psi_b,$$ 
is $\ssc^k$. Moreover, assuming that $\pi\colon T^kL\to L^j$ is the  projection onto a factor of  $T^kL$,  the composition $\pi\circ T^k\wt{\Phi} $ is a linear combination of maps  of the form
\begin{equation*}
\begin{gathered}
\bd \oplus V\oplus L^m\oplus(\R^{2+N})^p\to  L^j,\\
(b,w,c_1, \ldots ,c_p)\mapsto
\wt{\Phi} (b,D^{\alpha}w)\cdot f(b, c_1, \ldots ,c_p), 
\end{gathered}
\end{equation*}
where $f\colon \bd \oplus V\oplus  (\R^{2+N})^p\times \R^+\times S^1\to  \R$ is a smooth function which is  linear in every variable $c_i$.  Moreover,  $\abs{\alpha}\leq m-j$ and $p\leq k$.
\end{lemma}
\begin{proof} We prove the lemma by induction with respect to $k$ starting with $k = 0.$
In this case, the statement is trivially satisfied,  since by Lemma \ref{lem_M_3} the map 
$\wt{\Phi} $ is of class $\ssc^0$. Moreover,  there is exactly one projection $\pi:T^0L=L\to L$, namely the identity map, so that $\pi\circ \wt{\Phi} =\wt{\Phi} $ has the required form
with  $m=j=0$, $\alpha=(0,0)$, $p=0$, and $f\equiv 1$.
The statement of the lemma also holds for $k=1$. It  follows from  \eqref{eq_lin_1} and \eqref{guess1},  that the compositions of the tangent map $T\Phi$ with projections $\pi$ onto factors of $TL=L^1\oplus L$  are linear combinations of  maps  of the required form.

We assume that the statement has been proved for $k$ and verify that it holds for $k+1$. It suffices to show that the compositions of the  iterated tangent map $T^{k}\wt{\Phi} \colon T^k(\bd\oplus V\oplus L)\to T^kL$ with the projections $\pi\colon T^kL\to L^j$ onto the  factors of $T^kL$
are of class $\ssc^1$ and their linearizations have the required form.
By inductive assumption, $\pi\circ T^{k}\wt{\Phi} $ is the linear combination of maps  having the particular forms, 
\begin{equation*}
\Gamma\colon \bd\oplus V\oplus L^m\oplus (\R^{2+N})^{p}\to  L^j,
\quad (b,w,c)\mapsto
\wt{\Phi}(b,D^{\alpha}w)\cdot f(b,c),
\end{equation*}
where we abbreviated $c=(c_1,\ldots ,c_p)\in  (\R^{2+N})^{p}$  and $\abs{\alpha}\leq m-j$ and $p\leq k$.
It  is enough  to show that our claim holds for each of these  maps. By assumption, 
the function $f\colon \bd\oplus V\times (\R^{2+N})^p\times \R^+\times S^1 \to \R$ is  smooth and linear in each variable $c_i$. 

The map $\Gamma$ is the composition of the following maps. The sc-operator $L^m\to L^{m-\abs{\alpha}}$,  defined by $h\mapsto D^{\alpha}h$, is composed with  the map
$$\wt{\Phi} \colon V\oplus L^m \to  L^j, \quad (v,u)\mapsto \wt{\Phi} (v,u)$$
which we already know is of class $\ssc^1$. By the chain rule, this composition is at least of class $\ssc^1$. Multiplication of this composition  by a smooth function $\bd \oplus V\oplus (\R^{2+N})^p\to \R$ defined by $(b, c)\mapsto f(b, c)$ gives  an $\ssc^1$-map.   Hence the map $\Gamma$ is of class $\ssc^1$ and it  remains to show that
the compositions $\pi\circ T\Gamma$  of the tangent map $T\Gamma\colon T(\bd\oplus V\oplus L^m\oplus (\R^{2+N})^{p})\to T(L^j)$  with the projections onto factors of $T(L^j)$ are linear combinations of maps of the required form.
The tangent map of $\Gamma$ is given by
$$T\Gamma(b, w, c, \delta b, \delta w, \delta c)=\bigl( \Gamma (b, w, c), D\Gamma (b, w, c)( \delta b, \delta w, \delta c)\bigr)$$
where $(b, w, c)\in \bd\oplus V\oplus L^{m+1}\oplus (\R^{2+N})^{p}$ and $( \delta b, \delta w, \delta c)\in \R^{2+N}\oplus L^m\oplus (\R^{2+N})^{p}.$ Denoting by $\pi\colon T(L^j)=L^{j+1}\oplus L^j\to L^{j+1}$  the projection onto  the first factor of $T(L^j)$, then, by our inductive assumption, 
$\pi\circ T\Gamma =\Gamma$  is a linear combination of terms having required form but with the indices $m$ and $j$ raised by $1$.

So, we consider the projection $\pi\colon T(L^j)\to L^j$  onto the second factor of $T(L^j)$ and the map $\pi\circ T\Gamma=D\Gamma$. Using the chain rule and the linearization of $\Gamma$ given by 
\eqref{eq_lin_1},
the linearization $D\wt{\Phi} $ is a linear combination of the following four types of maps:
\begin{align*}
(1)&\quad \bd\oplus V\oplus L^{m}\oplus (\R^{2+N})^p \to L,\\
&\qquad \qquad(b, \delta w, c)\mapsto  \Gamma (b,D^{\alpha} (\delta w) )\cdot f(b,c).&\\
(2)&\quad\bd\oplus V\oplus  L^{m+1}\oplus (\R^{2+N})^{p+1}\to L,\\
&\qquad \qquad(b, w,  (\delta b, c))\mapsto    \Gamma (b,D^{\alpha +(1,0)} w)\cdot (D\phi^1_{b}\cdot \delta b) f(b,c),\;\; \text{and}\\
&\qquad \qquad(b, w,  (\delta b, c))\mapsto  \Gamma(b,D^{\alpha +(0,1)} w)\cdot (D\phi^2_{b}\cdot \delta b) f(b,c).&\\
(3)&\quad\bd\oplus V\oplus  L^{m+1}\oplus (\R^{2+N})^{p+1}\to L,\\
&\qquad \qquad(b, w, (\delta b, c))\mapsto   \Gamma(b,D^{\alpha} w )\cdot D_bf(b,c)\cdot \delta b.&\\
(4)&\quad\bd\oplus V\oplus  L^{m+1}\oplus (\R^n)^{p+1}\to L,\\
&\qquad \qquad(b, w, (c, \delta c_i))\mapsto   \Gamma(b,D^{\alpha} w )\cdot  f(b,(c_1,\ldots ,\delta c_i, \ldots , c_p))\\
\phantom{(4)}&\qquad \text{for every $1\leq i\leq p$}.
\end{align*}

All of these four types of maps have the desired form. Having verified the statement for $k+1$, the proof of Lemma  \ref{structure} is complete.
\end{proof}

With the proof of Lemma \ref{structure}, the proof of Theorem  \ref{thm_m3_1} is finished.\\[0.5ex]

This completes the proof of Theorem \ref{theorem_neck}.
\end{proof}

\section{Proof of Lemma \ref{GROMOVCONV} }\label{section5.1}
\begin{L3.4}
We consider the stable map  $(S, j, M, D, u)$ and two sequences $u_k, u_k'\in H^{3,\delta_0}(S, Q)$ of maps from the noded Riemann surface $S$ into the symplectic manifold $Q$   converging  to $u$ in $C^0$. We assume 
that $(a_k,v_k)\rightarrow (0,0)\in O$ and $(b_k,w_k)\rightarrow (0,0)\in O$ and  assume that  
$$
\phi_k:(S_{a_k},j(a_k,v_k),M_{a_k},D_{a_k},\oplus_{a_k}(u_k))\rightarrow
(S_{b_k},j(b_k,w_k),M_{b_k},D_{b_k},\oplus_{b_k}(u'_k))
$$
is a sequence of isomorphisms. Then there is a subsequence of $(\phi_k)$ which converges in $C^{\infty}_{\text{loc}}$ away from the nodes to an automorphism $\phi_0$ of $(S, j, M, D, u)$. 
\end{L3.4}
\begin{proof}
We abbreviate $S_k=S_{a_k}$ and $S'_k=S_{b_k}$ and choose two sequences $g_k$ and $g_k'$ of Riemann metrics on $S_k$ and $S_k'$, which are independent of $k$ on the core regions of $S_k$ and $S_k'$, respectively, and which are translation and rotation invariant standard metrics on the neck regions of $S_k$ and $S_k'$
(identifying the necks with cylinders). We abbreviate  the norms  by  $\abs{ \cdot  }_k=\sqrt{g_k'(\cdot , \cdot )}$.

In a first step we  show that the sequence $\abs{\nabla \phi_k}_k$ is uniformly bounded.  Arguing by contradiction we assume that there is a sequence of  points $z_k\in S_k$ satisfying 
$$
R_k=\abs{\nabla \phi_k (z_k)}_k\to \infty. 
$$

We next show that the sequence $(z_k)$ stays at a finite distance to  the core regions of $S_k$. Arguing by contradiction, we assume, after taking a subsequence, that $z_k$ lies in the neck region 
$N_k$ associated with the nodal pair $\{x, y\}\in D$ and that $\text{dist}\ (z_k, \partial N_k)\to \infty$. 
We distinguish the following two cases. 
\begin{itemize}
\item[(1)] The image  sequence  $(\phi_k(z_k))$ lies in the neck $N_k'$ of $S_k'$  associated with the nodal pair $\{x',y'\}\in D$ and $\text{dist}\ (\phi_k(z_k), \partial N_k')\to \infty$.
\item[(2)] The image  sequence  $(\phi_k(z_k))$ lies at a  finite distance to the core region of $S_k'$.
\end{itemize}

To deal with the case (1) we recall that  the necks  $N_k$ of $S_k$ are   the finite cylinders $Z^{\{x,y\}}_{a_k}:=Z^{\{x,y\}}_{a_k^{\{x,y\}}}$  connecting the boundaries of the associated disks $D_x$ and $D_y$ of the small disk structure. These  cylinders are defined as follows. 
We choose the positive and negative  holomorphic polar coordinates
\begin{equation*}
\begin{aligned}
&h_x:[0,\infty)\times S^1\rightarrow D_x\setminus\{x\},&\qquad 
&\bar{h}_x(s, t)=\ov{h}_x(e^{-2\pi (s+it)}),\\
&h_y:(-\infty, 0]\times S^1\rightarrow D_y\setminus\{y\},&\qquad 
&\bar{h}_y(s, t)=\ov{h}_y(e^{2\pi (s+it)}),
\end{aligned}
\end{equation*}
where
\begin{align*}
\ov{h}_x:\{w\in \C\vert \, \abs{w}\leq  1\}\to D_x\quad \text{and}\quad 
\ov{h}_y:\{w\in \C\vert \, \abs{w}\leq  1\}\to D_y
\end{align*}
are biholomorphic mappings. 
We remove from the disk $D_x$ the  points  $z=h_x(s,t)\in D_x$  for $s>\varphi(|a_k|)$ and from $D_y$ the points $z'=h_y(s',t')\in D_y$  for  $s'<-\varphi(|a_k|)$. Here $\varphi $ is the exponential gluing profile.  The remaining  points of the disks $D_x$ and $D_y$ are identified as follows. The points 
$z=h_x(s,t)$ and $z'=h_y(s',t')$
are equivalent if 
$s=s'+R_k$ and $t=t'+\vartheta    \pmod 1$, where $R_k=\varphi (\abs{a_k})$  and
 $a_k=\abs{a_k}e^{-2\pi i\vartheta}.$ 
 
 Assuming that  the sequence $(\phi_k(z_k))$ lies in the neck $N'_k=Z^{\{x', y'\}}_{b_k}$  
 associated with the nodal pair $\{x',y'\}\in D$ and $\text{dist}\ (\phi_k(z_k), \partial N_k')\to \infty$, we identify 
 $Z^{\{x,y\}}_{a_k}$ and $Z^{\{x',y'\}}_{b_k}$ with the cylinders  $[0,R_k]\times S^1$ and $[0, R'_k]\times S^1$. 
 Then  $z_k=(s_k, t_k)\in [0,R_k]\times S^1$ satisfies $s_k\to \infty$ and $R_k-s_k\to \infty$. Similarly, if $\phi_k (z_k)=(r_k, \theta_k)\in [0, R_k']\times S^1$, then $r_k$  satisfies $r_k\to \infty$ and $R_k'-r_k\to \infty$. In view of the periodicity in the $t$-variable, we view the maps $\phi_k$ as maps on (subsets of)  $\C$. 
 
Applying   the bubbling off arguments  from  \cite{H0}, we choose a sequence $(\varepsilon_k)$ of positive numbers such that $\varepsilon_k\to 0$ and $\varepsilon_kR_k\to \infty$.  In view of 
Lemma  26 in  \cite{H0} we can modify the sequences $(\varepsilon_k)$ and $(z_k)$ so that the new sequences  satisfy 
 \begin{gather}\label{gr-eq1}
\varepsilon_k\to 0,   \quad  R_k\varepsilon_k\to \infty\\
 \abs{\nabla \phi_k (z)}\leq 2R_k\quad \text{if \ $\abs{z-z_k}\leq \varepsilon_k$}.
 \end{gather}
Introducing the representation  $\phi_k (z)=(a_k(z), \theta_k (z))\in \R\times S^1$ and  $m_k=a_k(z_k)$, we define the sequence $(\wt{\phi}_k)$ of rescaled holomorphic maps  on the disks $\abs{z}\leq \varepsilon_k R_k$ by 
$$\wt{\phi}_k(z)=
\bigl(\wt{a}_k(z), \wt{\theta}_k(z)\bigr)=\bigl( a_k\big( z_k+\frac{z}{R_k}\bigr)-m_k, \theta_k 
\bigl(z_k+\frac{z}{R_k}\bigr)\bigr).
$$
The maps $\wt{\phi}_k$ satisfy
$$
\abs{\nabla \wt{\phi}_k(0)}_k=1 \ \  \text{and}\quad  \abs{\nabla \wt{\phi}_k(z)}_k\leq 2\quad \text{if}\ \ \abs{z}\leq \varepsilon_kR_k.
$$
We also note that the maps are injective. 
From the gradient bounds, one derives $C^\infty_{\text{loc}}$-bounds for the sequence $(\wt{\phi}_k)$. \
Since  the maps $\wt{\phi}_k$ are, by construction,  locally bounded in k, we find by  Ascoli-Arzela's theorem a converging subsequence $\wt{\phi}_k\to \phi $ in $C_{\textrm{loc}}^\infty (\C)$. The limit map $\phi (z)=(a, \theta):\C\to \R\times S^1$ is  a non-constant injective  holomorphic map satisfying 
$$
\text{$ \abs{ \nabla \phi (0) }=1$}.
$$
We may view ${\mathbb R}\times S^1$    as $S^2\setminus\{0,\infty\}$ conformally. Using the removable singularity theorem
we obtain the non-constant holomorphic map
$$
\phi:S^2\rightarrow S^2
$$
which misses  at least one point. This contradicts the fact that there is no non-constant holomorphic map $S^2\rightarrow {\mathbb C}$.
 
In case (2), the bubbling off analysis as above,   but replacing the maps $\wt{\phi}_k$ by the maps
$$\wt{\phi}_k(z)=\phi_k \biggl( z_k+\frac{z}{R_k}\biggr)
\quad \text{for $\abs{z}\leq \varepsilon_k R_k$},$$
produces a non-constant  holomorphic map $\phi:\C\to C$  into a connected and compact component $C$ of the noded  Riemann surface  $S$. In view of the removable singularity theorem, we can extend the holomorphic map $\phi$ from $\C$  to a holomorphic map $\wt{\phi}:S^2=\C\cup \{\infty\}\to C$.
Since non-constant  holomorphic maps are open and since $S^2$ is compact, the map $\wt{\phi}$ is surjective onto $C$, and we conclude from the Riemann-Hurwitz formula that the genus  $g(C)$ is equal to $0$ so that $C=S^2$. 
Since the maps $\phi_k$, by definition,  map special  points onto special  points, and since the neck regions do not  contain any  special  points, we conclude  from   the surjectivity of $\wt{\phi}$ and the connectedness of the noded Riemann  surface $(S, j, M, D)$ that  the component $C=S^2$ possesses precisely one special point, namely a nodal point, which is hit by the extension $\wt{\phi}:S^2\to C\equiv S^2$.  Since the number of special points is equal to $1$, we conclude from the stability assumption  of the stable map $(S, j, M, D, u)$ that $\int_Cu^*\omega>0$. Consequently,  we  find two  points $p$ and $q$ in $\C$ such that  their image points  $p'=\phi (p)$ and $q'=\phi (q)$ are different  and  such that $\text{dist}\ (u(p'), u(q'))>0$.
Since $\phi_k$ is an isomorphism, we know that $\oplus_{b_k}(u_k')\circ \phi_k=\oplus_{a_k}(u_k)$. 
Hence at a finite distance to the core regions we have $u'_k\circ \phi_k=u_k$ for $k$ large.  Abbreviating by $\tau_k(z)=z_k+\frac{z}{R_k}$ the maps introduced for the rescaled map $\wt{\phi}_k$, we  compute,  
\begin{equation*}
\begin{split}
0<\text{dist}\ (u(p'), u(q'))&=\text{dist}\ (u\circ \phi (p), u\circ  \phi (q))\\
&=\lim_k \text{dist}\ (u'_k\circ  \phi_k\circ \tau_k (p), u'_k\circ  \phi_k\circ \tau_k (q))\\
&=\lim_k  \text{dist}\ (u_k\circ \tau_k (p), u_k\circ \tau_k (q))\\
& =\text{dist} (u(x), u(x))=0,
\end{split}
\end{equation*}
which again is a contradiction. 
 
Consequently, we  have proved that  the sequence $(z_k)$ necessarily stays  at a  finite distance to  the core region of $S_k$. Since there are only finitely many connected components of $S_k$, we may assume that the sequence $(z_k)$ stays at a finite distance to the core region of the connected and compact component $C$ of the noded Riemann surface $S$.  Again we distinguish the two cases (1) and (2). 
Arguing as before, both cases lead to contradictions and we have proved the claim  that bubbling off in the sequence $(\phi_k)$ of isomorphisms does not occur.

Therefore,  one concludes  by Gromov--compactness that any subsequence of $(\phi_k)$ possesses a subsequence which converges in $C^\infty_{\text{loc}}$ away from the nodes to some automorphism 
$\phi_0:(S,j, M, D, u)\to (S,j, M, D, u)$ of the stable map as claimed in Lemma \ref{GROMOVCONV}.
\end{proof}

\section{Linearization of the CR-Operator}\label{connection}
In order to linearize the Cauchy-Riemann operator we first introduce some notation. If 
$\nabla$ is a connection of the vector bundle $E\to Q$ over the smooth manifold $Q$, then the covariant derivative $\nabla_Xs$  of the  section $s$ in the direction of the vector field $X$ on $Q$ is the section, which in local coordinates is given by 
$$
\nabla_Xs(x)=Ds (x)\cdot X(x) + \Gamma(x)(X(x),s(x)).
$$
Here, $Ds$ is the derivative of $s$ and $\Gamma$ is the local connector associated with $\nabla$. In the case of the tangent bundle $E=TQ\to Q$, the torsion $N$  of the connection $\nabla$ is, in local coordinates given by 
$$
N(x)(X,Y)=\Gamma(x)(X,Y)-\Gamma(x)(Y,X).
$$
If $u:S\to Q$ is a smooth map, the pull-back connection  $u^*\nabla$ on the pull-back bundle $u^*TQ\to S$ is defined as follows. The covariant derivative $(u^*\nabla)_X\eta$  of a section $\eta$ of the pull-back bundle (hence satisfying $\eta (z)\in T_{u(z)}Q$) in the direction of the vector field $X$ on $S$, is the section, which in local coordinates is given by 
$$
(u^*\nabla )_X\eta (z)=D\eta (z)\cdot X(z)+\Gamma (u(z))(Du(z)\cdot X(z), \eta (z)).
$$
Here $D\eta$ and $Du$ are the derivatives of $\eta$, respectively $u$, in local coordinates. 

Connections on the tangent bundle $TQ\to Q$ induce natural connections on function spaces associated with the manifold $Q$, as is explained in 
\cite{El}. We consider the function space of functions $u:S\to Q$ having sufficient regularity and denote by 
$$\partial : u\mapsto Tu$$
the map associating with the map $u$ its tangent map $Tu$, which is a section of the vector bundle $T^*S\otimes u^*(TQ)\to S$ so that $Tu(z)\in \call (T_zS, T_{u(z)}Q)$ for all $z\in S$. Let $\eta$ be a section of the pull-back bundle $u^*TQ\to S$. The covariant derivative $(\nabla_\eta \partial )(u)$ of the map $\partial$ at the point $u$ in the direction of $\eta$  is the section of $T^*S\otimes u^*(TQ)\to S$, in local coordinates given by 
$$(\nabla_\eta \partial )(u)(z)=D\eta (z)+\Gamma (u(z))\bigl(\eta (z), Du(z)\cdot \bigr).$$
Since, for the fixed map $u:S\to Q$ the covariant derivative $(u^*\nabla)_{(\cdot)}\eta$ of the section $\eta$ of $u^*TQ\to S$ is the  section of 
$T^*S\otimes u^*(TQ)\to S$ which in local coordinates is given by 
$$(u^*\nabla)_{(\cdot)}\eta (z)=D\eta(z)\cdot   +\Gamma (u(z))\bigl(Du(z)\cdot , \eta (z)\bigr),$$
we obtain the formula 
\begin{equation}\label{covariant1}
(\nabla_\eta\partial)(u)=(u^*\nabla )_{(\cdot)}\eta+N(\eta, Tu\cdot ).
\end{equation}
Let us simplify the notation by setting for a vector field $\eta$ along $u$, since $u$ is usually fixed,
$$
\nabla\eta := (u^\ast\nabla)_{(\cdot)}\eta.
$$
Note that $(\nabla\eta)(z)$ for $z\in S$ is a linear map $T_zS\rightarrow T_{u(z)}Q$.
We now assume that $(Q,\omega)$ is our  symplectic manifold  equipped with the compatible almost complex structure $J$. We choose a connection $\nabla$ on $TQ\to Q$ which satisfies  $\nabla J=0$. 
For fixed $(a, v)$ we consider the space of maps $u:S_a\to Q$ and compute the  covariant derivative $(\nabla_\eta \ov{\partial}_{J, j(a, v)})(u)$ of the Cauchy-Riemann map 
$$u\mapsto  \ov{\partial}_{J, j(a, v)}(u)=\frac{1}{2}\bigl[ Tu+J(u)\circ Tu \circ j(a, v)\bigr]$$
at the point $u$ in the  direction of the  section $\eta$ along $u$. In view of $\nabla J=0$ and the formula \eqref{covariant1} we obtain the following formula  for the section 
$(\nabla_\eta \ov{\partial}_{J, j(a, v)})(u)$  of the  vector bundle $T^*S\otimes u^*(TQ)\to S$,
\begin{equation}\label{covariant2}
\begin{split}
(\nabla_\eta \ov{\partial}_{J,j(a,v)})(u)&=\frac{1}{2} \bigl[ \nabla\eta  + J(u)\circ  (\nabla\eta) \circ j(a, v)\cdot \bigr]  \\
&\phantom{=}+\frac{1}{2}\bigl[ N(\eta ,Tu \cdot )  +J(u)\circ N(\eta,Tu\circ  j(a,v)\cdot )\bigr].
\end{split}
\end{equation}
The Cauchy-Riemann type operator in \eqref{covariant2} is in local coordinates of the form 
\begin{equation}\label{covariant2_b}
D\eta (z)+J(u(z))\circ D\eta (z)\circ  j(a, v)+A(z)(\eta (z), Tu(z)).
\end{equation}
Here $A(z)$ is a bilinear map.  Fixing the point $z_0\in S_a$, we choose a smooth family $\varphi_v:(D, i)\to ({\mathcal D}, j(a, v))$ of biholomorphic mappings which map $0\in \C$ onto $z_0$. 
Taking the composition $\xi=\eta\circ \varphi_v:D\to \R^{2n}$ and evaluating the above operator at the vectors $T\varphi_v \cdot \frac{\partial}{\partial s}$ we obtain the following local expression for 
$\nabla_\eta \ov{\partial }_{J,j (a, v)}(u)$, 
$$\frac{\partial }{\partial s}\xi+J(u\circ \varphi_v)\cdot \frac{\partial}{\partial t}\xi+A(\varphi_v)(\xi, \frac{\partial}{\partial s}\bigl( u\circ \varphi_v )\bigr) .$$
Let us apply the above considerations to our filled section ${\bf f}:O\to \cf$ of Section \ref{section4.4} which is defined by 
\begin{equation*}
\begin{aligned}
\Gamma [\exp_{u_a}(\oplus_a(\eta)), u_a]\circ \wh{\oplus}_a{\bf f}(a, v,\eta)\circ \sigma (a, v)&=\ov{\partial}_{J, j(a, v)}(\exp_{u_a}\bigl(\oplus_a(\eta))\bigr)\\
\wh{\ominus}_a{\bf f}(a, v, \eta)&=\ov{\partial}_0\bigl(\ominus_a(\eta)\bigr).
\end{aligned}
\end{equation*}
For fixed $(a, v)$ we take the covariant derivative with respect to the section $\eta$ at the section $k$ in the direction of the section $h$ of the bundle $u^*TQ\to S$ and obtain 
\begin{equation*}
\begin{split}
&\bigl(\nabla_{(T\exp_{u_a}(\oplus_a (k))\cdot \oplus_a(h),0)}\Gamma\bigr)[ \exp_{u_a}(\oplus_a(k)), u_a]\circ (\wh{\oplus}_a{\bf f}(a, v, k))\circ \sigma (a, v)\\
&\phantom{===}+\Gamma[\exp_{u_a}(\oplus_ak), u_a]\circ (\wh{\oplus}_aD_3{\bf f}(a, v, k)\cdot h)\circ \sigma (a, v)\\
&\phantom{==}=(\nabla_{(T\exp_{u_a}(\oplus_ak )\cdot \oplus_ah)}\ov{\partial}_{J, j(a, v)})(\exp_{u_a}(\oplus_a k))
\end{split}
\end{equation*}
and 
\begin{equation*}
\wh{\ominus}_a(D_3{\bf f}(a, v, k)\cdot h)=\ov{\partial}_0(\ominus_a (h)).
\end{equation*}
We note that 
$T\exp_{u_a}(\oplus_a(k))\cdot \oplus_a (h)$  is a section of $[\exp_{u_a}(\oplus_a (k))]^*TQ$.

\section{Consequences of  Elliptic Regularity}\label{appx19}
In the following we first  apply the  elliptic regularity theory to a family of Cauchy-Riemann operators on a fixed domain. In sharp contrast to this ``classical case'' we then apply the  elliptic regularity theory 
in the ``noded version'' to a family of Cauchy-Riemann operators  {\bf on  varying  domains}. \\[0.5ex]

\noindent{\bf Classical Case}\\
 We denote by  $D$ the closed unit disk in $\C$ and by $D_{\delta}\subset D$ the concentric subdisks $D_\delta =\{ z\in D\vert, \, \abs{z}\leq 1- \delta\}$ for $0<\delta<1$.  By $(J_k)$ we denote  
a sequence of complex multiplications on $\R^{2n}$ smoothly parametrized by $D$, i.e.,  $J_k(z)$ for $z\in D$ is a complex structure smoothly depending on $z$. We assume that $J_k\rightarrow J$
in $C^{\infty}(D)$  so that  $J$ is a again a smooth family of complex structures on $\R^{2n}$.  We also assume that $J(z)$ is close to $i$
so that in particular  for $z\in\partial D$ the space $({\mathbb R}\oplus \{0\})\oplus\ldots \oplus ({\mathbb R}\oplus \{0\})$ is totally real for $J(z)$.
 Moreover,  we consider a sequence $(A_k)$ of smooth maps 
$A_k:D\to \call (\R^{2n}, \R^{2n})$ which converge to the map $A:D\to \call (\R^{2n}, \R^{2n})$ in the $C^\infty$-topology. We assume that 
$(y_k)$ is a sequence in $H^{2+m}(D,\R^{2n})$ converging to $0$, 
$$y_k\to 0\quad \text{in $H^{2+m}(D,\R^{2n})$},$$
 and we assume that the sequence 
$(z_k)\subset H^{3+m}(D,{\mathbb R}^{2n})$ is  bounded. Finally,  we assume that the sequence $(u_k)$ is bounded  in  $H^{3+m}(D,\R^{2n})$ and satisfies the equations 
\begin{equation}\label{bond}
\partial_su_k +J_k\partial_t u_k + A_k u_k =y_k+z_k.
\end{equation}
\begin{proposition}\label{appx20}
Under the above assumptions,  every $u_k$ splits as $u_k=v_k+w_k$ where the sequence  $(w_k)\subset H^{m+3}(D,\R^{2n})$
converges to $0$ in $H^{m+3}(D,{\mathbb R}^{2n})$ and  the sequence  $(v_k)$,  if restricted to $D_\delta$,  is bounded in 
$H^{m+4}(D_\delta,\R^{2n})$ for every $0<\delta<1$. In particular, the sequence $(u_k)$ possesses  a subsequence which,  restricted to $D_\delta$,  converges  in
$H^{m+3}(D_\delta,\R^{2n})$.
\end{proposition}
\begin{proof}
We introduce the subspace  $W\subset H^{m+3}(D,{\mathbb R}^{2n})$ of maps satisfying the  boundary conditions 
 $u(1)=0$ and $u(z)\in ({\mathbb R}\oplus \{0\})\oplus\ldots \oplus ({\mathbb R}\oplus \{0\})$ for all $z\in \partial D$. Then the linear operator
$$
L:W\rightarrow H^{m+2}(D,{\mathbb R}^{2n}),\quad u\mapsto u_s +J u_t +Au
$$
is a Fredholm operator of index $0$ whose spectrum is a countable discrete set of points with no finite limit point.  The same holds true for the linear operator 
$$L_k:W\to H^{2+m}(D, \R^{2n}),\quad u\mapsto u_s+J_ku_t+A_ku$$
and there exists an $\varepsilon>0$ such that  for $k$ sufficiently large the operators $L_k+\varepsilon $ are linear isomorphisms. Moreover, 
$\norm{(L_k+\varepsilon )^{-1}}\leq C$ for all $k\geq k_0$, with  a constant $C$ independent of $k$.
For $k\geq k_0$, we  let $w_k\in W$ be the unique  solution of 
$$
(L_k +\varepsilon )w_k = y_k.
$$
Since $y_k\rightarrow 0$ in $H^{m+2}(D,\R^{2n})$, we conclude from 
$$\norm{w_k}_{H^{m+3}(D,{\mathbb R}^{2n})}=\norm{(L_k+\varepsilon )^{-1}(y_k)}_{H^{m+3}(D,{\mathbb R}^{2n})}\leq C\norm{y_k}_{H^{m+2}(D,{\mathbb R}^{2n})}$$
that $w_k\rightarrow 0$ in $H^{m+3}(D,{\mathbb R}^{2n})$.
Now we set $v_k=u_k-w_k$ and note that  the sequence $(v_k)$ is bounded in $H^{m+3}(D,{\mathbb R}^{2n})$ in view of the assumption  on the sequence $(u_k)$. We compute,  using \eqref{bond}, 
$$
L_k(v_k) = L_k(u_k)-L_k(w_k) = (y_k+z_k)-(y_k -\varepsilon w_k)= z_k+\varepsilon w_k =:z_k^{\ast}.
$$
Since $w_k\rightarrow 0$ in $H^{m+3}(D,{\mathbb R}^{2n})$ and since,  by assumption,   the sequence $(z_k)$ is bounded in $H^{m+3}(D,{\mathbb R}^{2n})$, the sequence $(z_k^*)$  is also  bounded in $H^{m+3}(D,{\mathbb R}^{2n})$. Now using that $(v_k)$ is bounded in $H^{m+3}(D,{\mathbb R}^{2n})$, 
we conclude by the interior elliptic regularity theory  that  the sequence $(v_k)$ is bounded in the space  $H^{m+4}_{loc}(\textrm{int} (D),{\mathbb R}^{2n})$. This implies the desired conclusion of Proposition \ref{appx20}.
\end{proof}

\noindent{\bf Noded Version}\\
 We assume that we  are given two smooth maps 
$$
[0,\infty)\times S^1\rightarrow \call ({\mathbb R}^{2n},\R^{2n}),\quad  (s,t)\mapsto  J^+(s,t)
$$
and 
$$
(-\infty,0]\times S^1\rightarrow \call ({\mathbb R}^{2n},\R^{2n}),\quad  (s',t')\mapsto  J^-(s',t')
$$
where  $J^+(s,t)$ and $J^-(s',t')$ are complex structures on ${\mathbb R}^{2n}$ for every $(s,t)$ and $(s',t')$. Moreover, we assume that there is a complex structure $J_0$  on $\R^{2n}$ such that 
$$J^+(s,t)\to J_0\qquad  \text{and}\qquad J^-(s',t)\to J_0$$
uniformly in $t$ as $s\to \infty$ and $s'\to -\infty$.  In addition,  we assume
that for every multi-index $\alpha$ and  every  positive number $\varepsilon$,  there exists a number $s_0>0$ such that 
\begin{equation*}
\begin{aligned}
&\abs{D^{\alpha} (J^+(s,t)-J_0)}\leq \varepsilon&\quad &\text{for $s\geq s_0$}\\
&\abs{D^{\alpha} (J^-(s',t')-J_0)}\leq \varepsilon&\quad &\text{for $s'\leq -s_0.$}
\end{aligned}
\end{equation*}

Further, we assume that  the sequence $(a_k)$ of gluing parameters in $\hb \setminus \{0\}$ converges to $0$,
$$a_k\to 0\qquad \text{as $k\to \infty$}.$$
By $Z_{a_k}$ we denote the associated glued finite cylinders introduced in Section \ref{dm-subsect}
and by $Z_{a_k}(-r)$ the sub-cylinder 
$$Z_{a_k}(-r)=\{[ s,t]\in Z_{a_k}\vert \, r\leq s\leq R_k-r\},$$
where, as usual, $R_k=e^{\frac{1}{\abs{a_k}}}-e$.

In addition, we let $z\mapsto  J_k(z)$ be a smooth family of complex structures on ${\mathbb R}^{2n}$ parametrized by $z\in Z_{a_k}$.  We assume that the sequence $(J_k)$ converges to   $J=(J^+, J^-)$ on $Z_0=({\mathbb R}^+\times S^1)\cup ({\mathbb R}^-\times S^1)$ in the following sense. Recall that on every cylinder $Z_{a_k}$ we have  the two  distinguished   canonical holomorphic coordinates
$(s,t)\in [0,R_k]\times S^1$ and $(s',t')\in [-R_k,0]\times S^1$. 
Then,  
\begin{itemize}
\item[$\bullet$] given $m\geq 0$ and $\varepsilon>0$ and $r>0$, there  exists  $k_0$ such that
\begin{equation}\label{dalpha1}
\abs{D^{\alpha}(J_k([s,t])-J_0)}<\varepsilon
\end{equation}
for all $[s,t]\in Z_{a_k}(-r)$, all multi-indices  $\alpha$ satisfying $\abs{\alpha}\leq m$, and all $k\geq k_0$.    
\item[$\bullet$]
Moreover, for every $r>0$ and $m\geq 0$,  
\begin{align*}
&\text{ $J_k([s,t])\rightarrow J^+(s,t) \quad $  in $C^m([0,2r]\times S^1)$}\\
&\text {$J_k([s',t'])\rightarrow J^-(s',t') \quad $  in $C^m([-2r,0]\times S^1)$ }
\end{align*}
as $k\to \infty$.
\end{itemize}

With the sc-Hilbert  spaces  $E$ and $F$ of pairs $h=(h^+, h^-)$ of maps 
$h^\pm:\R^\pm \times S^1\to \R^{2n}$ defined in Section 2.4 
the following result holds true.
\begin{proposition}\label{ELLIPTIC-X}
Fix  a level $m\geq 0$.  We assume  that $(y_k)$ and  $(y_k')$  are sequences  in $F_m$  which converge   to $0$ and $(z_k)$, $(z_k')$ are  two  bounded sequences  in $F_{m+1}$.  Moreover, we assume that $(h_k)$, where $h_k=(h_k^+, h_k^- )$,  is   a bounded sequence in $E_m$ satisfying the equations
\begin{equation}\label{contbound1}
\begin{aligned}
\partial_s (\oplus_{a_k}(h_k))+ J_k\partial_t (\oplus_{a_k}(h_k))&=\wh{\oplus}_{a_k}(y_k)+\wh{\oplus}_{a_k}(z_k)\\
\ov{\partial}_0(\ominus_{a_k}(h_k)) &= \wh{\ominus}_{a_k}(y_k')+\wh{\ominus}_{a_k}(z_k').
\end{aligned}
\end{equation}
Then the following statement holds true. If $\alpha:{\mathbb R}^+\rightarrow \R$  is a   smooth function  which vanishes near $0$ and   is equal to $1$ outside of a compact set,  and if   the functions $\alpha^\pm:\R^\pm\to \R$ are  defined by  $\alpha^+(s)=\alpha(s)$  for $s\geq 0$ and $\alpha^-(s')=\alpha(-s')$ for $s'\leq 0$,  then the sequence 
$$(\alpha^+h^+_k, \alpha^-h^-_k)\in E_m$$
possesses a subsequence  which converges in $E_m$. 
\end{proposition}
\begin{proof}
We take  a  function $\alpha:\R^+\to \R$ having the properties as in the statement of the lemma. Since $\alpha$ is equal to $0$ near $0$, we extend $\alpha$  by $0$ on the negative half-line. By assumptions, there are two numbers $0<r_0<r_1$ such that 
$\alpha (s)=0$ for all $s\leq r_0$ and $\alpha (s)=1$ for all $s\geq r_1$.
We shall show that the sequence $(\alpha^+ h_k^+, \alpha^- h^-_k)$ possesses  a convergent subsequence in $E_m$. This will be done in  two steps. In the first step we shall show that   the sequences   $(\alpha^+h_k^+)$  and $(\alpha^-h_k^-)$ possess converging subsequences in  the spaces $H^{m+3}([0,r]\times S^1,{\mathbb R}^{2n})$ and $H^{m+3}([-r,0]\times S^1,{\mathbb R}^{2n})$ respectively for every $r>0$.  In the second step we then  show that  for  large $r>0$  the sequences $(\alpha^+h_k^+)$  and $(\alpha^-h_k^-)$ possess converging subsequences in  $H^{3+m,\delta_m}([r,\infty)\times S^1,{\mathbb R}^{2n})$ and $H^{3+m,\delta_m}((-\infty, -r]\times S^1,{\mathbb R}^{2n})$, respectively. 

In order to carry out the first step of the proof we choose  $r>r_1$ and introduce the function $\gamma:\R\to \R$ by 
$$
\gamma (s)=\alpha^+(s)\cdot \alpha^-(s-2r)
$$
for $s\in \R$.

\begin{figure}[htbp]
\psfrag {1}{$1$}
\psfrag {0}{$0$}
\psfrag {r0}{$r_0$}
\psfrag {r1}{$r_1$}
\psfrag  {r}{$r$}
\psfrag  {r2}{$2r-r_1$}
\psfrag  {r3}{$2r-r_0$}
\psfrag {s}{$s$}
\psfrag {g}{$\gamma (s)$}
\centering
\includegraphics[width=4in]{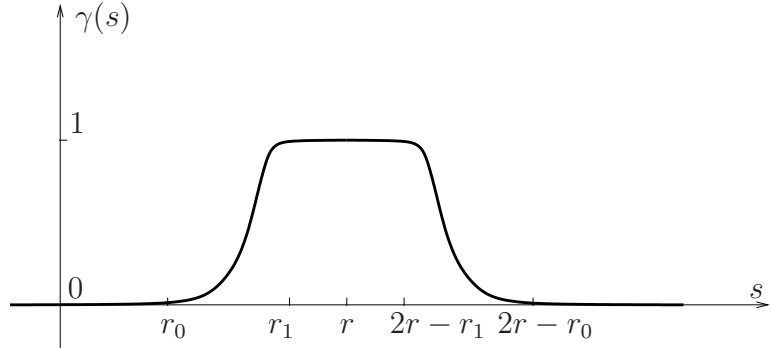}
\caption{The graph of the function $\gamma$.}
\label{Fig13}
\end{figure}

We shall prove  that  $(\gamma h^+_k)$ has a converging subsequence in the space $H^{m+3}([0,2r]\times S^1,{\mathbb R}^{2n})$. Since $\gamma h^+_k=\alpha^+h^+_k$ on $[0,r]\times S^1$, this then  implies that 
the  sequence $(\alpha^+h^+_k)$ has a converging subsequence in $H^{m+3}([0,r]\times S^1,{\mathbb R}^{2n})$.  
{Clearly,  the product rule gives 
\begin{equation*}
\begin{split}
\gamma \cdot \partial_s (\oplus_{a_k}h_k)&=\partial_s (\gamma \oplus_{a_k}h_k)- (\partial_s\gamma)\cdot \oplus_{a_k}h_k\\
&=\partial_s (\gamma \oplus_{a_k}h_k)-\dot{\gamma}\cdot \oplus_{a_k}h_k.
\end{split}
\end{equation*}
}
{Now recalling that  $R_k=e^{\frac{1}{\abs{a_k}}}-e$,  so that $R_k\to \infty$ as $a_k\to 0$, we take $k$  so large that $2r<\frac{R_k}{2}-1$. Observing that for large values of $k$,  
$\oplus_{a_k} (h_k)=h^+_k$ on $[0,2r]\times S^1$, we compute, using   \eqref{contbound1}, 
\begin{equation*}
\begin{split}
&\partial_s ( \gamma h^+_k)+J_k \partial_t ( \gamma h^+_k)\\
&\phantom{====}=\gamma \bigl[ \partial_s h^+_k +J_k\partial_t h^+_k\bigr]+\dot{\gamma}h^+_k\\
&\phantom{====}=\gamma \bigl[ \partial_s (\oplus_{a_k}h_k)  +J_k\partial_t  (\oplus_{a_k}h_k)  \bigr]+\dot{\gamma}h^+_k\\
&\phantom{====}=\gamma \bigl[ \wh{\oplus}_{a_k}(y_k)+ \wh{\oplus}_{a_k}(z_k)  \bigr]+\dot{\gamma}h^+_k\\
&\phantom{====}=\gamma y_k^+  +\gamma  z_k^+ +\dot{\gamma}h^+_k.
\end{split}
\end{equation*}
}
Because  $\gamma (s)=0$ for $s\geq -r_0+2r$ we may  consider the maps $\gamma h^+_k$, $\gamma y^+_k$, $\gamma z^+_k$  and  $\dot{\gamma}h^+_k$ as  equal to $0$ for  $s\geq -r_0+2r$.  
The maps $h^+_k$ are of class $H^{3+m,\delta_m}$ and form a bounded  sequence  in the space $H^{3+m, \delta_m}(\R^+\times S^1, \R^{2n})$. Since they are supported in 
the finite cylinder $[0,2r]\times S^1$, the sequences $ \dot{\gamma}h^+_k$ and  $\gamma z^+_k$ are bounded in $H^{3+m, \delta_{m+1}}(\R^+\times S^1, \R^{2n})$. Hence 
the above equation becomes 
$$\partial_s ( \gamma h^+_k)+J_k \partial_t ( \gamma h^+_k)=\gamma y_k^++\rho_k$$
where the sequence $(\rho_k)$ is bounded in $H^{3+m, \delta_{m+1}}(\R^+\times S^1, \R^{2n})$ and the sequence $(\gamma y_k^+)$ converges in 
 $H^{2+m, \delta_m}(\R^+\times S^1, \R^{2n})$.

Moreover, the  maps $\gamma y^+_k$ and $\rho_k$ are supported in the finite cylinder 
$[0,2r]\times S^1$. 
If $h:\R^+\times S^1\to D\setminus \{0\}$ is the  holomorphic map defined by $h (s, t)=e^{-2\pi (s+it)}$, we denote its inverse  map by $\phi:
D\setminus \{0\}\to \R^+\times S^1$ and define $v_k(z)=\gamma\cdot  h^+_k\circ \phi (z)$. The maps  $v_k$  have their supports  contained in a  fixed closed annulus $A$ and satisfy the equations 
$$\partial_xv_k+\wt{J}_k\partial_y v_k=\ov{y}_k+\ov{\rho}_k$$
on $D\setminus \{0\}$, where the maps $\ov{y}_k$ and $\ov{\rho}_k$ have their supports  in $A$. 
Here the parameterized complex structures $\wt{J}_k$ are equal to $\wt{J}_k(x, y)=J_k(\phi (x, y))$. 
Extending all the maps involved by $0$, we rewrite the  equation as 
$$\partial_xv_k+J_0\partial_y v_k =\ov{y}_k+(J_0-\wt{J}_k)\partial_y v_k+\ov{\rho}_k.$$
In view of our assumption \eqref{dalpha1} on the  sequence  $(J_k)$  and in view of  the fact that the maps $(v_k)$ are supported  in a  fixed closed annulus $A\subset D\setminus \{0\}$, we conclude that 
$$(J_0-\wt{J}_k)\partial_y v_k\to 0$$
in $H^{2+m}(D)$ as $k\to \infty$.
Consequently, the sequence  of maps $\ov{y}_k+(J_0-\wt{J}_k)\partial_y v_k$ converges to $0$ in $H^{2+m}(D)$,  the  sequence $(\ov{\rho}_k)$ is bounded in $H^{3+m}(D)$, and the equation is satisfied on the whole disk $D$. By Proposition \ref{appx20}, the  sequence $(v_k)$ has a converging subsequence in $H^{3+m}(D)$ which in turn implies that $( \gamma h^+_k)$ has a converging subsequence in $H^{3+m, \delta_m}(\R^+\times S^1).$   Hence $(\alpha^+h^+_k)$ has a converging subsequence in $H^{m+3}([0,r]\times S^1,{\mathbb R}^{2n})$.  By  similar arguments,   the sequence $(\alpha^-h^-_k)$ possesses  a converging subsequence in $H^{m+3}([-r,0]\times S^1,{\mathbb R}^{2n})$. This finishes the first step of the proof.

In order to carry out the second step of the proof we introduce  for $\tau>0$ 
the shifted functions $\alpha_\tau^\pm:\R \to \R$ by
$$\alpha_\tau^+(s):=\alpha^+(s-\tau)\quad \text{and}\quad \alpha_\tau^-(s):=\alpha^-(s+\tau)$$
and  define   
for $k$ large the  sequence $(\alpha_k^\tau)$ of  functions $\alpha_k^\tau:[0, R_k]\to \R$  by 
\begin{equation*}
\alpha_k^\tau (s)=
\begin{cases}\alpha^+_\tau (s) \quad &\text{$0\leq s\leq \frac{R_k}{2}$}\\
\alpha^-_\tau (s-R_k)\quad &\text{$\frac{R_k}{2}\leq s\leq R_k$.}
\end{cases}
\end{equation*}

\begin{figure}[htbp]
\psfrag  {1}{\small{$1$}}
\psfrag  {0}{\small{$0$}}
\psfrag {r0}{\small{$r_0+\tau$}}
\psfrag  {r1}{\small{$r_1+\tau$}}
\psfrag  {r}{\small{$\frac{R_k}{2}$}}
\psfrag  {r2}{\small{$R_k-(r_1+\tau)$}}
\psfrag {r3}{\small{$R_k-(r_0+\tau)$}}
\psfrag {s}{\small{$s$}}
\psfrag {g}{\small{$\alpha_k^\tau (s)$}}
\centering
\includegraphics[width=4.1in]{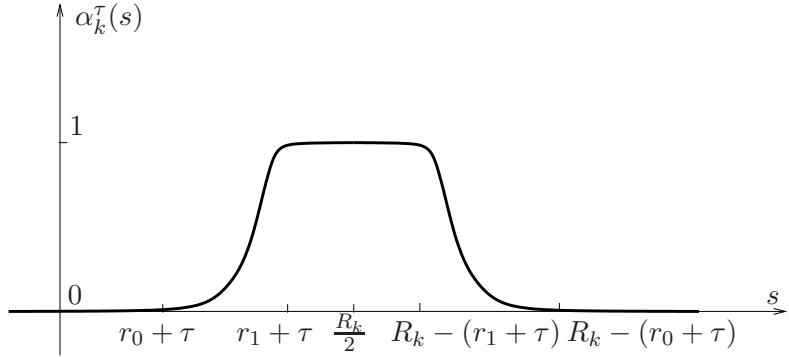}
\caption{The graph of the function $\alpha_k^\tau $.}
\label{Fig14}
\end{figure}

We abbreviate 
$$\wt{h}_k=(\alpha^+_\tau h^+_k, \alpha^-_\tau h^-_k).$$
{Since $\alpha^+_\tau(s)=1$ for $R_k/2\leq s\leq R_k$ and 
$\alpha^-_\tau (s-R_k)=1$ for $0\leq s\leq R_k/2$ for $k$ large, we observe that 
$$\alpha_k^\tau (s)=\alpha^+_\tau (s) \alpha^-_\tau (s-R_k).$$ Moreover, 
$\alpha_\tau^-(s-R_k)\beta_{a_k}(s)=1$ for $s\leq \R_k/2+1$. Hence
$$\alpha_k^\tau(s)\beta_{a_k}(s)=\alpha^+_\tau (s)\alpha^-_\tau (s-R_k)\beta_{a_k}(s)=\alpha^+_\tau (s)\beta_{a_k}(s),$$
and similarly
$$\alpha_k^\tau(s)(1-\beta_{a_k}(s))=\alpha^-_\tau (s-R_k)(1-\beta_{a_k}(s)).$$
Then, using these identities and the formula for the gluing $\oplus_{a_k}$, we find that 
\begin{equation}\label{sigmaeq1}
\alpha_k^\tau \oplus_{a_k}(h_k) = \oplus_{a_k}(\wt{h}_k)
\quad \text{and}\quad \ominus_{a_k}(\wt{h}_k) =\ominus_{a_k}(h_k).
\end{equation}
}
From   the first equations in \eqref{contbound1} and in \eqref{sigmaeq1},  we  deduce 
\begin{equation}\label{contbound2}
\begin{split}
&\partial_s (\oplus_{a_k}(\wt{h}_k))+ {J}_k\partial_t (\oplus_{a_k}(\wt{h}_k))\\
&\phantom{====}=
\partial_s (\alpha_k^\tau \oplus_{a_k}(h_k) )+ {J}_k\partial_t (\alpha_k^\tau \oplus_{a_k}(h_k) )\\
&\phantom{====}=\alpha_k^\tau \bigl[\partial_s(\oplus_{a_k}(h_k))+J_k\partial_ t(\oplus_{a_k}(h_k))\bigr] +\dot{ \alpha}_k^\tau\oplus_{a_k}(h_k)\\
&\phantom{====}=\alpha_k^\tau \bigl[\wh{\oplus}_{a_k}(y_k)+\wh{\oplus}_{a_k}(z_k)\bigr]+\dot{ \alpha}_k^\tau\oplus_{a_k}(h_k)\\
&\phantom{====}=\wh{\oplus}_{a_k}(\alpha_\tau^+y_k^+,\alpha_\tau^-y_k^-)+
\wh{\oplus}_{a_k}(\alpha_\tau^+z_k^+, \alpha_\tau^-z_k^-)+\dot{ \alpha}_k^\tau\oplus_{a_k}(h_k).
\end{split}
\end{equation}
The support of $\dot{\alpha}^\tau_k$ lies in the intervals $[r_0+\tau, r_1+\tau]$ and $[R_k-(r_1+\tau), R_k-(r_0+\tau)]$.
We set  
$$r=r_1+\tau$$
and   introduce for large $k$ the maps 
$g_k^\pm:\R^\pm \times S^1\to \R^{2n}$, defined by $g_k^\pm =\dot{\alpha}^\pm_{\tau}\cdot h^\pm_k$ on $[0,2r]\times S^1$ and $[-2r,0]\times S^1$, respectively, and   $g_k^\pm=0$ for $s\geq 2r$  and $s'\leq -2r$, respectively. 
Since $h_k\in E_m$,  the maps $h_k^\pm$ are of class $H^{3+m}$ so that the functions $g_k=(g_k^+, g^-_k)$ belong to $F_{m+1}$.  Since all  the functions $g_k^\pm$ are supported in the fixed finite  cylinders $[0,2r]\times S^1$ and $[-2r,0]\times S^1$, the sequence  $(g_k)$ is also bounded in $F_{m+1}$. 
Moreover,  it satisfies the equations
\begin{equation*}
\begin{split}
\oplus_{a_k}(g_k)&=\beta_a\dot{\alpha}^+_{\tau}\cdot h^+_k+(1-\beta_a)\dot{\alpha}^-_{\tau}(\cdot -R_k)\cdot h^-_k(\cdot -R_k, t-\vartheta_k)\\
&=\dot{\alpha}_k^\tau\cdot \bigl[\beta_a( h^+_k+(1-\beta_a)\cdot h^-_k(\cdot -R_k, t-\vartheta_k)\bigr]\\
&=\dot{\alpha}_k^\tau\cdot \oplus_{a_k}(h_k).
\end{split}
\end{equation*}
Consequently,  abbreviating  
$$\wt{y}_k=(\alpha_\tau^+ \cdot y_k^+, \alpha_\tau^- \cdot y_k^-)\qquad \text{and}\qquad \wt{z}_k=
(\alpha_\tau^+ \cdot z_k^+, \alpha_\tau^- \cdot z_k^-)+g_k, $$
the equation \eqref{contbound2} becomes
\begin{equation}\label{roppp}
\begin{split}
\partial_s (\oplus_{a_k}(\wt{h}_k))+ {J}_k\partial_t (\oplus_{a_k}(\wt{h}_k))=\wh{\oplus}_{a_k}(\wt{y}_k)+\wh{\oplus}_{a_k}(\wt{z}_k).
\end{split}
\end{equation}
We note that the  sequences  $\wh{\oplus}_{a_k}(\wt{y}_k)$ and  $\wh{\oplus}_{a_k}(\wt{z}_k)$,   as well as the sequence $(\oplus_{a_k}(\wt{h}_k))$,  are supported in the sub-cylinder $Z_{a_k}(-r)$.  Therefore, we may extend them by $0$ outside of the  finite cylinders $Z_{a_k}$ and 
since the $J_k$ are  defined on $Z_{a_k}$ we may consider  the equations in \eqref{roppp} as equations on the whole infinite cylinder. Also, in view of  the properties of $(y_k)$, $(z_k)$, and $(g_k)$, the sequence $(\wt{y}_k)$ converges to $0$ in $F_m$ and $(\wt{z}_k)$ is bounded in $F_{m+1}$.

We choose $\tau$ so large that $r_0+\tau -(r_1-r_0)>0$ and observe that $\alpha_k^{\tau-(r_1-r_0)}=1$ for all $r\leq s\leq R_k-r$.  Moreover, the  support of $\alpha_k^{\tau-(r_1-r_0)}$ is contained in $(0, R_k)$. Since the support of $\oplus_{a_k}(\wt{h}_k)$ is contained in $Z_{a_k}(-r)$, we conclude that 
\begin{equation}\label{split1}
\begin{split}
&\partial_s (\oplus_{a_k}(\wt{h}_k))+ {J}_k\partial_t (\oplus_{a_k}(\wt{h}_k))\\
&\phantom{===}=
\partial_s(\oplus_{a_k}(\wt{h}_k))+  J_0\partial_t (\oplus_{a_k}(\wt{h}_k))+ \bigl(J_k-J_0\bigr)\partial_t(\oplus_{a_k}(\wt{h}_k))\\
&\phantom{===}=\partial_s(\oplus_{a_k}(\wt{h}_k))+  J_0\partial_t (\oplus_{a_k}(\wt{h}_k))+ \alpha_k^{\tau-(r_1-r_0)}\cdot \bigl(J_k-J_0\bigr)\partial_t(\oplus_{a_k}(\wt{h}_k)).
\end{split}
\end{equation}
We extend the maps $\alpha_k^{\tau-(r_1-r_0)}\cdot \bigl(J_k-J_0\bigr)$ by $0$ onto  the rest of the infinite cylinder $Z_{a_k}^*$, which by construction contains the finite cylinder $Z_{a_k}$. 
Therefore, introducing the operators 
$$
B_k:H^{3+m+i,-\delta_{m+i}}(Z_{a_k}^\ast,{\mathbb R}^{2n})\rightarrow H^{2+m+i,-\delta_{m+i}}(Z_{a_k}^\ast,{\mathbb R}^{2n})
$$
by 
\begin{equation}\label{opk1}
B_k := \ov{\partial}_0 +\alpha_k^{\tau-(r_1-r_0)}(J_k-J_0)\partial_t,
\end{equation}
the equation  \eqref{roppp}  becomes  
\begin{equation}\label{contbound0}
 B_k(\oplus_{a_k}(\wt{h}_k)) =\wh{\oplus}_{a_k}(\wt{y}_k)+\wh{\oplus}_{a_k}(\wt{z}_k).
\end{equation}
By Proposition \ref{POLTER},   
 the standard Cauchy-Riemann operator
$$
\ov{\partial}_0:H^{3+m+i,-\delta_{m+i}}(Z_{a_k}^\ast,{\mathbb R}^{2n})\rightarrow H^{2+m+i,-\delta_{m+i}}(Z_{a_k}^\ast,{\mathbb R}^{2n}),
$$
is a surjective Fredholm operator and 
there exists a constant $C>0$ such that 
\begin{equation}\label{es1}
\begin{split}
\frac{1}{C}\nr u-[u]_{a_k}\nr_{3+m+i,-\delta_{m+i}}^\ast&\leq \nr \ov{\partial}_0(u-[u]_{a_k})\nr_{2+m+i,-\delta_{m+i}}^\ast\\
&\leq C\nr u-[u]_{a_k}\nr^\ast_{3+m+i,-\delta_{m+i}}
\end{split}
\end{equation}
for $i=0, 1$ and for all $u\in H^{3+m+i,-\delta_{m+i}}(Z_{a_k}^\ast,{\mathbb R}^{2n})$. 
We will need an estimate for $\nr \alpha_k^{\tau-(r_1-r_0)}(J_k-J_0)\partial_t (u-[u]_{a_k})\nr_{2+m+i,-\delta_{m+i}}^\ast$ for $k$ sufficiently  large.

\begin{lemma}\label{twoadd1}
\mbox{}
\begin{itemize}\label{addition1}
\item[{\em (1)}] Given $\varepsilon>0$, there is an integer $k_0$ such that  
$$\nr \alpha^{\tau-(r_1-r_0)}_k(J_k-J_0)(\partial_t u)\nr_{2+m+i,-\delta_{m+i}}^\ast\leq \varepsilon\nr u\nr_{3+m+i,-\delta_{m+i}}^*$$
for all $u\in H^{3+m+i,-\delta_{m+i}}(Z_{a_k}^*)$ and all $k\geq k_0$.
\item[{\em (2)}] 
There  exists $k_0$ such that for all $k\geq k_0$ the operators 
$$B_k:H^{3+m,-\delta_m}(Z_{a_k}^*, \R^{2n})\to H^{2+m,-\delta_m}(Z_{a_k}^*, \R^{2n})$$
are surjective Fredholm operators whose kernels consist of constant functions.  
\end{itemize}
\end{lemma}
\begin{proof}
(1)\, 
In view of our assumption on the sequence of almost complex structures, given $\varepsilon>0$, there exists $k_0$ such that 
\begin{equation}\label{add01}
\abs{D^\beta (J_k(s, t)-J_0)}\leq \varepsilon
\end{equation}
for all $s$ satisfying $r'_0:=r_0+\tau-(r_1-r_0)\leq s\leq R_k-r_0'$, for all $k\geq k_0$, and for all multi-indices $\beta$ satisfying  $\abs{\beta}\leq 3+m$.
We abbreviate by  $\Sigma_k$ the finite cylinders 
$$\Sigma_k=[ r_0', R_k-r_0']\times S^1.$$
Since 
$\alpha^\tau_k (s)=0$  for $s\leq r_0'$ and for $s\geq R_k-r_0'$ and the derivatives of $\alpha^\tau_k$ are bounded by a  constant independent of $k$ and $\tau$,  the square of the norm $\nr\alpha_k^{\tau-(r_1-r_0)} (J_k-J_0)(\partial_t u)\nr_{2+m+i,-\delta_{m+i}}^\ast$ is  bounded above by a constant $C$  (independent of $k$ and $\tau$) times 
the sum of  integrals of the form
\begin{equation*}
\int_{\Sigma_k}\abs{D^\beta (J_k(s, t)-J_0)}^2\abs{D^\gamma (\partial_tu)}^2e^{-2\delta_{m+i}\abs{s-\frac{R_k}{2}}}\ ds dt
\end{equation*}
where the multi-indices $\beta$ and $\gamma$ satisfy   $\abs{\beta}+\abs{\gamma}\leq 2+m+i$.
Using  \eqref{add01} we conclude that 
\begin{equation*}
\begin{split}
&\int_{\Sigma_k}\abs{D^\beta (J_k(s, t)-J_0)}^2\abs{D^\gamma (\partial_tu)}^2e^{-2\delta_{m+i}\abs{s-\frac{R_k}{2}}}\ ds dt\\
&\leq \varepsilon^2 \int_{\Sigma_k}\abs{D^\gamma (\partial_tu)}^2e^{-2\delta_{m+i}\abs{s-\frac{R_k}{2}}}\ ds dt
\leq  \varepsilon^2 ( \nr u\nr_{3+m+i,-\delta_{m+i}}^\ast )^2.
\end{split}
\end{equation*}
Consequently, 
$$\nr \alpha_k^{\tau-(r_1-r_0)} (J_k-J_0)(\partial_t u)\nr_{2+m+i,-\delta_{m+i}}^\ast\leq C\varepsilon'  \nr u\nr_{3+m+i,-\delta_{m+i}}^\ast.$$
This finishes the prove of the part (a) of the lemma.\\
(2)\,  By Proposition \ref{POLTER}, the operator  
$$
\ov{\partial}_0:H^{3+m+i,-\delta_{m+i}}(Z^*_{a_k})\to H^{2+m+i, -\delta_{m+i}}(Z^*_{a_k})
$$
 is a Fredholm operator whose index is equal to 
 $\text{ind} (\ov{\partial}_0)=2n$. Hence there exists $\varepsilon>0$ such that  if  $M:H^{3+m+i,-\delta_{m+i}}(Z^*_{a_k})\to H^{2+m,-\delta_m}(Z^*_{a_k})$ is a bounded linear operator whose operator norm is less than $\varepsilon$, then the  operator $\ov{\partial}_0+M$ is a  Fredholm operator having index  equal to $\text{ind}( \ov{\partial}_0+M)
 =\text{ind}( \ov{\partial}_0)=2n$.  In particular, introducing the operators   $M_k:H^{3+m+i,-\delta_{m+i}}(Z^*_{a_k})\to H^{2+m+i,-\delta_{m+i}}(Z^*_{a_k})$ by 
 $$M_k u:= \alpha_k^{\tau-(r_1-r_0)} (J_k-J_0)\partial_t (u),$$
 part (1) implies that for $k$ large  the operators $B_k=\ov{\partial}_0+M_k$ are Fredholm and $\text{ind} ( B_k)=2n$. If, in addition $\varepsilon <\frac{1}{2C}$ where $C$ is the constant in the estimate  \eqref{es1}, we conclude from  \eqref{es1} and  \eqref{opk1} that 
 \begin{equation}\label{es2}
\begin{split}
\frac{1}{2C}\cdot \nr u-[u]_{a_k}\nr_{3+m+i,-\delta_{m+i}}^\ast&\leq \nr B_k(u-[u]_{a_k})\nr_{2+m+i,-\delta_{m+i}}^\ast\\
&\leq 2C\cdot \nr u-[u]_{a_k}\nr^\ast_{3+m+i,-\delta_{m+i}}
\end{split}
\end{equation}
for $k$ sufficiently large. 
If $u$ belongs to the kernel of $B_k$, then the the above estimate implies that $u=[u]_{a_k}$ and since the constants belong to the  kernel $\ker (B_k)$ we conclude that the kernel consists of the constant functions. Since the Fredholm  index of $B_k$ is equal to the dimension of the constant functions, namely to  $\R^{2n}$, the operator $B_k$ is surjective. 
 This finishes the proof of the part (2) and hence  the proof of Lemma \ref{twoadd1}
\end{proof}
We continue with the  proof of  Proposition \ref{ELLIPTIC-X}.  
Using the second equation in \eqref{contbound1} and the second identity in \eqref{sigmaeq1}, we obtain 
\begin{equation}\label{contbound4}
\ov{\partial}_0(\ominus_{a_k}(\wt{h}_k) ) = \wh{\ominus}_{a_k}(y_k')+\wh{\ominus}_{a_k}(z_k').
\end{equation}

In view of Lemma \ref{twoadd1}  and Proposition  \ref{c-iso},  
we find   a unique map $\wh{f}_k$ belonging to $H^{3+m,-\delta_m}(Z^{\ast}_{a_k}, \R^{2n})$ having vanishing mean value,  $[\wh{f}_k]_{a_k}=0$,   and  a unique map $f_k\in H_c^{3+m,\delta_m}(C_{a_k}, \R^{2n})$ such that 
\begin{equation}\label{contbound10}
B_k (\wh{f}_k)=\wh{\oplus}_{a_k}(\wt{y}_k)\quad \text{and}\quad \ov{\partial}_0(f_k)=\wh{\ominus}_{a_k}(y_k').
\end{equation}
The map $\wh{f}_k$, if restricted to the finite cylinder  $Z_{a_k}$,  belongs to the space $H^{3+m}(Z_{a_k}, \R^{2n})$ so that the pair $(\wh{f}_k, f_k)$ belongs to the sc-space $G^{a_k}$ introduced in Section \ref{preliminaries}. 
By Theorem \ref{sc-splicing-thm},  the total gluing map $\boxdot_a:E\to G_a$ is an sc-isomorphism. Therefore, 
there exists a pair $h_{0,k}=(h_{0,k}^+, h_{0,k}^-)\in E_m$  satisfying 
$$\oplus_{a_k}(h_{0,k})=\wh{f}_k\quad \text{and}\quad \ominus_{a_k}(h_{0,k})=f_k.$$
Consequently, 
\begin{equation}\label{h0jeq}
B_k(\oplus_{a_k}(h_{0,k}) )=\wh{\oplus}_{a_k}(\wt{y}_k)\quad \text{and}\quad  \ov{\partial}_0( \ominus_{a_k}(h_{0,k}) )=\wh{\ominus}_{a_k}(y_k').
\end{equation}
In addition,  the averages $[h_{0,k}]_{a_k}$ are equal to 
\begin{equation*}
\begin{split}
[h_{0,k}]_{a_k}&=\frac{1}{2}\int_{S^1}\bigl[ h^+_{0,k}(R_k/2, t)+h^-_{0,k}(-R_k/2, t)\bigr]\ dt\\
&=\int_{S^1}\oplus_{a_k}(h_{0,k})(R_k/2, t)\ dt=
\int_{S^1}\wh{f}_k(R_k/2, t)\ dt=[\wh{f}_k]_{a_k}=0.
\end{split}
\end{equation*}
By similar arguments, we find  a sequence  $(h_{1,k})\subset  E_{m+1}$ satisfying 
\begin{equation}\label{h1jeq}
B_k(\oplus_{a_k}(h_{1,k}) )=\wh{\oplus}_{a_k}(\wt{z}_k)\quad \text{and}\quad  \ov{\partial}_0( \ominus_{a_k}(h_{1,k}) )=\wh{\ominus}_{a_k}(z_k')
\end{equation}
having vanishing averages, $[h_{1,k}]_{a_k}=0$.  We claim that 
$$\wt{h}_k=([\wt{h}_k]_{a_k}, [\wt{h}_k]_{a_k})+h_{0,k}+h_{1,k}.$$
Indeed, abbreviating  
$\rho_k=\wt{h}_k-([\wt{h}_k]_{a_k}, [\wt{h}_k]_{a_k})-h_{0,k}-h_{1,k}, $
we have  $[\rho_k]_{a_k}=0$ and, in view of equations  \eqref{roppp}, \eqref{h0jeq}, and \eqref{h1jeq}, we obtain 
\begin{equation}\label{h3jeq}
B_k(\oplus_{a_k}(\rho_k))=0\quad \text{and}\quad  \ov{\partial}_0( \ominus_{a_k}(\rho_k ))=0.
\end{equation}
From the second equation and Proposition \ref{c-iso} 
we conclude that  $\ominus_{a_k}(\rho_k )=0$. In view of the definition of  the anti-gluing $\ominus_{a_k}$, we conclude that $\rho_k^+(s, t)=0$ for $s\geq \frac{R_k}{2}+1$ and $\rho_k^-(s', t')=0$ for $s'\leq -\frac{R_k}{2}-1.$ Consequently, it follows from the definition of the gluing that $\oplus_{a_k}(\rho_k )(s, t)=0$ on  the finite cylinders $[0, R_k/2-1]\times  S^1]$ and $[R_k/2+1, R_k]\times S^1$ so that we can extend  the map $\oplus_{a_k}(\rho_k )$  by $0$ beyond the finite cylinders $Z_{a_k}$.  Hence $\oplus_{a_k}(\rho_k )\in H^{3+m,-\delta_m}(Z_{a_k}^\ast,{\mathbb R}^{2n})$ and 
$B_k(\oplus_{a_k}(\rho_k )=0$.  Now  Proposition \ref{POLTER} 
implies that $\oplus_{a_k}(\rho_k )=0$.  Since by Theorem \ref{sc-splicing-thm} the total gluing map $\boxdot_a:E\to G_a$ is an sc-isomorphism,  we conclude that indeed $\rho_k=0$  and our claim is proved.

 From now on we view  the maps $\oplus_{a_k}h_{i,k}$, $i=0,1$,   as defined on  the finite cylinders  $Z_{a_k}$. We shall use extensively the estimates for the global gluing constructions in  Proposition \ref{HAT_HAT} 
 and  Proposition  \ref{total-gluing}.
{First, recalling the $\wh{G}^{a_k}$-norm introduced in 
Section \ref{preliminaries}, we have the following estimates
\begin{align*}
e^{\delta_m \frac{R_k}{2} }\nr \wh{\oplus}_{a_k}(\wt{y}_k)\nr_{m+2,-\delta_m}
\leq |\wh{\boxdot}_a(\wt{y}_k) |_{\wh{G}^a_m}
\intertext{and}
e^{\delta_m\frac{R_k}{2} }\nr \wh{\ominus}_{a_k}(y_k) \nr_{m+2,\delta_m}
\leq |\wh{\boxdot}_a(y_k) |_{\wh{G}^a_m}.
\end{align*}
Hence, in view of Proposition \ref{HAT_HAT}, it  follows that 
 \begin{equation}\label{hatestimates1}
 e^{\delta_m\frac{R_k}{2}}\nr \wh{\oplus}_{a_k}(\wt{y}_k)\nr_{m+2,-\delta_m}\to 0\quad \text{and}\quad e^{\delta_m\frac{R_k}{2}}\nr \wh{ \ominus}_{a_k}(y_k)\nr_{m+2,\delta_m}\rightarrow 0.
\end{equation}
since the sequences  $(\wt{y}_k)$ and $(y_k)$ converge to $0$ in $F_m$.}
{
Now,  $[h_{0,k}]_{a_k}=0$ implies that $[\wh{\oplus}_{a_k}(h_{0,k}]_{a_k}=0$ and we conclude from \eqref{h0jeq}, using \eqref{es2}, \eqref{hatestimates1},   and 
Proposition \ref{c-iso}, that }
\begin{equation}\label{hatestimates2}
e^{\delta_m\frac{R_k}{2}}\nr \oplus_{a_k}(h_{0,k})\nr_{m+3,-\delta_m}\to 0\quad \text{and}\quad 
e^{\delta_mt\frac{R_k}{2}} \norm{ \ominus_{a_k} (h_{0,k} ) }_{H^{m+3,\delta_m}_c(C_{a_k},{\mathbb R}^{2n})}\to 0.
\end{equation}

Abbreviating by $c_k$ the asymptotic constants of the  maps $h_{0,k}$, it follows from the definition of the norm $\norm{\cdot}_{H^{m+3,\delta_m}_c(C_{a_k},{\mathbb R}^{2n})}$ and  from the second limit in \eqref{hatestimates2} that also $e^{\delta_m\cdot\frac{R_k}{2}}\cdot  c_k\to 0$.  Treating $c_k$ as constant functions  defined on the finite cylinders $Z_{a_k}$, we estimate 
$$\nr c_k\nr^2_{m+3,-\delta_m}=\abs{c_k}^2\int_{[0, R_k]}e^{-2\delta_m \abs{s-\frac{R_k}{2}}}\leq C^2\abs{c_k}^2.$$
We next compute the $G^{\alpha_k}$-norm of the pair $(\oplus_{a_k}(h_{0,k}), \ominus_{a_k}(h_{0,k}))$. 
 Using $[\oplus_{a_k}(h_{0,k})]_{a_k}=0$,  we have 
 \begin{equation*}\label{nh0j}
 \begin{split}
 &\abs{ (\oplus_{a_k}(h_{0,k}), \ominus_{a_k}(h_{0,k})}^2_{G^{a_k}_m}\\
 &=
 \abs{c_k}^2+e^{\delta_mR_k}\bigl[ \nr \oplus_{a_k}(h_{0,k})+c_k\nr^2_{m+3,-\delta_m}+
 \nr \ominus_{a_k}(h_{0,k})-[1-2\beta_{a_k}]\cdot c_k\nr^2_{m+3,\delta_m}\bigr].
 \end{split}
 \end{equation*}
 We already know that $\abs{c_k}\to 0$ and 
 \begin{equation*}
 \begin{split}
\nr \oplus_{a_k}(h_{0,k})+c_k\nr_{m+3,-\delta_m}&\leq \nr \oplus_{a_k}(h_{0,k})\nr_{m+3,-\delta_m}+ \nr c_k\nr_{m+3,-\delta_m}\\
&\leq   \nr \oplus_{a_k}(h_{0,k})\nr_{m+3,-\delta_m}+C\abs{c_k}
\end{split}
\end{equation*}
so that, in view of $e^{\delta_m\cdot\frac{R_k}{2}}\cdot  c_k\to 0$ and in view of the first limit in \eqref{hatestimates2}, we conclude 
$$ e^{\delta_mR_k}\nr \oplus_{a_k}(h_{0,k})+c_k\nr^2_{m+3,-\delta_m}\to 0.$$
The second limit in  \eqref{hatestimates2} implies the convergence
$$e^{\delta_m\cdot\frac{R_k}{2} }\cdot \nr \ominus_{a_k}(h_{0,k})-[1-2\beta_{a_k}]\cdot c_k\nr_{m+3,\delta_m}\to 0.$$
Summing up, we have proved that  $\abs{ (\oplus_{a_k}(h_{0,k}), \ominus_{a_k}(h_{0,k})}^2_{G^{a_k}_m}\to 0$.  Now  Proposition \ref{total-gluing}
 implies that $h_{0,k}$ converges to $0$ in $E_m$.  In a similar way, using the equations \eqref{h1jeq} and the fact that the  sequences 
$(\wt{z}_k)$ and $(z_k)$ are bounded, one shows that the sequence $(h_{1,k})$ is bounded in $E_{m+1}$. Since the sequence $(\wt{h}_k)$ is bounded, we conclude,  using the Sobolev embedding theorem, that 
the sequence of averages $([h_k]_{a_k})$ is bounded.  Hence,  the sequence 
 $\wt{h}_k$ admits the following decomposition
 $$
 \wt{h}_k = h_{0,k}+ \bigl( h_{1,k } +([\wt{h}_k]_{a_k},[\wt{h}_k]_{a_k}\bigr) =:h_{0,k} + h_{2,k}.
 $$
The sequence  $(h_{0,k})$ converges to $0$  in  $E_m$ and  the  sequence $(h_{2,k})$ is bounded in $E_{m+1}$. 
Since,   by the definition of the sc-structure,  the embedding $E_{m+1}\to E_m$ is compact,  we conclude that $( \wt{h}_k )$ possesses a convergent subsequence in $E_m$. This completes the proof of Proposition \ref{ELLIPTIC-X}.

\end{proof}
\section{Proof of Proposition \ref{XCX1}}\label{XCX10}
\noindent In this section we give a proof of Proposition \ref{XCX1}. 
We employ the notation $\eta=(\eta^+,\eta^-)$ and $\xi=(\xi^+, \xi^-)$. We start with  the map $B_\frac{1}{2}\oplus E\rightarrow F$, defined by  
$(a,\eta)\mapsto  D^a_{t}(\eta)$. The  pair $\xi:=D^a_{t}(\eta)$ is uniquely determined by the two equations
\begin{equation*}
\begin{aligned}
\wh{\oplus}_a(\xi) &= \pt (\oplus_a\eta )\\
\wh{\ominus}_a(\xi)&=0.
\end{aligned}
\end{equation*}
Since $\pt (\oplus_a(\eta))= \wh{\oplus}_a(\pt\eta)$, the pair $\xi$  solves the  equations
\begin{equation}\label{xi1}
\begin{aligned}
\wh{\oplus}_a(\xi)&= \wh{\oplus}_a(\pt\eta)\\
\wh{\ominus}_a(\xi)&=0.
\end{aligned}
\end{equation}
Recalling the projection $\wh{\pi}_a:F\to F$  onto $\ker \wh{\ominus}_a$  along  $\ker \wh{\oplus}_a$ introduced in Section \ref{gluinganti-sect}, 
 we find that 
$$
\xi  =\wh{\pi}_a(\partial_t \eta ).
$$
The map 
$\pt:E\rightarrow F,$   defined by $\eta \mapsto \pt\eta $,  is an sc-operator,  so that for every $m$ there is 
a positive constant $C'_m$ such that 
\begin{equation}\label{pi1}
\abs{\pt\eta}_{F_m}\leq C_m' \cdot \abs{\eta}_{E_m}.
\end{equation}
for all  $\eta \in E_m$.   Moreover,  in view of Theorem 1.29 in \cite{HWZ8.7}, 
 the map 
$\wh{\pi}: \bt \oplus F\to F$, defined by $(a, h)\mapsto \wh{\pi}_a(h)$,  is 
sc-smooth. Consequently,  applying the chain rule for sc-smooth maps,  
the map  $\xi$ depends sc-smoothly  on $(a,\eta)$. 
With  the total hat-gluing $\wh{\boxdot}_a(h):=(\wh{\oplus}_a(h),\wh{\ominus}_a(h))$, the equations \eqref{xi1} can be written as 
$$\wh{\boxdot}_a(\xi)=(\wh{\oplus}_a(\pt \eta), 0).$$
Now applying the estimates of  Proposition \ref{HAT_HAT}  
and \eqref{pi1},   there is a constant $C_m$ independent of $a$ such  that the following estimate  holds, 
\begin{equation*}
\begin{split}
\abs{D^a_t(\eta)}_{F_m}&=\abs{\xi}_{F_m}
\leq C''_m \abs{\wh{\boxdot}_a (\xi)}_{\wh{G}_m^a}= 
C''_m \abs{(\wh{\oplus}_a(\pt \eta), 0)}_{\wh{G}_m^a}\\
&\leq C''_m \abs{(\wh{\oplus}_a(\pt \eta), \wh{\ominus}_a(\pt \eta))}_{\wh{G}_m^a}=
C_m''\abs{\wh{\boxdot}_a(\pt \eta)}_{\wh{G}_m^a}\\
&\leq C_m''' \abs{(\pt \eta)}_{F_m}\leq C_m\abs{(\eta^+,\eta^-)}_{E_m}.
\end{split}
\end{equation*}
We now turn to  the operator $D^a_s:E\to F$.  With $\xi=D_s^a(\eta)$, the pair $\xi$  solves the  equations 
\begin{equation*}
\begin{aligned}
\wh{\oplus}_a(\xi)& = \ps(\oplus_a(\eta))\\
 \wh{\ominus}_a(\xi)&=0.
\end{aligned}
\end{equation*}
The first equation  can be written as
\begin{equation*}
\wh{\oplus}_a(\xi)
=\wh{\oplus}_a(\ps \eta)+ \beta_a'(s)
\cdot [ \eta^+(s,t)-\eta^-(s-R,t-\vartheta)], 
\end{equation*}
where $\beta_a(s)=\beta \bigl(s-\frac{R}{2}\bigr).$
To prove that the  map $(a,\eta)\mapsto  D^a_{t}(\eta)$ is sc-smooth, we introduce  the solution  $\xi_1:=(\xi^+_1, \xi^+_1)$  of the equations 
\begin{equation}\label{sy1}
\begin{aligned}
\wh{\oplus}_a(\xi_1)(s, t)&=\beta'_a(s)(\eta^+(s,t)-\eta^-(s-R,t-\vartheta))\\
 \wh{\ominus}_a(\xi_1)(s,t)&=0
\end{aligned}
\end{equation}
so that 
\begin{equation*}
\begin{aligned}
\wh{\oplus}_a(\xi-\xi_1)&=\wh{\oplus}_a(\ps \eta) \\
\wh{\ominus}_a(\xi-\xi_1)&=0.
\end{aligned}
\end{equation*}
Using  again the fact that $\wh{\pi}_a$ is a projection 
onto $\ker \wh{\ominus}_a$  along  $\ker \wh{\oplus}_a$, we find 
\begin{equation}\label{de1}
\xi-\xi_1=\wh{\pi}_a(\ps \eta).
\end{equation}
Hence the map $B_{\frac{1}{2}}\oplus E\to F$, defined by 
$$\bigl(a, \eta)\mapsto \xi-\xi_1,$$ 
is sc-smooth as a compositions of sc-smooth maps $\ps:E\to F$, $\eta\mapsto  \ps \eta$,  and  $B_{\frac{1}{2}}\oplus E\to E$, $\bigl(a, h\bigr)\mapsto \wh{\pi}_a (h)$. 
To show that the map $(a, \eta)\mapsto \xi_1$ is sc-smooth, we  write  the system \eqref{sy1} in matrix form 
\begin{equation*}
\begin{bmatrix}\beta_a&1-\beta_a\\
\beta_a-1&\beta_a\end{bmatrix}\cdot \begin{bmatrix}\xi_1^+(s, t)\\ \xi_1^-(s-R,t-\vartheta)
\end{bmatrix}=\begin{bmatrix}f(s, t)\\0\end{bmatrix},
\end{equation*}
where we have abbreviated $\beta_a=\beta_a(s)$ and $f(s,t)=\beta'_a(s)(\eta^+(s,t)-\eta^-(s-R,t-\vartheta))$. 
Introducing the  determinant $\gamma_a(s)=\beta_a^2(s)+(1-\beta_a)^2(s)$  of the matrix on the left hand side and multiplying both sides by the inverse of this matrix we arrive at the formula
\begin{equation*}
\begin{bmatrix}\xi_1^+(s, t)\\ \xi_1^-(s-R,t-\vartheta)\end{bmatrix}=\frac{1}{\gamma_a}
\begin{bmatrix}\beta_a&\beta_a-1\\
1-\beta_a&\beta_a\end{bmatrix}\cdot \begin{bmatrix}f(s,t)\\ 0\end{bmatrix}.
\end{equation*}
From this formula we read off the formulae for $\xi_1^+$ and $\xi_1^-$,
\begin{equation*}
\begin{split}
\xi_1^+(s, t)&=
 \frac{\beta_a\beta'_a}{\gamma_a}(s)\cdot [ \eta^+(s,t) -\eta^-(s-R,t-\vartheta)] \\
&= \frac{\beta_a\beta'_a}{\gamma_a}(s)\cdot [r^+(s,t) -r^-(s-R,t-\vartheta)]
\end{split}
\end{equation*}
and 
\begin{equation*}
\begin{split}
\xi_1^-(s-R, t-\vartheta)&=\frac{(1-\beta_a)\beta_a}{\gamma_a}(s)\cdot 
 [ \eta^+(s,t) -\eta^-(s-R,t-\vartheta)] \\
&=\frac{(1-\beta_a)\beta_a}{\gamma_a}(s)\cdot [r^+(s,t) - r^-(s-R,t-\vartheta)],
\end{split}
\end{equation*}
where we have decomposed  the maps  $\xi_1^\pm$ as $\xi_1^\pm =c+r^\pm$. Here $c$ is  the common asymptotic  constant of $\xi_1^\pm$. 
In view of Proposition 2.17  in \cite{HWZ8.7}, 
 the map
$\Phi:B_\frac{1}{2}\oplus E\rightarrow F$,  defined by 
$$
(a,\eta)\mapsto \begin{cases}\xi_1&\quad \text{$a\neq 0$}\\
0&\quad \text{$a=0$,}
\end{cases}
$$
is sc-smooth.   Summing up,  we have proved that the map 
$\bigl(a, \eta)\mapsto \xi$ is sc-smooth. 

To obtain the desired estimate, we first note that for every level $m$ there is a positive constant $C'_m$ such that 
\begin{equation}\label{pi2}
\abs{\ps \eta}_{F_m}\leq C'_m\cdot \abs{\eta}_{E_m}
\end{equation}
for all pairs $\eta\in E_m$. 
Writing the equation \eqref{sy1}  as 
$$\wh{\boxdot}_a(\xi-\xi_1)=(\wh{\oplus}_a(\ps \eta), 0)$$
and applying the estimates of  Proposition \ref{HAT_HAT}  
and \eqref{pi2}, we find  a constant $C_m$ independent of $a$ so that the following estimate  holds, 
\begin{equation}\label{ps2}
\begin{split}
\abs{\xi-\xi_1}_{F_m}&\leq C''_m \abs{\wh{\boxdot}_a (\xi-\xi_1)}_{\wh{G}_m^a}= C''_m \abs{\bigl(\wh{\oplus}_a(\ps \eta), 0\bigr)}_{\wh{G}_m^a}\\
&\leq C''_m \abs{(\wh{\oplus}_a(\ps \eta), \wh{\ominus}_a(\ps \eta)}_{\wh{G}_m^a}\\
&=
C_m''\abs{\wh{\boxdot}_a(\ps \eta)}_{\wh{G}_m^a}
\leq C_m'''\abs{\ps \eta}_{F_m}\\
&\leq C_m\abs{\eta}_{E_m}.
\end{split}
\end{equation}
In order to obtain similar estimates for $\xi_1$, we use  the fact that  the map $\Phi$ is, in particular, of class $\ssc^0$.
Hence, given  the level $m$, there exists  $\delta>0$ such  that
$|\eta|_{E_m}\leq \delta$ and $|a|\leq \delta$ implies  $|\Phi(a,\eta)|_{F_m}\leq 1$.
Since $\Phi (a, \cdot )$ is linear, we conclude that 
$$
|\Phi (a,\eta )|_{F_m}\leq \delta^{-1}\cdot |\eta |_{E_m}
$$
for $|a|\leq \delta$ and $\eta \in E_m$.   Moreover,   it follows from the formulae for $\xi_1^\pm$  that there exists a positive constant $C_m''>\delta^{-1}$ for which  
$$
|\Phi(a,\eta)|_{F_m}=|\xi_1|_{F_m}\leq C_m''\cdot | r |_{F_m}\leq C_m'' \cdot |\eta |_{E_m}
$$
for every $a$ satisfying  $\delta \leq |a|<\frac{1}{2}$ and for all $\eta \in E_m$. 
Combining   the estimate \eqref{ps2} with the above estimates,  we find a constant $C_m$ independent of $a$ such that 
$$
|D^a_s(\eta)|_{F_m}=\abs{\xi}_{F_m}\leq C_m\cdot |\eta|_{E_m}
$$
for all $\eta\in E_m$.

We next discuss the operators $C^a_s, C^a_t:E\to F$. Setting $C_t^a(\eta)=\xi$, the map $\xi$  satisfies  the equations
\begin{equation*}
\begin{aligned}
\wh{\oplus}_a(\xi)&=0\\
\wh{\ominus}_a(\xi)&=\partial_t (\ominus_a(\eta)).
\end{aligned}
\end{equation*}
Solving for $\xi=(\xi^+, \xi^-)$, we find 
\begin{equation}\label{cat}
\begin{split}
\xi^+(s, t)&=\frac{\beta_a-1}{\gamma_a}(s)\cdot \partial_t(\ominus_a (\eta))(s, t)\\
&=\frac{(\beta_a-1)^2}{\gamma_a}(s)\cdot \partial_t \eta^+(s, t)+\frac{\beta_a(\beta_a-1)}{\gamma_a}(s)\cdot \partial_t\eta^-(s-R, t-\vartheta)
\end{split}
\end{equation}
and
\begin{equation*}
\begin{split}
\xi^-(s-R, t-\vartheta)=\frac{\beta_a(\beta_a-1)}{\gamma_a}(s)\cdot\partial_t \eta^+(s, t)+\frac{\beta_a^2}{\gamma_a}(s)\cdot \partial_t \eta^-(s-R, t-\vartheta).
\end{split}
\end{equation*}
The map $E\to F$, defined by $\eta\mapsto \partial_t\eta$, is sc-smooth.  Also, in view of Proposition 2.8  in  \cite{HWZ8.7}, the maps 
$\hb \oplus F\to F$, defined by 
$$(a, \eta)\mapsto \frac{(\beta_a-1)^2}{\gamma_a}(s)\cdot \eta^+(s, t)\qquad \text{and}\qquad 
(a, \eta)\mapsto \frac{\beta_a^2}{\gamma_a}(s)\cdot \eta^-(s-R, t-\vartheta),$$
are sc-smooth.  Hence, by the chain rule, the map $\hb\times E\to H^{2,\delta_0}(\R^+\times S^1)$, $(a, \eta)\mapsto \xi^+$ is sc-smooth. The same conclusion holds for the map $\hb\times E\to H^{2,\delta_0}(\R^-\times S^1)$ given by  $(a, \eta)\mapsto \xi^-$. Summing up, 
the map $\hb\times E\to F$, $(a, \eta)\mapsto \xi$ is sc-smooth.  

We finally consider the map 
$\hb\oplus E\to F$, defined by $(a, \eta)\mapsto C_s^a(\eta)$. Setting $\xi=C_s^a(\eta)$, the map $\xi$ solves the following two equations,
\begin{equation*}
\begin{aligned}
\wh{\oplus}_a(\xi)&=0\\
\wh{\ominus}_a(\xi)&=\partial_s (\ominus_a(\eta)).
\end{aligned}
\end{equation*}
We obtain for the pair $\xi=(\xi^+, \xi^-)$ the  formulae
\begin{equation*}
\begin{split}
\xi^+(s, t)&=\frac{\beta_a-1}{\gamma_a}(s)\cdot \wh{ \ominus}_a (\partial_s\eta )(s, t)\\
&\phantom{==}+\frac{(\beta_a-1)\beta_a'}{\gamma_a}\cdot \bigl(\eta^+(s, t)-\av_a (\eta)\bigr)\\
&\phantom{==}+\frac{(\beta_a-1)\beta_a'}{\gamma_a}\cdot \bigl(\eta^-(s-R, t-\vartheta)-\av_a (\eta)\bigr)
\end{split}
\end{equation*}
and an analogous expression for $\xi^-(s-R, t-\vartheta)$. 
The average $\av_a (\eta)$ is the  number 
$$\av_a (\eta):=\frac{1}{2}\biggl( \int_{S^1}\eta^+\Bigl(\frac{R}{2}, t\Bigr)\ dt+ \int_{S^1}\eta^-\Bigl(-\frac{R}{2}, t\Bigr)\ dt\biggr).$$
The  map $\hb \times E\to F$, 
$$(a, \eta)\mapsto \frac{\beta_a-1}{\gamma_a}\cdot \wh{ \ominus}_a (\partial_s\eta))$$
is sc-smooth by the same argument as above. 
In view of the results in  Section 2.4 of  \cite{HWZ8.7},  the maps 
\begin{equation*}
\begin{aligned}
(a, \eta^+)&\mapsto \frac{(\beta_a-1)\beta_a'}{\gamma_a}\cdot \bigl(\eta^+ -\av_a (\eta)\bigr)\\
(a, \eta^-)&\mapsto \frac{(\beta_a-1)\beta_a'}{\gamma_a}\cdot \bigl(\eta^-(s-R, t-\vartheta)-\av_a (\eta)\bigr)
\end{aligned}
\end{equation*}
from $\hb\oplus H^{3,\delta_0}(\R^+\times S^1)$ to $H^{3,\delta_0}(\R^+\times S^1)$ and from  $\hb\oplus H^{3,\delta_0}(\R^-\times S^1)$ to $H^{3,\delta_0}(\R^-\times S^1)$, respectively, are sc-smooth.  The same arguments apply to the component $\xi^-$.  Therefore,   the map  $\hb\times E\to F$, $(a, \eta)\mapsto \xi$ is sc-smooth. The desired estimates follow from the formulae for the components of the maps 
$(a, \eta)\mapsto C^a_t(\eta)$ and $(a, \eta)\mapsto C^a_s(\eta)$.  This completes the proof of 
Proposition  \ref{XCX1}. \hfill $\blacksquare$

\section{Banach Algebra Properties}\label{XCX19} 
Our aim is to prove the following Banach algebra properties  for the Banach spaces  $H^{m,\delta} (\R^+\times S^1)$ and $H_{\textrm{c}}^{m,\delta} (\R^+\times S^1)$  equipped with the norms
$$\norm{f}_{m, \delta}^2= \sum_{ \abs{\alpha}\leq m}\int_{\R^+\times S^1}\abs{ D^{\alpha}f(s, t) }^2e^{2\delta s}\  dsdt$$
for $f\in H^{m,\delta} (\R^+\times S^1)$ and 
$$\norm{g}_{m, \delta}^2=\abs{c}^2+\sum_{ \abs{\alpha}\leq m}\int_{\R^+\times S^1}\abs{ D^{\alpha}( g(s, t) -c)}^2e^{2\delta s}\  dsdt$$
for $g\in H_{\textrm{c}}^{m,\delta} (\R^+\times S^1)$, where $c$ is the asymptotic constant. 

\begin{theorem}\label{balgebra}
For $m\geq 3$ and $2\leq k\leq m$ and $0\leq  \sigma\leq \delta$, there exists a constant $C$ such that  
$$\norm{f\cdot g}_{k,\delta}\leq C\norm{f}_{m,\delta}\cdot \norm{g}_{k,\sigma}$$
for all $f\in H^{m,\delta}(\R^+\times S^1)$ and $g\in H^{k,\sigma}(\R^+\times S^1)$. Moreover,  
\begin{itemize}
\item[{\em (1)}] If $m\geq 3$ and $2\leq k\leq m$ and $0\leq \delta$,  
$$\norm{f\cdot g}_{k,\delta}\leq C\norm{f}_{H_c^{m,\delta}(\R^+\times S^1)}\cdot  \norm{g}_{k,\delta}$$
for all $f\in H_c^{m,\delta}(\R^+\times S^1)$ and all $g\in H^{k,\delta}(\R^+\times S^1)$.
\item[{\em (2)}] If $k\geq 3$ and $\delta>\sigma\geq 0$, then 
$$\norm{f\cdot g}_{k,\delta}\leq C\norm{f}_{H_c^{k,\sigma}(\R^+\times S^1)}\cdot  \norm{g}_{k,\delta}$$
for all $f\in H_c^{k,\sigma}(\R^+\times S^1)$ and all $g\in H^{k,\delta}(\R^+\times S^1)$.
\end{itemize}
\end{theorem}
\begin{proof}
We first assume that $k\geq 3$.  Then the square of the norm $\norm{f\cdot g}_{k,\sigma}$ is bounded by some constant  times the sum of integrals of the form 
\begin{equation}\label{intba1}
\int_{\R^+\times S^1}\abs{D^\alpha f}^2\abs{D^\beta g}^2\cdot e^{2\sigma s}
\end{equation}
where the multi-indices $\alpha$ and $\beta$ satisfy $\abs{\alpha}+\abs{\beta}\leq k$.  Then, since $k\geq 3$, either $\abs{\alpha}\leq k-2$ or $\abs{\beta}\leq k-2$. We assume that $\abs{\alpha}\leq k-2$.  Since $e^{\delta \cdot }f$ belongs to $H^{k}$, it follows from the Sobolev embedding theorem that 
$\max_{0\leq \abs{\alpha}\leq k-2} \abs{D^{\alpha} (e^{\delta \cdot} f )}_{L^\infty(\R^+\times S^1)}\leq C'\norm{e^{\delta \cdot }f}_{H^k}\leq C\norm{f}_{k,\delta}$ for some constant $C$. In particular, we conclude that $\abs{e^{\delta \cdot}D^{\alpha}f}_{L^\infty(\R^+\times S^1)}\leq C\norm{f}_{k,\delta}$ for all multi-indices $\alpha$ satisfying 
$\abs{\alpha}\leq k-2$. 
Hence, if $\abs{\alpha}\leq k-2$, we obtain the estimate
\begin{equation*}
\begin{split}
\int_{\R^+\times S^1}\abs{D^\alpha f}^2\abs{D^\beta g}^2\cdot e^{2\delta s}
&\leq C^2\norm{f}^2_{k,\delta}\cdot \int_{\R^+\times S^1}\abs{D^\beta g}^2\leq C^2\norm{f}_{k,\delta}^2\cdot \norm{g}^2_{k,\sigma}\\
&\leq 
C^2\norm{f}_{m,\delta}^2\cdot \norm{g}^2_{k,\sigma}
\end{split}
\end{equation*}
since $m\geq k$.  If  the multi-index $\beta$ satisfies $\abs{\beta}\leq k-2$,  then the derivative $D^\beta g$ belongs to $H^{k-\abs{\alpha}}\subset H^2$  so that, in view of the Sobolev embedding  theorem,  $\abs{D^\beta g}_{L^\infty (\R^+\times S^1)}\leq  C\norm{g}_{k,\sigma}$, and hence 
$$
\int_{\R^+\times S^1}\abs{D^\alpha f}^2\abs{D^\beta g}^2\cdot e^{2\delta s}
\leq C^2\norm{g}^2_{k,\sigma}\cdot \int_{\R^+\times S^1}\abs{D^\alpha f}e^{2\delta s}\leq C^2\norm{f}_{m,\delta}^2\cdot \norm{g}^2_{k,\sigma}. 
$$
Next we consider the  remaining case  $k=2$ and $m\geq 3$. 
Then the multi-indices $\alpha$ and $\beta$ in \eqref{intba1} satisfy  $|\alpha|+|\beta|\leq 2$.
Since $m\geq 3$, it follows as before that 
$\max_{\abs{\alpha}\leq 1}\abs {e^{\delta \cdot }D^{\alpha}f}_{L^\infty (\R^+\times S^1)}\leq C\norm{f}_{m,\delta}$ and we conclude for the multi-indices $\abs{\alpha}\leq 1$ that 
\begin{equation*}
\begin{split}
&\int_{\R^+\times S^1}\abs{f}^2\cdot |D^\beta g|^2\cdot e^{2\delta s}\leq  C^2\norm{f}^2_{m,\delta}\int_{\R^+\times S^1} |D^\beta g|^2\leq C^2\norm{f}^2_{m, \delta}\cdot \norm{g}^2_{k,\sigma}.
\end{split}
\end{equation*}
If $\abs{\alpha}=2$ so that $\beta=0$,  then, in view of the Sobolev embedding theorem applied to $g\in H^{2}$, we can   estimate 
 $\abs{g}_{L^\infty (\R^+\times S^1)}\leq C\norm{g}_{k,\sigma}$,  and obtain
\begin{equation*}
\begin{split}
\int_{\R^+\times S^1}\abs{D^\alpha f }^2\cdot \abs{g}^2\cdot e^{2\delta s}\leq  C^2\norm{g}^2_{k,\sigma}\int_{\R^+\times S^1} \abs{D^\alpha f}^2e^{2\delta s}\leq C^2\norm{f}^2_{m, \delta}\cdot \norm{g}^2_{k,\sigma}.
\end{split}
\end{equation*}
In order to  prove the statements (1) and (2), we split $f=c+r$ where $c$ is the asymptotic constant and $r$ belongs to $H^{m,\delta}(\R^+\times S^1)$ in case (1) and to $H^{k,\sigma}(\R^+\times S^1)$ in case (2). Since  in both cases $g\in H^{k,\delta}(\R^+\times S^1)$, we have $cg\in  H^{k,\delta}(\R^+\times S^1)$ and, in  view of the previous part,  also $rg\in H^{k,\delta}(\R^+\times S^1)$.  Recalling that $\norm{g}^2_{H_c^{k,\delta} (\R^+\times S^1)}=\abs{c}^2+\norm{r}^2_{k,\delta}$, we estimate in case (1),
\begin{equation*}
\begin{split}
\norm{fg}_{k,\delta}&=\norm{cg+rg}_{k,\delta}\leq \abs{c}\cdot \norm{g}_{k,\delta}+\norm{rg}_{k,\delta}\\
&\leq  \abs{c}\cdot \norm{g}_{k,\delta}+C\norm{r}_{m,\delta}\cdot \norm{g}_{k,\delta}\leq 2C\norm{f}_{H_c^{k,\delta}(\R^+\times S^1)}\cdot \norm{g}_{k,\delta}, 
\end{split}
\end{equation*}
and in case (2),
\begin{equation*}
\begin{split}
\norm{fg}_{k,\delta}&=\norm{cg+rg}_{k,\delta}\leq \abs{c}\cdot \norm{g}_{k,\delta}+\norm{rg}_{k,\delta}\\
&\leq  \abs{c}\cdot \norm{g}_{k,\delta}+C\norm{r}_{k,\sigma}\cdot \norm{g}_{k,\delta}\leq 2C\norm{f}_{H_c^{k,\sigma}(\R^+\times S^1)}\cdot \norm{g}_{k,\delta}, 
\end{split}
\end{equation*}
The proof of Theorem \ref{balgebra} is complete.
\end{proof}
\section{Proof of Proposition \ref{XCX2}}\label{XCX20}
We recall that,  given a smooth map $A:\R^{N}\to \call (\R^K, \R^M)$, the map 
$\wt{A}:\hb\oplus E^{(N)}\oplus F^{(K)}\to F^{(M)}$, $(a, u, \eta)\mapsto \xi=\wt{A}(a, u,\eta)$,  is defined as the unique solution of the  equations
\begin{equation}\label{ainduced}
\begin{aligned}
\wh{\oplus}_a(\xi) (s,t)&= A\bigl(\oplus_a (u)(s,t)\bigr) \cdot \wh{\oplus}_a  (\eta) (s,t)\\
\wh{\ominus}_a(\xi)(s,t)&=0,
\end{aligned}
\end{equation}
for $[(s, t)]\in Z_a$ where we have abbreviated $u=(u^+, u^-)\in E^{(N)}$, $\eta=(\eta^+, \eta^-)\in F^{(K)}$, and $\xi=(\xi^+, \xi^-)\in F^{(M)}$.  Here the spaces are defined as follows. 
The sc-Hilbert  space $E^{(N)}$ consists of pairs $(u^+, u^-)$ belonging to $H_c^{3,\delta_0}(\R^+\times S^1, \R^N)\oplus 
H_c^{3,\delta_0}(\R^-\times S^1, \R^N)$ having common asymptotic constants as $s\to \pm\infty$ and  the sc-Hilbert space $F^{(K)}$ is the space $H^{2,\delta_0}(\R^+\times S^1, \R^K)\oplus H^{2,\delta_0}(\R^+\times S^1, \R^K)$.  We also point out that in view of  the  second equation and the properties of  the  function $\beta_a$, we conclude $\xi^+(s, t)=0$ for  $s\geq \frac{R}{2}+1$ and $\xi^-(s', t')=0$ for $s'\leq -\frac{R}{2}-1$.  

Without loss of generality we may assume that $N=K=M=1$ so that $A:\R\to \R$ is a smooth function. We also drop the  superscript from $E^{(1)}$ and $F^{(1)}$ and   write  instead simply $E$ and $F$, respectively.  \\[1ex]
{\bf (1).} \   In order to prove the first  statement in Proposition \ref{XCX2}, we derive a  formula for  $\wt{A}\bigl(a, u, \eta\bigr)=\xi\in F$. Abbreviating  $\beta_a=\beta_a(s)$  and 
$$u_a=\oplus_a(u)\quad \text{and}\quad \eta_a=\wh{\oplus}_a(\eta),$$
 we write  the equations \eqref{ainduced}  in matrix form as 
\begin{equation*}
\begin{bmatrix}\beta_a&1-\beta_a\\
\beta_a-1&\beta_a\end{bmatrix}\cdot \begin{bmatrix}\xi^+(s, t)\\ \xi^-(s-R,t-\vartheta)
\end{bmatrix}=\begin{bmatrix}A(u_a(s,t)) \cdot   \eta_a(s,t)\\ 0\end{bmatrix}. 
\end{equation*}
With $\gamma_a(s)=\beta_a^2(s)+(1-\beta_a)^2(s)$ denoting the  determinant  of the matrix on the left hand side and multiplying both sides by the inverse of this matrix,  we arrive at the formula 
\begin{equation*}
\begin{bmatrix}\xi^+(s, t)\\ \xi^-(s-R,t-\vartheta)\end{bmatrix}=\frac{1}{\gamma_a}
\begin{bmatrix}\beta_a&\beta_a-1\\
1-\beta_a&\beta_a\end{bmatrix}\cdot \begin{bmatrix}  A(u_a(s,t)) \cdot   \eta_a(s,t)  \\ 0\end{bmatrix}.
\end{equation*}
The formulae for $\xi^+$ and $\xi^-$ are 
\begin{equation*}
\begin{split}
\xi^+(s, t)=
 \frac{\beta_a}{\gamma_a}(s)\cdot f(s,t)= \frac{\beta_a}{\gamma_a}(s)\cdot A\bigl(u_a(s,t)\bigr) \cdot \eta_a(s,t)
\end{split}
\end{equation*}
and $\xi^+(s, t)=0$ for $s\geq \frac{R}{2}+1$, and 
\begin{equation*}
\begin{split}
\xi^-(s-R, t-\vartheta)=\frac{1-\beta_a}{\gamma_a}(s)
\cdot A\bigl(u_a(s,t)\bigr) \cdot \eta_a(s,t)
\end{split}
\end{equation*}
and $\xi^-(s-R, t-\vartheta)=0$ for $s\leq \frac{R}{2}-1$.  
Our aim is to prove  that for $i=0, 1$ the map $\hb\oplus E\oplus F^i\to F^i$, 
$$(a, u, \eta)\mapsto  \xi,$$
is sc-smooth. To do this we  consider the formula for $\xi^+$. 
We  introduce the map $f:\hb\oplus F^i\to H^{2+i,\delta_i}(\R^+\times S^1)$, defined for $a\neq 0$  by 
\begin{equation}\label{phieta1}
\begin{split}
f(a, \eta)(s, t)=\frac{\beta_a}{\gamma_a}(s)\cdot \wh{\oplus}_a(\eta)(s, t)
\end{split}
\end{equation}
where $(s, t)\in \R^+\times S^1$. If $a=0$,  we define 
$f(0, \eta)(s, t)=\eta^+(s, t)$ for $(s, t)\in \R^+\times S^1$.
 We  point out   that if $a\neq 0$, then 
  $$f (a, \eta)(s, t)=0\quad \text{for $s\geq \frac{R}{2}+1.$}$$
 Decomposing $u^\pm=c+r^\pm$ where $c$ is the common asymptotic constant of the maps $u^\pm$,  the glued map $\oplus_a (u)$ is, in view of Section 2.4 of \cite{HWZ8.7},  given  by 
 \begin{equation*}
 \oplus_a(u)(s, t)=c+\oplus_a(r)(s, t)
 \end{equation*}
where $0\leq s\leq R$ and where $r=(r^+, r^-)$. 
Next we introduce  the map $g:  \hb\times H_c^{3,\delta_0}(\R^+\times S^1)\to H^{3,\delta_0}(\R^+\times S^1)$,  defined by 
\begin{equation}\label{g1}
g(a, u)(s, t)=c+\beta_a(s-2)\cdot \oplus_a(r)(s, t), 
\end{equation}
if $a\neq 0$. If $a=0$, we set $g(0, u)=c+r^+=u^+$.
We note that $g(a,u)(s, t)=0$ for $s\geq \frac{R}{2}+3$ and $\oplus_a(u)(s, t)=g(a, u)(s, t)$ for $s\leq \frac{R}{2}+1$ since $\beta_a(s-2)=1$ for $s\leq \frac{R}{2}+1$.
Since $f (a, \eta)(s, t)=0$ for $s\geq  \frac{R}{2}+1$, we therefore find that 
$$\xi^+(s, t)=A(g(a,u)(s, t))\cdot f(a, \eta)(s, t)$$
if $a\neq 0$, and if $a=0$, then 
$$\xi^+(s, t)=A(u^+(s, t))\cdot \eta^+(s, t)=A(g(0, u)(s, t))\cdot f(0, \eta)(s, t)$$
for $(s, t)\in \R^+\times S^1$.
Consequently, $\xi^+$ is the  product of two maps defined on all of $\R^+\times S^1$ and we shall show that  each map is sc-smooth.

We start with the map $\hb \oplus E\to H^{3,\delta_0}_c(\R^+\times S^1)$, defined by $(a, u)\mapsto A(g(a, u))$.  First, using the arguments in Section 2.4  of  \cite{HWZ8.7} 
one shows that the map $\hb \oplus E\to H^{3,\delta_0}_c(\R^+\times S^1)$, defined by $(a, u)\mapsto  g(a,u)$ is sc-smooth.
The smooth  function  $A:\R\to \R$ induces   
a smooth substitution operator 
$$\wh{A}:H^{m,\delta}_c(\R^+\times S^1)\to H^{m,\delta}_c(\R^+\times S^1), \qquad \wh{A}(u)=A\circ u$$
 for every $m\geq 2$ and $\delta\geq 0$. In particular, the map 
 $$\wh{A}: H^{3,\delta_0}_c(\R^+\times S^1)\to H^{3,\delta_0}_c(\R^+\times S^1)$$ 
 is sc-smooth. 
Applying the chain-rule for sc-smooth maps, we conclude that the map 
$$
\hb\oplus E\to H^{3,\delta_0}_c({\mathbb R}^+\times S^1),\qquad (a, u)\mapsto \wh{A}(g(a,u))=A\circ  g(a, u)
$$
is sc-smooth.   Also  the map $\hb\oplus F^i\to H^{2+i,\delta_i}(\R^+\times S^1)$, defined by \eqref{phieta1} is sc-smooth for $i=0$ and $i=1$,  again by the arguments  in Section 2.4 of  \cite{HWZ8.7}. 

The product map  
$\Phi:H^{3+m,\delta_m}_c(\R^+\times S^1)\oplus H^{2+m+i,\delta_{m+i}} (\R^+\times S^1)\to  H^{2+m+i,\delta_{m+i}}(\R^+\times S^1)$, defined by $\Phi(f, g)=f\cdot g$,   is smooth for every level $m\geq 0$ and $i=0,1$.  This follows from the statement (1) of   Theorem \ref{balgebra} if $i=0$ and from the statement (2) if $i=1$. 
Consequently, applying the chain-rule, the map $\hb\oplus E\oplus F^i\to H^{2+i,\delta_i}(\R^+\times S^1)$, defined by 
$(a, u,\eta)\mapsto \Phi\bigl(\wh{A}(g(a,u)),f (a, \eta)\bigr)$,  is sc-smooth and hence the map $(a, u,\eta)\mapsto \xi^+$ is sc-smooth. 
The same  arguments  applied to  the second map $\hb\oplus E\oplus F^i\to H^{2+i,\delta_i}(\R^-\times S^1)$,\ $(a, u, \eta)\mapsto \xi^-$ show that it  is also sc-smooth. 

Consequently, the map $\hb\times E\oplus F^i\to F^i$, $(a, u,\eta)\mapsto  \xi=\wt{A}(a, u,\eta)$, is sc-smooth as claimed. \\[1ex]

  To obtain the desired estimate we first  consider the map $(a, \eta)\mapsto f(a, \eta)$.  The map is sc-smooth and, in particular, of class $\ssc^0$. Hence, given the  level $m\geq 0$, there exists a constant $\rho>0$ such that 
 \begin{equation*}
 \norm{ f (a, \eta)}_{H^{2+m+i,\delta_{m+i}}(\R^+\times S^1)}\leq 1
 \end{equation*}
 for all $\abs{a}\leq \rho$, $i=0,1$,  and  all $\abs{\eta}_{F_{m+i}} \leq \rho$.  Since the map is linear in the variable $\eta$, we conclude that 
 \begin{equation*}\label{linvar1}
 \norm{ f (a, \eta)}_{H^{2+m+i,\delta_{m+i}}(\R^+\times S^1)}\leq \rho^{-1}\abs{\eta}_{F_{m+i}}
 \end{equation*} 
for all $\abs{a}\leq \rho$, $i=0,1$,  and all $\eta \in F_{m+i}$.  Also, it follows from \eqref{phieta1} that there is a constant $C_m>\rho^{-1}$ such that 
 \begin{equation*}
 \norm{ f (a, \eta)}_{H^{2+m+i,\delta_{m+i}}(\R^+\times S^1)}\leq  C_m\abs{\eta}_{F_{m+i}}
 \end{equation*} 
 for all $\rho\leq \abs{a}\leq \frac{1}{2}$, $i=0,1$,  and all $\eta \in F_{m+i}$.  Consequently, 
\begin{equation}\label{feta}
 \norm{f(a, \eta)}_{H^{2+m+i,\delta_{m+i}}(\R^+\times S^1)}\leq  C_m\abs{\eta}_{F_{m+i}}
 \end{equation} 
 for all $\abs{a}\leq \frac{1}{2}$, $i=0,1$, and all  $\eta\in F_{m+i}$.  Now,  given a pair $u_0=(u_0^+,u_0^-)\in E_m$ where $u^\pm_0=c_0+r^\pm_0$, we  define 
 $$g(a, u_0)(s, t)=c_0+\beta_a(s-2)\cdot  \wh{\oplus}_a(r_0)(s,t)$$
 if $a\neq 0$ and $g(0, u_0)=u_0^+$ if $a=0$. 
We already know that the map $(a, u)\mapsto A \circ g(a, u))$ is sc-smooth,  hence, in particular,  of class $\ssc^0$.  Therefore, fixing a level $m\geq 0 $,  we find  for every $a_0\in \hb$ and $\varepsilon>0$   a positive number $\rho_{a_0}$ such that 
$$\norm{A(g(a, u))-A(g(a_0, u_0))}_{H_c^{m+3,\delta_m}(\R^+\times S^1)}\leq \varepsilon/2$$
 for all $\abs{a-a_0}<\rho_{a_0}$ and $\abs{u-u_0}_{E_m}<\rho_{a_0}.$
 In particular, 
 $$\norm{A(g(a, u_0))-A(g(a_0, u_0))}_{H_c^{m+3,\delta_m}(\R^+\times S^1)}\leq \varepsilon/2$$
implying
\begin{equation*}
\begin{split}
&\norm{A(g(a, u))-A(g(a, u_0)) }_{H_c^{m+3,\delta_m}(\R^+\times S^1)}\\
&\phantom{====}\leq  \norm{A(g(a, u))-A(g(a_0, u_0))}_{H_c^{m+3,\delta_m}(\R^+\times S^1)}\\
&\phantom{====\leq}+ \norm{A(g(a_0, u_0))-A(g(a, u_0))}_{H_c^{m+3,\delta_m}(\R^+\times S^1)}\leq  \varepsilon.
\end{split}
\end{equation*}
Covering the closed disk $\hb$ by open disks  $B(a_0,\rho_{a_0})$ and using compactness we find a constant $\rho>0$ such that 
$$\norm{A(g(a, u))-A(g(a, u_0))}_{H_c^{m+3,\delta_m}(\R^+\times S^1)}\leq \varepsilon$$
for all $a$ and for $u\in E_m$ satisfying $\abs{u-u_0}_{E_m}<\rho.$
 
Now, using  Theorem \ref{balgebra} on  level $m+i$ for $i=0,1$ and the estimate \eqref{feta},  we conclude for given 
$u_0\in E_m$ and $\varepsilon>0$ that there exist $\rho>0$  and a constant  $C_m$ depending on $m$ but not on $a$ such that  
 \begin{equation*}
\begin{split}
\norm{\bigl(A(g(a, u))-A(g(a, u_0))\bigr)\cdot f(a, \eta)}_{H^{m+2+i,\delta_{m+i}}(\R^+\times S^1)}\leq  C_m\cdot \varepsilon\cdot \abs{\eta}_{F_{m+i}}
\end{split}
\end{equation*}
for all $a\in B_{\frac{1}{2}}$, all  $\eta\in F_{m+i}$ and for $u\in E_m$ satisfying  $\abs{u-u_0}_{E_m}<\rho$. This completes the proof of the part (1) of Proposition 4.6 and we next  turn to part (2).  \\[1ex]
{\bf (2).} \  We recall  that,  for a fixed $a\in \hb$,  the map $L_a:E\to \call (F, F)$ is defined by 
$$L_a(u)=\wt{A}(a, u, \cdot )\in \call (F, F).$$
That for fixed $a$ and level $m$  the map $L_a$ is of class $C^1$ follows from standard classical argument
and the fact that gluing is linear. Then one computes the form of $DL_a$ by a formal calculation.
For this we consider  the  smooth  function  $B:\R\oplus \R\to \R$, defined by 
$$B(p, q)=DA(p)\cdot q$$
where $DA$ stands for the derivative of $A$, and introduce the map $\wt{B}:\hb\oplus E\oplus E\oplus F\to F$, $(a, u,\wh{u},\eta)\mapsto \wt{B}(a, u, \wh{u}, \eta)=\xi$,  defined as the unique solution of the two equations
\begin{equation*}
\begin{aligned}
\wh{\oplus}_a(\xi)&=B(u_a, \wh{u}_a)\cdot \eta_a\\
\wh{\ominus}_a(\xi)&=0
\end{aligned}
\end{equation*}
if $a\neq 0$. Here we have  used the notation 
$u_a=\oplus_a (u)$, $\wh{u}_a=\oplus_a (\wh{u})$,  and $\eta_a=\wh{\oplus}_a(\eta).$ If 
$a=0$, then 
$\wt{B}(a, u, \wh{u}, \eta)=\bigl([DA(u^+)\cdot \wh{u}^+]\cdot \eta^+, [DA(u^-)\cdot \wh{u}^-]\cdot \eta^-\bigr).$
 We claim that the derivative $DL_a(u)\wh{u}$ of the  map $L_a$  at the  point $u$ in the direction of $\wh{u}$ is equal to 
\begin{equation}\label{btilde_eq1}
DL_a(u)\wh{u}=\wt{B}(a, u, \wh{u}, \cdot ).
\end{equation}

The continuity of the map $DL_a:E_m\to \call (E_m, \call (F_m, F_m))$  is a consequence of the estimate in the part (3) below. 

 The sc-smoothness  of the map 
$\hb\oplus E\oplus E \oplus F\to F$,  defined by 
$$(a, u, \wh{u}, \eta)\mapsto [DL_a(u)\wh{u}]\eta=\wt{B}(a, u, \wh{u}, \eta), $$
can be reduced to the previous discussion in part (1) by replacing the function $A$ by  the function $B$. This completes the proof of part (2) in Proposition \ref{XCX2}, and we shall prove part (3).\\
 {\bf (3).} \  Fixing a level $m\geq 0$ and $u_0\in E_m$, we have  to estimate
\begin{equation*}
\begin{split}
&\abs{\bigl[DL_a(u)\wh{u}\bigr]\eta -\bigl[DL_{a}(u_0)\wh{u}\bigr]\eta}_{F_m}=
\abs{\wt{B}(a, u, \wh{u}, \eta)-\wt{B}(a, u_0, \wh{u}, \eta)}_{F_m}. 
\end{split}
\end{equation*}
For $a\neq 0$,  we abbreviate  $\xi=\wt{B}(a, u, \wh{u}, \eta)$ and $\xi_0=\wt{B}(a, u_0, \wh{u}, \eta)$. Then, 
 \begin{equation*}
 \begin{split}
 \xi^+-\xi^+_0&=\bigl[ DA(u_a)-DA((u_0)_{a})\bigr] \cdot \wh{u}_{a}\cdot\biggl[\frac{\beta_a}{\gamma_a} \eta_{a}\biggr]
 \end{split}
 \end{equation*}
With the maps $f:\hb\oplus F\to H^{2,\delta_0}(\R^+\times S^1)$, $(a, \eta)\mapsto f(a, v)$,  and $\wh{f}:\hb\oplus E\to H^{3,\delta_0}(\R^+\times S^1)$, $(a,\wh{u} )\mapsto \wh{f}(a, \wh{u})$,  both defined by \eqref{phieta1},  and the map $g: \hb\times H_c^{3,\delta_0}(\R^+\times S^1)\to H^{3,\delta_0}(\R^+\times S^1)$, $(a, u)\mapsto g(a, u)$ defined by \eqref{g1},  the map $\xi^+-\xi^+_0$ is equal to 
 \begin{equation*}
 \xi^+-\xi^+_0=\bigl[ DA(g(a, u))-DA(g(a, u_0))\bigr]\cdot \wh{f}(a, \wh{u})\cdot f(a,\eta)
 \end{equation*}
 For the map $f$ we have the estimate \eqref{feta} and for the map $\wh{f}$ we similarly obtain  the estimate 
 \begin{equation}\label{whfeta}
 \norm{\wh{f}(a, \wh{u})}_{H^{3+m,\delta_m}(\R^+\times S^1)}\leq  C_m\abs{\wh{u}}_{E_m}
 \end{equation} 
 for all $\abs{a}\leq \frac{1}{2}$ and all  $\wh{u}\in E_m$. 
 Hence,  using  the Banach algebra property in Theorem \ref{balgebra} twice, and the estimates \eqref{feta} and \eqref{whfeta},  we obtain 
\begin{equation*}
\begin{split}
&\norm{ \xi^+-\xi^+_0}_{H^{2+m,\delta_m}(\R^+\times S^1)}\\
&\leq C_m\norm{DA(g(a, u))-DA(g(a, u_0))}_{H_c^{3+m,\delta_m}(\R^+\times S^1)}\cdot \abs{\wh{u}}_{E_m}\cdot  \abs{\eta}_{F_m}
\end{split}
\end{equation*}
with a constant $C_m$ depending on $m$ but not on $a$. 
Finally,  repeating the same argument as in the part (1) but with the derivative  $DA:\R\to \R$ instead of $A:\R\to \R$,  we see that for a given $\varepsilon>0$,  there is  a constant $\rho>0$ such  that 
$$\norm{DA (g(a, u)-DA (g(a, u_0))}_{H_c^{3+m,\delta_m}(\R^+\times S^1)}<\varepsilon$$
for all $\abs{a}\leq \frac{1}{2}$  and for $u\in E_m$ satisfying $\abs{u-u_0}_{E_m}<\rho.$
Therefore,
$$\norm{\xi^+-\xi^+_0}_{H^{2+m,\delta_m}(\R^+\times S^1)}<C_m\cdot \varepsilon \cdot \abs{\wh{u}}_{E_m}\cdot \abs{\eta}_{F_m}$$
for $u\in E_m$ satisfying $\abs{u-u_0}_{E_m}<\rho$,  for all  $\wh{u}\in E_m$ and all $\eta\in F_m$.  Since the same estimate also holds for the map $\xi^--\xi^-_0$ ,  we conclude that 
$$\abs{\wt{B}(a, u, \wh{u}, \eta)-\wt{B}(a, u_0, \wh{u}, \eta)}_{F_m}\leq C_m\cdot \varepsilon \cdot \abs{\wh{u}}_{E_m}\cdot \abs{\eta}_{F_m}$$
for $u\in E_m$ satisfying $\abs{u-u_0}_{E_m}<\rho$, for  all $\wh{u}\in E_m$ and all $\eta\in F_m$. The estimate for $a=0$ is obtained  the same way.

The proof of Proposition 4.6 is complete. \hfill $\blacksquare$
\section{Proof of Proposition \ref{POLTER}}\label{POLTERGG}

We recall that 
the Hilbert spaces  $H^{3+m,-\delta_m}(Z_a^\ast,\R^{2n} )$ for $m\geq 0$ consist of maps $u:Z_a^\ast\rightarrow\R^{2n}$ for which the associated maps $v:{\mathbb R}\times S^1\rightarrow \R^{2n} $, defined by $ (s,t)\rightarrow u([s,t])$ have partial derivatives up to order $3+m$ which, if  weighted by $e^{-\delta_m|s-\frac{R}{2}|}$ belong to the space $L^2(\R\times S^1)$. The spaces $H^{2+m,-\delta_m}(Z_a^\ast,\R^{2n})$  are defined analogously.
 The norm of $u\in H^{3+m,-\delta_m}(Z_a^\ast, \R^{2n} )$ is defined  as 
 $$\nr u\nr_{3+m,-\delta_m}^{\ast}=\bigl( \sum_{\abs{\alpha}\leq 3+m }\int_{\R\times S^1}\abs{D^\alpha u}^2e^{-2\delta_m\abs{s-\frac{R}{2}}}\ dsdt\bigr)^{1/2} .$$
The number  $R$ is equal to $R =\varphi (\abs{a})$  where  $\varphi$ is the exponential   gluing profile.
Moreover, the mean value $[u]_a$ of a map $u:Z^\ast_{a}\rightarrow  \R^{2n}$ over the  circle at $\{\frac{R}{2}\} \times S^1$
is   defined as 
$$
[u]_a:=\int_{S^1} u\left(\left[\frac{R}{2},t\right]\right) dt.
$$

If $u\in H^{3+m,-\delta_m}(Z_a^\ast, \R^{2n} )$, then the norm $\nr u\nr^{\ast}_{3+m,-\delta_m}$ is equal to 
\begin{equation*}
\begin{split}
\nr u\nr^{\ast}_{3+m,-\delta_m}&=\bigl(\sum_{\abs{\alpha}\leq 3+m}\int_{\R\times S^1}\abs{D^\alpha u(s,t)}^2e^{-2\delta_m \abs{s-\frac{R}{2}}}\ dsdt\bigr)^{1/2}\\
&=\bigl(\sum_{\abs{\alpha}\leq 3+m}\int_{\R\times S^1}\abs{D^\alpha u(s+R/2,t)}^2e^{-2\delta_m \abs{s}}\ dsdt\bigr)^{1/2}\\
&=\norm{u(\cdot +R/2, \cdot)}_{H^{3+m,-\delta_m}(\R\times S^1, \R^{2n})},
\end{split}
\end{equation*}
where  the  $H^{3+m,-\delta_m}(\R\times S^1, \R^{2n})$-norm of the map $u:\R\times S^1\to \R^{2n}$ is defined by 
$$\norm{u}^2_{H^{3+m,-\delta_m}(\R\times S^1, \R^{2n})}=
\sum_{\abs{\alpha}\leq 3+m}\int_{\R\times S^1}\abs{D^\alpha u(s,t)}^2e^{-2\delta_m \abs{s}}\ dsdt.$$
Thus, the map $I_a:H^{3+m,-\delta_m}(Z_a^\ast,\R^{2n} )\to H^{3+m,-\delta_m}(\R\times S^1,\R^{2n} )$, defined by $u(\cdot +R/2, \cdot )$, is a linear isometry.  The same holds for the spaces $H^{2+m,-\delta_m}(Z_a^\ast,\R^{2n} )$ and $H^{2+m,-\delta_m}(\R\times S^1),\R^{2n} )$. Moreover, if $v(s, t)=u(s+R/2, t)$, then $v-[v]_0=u-[u]_a$ where $R=\varphi (\abs{a})$.  Consequently, it suffices to study 
 the standard Cauchy-Riemann operator
$$
\ov{\partial}_0:H^{3+m,-\delta_m}({\mathbb R}\times S^1,{\mathbb R}^{2n})\rightarrow H^{m+2,-\delta_m}({\mathbb R}\times S^1,{\mathbb R}^{2n}).
$$
It is known that this operator is a surjective Fredholm operator whose kernel $N$ consists of constant functions.
The subspace $X\subset H^{3+m,-\delta_m}({\mathbb R}\times S^1,{\mathbb R}^{2n})$ of maps $u$ satisfying 
$$[u]_0:=\int_{S^1}u(0, t)\ dt=0$$ 
is closed and an  algebraic complement of the kernel $N$. Since $N$ is of  finite dimension, the subspace $X$ is a topological complement of $N$.  Applying  the open mapping theorem, 
there exists a constant $C_m$ such that 
$$
\frac{1}{C_m} \cdot \norm{u}_{m+3,-\delta_m}\leq \norm{\ov{\partial}_0u}_{m+2,-\delta_m}\leq C_m\cdot \norm{u}_{m+3,-\delta_m}
$$
for every $u\in X$.   Now the desired estimate is a consequence of the fact that if $u\in H^{3+m,-\delta_m}({\mathbb R}\times S^1,{\mathbb R}^{2n})$, then the map 
$u-[u]_0$ belongs to $X$. This completes the proof of Proposition \ref{POLTER}.

\section{Proof of Lemma \ref{TORTE} }\label{Lemma-4.18}
We recall the statement for convenience.
\begin{LXX}
If $\varepsilon>0$,  there exists $R_\varepsilon>0$ such that for all $R\geq R_\varepsilon$ and for $i=0,1$ the estimate
$$
\norm{ \gamma_R\cdot (J(v)-J(0))h_t}_{H^{m+2+i,\delta_{m+i}}(\R\times S^1,\R^{2n})}\leq \varepsilon \cdot\norm{h}_{H^{m+3+i,\delta_{m+i}}_c(\R\times S^1,\R^{2n})}
$$
holds for  all $h\in  H^{m+3+i,\delta_{m+i}}_c(\R\times S^1,\R^{2n})$.
\end{LXX}

\begin{proof}  
We only consider the case $i=0$; the arguments for the case $i=1$  are the  same. We may assume without loss of generality that  $J$ and $h$ are real-valued. Since we have uniform  bounds on the derivatives of $\gamma_R$ up to any order and $\gamma_R$ vanishes for $s\leq R+1$,  the norm 
$$\norm{ \gamma_R\cdot (J(v)-J(0))h_t}_{H^{m+2,\delta_{m}}}^2$$
is estimated from above  by a constant independent of $R$ times a linear combination of  integrals  of the form 
\begin{equation}\label{int1}
\int_{[R,\infty)\times S^1}\abs{D^{\alpha} \bigl[ \bigl( J(v)-J(0)\bigr)h_t\bigr] }^2e^{2\delta_m\abs{s}}\ dsdt
\end{equation}
where 
the multi-indices $\alpha =(\alpha_1,\alpha_2)$  are of  order $\abs{\alpha}\leq m+2.$ 
Since 
\begin{equation*}
D^{\alpha} \bigl[ \bigl( J(v)-J(0)\bigr)h_t\bigr] =\sum_{\beta\leq \alpha} \binom{\alpha}{\beta} D^{\beta}\bigl( J(v)-J(0)\bigr)\cdot D^{\alpha-\beta}h_t ,
\end{equation*}
the integrals \eqref{int1} are estimated by a  constant independent of $R$,  times a  linear combination of integrals of the form 
\begin{equation}\label{int2}
\int_{[R,\infty)\times S^1}\abs{D^{\beta}\bigl( J(v)-J(0)\bigr)}^2\cdot \abs{D^{\alpha-\beta}h_t}^2e^{2\delta_ms}\ ds dt
\end{equation}
with multi-indices $\alpha$ and $\beta$ satisfying $\beta\leq \alpha$ and $\abs{\alpha}\leq m+2.$ If $\abs{\beta}=0$, then using that $J$ is smooth and $v(s, t)$ approaches $0$ as $s\rightarrow \infty$  uniformly in $t$, the above integral can be estimated by 
$$C_R\int_{[R, \infty)\times S^1}\abs{D^{\alpha}h_t}^2e^{\delta_m s}\ ds dt\leq C'_R\norm{h}^2_{H_c^{m+3,\delta_m}(\R\times S^1,\R^{2n})}$$
where the constants $C_R$ and $C'_R$ converge to $0$ as $R\to \infty.$  We next assume that $1\leq \abs{\beta}\leq m+2$.  Then 
$D^{\beta}\bigl( J(v)-J(0)\bigr)=D^{\beta}\bigl( J(v)\bigr).$  The derivative $D^\beta (J(v))$ is a linear combination of expressions of the form 
\begin{equation}\label{exp1}
(D^k_{I}J)(v)\cdot D^{\gamma_{i_1}}v_{i_1}\cdot \cdots  \cdot D^{\gamma_{i_k}}v_{i_k},
\end{equation}
where $D^k_{I}J$ stands for the $k$-order derivative with respect to the variables $\{i_1,\ldots ,i_k\}$ and $v_i$'s are the components of the vector $v$. Moreover, $k\leq \abs{\beta}$ and the multi-indices $\gamma_{i_1},\ldots, \gamma_{i_k}$ satisfy $\abs{\gamma_{i_1}}+\ldots +\abs{\gamma_{i_k}}=\abs{\beta}$.  Hence 
using that  $v(s, t)$ approaches $0$ as $s\rightarrow \infty$  uniformly in $t$ and that $J$ is smooth, we see 
that for multi-indices $\beta$ satisfying $1\leq \abs{\beta}\leq m+2$ the integrals in \eqref{int2} are  bounded above by integrals of the form 
\begin{equation}\label{int3}
C_R\int_{[R,\infty)\times S^1} \abs{D^{\gamma_1}v_1}^{2} \cdots   \abs{D^{\gamma_{2n}}v_{2n}}^2\cdot \abs{D^{\alpha-\beta}h_t}^2e^{2\delta_ms}\ ds dt
\end{equation}
with a constant $C_R$ tending to $0$ as $R\to \infty$.
To estimate the integrals in \eqref{int3} we first assume that $\abs{\gamma_{i_1}}, \ldots, \abs{\gamma_{i_k}}$ are less or equal to $m+1$. We recall that $v$ is a fixed map belonging to $H^{m+3,\delta_m}(\R^+\times S^1, \R^{2n})$. This is equivalent to saying that $e^{\delta_ms}v$ belongs to $H^{m+3}(\R^+\times S^1, \R^{2n})$. In view of the Sobolev embedding theorem, we conclude that  $e^{\delta_ms}v$ is of class $C^{m+1}((1, \infty)\times S^1)$ and there exists a constant $C$ such that 
$\max_{0\leq \abs{\gamma}\leq m+1}\sup_{(1, \infty)\times S^1}\abs{D^\gamma (e^{\delta_m s}v)}\leq C\norm{v}_{H^{m+3,\delta_m}((1,\infty)\times S^1,\R^{2n})}$. From this we conclude that  the partial derivatives $D^{\gamma_i}v_i$ are uniformly  bounded by a constant independent of $R>1$. If $\abs{\gamma_{i_1}}, \ldots, \abs{\gamma_{i_k}}\leq m+1$,  then
 the integrals in \eqref{int3} are estimated from above by 
$$
K_R\int_{[R,\infty )\times S^1}\abs{D^{\alpha-\beta}h_t}^2e^{2\delta_ms}\ ds dt\leq C'_R\norm{h}_{H^{m+3,\delta_m}(\R\times S^1, \R^{2n})}.
$$
It remains to consider the  case in which there is exactly one multi-index, say $\gamma_{i}$,  for which $\abs{\gamma_{i}}=m+2$. In this case $\alpha=\beta=\gamma_{i}$ and the integral \eqref{int3} takes the form 
\begin{equation}\label{int4}
C_R\int_{[R,\infty)\times S^1}\abs{D^{\alpha}v_i}^2\cdot \abs{h_t}^2e^{2\delta_m s}\ dsdt.
\end{equation}
Since the derivative $h_t$ belongs to $H^{m+2,\delta_m}(\R^+\times S^1, \R^{2n})$, the map 
$e^{\delta_m s}h_t$ belongs to $H^{m+2}(\R^+\times S^1, \R^{2n})$. In particular,  in view of the Sobolev embedding theorem, $e^{\delta_m s}h_t$ belongs to 
the space $C^0(\R^+\times S^1, \R^{2n})$ and 
$$
\norm{e^{\delta_m s}h_t}_{C^0(\R^+\times S^1)}\leq C\norm{e^{\delta_ms}h_t}_{H^{m+2}(\R^+\times S^1,\R)}\leq C\norm{h}_{H_c^{m+3,\delta_m}(\R\times S^1)}.
$$
Consequently, 
\begin{equation}\label{int5}
\begin{split}
C_R\int_{[R,\infty)\times S^1} &\abs{D^{\alpha} v_i}^{2} \cdot \abs{h_t}^2e^{2\delta_ms}\ ds dt\\
&\leq C'_R\norm{h_t}_{C^0(\R^+\times S^1)} \norm{v}_{H^{m+3}(\R^+\times S^1, \R^{2n})}\\
&\leq  C'_R\norm{v}_{H^{m+3}(\R^+\times S^1, \R^{2n})}\cdot \norm{h}_{H^{m+3,\delta_m}_c(\R\times S^1,\R^{2n})}.
\end{split}
\end{equation}
Summing up, each of the integrals \eqref{int1} can be estimated from above by $C_R\norm{h}_{H^{m+3,\delta_m}_c(\R\times S^1,\R^{2n})}$ in which  $C_R\to 0$ as $R\to \infty$. This completes the proof of Lemma \ref{TORTE}.
\end{proof}

\section{Orientations for Sc-Fredholm Sections}\label{orientations-abstract}
The Appendices \ref{orientations-abstract} and \ref{orientations} are devoted to the orientation in Gromov-Witten theory.

\subsection{Basic Ideas}
We start with some general considerations well known from index theory, and refer to \cite{FH} and \cite{HWZ10,HWZ11}  for more details. 
Here we shall restrict ourselves to the case of Fredholm operators in Hilbert spaces. The Banach space case is technically a little bit more complicated and the basic constructions can be found in \cite{HWZ10}.

If $E$ is a finite-dimensional vector space, we denote its dual space by $E^\ast$.  The maximal wedge of $E$ is the $1$-dimensional vector space 
$$
\Lambda^{\textrm{max}} E =\Lambda^{\dim(E)} E.
$$
In case $E=\{0\}$, we put  $\Lambda^{\textrm{max}} E={\mathbb R}$.  
The dual ${\mathbb R}^\ast$ of ${\mathbb R}$ is  canonically identified with ${\mathbb R}$ by the isomorphism ${\mathbb R}^\ast\rightarrow {\mathbb R}$, $\lambda\rightarrow \lambda(1)$. 
Henceforth ${\mathbb R}^\ast={\mathbb R}$. Moreover, there is a canonical  isomorphism
$$
\gamma:\Lambda^{\textrm{max}} ( E^\ast) \rightarrow (\Lambda^{\textrm{max}} E)^\ast
$$
defined by 
$$\gamma( e_1^\ast\wedge..\wedge e_n^\ast)(a_1\wedge \ldots \wedge a_n)=\det (e_i^\ast(a_j)),
$$
where $n=\dim(E)$. If $n=0$, then 
$\Lambda^{\textrm{max}} ( E^\ast) =\Lambda^{\textrm{max}} (0)=\R=\R^\ast$ and $\gamma(r)t=r\cdot t$.

\begin{definition}\label{def_determinant_1}
The determinant $\text{det}(T)$ of a bounded linear Fredholm operator $T:E\to F$ between Hilbert spaces is the $1$-dimensional real vector space defined by
$$
\text{det}(T)= (\Lambda^{\textrm{max}} \ker(T))\otimes (\Lambda^{\textrm{max}} \text{coker(T)}^\ast).
$$
\end{definition}

Our aim is to study orientation questions for sc-Fredholm sections. The situation is more complicated than  in the classical case.
In fact there are many possibilities for natural constructions
and different conventions can  lead to different signs in the  orientations of the  moduli spaces. 
In the SFT the orientation analysis requires many 
conventions due to a wealth of additional  structures. This analysis will be carried out in the  forthcoming paper \cite{HWZ5} and,  of course, it covers GW as  a special case. However, the special case of the Gromov-Witten theory allows a somewhat  simpler approach.

Given an exact sequence 
$$
{\bf E:}\ \ 0\rightarrow A\xrightarrow{\alpha}B\xrightarrow{\beta}C\xrightarrow{\gamma} D\rightarrow 0
$$
of finite dimensional vector spaces, we shall construct two associated natural isomorphisms 
\begin{align*}
\Phi:&=\Phi_{\bf E}:\, \Lambda^{\textrm{max}} A\otimes \Lambda^{\textrm{max}} D^\ast \rightarrow \Lambda^{\textrm{max}} B\otimes\Lambda^{\textrm{max}} C^\ast,\\
\Psi:&=\Psi_{\bf E}:\, \Lambda^{\textrm{max}} C\otimes \bigl(\Lambda^{\textrm{max}} A \oplus \Lambda^{\textrm{max}} D^\ast\bigr) \rightarrow \Lambda^{\textrm{max}} B.
\end{align*}

Let $\dim(A)=n$, $\dim(B)=m$, $\dim(C)=k$,  and $\dim(D)=l$. 
The exactness of ({\bf E}) implies the equality 
$$m-n= k-l.$$
\begin{definition}\label{lem_isom_phi_1}
The  isomorphism
$$
\Phi:=\Phi_{\bf E}\, :\Lambda^{\textrm{max}} A\otimes \Lambda^{\textrm{max}} D^\ast \rightarrow \Lambda^{\textrm{max}} B\otimes\Lambda^{\textrm{max}} C^\ast
$$
associated with the exact sequence $({\bf E})$  is constructed as follows. We fix a basis $d_1^\ast,\ldots ,d^\ast_l$ of $D^\ast$. If $A\neq \{0\}$, given $h\in \Lambda^{\textrm{max}} A\otimes \Lambda^{\textrm{max}} D^\ast$, we choose vectors $a_1,\ldots, a_n\in A$ such that 
$$
h=(a_1\wedge\ldots \wedge a_n)\otimes (d_1^\ast\wedge\ldots \wedge d_l^\ast).
$$
Next we choose  linearly independent vectors $b_1,\ldots ,b_{m-n}$  which are not in the image of the linear map $\alpha\colon   A\to B$ and define the vectors $c_i=\beta (b_i)$ in $C$,  for  $i=1,\ldots ,m-n$. Then $c_1,\ldots ,c_{m-n}$  
are linearly independent in $C$ and span the image of the linear map $\beta\colon  B\to C$. 
We take  dual vectors  $c_1^\ast,\ldots,c_{k-l}^\ast$ in $C^\ast$ so that the square matrix 
$(c_i^\ast(c_j))$ has determinant $1$.
The vectors 
$d_1^\ast\circ\gamma,\ldots ,d_l^\ast\circ\gamma, c_1^\ast,\ldots,c_{m-n}^\ast$
form  a basis in  $C^\ast$ and the map $\Phi$ is defined by 
$$
\Phi(h)= 
(\alpha (a_1)\wedge..\wedge \alpha (a_n)\wedge b_1\wedge..\wedge b_{m-n})\otimes (d_1^\ast\circ \gamma\wedge..\wedge d_l^\ast\circ \gamma\wedge c_1^\ast\wedge..\wedge c_{m-n}^\ast).
$$
If  $A=\{0\}$, we replace $ \alpha (a_1)\wedge\ldots \wedge \alpha (a_n)$ by the real number $r$  which replaces
$a_1\wedge\ldots \wedge a_n$, so that $h=r (d_1^\ast\wedge\ldots \wedge d_l^\ast)$ and then define 
$$\Phi(h)=r( b_1\wedge\ldots \wedge b_{m-n})\otimes (d_1^\ast\circ \gamma\wedge\ldots \wedge d_l^\ast\circ \gamma\wedge c_1^\ast\wedge\ldots \wedge c_{m-n}^\ast).$$
\end{definition}

The  definition of the isomorphism $\Phi_{\bf E}$ does  not depend on the choices involved,  apart from the two convention of the  order in which  we write the right-hand side of $\Phi (h)$.
The two conventions are that $b_1,\ldots ,b_{m-n}$ are listed after $\alpha (a_1),\ldots ,\alpha (a_n)$ and $c^\ast_1,\ldots, c^\ast_{m-n}$ are listed after $d_1^\ast\circ\gamma,\ldots ,d_l^\ast\circ \gamma$.  
Apart from these two conventions, the resulting definition does not depend on the choices involved, see Lemma \ref{very_new_lemma} below.

\begin{definition}\label{lem_isom_phi_2}
The natural isomorphism 
$$\Psi=\Psi_{\bf E}:
\Lambda^{\textrm{max}} C\otimes \bigl( \Lambda^{\textrm{max}}A \otimes 
\Lambda^{\textrm{max}}D^*\bigr)\to 
\Lambda^{\textrm{max}}B$$ 
is constructed as follows. Again we fix a basis 
$d_1^*,\ldots d^*_l\in D^*$. Then we choose non-vanishing elements $c_1^*,\ldots ,c^*_{m-n}\in C^*$, which do not vanish on the image of the map $\beta:B\to C$, such that the functionals  
$$(c_1',\ldots , c_k')=(d_1^*\circ \gamma,\ldots ,
d_l^*\circ \gamma, c_1^*,\ldots , c_{m-n}^*)$$
form a basis of the dual space $C^*$, recalling that 
$k=l+(m-n)$.  Then we choose vectors $b_1,\ldots ,b_{m-n}$ such that the matrix 
$(c_i^*(\beta(b_j))$ has determinant $1$. The functionals 
$(c_1',\ldots ,c_l')=(d_1^*\circ \gamma,\ldots ,
d_l^*\circ \gamma)$ vanish on $\beta(b_1), \ldots , \beta(b_{m-n})$. Then we choose the vectors $c_1,\ldots ,c_l\in C$ such that the matrix $(c_i'(c_j))$, where $i, j=1,\ldots ,l$, has determinant equal to $1$. Denoting by 
$$(c_1,\ldots ,c_k)=(c_1,\ldots ,c_l, \beta(b_1), \ldots , \beta(b_{m-n})$$
the basis of $C$ we have, by construction,
$$\det \bigl(c_i'(c_j)\bigr)_{1\leq i,j\leq k}=1.$$
Since we can choose the vectors $a_1,\ldots, a_n\in A$ arbitrarily, every vector 
$h\in \Lambda^{\textrm{max}}C\otimes \bigl(
 \Lambda^{\textrm{max}}A\otimes  \Lambda^{\textrm{max}}D^*)$ can be represented in the form 
 \begin{eqnarray}\label{must}
 h=(c_1\wedge \ldots \wedge c_k)\otimes (a_1\wedge \ldots \wedge a_n)\otimes (d_1^*\wedge \ldots \wedge d_l^*)
 \end{eqnarray}
 for a suitable choice of $a_1,\ldots ,a_n\in A$. Finally, the isomorphism $\Psi$ is defined by 
 $$\Psi(h)=(\alpha (a_1)\wedge \ldots \wedge \alpha (a_n))\wedge b_1\wedge \ldots \wedge b_{k-l}.$$
 The case $A=\{0\}$ is dealt with as in the previous Definition \ref{lem_isom_phi_1}.
 \end{definition}
An alternative definition of $\Psi_{\bf E}$ is given by  the composition
 $$
 \Lambda^{\textrm{max}}C\otimes (\Lambda^{\textrm{max}}\otimes\Lambda^{\textrm{max}} D^\ast)\xrightarrow{Id\otimes\Phi_{\bf E}} \Lambda^{\textrm{max}} C\otimes\Lambda^{\textrm{max}}B\otimes\Lambda^{\textrm{max}}C^\ast\rightarrow \Lambda^{\textrm{max}} B.
 $$
The last map is the  isomorphism 
 $$
 (c_1\wedge\ldots \wedge c_k)\otimes(b_1\wedge\ldots \wedge b_m) \otimes (c_1^\ast\wedge\ldots\wedge c_k^\ast)\mapsto  \det(c_j^\ast(c_i))\cdot (b_1\wedge\ldots \wedge b_m).
 $$
 To see the last assertion we start with the element $h$ in \eqref{must}. Writing 
 $$
 h= (c_1\wedge \ldots \wedge c_k)\otimes h', 
 $$
 we compute with $\Phi={\Phi}_{\bf E}$ and the vectors previously constructed 
 $$
 \Phi(h')=(\alpha (a_1)\wedge\ldots\wedge \alpha (a_n)\wedge b_1\wedge\ldots\wedge b_{m-n})\otimes (c_1'\wedge\ldots\wedge c_k').
  $$
  Finally,  we observe that $(c_1\wedge\ldots \wedge c_k)\otimes \Phi(h')$ is mapped by the natural isomorphism onto $(\alpha (a_1)\wedge\ldots\wedge \alpha (a_n)\wedge b_1\wedge\ldots\wedge b_{m-n})$. This  proves our assertion noting that $m-n=k-l$.
\begin{lemma}\label{very_new_lemma}
The definitions of the maps $\Phi=\Phi_{\bf E}$ and $\Psi=\Psi_E$ are independent  of the choices involved. Moreover, the maps $\Phi$ and $\Psi$ are isomorphisms.
\end{lemma}
\begin{proof}
In view of the previous discussion we only  need to consider  the map $\Phi$  and prove that it is well-defined  in the case $A\neq \{0\}$. 
Assume we have chosen  a basis $a_1,\ldots,a_n$ for $A$, a basis  $d_1^\ast,\ldots,d^\ast_l$ for $D^\ast$, and linearly independent vectors
$b_1,\ldots,b_{m-n}$ in $B$ which are not in the image of $\alpha$. With these choices fixed, we choose  $c_1^\ast,\ldots,c^\ast_{m-n}$ in $C^\ast$ such  that the matrix 
$(c_i^\ast(\beta(b_j))$ has determinant $1$.  Consider the vector $H$ given by
$$
(\alpha(a_1)\wedge\ldots\wedge\alpha(a_n)\wedge b_1\ldots\wedge b_{m-n})\otimes(d_1^\ast\circ \gamma\wedge\ldots\wedge d^\ast_l\circ\gamma\wedge c_1^\ast\wedge\ldots\wedge c^\ast_{m-n}).
$$
Assuming  we would have made a different choice for the vectors $c_1^\ast,\ldots,c_{m-n}^\ast$, say $e_1^\ast,\ldots,e^\ast_{m-n}$, we can write
$$
e^\ast_j = \sum_{i=1}^{m-n} \lambda_{ji} c_i^\ast + \Delta_j,
$$
where $\Delta_j$ is a linear combination of the vectors $d_1^\ast\circ\gamma,\ldots ,d_l^\ast\circ\gamma$.  Using 
$\gamma\circ\beta=0$, we conclude that 
$$
e^\ast_j(\beta(b_\tau))= \sum_{i=1}^{m-n} \lambda_{ji} c_i^\ast(\beta(b_\tau))
$$
for all $1\leq \tau\leq m-n$. Since by assumption $\det (e^\ast_j(\beta(b_\tau))=\det (e^\ast_j(\beta(b_\tau))=1$, the matrix $(\lambda_{ji})$ has determinant $1$. Therefore,  
$$
d_1^\ast\circ \gamma\wedge\ldots \wedge d^\ast_l\circ\gamma\wedge c_1^\ast\wedge\ldots \wedge c^\ast_{m-n}
=d_1^\ast\circ \gamma\wedge..\wedge d^\ast_l\circ\gamma\wedge e_1^\ast\wedge(e^\ast_j(\beta(b_\tau)) \wedge e^\ast_{m-n}.
$$
Consequently,  the vector $H$ does not depend on the choice of $c_1^\ast,\ldots,c_{m-n}^\ast$, which is the last choice in the construction. Next we show that the vector $H$ is independent of the choice of $b_1,\ldots ,b_{m-n}$. Assume that we have made a different choice for the vectors $b_1,\ldots, b_{m-n}$ which we again denote by  $e_1,\ldots, e_{m-n}$.  The vectors 
$e_1,\ldots, e_{m-n}$ do not belong to the image of the map $\alpha$ and 
$$
e_j =\sum_{i=1}^{m-n} \lambda_{ji} b_i +\Delta_j
$$
for all $1\leq j\leq m-n$ and where $\Delta_j$ is a vector in  the image of $\alpha$.  Denoting by  $c_1^\ast,\ldots ,c_{k-l}^\ast\in C^\ast$ the vectors associated with $b_1,\ldots ,b_{m-n}$ so that the matrix $(c_\tau (\beta (b_i))$ has determinant $1$,  we find, using $\beta\circ \alpha=0$,  that 
$$
c_\tau^\ast(\beta(e_j))=\sum_{i=1}^{m-n} \lambda_{ji} c_\tau^\ast(\beta(b_i)), 
$$
which implies
$$
\det(c^\ast_\tau\beta(e_j))= \det(\lambda_{ji}).
$$ 
Then, abbreviating $r=\det (\lambda_{ji})$,  we can take for the choice $e_1,\ldots, e_{m-n}$ the dual vectors
\begin{equation}\label{must1}
\frac{1}{r}c_1^\ast,c_2^\ast,\ldots ,c_{m-n}^\ast.
\end{equation}
With the choice above we observe that 
$$
r\cdot(\alpha(a_1)\wedge\ldots \wedge \alpha(a_n)\wedge b_1\ldots \wedge b_{m-n}=(\alpha(a_1)\wedge\ldots \wedge \alpha(a_n)\wedge b_1'\ldots \wedge b_{m-n}'
$$
and consequently, using that our new choice for the $c_i^\ast$ is \eqref{must1}, we conclude that the result does not depend on the the third choice
in our construction, namely the $b_i$, either. Finally,  we need to understand the dependency on the choice of $a_1,\ldots,a_n$ and $d_1^\ast,\ldots,d_l^\ast$.
If $a_1',\ldots ,a_n'$ is a second choice for $a_1,\ldots ,a_n$,  there is $r$ such  that
$$
\frac{1}{r}\cdot (a_1\wedge\ldots\wedge a_n) =a_1'\wedge\ldots \wedge a_n'.
$$
Similarly, if  $e_1^\ast,\ldots,e_l^\ast$ is a different choice for $d_1^\ast,\ldots ,d_l^\ast$,  there is $s$ such that 
$$
e_1^\ast\wedge\ldots \wedge e_l^\ast = s\cdot d_1^\ast\wedge\ldots \wedge d_l^\ast.
$$
Assuming that
$$
(a_1\wedge\ldots \wedge a_n)\otimes(d_1^\ast\wedge\ldots\wedge d_l^\ast)=(a_1'\wedge\ldots\wedge a_n')\otimes (e_1^\ast\wedge\ldots\wedge e_l^\ast)
$$
we infer that $r=s$. Therefore,  we conclude, abbreviating $e^\ast=e_1^\ast\wedge\ldots\wedge e_l^\ast$, $a=a_1\wedge\ldots \wedge a_n$, $a'=a_1'\wedge\ldots\wedge a_n'$,
and $d^\ast=d_1^\ast\wedge\ldots\wedge d_l^\ast$,  that the following holds with the obvious abbreviations
\begin{equation*}
\begin{split}
&(\alpha(a)\wedge b_1\wedge\ldots \wedge b_{m-n})\otimes(d^\ast\circ \gamma\wedge c_1^\ast\wedge\ldots\wedge c_{m-n}^\ast)\\
&\phantom{=}=(\alpha(a')\wedge b_1\wedge\ldots\wedge b_{m-n})\otimes(e^\ast\circ \gamma\wedge c_1^\ast\wedge..\wedge c_{m-n}^\ast).
\end{split}
\end{equation*}
This concludes the proof that the definition of $\Phi_{\bf E}$ is independent of the choices involved.
\end{proof}

The Definitions  \ref{lem_isom_phi_1} and \ref{lem_isom_phi_2} will be used in the following constructions. 

If $T:E\rightarrow F$ is  a 
linear Fredholm operator between Hilbert spaces, we 
denote by $\Pi_T$ the collection of all orthogonal  projections $P:F\rightarrow F$ satisfying 
\begin{itemize}
\item[(1)] $\text{dim}(F/R(P))<\infty$.
\item[(2)] $R(P\circ T)= R(P)$.
\end{itemize}
Here we denote by  $R(A)$  the range of the linear operator $A$.

We view the composition $P\circ T$ as an operator $E\rightarrow F$ so that $T$ and $P\circ T$ have the same index.
We introduce a partial ordering $\leq $ on $\Pi_T$ by defining  
$P\leq P'$ if $P=P'\circ P=P\circ P'$. The following holds.
\begin{lemma}
Let $T:E\rightarrow F$ be a Fredholm operator between Hilbert spaces. Assume that $P,P'\in\Pi_T$. Then there exists $Q\in\Pi_T$ with $Q\leq P$ and $Q\leq P'$.
\end{lemma}
\begin{proof}
We abbreviate  by $H$ the subspace $H=R(P)\cap R(P')$ of $F$. The subspace $H$ has finite codimension in $F$  Let $Q$ be the orthogonal projection onto $H$. Denoting by 
$A$ the orthogonal complement  of $H$ in $R(P)$ and by $B$  the orthogonal complement of $H$ in $R(P')$,  we obtain the orthogonal decompositions
$$
F=A\oplus H\oplus R(P)^\perp \ \text{and}\ F=B\oplus H\oplus R(P')^\perp.
$$
Accordingly we can write an element $f\in F$ as $f=a+h+z$ where $a\in A$, $h\in H$,  and $z\in R(P)^\perp$.  Then 
$$
Q(f)=PQ(f) =QP(f).
$$
Similarly, 
$$
Q=QP'=P'Q.
$$
Since $PT\colon E\rightarrow R(P)$ is surjective,  the same holds for $QT=QPT$. Hence  $R(Q)=R(QT)$ and $Q\in \Pi_T$, as claimed.
\end{proof}

Given $P\in \Pi_T$, we  consider the sequence
$$
{\bf (P)}\qquad 0\rightarrow \ker(T)\xrightarrow{j}\ker(PT)\xrightarrow{\Phi_T^P}F/R(P)\xrightarrow{\pi}\text{coker}(T)\rightarrow 0
$$
in which  $j$ is the inclusion map, and 
$\Phi^P_T(x) = T(x) +R(P)$, and  $\pi$ is the surjection  defined by $\pi(y+R(P))=(I-P)y+R(T)$.
\begin{lemma}\label{lem_seq_exact_1}
The sequence ${\bf (P)}$ is exact.
\end{lemma}
\begin{proof}
The inclusion $j$ is injective and $\Phi^P_T\circ j=0$. If $\Phi^P_T(x)=0$, where $PT(x)=0$, then  
$T(x)\in R(P)$. This  implies $T(x)=PT(x)=0$, so that $x$ is in the image of $j$. If $x\in \ker(PT)$ we see that
$\pi\circ \Phi^P_T(x)=\pi(Tx+R(P))=(I-P)T(x) +R(T)=T(x)+R(T)=R(T)$. Hence $\pi\circ \Phi^P_T=0$.
If $\pi(y+R(P))=0$ then by definition $(I-P)y+R(T)=R(T)$, which implies $(I-P)y\in R(T)$. Hence we find $(I-P)y=T(x)$
for some $x\in E$. Applying $P$ to both sides of this equation shows that $PT(x)=0$, and therefore $(I-P)y = (I-P)T(x)$. This  implies $\Phi^P_T(x) =(I-P)T(x)+R(P)=(I-P)y +R(P)=y+R(P)$.
Finally we note that $\pi$ is surjective. Indeed, given $y+R(T)$, choose $x\in E$ with $PT(x)=Py$. Then
$y+R(T)=y-T(x) +R(T)= (I-P)(y-T(x)) +R(T)$. Then we compute
$\pi(y-T(x)+R(P))= y-T(x)+R(T)=y+R(T)$ showing that $\pi$ is surjective. Hence the sequence is exact.

\end{proof}
We now assume that $P\leq Q$ so that $P=Q\circ P=P\circ Q$. Then $P\in \Pi_{QT}$ and one verifies as in Lemma \ref{lem_seq_exact_1} that the following sequences are exact.
\begin{gather*}
0\rightarrow \ker(T)\xrightarrow{j}\ker(QT)\xrightarrow{\Phi_T^Q}F/R(Q)\xrightarrow{\pi}\text{coker}(T)\rightarrow 0\\
0\rightarrow \ker(QT)\xrightarrow{j}\ker(PT)\xrightarrow{\Phi_{QT}^P}F/R(P)\xrightarrow{\pi}\text{coker}(QT)\rightarrow 0\\
0\rightarrow \ker(T)\xrightarrow{j}\ker(PT)\xrightarrow{\Phi_T^P}F/R(P)\xrightarrow{\pi}\text{coker}(T)\rightarrow 0.
\end{gather*}
In view of Definition \ref{def_determinant_1} of the determinant we deduce from Definition \ref{lem_isom_phi_1} and Lemma \ref{very_new_lemma} 
the following isomorphisms.
\begin{gather*}
\gamma_T^Q: \text{det}(T)\rightarrow \text{det}(QT)\\
\gamma_{QT}^P:\text{det}(QT)\rightarrow \text{det}(PT)\\
\gamma_{T}^P:\text{det}(T)\rightarrow\text{det}(PT).
\end{gather*}
\begin{lemma}\label{revision_lem_1}
If  $T:E\rightarrow F$ is a linear  Fredholm operator between Hilbert spaces and if the projections $P,Q\in\Pi_T$
satisfy  $P\leq Q$ we have the following relationship, 
$$
\gamma^P_{QT}\circ \gamma^Q_T=\gamma^P_T.
$$
\end{lemma}
\begin{proof}
We set  $K=\ker(T)$ and denote by $A\subset \ker(QT)$ the  orthogonal complement of $K$ in $ \ker(QT)$ and by $B$ the  orthogonal complement of $\ker(QT)$ in $\ker(PT)$.
Then we have the orthogonal decompositions 
$$
\ker(T)=K,\ \ker(QT)= K\oplus A,\ \text{and}\ \ker(PT) =K\oplus A\oplus B.
$$
Note that we have a canonical identification
$$
F/R(P) = F/R(Q)\oplus R(Q)/R(P).
$$
To see this we  note that every class in $F/R(P)$ has a unique representative $f+R(P)$ with $f\in R(P)^\perp$.
Then $f$ has a unique decomposition $f=q+p$ where  $q\in R(Q)^\perp $ and $p\in R(Q)\cap R(P)^\perp$. Then we  associate with 
 $f+R(P)$ the element $(q+R(Q),p+R(P))$. Consider the map $A\rightarrow F/R(Q)\colon a\mapsto T(a)+R(Q)$.
This map is injective and its image can be written as $L/R(Q)$. Then we can as before canonically  identify $F/R(Q)=F/L\oplus L/R(Q)$.
We introduce the following abbreviations $X=F/L$, $Y=L/R(Q)$, and $Z=R(Q)/R(P)$. 
We can identify the exact sequence associated with $T$ and $P$ with the exact sequence
$$
0\rightarrow K\xrightarrow{j} K\oplus A\oplus B\xrightarrow{\Phi}X\oplus Y\oplus Z\xrightarrow{\pi} F/R(T)\rightarrow  0.
$$
Here $j(k)=(k,0,0)$, $\Phi(k,a,b)=(0,T(a)+R(Q),QT(b)+R(P))$ and $\pi(g+L,f+R(Q),e+R(P))= (Id-Q)g+R(T)$. 
The maps $B\rightarrow Z: b\rightarrow QT(b)+R(P)$ and $A\rightarrow Y:a\rightarrow T(a)+R(Q)$ are bijections.
The exact sequence associated with  $T$ and $Q$
is then given by
$$
0\rightarrow K\xrightarrow{j'} K\oplus A \xrightarrow{\Phi'} X\oplus Y\xrightarrow{\pi'} F/R(T)\rightarrow 0.
$$
Here $j'(k)=(k,0)$, $\Phi'(k,a)=(0,T(a)+R(Q))$, and $\pi'(g+L,f+R(Q))=(Id-Q)g + R(T)$.  Finally the exact sequence associated with  $ QT$ and $P$ is given by
$$
0\rightarrow K\oplus A\xrightarrow{j''} K\oplus A\oplus B\xrightarrow{\Phi''} X\oplus Y \oplus Z\xrightarrow{\pi''} X\oplus Y\rightarrow 0,
$$
where $j''(k,a)=(k,a,0)$, $\Phi''(k,a,b)=(0,0,QT(b)+R(P))$ and $\pi''(g+L,f+R(Q),e+R(P))=(g+L, f+R(Q))$.
Consider first the sequence associated withe  $T$ and $Q$. Take a basis $k_1,\ldots,k_l$ for $K$ and a basis $d_1^\ast,\ldots,d_n^\ast$
for $(F/R(T))^\ast$. Applying $j'$ to the $k_i$,  we obtain
$$
\ov{k}_1=(k_1,0),\ldots,\ov{k}_l=(k_\ell,0)
$$
and pulling back by $\pi'$ the $d_j^\ast$,  we obtain
$$
\ov{x}_1^\ast = (x_1^\ast,0),\ldots, \ov{x}_n^\ast =(x_n^\ast,0).
$$
Using that $A\rightarrow Y:a\rightarrow T(a)+R(Q)$ is a bijection,  we take a basis $a_1,\ldots ,a_m$ for $A$ and a dual basis $y_1^\ast,\ldots,y^\ast_m$ for $Y^\ast$ associated with  the image basis of the $a_1,\ldots,a_m$. Then
$$
\det (y^\ast_j(T(a_i)+R(Q)))=1.
$$
Defining  $\ov{a}_j=(0,a_j)$ and $\ov{y}_j^\ast=(0,y_j^\ast)$, the vectors 
$$
\ov{k}_1,\ldots ,\ov{k}_l,\ov{a}_1,\ldots,\ov{a}_m
$$
for a basis for $K\oplus A$ and   the vectors  
$$
\ov{x}_1^\ast,\ldots,\ov{x}_n^\ast,\ov{y}_1^\ast,\ldots,\ov{y}_m^\ast
$$
form a basis
for $(X\oplus Y)^\ast$. We observe that these two bases provide the input data for the diagram associated with $QT$ and $P$. We also note that
the map $\Phi"$ induces the bijection $B\rightarrow Z:b\rightarrow QT(b)+R(P)$.  Hence after fixing a basis $b_1,\ldots,b_q$ for $B$
and a dual basis for $Z^\ast$ for the image basis we obtain a basis for $K\oplus A\oplus B$ and for $X\oplus Y\oplus Z$.
Wedging the vectors of each of these bases together and then taking the tensor product is the image in $\det(PT)$
by mapping $h=(k_1\wedge\ldots \wedge k_l)\otimes (x_1^\ast\wedge\ldots\wedge x_n^\ast)$ first under $\gamma_T^Q$ to an element
in $\det(QT)$ and then under $\gamma^P_{QT}$ to $\det(PT)$. Of course, we could have alternatively taken the exact sequence 
associated with $T$ and $P$ and started with $k_1,\ldots,k_l$ and $x_1^\ast,\ldots,x_n^\ast$ and then make the same choice of vectors
for $A$, $B$, $Y$, and $Z$. The result would be the same. This completes the proof.
\end{proof}

\subsection{The Line Bundle Structure on $\text{DET}(E,F)$}
We assume that $E$ and $F$ are two Hilbert spaces and   consider the open subset ${\mathcal Fred} (E, F)$ of 
${\mathcal L}(E, F)$ consisting of all linear Fredholm operators.  We  introduce  the line bundle $\text{DET}(E, F)$ by  
$$\text{DET}(E, F)=\bigcup_{T\in { \mathcal Fred}(E, F)}\{T\}\times \det (T).$$
Our aim in this section is to show that
$$
\text{DET}(E,F)\rightarrow {\mathcal Fred}(E,F)
$$
has in  a natural way the structure of topological line bundle. At this point $\text{DET}(E,F)$ is only a set which fibers over ${\mathcal Fred}(E,F)$
and where the fibers have natural structures as one-dimensional real vector spaces.

We define a collection of bijections possessing additional properties as follows.
We choose $T_0\in {\mathcal Fred}(E,F)$ and let $P$ be a projection in $\Pi_{T_0}$. Then we find $\varepsilon_0>0$ such  that
$R(PT)=R(P)=:H$ for all $T$ satisfying  $\norm{T-T_0}<\varepsilon_0$. We fix the orthogonal complement  $X$ of  $\ker(PT_0)$ in $E$ 
so that
$$
E=\ker(PT_0)\oplus X.
$$
We define $B_{\varepsilon}(T_0)=\{T\in {\mathcal L}(E,F)\,\vert \, \norm{T-T_0}<\varepsilon_0\}$ and consider the map
$$
\Gamma:B_{\varepsilon}(T_0)\times  \ker(PT_0)\oplus X\rightarrow H,\quad (T,y,x)\rightarrow PT(y+x).
$$
Replacing $\varepsilon_0$ by $\varepsilon_1\in (0,\varepsilon_0]$,  we may assume that 
$$
X\rightarrow H, \quad x\mapsto PT(y+x)
$$
is surjective for all $y\in \ker(PT_0)$ and  all $\norm{T-T_0}<\varepsilon_1$. This allows us to define a uniquely determined continuous map
$$
\Phi: B_{\varepsilon_1}(T_0)\times \ker(PT_0)\rightarrow X
$$
satisfying
$$
PT(y+\Phi(T,y))=0.
$$
For a fixed $T\in B_{\varepsilon_1}(T_0)$, the map $y\mapsto \Phi (T, y)$ is linear and  $\Phi(T_0,y)=0$ for all $y\in \ker(PT_0)$.

Consider  the subset $\Theta$ of $\{T\in {\mathcal L}(E,F)\ |\ \norm{T-T_0}<\varepsilon_1\}\times E$
consisting of all pairs $(T,k)$ satisfying  $PT(k)=0$. Then $\Theta$ has an induced topology from the ambient space and the map
$$
\Phi:B_{\varepsilon_1}(T_0)\times \ker(PT_0)\rightarrow \Theta, \quad (T,y)\mapsto (T,y+\Phi(T,y))
$$
is a homeomorphism which is  linear in the fibers and whose  inverse is given by $(T,k)\rightarrow (T,Q_0k)$, where $Q_0$ is the projection onto $\ker(PT_0)$
along $X$. This shows the set $\Theta$ is a topological vector bundle in a natural way and we have given a trivialization. 
Trivially $\{T\ |\ \norm{T-T_0}<\varepsilon_1\}\times (F/H)^\ast$ is a smooth vector bundle. This immediately gives us the structure of a topological line bundle
$$
{\mathcal L}:=\bigcup_{\{T\ |\ \norm{T-T_0}<\varepsilon_1\}} \{T\}\times \det(PT)\rightarrow \{T\ |\ \norm{T-T_0}<\varepsilon_1\}.
$$
The collection of maps $\gamma_T^P:\det{T}\rightarrow \det(PT)$ defines a bijection 
$$
\text{DET}(E,F)|\{T\ |\ \norm{T-T_0}<\varepsilon_1\} \rightarrow {\mathcal L}
$$
which is linear on the fibers
and covers the identity. We equip $$\text{DET}(E,F)|\{T\ |\ \norm{T-T_0}<\varepsilon_1\}$$  with the unique structure of a topological line bundle
making the latter map a topological bundle isomorphism. Hence we have proved the following lemma.
\begin{lemma}
Given $T_0\in {\mathcal Fred}(E,F)$ and $P\in \Pi_{T_0}$, there exists  $\varepsilon>0$ such  that the following holds.
\begin{itemize}
\item[(i)] Every $T\in {\mathcal L}(E,F)$ satisfying  $\norm{T-T_0}<\varepsilon$ is Fredholm.
\item[(ii)] $P\in \Pi_{T}$ for all $T$ satisfying  $\norm{T-T_0}<\varepsilon$.
\item[(iii)] The set ${\mathcal L}_{T_0,P,\varepsilon}=\bigcup_{\{T\ |\ \norm{T-T_0}<\varepsilon\}} \{T\}\times \det(PT)$ has in a natural way the structure of a topological line bundle 
$$
{\mathcal L}_{T_0,P,\varepsilon}\rightarrow \{T\ |\ \norm{T-T_0}<\varepsilon\}.
$$
\item[(iv)] There exists a natural bijection associated with the $\gamma_T^P$ which is linear on the fibers and  which makes the following diagram commutative
$$
\begin{CD}
\bigcup_{\{T\ |\ \norm{T-T_0}<\varepsilon\}} \{T\}\times \det(T) @>\Phi_{T_0,P,\varepsilon}>> {\mathcal L}_{T_0,P,\varepsilon}\\
@VVV @VVV\\
\{T\ |\ \norm{T-T_0}<\varepsilon\} @= \{T\ |\ \norm{T-T_0}<\varepsilon\}
\end{CD}
$$
Here $\Phi_{T_0,P,\varepsilon}(T,h)=(T,\gamma_T^P(h))$.
\end{itemize}
\end{lemma}
Let $T_0\in {\mathcal Fred}(E,F)$ and $P,Q\in \Pi_{T_0}$ with $P\leq Q$. We find $\varepsilon>0$ so that every bounded linear operator
$T:E\rightarrow F$ with $\norm{T-T_0}<\varepsilon$ is Fredholm and $P,Q\in\Pi_{T}$.
Using the exact sequences 
$$
0\rightarrow\ker(QT)\xrightarrow{j}\ker(PT)\xrightarrow{\Phi^P_{QT}}F/R(P)\xrightarrow{\pi} F/R(Q)\rightarrow 0
$$
we obtain the family of maps
$$
\gamma^P_{QT}:\det(QT)\rightarrow \det(PT).
$$
Near $T_0$ we have the topological vector bundles coming from the families
$\{T\}\times \ker(PT)\rightarrow T$, $\{T\}\times \ker(QT)\rightarrow T$, and the trivial bundles
$\{T\}\times F/R(P)\rightarrow T$, and $\{T\}\times F/R(Q)\rightarrow T$, which fit well with the maps occurring in the exact sequence.
Hence the maps $\gamma^P_{QT}$ define a topological line bundle isomorphisms
$$
{\mathcal L}_{T_0,Q,\varepsilon}\rightarrow {\mathcal L}_{T_0,P,\varepsilon}.
$$
 Consider the transition map 
$$
\Phi_{T_1,Q_1,\varepsilon_1}\circ \Phi_{T_2,Q_2,\varepsilon_2}^{-1}:
{\mathcal L}_{T_2,Q_2,\varepsilon_2}|U\rightarrow 
{\mathcal L}_{T_1,Q_1,\varepsilon_1}|U
$$
where $U=\{T\ |\ \norm{T-T_1}<\varepsilon_1,\norm{T-T_2}<\varepsilon_2\}$. This map is a bijection and is  linear on the fibers. It suffices to show that the transition map is continuous. Let $T_3\in U$. Then $Q_1,Q_2\in \Pi_{T_3}$ and we pick $P\in \Pi_{T_3}$ satisfying
$P\leq Q_1$ and $P\leq Q_2$. We take $\varepsilon$ small enough so that $V:=\{T\ |\ \norm{T-T_3}<\varepsilon\}\subset U$
and consider
$$
\text{DET}(E,F)|V\xrightarrow{\Phi_{T_3,P,\varepsilon}} {\mathcal L}_{T_3,P,\varepsilon}.
$$
First of all we have  the following commutative diagram of bijections which are  linear in the fibers, 
$$
\begin{CD}
\text{DET}(E,F)|V @= \text{DET}(E,F)|V\\
@V \Phi_{T_2,P_2,\varepsilon_2} VV @V \Phi_{T_1,P_1,\varepsilon_1} VV \\
{\mathcal L}_{T_2,P_2,\varepsilon_2}|V @. {\mathcal L}_{T_1,P_1,\varepsilon_1}|V\\
@VVV   @VVV\\
{\mathcal L}_{T_3,Q,\varepsilon} @= {\mathcal L}_{T_3,Q,\varepsilon}
\end{CD}
$$
The bottom vertical arrows are topological bundle isomorphisms. The vertical compositions on the left and on the right are
the same map in view of Lemma \ref{revision_lem_1}. Hence the transition map $\Phi_{T_1,Q_1,\varepsilon_1}\circ \Phi_{T_2,Q_2,\varepsilon_2}^{-1}
$ over the set $V$ can  be written as the transition map associated with the two lower vertical arrows, but the latter is a topological bundle isomorphism.

Hence we have proved the following theorem.
\begin{theorem}
 $\text{DET}(E,F)$ has in a natural way the structure of a topological line bundle
over ${\mathcal Fred}(E,F)$.
\end{theorem}

\subsection{Orientations in the M-Polyfold Case}
Now we are ready to deal with the M-polyfold case. In the following we shall assume that our polyfolds and strong bundles
are modeled on sc-smooth retracts in sc-Hilbert spaces, so that the orientation discussion from the previous subsection is applicable.
A more general orientation discussion covering the Banach space case is contained in \cite{HWZ10,HWZ11}. 
Let $E\rightarrow X$ be a strong M-polyfold bundle and $f$ a sc-Fredholm section. Given a smooth point $x_0\in X$
we can construct near $x_0$ a local sc$^+$-section $s$ satisfying $s(x_0)=f(x_0)$. Taking any such section $s$, the linearization
of $f-s$ at $x_0$, denoted by 
$$
(f-s)'(x_0):T_{x_0}X\rightarrow E_{x_0}, 
$$
 is a well defined (linear) sc-Fredholm operator between sc-Hilbert spaces, see 
\cite{HWZ2,HWZ3}. Given two such sections $s_1$ and $s_2$ we deduce from $f-s_1 =f-s_2 + (s_2-s_1)$ and the property
$(s_2-s_1)(x_0)=0$, that the linearisations $(f-s_1)'(x_0)$ and $(f-s_2)'(x_0)$ differ by the linearisation $(s_2-s_1)'(0)$,
which is a linear sc$^+$-operator and therefore level-wise compact.  We recall from \cite{HWZ2} that sc-Fredholm operators stay
sc-Fredholm operators after a perturbation by  an sc$^+$-operator. 
\begin{definition} {\bf The set $\text{Lin}(f,x_0)$} consists of all linear sc-Fredholm operators
of the form $(f-s)'(x_0)+a$, where $s$ is a fixed   local sc$^+$-section satisfying $s(x_0)=f(x_0)$ and $a:T_{x_0}X\rightarrow E_{x_0}$ is any linear sc$^+$-operator. In a formula, 
$$
\text{Lin}(f,x_0)=\{ (f-s)'(x_0) + a\, \vert \, a\, \text{is an $\ssc^+$-operator}\}.
$$
We call $\text{Lin}(f,x_0)$ the set of linearisations of $f$ at the smooth point $x_0$. The definition of
$\text{Lin}(f,x_0)$ does not depend on the choice of the $\ssc^+$-section $s$ as along as $s(x_0)=f(x_0)$.
\end{definition}
The set  $\text{Lin}(f,x_0)$ is a convex set to which 
we shall refer as the family of linearisations of $f$ at $x_0$.
Every element in $\text{Lin}(f,x_0)$ is an sc-Fredholm operator. In particular,  viewing them
as operators
$$
T_{x_0}X\rightarrow E_{x_0},
$$
i.e. on level $0$, 
we obtain classical Fredholm operators between Hilbert spaces.  If $A_1$ and $A_2$ are two such operators, then $A_1-A_2$ is a compact operator, hence belonging to
the Banach space $K(T_{x_0}X,E_{x_0})$ of compact operators.  The space  $\text{Lin}(f,x_0)$ is a metric space  equipped with the metric 
$$
d(A_1,A_2)=\norm{A_1-A_2}_{{\mathcal L}(T_{x_0}X,E_{x_0})}.
$$
By the previous discussion we have for every
classical Fredholm operator $T$ its determinant $\text{det}(T)$ and conclude the following result.
\begin{proposition}
The set $\text{DET}(f,x_0)$, defined by
$$
\text{DET}(f,x_0)=\bigcup_{A\in \text{Lin}(f,x_0)} \{A\}\times \det(A), 
$$
has in a natural way the structure of a topological line bundle over $\text{Lin}(f,x_0)$.
\end{proposition}
In view of the 
continuous map $\text{Lin}(f,x_0)\rightarrow {\mathcal Fred}(T_{x_0}X,E_{x_0})$ , we  can view the line bundle above as the pull-back bundle
of the determinant bundle over the space of Fredholm operators.
Since the  base space $\text{Lin}(f,x_0)$ is contractible  our bundle is trivial
and has precisely two possible orientations. 
\begin{definition}
{\bf An orientation of the sc-Fredholm section}  $f$ of
the strong $M$-polyfold bundle  $E\rightarrow X$ at the smooth point $x_0$ is a choice of one of the possible two orientations of
$\text{DET}(f,x_0)$.
\end{definition}

We need the  notion of {\bf ``local continuation''} of the orientation. For the following constructions we shall assume the existence of an sc-smooth partitions of unity.  (This is only for convenience and not really necessary, see \cite{HWZ10} for the general argument.)

We start with an sc-Fredholm section $f$ of the strong bundle
$E\rightarrow X$ and write $E\rightarrow [0,1]\times X$ for the pull-back of $E$ via the projection map  $[0,1]\times X\rightarrow X$.
We shall always assume in the following that the M-polyfod $X$ has no boundary, i.e., $\partial X=\emptyset$. The more general argument in the boundary case  $\partial X\neq \emptyset$ will be given in 
\cite{HWZ10}. 

Following the constructions in \cite{HWZ3}, we 
consider   an sc-smooth path $\phi:[0,1]\rightarrow X$ connecting the smooth point $x_0=\phi(0)$ with the smooth point  $x_1=\phi(1)$.
Using a partition of unity argument the following two Lemmata are proved in \cite{HWZ3}.
\begin{lemma}\label{lem_graph_1}
There exists for the strong bundle $E\rightarrow [0,1]\times X$ an sc$^+$-section $s=s(t,x)$ which satisfies
$s(t,\phi(t))=f(\phi(t))$ for $t\in [0,1]$.  
\end{lemma}
The solution set,  consisting  of all $(t,x)\in [0,1]\times X$ solving the equation 
$f(x)-s(t,x)=0$, 
contains the graph of $\phi$.
\begin{lemma}  \label{lem_property_star}
There exist finitely many sc$^+$-sections $s_1,\ldots ,s_k$ of the bundle  $E\rightarrow [0,1]\times X$
such  that the sc-smooth Fredholm section $F$ of the bundle $E\rightarrow [0,1]\times {\mathbb R}^k\times X$,  defined by
$$
F(t,\lambda,x) = f(x)-s(t,x)-\sum_{i=1}^k \lambda_i\cdot s_i(t,x),
$$
has the following property $(\ast)$ at the points $(t,\phi(t))$, $t\in [0,1]$.
\begin{itemize}
\item[($\ast$)] For every $t\in [0,1]$,  the linearization of $F_t$ at the point $(0,\phi(t))$ is surjective.
\end{itemize}
Here $F_t:=F(t,\cdot ,\cdot )$ is the obvious section of $E\rightarrow {\mathbb R}^k\times X$.
\end{lemma}

The implicit function theorem  in  \cite{HWZ3} or \cite{HWZ10} implies that the solution set $\{F(t,\lambda,x)=0\}$  near
$\{(t,0,\phi(t))\ |\ t\in [0,1]\}$   is in a natural way a smooth manifold $M$. By the  property $(\ast)$,
M also
fibers over $[0,1],$
$$
p\colon M\rightarrow [0,1],\quad p(t,\lambda,x)\mapsto  t
$$
and  every fiber $M_t=p^{-1}(t)$ is a manifold as well.
Abbreviating by  $L_t:=T_{(t,0,\phi(t))}M_t$ the tangent space, the bundle  
$$
L=\bigcup_{t\in [0,1]} L_t
$$
is a smooth vector bundle over $[0,1]$ and   
$L_t$ is the kernel of the  linearisation $DF_t(0,\phi(t))$, denoted by $F_t'(0,\phi (t))$. There is the associated smooth line bundle
$$
\Lambda^{\textrm{max}}L\rightarrow [0,1].
$$
An orientation of any of the fibers $L_t$ (any of the lines  $\Lambda^{\textrm{max}}L_t$)  determines an orientation  of  all the other fibers.
It remains to relate these orientations to the orientations of $\det(  (f-s(t,\cdot ))'(\phi(t)) )$ for 
$t\in [0,1]$.
To this aim we introduce for fixed  $t$  the sequence 
\begin{equation}\label{eexact}
\begin{gathered}
0\rightarrow \ker( (f-s(t,\cdot ))'(\phi(t)) )\xrightarrow{j} \ker({F}_t'(0,\phi(t)))\xrightarrow{p}\\
\xrightarrow{p} {\mathbb R}^k\xrightarrow{c} E_{\phi(t)}/R((f-s(t,\cdot ))'(\phi(t)))\rightarrow 0.
\end{gathered}
\end{equation}
Here $j$ is the inclusion map and $p$ the projection onto the ${\mathbb R}^k$-factor. The map $c$ is defined
by 
$$
c(\lambda) =\left(\sum_{i=1}^k \lambda_is_i(t,\phi(t))\right)+ R((f-s(t,\cdot ))'(\phi(t))).
$$
\begin{lemma}\label{lem_exact_seq_2}
The sequence \eqref{eexact} is exact.
\end{lemma}
\begin{proof}
The inclusion map $j$ is injective and $p\circ j=0$. From  $p(\lambda,h)=0$ it follows that 
$\lambda=0$
so that $h\in \ker((f-s(t,\cdot ))'(\phi(t)))$. If $(\lambda,h)\in \ker(F_t'(0,\phi(t)))$,  then $\sum \lambda_is_i(t,\phi(t))$
belongs to the image of $(f-s(t,\cdot ))'(\phi(t))$ which implies $c\circ p=0$. It is also immediate that an
element $\lambda\in {\mathbb R}^k$ satisfying  $c(\lambda)=0$ implies that $\sum \lambda_is_i(t,\phi(t))$ belongs to the image
of $(f-s(t,\cdot ))'(\phi(t))$. This allows us to construct an element $(\lambda,h)\in \ker(F_t'(0,\phi(t)))$ satisfying  $p(\lambda,h)=\lambda$. 
Finally, it follows from the property $(\ast)$ in Lemma \ref{lem_property_star}  that the map $c$ is  surjective. 
The proof of Lemma 
\ref{lem_exact_seq_2} is complete.
\end{proof}

In view of the exactness of the sequence \eqref{eexact} we deduce from Definition
 \ref{lem_isom_phi_1}, for every $t\in [0,1]$,  a uniquely defined isomorphism.

$$
\Psi_t:\Lambda^{\textrm{max}}{\mathbb R}^k \otimes \det(f-s(t,.)'(\phi(t)))\rightarrow
\Lambda^{\textrm{max}} \ker({F}_t'(0,\phi(t))).
$$
There is a uniquely determined smooth structure on  the line bundle 
$$
\bigcup_{t\in [0,1]} \det((f-s(t,\cdot ))'(\phi(t)))\rightarrow [0,1]
$$
 such  that 
the maps $\Psi_t$, $t\in [0,1]$,  are smooth line bundle isomorphism.

The following lemma is proved in \cite{HWZ10}. 
\begin{lemma} The smooth structure on the line bundle  
$$
\bigcup_{t\in [0,1]} \det((f-s(t,\cdot ))'(\phi(t)))\rightarrow [0,1]
$$
 does not depend
on the choice of the  $\ssc^+$-sections $s_i(t,x)$ having the property $(\ast)$ of 
Lemma \ref{lem_property_star}.
\end{lemma}

By construction,  $(f-s(t,\cdot ))'(\phi(t))\in \text{Lin}(f,\phi(t))$. Using the above isomorphism 
$\Psi_t$ and
the standard orientations  for ${\mathbb R}^k$ the following holds true. If for one $t_0 \in [0,1]$ we have an orientation
$\mathfrak{o}_{\phi(t_0)}$ of  $\text{DET}(f,\phi(t_0))$, then we get an orientation of 
$\det((f,s(t_0,\cdot ))'(\phi(t_0)))$ which can be propagated via the determinant line  bundle to an orientation
$\mathfrak{o}_{\phi(t)}$ of  $\text{DET}(f,\phi(t))$ for every $t\in [0,1]$. 
 
 A priori,  the method of propagation of an orientation might still depend on the choice of the $\ssc^+$-section $s(t,x)$ satisfying  $s(t,\phi(t))=f(\phi(t))$.
That  this is not the case is proved in \cite{HWZ10}.  Moreover,  an sc-smooth map $\Phi:[0,1]\times [0,1]\rightarrow X$,
 satisfying $\Phi(s,i)=\Phi(0,i)$ for $i=0,1$ may be viewed as a homotopy between $\phi_i:=\Phi(i,.)$ for $i=0,1$.
 One can carry out the above constructions also in this case. The conclusion  is the following result.
 \begin{lemma}
 Given an sc-smooth path $\phi:[0,1]\rightarrow X$ connecting the smooth points $x_0$ and $x_1$,
 the propagation of an orientation of $\text{DET}(f,\phi(t))$ does not depend
 on the choices of the $\ssc^+$-sections $s(t,x)$ and $s_i(t, x)$.  Of course, it might depend on the path connecting $x_0$ and $x_1$.
However, if two smooth paths connect the same points and are sc-smoothly homotopic, then they 
 propagate the initial orientation to the same orientation at the end point.
 \end{lemma}
As before we assume also in the following that 
the M-polyfold $X$ admits sc-smooth partitions of unity. It comes handy for the constructions in \cite{HWZ3},  to which we  have refered  for the proofs of the Lemmata \ref{lem_graph_1} and \ref{lem_property_star} above.
As already pointed out, the  propagation of an orientation along a path  can be established without the requirement of the existence of an sc-smooth partition of unity, \cite{HWZ10}.

\begin{definition} Let $E\rightarrow X$ be a strong bundle over the M-polyfold $X$ without boundary and let $f$ be an sc-Fredholm section.  An {\bf orientation $\mathfrak{o}$ for $f$}  consists of a choice of orientation $\mathfrak{o}_x$ of $\text{DET}(f,x)$ at  every smooth point $x\in X$, such  that
along  every sc-smooth path $\phi:[0,1]\rightarrow X$,  the propagation of the orientation $\mathfrak{o}_{\phi(0)}$ 
is the orientation $\mathfrak{o}_{\phi(1)}$ of 
$\text{DET}(f,\phi (1))$.
\end{definition}
In general,  it is possible (not for Gromov-Witten)
that $\phi(t)=\phi(t')$ for some  $t\neq t'$ but  the associated orientations are opposite.

\begin{proposition}
Assume that $E\rightarrow X$ is a strong bundle over the M-polyfold $X$ and $f$ an {\bf oriented} proper sc-Fredholm section.
Suppose that for every solution $x\in f^{-1}(0)$ the linearization $f'(x)$ is onto. Then the solution set $M=f^{-1}(0)$ is a smooth compact
orientable manifold inheriting a natural orientation from the orientation $\mathfrak{o}$ of $f$.
\end{proposition}
\begin{proof}
From the implicit function theorem for sc-Fredholm section we deduce that $M$ is a smooth manifold,
which in addition is compact, since $f$ is proper. Then $T_xM =\ker(f'(x))$ for $x\in M$ and using that
$$
\Lambda^{max} T_xM\simeq \Lambda^{max} T_xM \otimes {\mathbb R} =\det(f'(x))
$$
we see that the orientation $\mathfrak{o}_x$  determines an orientation on the tangent space $T_xM$ which we also denote by  $\mathfrak{o}_x$.  Due to the property of propagation 
along paths,  the orientation of $T_yM$ for $y$ near $x\in M$ is precisely the orientation 
$\mathfrak{o}_y$.
\end{proof}

\subsection{Orientations in the Polyfold Case}\label{SS5.11}
In the case of a strong polyfold bundle $W\rightarrow Z$,  an sc-mooth Fredholm section $f$ is represented
by an sc-smooth section functor $F$ of a strong bundle 
$(E\rightarrow X, \mu)$
over the ep-groupoid $X$.

We consider a morphism 
$\phi_0:x_0=s(\phi_0)\rightarrow y_0=t(\phi_0)$, denoting by $s$ and $t$ the source and  target maps from the morphism set of the ep-groupoid $X$ to its object set.
Associated with $\phi_0$ is the 
germ of an sc-diffeomorphism 
$\Phi=t\circ s^{-1}: {O}(x_0)\rightarrow {O}(y_0)$ between open neighborhoods. Its derivative at $x_0$ is, by definition,  the tangent map $T\phi_0:T_{x_0}X\rightarrow T_{y_0}X$. We denote by $\wh{\phi}_0:E_{x_0}\rightarrow E_{y_0}$ the lift of $\phi_0$,  viewed as a linear operator.  
More precisely,  $\wh{\phi}_0=\mu (\phi_0,\cdot ):E_{x_0}\rightarrow E_{y_0}$ where the sc-smooth strong bundle map $\mu\colon {\bf E}\to E$ satisfies the properties of Definition \ref{stbundleep}.

\begin{definition}
Consider an operator $L\in \text{Lin}(F,x_0)$ and let $\phi_0:x_0\rightarrow y_0$ be a morphism. The {\bf push forward
of $L$}, denoted by ${(\phi_0)}_\ast L$,  is the  linear sc-Fredholm operator belonging  to the set $\text{Lin}(F,y_0)$ defined by
\begin{equation}\label{po<po}
{(\phi_0)}_\ast L = \what{\phi}_0\circ L\circ {(T\phi_0)}^{-1}.
\end{equation}
If $\mathfrak{o}$ is an orientation of  $\det(L)$ defined by the vector $(h_1\wedge\ldots \wedge h_k)\otimes (h_1^\ast\wedge\ldots \wedge h_\ell^\ast)$, then the push-forward orientation ${(\phi_0)}_\ast\mathfrak{o}$ of  ${(\phi_0)}_\ast L$ is defined by the vector
$(T\phi_0(h_1)\wedge..\wedge T\phi_0(h_k))\otimes (h_1^\ast\circ \what{\phi}_0^{-1}\wedge\ldots \wedge h_\ell^\ast\circ \what{\phi}_0^{-1})$.
\end{definition}
It is easy to see  that (\ref{po<po}) defines an element in $\text{Lin}(F,y_0)$.
In view of the discussion in the previous subsection,   the proof of the following lemma is straightforward.
\begin{lemma}
If $\phi:x\rightarrow y$ is a morphism between smooth points, then
the map $L\rightarrow \phi_\ast(L) $ defines a topological isomorphism of line bundles
$$
\phi_\ast :\text{DET}(F,x)\rightarrow \text{DET}(F,y),
$$
via $(L,(h_1\wedge\ldots \wedge h_k)\otimes (h_1^\ast\wedge\ldots \wedge h_l^\ast))\rightarrow
(\phi_\ast L, (T\phi(h_1)\wedge\ldots \wedge T\phi(h_k))\otimes (h_1^\ast\circ \what{\phi}^{-1}\wedge\ldots \wedge h_l^\ast\circ \what{\phi}^{-1}))$. In particular,  if   $\mathfrak{o}$ is an orientation
of  $\text{DET}(F,x)$,  we obtain the  induced orientation $\phi_\ast\mathfrak{o}$ of $\text{DET}(F,y)$.
\end{lemma}
\begin{remark}
A particular case arises  if the morphism is an automorphisms $\phi:x\rightarrow x$ between the smooth points $x\in X$.  It gives
us a bundle automorphisms  of $\text{DET}(F,x)$.  In general, an automorphism can act in an orientation-preserving or an orientation-reversing
way. We shall denote the automorphism group of $x$ by $\text{Aut}(x)$. In the case that the automorphisms always act in an orientation-preserving way,  we  conclude  
$\phi_\ast\mathfrak{o}=\psi_\ast \mathfrak{o}$
for arbitrary morphisms $\phi,\psi:x\rightarrow y.$
\end{remark}

If $E\rightarrow X$ and $E'\rightarrow X'$ are two strong bundles over ep-groupoids, we let 
$\Gamma:E\rightarrow E'$ be a strong bundle equivalence covering the equivalence $\gamma:X\rightarrow X'$ in the sense of Definition 2.17 in \cite{HWZ3.5}. 
Let $F$ and $G$ be sc-Fredholm sections of these bundles satisfying  $\Gamma^\ast(G)=F$.
Then given  the smooth points $x\in X$  and  $y'=\gamma(x)$, we define  the  topological  isomorphism
$$
\text{Lin}(F,x)\rightarrow \text{Lin}(G,y),\quad  L\rightarrow \Gamma_\ast(L), 
$$
by
$$
\Gamma_\ast(L) = \Gamma_x \circ L\circ T\gamma(x)^{-1}.
$$
To verify that $\Gamma_\ast(L)$ is  well-defined, we  take a germ of a sc$^+$-section $u$ satisfying  $F(x)-u(x)=0$ and 
let $v$ be the push-forward germ.  Recalling that  a general element $L$ in the set $\text{Lin}(F,x)$  has the form 
$(F-u)'(x)+a$, where $a$ is an $\ssc^+$-operator, 
we compute
$$
\Gamma_\ast((F-u)'(x)+a) = (G-v)'(y) +\Gamma_\ast(a)=: (G-v)'(y) +b.
$$
The section  $b$ is also an $\ssc^+$-section, hence proving the claim.

We can go one step further and obtain along the same lines an isomorphism 
$$
{(\Gamma_x)}_\ast:\text{DET}(F,x)\rightarrow \text{DET}(G,y)
$$
between topological line bundles.
\begin{lemma}
Assume that $E\rightarrow X$ and $E'\rightarrow X'$ are strong bundles over ep-groupoids coming with the sc-Fredholm sections $F$ and $G$. Also assume that $\Gamma:E\rightarrow E'$ is a strong bundle equivalence covering the equivalence  $\gamma:X\rightarrow X'$ of ep-groupoids
and satisfying  $\Gamma^\ast G=F$.  Then for every smooth $x\in X$ and $y=\gamma(x)$,  we have an induced
topological bundle isomorphism
$$
{(\Gamma_x)}_\ast: \text{DET}(F,x)\rightarrow \text{DET}(G,y)
$$
\end{lemma}
\begin{proof}
For $\gamma$ we have the tangent map $T\gamma(x):T_xX\rightarrow T_yX'$. Moreover,  $\Gamma$ gives us an sc-bundle isomorphism
$\Gamma_x:E_x\rightarrow E_y'$. We define
$$
{(\Gamma_x)}_\ast(L) = \Gamma_x \circ L \circ T\gamma(x)^{-1}.
$$
The vectors in $\det(L)$ are pushed forward by the usual coordinate change.
This completes the proof.
\end{proof}

\begin{definition}\label{definition_sc_fred_orient}
Let $E\rightarrow X$ be a strong bundle over an ep-groupoid.
Then an {\bf sc-Fredholm section  $F$ is  called orientable}  if  we can choose for every smooth point $x\in X$ an orientation $\mathfrak{o}_x$
of $\text{DET}(F,x)$ such  that the following holds.
\begin{itemize}
\item[(1)] For every smooth path $\varphi:[0,1]\rightarrow X$, the orientation $\mathfrak{o}_{\varphi(1)}$ is the continuation of $\mathfrak{o}_{\varphi(0)}$.
\item[(2)] For every morphism $\phi:x\rightarrow y$ between smooth points the push forward of $\mathfrak{o}_x$ is $\mathfrak{o}_y$.
\end{itemize}
An {\bf orientation of  the sc-Fredholm section $F$}  of $E\rightarrow X$ is given by a choice
of  an orientation $\mathfrak{o}_x$ of  $\text{DET}(F,x)$ for every smooth $x$
satisfying the above properties $(1)$ and $(2)$. We shall write $\mathfrak{o}$ for an orientation of $F$. The orientation $\mathfrak{o}$ associates with every smooth point $x\in X$ the orientation $\mathfrak{o}_x$ of the topological line bundle $\text{DET}(F,x)\rightarrow \text{Lin}(F,x)$.
\end{definition}

We now assume that  $\mathfrak{o}$  is an orientation of  the sc-Fredholm section $F$ of the strong bundle $E\rightarrow X$.
If $\Gamma:E\rightarrow E'$ is a strong bundle equivalence covering the equivalence $\gamma:X\rightarrow X'$, 
and if $G$ is an sc-Fredholm section of the strong bundle $E'\to X'$ satisfying $\Gamma^\ast G=F$, then there is an induced orientation $\Gamma_\ast \mathfrak{o}$ of the sc-Fredholm section  $G$. It is defined as follows. If  $y\in X'$,  we find $x\in X$ and a morphism $\phi:\gamma(x)\rightarrow y$. Since $F$ is orientable, 
the automorphism group $\text{Aut}(x)$  of $x$ acts in an orientation-preserving way on $\text{DET}(F,x)$. We have the topological 
bundle isomorphism ${(\Gamma_x)}_\ast :\text{DET}(F,x)\rightarrow \text{DET}(G,\gamma(x))$. The automorphism group $\text{Aut}(\gamma (x))$ of $\gamma(x)$ is isomorphic to $\text{Aut}(x)$ and acts also  in an orientation-preserving way. Due to the latter,  the push forward of the 
orientation via the morphism $\phi:\gamma(x)\rightarrow y$ to $\text{DET}(G,y)$ does not depend on the actual choice of the morphism $\phi$.
We also note that, conversely,  given an orientation of the sc-Fredholm section  $G$ we obtain an orientation of the section $F$.  Hence we have the operations
$$
\mathfrak{o}\rightarrow\Gamma_\ast \mathfrak{o}\ \ \text{and}\ \ \mathfrak{o}'\rightarrow\Gamma^\ast\mathfrak{o}'
$$
which associate to an orientation of $F$ an orientation of $G$ and, conversely,  to an orientation of $G$ an orientation of $F$.
These operations are mutually inverse.  

More generally,  if we have two sc-Fredholm sections $F$ and $G$ of the strong bundles $E\rightarrow X$ and $E'\rightarrow X'$, respectively,  and if $\mathfrak{A}:E\Longrightarrow E'$ is a generalized bundle isomorphism satisfying $\mathfrak{A}_\ast F=G$, then an orientation on either side determines an orientation on the other side.
This is an immediate consequence of the previous discussion and the definition of a generalized strong bundle isomorphism in Definition 2.20 of \cite{HWZ3.5}. 
Now,  we are ready for the following definition.

\begin{definition}
Let $f$ be an sc-Fredholm section of the strong polyfold bundle $W\rightarrow Z$.  We call $f$ {\bf orientable} provided
there exists a model $E\rightarrow X$ and an sc-smooth section functor $F$ representing $f$
which admits an orientation.  An {\bf orientation of $f$}  is by definition the  choice of an orientation of $F$ for a suitable model.
\end{definition}

The definition does not depend on the choice of the model because   the orientation of one model determines  orientations of every other model,   by means of a generalized strong bundle isomorphism. 

\begin{remark} We would like to point out that in SFT the orientability is sometimes obstructed due to automorphisms acting in an orientation-reversing way.  The algebraic constructions in SFT have to take this into consideration. 
In the case of Gromov-Witten theory the situation is more pleasant because there is even a natural orientation, called the complex orientation, which we shall discuss next.
\end{remark}

\section{The Canonical Orientation in Gromov-Witten Theory}\label{orientations}

We denote by $Z$ the space of stable curves into  to  a closed symplectic manifold $(Q,\omega)$. 
Having fixed a compatible almost complex structure $J$ on $Q$ we have the associated strong bundle
$W\rightarrow Z$ and the sc-smooth Cauchy-Riemann section $\ov{\partial}_J$.
  Our main result in this section is the following theorem.
\begin{theorem}\label{BINGO}
The Cauchy-Riemann section $\ov{\partial}_J$ of the strong bundle $W\rightarrow Z$ is orientable. Moreover it has a natural orientation, called the complex orientation.
\end{theorem} 
In order to briefly outline the strategy of the proof we recall that the models $F:X\rightarrow E$ are constructed from local models
 $F_\cg:\cg\rightarrow \what{\cg}$. We shall show that these local models have canonical orientations, called complex orientations. They 
are compatible with the morphisms in $M(\cg,\cg')$ and $M(\what{\cg},\what{\cg}')$, respectively.
From this Theorem \ref{BINGO} will then follow easily.

\subsection{A Basic Class of Linear Fredholm Operators}
By ${\mathcal J}={\mathcal J}(Q,\omega)$ we denote the contractible space of compatible smooth  almost complex structures on 
the closed symplectic manifold $(Q,\omega)$. We consider a closed smooth oriented   surface $S$ equipped with the finite set $M$ of ordered marked points.   By $\mathfrak{j}=\mathfrak{j}_{(S,M)}$ we denote the collection of smooth almost complex structures on $S$ which induce the given orientation.

Next we consider the collection of all {\bf un-noded stable} maps $\alpha=(S,j,M,\emptyset,u)$, where $u:S\rightarrow Q$ is smooth and  $j\in \mathfrak{j}$.
 We associate with $\alpha$ the Hilbert space
$H^3(u^\ast TQ)$.  Given the  almost complex structure $J\in {\mathcal J}$ on $Q$, 
we denote by $\Omega^{0,1}_{j,J}$ the bundle over $S\times Q$, whose fiber over the point $(z,q)$ consists of all complex anti-linear maps $(T_zS,j)\rightarrow (T_qQ,J)$.  Whenever  we write $u^\ast\Omega^{0,1}_{j,J}$ for a smooth map $u:S\rightarrow Q$ we actually mean the pullback of the bundle $\Omega^{0,1}_{j,J}$  by the graph map $\text{gr}(u):S\rightarrow S\times Q$. 

\begin{definition} \label{<<<}
 {\bf The collection ${\mathcal F}(S,M)$}  consists of all linear Fredholm operators
\begin{equation}\label{999}
L: H^3(u^\ast TQ)\rightarrow H^2(u^\ast\Omega_{j,J})
\end{equation}
for  suitable $J\in {\mathcal J}$, $j\in \mathfrak{j}$,  and smooth un-noded stable maps $\alpha=(S,j,M,\emptyset, u)$, 
of the form
\begin{equation}\label{n<n}
L(\eta)= \nabla\eta +J(u)(\nabla\eta) j +A(\eta).
\end{equation}
Here the covariant derivative  $\nabla$,   associated with a  complex connection for $(TQ,J)\rightarrow Q$, satisfies $\nabla J=0$. Moreover $A$ is a compact operator.

Fixing the Riemann surface $(S, j, M)$,  {\bf we  denote by 
${\mathcal F}_{\alpha}$} the collection of all Fredholm operators  
$L$  of the  form \eqref{n<n}, associated  with the stable map 
$\alpha=(S, j, M, \emptyset, u)$ and $J$ varying in ${\mathcal J}$.

\end{definition}
\begin{remark}
Note that among the operators in ${\mathcal F}(M,S)$ are linear sc-Fredholm operators if the spaces 
$H^3(u^\ast TQ)$ and $H^2(u^\ast\Omega_{j,J})$ are equipped with their sc-structure.
Although the sc-structure is important in orientation questions,  as shown in the previous subsection,
it will not be needed for a while, but  ultimately it will be incorporated.
\end{remark}

For fixed $u$, $J$, and $j$ the collection of all operators $L$ of the form  \eqref{999} and \eqref{n<n}  is  a convex space, which is contractible. 
Also the collections ${\mathcal J}$ and $\mathfrak{j}$ are contractible. This immediately implies the following lemma.
\begin{lemma}
The space of Fredholm operators of the form \eqref{999}  and \eqref{n<n}, where $u$ is fixed, equipped with the operator norm,  is contractible.
\end{lemma}

\begin{remark} Let us observe that a given $L$ in general can be written
in different ways. However, if 
$$
\nabla\eta +J(u)(\nabla\eta) j +A(\eta) =\nabla'\eta +J'(u)(\nabla'\eta) j' +A'(\eta) \quad  \text{for all $\eta$}
$$
we must always  have $J=J'$, and $j=j'$ because this  is  part of the principal symbol.  Note that the difference 
$\nabla-\nabla'$ is a compact operator.
\end{remark}
If  the operator $L\in {\mathcal F}(S,M)$ is  associated with the almost complex structures $J$ and $j$, we can turn the domain and codomain of $L$ into complex vector spaces by defining $i\eta:=J(u)\eta$ for the domain and $i\xi:=J(u)\circ \xi$ for the codomain.
The operator  $L$ is called complex linear if  $L(i\eta) = iL(\eta)$. This gives in certain cases an additional structure on $L$. Let us observe that if $L\in {\mathcal F}_\alpha$, as described in Definition \ref{<<<},  so is $L'$,  defined by
$$
L'(\eta)=  \nabla\eta +J(u)(\nabla\eta) j. 
$$
The  operator $L'$ is complex linear. This follows from  
$\nabla (J(u)\eta)=J(u)\nabla \eta$ which is implied by  $\nabla J=0$.

For fixed $(S,M)$,  we consider two operators
 $L_0\in {\mathcal F}_{\alpha_0}$ and $L_1\in {\mathcal F}_{\alpha_1}$ associated with the stable maps  $\alpha_i=(S,j_i,M,\emptyset, u_i)$,  and the almost complex structures $J_i$ and $j_i$ for $i=0,1$.
 If $u_0$  and $u_1$ smoothly homotopic by the homotopy $u^t$, $0\leq t\leq 1$, we can construct a smooth path $L_t$ in  ${\mathcal F}$ connecting 
 $L_0$ with $L_1$ of the form
$$
L_t = \nabla^t(\cdot ) +J^t(u^t)(\nabla^t(\cdot ))j^t +A^t(\cdot ).
$$
Here $J^t$ is a smooth path in ${\mathcal J}$ connecting $J_0$ with $J_1$ and $j^t$ is a smooth path in $\mathfrak{j}$ connecting $j_0$ with $j_1$.
The covariant derivatives $\nabla^t$ are associated with the tangent bundle $(TQ, J^t)\to Q$ and satisfy 
$\nabla^tJ^t=0$. The data indexed by $t$ on the right-hand side of $L_t$ depend smoothly on $t$. The domains  of the operators  $L_t$  form  a smooth Hilbert space bundle over $[0,1]$  having the  fiber $H^3({(u^t)}^\ast TQ)$ over $t$, 
whereas the codomains  form a smooth bundle over $[0,1]$. To the path $(L_t)_{0\leq t\leq 1}$ there belongs the determinant bundle having the fiber $\det (L_t)$ over $t\in [0,1]$. 
An orientation of $\det (L_0)$ determines
an orientation of $\det(L_1)$, called the orientation continued along the path $(L_t)_{0\leq t\leq 1}$. In the following  shall refer to $(L_t)$ as a
{\bf smooth path of operators in ${\mathcal F}(S,M)$}.

The fundamental result about the operators in ${\mathcal F}(S, M)$ is  the  following proposition.
\begin{proposition}\label{p<ppp}
For every Fredholm operator $L\in {\mathcal F}(S,M)$ there exists a  natural orientation $\mathfrak{o}_L$ of  $\det(L)$
characterized by the following two requirements.
\begin{itemize}
\item[(1)] Along  every smooth path $(L_t)_{0\leq t\leq 1}$ in ${\mathcal F}(S,M)$, the continuation of $\mathfrak{o}_{L_0}$ is the orientation $\mathfrak{o}_{L_1}$.
\item[(2)] If $L$ is complex linear,  then $\mathfrak{o}_L$ is defined by the vector 
$$
(e_1\wedge ie_1\ldots \wedge e_n\wedge  ie_n)\otimes (h_1^\ast\wedge ih^\ast_1\ldots \wedge h_k\wedge ih^\ast_k),
$$
where $e_1,\ldots \, e_n$ is a complex basis of  $\ker(L)$ and $h_1^\ast,\ldots, h_k^\ast$  is a complex basis of  $\text{coker}(L)^\ast$.
This orientation does not depend on choices involved and is called the {\bf complex orientation
of the (real) determinant of $L$}. 
\end{itemize}
\end{proposition}
\begin{proof}
We define an equivalence relation by calling two Fredholm operators in $F(S,M)$ equivalent if 
they can be connected by a smooth path of operators. 
In  an equivalence class $P$  we find a complex linear operator $L_0$  associated with 
$J_0$ and the smooth stable map $\alpha_0=(S,j_0,M,\emptyset,u_0)$, defined by
$$
L_0\eta  =\nabla\eta +J_0(u_0)(\nabla\eta) j_0. 
$$
The operator $L_0$  is complex linear since, by assumption,  $\nabla J=0$. We fix the complex orientation $\mathfrak{o}_{L_0}$.

If  $L_1\in P$ we take a smooth path $(L_t)$ of operators connecting $L_0$ with $L_1$ and  
 equip $\det(L_1)$ with the continuation of the already chosen orientation 
of $\det(L_0)$ along the  determinant bundle of the path.  A priori the orientation $\mathfrak{o}_{L_1}$ might depend on the choice of the path.
If we take another path,  the two paths together define a smooth loop $\theta\mapsto  L_\theta$, $\theta\in S^1$, 
of  the form
$$
L_\theta(\eta) =\nabla^\theta\eta +J(u^\theta)(\nabla^\theta\eta)j^\theta +A^\theta(\eta),
$$
where $\nabla^\theta J^\theta=0$. Our orientation is well-defined precisely if the determinant bundle $\det (L_\theta)$ over
$S^1$ is orientable.  In order to prove this, 
we homotope the family $\theta\mapsto  L_\theta$ to a family of the form
$$
L^0_\theta(\eta)=\nabla^\theta\eta +J(u^\theta)(\nabla^\theta\eta)j^\theta,
$$
consisting   of a complex linear operator over every fiber. This implies that the family $(L^0_\theta)$ over $S^1$
has an orientable determinant bundle. Hence the original loop is also orientable. A priori the definition of the orientation of the operators in $P$ might still depend on the choice of  the operator $L_0\in P$.
However, if $L_0'$ is a second complex linear operator in $P$, we connect $L_0$ with  $L_0'$ by a path of complex linear operators and see that the complex orientation for $L_0'$ is the one obtained from $L_0$ by propagation along the  path.
This proves that the definition of the orientation for the operators in $P$ is independent of all choices.
We carry out the process for all equivalence classes $P$ and the proof of Proposition \ref{p<ppp} is complete.
\end{proof}

\begin{definition}
The assignment $L\rightarrow \mathfrak{o}_L$ is called (in view of property (2) in Proposition \ref{p<ppp}) the {\bf complex orientation of  ${\mathcal F}(S,M)$},
and for each individual operator $L$, whether  complex linear or not, it is called the {\bf complex orientation of  $L$}.
We also refer interchangeably to $\mathfrak{o}_L$ as the complex orientation for $L$ or $\det(L)$.
\end{definition}

Next we describe  what happens under certain coordinate changes. We assume that the two smooth un-noded stable maps $\alpha=(S,j,M,\emptyset,u)$
and $\alpha'=(S',j',M',\emptyset, u')$ are isomorphic by the isomorphism 
$$
\phi:\alpha\rightarrow\alpha'.
$$

We shall define a map ${\mathcal F}_{\alpha}\rightarrow {\mathcal F}_{\alpha'}$ from ${\mathcal F}_{\alpha}\subset {\mathcal F}(S,M)$ to  ${\mathcal F}_{\alpha'}\subset {\mathcal F}(S',M')$.
Given  the Fredholm operator ${\bf L}$  in ${\mathcal F}_{\alpha}$,  we define ${\bf L}'=\phi_\ast{\bf L}\in 
{\mathcal F}_{\alpha'}$ by the formula
$$
{\bf L}'(\eta') := ({\bf L}(\eta'\circ\phi))T\phi^{-1}, \quad 
\eta'\in H^3((u')^*TQ).
$$
It is useful to recall that, by definition of an isomorphism, $u'(\phi (z))=u(z)$ for all $z\in S$, so that 
$\eta'\circ \phi\in H^3(u^*TQ).$
\begin{lemma}
If  ${\bf L}\in {\mathcal F}_{\alpha}$,  the push forward operator ${\bf L}'=\phi_\ast{\bf L}$ belongs to ${\mathcal F}_{\alpha'}$.
\end{lemma}
\begin{proof}
In order to verify  that the above definition of  ${\bf L}'$ has the required  form we shall use that, by definition of an
isomorphism, $T\phi (z)\cdot j(z)=j'(\phi (z))T\phi (z)$, $z\in S$, and compute, abbreviating  $\eta=\eta'\circ\phi$, 
\begin{equation*}
\begin{split}
{\bf L}'(\eta')&
({\bf L}(\eta'\circ \phi)) T\phi^{-1}\\
&=[\nabla\eta + J(u)(\nabla\eta)j + A(\eta)]T\phi^{-1}\\
&=[(\nabla\eta')T\phi +J(u'\circ\phi)((\nabla\eta')T\phi)j +A(\eta)]T\phi^{-1}\\
&=[(\nabla\eta')T\phi +J(u'\circ\phi)((\nabla\eta')T\phi)j +A'(\eta')T\phi]T\phi^{-1}\\
&=([\nabla\eta' +J(u')(\nabla\eta')j']T\phi)T\phi^{-1}+A'(\eta')\\
&=\nabla\eta' +J(u')(\nabla\eta')j' + A'(\eta').
\end{split}
\end{equation*}
\end{proof}
If the  orientation $\mathfrak{o}$ of ${\bf L}$ is defined by  the vector
$$
(h_1\wedge..\wedge h_k)\otimes(h_1^\ast\wedge\ldots \wedge h_l^\ast),
$$
then the associated orientation $\phi_{\ast}\mathfrak{o}$ of  $\phi_\ast {\bf L}$, called the {\bf orientation obtained by coordinate change},  is given by the
vector
$$
(h_1\circ\phi^{-1}\wedge\ldots \wedge h_k\circ\phi^{-1})\otimes ( h_1^\ast\circ  T\phi^{-1}\wedge\ldots \wedge h_l^\ast\circ T\phi^{-1}).
$$
In the case that ${\bf L}$ is complex linear the same is true for $\phi_\ast {\bf L}$ and the coordinate change 
maps the complex orientation to the complex one. Since the complex orientation for the operators in ${\mathcal F}_{\alpha}$
as well as in ${\mathcal F}_{\alpha'}$ enjoy the continuation along path property, it follows 
 by the same type of (homotopy) argument already used before,
  that $\phi_{\ast}\mathfrak{o}_{\bf L} = \mathfrak{o}_{\phi_\ast {\bf L}}$,
 i.e. the coordinate change maps the complex orientation to the complex orientation.
 Hence we have proved.
\begin{proposition}\label{m<0}
If $\phi:\alpha\rightarrow\alpha'$ is an isomorphism between un-noded smooth stable maps and ${\bf L}\in {\mathcal F}_\alpha$,
then the complex orientation $\mathfrak{o}_{\bf L}$ is mapped via $\phi$ to the complex orientation 
$\mathfrak{o}_{{\bf L}'}$ of the push-forward operator ${\bf L}':=\phi_\ast{\bf L}\in {\mathcal F}_{\alpha'}$.
\end{proposition}

\subsection{Modifications of the  Basic Class}
Up to now the stability of the stable map $\alpha$ did not play any role. However, this will change now. We fix $(S,M)$ in the definition of the space ${\mathcal F}(S, M)$. The following discussion will focus on operators which we might consider as finite-dimensional modifications of Fredholm operators.

Our aim is to establish canonical orientations (again called the complex orientations) for these modified Fredholm operators and to study their behavior under 
coordinate changes which are more general than those considered so far. 

We recall that the stability of a smooth stable map 
$\alpha=(S, j, M, \emptyset, u)$ requires either that 
$(S,j,M)$ is a stable Riemann surface, or, if this is not the case, that the map $u$ satisfies $\int_S u^\ast\omega>0$.

\begin{definition}
A {\bf stab-set} $\Xi$ for the smooth stable map $\alpha=(S, j, M, \emptyset, u)$ is a finite subset $\Xi$  of $S\setminus M$ having the property, that for every $z\in \Xi$
the vector space  $R(Tu(z))$, i.e.,   the image of the tangent map $Tu(z)$,  is  a symplectic subspace of the tangent space $(T_{u(z)}Q, \omega(u(z)))$,  and the orientation induced by $\omega(u(z))$ is the same as the orientation  induced 
from $(T_zS,j)$ via the tangent map  $Tu(z)$.  
\end{definition}

An example  of a stab-set is the stabilization set of a smooth un-noded  stable map as introduced in Definition \ref{stab-f}. In contrast to the stabilization set, the stab-set is not required to have any properties  related to the automorphim group, and the Riemann surface $(S,j,M\cup\Xi)$ is also not required to be stable.
The empty stab-set is a possible stab-set. 

For every point $z\in \Xi$ we introduce the set ${\mathcal H}_z$ consisting of all linear 
subspaces of $T_{u(z)}Q$ of codimension $2$ which are transversal to $R(Tu(z))$.
The set ${\mathcal H}_z$ contains, for example, the symplectic complement of 
$R(Tu(z))$, defined by
$$R(Tu(z))^\omega=\{v\in T_{u(z)}Q\, \vert\, \text{$\omega (u(z))(v, w)=0$ for all $w\in R(Tu(z))$}\}.$$

By ${\mathcal H}$ we denote the collection of all $H={(H_z)}_{z\in\Xi}$ and $H_z\in {\mathcal H}_z$ for $z\in\Xi$.
\begin{definition}
Associated with $H$ is the subspace 
$$H^3_H(u^\ast TQ)\subset H^3(u^\ast TQ)$$
of finite codimension,  consisting of all sections $\eta$ satisfying  $\eta(z)\in H_z$ for $z\in\Xi$. 
\end{definition}

In the following we shall orient (in a coherent way) the Fredholm operators 
$$
L: H^3_H(u^\ast TQ)\rightarrow H^2(u^\ast \Omega^{0,1}_{j,J})
$$
of the  usual form 
$$
L(\eta)=\nabla\eta +J(u)(\nabla\eta)\cdot  j +A(\eta),
$$
where $A$ is a compact operator. We observe that these are the same operators as the ones studied before, but this time  restricted to  a subspace of finite codimension.  We denote this collection of Fredholm operators by
 ${\mathcal F}_{\alpha,\Xi}$.  In ${\mathcal F}_{\alpha,\Xi}$ we allow $J\in{\mathcal J}$ and $H\in {\mathcal H}$ to vary,
whereas $u$ and $j$ are fixed.
In the   case that $H$ has the property $J(u(z))H_z=H_z$ for all $z\in \Xi$,  the space $H^3_H(u^\ast TQ)$ has a complex structure 
 via $i\eta:=J(u)\eta$ so that  in this case  the operator $\eta\rightarrow \nabla\eta +J(u)(\nabla\eta) j $ is a complex linear operator if $\nabla J=0$.
 Hence we have complex linear operators in ${\mathcal F}_{\alpha,\Xi}$ and complex  orientations (for the real determinants).  
Since the space ${\mathcal J}$ of almost complex structures on the manifold $(Q, \omega)$ is contractible, and for fixed $\Xi$ the set ${\mathcal H}$ of constraints is contractible, the space ${\mathcal F}_{\alpha,\Xi}$ of Fredholm operators (for fixed $\alpha$) is contractible. Therefore, we can repeat the arguments in the proof of 
Proposition \ref{p<ppp} to obtain the following result.  
 
  \begin{proposition}
 There is for every $L\in {\mathcal F}_{\alpha,\Xi}$ a natural orientation $\mathfrak{o}_L$
 for $\det(L)$,  characterized by the following requirements.
 \begin{itemize}
 \item[(1)]  For every smooth path $(L_t)_{0\leq t\leq 1}$  in ${\mathcal F}_{\alpha,\Xi}$, the continuation of $\mathfrak{o}_{L_0}$ along the path is $\mathfrak{o}_{L_1}$. 
 \item[(2)] If $L$ is complex linear then $\mathfrak{o}_{L}$ is the complex orientation.
 \end{itemize}
 \end{proposition}
 
If the Fredholm operator $L$ in ${\mathcal F}_{\alpha,\Xi}$ belongs to  the constraints $H=(H_z)$ for $z\in \Xi$ and if $\phi:\alpha\rightarrow\alpha'$ is an isomorphism between the two smooth stable maps, which maps $\Xi$ to $\Xi'$, then the push-forward operator $L'=\phi_\ast L$ 
belongs to the space ${\mathcal F}_{\alpha',\Xi'}$ and satisfies the push-forward constraints $\phi_\ast H$ defined as follows.
 If $H=(H_z)$, $z\in \Xi$, then $\phi_\ast H=(H_{z'})$, $z'\in \Xi'$, is defined as $H_{z'}=H_{\phi (z)}=H_z$, so that  $H_{z'} \subset T_{u'(z')}Q=T_{u(z)}Q$, since, by definition of an isomorphism, 
 $u'(z')=u'(\phi (z))=u(z)$. Therefore, we have a well-defined push-forward of an orientation. Proceeding now as in the proof of 
 Proposition \ref{m<0}, considering  constraints which are $J$-invariant
and applying  the homotopy arguments from above, one  verifies  that the complex orientation $\mathfrak{o}_{\bf L}$ is pushed forward to  
the complex orientation $\mathfrak{o}_{\bf \phi_\ast L}$ and obtains the following proposition.
 
 \begin{proposition}\label{b<b}
 Given an isomorphism $\phi:\alpha\rightarrow \alpha'$ between two smooth un-noded stable maps,   and given stab-sets $\Xi$ and $\Xi'$ satisfying 
 $\phi(\Xi)=\Xi'$, the push-forward operation $\phi_\ast:{\mathcal F}_{\alpha,\Xi}\rightarrow {\mathcal F}_{\alpha',\Xi'}$
 maps the complex orientation $\mathfrak{o}_{\bf L}$ to the complex orientation  $\mathfrak{o}_{\phi_\ast{\bf L}}$.
 \end{proposition}

Next we consider the relationship between operators in ${\mathcal F}(S,M)$ and ${\mathcal F}_{\alpha,\Xi}$.
If  $\alpha=(S,j,M,\emptyset,u)$ is,   as before,  a smooth un-noded stable map, we introduce  the following definition.
\begin{definition}
{\bf A stab-set pair $(\Xi,\Xi_0)$ for $\alpha$}  consists of two stab-sets $\Xi$ and $\Xi_0$ for $\alpha$ such  that $\Xi_0\subset \Xi$.
\end{definition}
We allow $\Xi_0=\emptyset$ and identify in this case the stab-set pair $(\Xi,\emptyset)$ with $\Xi$. Associated with  a stab-set pair we introduce the following distinguished  class of finite-dimensional subspaces of $H^3(u^\ast TQ)$. 

\begin{definition} Let $\alpha$ be a smooth unnoded  stable map as before and $(\Xi,\Xi_0)$ a stab-set pair.
A {\bf $(\Xi,\Xi_0)$-subspace associated with  $\alpha$} is
a finite-dimensional linear subspace
$N$ of $H^3(u^\ast TQ)$ having the following properties.
\begin{itemize}
\item[(1)] $\dim_{\mathbb R}(N)=2\cdot (\sharp\Xi-\sharp\Xi_0)$.
\item[(2)] If  $\eta\in N$,  then  $\eta(z)\in R(Tu(z))$ for all $z\in\Xi\setminus\Xi_0$ and $\eta(z)=0$ for $z\in\Xi_0$.
\item[(3)] The evaluation map $N\rightarrow \bigoplus_{z\in\Xi\setminus\Xi_0} R(Tu(z))$, $\eta\mapsto (\eta(z))_{z\in\Xi\setminus\Xi_0}$
is a linear isomorphism.
\end{itemize}
\end{definition}
{The $(\Xi,\Xi_0)$-subspace  $N$ inherits a natural orientation by the isomorphism  in (3).}

\begin{remark}
For a smooth stable map 
$\alpha=(S,j,M,\emptyset,u)$ and a stab-set $\Xi$ having the associated associated constraints $H$,  we consider the linear subspace $H^3_H(u^\ast TQ)$ of $H^3(u^\ast TQ)$. If $N$ is a $\Xi$-subspace  (where we have identified $\Xi$ with the stab-set $(\Xi, \emptyset)$, the map
$$
N\oplus H^3_H(u^\ast TQ)\rightarrow H^3(u^\ast TQ), \quad (v,\eta)\mapsto v+\eta
$$
is a topological linear isomorphism. 
We also note that,  given a stab-set pair $(\Xi,\Xi_0)$ and the constraints $H$ associated with  $\Xi$,  we can take 
a $\Xi_0$-subspace $N_0$ and a $(\Xi,\Xi_0)$-subspace $N_1$,  and obtain the linear topological isomorphism
$$
N_0\oplus N_1\oplus H^3_H(u^\ast TQ)\rightarrow H^3(u^\ast TQ)\quad (a,b,\eta)\rightarrow a+b+\eta.
$$
Similarly,  if $H_0$ is the subset of $\Xi_0$-constraints which is obtained by restricting the constraint $H$ of $\Xi$ to the subset  $\Xi_0\subset \Xi$, and if  $N_1$ is a $(\Xi,\Xi_0)$-subspace,  we obtain the linear topological isomorphism
$$
N_1\oplus H^3_H(u^\ast TQ)\rightarrow H^3_{H_0}(u^\ast TQ),\quad  (a,\eta )\mapsto a+\eta.
$$
\end{remark}

 By ${\mathcal S}_{\alpha,(\Xi,\Xi_0)}$  we denote the  collection of all $(\Xi,\Xi_0)$-subspaces of $H^3(u^\ast TQ)$.  Fixing a norm $\norm{\cdot}$  on $H^3(u^\ast TQ)$,  we equip ${\mathcal S}_{\alpha,(\Xi,\Xi_0)}$  with the metric defined  by  the Hausdorff distance of the unit-spheres in  the  $(\Xi,\Xi_0)$-subspaces.

\begin{lemma}\label{vvvv}
The topological space ${\mathcal S}_{\alpha,(\Xi,\Xi_0)}$ is contractible.
\end{lemma}
\begin{proof}
Every subset of a stab-set $\Xi$  is again a stab-set. In particular, if 
$z\in \Xi$, then the set $\{z\}$ is a stab-set and we associate it with a special two-dimensional subspace. We modify the element $\eta$ of the $\{z\}$-subspace by 
a local construction using  a  cut-off function to achieve that 
$\eta(z')=0$ for all $z'\in \Xi\setminus\{z\}$ and such that  the properties (1)-(3) hold true for the stab-set $\{z\}$.
If $\Xi_0\subset \Xi$, then the direct sum over all 
$z\in \Xi\setminus\Xi_0$ of the  subspaces constructed this way is a $(\Xi,\Xi_0)$-subspace possessing the properties (1)-(3).
This shows that there exist  $(\Xi,\Xi_0)$-subspaces.  

For every $z\in \Xi$  we fix an ordered basis $\{h_z, k_z\}$ of the symplectic vector space $R(Tu(z))$ satisfying 
$\omega (u(z))(h_z, k_z)>0$ which defines the same orientation as the orientation induced from $j$ on $T_zS$ via $Tu(z)$. In view of the properties (1)-(3), there exist in every 
$(\Xi,\Xi_0)$-subspace $N$, for every $z\in \Xi\setminus\Xi_0$, uniquely determined  vectors $\delta_z^N, \gamma_z^N\in N$ satisfying  $\delta_z^N(z)=h_z$, $\gamma_z^N(z)=k_z$
 and $\delta_z^N(z'),\gamma_z^N(z')=0$ for $z'\neq z$. If we fix a $(\Xi,\Xi_0)$-subspace $N_0$, there is,  for every $(\Xi\setminus \Xi_0)$-subspace,   a uniquely determined isomorphism, mapping $\delta^N_z$ and $\gamma_z^N$, respectively,
 to $\delta^{N_0}_z$ and $\gamma_z^{N_0}$, respectively,  for all $z\in \Xi\setminus\Xi_0$.
The map 
$$
\Gamma:[0,1]\times {\mathcal S}_{\alpha,(\Xi,\Xi_0)}\rightarrow {\mathcal S}_{\alpha,(\Xi,\Xi_0)}
$$
is defined by
$$
\Gamma(t,N) =\text{span}\{ (1-t)\delta_z^N +t\delta_z^{N_0}, (1-t)\gamma_z^N +t\gamma_z^{N_0}\ |\ z\in \Xi\setminus\Xi_0\}.
$$
In order to verify that for $t$ fixed  the vector space
$\Gamma(t,N)$ is a $(\Xi,\Xi_0)$-subspace we first observe that, by construction,  
$\dim \Gamma(t,N)\leq  2\cdot(\sharp\Xi-\sharp\Xi_0)$.
On the other hand, we conclude from 
$$
((1-t)\delta_z^N +t\delta_z^{N_0})(z)=h_z\quad \text{and}\quad  ((1-t)\delta_z^N +t\delta_z^{N_0})(z')=0,
$$
and 
$$
((1-t)\gamma_z^N +t\gamma_z^{N_0})(z)=k_z\quad \text{and}
\quad ((1-t)\gamma_z^N +t\gamma_z^{N_0})(z')=0,
$$
for $z\in\Xi\setminus\Xi_0$ and $z'\in\Xi_0$, that 
$\dim \Gamma(t,N)\leq  2\cdot(\sharp\Xi-\sharp\Xi_0)$ and the evaluation map 
$$ \Gamma(t,N)\to \oplus_{z\in \Xi\setminus \Xi_0}R(Tu(z)), \quad 
\eta\mapsto \bigl(\eta (z)\bigr)_{z\in \Xi\setminus \Xi_0}$$
is surjective.
Hence, $\dim \Gamma(t,N)=2\cdot(\sharp\Xi-\sharp\Xi_0)$ and $\Gamma(t,N)$ belongs indeed to the space ${\mathcal S}_{\alpha, \Xi,\Xi_0}$. We note that, in particular, 
$$
\delta^{\Gamma(t,N)}_z =(1-t)\delta_z^N +t\delta_z^{N_0}\quad  \text{and}\quad  \gamma^{\Gamma(t,N)}_z =(1-t)\gamma_z^N +t\gamma_z^{N_0}.
$$
In view of 
$\Gamma(0,N)=N$ and $\Gamma(1,N)=N_0$, the space 
${\mathcal S}_{\alpha, \Xi\setminus \Xi_0}$ is contractible.
\end{proof}

Let us recall  that $R(Tu(z))$ has a natural orientation given by the ordered basis $\{h_z, k_z\}$, which agrees with  the
orientation induced  from $(T_zS,j)$ by  $Tu(z)$.
Thus we obtain a natural orientation of  the vector space 
$ \bigoplus_{z\in\Xi\setminus\Xi_0} R(Tu(z))$.  This space is,  by definition,  the space
 of maps which associate with every point  $z\in\Xi\setminus\Xi_0$ a vector in $R(Tu(z))$.  Take a numbering $z_1,\ldots ,z_{\sharp(\Xi\setminus\Xi_0)}$ and take
 as a basis 
 $[z\rightarrow \delta_{z_1}(z)h_{z_1}],[z\rightarrow \delta_{z_1}(z)k_{z_1}],\ldots$. Here $\delta_{z_i}(z)=1$ if and only if  $z=z_i$ and otherwise its value is $0$.
 Then the associated orientation
 does not depend on the numbering of the points in $\Xi\setminus \Xi_0$. Using the canonical isomorphism from property (3), 
 \begin{equation}\label{o<?}
 N\rightarrow \bigoplus_{z\in\Xi\setminus\Xi_0} R(Tu(z)),\quad  \eta \mapsto  {(\eta (z))}_{z\in\Xi\setminus\Xi_0},
\end{equation}
 we see that every $(\Xi,\Xi_0)$-subspace $N$ inherits  a natural orientation for which  
 the map \eqref{o<?} is orientation-preserving.

 \begin{definition}
 Let $\alpha$ be a smooth un-noded stable map and $(\Xi,\Xi_0)$ a stab-set pair. The orientation
of the $(\Xi,\Xi_0)$-subspace $N$ of $H^3(u^\ast TQ)$,  for which  the linear isomorphism
 in (\ref{o<?}) is orientation preserving,  is called the {\bf natural orientation for the $(\Xi,\Xi_0)$-subspace $N$}.
 \end{definition}
 
 Note that  the proof of Lemma \ref{vvvv} proves also  the following useful lemma.
 \begin{lemma}
 Let $\alpha$ be a smooth un-noded stable map and $(\Xi,\Xi_0)$ a stab-pair.  If $N_0$ is a $\Xi_0$-subspace and
 $N_1$ a $(\Xi,\Xi_0)$-subspace, then $N=N_0\oplus N_1$ is a $\Xi$-subspace and the map
 $$
 N_0\oplus N_1\rightarrow N, \quad (a,b)\mapsto  a+b
 $$
 is orientation-preserving. (This is also true if we replace $N_0\oplus N_1$ by $N_1\oplus N_0$  since $N_0$ and $N_1$ are even-dimensional.)
 \end{lemma}

The  stab-set pair $(\Xi,\Xi_0)$ and the constraints $H$ for $\Xi$ induce the  set $H_0$ of constraints for $\Xi_0$. Given a $(\Xi,\Xi_0)$-subspace $N$ we have the topological isomorphism
$$
N\oplus H_{H}^3(u^\ast TQ)\rightarrow H^3_{H_0}(u^\ast TQ),\quad (v,\eta)\mapsto  v+\eta.
$$
We note that $N$ and $ H_{H}^3(u^\ast TQ)$ are linear subspaces of $H_{H_0}^3(u^\ast TQ)$.
Given the operator ${\bf L}\in {\mathcal F}_{\alpha,\Xi_0}$ we can define its restriction ${\bf L}'$ to 
$H_{H}^3(u^\ast TQ)$. If  $P:H^3_{H_0}(u^\ast TQ)\rightarrow H^3_{H_0}(u^\ast TQ)$ is  the projection onto
$H^3_{H}(u^\ast TQ)$ along $X$,  the composition 
 ${\bf L}\circ P$  is also an operator in ${\mathcal F}_{\alpha,\Xi_0}$, because 
${\bf L}\circ (Id-P)$ is a compact operator.

\begin{proposition}
We consider for a stab-pair $(\Xi,\Xi_0)$ and the constraints $H$ associated  with  $\Xi$ and the restriction  $H_0=H\vert \Xi_0$,   an operator ${\bf L}\in {\mathcal F}_{\alpha,\Xi_0}$ on 
$H^3_{H_0}(u^\ast TQ)$
and its restriction  ${\bf L}'$ to  $H^3_{H}(u^\ast TQ)$ as described above.  
We assume that the operators in ${\mathcal F}_{\alpha,\Xi}$ and ${\mathcal F}_{\alpha,\Xi_0}$ are equipped with the complex orientations. If the complex orientation of ${\bf L}'$ is defined by the vectors  by $(h_1\wedge\ldots \wedge h_k)\otimes (h_1^\ast\wedge\ldots\wedge h_l^\ast)$, then the complex orientation of ${\bf L}\circ P$ is given by
$(a_1\wedge\ldots\wedge a_{2\sharp(\Xi\setminus\Xi_0)}\wedge h_1\wedge\ldots\wedge h_k)\otimes (h_1^\ast\wedge\ldots\wedge h_l^\ast)$,
where $a_1\wedge\ldots\wedge a_{2\sharp(\Xi\setminus\Xi_0)}$ defines  the orientation of the 
$(\Xi, \Xi_0)$-subspace $N$.
\end{proposition}

\begin{proof}
By a homotopy argument we may assume that ${\bf L}$ is complex linear. For example,  we take the constraints
$(H_z)$ and $J$  satisfying  $J(u(z))H_z=H_z$ and $\nabla J=0$,  and ${\bf L}$ of the form 
$$
\eta\mapsto \nabla\eta+J(u)(\nabla\eta)j.
$$
In this case ${\bf L}'$ is also complex linear. Moreover, we can even assume,  in addition,  that $J(u(z))(R(Tu(z)))=R(Tu(z))$.
This implies that there is a $(\Xi,\Xi_0)$-subspace $N\in {\mathcal S}_{\alpha,(\Xi,\Xi_0)}$  which is complex linear.  Observe that the canonical orientation
on $N$ is the same as the complex orientation. In this situation the projection  $P$ is complex linear
and the homotopy
$$
[0,1]\ni t\mapsto {\bf L}\circ P + t {\bf L}\circ (\text{id}-P)\colon H^3_{H_0}(u^\ast TQ)\rightarrow H^2(u^\ast \Omega_{j,J})
$$
is a homotopy consisting of complex linear operators. Therefore,  the complex orientation of ${\bf L}$ corresponds to the 
complex orientation of ${\bf L}\circ P$. We note further that 
$$
{\bf L}\circ P | H^3_{H}(u^\ast TQ) ={\bf L}'\ \ \text{and}\ \  {\bf L}\circ P|X=0.
$$
Consequently,  $\ker({\bf L}\circ P)= N\oplus \ker({\bf L}')$ and $\ker({\bf L}')$ is a complex linear subspace,
whereas the cokernels are the same complex linear subspaces. This implies the desired result.
\end{proof}

\subsection{A Further Extension}
Next we extend the construction as follows.  We assume that $\alpha$ and $\Xi$ are  as before, and let  $E$ be a finite-dimensional oriented real vector space. In most of the applications $E$ will be a complex vector space, which we shall view as a real vector space equipped
with the orientation defined by the complex structure.
 By ${\mathcal F}_{\alpha,\Xi,E}$ we denote the collection of all operators
of the form
$$
E\oplus H^3_H(u^\ast TQ)\rightarrow H^2(u^\ast\Omega^{0,1}_{j,J}),\quad L(e,\eta)=\nabla\eta +J(u)(\nabla\eta) j +A(e,\eta),
$$
where $A$ is compact. Every 
$L\in {\mathcal F}_{\alpha,\Xi}$
 defines an operator $L_E$ in ${\mathcal F}_{\alpha,\Xi, E}$ by
$L_E(e,\eta)= L\eta$.  For fixed $\alpha$ and $E$,  two operators in ${\mathcal F}_{\alpha,\Xi,E}$ can be connected by a smooth path so that we can talk about continuation of an orientation. If $L\in {\mathcal F}_{\alpha,\Xi}$ and $\mathfrak{o}_L$
is the  orientation of  $\det(L)$ defined by the vector
$(k_1\wedge\ldots \wedge k_n)\otimes (h_1^\ast\wedge\ldots \wedge h_l^\ast)$, 
we obtain the orientation $\mathfrak{o}_{L_E}$ of  $\det(L_E)$ defined by the vector $(e_1\wedge\ldots \wedge e_{2m}\wedge k_1\wedge\ldots \wedge k_n)\otimes(h_1^\ast\wedge\ldots\wedge h_l^\ast)$,
where $e_1,\ldots ,e_{2m}$ is an oriented  basis for the even-dimensional oriented vector space $E$. The orientation $\mathfrak{o}_{L_E}$ only depends on $\mathfrak{o}_L$ and not on the choices involved in its construction.

\begin{proposition}
There is for every $\what{L}$ in ${\mathcal F}_{\alpha,\Xi,E}$ a natural orientation $\mathfrak{o}_{\what{L}}$
of $\det(\what{L})$ characterized by the following.
\begin{itemize}
\item[(1)] Along every smooth path $(\what{L}_t)_{0\leq t\leq 1}$ in ${\mathcal F}_{\alpha,\Xi,E}$, the continuation of $\mathfrak{o}_{\what{L}_0}$ 
is $\mathfrak{o}_{\what{L}_1}$.
\item[(2)] If $L\in {\mathcal F}_{\alpha,\Xi}$ then $\mathfrak{o}_{L_E}$ corresponds to the complex orientation
$\mathfrak{o}_L$ defined for $L\in {\mathcal F}_{\alpha,\Xi}$ via the above construction.
\end{itemize}
\end{proposition}
\begin{proof}
Having fixed an operator $\what{L}_0$ and  its orientation 
$\mathfrak{o}_{\what{L}_0}$,  we obtain an orientation $\mathfrak{o}_{\what{L}}$ 
for any other operator $\what{L}$ by continuation along a path connecting $\what{L}_0$ with  $\what{L}$.
The outcome does not depend on the choice of path since loops are contractible.  We now choose 
$\what{L}_0$ of the form 
$$
\what{L}_0(e,\eta)= \nabla\eta +J(u)(\nabla\eta)j,
$$
where the domain $H^3_H(u^\ast TQ)$ for $\eta$ is  complex. Then we have the distinguished
orientation for the complex operator $\eta\mapsto  \nabla\eta +J(u)(\nabla\eta)j$ and taking a complex basis for $E$
we define the  orientation $\mathfrak{o}_{L_E}$ as  described above the proposition.
Having fixed this orientation and proceeding as in the  previous proofs,  we obtain an orientation 
$\mathfrak{o}_{L}$  for every
$L\in {\mathcal F}_{\alpha,\Xi,E}$. We only need to show that it does not depend
on the specific choice of $L_0$.  But as already used before,  we can connect any two such choices
by an arc of operators in the same class.Therefore,  the propagation of the 
complex orientation is the complex orientation. This completes the proof. 
\end{proof}

It is useful to summarize our findings in this subsection.  Given a smooth stable un-noded $\alpha=(S,j,M,\emptyset,u)$,
 a stabset $\Xi$,  a collection of transversal constraints $H={(H_z)}_{z\in \Xi}$,  an oriented even-dimensional vector space $E$, 
 and a compatible almost complex structure $J$, we can consider the special operators on the  domain $E\oplus H^3_H(u^\ast TQ)$ into the  codomain $H^2(u^\ast\Omega_{j,J})$,  of the form
 $$
L(e,\eta)=\nabla\eta +J(u)(\nabla\eta)j+A(e,\eta).
 $$
 The determinants of these operators have a canonical orientation, called the complex orientation.
If  we change for a fixed stab-set $\Xi$,  any of the data $j$, $J$ or $H$,  the continuation of the complex orientation is the complex orientation.
We have also seen that the push-forward of such an operator by an isomorphism between smooth stable maps produces an operator in the same class
and that the push forward of its  complex orientation is the complex orientation.
In order to proof our main result about the orientations 
we need to  deal with  more general coordinate changes, which will be  done in the next subsection.

\begin{remark}\label{mm<0}
The obvious version of Proposition \ref{m<0} also holds.
\end{remark}

\subsection{General Coordinate Changes}

We consider the  stable un-noded marked Riemann surface $\beta=(S,j,M)$. In  the following it is not relevant  whether  $M$ is ordered, partially ordered or un-ordered.  The real dimension of the component of the Deligne-Mumford space containing the isomorphism class $[S,j,M]$ is equal to $d_\beta=6g(S)-6 + 2\sharp M$.
\begin{definition}\label{0-general}
An {\bf effective, oriented deformation germ for the stable marked surface $\beta=(S,j,M)$} consists of a real oriented vector space
$E$ of dimension $d_\beta$ and a smooth germ of deformations of $j$, 
$$
{\mathcal O}(E,0)\ni v\rightarrow j(v),
$$
satisfying  $j(0)=j$. Moreover, the Kodaira differential
$$
[Dj(0)]:E\rightarrow H^1(\alpha)
$$
is an orientation-preserving isomorphism, where the target $H^{1}(\alpha )$ carries the  orientation coming from its complex structure.
\end{definition}

We consider  tuples $(\alpha,\Xi,E)$ consisting of a smooth un-noded stable map $\alpha=(S,j,M,\emptyset, u)$, a stabilization set $\Xi$ and 
an oriented even-dimensional vector space $E$ of dimension 
$$
\dim_{\mathbb R}(E) = 6g(S)-6 + 2\sharp M+2\sharp\Xi.
$$
We would like to emphasize that $\Xi$ here is a stabilization in the original sense of Definition \ref{stab-f}. It is, in particular,  also a
stab-set.
Fix such a tuple and consider the set ${\mathcal H}$ of all associated linear constraints $H={(H_z)}_{z\in\Xi}$.
For fixed $z$ we can view the set of all associated linear subspaces  $H_z\subset T_zQ$ as an open subset of a Grassmannian. Consequently, 
this collection has in a natural way a smooth manifold structure. The collection ${\mathcal H}$ is a product of these manifolds
and therefore also a smooth manifold. We also recall that ${\mathcal H}$ is contractible. We obtain the  Hilbert space bundle ${\mathcal E}$ over ${\mathcal H}$ whose  fiber over the constraint $H$, ${\mathcal E}_H$,  is  the Hilbert space $H^3_H(u^\ast TQ)$ so that 
$$
{\mathcal E}_H:=H^3_H(u^\ast TQ).
$$

In the following we shall interpret the triples 
$$(\alpha,H,E)$$ 
as objects of a category. We recall that $\alpha$ is an un-noded smooth stable  map $\alpha=(S, j, M, \emptyset, u)$, the marked Riemann surface $(S, j, M)$ is equipped with the stabilization $\Xi$ and $H=(H_z)_{z\in \Xi}$ is an associated constraint. Moreover, $E$ is an oriented even dimensional vector space of real dimension $\dim_{\mathbb R}(E) = 6g(S)-6 + 2\sharp M+2\sharp\Xi$.

In order to define the morphisms of the category having as objects the triples $(\alpha,H,E)$, we take a second 
object $(\alpha',H',E')$ in which $\alpha'=(S', j', M', \emptyset, u')$ is a stable map with the 
stabilization $\Xi'$, the constraint $H'$, and the vector space $E'$. 
 The morphisms 
$$
(\alpha,H,E)\rightarrow (\alpha',H',E')
$$
of the category consists of pairs 
$(\phi,K)$ in which 
$$\phi:\alpha\rightarrow \alpha'$$
is an isomorphism of un-noded stable maps and  
$$K:E\oplus H^3_H(u^\ast TQ)\rightarrow E'\oplus  H^3_{H'}({(u')}^\ast TQ)$$ is a topological linear isomorphism
which is constructed as follows.

We choose a smooth map $\Lambda_u$  defined on an open neighborhood of the zero-section in $u^\ast TQ$ with image in $Q$
and having the following properties.
\begin{itemize}
\item[(1)]  $\Lambda_u(z,0_{u(z)})=u(z)$ for $z\in S$.
\item[(2)] $T_{(z,0_{u(z)})}\Lambda_u:{( u^\ast TQ)}_z=T_{u(z)}Q\rightarrow T_{u(z)}Q$ is the identity map.
\end{itemize}

A typical example is $\Lambda_u=\exp_u$ where $\exp:TQ\rightarrow Q$ is the exponential map of a  Riemann metric on the manifold $Q$ and 
$\exp_u$  is the pull-pack defined by the smooth map by $u:S\to Q$.  For the second triple $(S', j', M')$ we choose an analogous map $\Lambda_{u'}'$ having the analogous properties. Next we choose 
an effective oriented deformation germ  $v\mapsto j(v)$ for $v\in E$ associated with the stable surface $(S,j,M\cup\Xi)$ and satisfying $j(0)=j$. Similarly we take another
effective oriented deformation germ  $v'\mapsto j'(v')$ on $E'$ for the stable surface $(S', j', M'\cup \Xi')$ satisfying $j'(0)=j'$. 

In view of Proposition \ref{imremark}, there are associated uniformizers 
\begin{align*}
\alpha (v, \eta)&=(S, j(v), M, \emptyset, \Lambda_u(\eta))\\
\intertext{and}
\alpha' (v', \eta')&=(S', j'(v'), M', \emptyset, \Lambda'_{u'}(\eta')),
\end{align*}
defined for $(v, \eta)$ near $(0, 0)$ in $E\oplus H^3_{H}(u^*TQ)$, respectively for 
$(v', \eta')$ near $(0, 0)$ in $E'\oplus H^3_{H'}((u')^*TQ)$. Theorem \ref{key-z} guarantees a unique germ 
$$f:(v, \eta)\mapsto (v', \eta')=f(v, \eta)$$
of an sc-diffeomorphism between the parameter spaces satisfying $f(0, 0)=(0, 0)$ and a smooth germ 
$$\phi_{(v,\eta)}:\alpha_{(v,\eta)}\rightarrow \alpha'_{f(v,\eta)}$$
of isomorphisms between the stable maps satisfying  $\phi_{(0,0)}=\phi$.

The topological linear isomorphism  $K$ is defined by
$$
K:=Df(0,0):E\oplus H^3_H(u^*TQ)\to 
E'\oplus H^3_{H'}((u')^*TQ).
$$

In  other words,  the morphisms in our category are pairs $(\phi,K)$ in which  $\phi:\alpha \to \alpha'$ is an isomorphism between stable maps and $K$ is the linearization of a coordinate change associated with  $\phi$.

The identity morphisms of the category are given by  
$1_{(\alpha,H,E)}=(\text{id},\text{id})$ and if 
the morphisms can be composed, the multiplication of two morphisms is simply 
 $(\phi,K)\circ (\psi,K')=(\phi\circ\psi,K\circ K')$.  
\begin{definition}\label{GENERAL}
The morphism $(\phi, K):(\alpha, H, E)\to (\alpha', H', E')$ constructed above is called 
a {\bf  general linear coordinate change}.
\end{definition}
\begin{lemma}
Given $\alpha$, $\alpha'$, the isomorphism $\phi:\alpha\rightarrow \alpha'$, the constraints $H$ and $H'$,
 and the effective oriented germs $v\rightarrow j(v)$ and $v'\rightarrow j'(v')$,  the topological linear isomorphism
 $$
 K:E\oplus H^3_H(u^\ast TQ)\rightarrow E'\oplus H^3_{H'}({(u')}^\ast TQ)
 $$
 is uniquely determined by $\phi$ and the  differentials $Dj(0)$ and $Dj'(0)$. It does not depend on the choices
 of $\Lambda_u$ and $\Lambda_{u'}'$ as long as they have the stated properties.
 \end{lemma}
 \begin{proof}
  To see this we define $\Gamma_0=\phi^{-1}(\Xi')$. Since $u'\circ \phi =u$ we see that
$u$ intersects the nonlinear constraints associated with  $\Xi'$ at the points in $\Gamma_0$. If we deform $u$,  say to
$\Lambda_u(\eta)$,  the intersection set will be a small deformation $\Gamma=\Gamma_\eta$ of $\Gamma_0$.
For convenience we may assume that the points in $\Gamma_0$ are numbered and those in $\Gamma$ are numbered as well,
so that the $i$-th point in $\Gamma$ is the small deformation of the $i$-th point in $\Gamma_0$. Since $\Gamma_0$ is numbered
we obtain a numbering of $\Xi'$ via $\phi$.
Now there exists for given $v\in E$ near $0$ a unique $v'=v'(v,\Gamma)$
near $0\in E'$, and a unique  diffeomorphism $\phi_{v,\Gamma}$   near $\phi$,  which maps $\Gamma$ to $\Xi'$ preserving the numbering, but also
satisfies for the marked points $\phi_{(v,\Gamma)}(m_i)=m_i'$,  so that 
\begin{equation}\label{r<m}
j'(v'(v, \Gamma))\circ T\phi_{(v,\Gamma)} = T\phi_{(v,\Gamma)}\circ j(v).
\end{equation}
 This is proved by an implicit functions theorem in \cite{HWZ-DM}. Indeed,  following the argument there, $F(v,\Gamma,v', \phi_{(v,\Gamma)})=0$, where 
 $F$ is a first order elliptic differential operator in the fourth entry satisfying $F(0,\Gamma_0,0,\phi)=0$,  and for every small deformation $(v,\Gamma)$ of $(0,\Gamma_0)$ there is a unique solution $(v'(v, \Gamma),\phi_{(v,\Gamma)})$ near $(0,\phi)$. The linearization of $F$ with respect to the third and fourth variable
  at $(0,\phi)$ is an isomorphism. If we differentiate $F(v,\Gamma,v'(v,\Gamma),\phi_{(v,\Gamma)})=0$ with respect $(v,\Gamma)$ at $(0,\Gamma_0)$
  we obtain
  \begin{eqnarray*}
  0 &= &(D_{1,2}F)(0,\Gamma_0,0,\phi)(\delta v,\delta\Gamma)\\
  & +& (D_{3,4}F)(0,\Gamma_0,0,\phi)(Dv'(0,\Gamma_0)(\delta v,\delta\Gamma),(D\phi)_{(0,\Gamma_0)}(\delta v,\delta\Gamma)).
  \end{eqnarray*}
  The $D_{(3,4)}$-term is a linear isomorphism determining 
  $$
Dv'(0,\Gamma_0)(\delta v,\delta\Gamma),(D\phi)_{(0,\Gamma_0)}(\delta v,\delta\Gamma)).
  $$
  uniquely in terms of the $D_{(1,2)}$-term.  From (\ref{r<m}) it follows that it is uniquely determined in terms of $\phi$, $Dj(0)$ and $Dj'(0)$. 
  The linear isomorphism $K$ is the linearization of the following composition. For small $(v,\eta)$ we obtain a deformation $\Gamma_\eta$ of $\Gamma_0$
  where $\Gamma_\eta$ consists of the points $\gamma_i$ at which the map $\exp_u(\eta)$ intersects $H_{z_i'}'$. Then we consider the germ
  $$
  (v,\eta)\rightarrow (v,\Gamma_\eta,\Lambda_u(\eta))\rightarrow (v'(v,\Gamma_\eta),\eta'(v,\eta)),
  $$
  where $\eta'=\eta'(v,\eta)$ is defined by $\Lambda_{u'}'(\eta') \circ \phi_{(v,\Gamma_{\eta})} =\Lambda_u(\eta)$.
  From this expression we conclude  that $K$, the derivative at $(0,0)$,  is determined by the derivative of $\phi_{(v,\Gamma)}$ at $(0,\Gamma_0)$.
\end{proof}

We consider,  for two triplets $(\alpha,H,E)$ and $(\alpha',E',H')$ and a fixed isomorphism $\phi:\alpha\rightarrow \alpha'$,  the collection of all morphisms
$(\phi, K):(\alpha,H,E)\rightarrow (\alpha',H',E')$. We equip  this collection  with  a metric defining the distance
between $(\phi,K)$ and $(\phi,K')$  by  the 
 operator norm $\norm{K-K'}$. 
 
\begin{lemma}\label{i<<i}
For fixed triples $(\alpha,H,E)$ and $(\alpha',H',E')$  and fixed isomorphism  $\phi:\alpha\rightarrow \alpha'$, 
the collection of all morphisms of the form $(\phi,K):(\alpha,H,E)\rightarrow (\alpha',H',E')$ is a connected space.
\end{lemma}

\begin{proof}
Having $(H,E)$ and $(H',E')$ fixed,  the operator  $K$ depends only on $\phi$, and the differentials of $v\rightarrow j(v)$ and $v'\rightarrow j'(v')$ at $0$.
Given two deformations $j_1(v)$ and $j_2(v)$  of $j$ with effective Kodaira differentials at $v=0$, we find a smooth arc of germs of deformations
$(k_t)_{1\leq t\leq 2}$ connecting them such  that,  for $i=1,2$,   $k_i$ and $j_i$ have the same differentials and the Kodaira differentials are effective for all $t\in [1,2]$.
Using the deformation $t\mapsto  k_t$
and the analogous deformation $t\mapsto k_t'$,  we can construct a smooth family $(f_t, \phi_i)$ of germs 
of coordinate changes and the resulting arc $Df_t(0, 0)$ of derivatives at $(0,0)$ is a homotopy between the morphisms
associated with the morphisms   for  $ i=1,2$. This completes the proof.
\end{proof}
\begin{remark}In fact,  something more general is true due to the contractibility and therefore connectivity of ${\mathcal H}$ and ${\mathcal H}'$, respectively.
Namely,  given $\phi:\alpha\rightarrow \alpha'$,  the collection of all morphisms
$$
(\phi,K):(\alpha,H,E)\rightarrow (\alpha',H',E')
$$
where $H$ varies in ${\mathcal H}$ and $H'$ in ${\mathcal H}'$,  is connected.
\end{remark}
 \begin{lemma}
If  $(\phi,K):(\alpha,H,E)\rightarrow (\alpha',H',E')$ is  a morphism, and if 
the  the operator ${\bf L}:E\oplus H^3_H(u^\ast TQ)\rightarrow H^2(u^\ast\Omega_{j,J})$
belongs to the distinguished class ${\mathcal F}_{\alpha,\Xi,E}$, 
then the   operator ${\bf L}':E'\oplus H^3_{H'}(u^\ast TQ)\rightarrow H^2({(u)}^\ast\Omega_{j',J})$, 
defined by
$$
{\bf L}'(K(v,\eta))\circ T\phi = {\bf L}(v,\eta),
$$
 belongs to the class ${\mathcal F}_{\alpha',\Xi',E'}$.
 \end{lemma}
 \begin{remark}
 We shall call ${\bf L}'$ the {\bf push-forward}  of ${\bf L}$ and denote it by $(\phi,K)_\ast{\bf L}$. We can use $(\phi,K)$
 in the obvious way to push-forward an orientation $\mathfrak{o}$ of ${\bf L}$ to an orientation $(\phi,K)_\ast\mathfrak{o}$ of
 $(\phi,K)_\ast{\bf L}$.
 \end{remark}
 \begin{proof}
We recall that the morphism $(\phi,K)$ is obtained from a germ of an sc-diffeomorphism $f:(v, \eta)\to (v', \eta')=f(v, \eta)$ in the parameter space, satisfying $f(0, 0)=(0, 0)$,  together with a smooth germ $\phi_{(v, \eta)}:\alpha_{(v, \eta)}\rightarrow \alpha_{f(v, \eta)}'$ of isomorphism between stable maps satisfying  $\phi_{(0,0)}=\phi$.
Abbreviating the maps $\Lambda_u (\eta)=\wt{u}$ and $\Lambda'_{u'} (\eta')=\wt{u}'$ and recalling that $\wt{u}'\circ \phi_{(v,\eta)}=\wt{u}$ (by definition of an isomorphism between stable maps), we obtain, for suitable chosen $\ssc^+$-germs $s$ and $s'$ of sections, the identity 
\begin{equation*}
\begin{split}
\biggl(\dfrac{1}{2}\bigl[ &T\wt{u}'+J(\wt{u}')(T\wt{u}')j'(v')\bigr]-s'(v',\eta')\biggr)\circ T\phi_{(v, \eta)}\\
&= \dfrac{1}{2}\bigl[ T\wt{u}+J(\wt{u})\circ T\wt{u}\circ j(v)\bigr]-s(v,\eta).
\end{split}
\end{equation*}
 Recalling that $Df(0, 0)=K$, we set 
 $K(\delta v,\delta \eta)=(\delta v', \delta \eta')$ for $(\delta v,\delta \eta)\in E\oplus H^3_H(u^*TQ)$, and obtain at the point $(v, \eta)=(0, 0)$ the identity 
\begin{equation*}
\begin{split}
\biggl( \dfrac{1}{2}\bigl[&\nabla'(\delta \eta')+J(u')\nabla'(\delta \eta')\circ j' \bigr] +A'(\delta v', \delta\eta')\biggr)\circ T\phi\\
&=
 \dfrac{1}{2}\bigl[\nabla (\delta\eta)+J(u)\nabla(\delta \eta))j\bigr]+A(\delta v, \delta \eta).
 \end{split}
\end{equation*}
Therefore, 
$$
{\bf L}'\bigl(K(\delta v', \delta \eta')\bigr)\circ T\phi={\bf L}(\delta v, \delta \eta)
$$
and we see that ${\bf L}'$ belongs to the class ${\mathcal F}_{\alpha', \Xi', E'} $ of operators.
\end{proof}
The main result of this section is the following theorem.
\begin{theorem}\label{kkk<0}
Let $(\phi,K):(\alpha,H,E)\rightarrow (\alpha',H',E')$ be a morphism and ${\bf L}:E\oplus H^3_H(u^\ast TQ)\rightarrow H^2(u^\ast\Omega_{j,J})$ an operator in $ {\mathcal F}_{\alpha,\Xi,E}$.  Denote by ${\bf L}'=(\phi,K)_\ast{\bf L}:E'\oplus H^3_{H'}({(u')}^\ast TQ)\rightarrow H^2({(u')}^\ast\Omega_{j,J})$ the operator in ${\mathcal F}_{\alpha',\Xi',E'}$ which is the push-forward of ${\bf L}$.
 Then the push-forward of the complex orientation by $(\phi,K)$ is the complex orientation, i.e.
 $$
 (\phi,K)_{\ast}\mathfrak{o}_{\bf L} =\mathfrak{o}_{(\phi,K)_\ast{\bf L}}.
 $$
\end{theorem}
 We first prove the following special case of the theorem.

\begin{proposition}\label{ert<ll}   We consider  the smooth stable map 
$\alpha=(S,j,M,\emptyset,u)$  and let $\Xi$ and $\Xi_0$ be stabilizations  satisfying  $ \Xi_0\subset\Xi$.
Let $H_0$ be the constraints associated with $\Xi_0$ and $H$ the constraints associated with  $\Xi$. Let
$$
(\phi,K):(\alpha,H_0,E_0)\rightarrow (\alpha,H,E)
$$
be a morphism.  If ${\bf L}=(\phi,K)_\ast{\bf L}_0\in {\mathcal F}_{\alpha, \Xi,E}$ is the push-forward of  ${\bf L}_0\in {\mathcal F}_{\alpha, \Xi_0,E_0}$, 
then the push forward of the complex orientation 
$\mathfrak{o}_{\bf L_0}$ is the complex orientation $\mathfrak{o}_{\bf L}$, 
so that $$
(\phi,K)_\ast\mathfrak{o}_{{\bf L}_0} = \mathfrak{o}_{\bf L}.
$$
\end{proposition}
\begin{proof}
In order to verify the statement we are allowed, using a homotopy argument, to assume that $H_0$ is the restriction of $H$ to $\Xi_0$.
Again by a homotopy argument we also may take particular  deformations $v\mapsto  j_0(v)$ and $w\mapsto j(w)$. We can also change $J$ for the same reason, but we have to use the same structure
for both operators. Further we may assume that $\phi=\text{id}$. 

By the previous discussion,   being able to homotope problems and morphisms, it suffices to make a computation for well-chosen data. 
Consider $(S,j,M\cup\Xi_0)$. We take an effective orientation preserving germ 
$v\rightarrow j_0(v)$ defined on $E_0$ which, therefore, has a bijective, orientation-preserving  Kodaira differential at $v=0$. Moreover we assume  for small $v\in E_0$ that $j_0(v)=j$ near the points in $\Xi$.  Writing $\beta=(S,j,M\cup\Xi_0)$
and letting $\Gamma=\Xi\setminus\Xi_0$ we are going to apply the following lemma whose proof is based on a calculation of the Kodaira differential in \cite{HWZ-DM}.
\begin{lemma}\label{lem_Kodairs_diff}
Let $\beta=(S,j, M)$ be a stable marked Riemann surface  and let $\Gamma=\{z_1,\ldots, z_k\}$ be a collection of finitely many mutually different points not belonging to the set $M$ of marked points. We assume that $j_0(v)$ is an effective oriented deformation of $j$ defined for $v\in E_0$. We assume that $\phi_i:{\mathcal O}(T_{z_i}S, 0)\to S$ are germs of smooth embeddings satisfying 
$$\phi_i(0)=z_i\quad \text{and}\quad \dfrac{\partial \phi_i}{\partial w_i}(0)=\text{id}.$$
Define for small $w=(w_1,\ldots ,w_k)\in E_1$, where 
$$E_1=\bigoplus_{i=1}^kT_{z_i}S,$$
the deformation $\Gamma_w$ of $\Gamma$ by 
$\Gamma_w=\{\phi_1(w_1),\ldots ,\phi_k(w_k)\}.$ Finally, let 
$\Psi_w=\Psi (w, \cdot ):S\to S$ be a germ of smooth family of diffeomorphisms having compact supports near the points in $\Gamma$ and satisfying $\Psi_w(M)=M$ and $\Psi_w(\phi_i(w_i))=z_i$.  Moreover, $\Psi(0, z)=z$  for every $z\in S$ at $w=0$. Define for $(v, w)$ near $(0,0)\in E_0\oplus E_1$ the almost complex structure $j(v, w)$ on $S$ by $j(v, w)=\Psi^*_wj_0(v)$. Then the germ 
$$(v, w)\mapsto (S, j(v, w), M\cup \Gamma)$$
is an effective and orientation preserving deformation germ of $(S, j, M\cup \Gamma)$ on $(v, w)\in E=E_0\oplus E_1$.
\end{lemma}

We apply the lemma to the situation  in which the stable surface is $\beta=(S, j, M\cup \Xi_0)$ and 
$\Gamma$ is the set $\Gamma=\Xi\setminus \Xi_0=\{z_1,\ldots ,z_k\}$. From the lemma we obtain, using that $\Psi_w$ is the identity away from the points in $\Gamma$, for small $(v, w)\in E_0\oplus E_1$, the isomorphisms
$$\Psi_w:(S, j_0(w), M\cup \Xi_0\cup\Gamma_w)\to (S, j(v,w), M\cup \Xi)$$
between marked Riemann surfaces. In particular, 
$$T\Psi_w\circ j_0(v)=j(v, w)\circ T\Psi_w.$$
Moreover, 
\begin{equation}\label{eq_z_i}
\Psi_w(\phi_{i}(w_i))=z_i\quad  \text{for all $z_i\in\Xi\setminus\Xi_0$.}
\end{equation}
Taking the derivative of the map 
$w\mapsto \Psi_w(\phi_{i}(w))$ at the point $w=0$ in the direction 
$\delta w=(\delta w_1,\ldots,\delta w_k)\in \bigoplus_{i=1}^k T_{z_i}S$, and using that 
$\frac{\partial \phi_i}{\partial w_i}(0)=\text{id}$ and 
$\Psi (0, z)=z$ for all $z\in S$, we obtain the relation 
\begin{equation}\label{eq_z_i_1}
\sum_{s=1}^k\dfrac{\partial\Psi}{\partial w_s}(0, z_i)\delta w_s+\delta w_i=0
\end{equation}
for all $i=1,\ldots ,k$.

We  next consider  the germs 
\begin{gather*}
(v,\eta)\rightarrow (S,j_0(v),M,\emptyset,\Lambda_u(\eta))\\
(v,w,\xi)\rightarrow (S,j(v,w),M,\emptyset,\Lambda_u(\xi))
\end{gather*}
of good uniformizing families of stable maps. From Theorem \ref{key-z} we obtain a germ of an 
sc-smooth diffeomorphism in the parameter space
$$f:(v,\eta)\mapsto ((v, w), \xi)=f(v, \eta),$$
where $w=w(\eta)$ and $\xi =\xi (\eta)$ 
satisfy $w(0)=0$ and $\xi(0)=0$. The first aim is to compute the derivative $K=Df(0,0)$,
$$
Df(0,0)(\delta v,\delta \eta)=\biggl(\delta v, \dfrac{\partial w}{\partial \eta}(0)\delta \eta, \dfrac{\partial \xi}{\partial \eta}(0)\delta \eta \biggr).
$$
The diffeomorphism $f$ satisfies 
\begin{equation}\label{e<rr}
\Lambda_u(\eta) = (\Lambda_u(\xi (\eta)))\circ \Psi_{w(\eta)}.
\end{equation}
We recall from the proof of 
Theorem \ref{key-z}
that the map $w(\eta)$ is uniquely determined by the requirement 
$$\Lambda_u(\eta)(\phi_i(w_i(\eta))\in \Lambda_{u}(z_i, H_{z_i})$$
for $i=1,\ldots, k$, so that the  section 
$\Lambda_u(\eta)$ satisfies the constraints at the points $\phi_i(w_i(\eta))$. 
Taking the derivative of the maps 
$\eta\mapsto \Lambda_u(\eta)(\phi_i(w_i(\eta)))$ in the direction of $\delta\eta$ at the point $\eta=0$, 
we obtain, in view of $w_i(0)=0$ and $\phi_i(0)=z_i$, the relations 
\begin{equation*}\label{l<<L}
\delta\eta (z_i) + Tu(z_i)(\delta w_i) \in H_{z_i}, \quad i=1,\ldots, k,
\end{equation*}
where 
$$\delta w_i=\dfrac{\partial w_i}{\partial \eta}(0)\delta \eta\in T_{z_i}S.$$
Denoting by  $\pi_i:T_{u(z_i)}Q\rightarrow T_{u(z_i)}Q$ the projection onto $R(Tu(z_i))$, we conclude from  $\pi_i(H_{z_i})=0$ the equation 
\begin{equation}\label{u<u}
\delta w_i =\dfrac{\partial w_i}{\partial \eta}(0)\delta \eta = - Tu(z_i)^{-1}\pi_i (\delta\eta(z_i)).
\end{equation}
Differentiating the equation \eqref{e<rr} at the point $\eta=0$ in the direction of $\delta\eta$ and 
recalling that $\Psi(0, z)=z$ for all $z\in S$ and $\phi_i(w_i(0))=\phi_i(0)=z_i$,
we obtain, using \eqref{u<u}, the equation 
\begin{equation}\label{o<u}
\delta \eta (z) - Tu (z) \bigl(  \sum_{s=1}^k \frac{\partial\Psi}{\partial w_s}(0, z)\delta w_s\bigr) =\delta \xi (z)
\end{equation}
where, recalling that $\eta\mapsto \xi (\eta)$ satisfies $\xi(0)=0$, 
$$\delta \xi (z)=\biggl(\dfrac{\partial \xi}{\partial \eta}(0)\delta \eta\biggr)(z).$$
We denote the left hand side  of \eqref{o<u} by 
$P(\delta\eta)$,  so that \eqref{o<u}  becomes 
\begin{equation}\label{proj_eta}
P(\delta\eta) =\delta \xi.
\end{equation}
From $P(\delta\eta)(\zeta)=\delta \xi (\zeta)\in H_{\zeta}$ for all $\zeta\in \Xi$ we deduce that 
$\pi_i(P(\delta\eta)(z_j)=0$ for $j=1,\ldots ,k$, and conclude that 
$$P(P\delta\eta)=P(\delta\eta).$$
Hence $P$ is a bounded linear projection 
$$P:H^3_{H_0}(u^*TQ)\to H^3_{H_0}(u^*TQ).$$
Since $\frac{\partial \psi}{\partial w_i}(0, z)=0$ at the points $z$ away from $\Gamma=\{z_1,\ldots,z_k\}$,  the image of the operator $P$ is equal to 
$$P(H^3_{H_0}(u^*TQ))=H^3_{H}(u^*TQ).$$
Moreover, in view of \eqref{o<u} and \eqref{eq_z_i_1} and \eqref{u<u},
$$
[(Id-P)\delta\eta](z_i)= - Tu (z_i)\delta w_i=\pi_i\delta \eta (z_i).
$$
Consequently, the image of the projection 
$(Id-P)$,
$$N_1=(Id-P)(H^3_{H_0}(u^*TQ)),$$
is a $(\Xi,\Xi_0)$-subspace of  $H^3_{H_0} (u^*TQ)$ and we therefore obtain the topological direct sum 
$$
E_0\oplus H^3_{H_0}(u^\ast TQ) = E_0\oplus  N_1\oplus H^3_{H}(u^\ast TQ).
$$

Accordingly  we  decompose  $\delta v\oplus \delta \eta=\delta v\oplus \delta e\oplus\delta \xi$,  and  rewrite the linear isomorphism $K=Df(0, 0)$ in the form 
$$
K(\delta v,\delta e,\delta \xi) = (\delta v,\delta w (\delta e),\delta \xi),
$$
where $\delta w(\delta e)=(\delta w_1(\delta e),\ldots ,\delta w_k(\delta e))\in E_1$ is explicitly given  by 
$$\delta w_i(\delta e)=-Tu(z_i)^{-1}\pi_i(\delta e(z_i)).$$
The operator 
$$
{\bf L}_0(\delta v,\delta e,\delta \xi):= \nabla\xi +J(u)(\nabla\xi)j
$$ 
belongs to ${\mathcal F}_{\alpha, \Xi_0,E_0}$ and its complex orientation $\mathfrak{o}_{L_0}$ is obtained from the complex orientation of the operator $H^3_{H_0}(u^*TQ)\to H^2(u^*\Omega_{j, J})$ on the right hand side and the orientation of $E_0\oplus N_1$.

Since the map $(\delta v,\delta e)\mapsto (\delta v, \delta w)$ is orientation preserving, the orientation of the push-forward operator 
${\bf L}=(\text{id}, K)_\ast{\bf L}_0$ belonging to ${\mathcal F}_{\alpha, \Xi, E}$,  is the complex orientation. The proof of Lemma  
\ref{lem_Kodairs_diff} is complete. 

\end{proof}

Now we are in the position to prove Theorem \ref{kkk<0}.
\begin{proof}[Proof of Theorem \ref{kkk<0}]
We consider the morphism 
 $$
 (\phi,K):(\alpha,H,E)\rightarrow (\alpha',H',E')
 $$
which can be factored into the product
$$
(\alpha,H,E)\xrightarrow{(Id,L)}(\alpha,H^\ast,E')\xrightarrow{(\phi, L')}(\alpha',H',E').
$$
The second morphism is just the geometric pull-back of the data in $(\alpha',H',E')$ via the isomorphism $\phi:\alpha\rightarrow\alpha'$.
We already know from Proposition \ref{m<0} and Remark \ref{mm<0} that the geometric pull-pack and the  push-forward preserve the complex orientation. Therefore,  without loss of generality,  we may assume that
we are in the following situation
$$
(\alpha,H,E)\xrightarrow{(Id,K)} (\alpha,H',E').
$$
Then $\Xi,\Xi'\subset S$ and we define $\Xi''=\Xi\cup \Xi'$. We may assume,  by homotopy,  that the constraints over
$\Xi\cap \Xi'$ coincide. Associated  with $\Xi''$ we have the  constraints $H''$ satisfying $H''\vert \Xi=H$ and $H''\vert \Xi'=H'$. The associated
oriented vector space is denoted by $E''$. The above morphism factors as follows,
$$
(\alpha,H,E)\xrightarrow{(Id,L')}(\alpha,H'',E'')\xrightarrow{(Id,L'')} (\alpha,H',E').
$$
By applying Proposition \ref{ert<ll} to the first morphism and then to  the inverse of the second we see that  the complex orientation is preserved.
This implies the result we are aiming for and the proof of Theorem \ref{kkk<0} is complete.
\end{proof}

 The operators  in this subsection are not yet the operators we have to orient. However, they are very closely  related to the operators arising from the study of the Cauchy-Riemann  section $\ov{\partial}_J$ of the bundle $W\rightarrow Z$.

\subsection{Canonical Orientations for the CR-Operator}
Previously we have introduced a class of linear operators associated with the Cauchy-Riemann operator and have these operators  canonically oriented. In this subsection we relate them to the operators we actually have to orient
in order to turn the Cauchy-Riemann section  $(W\rightarrow Z,\ov{\partial}_J)$ into an oriented sc-Fredholm problem.

From Chapter \ref{four} and Chapter \ref{fredholm-section-z} we recall that the global models $E\rightarrow X$ and the  sections  $F$ representing the CR-section of $W\rightarrow Z$ are constructed from  local situations $F_\cg :\cg\rightarrow \what{\cg}$ in which  $F_\cg$ is represented
by an sc-Fredholm section ${\bf f}_\cg:{\mathcal O}\rightarrow\what{\mathcal O}$ such that the  diagram
$$
\begin{CD}
\cg @>F_\cg >> \what{\cg}\\
@VV \pi V    @VV\what{\pi} V\\
{\mathcal O}@>{\bf f}_\cg>> \what{\mathcal O}
\end{CD}
$$
commutes. 
Here $\pi$ and $\what{\pi}$ are the obvious projections from the graph to the domain. The  parameter spaces ${\mathcal O}$ are  open subsets
of splicing cores which are contractible sets. Later on, in the construction of the ep-groupoid $X$,  we take suitable open subsets of ${\mathcal O}$
such  that the collection ${\mathcal O}_\lambda$ will give on $Z$ a locally finite covering.  In any case, for the current discussion we may assume that the local section 
${\bf f}_\cg$ is defined on a  contractible set.  Therefore there are only two possible orientations
having the continuity property (over the smooth subset).  In view of  the above diagram the same is true for the section $F_\cg$. 

Let us refer to $F_\cg:\cg\rightarrow \what{\cg}$ as a {\bf local situation}.
We shall prove the following result.
\begin{theorem}\label{Maine<}
Every un-noded local situation $F_\cg:\cg\rightarrow \what{\cg}$ has a preferred orientation, called the complex orientation, which enjoys the continuation property along sc-smooth paths. In  other words,  we have for a smooth $x\in \cg$ a canonical orientation $\mathfrak{o}_x$ of $\text{DET}(F_\cg,x)$ so that continuation
along an sc-smooth path connecting two points relates the corresponding orientations. Moreover, the following holds. If 
$F_{\cg}:\cg\rightarrow \what{\cg}$ and $F_{\cg'}:\cg'\rightarrow \what{\cg}'$ are  two  un-noded local situations and if $\Phi:x\rightarrow x'$ is a morphism in $M(\cg,\cg')$ between two smooth points
with lift $\what{\Phi}$, then the associated push forward of the complex orientation is the complex orientation.
\end{theorem}

The proof of Theorem \ref{Maine<} will follow from Propositions \ref{prop_5.61} and \ref{A<A} below. 

We first construct the complex orientation for the local 
situations $F_\cg:\cg\rightarrow \what{\cg}$ built  on the un-noded stable map $\alpha=(S, j, M, \emptyset, u)$ using the notations of Chapter \ref{four} and Chapter \ref{fredholm-section-z}.  We take a good uniformizer 
$$q\rightarrow \alpha_q=(S,j(v),M,\emptyset,\exp_u(\eta)), \quad q=(v, \eta).$$
Note, there is no gluing parameter $a$! The section $\eta\in u^\ast TQ$ satisfies the constraints $\eta(z)\in H_z\equiv H_{u(z)}$ at the points $z\in \Xi$ of the stabilization set $\Xi$. The parameter $q=(v, \eta)$ varies in an open subset ${\mathcal O}$ of an sc-Banach space. We denote by 
$$\cg=\{(q, \alpha_q)\, \vert \, q\in {\mathcal O}\}$$
the graph of the uniformizing family and recall that  $\cg$ possesses the natural M-polyfold structure (in our case even an sc-manifold structure), which turns the projection $\pi:\cg\to \co$
into an sc-diffeomorphism. 

In order to recall also the strong bundle $\wh{\cg}$ over $\cg$ we consider the set $\wh{\co}=\{\wh{q}=(q, \xi)\}$ of pairs, in which $q\in \co$ and $\xi$ is a map 
$$
z\rightarrow [\xi(z):(T_zS,j)\rightarrow (T_{u(z)}Q,J)]
$$
into  complex anti-linear maps of class $H^{2,\delta_0}$ as introduced previously. The bundle $\wh{\co}\to \co$ is a strong bundle.  By 
$$\wh{\cg}=\{(\wh{q}, \wh{\alpha}_{\wh{q}}\, \vert \, \wh{q}\in \wh{\co}\}$$
we denote the graph of the lifted family $\wh{q}\mapsto  \wh{\alpha}_{\wh{q}}$, where $\wh{q}=(q, \xi)$. It is  defined by 
$$\wh{\alpha}_{\wh{q}} =(S,j(v),M,\emptyset,\exp_u(\eta),\Gamma(\exp_u(\eta),u)\circ\xi\circ\delta(v)). 
$$
Here the maps $\Gamma(\exp_u(\eta),u)$ and  $ \delta(v):(TS, j(v)\to (TS, j)$ are the complex linear maps introduced in Section \ref{polstrbundle}. As in the case of the sc-manifold $\cg$ above, the sc-smooth structure on $\wh{\cg}$ is defined by the requirement that the projction $\wh{\cg}\to \wh{\co}$ is an sc-diffeomorphism. 

Continuing with recollections, the section $F_{\cg}$ of the strong bundle $\wh{\cg}\to \cg$ representing the Cauchy-Riemann section is defined by 
\begin{equation}\label{eq_5_59} 
F_{\cg}(q, \alpha_q)=\bigr((q, {\bf f}(q)), \wh{\alpha}_{(q, {\bf f}(q))}\bigr).
\end{equation}
The section ${\bf f}$ of the bundle $\wh{\co}\to \co$ is  the anti-linear map ${\bf f }(q)$, defined at the point $q=(v, \eta)$ by the Cauchy-Riemann section as follows,
\begin{equation}\label{e<r}
\begin{split}
\Gamma(&\exp_u(\eta),u)\circ {\bf f}(q)\circ \delta(v)\\
& =\dfrac{1}{2}\bigl[ T(\exp_u(\eta)) + J(\exp_u(\eta))\circ (T(\exp_u(\eta)))\circ j(v)\bigr].
\end{split}
\end{equation}
It is clear that the linearizations of ${\bf f}$ are related to the linearizations of the Cauchy-Riemann section on the right-hand side of \eqref{e<r}. This will lead to the class of operators introduced in the previous subsection. 

Given a smooth point $q_0=(v_0, \eta_0)$,  we choose an $\ssc^+$-section ${\bf s}$ of $\wh{\co}\to \co$ satisfying ${\bf s}(q_0)={\bf f}(q_0)$ and define the section 
$s$ by 
\begin{equation*}
s(q)=\Gamma(\exp_u(\eta),u)\circ {\bf s}(q)\circ \delta(v), \quad q=(v, \eta).
\end{equation*}
From  \eqref{e<r} we deduce the identity for all $q=(v, \eta)\in \co$, 
\begin{equation}\label{eq_5.61}
\begin{split}
\Gamma(&\exp_u(\eta),u)\circ [{\bf f}(q)-{\bf s}(q)]\circ \delta (v)\\
&= \dfrac{1}{2}\bigl[ T(\exp_u(\eta)) + J(\exp_u(\eta))(T(\exp_u(\eta)))\circ j(v)\bigr]-s(q).
\end{split}
\end{equation}
Abbreviating the fiber derivative at smooth points $q=(v, \eta)$ by 
$\Phi_q(\delta\eta)=\nabla_2\exp_u(\eta)\delta \eta$, we obtain for  the linearization of the identity \eqref{eq_5.61} at the special smooth point $q_0=(v_0, \eta_0)\in \co$ in the direction of $\delta q=(\delta v, \delta \eta)$, recalling that ${\bf f}(q_0)={\bf s}(q_0)$, the identity 
\begin{equation}\label{bb2}
\begin{split}
\Gamma(&\exp_u(\eta_0),u)\circ ({\bf (f-s)}'(q_0)\delta q)\circ \delta(v)\\
& =\bigl[  \nabla (\Phi_{q_0}(\delta\eta)) + J(\exp_u(\eta_0))\nabla(\Phi_{q_0}(\delta \eta))\circ j(v_0)\bigr]\\
&\phantom{=}+A(\delta v,\Phi_{q_0}(\delta\eta))\\
&\equiv  {\bf L}_{q_0}(\delta v,\Phi_{q_0}(\delta\eta)).
\end{split}
\end{equation}
The covariant derivative $\nabla$ has already been used before. The operator  $ {\bf L}_{q_0}$ is defined by the formula 
\eqref{bb2} and belongs to the class of operators discussed in the previous subsection.

 Introducing the linear maps $\wt{\Phi}_{q_0}$ by 
$\wt{\Phi}_{q_0}(\delta v, \delta \eta)=(\delta v, \Phi_{q_0}(\delta \eta))$ and $\Theta_{q_0}(\gamma) =\Gamma(\exp_u(\eta_0),u)\circ \gamma \circ \delta(v_0)$, the equation \eqref{bb2}  
becomes 
\begin{equation}\label{o1}
\Theta_{q_0}\circ \bigl[({\bf f}-{\bf s})'(q_0)(\delta q)\bigr]= {\bf L}_{q_0} \circ \tilde{\Phi}_{q_0}(\delta q)
\end{equation}
where $\delta q=(\delta v, \delta\eta)$.

We now equip the Fredholm operator ${\bf L}_{q_0}$ with the 
the complex orientation from the previous subsection, which, using the coordinate change associated with $ \tilde{\Phi}_{q_0}$ and 
$\Theta_{q_0}$ defines the orientation of the Fredholm operator 
${\bf (f-s)}'(q_0)$.

To be explicit, the orientation $\mathfrak{o}$ of the operator 
$({\bf f}-{\bf s})'(q_0)$ defined by the vector $(h_1\wedge\ldots\wedge h_k)\otimes (h_1^\ast\wedge\ldots \wedge h_l^\ast)$ 
corresponds to the orientation $\mathfrak{o}'$ of ${\bf L}_{q_0}$ defined by the vector
$(\wt{\Phi}_{q_0}(h_1)\wedge\ldots\wedge\wt{\Phi}_{q_0}(h_k))\otimes(h_1^\ast\circ\Theta^{-1}_{q_0}\wedge\ldots\wedge h_l^\ast\circ\Theta^{-1}_{q_0})$.

As a consequence we obtain  an orientation for the section $F_\cg$ of the bundle $\wh{\cg}\to \cg$. The transformations used to transport the orientation 
from ${\bf L}_{q_0}$ to $({\bf f}-{\bf s})'(q_0)$ are level-wise smooth and the orientations for the  operators of type  
${\bf L}_{q_0}$ possess, as we already know, the continuation property along sc-smooth paths. This implies that the orientations of the linearizations at the section ${\bf f}$ at the 
 various smooth points have the continuation property along sc-smooth paths. So far we have proved the following result.

\begin{proposition}\label{prop_5.61}
Given a good uniformizing family $q\mapsto \alpha_q$ around an un-noded  stable map $\alpha$, and given the section $F_\cg$ of the associated strong bundle $\wh{\cg}\rightarrow \cg$, which  represents  the Cauchy-Riemann section,
there is at every smooth point  $(q, \alpha_q)$ a natural orientation of the determinants of the linearizations, called the the complex orientation.
It is the orientation corresponding at a smooth $(q,\alpha_q)$ via the identity \eqref{o1} to the complex orientation of the Cauchy-Riemann operator ${\bf L}_q$.
Since the coordinate change depends smoothly on $q$,  the complex orientation for $F_\cg$ has the continuation property along sc-smooth paths.
\end{proposition}

At this point we have constructed for the section $F_\cg$ of the strong bundle $\wh{\cg}\rightarrow \cg$, built on an un-noded stable map $\alpha$,
a canonical orientation called the {\bf complex orientation}. We shall call  the section $F_\cg:\cg\rightarrow \what{\cg}$ an {\bf un-noded local model} and,  if equipped 
with the complex orientation, we call it  a {\bf canonically oriented un-noded local model}. 

We  next prove the following statement  which relates the orientations between two such local models. 

\begin{proposition}\label{A<A}
Let  $F_\cg:\cg\rightarrow \wh{\cg}$ and $F_{\cg'}:\cg'\rightarrow \wh{\cg}'$ be  canonically oriented un-noded local models.
Let  $\Phi_0\colon  x_0\rightarrow x'_0$ be  an isomorphism between smooth points in $\cg$ resp. $\cg'$,  belonging to to the set $M(\cg,\cg')$ of morphisms. 
Then the push-forward ${(\Phi_0)}_\ast: \text{DET}(F_\cg,x_0)\rightarrow \text{DET}(F_{\cg'},x'_0)$ is orientation preserving. 
\end{proposition}
The proof requires some preparation.  We start with two  smooth un-noded stable maps $\alpha=(S,j,M,\emptyset,u)$ and $\beta=(S',j',M',\emptyset,u')$ having the associated good uniformizing 
families $q\mapsto \alpha_q$  on $q\in \co$ and $p\mapsto  \beta_p$ on $p\in \co'$. We denote the corresponding graphs by 
$\cg$ and $\cg'$.  By construction, the points $q=(v, \eta)\in \co$ belong to $E\oplus H^3_H(u^\ast TQ)$, where the subscript $H$ indicates
the constraints $\eta(z)\in H_{u(z)}$ for $z$  belonging to the  stabilization set $\Xi$. As usual we shall  abbreviate $H_z\equiv H_{u(z)}$. Analogously, the points $p=(w, \xi)\in \co'$ belong 
$E'\oplus H^3_{H'}((u')^\ast TQ)$ satisfying the constraints $H'$ associated with the stabilization set $\Xi'$.

We now assume that there exists an isomorphism 
$$\Phi_0=\bigl((q_0,\alpha_{q_0}), \phi_0, (p_0,\beta_{p_0})\bigr)\in M(\cg, \cg')$$
in the morphism set in which  
$\phi_0\colon \alpha_{q_0}\to \beta_{p_0}$ is an isomorphism between stable maps. 
We use the notation $q_0=(v_0, \eta_0)$ and $p_0=(w_0,\xi_0)$. By Theorem \ref{key-z} there is a germ of an sc-diffeomorphism $p=f(q)$ between the parameters satisfying $p_0=f(q_0)$, and a core-smooth family $q\mapsto  \phi_q$ of isomorphims  
$$
\phi_q:\alpha_q\rightarrow \beta_{f(q)}
$$
between the stable maps satisfying $\phi_{q_0}=\phi_0$.
Differentiating the sc-smooth map $f$  at  the distinguished point $q_0$ we obtain a topological linear isomorphism
\begin{equation}\label{eq_5.65}
\wt{K}=Df(q_0)\colon E\oplus H^3_H(u^\ast TQ)\rightarrow E'\oplus H^3_{H'}((u')^\ast TQ).
\end{equation}
Next we consider the second un-noded local model $F_{\cg'}\colon \cg'\to \wh{\cg}'$ representing the Cauchy-Riemann section and defined by 
$$F_{\cg'}(p, \beta_p)=\bigl((p, {\bf g}(p)), 
\wh{\beta}_{(p, {\bf g}(p))}\bigr),$$
where the section ${\bf g}$ of the bundle $\wh{\co}'\to \co'$ is defined, at $p=(w, \xi)$, by 
\begin{equation}\label{bb11_0}
\begin{split}
\Gamma'(&\exp_{u'}'(\xi),u')\circ ({\bf g}(p)\circ \delta'(w)\\
& =\dfrac{1}{2}\bigl[ T(\exp_{u'}'(\xi)) + J(\exp'_{u'}(\xi))(T(\exp_{u'}'(\xi)))\circ j'(w)\bigr].
\end{split}
\end{equation}
Using the notation $f(q)=(w(q), \xi (q))$ we conclude,  from the properties of the isomorphisms $\phi_q$, that 
$\exp_{u'}'(\xi(q))\circ \phi_q=\exp_u(\eta)$
where $q=(v,\eta)$, and hence 
$$T(\exp_{u'}'(\xi (q))) )\circ T\phi_q=T(\exp_u(\eta)).$$
Moreover, 
$j'(w(q))\circ T\phi_q=T\phi_q\circ j(v).$
Therefore, for all $q=(v, \eta)$,
\begin{equation*}
\begin{split}
\dfrac{1}{2}&\bigl[ T(\exp_{u'}'(\xi(q))) + J(\exp'_{u'}(\xi(q)))(T(\exp_{u'}'(\xi(q)))\circ j'(w(q))\bigr]\circ T\phi_q\\
&\phantom{==}=
\dfrac{1}{2}\bigl[ T(\exp_{u}(\eta)) + J(\exp_{u}(\eta))(T(\exp_{u}(\eta))\circ j(v)\bigr].
\end{split}
\end{equation*}
We obtain the following identity in $q=(a, \eta)$,
\begin{equation*}
[\Gamma'(\exp_{u'}'(\xi (q)),u')\circ {\bf g}(f(q))\circ \delta'(w(q))]\circ  T\phi_q=
\Gamma(\exp_u(\eta ),u)\circ {\bf (f}(q)\circ \delta(v).
\end{equation*}
Hence, for all $q\in \co$ near $q_0$, 
\begin{equation}\label{bb11}
\begin{split}
\bigl[ \Gamma'(&\exp_{u'}'(\xi (q)),u')\circ [{\bf g}(f(q))-{\bf t}(f(q))]\circ \delta'(w(q))\bigr]\circ T\phi_q\\
& = \Gamma(\exp_u(\eta),u)\circ [{\bf (f-s)}(q)\delta q ]\circ \delta(v),
\end{split}
\end{equation}
 where  the section ${\bf t}$ of $\wh{\co}'\to \co$ is defined by 
 $[\Gamma'(\exp_{u'}'(\xi(q)), u')\circ {\bf t}(f(q))\circ \delta'(w(q))]\circ T\phi_q=\Gamma(\exp_u (\eta), u)\circ {\bf s}(q)\circ \delta (v).$
 Differentiating the identity \eqref{bb11} in $q$ at the distinguished  point $q_0=(v_0, \eta_0)$ where $({\bf g}-{\bf t})(f(q_0))=0$ and $({\bf f}-{\bf s})(q_0)=0$, and recalling that $\phi_{q_0}=\phi_0$ and $f(q_0)=p_0$, we obtain the equation 
$$ \Theta'_{p_0}\circ ({\bf g-t})'(p_0)\circ  \wt{K}(\delta q)\circ T\phi_0 = \Theta_{q_0}\circ ({\bf f-s})'(q_0)\delta q,$$
where $\delta q=(\delta v, \delta \eta)$.
In view of the definition of the operator ${\bf L}_{q_0}$ in \eqref{bb2} and the analogous definition of ${\bf L}'_{p_0}$ for the above primed version, we find the equation
\begin{equation}\label{o101}
({\bf L}'_{p_0}\circ \wt{\Phi}'_{p_0}\circ \wt{K}(\delta q))\circ T\phi_0 = {\bf L}_{q_0}\circ\tilde{\Phi}_{q_0}(\delta q).
\end{equation}
Recalling the fiber derivatives 
$$\wt{\Phi}_{q_0}(\delta q)=(\delta v, \nabla_2\exp_u (\eta_0)\delta \eta),$$
where $q_0=(v_0, \eta_0)$ and $\delta q=(\delta v, \delta \eta)$, 
we introduce  the analogous primed version at $p_0=(w_o, \xi_0)$, defined by 
 $$\wt{\Phi}'_{p_0}(\delta p)=(\delta w, \nabla_2\exp'_u (\xi_0)\delta \xi)$$
for $\delta p=(\delta w, \delta \xi)$. We define a new set of constraints $H^0$ by $H^0_\zeta=\nabla_2\exp_u(\eta_0)(H_z)$ for $\zeta=\exp_u (\eta_0)(z)$ and $z\in \Xi$. These are the constraints for the map $u_0=\exp_u(\eta_0)$. Analogously we define the constraints ${H^0}'$ in the primed version for the map 
$u'_0=\exp_{u'}'(\xi_0)$. 
Then 
$$\wt{\Phi}_{q_0}:E\oplus H^3_H(u^*TQ)\to E\oplus H^3_{H^0}(u^*_0TQ)$$
is a topological isomorphism. 
Analogously, in the primed case,
$$\wt{\Phi}'_{q_0}:E'\oplus H^3_{H'}((u')^*TQ)\to E'\oplus H^3_{{H^0}'}((u')^*_0TQ)$$
is a topological isomorphism. The topological linear isomorphism $K$ is defined by requiring that the following diagram is commutative, 
$$
 \begin{CD}
E\oplus H^3_H(u^\ast TQ)@>\wt{K}>> E'\oplus H^3_{H'}({(u')}^\ast TQ)\\
@VV\tilde{\Phi}_{q_0} V    @VV\tilde{\Phi}'_{p_0}V\\
 E\oplus H_{H^0}^3(u_0^\ast TQ)@>K>> E'\oplus H^3_{H^{0'}}({(h_0)}^\ast TQ).
 \end{CD}
 $$
Here $\wt{K}=Df(q_0)$.  The equation \eqref{o101} can be rewritten in the form 
\begin{equation}\label{eq_5.68}
[{\bf L}_{p_0}'\circ K(\delta q)]\circ T\phi_{0} = {\bf L}_{q_0}(\delta q).
\end{equation}
In order to complete the proof of Proposition \ref{A<A} we use the following computation.
 \begin{lemma}\label{hofer}
 The topological linear isomorphism 
 $$
 K: E\oplus H_{H^0}^3(u_0^\ast TQ)\rightarrow E'\oplus H^3_{H^{0'}}({(h_0)}^\ast TQ)
 $$
together with the isomorphism $\phi_0:\alpha_{q_0}\to \beta_{p_0}$  is a general linear coordinate change  $(\phi_0,K)$ in the sense of Definition \ref{GENERAL}.
 \end{lemma}
 \begin{proof}
By a change of coordinates we shall arrive at the setting of Definition \ref{GENERAL}. First we recall the sc-smooth parameter transformation $p=f(q)$. In detail, $p=(w, \xi)=(f_1(q), f_2(q))$ and 
 $q=(v, \eta)$. It satisfies, for all $q$ near $q_0$ the relations
 \begin{align*}
 \exp_{u'}'(\xi)\circ \phi_q &=\exp_u(\eta)\\
 T\phi_q\circ j(v)&=j'(w)\circ T\phi_q.
  \end{align*}
  Moreover, $\phi_q(m)=m'$ for all $m\in M$ and $\phi_{q_0}=\phi_0$. The constraints are $\eta(z)\in H_z$ and $\xi (z')\in H_{z'}'$ for all $z\in \Xi$ and $z'\in \Xi'$. We shall reformulate all these relations in the new coordinates 
 $(\wh{v},\wh{\eta})\in E\oplus H^3_{H^0}(u^\ast_0 TQ)$ and $(\wh{w},\wh{\xi})\in E'\oplus H^3_{H^{0'}}((u_0')^\ast TQ)$,
 where 
 $u_0=\exp_u (\eta_0)$ and $u_0'=\exp_{u'}' (\xi_0)$, defined by the formulae  
 {
 \begin{equation}\label{BUD}
\begin{aligned}
&\text{$v=v_0+\wh{v}$\quad\,\,   \text{and}\quad $\eta=\eta_0+\nabla_2\exp_u(\eta_0)^{-1}\wh{\eta} $}\\
&\text{$w=w_0+\wh{w}$\quad  \text{and}\quad $\xi=\xi_0+\nabla_2\exp'_{u'}(\xi_0)^{-1}\wh{\xi}$}.
\end{aligned}
\end{equation}
}
Correspondingly the almost complex structures become 
$\wh{j}(\wh{v})=j(v_0+\wh{v})$ and $\wh{j}'(\wh{w})=j'(w_0+\wh{w})=j'(w)$ so that 
$$\phi_0\colon (S, \wh{j}(0), M, \emptyset, u_0)\to 
(S', \wh{j}'(0), M', \emptyset, u_0')$$
is the isomorphism from before between the stable maps. We have the associated uniformizers, abbreviating
$\wh{q}=(\wh{v}, \wh{\eta})$ and $\wh{p}=(\wh{w}, \wh{\xi})$, 
 {
 \begin{align*}
\wh{q}\mapsto \wh{\alpha}_{\wh{q}}&=
 \bigl(S, \wh{j}(\wh{v}), M, \emptyset, \exp_u\bigl(\eta_0+\nabla_2\exp_u(\eta_0)^{-1}\wh{\eta}\ \bigr)\\
\wh{p}\mapsto \wh{\beta}_{\wh{p}}&=
 \bigl(S', \wh{j}'(\wh{w}), M', \emptyset, \exp'_{u'}\bigl(\xi_0+\nabla_2\exp_{u'}'(\xi_0)^{-1}\wh{\xi}\ \bigr).
 \end{align*}
}	 
The parameter transformation $f$ becomes, in the new coordinates, {$\wh{p}=(\wh{w}, \wh{\xi})=\wh{f}(\wh{q})$}, and looks in detail as follows,
{
\begin{align*}
\wh{w}&=f_1\bigl( v_0+\wh{v}, \eta_0+\nabla_2\exp_u(\eta_0)^{-1}\wh{\eta}\bigr)-w_0\\
\wh{\xi}&=\nabla_2\exp'_{u'}(\xi_0)\bigl[ f_2\bigl( v_0+\wh{v}, \eta_0+\nabla_2\exp_u(\eta_0)^{-1}\wh{\eta}\bigr)-\xi_0\bigr].
\end{align*}
}
The local sc-diffeomorphism $\wh{f}$  satisfies $\wh{f}(0)=0$.
Rewriting the smooth family $\phi_q=\phi_{(v,\eta)}$ of isomorphisms,  we introduce the map
 {
$$\wh{\phi}_{(\wh{v},\wh{\eta})} :=\phi_{(v_0+\wh{v},\eta_0+\nabla_2\exp_u(\eta_0)^{-1}\wh{\eta})}
$$
}
and obtain that $T\wh{\phi}_{\wh{q}}\circ \wh{j}(\wh{v})=
\wh{j}'(\wh{w})\circ T\wh{\phi}_{\wh{q}}$ and 
 {
$$\exp_{u'}'\bigl(\xi_0+\nabla_2\exp'_{u'}(\xi_0)^{-1}\wh{\xi}\ \bigr)\circ \wh{\phi}_{\wh{q}}=
\exp_{u}\bigl(\eta_0+\nabla_2\exp_{u}(\eta_0)^{-1}\wh{\eta}\bigr).
$$
}
Therefore, $\wh{\phi}_{\wh{q}}:\wh{\alpha}_{\wh{q}}\to 
\wh{\beta}_{\wh{f}(\wh{q})}$ is a smooth family of isomorphisms between stable maps, satisfying $\wh{\phi_0}=\phi_0$. By construction, the constraints become $\wh{\eta}(\zeta)\in H^0_\zeta$ where 
$\zeta=\exp_u (\eta_0)(z)$ for $z\in \Xi$ and $\wh{\xi}(\zeta')\in {H^0}'_{\zeta'}$ where $\zeta'=\exp_{u'}'(\xi_0)(z')$ and $z'\in \Xi'$.
Finally, a computation shows that 
$$\wt{\Phi}'_{p_0}=D\wh{f}(0)=Df(p_0)\circ \wt{\Phi}_{q_0}.$$
Consequently, $K=D\wh{f}(0)$ and we have verified our claim that the pair $(\phi_0, K)$ is indeed a morphism in the sense of 
Definition \ref{GENERAL}. 
 \end{proof}
 In view of Lemma \ref{hofer} and Theorem \ref{kkk<0}  the proof of Proposition \ref{A<A} is finished  \hfill  $\blacksquare$

With Proposition \ref{prop_5.61} and Proposition \ref{A<A} , the proof of Theorem \ref{Maine<} is complete.\hfill 
$\blacksquare$
 
So far we have considered only un-noded local models $\what{\cg}\rightarrow \cg$ with local sections $F_\cg$, where $\cg$ is  built  on families of un-noded stable maps on which the set of nodal pairs is empty. For these sections we have constructed a canonical orientation 
for $\text{DET}(F_\cg,x)$ at smooth points $x\in \cg$, called the complex orientation. It  enjoys the continuation property along smooth paths. Moreover, in view of Proposition \ref{A<A}  a morphism $\Phi\in M(\cg,\cg')$ between two un-noded models  pushes the complex  orientation forward to  the complex orientation.

In the next step we shall canonically orient the section  $F_\cg$ of a noded local  model $\wh{\cg}\to \cg$. We recall that $\cg$ is the graph 
$\cg=\{(q, \alpha_q)\, \vert \, q\in \co\}$ and the family $q\mapsto \alpha_q$ is now a noded uniformizing family where $q$ varies in the splicing core $\co$.  By construction, the splicing core $\co$ is connected and has the property that 
that every sc-smooth loop is contractible. 
Consequently, the section $F_\cg$ has two possible orientations possessing the continuation property along sc-smooth paths. We shall show that there is a canonical choice of one of the  two possible orientations. 
In order to do so we choose a smooth point $x_0=(q_0, \alpha_{q_0})\in \cg$ which  represents an un-noded map.  This can be done since
the un-noded elements are open and dense.

We shall use in the following  that the subset consisting of  the noded stable maps in $\cg$ behaves
as if it has ``codimension $2$''. Namely if $\phi:[0,1]\rightarrow {\mathcal G}$ is an sc-smooth path starting and ending at un-noded elements 
we can take a small deformation avoiding all noded elements. This is  true
because a complex gluing parameter associated with a node has real dimension $2$.  Similarly, if $\phi$ is an sc-smooth path starting at an un-noded element
and ending at a noded element, then it can be deformed to an sc-smooth  path $\psi$, where $\psi(1)$ is the original end point, $\psi(0)$ is the original starting point,
and $\psi(t)$ is un-noded for all $t\in [0,1)$.

 At the previously chosen un-noded element $x_0$ we construct the un-noded local model $\wh{\cg}'\to \cg'$ having the section $F_{\cg'}$. It is represented by the graph $\cg'=\{(p, \alpha'_p)\}$ and satisfies $(0, \alpha'_0)\equiv x_0'=(q_0, \alpha_{q_0})=x_0$. There exists an isomorphism $\Phi=(x_0,\text{id}, x_0')\in M(\cg, \cg')$.

We are going to equip the orientation $\mathfrak{o}$ of $F_\cg$ with the orientation coming from the orientation $\mathfrak{o}_{x_0}$ by continuation along sc-smooth paths.  The orientation 
 $\mathfrak{o}_{x_0}$ is defined by the requirement that its push-forward $\Phi_\ast \mathfrak{o}_{x_0}$ by the coordinate change, is the complex orientation  $\mathfrak{o}_{x_0'}$ of 
 $x_0'\in \cg'$. This definition of the orientation 
  $\mathfrak{o}_{x_0}$ does not depend on the choice of $x_0'$. Indeed, assume that $x_0=x_0''\in \cg''$ and $x_0''=(p_0', \alpha''_{p_0'})$ belongs to a second such un-noded local model $\wh{\cg}''\to \cg''$ having the local section $F_{\cg''}$. Then we have the isomorphism $\Phi'=(x_0, \text{id}, x_0'')\in M(\cg, \cg'')$. By Proposition  \ref{A<A}, the isomorphism
$$\Phi'\circ \Phi^{-1}=(x_0',\text{id}, x_0'')\in M(\cg',\cg'')$$
preserves the complex orientation. Therefore,  
the orientation $\mathfrak{o}_{x_0}$ induced from the complex orientation $\mathfrak{o}_{x_0'}$ agrees with the orientation induced from the complex orientation $\mathfrak{o}_{x_0''}$. The construction is clearly locally stable. Using that the splicing core $\co$ is connected and every sc-smooth loop is contractible, we can, starting from our choice of the orientation of $\text{DET}(F_\cg,x_0)$, using the property of continuation  along smooth paths, orient the set $\text{DET}(F_\cg, x)$ at every smooth point $x\in \cg$. We call this orientation  $\mathfrak{o}$ temporarily the canonical orientation of $\cg$.

We still have to show for an isomorphism $\Phi=(x, \phi, x')\in M(\cg, \cg')$ that the push-forward of the canonical orientation $\mathfrak{o}_x$ is the canonical orientation $\mathfrak{o}_{x'}$, which, as we already know, holds true in the un-noded case. To verify this we take the local sc-diffeomorphism $\phi=t\circ s^{-1}:O(x)\to O(x')$ between the two open neighborhoods such that  the morphisms $\Phi_y=(y, \phi, \phi (y))\in M(\cg, \cg')$ are  associated isomorphisms  for all $y\in O(x)$. We take a smooth point $y_0\in O(x)$ representing an un-noded element and an sc-smooth path 
$\gamma:[0,1]\to O(x)\subset \cg$ connecting $y_0=\gamma (1)$ with $x=\gamma (0)$, such that $\gamma (\tau )$ is un-noded for $\tau\in (0,1]$. Then $\phi (\gamma (\tau))$ is an sc-smooth path connecting $\phi (\gamma (1))$ with $\phi (\gamma (0))=x'$, such that $\phi (\gamma (\tau))$ is un-noded for $\tau \in (0, 1]$. In view of Proposition  \ref{A<A}, the isomorphism 
$\Phi_{y_0}=(y_0, \phi, \phi (y_0))\in M(\cg, \cg')$ pushes the complex orientation $\mathfrak{o}_{y_0}$ to the complex orientation $\mathfrak{o}_{\phi (y_0)}$. Since the canonical orientation $\mathfrak{o}_x$ is obtained by continuation  
 along the path $\gamma$ from the complex orientation $\mathfrak{o}_{y_0}$  and similarly,  the orientationation $\mathfrak{o}_{x}$  is obtained by continuation along the path $\phi\circ \gamma$ from the complex orientation $\mathfrak{o}_{\phi (y_0)}$, we conclude, by continuity, that indeed the isomorphism $\Phi=(x, \phi, x')$ pushes the canonical orientation  $\mathfrak{o}_{x}$  to the canonical orientation  $\mathfrak{o}_{x'}$ as we wanted to show.
 
Now, for every local model $\wh{\cg}\to \cg$ having the section $F_\cg$, the determinant space 
 $\text{DET}(F_\cg, x)$ at a smooth point possesses a canonical orientation, now called {\bf complex orientation}. It has the continuation property along sc-smooth paths. Moreover, a morphism $\Phi=(x, \phi, x')\in M(\cg, \cg')$ between any two local models  connecting smooth points pushes the complex orientation forward to the complex orientation.

Finally we are ready to prove the main result.

\begin{proof}[Proof of Theorem \ref{BINGO}]
In order to complete our construction of the canonical orientation we recall that the models $(E\to X, F)$ for the Cauchy-Riemann section $(W\to Z, \ov{\partial}_J)$ are obtained by gluing the local models $(\wh{\cg}\to \cg, F_\cg)$ by means of isomorphisms together.

At this point all local situations $F_\cg:\cg\rightarrow \what{\cg}$ are oriented with a canonical orientation, called the complex orientation. These orientations are compatible with the morphisms in $M(\cg,\cg')$ and $M(\wh{\cg},\wh{\cg}')$.  In the construction of a model 
$F:X\rightarrow E$ we take open subsets of the graphs $\cg$.  The thereby restricted local models are still oriented. 

It is clear that 
$X$ and $E$,  as disjoint unions of local models, possess a canonical orientation which is compatible with morphisms.

Moreover, if  $X$ and $X'$  are two models constructed  this way we take  the union of both families to obtain the third model $X''$ and two equivalences
$X\rightarrow X''$ and $X'\rightarrow X''$. The equivalences preserve the orientations of the sections of the corresponding strong bundles. 
Whatever model of this kind (it depends on the choice of the local models used),  the orientations are compatible. This implies that the  Cauchy-Riemann section $\ov{\partial}_J:Z\rightarrow W$ possesses  a canonical orientation.
The proof of Theorem \ref{BINGO} is complete.  
\end{proof}

\backmatter
\bibliographystyle{amsalpha}

\bibliographystyle{amsalpha}


\bibliography{Bibliography}

\printindex

\end{document}